\documentclass{article}
\usepackage{amsmath}
\usepackage{amssymb}
\usepackage{amsthm}
\usepackage{cite}
\usepackage{euscript}
\usepackage[arrow,curve,matrix,arc,2cell]{xy}
\usepackage{hyperref}
\UseAllTwocells
\DeclareFontFamily{U}{rsfs}{} \DeclareFontShape{U}{rsfs}{n}{it}{<->
rsfs10}{} \DeclareSymbolFont{mscr}{U}{rsfs}{n}{it}
\DeclareSymbolFontAlphabet{\scr}{mscr}
\def\mathscr{\scr}
\begin{document}
\def\e#1\e{\begin{equation}#1\end{equation}}
\def\ea#1\ea{\begin{align}#1\end{align}}
\def\eq#1{{\rm(\ref{#1})}}
\theoremstyle{plain}
\newtheorem{thm}{Theorem}[section]
\newtheorem{lem}[thm]{Lemma}
\newtheorem{prop}[thm]{Proposition}
\newtheorem{cor}[thm]{Corollary}
\newtheorem{princ}[thm]{Principle}
\theoremstyle{definition}
\newtheorem{dfn}[thm]{Definition}
\newtheorem{altdfn}[thm]{Alternative Definition}
\newtheorem{ex}[thm]{Example}
\newtheorem{rem}[thm]{Remark}
\numberwithin{figure}{section}
\numberwithin{equation}{section}
\def\dim{\mathop{\rm dim}\nolimits}
\def\codim{\mathop{\rm codim}\nolimits}
\def\vdim{\mathop{\rm vdim}\nolimits}
\def\depth{\mathop{\rm depth}\nolimits}
\def\Im{\mathop{\rm Im}\nolimits}
\def\Ker{\mathop{\rm Ker}}
\def\Coker{\mathop{\rm Coker}}
\def\Ho{\mathop{\rm Ho}}
\def\GL{\mathop{\rm GL}}
\def\Stab{\mathop{\rm Stab}\nolimits}
\def\supp{\mathop{\rm supp}}
\def\rank{\mathop{\rm rank}\nolimits}
\def\Sym{\mathop{\rm Sym}\nolimits}
\def\Hom{\mathop{\rm Hom}\nolimits}
\def\bHom{\mathop{\bf Hom}\nolimits}
\def\cHom{\mathop{\mathcal{H}om}\nolimits}
\def\cIso{\mathop{\mathcal{I}so}\nolimits}
\def\bcHom{\mathop{\bs{\mathcal{H}om}}\nolimits}
\def\bcEqu{\mathop{\bs{\mathcal{E}qu}}\nolimits}
\def\id{{\mathop{\rm id}\nolimits}}
\def\Obj{{\rm Obj}}
\def\CSch{{\mathop{\bf C^{\bs\iy}Sch}}}
\def\muKur{{\mathop{\bs\mu\bf Kur}\nolimits}}
\def\muKurb{{\mathop{\bs\mu\bf Kur^b}\nolimits}}
\def\muKurc{{\mathop{\bs\mu\bf Kur^c}\nolimits}}
\def\muKurcin{{\mathop{\bs\mu\bf Kur^c_{in}}\nolimits}}
\def\muKurcsi{{\mathop{\bs\mu\bf Kur^c_{si}}\nolimits}}
\def\muKurcis{{\mathop{\bs\mu\bf Kur^c_{is}}\nolimits}}
\def\muKurcst{{\mathop{\bs\mu\bf Kur^c_{st}}\nolimits}}
\def\muKurgc{{\mathop{\bs\mu\bf Kur^{gc}}}}
\def\muKurgcin{{\mathop{\bs\mu\bf Kur^{gc}_{in}}}}
\def\muKurgcsi{{\mathop{\bs\mu\bf Kur^{gc}_{si}}}}
\def\cmuKurc{{\mathop{\bs\mu\bf\check Kur^c}\nolimits}}
\def\cmuKurcin{{\mathop{\bs\mu\bf\check Kur^c_{in}}\nolimits}}
\def\cmuKurcsi{{\mathop{\bs\mu\bf\check Kur^c_{si}}\nolimits}}
\def\cmuKurcis{{\mathop{\bs\mu\bf\check Kur^c_{is}}\nolimits}}
\def\cmuKurcst{{\mathop{\bs\mu\bf\check Kur^c_{st}}\nolimits}}
\def\cmuKurgc{{\mathop{\bs\mu\bf\check Kur^{gc}}}}
\def\cmuKurgcin{{\mathop{\bs\mu\bf\check Kur^{gc}_{in}}}}
\def\cmuKurgcsi{{\mathop{\bs\mu\bf\check Kur^{gc}_{si}}}}
\def\mKur{{\mathop{\bf mKur}\nolimits}}
\def\mKurb{{\mathop{\bf mKur^b}\nolimits}}
\def\mKurc{{\mathop{\bf mKur^c}\nolimits}}
\def\mKurcin{{\mathop{\bf mKur^c_{in}}\nolimits}}
\def\mKurcsi{{\mathop{\bf mKur^c_{si}}\nolimits}}
\def\mKurcis{{\mathop{\bf mKur^c_{is}}\nolimits}}
\def\mKurcst{{\mathop{\bf mKur^c_{st}}\nolimits}}
\def\mKurgc{{\mathop{\bf mKur^{gc}}}}
\def\mKurgcin{{\mathop{\bf mKur^{gc}_{in}}}}
\def\mKurgcsi{{\mathop{\bf mKur^{gc}_{si}}}}
\def\cmKurc{{\mathop{\bf m\check{K}ur^c}\nolimits}}
\def\cmKurcin{{\mathop{\bf m\check{K}ur^c_{in}}\nolimits}}
\def\cmKurcsi{{\mathop{\bf m\check{K}ur^c_{si}}\nolimits}}
\def\cmKurcis{{\mathop{\bf m\check{K}ur^c_{is}}\nolimits}}
\def\cmKurcst{{\mathop{\bf m\check{K}ur^c_{st}}\nolimits}}
\def\cmKurgc{{\mathop{\bf m\check{K}ur^{gc}}}}
\def\cmKurgcin{{\mathop{\bf m\check{K}ur^{gc}_{in}}}}
\def\cmKurgcsi{{\mathop{\bf m\check{K}ur^{gc}_{si}}}}
\def\Kur{{\mathop{\bf Kur}\nolimits}}
\def\Kurb{{\mathop{\bf Kur^b}\nolimits}}
\def\Kurc{{\mathop{\bf Kur^c}\nolimits}}
\def\Kurcin{{\mathop{\bf Kur^c_{in}}\nolimits}}
\def\Kurcsi{{\mathop{\bf Kur^c_{si}}\nolimits}}
\def\Kurcis{{\mathop{\bf Kur^c_{is}}\nolimits}}
\def\Kurcst{{\mathop{\bf Kur^c_{st}}\nolimits}}
\def\Kurgc{{\mathop{\bf Kur^{gc}}}}
\def\Kurgcin{{\mathop{\bf Kur^{gc}_{in}}}}
\def\Kurgcsi{{\mathop{\bf Kur^{gc}_{si}}}}
\def\cKurcin{{\mathop{\bf \check{K}ur^c_{in}}\nolimits}}
\def\cKurcsi{{\mathop{\bf \check{K}ur^c_{si}}\nolimits}}
\def\cKurcis{{\mathop{\bf \check{K}ur^c_{is}}\nolimits}}
\def\cKurcst{{\mathop{\bf \check{K}ur^c_{st}}\nolimits}}
\def\cKurc{{\mathop{\bf \check{K}ur^c}\nolimits}}
\def\cKurgc{{\mathop{\bf \check{K}ur^{gc}}\nolimits}}
\def\cKurgcin{{\mathop{\bf \check{K}ur^{gc}_{in}}\nolimits}}
\def\cKurgcsi{{\mathop{\bf \check{K}ur^{gc}_{si}}\nolimits}}
\def\KurtrG{{\mathop{\bf Kur_{trG}}\nolimits}}
\def\KurtrGc{{\mathop{\bf Kur_{trG}^c}\nolimits}}
\def\KurtrGgc{{\mathop{\bf Kur_{trG}^{gc}}\nolimits}}
\def\GKur{{\mathop{\bf GKur}\nolimits}}
\def\GKurb{{\mathop{\bf GKur^b}\nolimits}}
\def\GKurc{{\mathop{\bf GKur^c}\nolimits}}
\def\GmuKur{{\mathop{\bf G\bs\mu Kur}\nolimits}}
\def\GmuKurb{{\mathop{\bf G\bs\mu Kur^b}\nolimits}}
\def\GmuKurc{{\mathop{\bf G\bs\mu Kur^c}\nolimits}}
\def\Man{{\mathop{\bf Man}}}
\def\Manb{{\mathop{\bf Man^b}}}
\def\Manc{{\mathop{\bf Man^c}}}
\def\Mangc{{\mathop{\bf Man^{gc}}}}
\def\Mancst{{\mathop{\bf Man^c_{st}}}}
\def\Mancsi{{\mathop{\bf Man^c_{si}}}}
\def\Mancis{{\mathop{\bf Man^c_{is}}}}
\def\Mancin{{\mathop{\bf Man^c_{in}}}}
\def\Mangcin{{\mathop{\bf Man^{gc}_{in}}}}
\def\Mangcsi{{\mathop{\bf Man^{gc}_{si}}}}
\def\cManc{{\mathop{\bf\check{M}an^c}}}
\def\cMancst{{\mathop{\bf\check{M}an^c_{st}}}}
\def\cMancis{{\mathop{\bf\check{M}an^c_{is}}}}
\def\cMancsi{{\mathop{\bf\check{M}an^c_{si}}}}
\def\cMancin{{\mathop{\bf\check{M}an^c_{in}}}}
\def\cMangc{{\mathop{\bf\check{M}an^{gc}}}}
\def\cMangcin{{\mathop{\bf\check{M}an^{gc}_{in}}}}
\def\cMangcsi{{\mathop{\bf\check{M}an^{gc}_{si}}}}
\def\dMan{{\mathop{\bf dMan}}}
\def\dOrb{{\mathop{\bf dOrb}}}
\def\Orb{{\mathop{\bf Orb}}}
\def\Orbb{{\mathop{\bf Orb^b}}}
\def\Orbc{{\mathop{\bf Orb^c}}}
\def\OrbKur{{\mathop{\bf Orb_{\rm Kur}}}}
\def\Mon{{\mathop{\bf Mon}}}
\def\ul{\underline}
\def\bs{\boldsymbol}
\def\ge{\geqslant}
\def\le{\leqslant\nobreak}
\def\pr{\prec}
\def\O{{\mathcal O}}
\def\R{{\mathbin{\mathbb R}}}
\def\Z{{\mathbin{\mathbb Z}}}
\def\Q{{\mathbin{\mathbb Q}}}
\def\N{{\mathbin{\mathbb N}}}
\def\C{{\mathbin{\mathbb C}}}
\def\CP{{\mathbin{\mathbb{CP}}}}
\def\RP{{\mathbin{\mathbb{RP}}}}
\def\fC{{\mathbin{\mathfrak C}\kern.05em}}
\def\fD{{\mathbin{\mathfrak D}}}
\def\fE{{\mathbin{\mathfrak E}}}
\def\fF{{\mathbin{\mathfrak F}}}
\def\cA{{\mathbin{\cal A}}}
\def\cB{{\mathbin{\cal B}}}
\def\cC{{\mathbin{\cal C}}}
\def\cD{{\mathbin{\cal D}}}
\def\cE{{\mathbin{\cal E}}}
\def\cF{{\mathbin{\cal F}}}
\def\cG{{\mathbin{\cal G}}}
\def\cH{{\mathbin{\cal H}}}
\def\cI{{\mathbin{\cal I}}}
\def\cJ{{\mathbin{\cal J}}}
\def\cK{{\mathbin{\cal K}}}
\def\cL{{\mathbin{\cal L}}}
\def\cM{{\mathbin{\cal M}}}
\def\cN{{\mathbin{\cal N}}}
\def\cP{{\mathbin{\cal P}}}
\def\cQ{{\mathbin{\cal Q}}}
\def\cR{{\mathbin{\cal R}}}
\def\cS{{\mathbin{\cal S}}}
\def\cT{{\mathbin{\cal T}\kern -0.1em}}
\def\cW{{\mathbin{\cal W}}}
\def\fV{{\mathbin{\mathfrak V}}}
\def\fW{{\mathbin{\mathfrak W}}}
\def\fX{{\mathbin{\mathfrak X}}}
\def\fY{{\mathbin{\mathfrak Y}}}
\def\fZ{{\mathbin{\mathfrak Z}}}
\def\fe{{\mathfrak e}}
\def\ff{{\mathfrak f}}
\def\fg{{\mathfrak g}}
\def\fh{{\mathfrak h}}
\def\oM{{\mathbin{\smash{\,\,\overline{\!\!\mathcal M\!}\,}}}}
\def\bcE{{\mathbin{\bs{\cal E}}}}
\def\ur{{\underline{r\kern -0.15em}\kern 0.15em}{}}
\def\uU{{{\underline{U\kern -0.25em}\kern 0.2em}{}}{}}
\def\uX{{{\underline{X\!}\,}{}}{}}
\def\cU{{\cal U}}
\def\cV{{\cal V}}
\def\cW{{\cal W}}
\def\bA{{\bs A}}
\def\bB{{\bs B}}
\def\bC{{\bs C}}
\def\bD{{\bs D}}
\def\bE{{\bs E}}
\def\bF{{\bs F}}
\def\bG{{\bs G}}
\def\bH{{\bs H}}
\def\bM{{\bs M}}
\def\bN{{\bs N}}
\def\bO{{\bs O}}
\def\bP{{\bs P}}
\def\bQ{{\bs Q}}
\def\bS{{\bs S}}
\def\bT{{\bs T}}
\def\bU{{\bs U}}
\def\bV{{\bs V}}
\def\bW{{\bs W}\kern -0.1em}
\def\bX{{\bs X}}
\def\bY{{\bs Y}\kern -0.1em}
\def\bZ{{\bs Z}}
\def\al{\alpha}
\def\be{\beta}
\def\ga{\gamma}
\def\de{\delta}
\def\io{\iota}
\def\ep{\epsilon}
\def\la{\lambda}
\def\ka{\kappa}
\def\th{\theta}
\def\ze{\zeta}
\def\up{\upsilon}
\def\vp{\varphi}
\def\si{\sigma}
\def\om{\omega}
\def\De{\Delta}
\def\La{\Lambda}
\def\Om{\Omega}
\def\Io{{\rm I}}
\def\Ka{{\rm K}}
\def\Mu{{\rm M}}
\def\Nu{{\rm N}}
\def\Tau{{\rm T}}
\def\Up{\Upsilon}
\def\Al{{\rm A}}
\def\Be{{\rm B}}
\def\Ga{\Gamma}
\def\Si{\Sigma}
\def\Th{\Theta}
\def\utau{{\underline{\tau\kern -0.2em}\kern 0.2em}{}}
\def\uchi{{\underline{\chi\kern -0.1em}\kern 0.1em}{}}
\def\ueta{{\underline{\eta\kern -0.1em}\kern 0.1em}{}}
\def\pd{\partial}
\def\ts{\textstyle}
\def\st{\scriptstyle}
\def\sst{\scriptscriptstyle}
\def\w{\wedge}
\def\sm{\setminus}
\def\lt{\ltimes}
\def\bu{\bullet}
\def\sh{\sharp}
\def\op{\oplus}
\def\od{\odot}
\def\op{\oplus}
\def\ot{\otimes}
\def\ov{\overline}
\def\bigop{\bigoplus}
\def\bigot{\bigotimes}
\def\iy{\infty}
\def\es{\emptyset}
\def\ra{\rightarrow}
\def\rra{\rightrightarrows}
\def\Ra{\Rightarrow}
\def\RRa{\Rrightarrow}
\def\Longra{\Longrightarrow}
\def\ab{\allowbreak}
\def\longra{\longrightarrow}
\def\hookra{\hookrightarrow}
\def\dashra{\dashrightarrow}
\def\ha{{\ts\frac{1}{2}}}
\def\t{\times}
\def\ci{\circ}
\def\ti{\tilde}
\def\d{{\rm d}}
\def\md#1{\vert #1 \vert}
\def\bmd#1{\big\vert #1 \big\vert}
\def\an#1{\langle #1 \rangle}
\def\ban#1{\bigl\langle #1 \bigr\rangle}
\title{A new definition of Kuranishi space}
\author{Dominic Joyce}
\date{Preliminary version, October 2015}
\maketitle

\begin{abstract} `Kuranishi spaces' were introduced in the work of Fukaya, Oh, Ohta and Ono \cite{Fuka,FOOO1,FOOO2,FOOO3, FOOO4,FOOO5,FOOO6,FOOO7,FOOO8,FuOn} in symplectic geometry, as the geometric structure on moduli spaces of $J$-holomorphic curves. Finding a satisfactory definition of Kuranishi space has been the subject of recent debate~\cite{FOOO6,McWe1,McWe2,McWe3,Yang1,Yang2,Yang3}.

We propose three new definitions of Kuranishi space: a simple `manifold' version, $\mu$-{\it Kuranishi spaces}, which form an ordinary category $\muKur$; a more complicated `manifold' version, {\it m-Kuranishi spaces}, which form a weak 2-category $\mKur$; and an `orbifold' version, {\it Kuranishi spaces}, which form a weak 2-category $\Kur$. These are related by an equivalence of categories $\muKur\simeq\Ho(\mKur)$, where $\Ho(\mKur)$ is the homotopy category of $\mKur$, and by a full and faithful embedding $\mKur\hookra\Kur$.

Most previous definitions of Kuranishi space, and related notions of `good coordinate system', `Kuranishi atlas', or `Kuranishi structure', on a topological space $X$, can be used to define a Kuranishi space $\bX$ in our sense, uniquely up to equivalence.

This book is surveyed in~\cite{Joyc11}.
\end{abstract}

\setcounter{tocdepth}{2}
\tableofcontents

\section{Introduction}
\label{ku1}

Kuranishi spaces were introduced in the work of Fukaya, Oh, Ohta and Ono \cite{Fuka,FOOO1,FOOO2,FOOO3,FOOO4,FOOO5,FOOO6,FOOO7,FOOO8,FuOn},
as the geometric structure on moduli spaces of $J$-holomorphic curves, which was to be used to define virtual cycles and virtual chains for such moduli spaces, for applications in symplectic geometry such as Gromov--Witten invariants, Lagrangian Floer cohomology, and Symplectic Field Theory. An alternative theory, philosophically rather different, but which does essentially the same job, is the polyfolds of Hofer, Wysocki and Zehnder~\cite{Hofe,HWZ1,HWZ2,HWZ3,HWZ4,HWZ5,HWZ6}.

Something which has consistently been a problem with Kuranishi spaces, since their introduction by Fukaya and Ono \cite[\S 5]{FuOn} in 1999, has been to find a satisfactory definition, preferably as a category (or higher category) of geometric spaces, with a well-behaved notion of morphism, and good functorial properties. The definition used by Fukaya et al.\ has changed several times as their work has evolved \cite{Fuka,FOOO1,FOOO2,FOOO3,FOOO4,FOOO5, FOOO6,FOOO7,FOOO8,FuOn}, and others including McDuff and Wehrheim \cite{McWe1,McWe2,McWe3}, Yang \cite{Yang1,Yang2,Yang3}, and the author \cite{Joyc1,Joyc2}, have proposed their own variations. For a review of previous definitions of Kuranishi space in the literature, see Appendix~\ref{kuA}.

The purpose of this book is to propose a new definition of Kuranishi space, which we hope will become accepted as final, replacing previous definitions. In fact, we give three variations:
\begin{itemize}
\setlength{\itemsep}{0pt}
\setlength{\parsep}{0pt}
\item[(i)] a simple `manifold' version, `$\mu$-Kuranishi spaces', with trivial isotropy groups, which form an ordinary category $\muKur$;
\item[(ii)] a more complicated `manifold' version, `m-Kuranishi spaces', with trivial isotropy groups, which form a weak 2-category $\mKur$; and
\item[(iii)] the full `orbifold' version, `Kuranishi spaces', with finite isotropy groups, which form a weak 2-category $\Kur$.
\end{itemize}
These are related by an equivalence of categories $\muKur\simeq\Ho(\mKur)$, where $\Ho(\mKur)$ is the homotopy category of $\mKur$, and by a full and faithful embedding $\mKur\hookra\Kur$. Here $\mu$-Kuranishi spaces and m-Kuranishi spaces are kinds of `derived manifolds', and Kuranishi spaces a kind of `derived orbifolds'. Kuranishi spaces are the most relevant to symplectic geometry.

The main technical innovation in our definition is our treatment of {\it coordinate changes\/} between Kuranishi neighbourhoods on a topological space $X$. In the `$\mu$-' version in \S\ref{ku2}, $\mu$-coordinate changes $\Phi_{ij}:(V_i,E_i,s_i,\psi_i)\ra (V_j,E_j,s_j,\psi_j)$ are {\it germs of equivalence classes\/} $[V_{ij},\phi_{ij},\hat\phi_{ij}]$ of triples $(V_{ij},\phi_{ij},\hat\phi_{ij})$, where $(V_{ij},\phi_{ij},\hat\phi_{ij})$ is a generalized Fukaya--Oh--Ohta--Ono-style coordinate change, and the equivalence relation is not obvious. In the full version in \S\ref{ku4}, our coordinate changes $\Phi_{ij}=(P_{ij},\pi_{ij},\phi_{ij},\hat\phi_{ij}):(V_i,E_i,\Ga_i,s_i,\psi_i)\ra (V_j,E_j,\Ga_j,s_j,\psi_j)$ are generalized Fukaya--Oh--Ohta--Ono-style coordinate changes, but we also introduce 2-{\it morphisms between coordinate changes}, involving germs of equivalence classes, and making Kuranishi neighbourhoods on $X$ into a 2-category.

One special feature of these new notions of coordinate change, is that in the \hbox{`$\mu$-'} version, {\it $\mu$-coordinate changes\/ $(V_i,E_i,s_i,\psi_i)\ra (V_j,E_j,s_j,\psi_j)$ form a sheaf on\/} $\Im\psi_i\cap\Im\psi_j$, as in \S\ref{ku22}, and in the full version, {\it coordinate changes\/ $(V_i,E_i,\ab\Ga_i,\ab s_i,\ab\psi_i)\ab\ra (V_j,E_j,\Ga_j,s_j,\psi_j)$ form a stack on\/} $\Im\psi_i\cap\Im\psi_j$, as in \S\ref{ku42}. These sheaf/stack properties are very useful --- for example, they are essential in defining composition of (1-)morphisms between ($\mu$-)Kuranishi spaces. Also, ($\mu$-)coordinate changes are exactly the morphisms between ($\mu$-)Kuranishi neighbourhoods which are invertible, or invertible up to 2-isomorphism.

In \S\ref{ku48} we show that a topological space $X$ equipped with a `Kuranishi structure' or a `weak good coordinate system' in the sense of Fukaya, Oh, Ohta and Ono \cite{FOOO1,FOOO2}, or with a `weak Kuranishi atlas' in the sense of McDuff and Wehrheim \cite{McWe2,McWe3}, or with a `Kuranishi structure' in the sense of Yang \cite{Yang1,Yang2,Yang3}, or with a `polyfold Fredholm structure' in the sense of Hofer, Wysocki and Zehnder \cite{Hofe,HWZ1,HWZ2,HWZ3,HWZ4,HWZ5,HWZ6} 
can all be made into a Kuranishi space $\bX$ in our sense, uniquely up to equivalence in the 2-category~$\Kur$. 

This implies that more-or-less all of the definitions of `Kuranishi space' and related notions of `good coordinate system', `Kuranishi atlas', or `Kuranishi structure' in \cite{Fuka,FOOO1,FOOO2,FOOO3,FOOO4,FOOO5,FOOO6,FOOO7,FOOO8,McWe2,McWe3,Yang1,Yang2,Yang3} can be transformed to Kuranishi spaces in our sense. Thus, the proofs of the existence of Kuranishi structures, polyfold structure, etc.\ on moduli spaces of $J$-holomorphic curves in Fukaya--Oh--Ohta--Ono \cite{Fuka,FOOO1,FOOO2,FOOO3,FOOO4,FOOO5,FOOO6}, or McDuff and Wehrheim \cite{McWe2,McWe3} or Hofer--Wysocki--Zehnder \cite{Hofe,HWZ1,HWZ2,HWZ3,HWZ4,HWZ5,HWZ6}, also make these moduli spaces into Kuranishi spaces in our sense.

For applications in symplectic geometry, it will be an advantage that our Kuranishi spaces form a well-behaved 2-category. Here are three examples:
\begin{itemize}
\setlength{\itemsep}{0pt}
\setlength{\parsep}{0pt}
\item[(a)] Defining a Kuranishi structure on a moduli space $\cM$ involves making many arbitrary choices, and it is helpful to know how different choices are related. In our theory, different choices of Kuranishi structure on $\cM$ should yield equivalent Kuranishi spaces $\bs\cM,\bs\cM'$ in the 2-category~$\Kur$.

Equivalence of Kuranishi spaces should be compared with McDuff and Wehrheim's `commensurate' Kuranishi atlases, as in \S\ref{kuA3}, and Dingyu Yang's `R-equivalent' Kuranishi structures, as in~\S\ref{kuA4}.
\item[(b)] Fukaya, Oh, Ohta and Ono use finite group actions on Kuranishi spaces in \cite{Fuka}, \cite[\S 7]{FOOO4}, and $T^n$-actions on Kuranishi spaces in \cite{FOOO2,FOOO3}. In our theory it is easy to define and study actions of Lie groups on Kuranishi spaces, in a more flexible way than in~\cite{FOOO2,FOOO3}.

\item[(c)] Fukaya \cite[\S 3, \S 5]{Fuka} (see also \cite[\S 4.2]{FOOO5}) works with a forgetful morphism $\mathfrak{forget}:\bs\cM_{l,1}(\be)\ra\bs\cM_{l,0}(\be)$ of $J$-holomorphic curve moduli spaces, which is clearly intended to be some kind of morphism of Kuranishi spaces, without defining the concept. Our theory sets this on a rigorous footing.
\end{itemize}

In ($\mu$-)Kuranishi neighbourhoods $(V_i,E_i,s_i,\psi_i)$ or $(V_i,E_i,\Ga_i,s_i,\psi_i)$, we can take the $V_i$ to be manifolds {\it without boundary}, or {\it with boundary}, or {\it with corners}, and so define ($\mu$- or m-)Kuranishi spaces $\bX$ without or with boundary, or with corners. The boundary and corners cases introduce extra complications. So for simplicity, in \S\ref{ku2} and \S\ref{ku4} we first define ($\mu$- or m-)Kuranishi spaces without boundary, and then explain in \S\ref{ku3} and \S\ref{ku5} how to modify the picture to include boundary and corners. When we just say `manifold' or `($\mu$-)Kuranishi space', we generally mean manifolds and ($\mu$-)Kuranishi spaces without boundary. 

This book will concentrate on the definitions of the (2-)categories of \hbox{($\mu$-)}\ab Kuranishi spaces, and their relation to other notions of Kuranishi space and similar ideas in the literature. In a sequel \cite{Joyc12}, we will study the `differential geometry' of ($\mu$-)Kuranishi spaces: submersions, immersions, embeddings, d-transverse fibre products, orientations, bordism, good coordinate systems, orbifold strata, blow-ups, and so on --- a toolkit of ideas, definitions and results for working with Kuranishi spaces. 

Our definitions of ($\mu$-)Kuranishi space are based on the author's theory of `d-manifolds' and `d-orbifolds' \cite{Joyc6,Joyc7,Joyc8}, which are `derived' smooth manifolds and orbifolds, where `derived' is in the sense of derived algebraic geometry. In the sequel \cite{Joyc12}, amongst other results we will prove:

\begin{thm} There is an equivalence of categories
\begin{equation*}
\smash{\muKur\simeq\Ho(\dMan),}
\end{equation*}
where $\muKur$ is the category of\/ $\mu$-Kuranishi spaces in {\rm\S\ref{ku2},} and\/ $\dMan$ is the strict\/ $2$-category of d-manifolds $\dMan$ from the author {\rm\cite{Joyc6,Joyc7,Joyc8},} and\/ $\Ho(\dMan)$ is its homotopy category. There are also equivalences of weak\/ $2$-categories
\begin{equation*}
\smash{\mKur\simeq\dMan,\qquad \Kur\simeq \dOrb,}
\end{equation*}
where $\mKur$ is the weak\/ $2$-category of m-Kuranishi spaces in {\rm\S\ref{ku47},} and\/ $\Kur$ the weak\/ $2$-category of m-Kuranishi spaces in {\rm\S\ref{ku4},} and\/ $\dOrb$ the strict\/ $2$-category of d-orbifolds $\dOrb$ from\/~{\rm\cite{Joyc6,Joyc7,Joyc8}}.
\label{ku1thm}
\end{thm}

Thus, for many purposes, Kuranishi spaces and d-orbifolds are interchangeable. The full theory of d-manifolds and d-orbifolds \cite{Joyc6,Joyc7,Joyc8} requires a significant amount of background in conventional algebraic geometry, derived algebraic geometry, and $C^\iy$-algebraic geometry. Symplectic geometers interested in Kuranishi spaces tend not to be familiar with these, and have (so far) mostly been unwilling to study them in order to learn about d-manifolds and d-orbifolds.

One of the main goals of this book is to provide an alternative to d-manifolds and d-orbifolds suited for symplectic geometers, that is as far as possible self-contained, has minimal prerequisites other than a good understanding of ordinary differential geometry, and uses little heavy machinery from algebraic geometry or derived algebraic geometry.

Therefore d-manifolds and d-orbifolds will only be mentioned very occasionally below, and no technical details about d-manifolds or d-orbifolds will be given. Readers who would like to know more are encouraged to begin with the survey paper~\cite{Joyc6}.

In \cite{Joyc10} the author defines `M-homology', a new homology theory $MH_*(Y;R)$ of a manifold $Y$ and a commutative ring $R$, canonically isomorphic to ordinary homology $H_*(Y;R)$. The chains $MC_k(Y;R)$ for $MH_*(Y;R)$ are $R$-modules generated by quadruples $[V,n,s,t]$ for $V$ an oriented manifold with corners (or something similar) with $\dim V=n+k$ and $s:V\ra\R^n$, $t:V\ra Y$ smooth maps with $s$ proper near 0 in $\R^n$. In future work \cite{Joyc12} the author will define virtual chains for Kuranishi spaces in M-homology, for applications in symplectic geometry. 

This book is surveyed in \cite{Joyc11}, concentrating on the definition of the 2-category $\Kur$ of Kuranishi spaces without boundary.
\medskip

\noindent{\it Acknowledgements.} I would like to thank Lino Amorim, Kenji Fukaya, Helmut Hofer, Dusa McDuff, and Dingyu Yang for helpful conversations, and a referee for many useful comments. I also thank the Simons Center for Geometry and Physics, Stony Brook, where this project began during a visit in May 2014. This research was supported by EPSRC grants EP/H035303/1 and EP/J016950/1.

\section{\texorpdfstring{The category of $\mu$-Kuranishi spaces}{The category of \textmu-Kuranishi spaces}}
\label{ku2}

We now define the category $\muKur$ of {\it $\mu$-Kuranishi spaces}. Here the `$\mu$-' stands for `manifold' (as opposed to orbifold). It means that our $\mu$-Kuranishi spaces have trivial isotropy groups, and we do not include finite groups $\Ga$ in Kuranishi neighbourhoods $(V,E,\Ga,s,\psi)$.

In this section all manifolds and $\mu$-Kuranishi spaces will be {\it without boundary}. We explain how to extend it to include boundaries and corners in~\S\ref{ku3}.

The proofs of Theorems \ref{ku2thm1} and \ref{ku2thm2} are deferred until \S\ref{ku61} and~\S\ref{ku62}.

\subsection{\texorpdfstring{$\mu$-Kuranishi neighbourhoods and their morphisms}{\textmu-Kuranishi neighbourhoods and their morphisms}}
\label{ku21}

We will use following `$O(s)$' and `$O(s^2)$' notation very often:

\begin{dfn} Let $V$ be a manifold, $E\ra V$ a vector bundle, and $s\in C^\iy(E)$ a smooth section.
\begin{itemize}
\setlength{\itemsep}{0pt}
\setlength{\parsep}{0pt}
\item[(i)] If $F\ra V$ is another vector bundle and $t_1,t_2\in C^\iy(F)$ are smooth sections, we write $t_1=t_2+O(s)$ if there exists $\al\in C^\iy(E^*\ot
F)$ such that $t_1=t_2+\al\cdot s$ in $C^\iy(F)$, where the
contraction $\al\cdot s$ is formed using the natural pairing of
vector bundles $(E^*\ot F)\t E\ra F$ over $V$.
\item[(ii)] We write $t_1=t_2+O(s^2)$ if there exists $\al\in
C^\iy(E^*\ot E^*\ot F)$ such that $t_1=t_2+\al\cdot(s\ot s)$ in
$C^\iy(F)$, where $\al\cdot(s\ot s)$ uses the pairing~$(E^*\ot
E^*\ot F)\t(E\ot E)\ra F$.
\end{itemize}

Now let $W$ be another manifold, and $f,g:V\ra W$ be smooth maps. 
\begin{itemize}
\setlength{\itemsep}{0pt}
\setlength{\parsep}{0pt}
\item[(iii)] We write $f=g+O(s)$ if whenever $h:W\ra\R$ is a smooth map, there
exists $\al\in C^\iy(E^*)$ such that $h\ci f=h\ci g+\al\cdot s$.
\item[(iv)] We write $f=g+O(s^2)$ if whenever $h:W\ra\R$ is a smooth
map, there exists $\al\in C^\iy(E^*\ot E^*)$ such that $h\ci f=h\ci
g+\al\cdot (s\ot s)$.
\item[(v)] If $\La\in C^\iy\bigl(E^*\ot f^*(TW)\bigr)$, we write $f=g+\La\cdot s+O(s^2)$ if whenever $h:W\ra\R$ is a smooth map, there exists $\al\in C^\iy(E^*\ot E^*)$ such that $h\ci f=h\ci g+\La\cdot(s\ot f^*(\d h))+\al\cdot (s\ot s)$. Here $s\ot f^*(\d h)$ lies in $C^\iy\bigl(E\ot f^*(T^*W)\bigr)$, and so has an obvious pairing with $\La$.
\end{itemize}

Next suppose $f,g:V\ra W$ with $f=g+O(s)$, and $F\ra V$, $G\ra W$ are vector bundles, and $t_1\in C^\iy(F\ot f^*(G))$, $t_2\in C^\iy(F\ot g^*(G))$. We wish to compare $t_1,t_2$, even though they are sections of different vector bundles.
\begin{itemize}
\setlength{\itemsep}{0pt}
\setlength{\parsep}{0pt}
\item[(vi)] We write $t_1=t_2+O(s)$ if for all $\be\in C^\iy(G^*)$ we have $t_1\cdot f^*(\be)=t_2\cdot g^*(\be)+O(s)$ in sections of $F\ra V$, as in (i).
\end{itemize}
Given any $t_1\in C^\iy(F\ot f^*(G))$, there exists $t_2\in C^\iy(F\ot g^*(G))$ with $t_1=t_2+O(s)$ in the sense of (vi), and if $t_2'$ is an alternative choice then $t_2'=t_2+O(s)$ in the sense of (i). The moral is that if $f=g+O(s)$ and we are interested in smooth sections $t$ of $F\ot f^*(G)$ up to $O(s)$, then we can treat the vector bundles $F\ot f^*(G)$ and $F\ot g^*(G)$ as essentially the same.

If instead $f,g:V\ra W$ with $f=g+O(s^2)$ and $F,G,t_1,t_2$ are as above\begin{itemize}
\setlength{\itemsep}{0pt}
\setlength{\parsep}{0pt}
\item[(vii)] We write $t_1=t_2+O(s^2)$ if for all $\be\in C^\iy(G^*)$ we have $t_1\cdot f^*(\be)=t_2\cdot g^*(\be)+O(s^2)$ in sections of $F\ra V$, as in (ii).
\end{itemize}
Given any $t_1\in C^\iy(F\ot f^*(G))$, there exists $t_2\in C^\iy(F\ot g^*(G))$ with $t_1=t_2+O(s^2)$ in the sense of (vii), and if $t_2'$ is an alternative choice then $t_2'=t_2+O(s^2)$ in the sense of~(ii). 

By a partition of unity argument, (i)--(vii) are local conditions on $V$. That is, if they hold on an open neighbourhood of each point of $s^{-1}(0)\subseteq V$, then they hold on~$V$.
\label{ku2def1}
\end{dfn}

We define $\mu$-Kuranishi neighbourhoods:

\begin{dfn} Let $X$ be a topological space. A {\it $\mu$-Kuranishi neighbourhood\/} on $X$ is a quadruple $(V,E,s,\psi)$ such that:
\begin{itemize}
\setlength{\itemsep}{0pt}
\setlength{\parsep}{0pt}
\item[(a)] $V$ is a smooth manifold, without boundary. We allow~$V=\es$.
\item[(b)] $\pi:E\ra V$ is a real vector bundle over $V$, called the {\it obstruction bundle}.
\item[(c)] $s:V\ra E$ is a smooth section of $E$, called the {\it Kuranishi section}.
\item[(d)] $\psi$ is a homeomorphism from $s^{-1}(0)$ to an open subset $\Im\psi$ in $X$, where $\Im\psi=\bigl\{\psi(x):x\in s^{-1}(0)\bigr\}$ is the image of $\psi$, and is called the {\it footprint\/} of~$(V,E,s,\psi)$.
\end{itemize}
If $S\subseteq X$ is open, by a {\it $\mu$-Kuranishi neighbourhood over\/} $S$, we mean a $\mu$-Kuranishi neighbourhood $(V,E,s,\psi)$ on $X$ with $S\subseteq\Im\psi\subseteq X$.
\label{ku2def2}
\end{dfn}

The next definition, of morphisms between $\mu$-Kuranishi neighbourhoods, is crucial for our programme. We comment on it in Remark~\ref{ku2rem1}.

\begin{dfn} Let $X,Y$ be topological spaces, $f:X\ra Y$ a continuous map, $(V_i,E_i,s_i,\psi_i)$, $(V_j,E_j,s_j,\psi_j)$ be $\mu$-Kuranishi neighbourhoods on $X,Y$ respectively, and $S\subseteq\Im\psi_i\cap f^{-1}(\Im\psi_j)\subseteq X$ be an open set. Consider triples $(V_{ij},\phi_{ij},\hat\phi_{ij})$ satisfying:
\begin{itemize}
\setlength{\itemsep}{0pt}
\setlength{\parsep}{0pt}
\item[(a)] $V_{ij}$ is an open neighbourhood of $\psi_i^{-1}(S)$ in $V_i$. We do not require that $V_{ij}\cap s_i^{-1}(0)=\psi_i^{-1}(S)$, only that~$\psi_i^{-1}(S)\subseteq V_{ij}\cap s_i^{-1}(0)\subseteq V_{ij}$.
\item[(b)] $\phi_{ij}:V_{ij}\ra V_j$ is a smooth map.
\item[(c)] $\hat\phi_{ij}:E_i\vert_{V_{ij}}\ra\phi_{ij}^*(E_j)$ is a morphism of vector bundles on $V_{ij}$.
\item[(d)] $\hat\phi_{ij}(s_i\vert_{V_{ij}})=\phi_{ij}^*(s_j)+O(s_i^2)$, in the sense of Definition~\ref{ku2def1}.
\item[(e)] $f\ci\psi_i=\psi_j\ci\phi_{ij}$ on $s_i^{-1}(0)\cap V_{ij}$.
\end{itemize}

Define a binary relation $\sim$ on such triples $(V_{ij},\phi_{ij},\hat\phi_{ij})$ by $(V_{ij},\phi_{ij},\hat\phi_{ij})\sim(V_{ij}',\phi_{ij}',\hat\phi_{ij}')$ if there exist an open neighbourhood $\dot V_{ij}$ of $\psi_i^{-1}(S)$ in $V_{ij}\cap V_{ij}'$ and a smooth morphism $\La:E_i\vert_{\dot V_{ij}}\ra \phi_{ij}^*(TV_j)\vert_{\dot V_{ij}}$ of vector bundles on $\dot V_{ij}$ satisfying
\e
\phi_{ij}'=\phi_{ij}+\La\cdot s_i+O(s_i^2)\;\>\text{and}\;\> \hat\phi_{ij}'=\hat\phi_{ij}+\La\cdot \phi_{ij}^*(\d s_j)+O(s_i)\;\> \text{on $\dot V_{ij}$,}
\label{ku2eq1}
\e
in the notation of Definition \ref{ku2def1}. Here $\d s_j$ is a shorthand for the derivative $\nabla s_j$ in $C^\iy(T^*V_j\ot E_j)$ with respect to any connection $\nabla$ on $E_j$. If $\ti\nabla$ is an alternative choice then $\ti\nabla s_j=\nabla s_j+O(s_j)$, so (d) implies that $\phi_{ij}^*(\nabla s_j)-\phi_{ij}^*(\ti\nabla s_j)=O(s_i)$. Thus $\phi_{ij}^*(\nabla s_j)+O(s_i)$ is independent of the choice of $\nabla$, so $\La\cdot\phi_{ij}^*(\nabla s_j)+O(s_i)$ is also independent of the choice of $\nabla$, and we write it as $\La\cdot\phi_{ij}^*(\d s_j)+O(s_i)$.

We will show that $\sim$ is an equivalence relation:
\begin{itemize}
\setlength{\itemsep}{0pt}
\setlength{\parsep}{0pt}
\item[(i)] To see that $(V_{ij},\phi_{ij},\hat\phi_{ij})\sim(V_{ij},\phi_{ij},\hat\phi_{ij})$, take $\dot V_{ij}=V_{ij}$ and $\La=0$.
\item[(ii)] Suppose $(V_{ij},\phi_{ij},\hat\phi_{ij})\sim(V_{ij}',\phi_{ij}',\hat\phi_{ij}')$, with $\dot V_{ij},\La$ as above. Then $\phi_{ij}'=\phi_{ij}+O(s_i)$, so the discussion after Definition \ref{ku2def1}(vi) shows that there exists $\ti\La:E_i\vert_{\dot V_{ij}}\ra\phi_{ij}^{\prime *}(TV_j)\vert_{\dot V_{ij}}$ with $\ti\La=-\La+O(s_i)$, and \eq{ku2eq1} gives
\begin{equation*}
\phi_{ij}=\phi_{ij}'+\ti\La\cdot s_i+O(s_i^2)\;\>\text{and}\;\> \hat\phi_{ij}=\hat\phi_{ij}'+\ti\La\cdot \phi_{ij}^{\prime *}(\d s_j)+O(s_i)\;\> \text{on $\dot V_{ij}$.}
\end{equation*}
Hence $(V_{ij}',\phi_{ij}',\hat\phi_{ij}')\sim(V_{ij},\phi_{ij},\hat\phi_{ij})$.
\item[(iii)] Suppose $(V_{ij},\phi_{ij},\hat\phi_{ij})\sim(V_{ij}',\phi_{ij}',\hat\phi_{ij}')$, with $\dot V_{ij},\La$ as above, and $(V_{ij}',\ab\phi_{ij}',\ab\hat\phi_{ij}')\ab\sim(V_{ij}'',\phi_{ij}'',\hat\phi_{ij}'')$ with $\ddot V_{ij},\La'$, so that
\e
\phi_{ij}''\!=\!\phi_{ij}'\!+\!\La'\cdot s_i\!+\!O(s_i^2)\;\>\text{and}\;\> \hat\phi_{ij}''\!=\!\hat\phi_{ij}'\!+\!\La'\cdot \phi_{ij}^{\prime *}(\d s_j)\!+\!O(s_i)\;\> \text{on $\ddot V_{ij}$.}
\label{ku2eq2}
\e
Set $\dddot V_{ij}=\dot V_{ij}\cap\ddot V_{ij}$. The discussion after Definition \ref{ku2def1}(vi) shows there exists $\ti\La'\in C^\iy\bigl((E_i^*\ot\phi_{ij}^*(TV_j))\vert_{\dddot V_{ij}}\bigr)$ with $\ti\La'=\La'\vert_{\dddot V_{ij}}+O(s_i)$. Define $\La''=\La\vert_{\dddot V_{ij}}+\ti\La'\vert_{\dddot V_{ij}}\in C^\iy\bigl((E_i^*\ot\phi_{ij}^*(TV_j))\vert_{\dddot V_{ij}}\bigr)$. Then \eq{ku2eq1}--\eq{ku2eq2} yield
\begin{equation*}
\phi_{ij}''\!=\!\phi_{ij}\!+\!\La''\cdot s_i\!+\!O(s_i^2)\;\>\text{and}\;\> \hat\phi_{ij}''\!=\!\hat\phi_{ij}\!+\!\La''\cdot \phi_{ij}^*(\d s_j)\!+\!O(s_i)\;\> \text{on $\dddot V_{ij}$,}
\end{equation*}
so $(V_{ij},\phi_{ij},\hat\phi_{ij})\sim(V_{ij}'',\phi_{ij}'',\hat\phi_{ij}'')$.
\end{itemize}
Thus $\sim$ is an equivalence relation. We may also write $\sim_S$ for $\sim$ if we want to make the choice of open set $S$ clear.

We write $[V_{ij},\phi_{ij},\hat\phi_{ij}]$ for the $\sim$-equivalence class of a triple $(V_{ij},\phi_{ij},\hat\phi_{ij})$, and we call $\Phi_{ij}=[V_{ij},\phi_{ij},\hat\phi_{ij}]:(V_i,E_i,s_i,\psi_i)\ra (V_j,E_j,s_j,\psi_j)$ a {\it morphism of\/ $\mu$-Kuranishi neighbourhoods over\/} $(S,f)$, or just a {\it morphism over\/} $(S,f)$. Write $\Hom_{S,f}\bigl((V_i,E_i,s_i,\psi_i),(V_j,E_j,s_j,\psi_j)\bigr)$ for the set of such~$\Phi_{ij}$.

If $Y=X$ and $f=\id_X$ then we call $\Phi_{ij}$ a {\it morphism of\/ $\mu$-Kuranishi neighbourhoods over\/} $S$, or just a {\it morphism over\/} $S$, and we write $\Hom_S\bigl((V_i,E_i,s_i,\psi_i),\ab(V_j,E_j,s_j,\psi_j)\bigr)$ for the set of such~$\Phi_{ij}$.
\label{ku2def3}
\end{dfn}

Next we define composition of morphisms, and identity morphisms.

\begin{dfn} Let $X,Y,Z$ be topological spaces, $f:X\ra Y$, $g:Y\ra Z$ be continuous maps, $(V_i,E_i,s_i,\psi_i),(V_j,E_j,s_j,\psi_j),(V_k,E_k,s_k,\psi_k)$ be $\mu$-Kuranishi neighbourhoods on $X,Y,Z$, and $T\subseteq \Im\psi_j\cap g^{-1}(\Im\psi_k)\ab\subseteq Y$ and $S\subseteq\Im\psi_i\cap f^{-1}(T)\subseteq X$ be open. Suppose $\Phi_{ij}=[V_{ij},\phi_{ij},\hat\phi_{ij}]:(V_i,E_i,s_i,\psi_i)\ra (V_j,\ab E_j,\ab s_j,\ab\psi_j)$ is a morphism of $\mu$-Kuranishi neighbourhoods over $(S,f)$, and $\Phi_{jk}=[V_{jk},\phi_{jk},\hat\phi_{jk}]:(V_j,E_j,s_j,\ab\psi_j)\ab\ra (V_k,E_k,s_k,\psi_k)$ is a morphism of $\mu$-Kuranishi neighbourhoods over $(T,g)$. Choose representatives $(V_{ij},\phi_{ij},\hat\phi_{ij}),(V_{jk},\phi_{jk},\hat\phi_{jk})$ in the $\sim$-equivalence classes $[V_{ij},\phi_{ij},\hat\phi_{ij}],[V_{jk},\phi_{jk},\hat\phi_{jk}]$. Define $V_{ik}=\phi_{ij}^{-1}(V_{jk})\ab\subseteq V_{ij}\subseteq V_i$, and $\phi_{ik}:V_{ij}\ra V_k$ by $\phi_{ik}=\phi_{jk}\ci\phi_{ij}\vert_{V_{ik}}$, and $\hat\phi_{ik}:E_i\vert_{V_{ik}}\ra\phi_{ik}^*(E_k)$ by $\hat\phi_{ik}=\phi_{ij}\vert_{V_{ik}}^*(\hat\phi_{jk})\ci\hat\phi_{ij}\vert_{V_{ik}}$. 

Using $\hat\phi_{ij}(s_i\vert_{V_{ij}})=\phi_{ij}^*(s_j)+O(s_i^2)$ and $\hat\phi_{jk}(s_j\vert_{V_{jk}})=\phi_{jk}^*(s_k)+O(s_j^2)$ from Definition \ref{ku2def3}(d) we see that $\hat\phi_{ik}(s_i\vert_{V_{ik}})=\phi_{ik}^*(s_k)+O(s_i^2)$. Also $\psi_i=\psi_j\ci\phi_{ij}$ on $s_i^{-1}(0)\cap V_{ij}$ and $\psi_j=\psi_k\ci\phi_{jk}$ on $s_j^{-1}(0)\cap V_{jk}$ from Definition \ref{ku2def3}(e) give $\psi_i=\psi_k\ci\phi_{ik}$ on $s_i^{-1}(0)\cap V_{ik}$. Thus $(V_{ik},\phi_{ik},\hat\phi_{ik})$ satisfies Definition \ref{ku2def3}(a)--(e), and $\Phi_{ik}:=[V_{ik},\phi_{ik},\hat\phi_{ik}]:(V_i,E_i,s_i,\psi_i)\ra (V_k,E_k,s_k,\psi_k)$ is a morphism of $\mu$-Kuranishi neighbourhoods over~$(S,g\ci f)$.

To see $\Phi_{ik}$ is independent of the choice of $(V_{ij},\phi_{ij},\hat\phi_{ij})$, $(V_{jk},\phi_{jk},\hat\phi_{jk})$, let $(V_{ij}',\phi_{ij}',\hat\phi_{ij}')$, $(V_{jk}',\phi_{jk}',\hat\phi_{jk}')$ be alternative choices, giving a triple $(V_{ik}',\phi_{ik}',\hat\phi_{ik}')$. As $(V_{ij}',\phi_{ij}',\hat\phi_{ij}')\sim (V_{ij},\phi_{ij},\hat\phi_{ij})$ and $(V_{jk}',\phi_{jk}',\hat\phi_{jk}')\sim (V_{jk},\phi_{jk},\hat\phi_{jk})$ we have open neighbourhoods $\dot V_{ij}$ of $\psi_i^{-1}(S)$ in $V_{ij}\cap V_{ij}'$ and $\dot V_{jk}$ of $\psi_j^{-1}(S)$ in $V_{jk}\cap V_{jk}'$ and morphisms $\La_{ij}:E_i\vert_{\dot V_{ij}}\ra \phi_{ij}^*(TV_j)\vert_{\dot V_{ij}}$, $\La_{jk}:E_j\vert_{\dot V_{jk}}\ra \phi_{jk}^*(TV_k)\vert_{\dot V_{jk}}$ satisfying
\e
\begin{aligned}
\phi_{ij}'&=\!\phi_{ij}\!+\!\La_{ij}\cdot s_i\!+\!O(s_i^2), & \hat\phi_{ij}'&=\!\hat\phi_{ij}\!+\!\La_{ij}\cdot \phi_{ij}^*(\d s_j)\!+\!O(s_i)\; \text{on $\dot V_{ij}$,}\\
\phi_{jk}'&=\!\phi_{jk}\!+\!\La_{jk}\cdot s_j\!+\!O(s_j^2), & \hat\phi_{jk}'&=\hat\phi_{jk}\!+\!\La_{jk}\cdot \phi_{jk}^*(\d s_k)\!+\!O(s_j)\; \text{on $\dot V_{jk}$.}
\end{aligned}
\label{ku2eq3}
\e

Set $\dot V_{ik}=\dot V_{ij}\cap\phi_{ij}^{-1}(\dot V_{jk})$, and define $\La_{ik}:E_i\vert_{\dot V_{ik}}\ra \phi_{ik}^*(TV_k)\vert_{\dot V_{ik}}$ by 
\e
\La_{ik}=\phi_{ij}^*(\d\phi_{jk})\ci\La_{ij}\vert_{\dot V_{ik}}+
\phi_{ij}^*(\La_{jk})\ci\hat\phi_{ij}\vert_{\dot V_{ik}}.
\label{ku2eq4}
\e
Then on $\dot V_{ik}$ we have
\e
\begin{split}
\phi_{ik}'&=\phi_{jk}'\ci\phi_{ij}'=\bigl[\phi_{jk}+\La_{jk}\cdot s_j+O(s_j^2)\bigr]\ci \bigl[\phi_{ij}+\La_{ij}\cdot s_i+O(s_i^2)\bigr]\\
&=\phi_{jk}\ci\phi_{ij}+\phi_{ij}^*(\d\phi_{jk})\ci\La_{ij}\cdot s_i+
\phi_{ij}^*(\La_{jk})\ci\phi_{ij}^*(s_j)\\
&\qquad +O(s_i^2)+O(\phi_{ij}^*(s_j^2))+O(s_i\cdot \phi_{ij}^*(s_j))\\
&=\phi_{ik}+\phi_{ij}^*(\d\phi_{jk})\ci\La_{ij}\cdot s_i+\phi_{ij}^*(\La_{jk})(\hat\phi_{ij}(s_i)+O(s_i^2))+O(s_i^2) \\
&=\phi_{ik}+\La_{ik}\cdot s_i+O(s_i^2),
\end{split}
\label{ku2eq5}
\e
using \eq{ku2eq3}--\eq{ku2eq4} and Definition \ref{ku2def3}(d). Similarly, on $\dot V_{ik}$ we have
\e
\begin{split}
\hat\phi_{ik}'&=\phi_{ij}'^*(\hat\phi_{jk}')\ci\hat\phi_{ij}'=\bigl(\phi_{ij}+O(s_i)\bigr){}^*\bigl[\hat\phi_{jk}+\La_{jk}\cdot \phi_{jk}^*(\d s_k)+O(s_j)\bigr]\\
&\quad\ci\bigl[\hat\phi_{ij}+\La_{ij}\cdot \phi_{ij}^*(\d s_j)+O(s_i)\bigr]\\
&=\phi_{ij}^*(\hat\phi_{jk})\ci\hat\phi_{ij}+\phi_{ij}^*(\hat\phi_{jk})\ci\La_{ij}\cdot \phi_{ij}^*(\d s_j)\\
&\quad+\phi_{ij}^*(\La_{jk})\cdot \phi_{ij}^*\ci\phi_{jk}^*(\d s_k)+O(s_i)+O(\phi_{ij}^*(s_j))\\
&=\hat\phi_{ik}+\La_{ij}\cdot\phi_{ij}^*(\hat\phi_{jk}\ci\d s_j)+\phi_{ij}^*(\La_{jk})\cdot \phi_{ik}^*(\d s_k)+O(s_i)\\
&=\hat\phi_{ik}\!+\!\La_{ij}\cdot\phi_{ij}^*\bigl(\phi_{jk}^*(\d s_k)\!\ci\!\d\phi_{jk}\!+\!O(s_j)\bigr)\!+\!\phi_{ij}^*(\La_{jk})\cdot \phi_{ik}^*(\d s_k)\!+\!O(s_i)\\
&=\hat\phi_{ik}+(\phi_{ij}^*(\d\phi_{jk})\ci\La_{ij})\cdot 
\phi_{ik}^*(\d s_k)+\phi_{ij}^*(\La_{jk})\cdot\phi_{ik}^*(\d s_k)+O(s_i)\\
&=\hat\phi_{ik}+\La_{ik}\cdot\phi_{ik}^*(\d s_k)+O(s_i),
\end{split}
\label{ku2eq6}
\e
using \eq{ku2eq3}--\eq{ku2eq4} and Definition \ref{ku2def3}(d) and its derivative. 

Equations \eq{ku2eq5}--\eq{ku2eq6} imply that $(V_{ik},\phi_{ik},\hat\phi_{ik})\sim(V_{ik}',\phi_{ik}',\hat\phi_{ik}')$, so the $\sim$-equivalence class $\Phi_{ik}=[V_{ik},\phi_{ik},\hat\phi_{ik}]$ is independent of choices, and depends only on $\Phi_{ij},\Phi_{jk}$. We write $[V_{jk},\phi_{jk},\hat\phi_{jk}]\ci [V_{ij},\phi_{ij},\hat\phi_{ij}]=[V_{ik},\phi_{ik},\hat\phi_{ik}]$, or $\Phi_{jk}\ci\Phi_{ij}=\Phi_{ik}$, and call it the {\it composition\/} of the morphisms $\Phi_{ij}=[V_{ij},\phi_{ij},\hat\phi_{ij}]$, $\Phi_{jk}=[V_{jk},\phi_{jk},\hat\phi_{jk}]$. It is easy to see that composition is associative.

If $(V_i,E_i,s_i,\psi_i)$ is a $\mu$-Kuranishi neighbourhood on $X$, and $S\subseteq\Im\psi_i$ is open, the {\it identity morphism\/} $\id_{(V_i,E_i,s_i,\psi_i)}$ is $[V_i,\id_{V_i},\id_{E_i}]$, considered as a morphism $(V_i,E_i,s_i,\psi_i)\ra(V_i,E_i,s_i,\psi_i)$ over $S$. Identity morphisms act as identities under composition of morphisms. 
\label{ku2def4}
\end{dfn}

We have now defined all the structures of a category: objects ($\mu$-Kuranishi neighbourhoods), morphisms, composition, and identity morphisms. However, our morphisms are defined over an open set $S\subseteq X$, which is not part of the usual category structure. There are two ways to make a genuine category: either (a) work on a fixed topological space $X$ and open $S\subseteq X$ with morphisms over $f=\id_X:X\ra X$ only; or (b) allow $X,f$ to vary but require $S=X$ throughout.

\begin{dfn} Let $X$ be a topological space and $S\subseteq X$ be open. Define the {\it category $\muKur_S(X)$ of $\mu$-Kuranishi neighbourhoods over\/} $S$ to have objects $\mu$-Kuranishi neighbourhoods $(V_i,E_i,s_i,\psi_i)$ on $X$ with $S\subseteq\Im\psi_i$, morphisms $\Phi_{ij}:(V_i,E_i,s_i,\psi_i)\ra (V_j,E_j,s_j,\psi_j)$ of $\mu$-Kuranishi neighbourhoods over $S$, and composition and identities as above.

Define the {\it category $\GmuKur$ of global\/ $\mu$-Kuranishi neighbourhoods\/} as follows. Objects $\bigl(X,(V,E,s,\psi)\bigr)$ in $\GmuKur$ are pairs of a topological space $X$ and a $\mu$-Kuranishi neighbourhood $(V,E,s,\psi)$ on $X$ with $\Im\psi=X$. Morphisms $(f,\Phi):\bigl(X,(V,E,s,\psi)\bigr)\ra\bigl(Y,(W,F,t,\chi)\bigr)$ in $\GmuKur$ are pairs of a continuous map $f:X\ra Y$ and a morphism $\Phi:(V,\ab E,\ab s,\ab\psi)\ab\ra(W,F,t,\chi)$ of $\mu$-Kuranishi neighbourhoods over $(X,f)$. Composition and identities are as above.
\label{ku2def5}
\end{dfn}

\begin{dfn} Let $\Phi_{ij}:(V_i,E_i,s_i,\psi_i)\ra(V_j,E_j,s_j,\psi_j)$ be a morphism of $\mu$-Kuranishi neighbourhoods on $X$ over $S$, where $S\subseteq\Im\psi_i\cap\Im\psi_j\subseteq X$ is open. We call $\Phi_{ij}$ a {\it $\mu$-coordinate change over\/} $S$ if it is invertible in $\muKur_S(X)$, that is, if there exists a morphism $\Phi_{ji}:(V_j,E_j,s_j,\psi_j)\ra (V_i,E_i,s_i,\psi_i)$ over $S$ satisfying $\Phi_{ji}\ci\Phi_{ij}=\id_{(V_i,E_i,s_i,\psi_i)}$ and $\Phi_{ij}\ci\Phi_{ji}=\id_{(V_j,E_j,s_j,\psi_j)}$.
\label{ku2def6}
\end{dfn}

\begin{dfn} Let $X,Y$ be topological spaces, $f:X\ra Y$ a continuous map, $(V_i,E_i,s_i,\psi_i)$, $(V_j,E_j,s_j,\psi_j)$ be $\mu$-Kuranishi neighbourhoods on $X,Y$ respectively, and $T\subseteq S\subseteq\Im\psi_i\cap f^{-1}(\Im\psi_j)\subseteq X$ be open. Then if $(V_{ij},\phi_{ij},\hat\phi_{ij})$ satisfies Definition \ref{ku2def3}(a)--(e) for $S$, it also satisfies them with $T$ in place of $S$, and if $(V_{ij},\phi_{ij},\hat\phi_{ij})\sim_S(V_{ij}',\phi_{ij}',\hat\phi_{ij}')$ then $(V_{ij},\phi_{ij},\hat\phi_{ij})\sim_T(V_{ij}',\phi_{ij}',\hat\phi_{ij}')$. Thus, we can define a map from morphisms of $\mu$-Kuranishi neighbourhoods over $(S,f)$ to morphisms of $\mu$-Kuranishi neighbourhoods over $(T,f)$, by mapping the $\sim_S$-equivalence class $\Phi_{ij}=[V_{ij},\phi_{ij},\hat\phi_{ij}]$ of $(V_{ij},\phi_{ij},\hat\phi_{ij})$, to the $\sim_T$-equivalence class of $(V_{ij},\phi_{ij},\hat\phi_{ij})$, which we write as $[V_{ij},\phi_{ij},\hat\phi_{ij}]\vert_T$ or $\Phi_{ij}\vert_T$, and call the {\it restriction of\/ $[V_{ij},\phi_{ij},\hat\phi_{ij}]$ to\/}~$T$.

Restriction is compatible with composition and identities, so for $f=\id_X$ it defines a {\it restriction functor\/} $\vert_T:\muKur_S(X)\ra\muKur_T(X)$ mapping $(V_i,\ab E_i,\ab s_i,\ab\psi_i)\ab\mapsto(V_i,E_i,s_i,\psi_i)$ on objects, and $\Phi_{ij}\mapsto\Phi_{ij}\vert_T$ on morphisms. If $\Phi_{ij}$ is a $\mu$-coordinate change on $S$, then $\Phi_{ij}\vert_T$ is a $\mu$-coordinate change on~$T$.
\label{ku2def7}
\end{dfn}

\begin{dfn} So far we have always discussed morphisms of $\mu$-Kuranishi neighbourhoods, and $\mu$-coordinate changes, {\it over a specified open set\/} $S\subseteq X$, or over $(S,f)$. We now make the convention that {\it when we do not specify a domain\/ $S$ for a morphism or $\mu$-coordinate change, the domain should be as large as possible}. For example, if we say that $\Phi_{ij}:(V_i,E_i,s_i,\psi_i)\ra (V_j,E_j,s_j,\psi_j)$ is a {\it morphism of $\mu$-Kuranishi neighbourhoods\/} ({\it over\/} $f$), {\it or a $\mu$-coordinate change}, without specifying $S$, we mean that $S=\Im\psi_i\cap\Im\psi_j$ (or~$S=\Im\psi_i\cap f^{-1}(\Im\psi_j)$).

Similarly, if we write a formula involving several morphisms or $\mu$-coordinate changes (possibly defined on different domains), without specifying the domain $S$, we make the convention {\it that the domain where the formula holds should be as large as possible}. That is, the domain $S$ is taken to be the intersection of the domains of each morphism in the formula, and we implicitly restrict each morphism in the formula to $S$ as in Definition \ref{ku2def7}, to make it make sense.

For example, if we say that $\Phi_{ij}:(V_i,E_i,s_i,\psi_i)\ra (V_j,E_j,s_j,\psi_j)$, $\Phi_{jk}:(V_j,E_j,s_j,\psi_j)\ra (V_k,E_k,s_k,\psi_k)$ and $\Phi_{ik}:(V_i,E_i,\ab s_i,\ab\psi_i)\ab\ra (V_k,E_k,s_k,\psi_k)$ are morphisms of $\mu$-Kuranishi neighbourhoods on $X$, and
\e
\Phi_{ik}=\Phi_{jk}\ci\Phi_{ij},
\label{ku2eq7}
\e 
we mean that $\Phi_{ij}$ is defined over $\Im\psi_i\cap\Im\psi_j$, and $\Phi_{jk}$ over $\Im\psi_j\cap\Im\psi_k$, and $\Phi_{ik}$ over $\Im\psi_i\cap\Im\psi_k$, and \eq{ku2eq7} holds over $\Im\psi_i\cap\Im\psi_j\cap\Im\psi_k$, that is, \eq{ku2eq7} is equivalent to
\begin{equation*}
\Phi_{ik}\vert_{\Im\psi_i\cap\Im\psi_j\cap\Im\psi_k}=
\Phi_{jk}\vert_{\Im\psi_i\cap\Im\psi_j\cap\Im\psi_k}\ci \Phi_{ij}\vert_{\Im\psi_i\cap\Im\psi_j\cap\Im\psi_k}.
\end{equation*}
Note in particular the potentially confusing point that \eq{ku2eq7} {\it does not determine\/ $\Phi_{ik}$ on\/ $\Im\psi_i\cap\Im\psi_k,$ but only on\/} $\Im\psi_i\cap\Im\psi_j\cap\Im\psi_k$.
\label{ku2def8}
\end{dfn}

\begin{rem}{\bf(i)} In our definition of $\mu$-Kuranishi spaces in \S\ref{ku23}, $\mu$-coordinate changes in Definition \ref{ku2def6} will be our substitute for coordinate changes between Kuranishi neighbourhoods in the Fukaya--Oh--Ohta--Ono theory of Kuranishi spaces \cite{Fuka,FOOO1,FOOO2,FOOO3,FOOO4,FOOO5,FOOO6,FOOO7,FOOO8,FuOn}, summarized in \S\ref{kuA1}. Morphisms over $f$ will be used in \S\ref{ku23} to define morphisms of Kuranishi spaces.
\smallskip

\noindent{\bf(ii)} In the definition of $\sim$ in Definition \ref{ku2def3}, restricting to an arbitrarily small neighbourhood $\dot V_{ij}$ of $\psi_i^{-1}(S)$ means that we are taking {\it germs of coordinate changes about\/ $\psi_i^{-1}(S)$ in\/} $V_i$. Fukaya--Ono's first definition of Kuranishi space \cite[\S 5]{FuOn} involved germs, though they were abandoned later. The difference is that Fukaya and Ono took germs of Kuranishi neighbourhoods at points, whereas we take germs of coordinate changes at larger subsets $\psi_i^{-1}(S)$ in~$V_i$.
\smallskip

\noindent{\bf(iii)} As those familiar with \cite{Fuka,FOOO1,FOOO2, FOOO3,FOOO4,FOOO5,FOOO6,FOOO7,FOOO8,FuOn,McWe1,McWe2,McWe3,Yang1,Yang2,Yang3} will know, in current theories of Kuranishi spaces, a lot of energy is spent worrying about the domains $V_{ij}$ of coordinate changes $\Phi_{ij}=(V_{ij},h_{ij},\phi_{ij},\hat\phi_{ij})$. Taking germs makes the domains $V_{ij}$ irrelevant in our theory, for better or for worse: in effect the domain of $[V_{ij},\phi_{ij},\hat\phi_{ij}]$ is~$\psi_i^{-1}(S)\subseteq V_i$.
\label{ku2rem1}
\end{rem}

\subsection{\texorpdfstring{Properties of morphisms and $\mu$-coordinate changes}{Properties of morphisms and \textmu-coordinate changes}}
\label{ku22}

We define sheaves on a topological space, following Hartshorne
\cite[\S II.1]{Hart}.

\begin{dfn} Let $X$ be a topological space. A {\it presheaf of sets\/} $\cE$ on $X$ consists of the data of a set $\cE(S)$ for every open set $S\subseteq X$, and a map $\rho_{ST}:\cE(S)\ra\cE(T)$ called the {\it restriction map\/} for every inclusion $T\subseteq S\subseteq X$ of open sets, satisfying the conditions that
\begin{itemize}
\setlength{\itemsep}{0pt}
\setlength{\parsep}{0pt}
\item[(i)] $\rho_{SS}=\id_{\cE(S)}:\cE(S)\ra\cE(S)$ for all open
$S\subseteq X$; and
\item[(ii)] $\rho_{SU}=\rho_{TU}\ci\rho_{ST}:\cE(S)\ra\cE(U)$ for all
open~$U\subseteq T\subseteq S\subseteq X$.
\end{itemize}

A presheaf of sets $\cE$ on $X$ is called a {\it sheaf\/} if it also satisfies
\begin{itemize}
\setlength{\itemsep}{0pt}
\setlength{\parsep}{0pt}
\item[(iii)] If $S\subseteq X$ is open, $\{T_i:i\in I\}$ is an open
cover of $S$, and $s,t\in\cE(S)$ have $\rho_{ST_i}(s)=\rho_{ST_i}(t)$ in
$\cE(T_i)$ for all $i\in I$, then $s=t$ in $\cE(S)$; and
\item[(iv)] If $S\subseteq X$ is open, $\{T_i:i\in I\}$ is an open cover of
$S$, and we are given elements $s_i\in\cE(T_i)$ for all $i\in I$
such that $\rho_{T_i(T_i\cap T_j)}(s_i)=\rho_{T_j(T_i\cap
T_j)}(s_j)$ in $\cE(T_i\cap T_j)$ for all $i,j\in I$, then there
exists $s\in\cE(S)$ with $\rho_{ST_i}(s)=s_i$ for all $i\in I$.
This $s$ is unique by~(iii).
\end{itemize}

Suppose $\cE,\cF$ are presheaves or sheaves of sets on $X$. A {\it morphism\/} $\phi:\cE\ra\cF$ consists of a map $\phi(S):\cE(S)\ra\cF(S)$ for all open $S\subseteq
X$, such that the following diagram commutes for all open $T\subseteq S\subseteq X$
\begin{equation*}
\xymatrix@C=95pt@R=15pt{
*+[r]{\cE(S)} \ar[r]_{\phi(S)} \ar[d]^{\rho_{ST}} & *+[l]{\cF(S)}
\ar[d]_{\rho_{ST}'} \\ *+[r]{\cE(T)} \ar[r]^{\phi(T)} & *+[l]{\cF(T),\!} }
\end{equation*}
where $\rho_{ST}$ is the restriction map for $\cE$, and $\rho_{ST}'$
the restriction map for~$\cF$.
\label{ku2def9}
\end{dfn}

Our first theorem, proved in \S\ref{ku61}, is an important property of morphisms of $\mu$-Kuranishi neighbourhoods on topological spaces. We will use it in \S\ref{ku23} to construct compositions of morphisms of $\mu$-Kuranishi spaces, and in \S\ref{ku24} to prove good behaviour of $\mu$-Kuranishi neighbourhoods on $\mu$-Kuranishi spaces. The first part of (a) is a special case of (b) for~$f=\id_X:X\ra X$.

\begin{thm}{\bf(a)} Let\/ $(V_i,E_i,s_i,\psi_i),$ $(V_j,E_j,s_j,\psi_j)$ be $\mu$-Kuranishi neighbourhoods on a topological space $X$. For each open $S\subseteq\Im\psi_i\cap\Im\psi_j,$ set
\begin{equation*}
\cHom\bigl((V_i,E_i,s_i,\psi_i),\kern -.1em(V_j,E_j,s_j,\psi_j)\bigr)(S)\!=\!\Hom_S\bigl((V_i,E_i,s_i,\psi_i),\kern -.1em(V_j,E_j,s_j,\psi_j)\bigr),
\end{equation*}
and for all open $T\subseteq S\subseteq \Im\psi_i\cap\Im\psi_j$ define
\begin{align*}
\rho_{ST}:\,&\cHom\bigl((V_i,E_i,s_i,\psi_i),(V_j,E_j,s_j,\psi_j)\bigr)(S)\longra\\
&\cHom\bigl((V_i,E_i,s_i,\psi_i),(V_j,E_j,s_j,\psi_j)\bigr)(T)
\quad\text{by}\quad
\rho_{ST}:\Phi_{ij}\longmapsto\Phi_{ij}\vert_T.
\end{align*}
Then $\cHom\bigl((V_i,E_i,s_i,\psi_i),(V_j,E_j,s_j,\psi_j)\bigr)$ is a sheaf of sets on $\Im\psi_i\cap\Im\psi_j$. 

Similarly, $\mu$-coordinate changes from $(V_i,E_i,s_i,\psi_i)$ to $(V_j,E_j,s_j,\psi_j)$ form a sheaf of sets\/ $\cIso\bigl((V_i,E_i,s_i,\psi_i),(V_j,E_j,s_j,\psi_j)\bigr)$ on $\Im\psi_i\cap\Im\psi_j,$ a subsheaf of\/~$\cHom\bigl((V_i,E_i,s_i,\psi_i),(V_j,E_j,s_j,\psi_j)\bigr)$.
\smallskip

\noindent{\bf(b)} Let\/ $f:X\!\ra\! Y$ be a continuous map of topological spaces, and\/ $(U_i,D_i,r_i,\chi_i),$ $(V_j,E_j,s_j,\psi_j)$ be $\mu$-Kuranishi neighbourhoods on $X,Y,$ respectively. Then morphisms from $(U_i,D_i,r_i,\chi_i)$ to $(V_j,E_j,s_j,\psi_j)$ over $f$ form a sheaf\/ $\cHom_f\bigl((U_i,\ab D_i,\ab r_i,\ab\chi_i),\ab (V_j,E_j,s_j,\psi_j)\bigr)$ over\/~$\Im\chi_i\cap f^{-1}(\Im\psi_j)\subseteq X$.
\label{ku2thm1}
\end{thm}

\begin{rem} We will refer to Theorem \ref{ku2thm1} as {\it the sheaf property of morphisms of\/ $\mu$-Kuranishi neighbourhoods}. Sheaves are basic objects in algebraic geometry. Informally, something forms a sheaf if it can be defined locally and glued. 

For example, if $X,Y$ are manifolds then smooth maps $f:X\ra Y$ form a sheaf on $X$, because an arbitrary map $f:X\ra Y$ is smooth if and only if it is a smooth on an open neighbourhood of each point in $X$. One uses the sheaf property of smooth maps all the time in differential geometry without noticing. 

It is clear that $\cHom\bigl((V_i,E_i,s_i,\psi_i),(V_j,E_j,s_j,\psi_j)\bigr)$, $\cHom_f\bigl((U_i,D_i,r_i,\chi_i),\ab (V_j,E_j,s_j,\psi_j)\bigr)$ are presheaves, the problem is to prove Definition~\ref{ku2def9}(iii),(iv).

\label{ku2rem2}
\end{rem}

The next theorem, proved in \S\ref{ku62}, characterizes $\mu$-coordinate changes:

\begin{thm} Let\/ $\Phi_{ij}=[V_{ij},\phi_{ij},\hat\phi_{ij}]:(V_i,E_i,s_i,\psi_i)\ra (V_j,E_j,s_j,\psi_j)$ be a morphism of $\mu$-Kuranishi neighbourhoods on a topological space $X$ over an open subset\/ $S\subseteq X$. Let\/ $x\in S,$ and set\/ $v_i=\psi_i^{-1}(x)\in V_i$ and\/ $v_j=\psi_j^{-1}(x)\in V_j$. Consider the sequence of finite-dimensional real vector spaces:
\e
\begin{gathered}
\smash{\xymatrix@C=18pt{ 0 \ar[r] & T_{v_i}V_i \ar[rrr]^(0.39){\d s_i\vert_{v_i}\op\d\phi_{ij}\vert_{v_i}} &&& E_i\vert_{v_i} \!\op\!T_{v_j}V_j 
\ar[rrr]^(0.56){-\hat\phi_{ij}\vert_{v_i}\op \d s_j\vert_{v_j}} &&& E_j\vert_{v_j} \ar[r] & 0. }}
\end{gathered}
\label{ku2eq8}
\e
Here $\d s_i\vert_{v_i}$ means $\nabla s_i\vert_{v_i}$ for any connection $\nabla$ on $E_i\ra V_i,$ and\/ $\d s_i\vert_{v_i}$ is independent of the choice of\/ $\nabla$ as $s_i(v_i)=0,$ and similarly for $\d s_j\vert_{v_j}$. Also $\d\phi_{ij}\vert_{v_i}:T_{v_i}V_i\ra T_{v_j}V_j$ and\/ $\hat\phi_{ij}\vert_{v_i}:E_i\vert_{v_i}\ra E_j\vert_{v_j}$ are independent of the choice of representative $(V_{ij},\phi_{ij},\hat\phi_{ij})$ for the $\sim$-equivalence class $[V_{ij},\phi_{ij},\hat\phi_{ij}]$. Thus \eq{ku2eq8} is well defined, and Definition\/ {\rm\ref{ku2def3}(d)} implies that\/ \eq{ku2eq8} is a complex.

Then $\Phi_{ij}$ is a $\mu$-coordinate change over $S$ in the sense of Definition\/ {\rm\ref{ku2def6}} if and only if\/ \eq{ku2eq8} is exact for all\/~$x\in S$.
\label{ku2thm2}
\end{thm}

The theorem is useful as it is often easier to check exactness of \eq{ku2eq8} for each $x\in S$ than to construct a two-sided inverse $\Phi_{ji}$ for $\Phi_{ij}$.

We now explain the relation of our $\mu$-coordinate changes to coordinate changes between Kuranishi neighbourhoods in the work of Fukaya--Oh--Ohta--Ono \cite{Fuka,FOOO1,FOOO2,FOOO3,FOOO4,FOOO5,FOOO6, FOOO7,FOOO8,FuOn} (often shortened to `FOOO'), as described in \S\ref{kuA1}. The next definition is the `$\mu$-' version of `FOOO coordinate changes' in Definition \ref{kuAdef2} below, taken from Fukaya et al.\ \cite[\S 4]{FOOO6}, with the finite groups $\Ga_i,\Ga_j$ omitted.

\begin{dfn} Let $(V_i,E_i,s_i,\psi_i),$ $(V_j,E_j,s_j,\psi_j)$ be $\mu$-Kuranishi neighbourhoods on a topological space $X$, and $S\subseteq\Im\psi_i\cap\Im\psi_j$ be open. Define a {\it FOOO $\mu$-coordinate change from\/ $(V_i,E_i,s_i,\psi_i)$ to\/ $(V_j,E_j,s_j,\psi_j)$ over\/} $S$ to be a triple $(V_{ij},\phi_{ij},\hat\phi_{ij})$ satisfying:
\begin{itemize}
\setlength{\itemsep}{0pt}
\setlength{\parsep}{0pt}
\item[(a)] $V_{ij}$ is an open neighbourhood of $\psi_i^{-1}(S)$ in $V_i$.
\item[(b)] $\phi_{ij}:V_{ij}\hookra V_j$ is a smooth embedding of manifolds.
\item[(c)] $\hat\phi_{ij}:E_i\vert_{V_{ij}}\hookra\phi_{ij}^*(E_j)$ is an embedding of vector bundles on~$V_{ij}$.
\item[(d)] $\hat\phi_{ij}(s_i\vert_{V_{ij}})=\phi_{ij}^*(s_j)$, in the sense of Definition~\ref{ku2def1}.
\item[(e)] $\psi_i=\psi_j\ci\phi_{ij}$ on $s_i^{-1}(0)\cap V_{ij}$.
\item[(f)] Let $x\in S,$ and set $v_i=\psi_i^{-1}(x)\in V_i$ and $v_j=\psi_j^{-1}(x)\in V_j$. Then we have a commutative diagram
\e
\begin{gathered}
\xymatrix@C=25pt@R=13pt{ 0 \ar[r] & T_{v_i}V_i \ar[rr]_{\d\phi_{ij}\vert_{v_i}} \ar[d]^{\d s_i\vert_{v_i}} && T_{v_j}V_j \ar[r] \ar[d]^{\d s_j\vert_{v_j}} & N_{ij}\vert_{v_i} \ar[r] \ar@{.>}[d]^{\d_{\rm fibre}s_j\vert_{v_i}} & 0 \\
0 \ar[r] & E_i\vert_{v_i} \ar[rr]^{\hat\phi_{ij}\vert_{v_i}} && E_j\vert_{v_j} \ar[r] & F_{ij}\vert_{v_i} \ar[r] & 0 
}
\end{gathered}
\label{ku2eq9}
\e
with exact rows, where $N_{ij}\ra V_{ij}$ is the normal bundle of $V_{ij}$ in $V_j$, and $F_{ij}=\phi_{ij}^*(E_j)/\hat\phi_{ij}(E_i\vert_{V_{ij}})$ the quotient bundle. We require that the induced $\d_{\rm fibre}s_j\vert_{v_i}$ in \eq{ku2eq9} should be an isomorphism for all~$x\in S$.
\end{itemize}
\label{ku2def10}
\end{dfn}

Note that FOOO $\mu$-coordinate changes are triples $(V_{ij},\phi_{ij},\hat\phi_{ij})$, not $\sim$-equivalence classes $[V_{ij},\phi_{ij},\hat\phi_{ij}]$ as in~\S\ref{ku21}.

\begin{lem} In Definition\/ {\rm\ref{ku2def10},} if\/ $(V_{ij},\phi_{ij},\hat\phi_{ij})$ is a FOOO $\mu$-coordinate change from $(V_i,E_i,s_i,\psi_i)$ to\/ $(V_j,E_j,s_j,\psi_j)$ over\/ $S,$ then the $\sim$-equivalence class $\Phi_{ij}=[V_{ij},\phi_{ij},\hat\phi_{ij}]$ from Definition\/ {\rm\ref{ku2def3}} is a $\mu$-coordinate change from $(V_i,\ab E_i,\ab s_i,\ab\psi_i)$ to\/ $(V_j,E_j,s_j,\psi_j)$ over\/ $S,$ in the sense of Definition\/~{\rm\ref{ku2def6}}.
\label{ku2lem1}
\end{lem}

\begin{proof} Definition \ref{ku2def10}(a)--(e) imply Definition \ref{ku2def3}(a)--(e), so the $\sim$-equivalence class $[V_{ij},\phi_{ij},\hat\phi_{ij}]$ in Definition \ref{ku2def3} is well-defined. That $\d_{\rm fibre}s_j\vert_{v_i}$ is an isomorphism in Definition \ref{ku2def10}(f) implies that \eq{ku2eq8} is exact for all $x\in S$, so Theorem \ref{ku2thm2} shows that $\Phi_{ij}=[V_{ij},\phi_{ij},\hat\phi_{ij}]$ is a $\mu$-coordinate change.
\end{proof}

Thus, FOOO $\mu$-coordinate changes yield examples of our $\mu$-coordinate cha\-nges. However, our $\mu$-coordinate changes are more general. For example, FOOO $\mu$-coordinate changes from $(V_i,E_i,s_i,\psi_i)$ to $(V_j,E_j,s_j,\psi_j)$ over $S\ne\es$ can exist only if $\dim V_i\le\dim V_j$, and so usually only in one direction. Our next example shows explicitly how FOOO $\mu$-coordinate changes can be invertible in our theory when $\dim V_i<\dim V_j$, although they are not invertible in~\cite{Fuka,FOOO1,FOOO2,FOOO3,FOOO4,FOOO5,FOOO6,FOOO7,FOOO8,FuOn}. 

\begin{ex} Let $V_1\subseteq\R^m$ be open, $f_1,\ldots,f_k:V_1\ra\R$ be smooth, and
\begin{equation*}
X=\bigl\{(x_1,\ldots,x_m)\in\R^m:f_1(x_1,\ldots,x_m)=\cdots=f_k(x_1,\ldots,x_m)=0\bigr\},
\end{equation*}
as a topological subspace of $\R^m$. Let $E_1\ra V_1$ be the trivial vector bundle $\R^k\t V_1\ra V_1$, and $s_1:V_1\ra E_1$ be the smooth section $s_1=(f_1,\ldots,f_k)$, so that $X=s_1^{-1}(0)\subseteq V_1$. Define $\psi_1=\id_X:s_1^{-1}(0)\ra X$. Then $(V_1,E_1,s_1,\psi_1)$ is a $\mu$-Kuranishi neighbourhood on $X$, with~$\Im\psi_1=X$.

Let $n$ be a positive integer, and define a second $\mu$-Kuranishi neighbourhood $(V_2,E_2,s_2,\psi_2)$ on $X$ by $V_2=V_1\t\R^n$, $E_2\ra V_2$ is the trivial vector bundle $\R^{k+n}\t V_2\ra V_2$, and $s_2:V_2\ra E_2$ is the smooth section
\begin{equation*}
s_2(x_1,\ldots,x_m,y_1,\ldots,y_n)=\bigl(f_1(x_1,\ldots,x_m),\ldots,f_k(x_1,\ldots,x_m),y_1,\ldots,y_n\bigr)
\end{equation*}
so that $s_2^{-1}(0)=X\t\{(0,\ldots,0)\}\subset V_1\t\R^n$. Define $\psi_2:s_2^{-1}(0)\ra X$ by $\psi_2:(x_1,\ldots,x_m,0,\ldots,0)=(x_1,\ldots,x_m)$, so that~$\Im\psi_2=X$.

Let $V_{12}$ be any open neighbourhood of $X$ in $V_1$. Define embeddings $\phi_{12}:V_{12}\hookra V_2$ by $\phi_{12}:(x_1,\ldots,x_m)\mapsto(x_1,\ldots,x_m,0,\ldots,0)$, and $\hat\phi_{12}:E_1\vert_{V_{12}}\hookra\phi_{12}^*(E_2)$ by $\hat\phi_{12}:(e_1,\ldots,e_k)\mapsto (e_1,\ldots,e_k,0,\ldots,0)$. It is now easy to check that $(V_{12},\phi_{12},\hat\phi_{12})$ is a FOOO $\mu$-coordinate change from $(V_1,E_1,s_1,\psi_1)$ to $(V_2,E_2,s_2,\psi_2)$ over $S=X$, as in Definition \ref{ku2def10}. Also $(V_{12},\phi_{12},\hat\phi_{12})$ satisfies Definition \ref{ku2def3}(a)--(e), so $\Phi_{12}=[V_{12},\phi_{12},\hat\phi_{12}]:(V_1,E_1,s_1,\psi_1)\ra (V_2,E_2,s_2,\psi_2)$ is a morphism of $\mu$-Kuranishi neighbourhoods on $X$ over $S=X$. Note that in Definition \ref{ku2def3}(d), $\hat\phi_{12}(s_1\vert_{V_{12}})=\phi_{12}^*(s_2)$ holds exactly, not just up to~$O(s_1^2)$.

Let $V_{21}$ be any open neighbourhood of $X\t\{(0,\ldots,0)\}$ in $V_2$. Define $\phi_{21}:V_{21}\ra V_1$ by $\phi_{21}:(x_1,\ldots,x_m,y_1,\ldots,y_n)\mapsto(x_1,\ldots,x_m)$, and $\hat\phi_{21}:E_2\vert_{V_{21}}\ra\phi_{21}^*(E_1)$ by $\hat\phi_{21}:(e_1,\ldots,e_k,f_1,\ldots,f_n)\mapsto(e_1,\ldots,e_k)$. Then $(V_{21},\phi_{21},\hat\phi_{21})$ is {\it not\/} a FOOO $\mu$-coordinate change, as $\phi_{21},\hat\phi_{21}$ are not embeddings. Nonetheless, $(V_{21},\phi_{21},\hat\phi_{21})$ satisfies Definition \ref{ku2def3}(a)--(e), where again part (d) holds exactly rather than up to $O(s_2^2)$, so $\Phi_{21}=[V_{21},\phi_{21},\hat\phi_{21}]:(V_2,E_2,s_2,\psi_2)\ra (V_1,E_1,s_1,\psi_1)$ is a morphism over $S=X$, as in~\S\ref{ku21}.

We claim that $\Phi_{12}$ and $\Phi_{21}$ are inverse, that is, that
\ea
[V_{21},\phi_{21},\hat\phi_{21}]\ci[V_{12},\phi_{12},\hat\phi_{12}]&=[V_1,\id_{V_1},\id_{E_1}],
\label{ku2eq10}\\
[V_{12},\phi_{12},\hat\phi_{12}]\ci[V_{21},\phi_{21},\hat\phi_{21}]&=[V_2,\id_{V_2},\id_{E_2}],
\label{ku2eq11}
\ea
hold in morphisms of $\mu$-Kuranishi neighbourhoods over $S=X$. By Definition \ref{ku2def4}, the left hand side of \eq{ku2eq10} is $[V_{11},\phi_{11},\hat\phi_{11}]$, where $V_{11}=\phi_{12}^{-1}(V_{21})\subseteq V_{12}\subseteq V_1$, and $\phi_{11}=\phi_{21}\ci\phi_{12}\vert_{V_{11}}=\id_{V_{11}}$, and $\hat\phi_{11}=\phi_{12}\vert_{V_{11}}^*(\hat\phi_{21})\ci\hat\phi_{12}\vert_{V_{11}}=\id_{E_1\vert_{V_{11}}}$. So using $\dot V_{11}=V_{11}\subseteq V_{11}\cap V_1$ and $\La=0$ in Definition \ref{ku2def2}, we see that $(V_{11},\phi_{11},\hat\phi_{11})\sim(V_1,\id_{V_1},\id_{E_1})$, proving~\eq{ku2eq10}.

Similarly, the left hand side of \eq{ku2eq11} is $[V_{22},\phi_{22},\hat\phi_{22}]$, where $V_{22}=\phi_{21}^{-1}(V_{12})\ab\subseteq V_{21}\subseteq V_2$, but now $\phi_{22},\hat\phi_{22}$ are not identities, rather we have
\begin{align*}
\phi_{22}&:(x_1,\ldots,x_m,y_1,\ldots,y_n)\longmapsto(x_1,\ldots,x_m,0,\ldots,0),\\
\hat\phi_{22}&:(e_1,\ldots,e_k,f_1,\ldots,f_n)\longmapsto(e_1,\ldots,e_k,0,\ldots,0).
\end{align*}

Define a vector bundle morphism $\La:E_2\vert_{V_{22}}\ra TV_2\vert_{V_{22}}=\id_{V_{22}}^*(TV_2)$ by
\e
\La:(e_1,\ldots,e_k,f_1,\ldots,f_n)\longmapsto \ts-f_1\frac{\pd}{\pd y_1}-\cdots-f_n\frac{\pd}{\pd y_n}.
\label{ku2eq12}
\e
Then we find that on $\dot V_{22}=V_{22}$ we have
\begin{equation*}
\phi_{22}=\id_{V_2}+\La\cdot s_2+O(s_2^2)\;\>\text{and}\;\> \hat\phi_{22}=\id_{E_2}+\La\cdot \id_{V_2}^*(\d s_2)+O(s_2),
\end{equation*}
so $(V_2,\id_{V_2},\id_{E_2})\sim(V_{22},\phi_{22},\hat\phi_{22})$ in the notation of Definition \ref{ku2def3}, giving $[V_{22},\phi_{22},\hat\phi_{22}]=[V_2,\id_{V_2},\id_{E_2}]$, which proves equation~\eq{ku2eq11}.
\label{ku2ex1}
\end{ex}

\begin{rem} One thing that makes the Kuranishi space theory of \cite{Fuka,FOOO1, FOOO2,FOOO3,FOOO4,FOOO5,FOOO6,FOOO7,FOOO8,FuOn} tricky is that coordinate changes go in only one direction. In our version, $\mu$-coordinate changes are invertible, which makes life easier in many ways.

In Example \ref{ku2ex1}, we invert the morphism $\Phi_{12}=[V_{12},\phi_{12},\hat\phi_{12}]$ associated to the FOOO $\mu$-coordinate change $(V_{12},\phi_{12},\hat\phi_{12})$ by using a nonzero $\La$ in \eq{ku2eq12}. This gives the first justification for including $\La$ in the definition of $\sim$ in Definition \ref{ku2def3}; without them, FOOO $\mu$-coordinate changes would not be invertible.

One can prove that every FOOO $\mu$-coordinate change $(V_{ij},\phi_{ij},\hat\phi_{ij})$ in Definition \ref{ku2def10} is modelled locally on $(V_{12},\phi_{12},\hat\phi_{12})$ in Example \ref{ku2ex1} near any point $x\in\psi_i^{-1}(S)$. So Example \ref{ku2ex1} shows that any morphism $\Phi_{ij}=[V_{ij},\phi_{ij},\hat\phi_{ij}]$ coming from a FOOO $\mu$-coordinate change $(V_{ij},\phi_{ij},\hat\phi_{ij})$ is locally invertible. Lemma \ref{ku2lem1} says that $\Phi_{ij}$ is globally invertible.
\label{ku2rem3}
\end{rem}

\subsection{\texorpdfstring{The definition of $\mu$-Kuranishi space}{The definition of \textmu-Kuranishi space}}
\label{ku23}

We can now define $\mu$-Kuranishi spaces (without boundary):

\begin{dfn} Let $X$ be a Hausdorff, second countable topological space (not necessarily compact), and $n\in\Z$. A {\it $\mu$-Kuranishi structure\/ $\cK$ on $X$ of virtual dimension\/} $n$ is data $\cK=\bigl(I,(V_i,E_i,s_i,\psi_i)_{i\in I},\Phi_{ij,\;i,j\in I}\bigr)$, where:
\begin{itemize}
\setlength{\itemsep}{0pt}
\setlength{\parsep}{0pt}
\item[(a)] $I$ is an indexing set (not necessarily finite).
\item[(b)] $(V_i,E_i,s_i,\psi_i)$ is a $\mu$-Kuranishi neighbourhood on $X$ for each $i\in I$, with $\dim V_i-\rank E_i=n$.
\item[(c)] $\Phi_{ij}=[V_{ij},\phi_{ij},\hat\phi_{ij}]:(V_i,E_i,s_i,\psi_i)\ra(V_j,E_j,s_j,\psi_j)$ is a $\mu$-coordinate change for all $i,j\in I$ (as in Definition \ref{ku2def8}, defined on $S=\Im\psi_i\cap\Im\psi_j$).
\item[(d)] $\bigcup_{i\in I}\Im\psi_i=X$. 
\item[(e)] $\Phi_{ii}=\id_{(V_i,E_i,s_i,\psi_i)}$ for all $i\in I$.
\item[(f)] $\Phi_{jk}\ci\Phi_{ij}=\Phi_{ik}$ for all $i,j,k\in I$ (as in Definition \ref{ku2def8}, this holds on $S=\Im\psi_i\cap\Im\psi_j\cap\Im\psi_k$).
\end{itemize}
We call $\bX=(X,\cK)$ a {\it $\mu$-Kuranishi space}, of {\it virtual dimension\/}~$\vdim\bX=n$.

When we write $x\in\bX$, we mean that $x\in X$.
\label{ku2def11}
\end{dfn}

Definition \ref{ku2def11} is what we will mean by $\mu$-Kuranishi space in the rest of the book. Here is an alternative definition, which we will not use, but which we include to facilitate comparison with the previous definitions of Kuranishi spaces and good coordinate systems described in Appendix~\ref{kuA}.

\begin{altdfn} Let $X$ be a Hausdorff, second countable topological space, and $n\in\Z$. A {\it $\mu$-Kuranishi structure\/ $\cK'$ on $X$ of virtual dimension\/} $n$ is data $\cK'=\bigl((I,\pr),(V_i,E_i,s_i,\psi_i)_{i\in I},\Phi_{ij,\; i\pr j\in I}\bigr)$, where:
\begin{itemize}
\setlength{\itemsep}{0pt}
\setlength{\parsep}{0pt}
\item[(a)] $I$ is an indexing set (not necessarily finite), and $\pr$ a partial order on $I$.
\item[(b)] $(V_i,E_i,s_i,\psi_i)$ is a $\mu$-Kuranishi neighbourhood on $X$ for each $i\in I$, with $\dim V_i-\rank E_i=n$.
\item[(c)] $\Phi_{ij}:(V_i,E_i,s_i,\psi_i)\ra(V_j,E_j,s_j,\psi_j)$ is a $\mu$-coordinate change for all $i,j\in I$ with $i\pr j$. 
\item[(d)] $\bigcup_{i\in I}\Im\psi_i=X$. 
\item[(e)] If $i\ne j\in I$ with $\Im\psi_i\cap\Im\psi_j\ne\es$ then $i\pr j$ or $j\pr i$.
\item[(f)] $\Phi_{jk}\ci\Phi_{ij}=\Phi_{ik}$ for all $i,j,k\!\in\! I$ with $i\!\pr\! j\!\pr\! k$.
\end{itemize}
We call $\bX'=(X,\cK')$ a {\it $\mu$-Kuranishi space}, of {\it virtual dimension\/}~$\vdim\bX'=n$.

\label{ku2def12}
\end{altdfn}

Definitions \ref{ku2def11} and \ref{ku2def12} are nearly equivalent. Here is how to relate them:
\begin{itemize}
\setlength{\itemsep}{0pt}
\setlength{\parsep}{0pt}
\item[(A)] If $\bX=(X,\cK)$ is as in Definition \ref{ku2def11}, choose a partial order $\pr$ on $I$, such that if $i\ne j\in I$ with $\Im\psi_i\cap\Im\psi_j\ne\es$ then $i\pr j$ or $j\pr\nobreak i$. For example, any total order $\pr$ on $I$ will do. Then $\cK'=\bigl((I,\pr\nobreak),\ab(V_i,E_i,s_i,\psi_i)_{i\in I},\ab\Phi_{ij,\; i\pr j\in I}\bigr)$ satisfies Definition \ref{ku2def12}. Note that $\cK'$ only remembers $\Phi_{ij}$ when $i\pr j$, which is less than half of the data $\Phi_{ij}$ for all $i,j\in I$ in~$\cK$.
\item[(B)] If $\bX'=(X,\cK')$ is as in Definition \ref{ku2def12}, define $\Phi_{ij}$ for all $i,j\in I$ as follows:
\begin{itemize}
\setlength{\itemsep}{0pt}
\setlength{\parsep}{0pt}
\item[(i)] If $i\pr j$ then $\Phi_{ij}$ is as in $\cK'$.
\item[(ii)] If $j\pr i$ then $\Phi_{ij}=\Phi_{ji}^{-1}$, for $\Phi_{ji}$ as in $\cK'$.
\item[(iii)] If $i\ne j$ with $i\not\pr j$ and $j\not\pr i$ then $\Im\psi_i\cap\Im\psi_j=\es$, and we set $\Phi_{ij}=[\es,\es,\es]$.
\item[(iv)] If $i=j$ then $\Phi_{ii}=\id_{(V_i,E_i,s_i,\psi_i)}$.
\end{itemize}
Then $\cK=\bigl(I,(V_i,E_i,s_i,\psi_i)_{i\in I},\Phi_{ij,\; i,j\in I}\bigr)$ satisfies Definition~\ref{ku2def11}.
\end{itemize}

Starting with $\bX=(X,\cK)$ and doing (A) and then (B) recovers the same $\bX=(X,\cK)$. Starting with $\bX'=(X,\cK')$ and doing (B) and then (A) recovers the same $\bX'=(X,\cK')$ if we choose the same partial order $\pr$. Thus, a Kuranishi space $\bX=(X,\cK')$ in the sense of Definition \ref{ku2def12} is equivalent to $\bX=(X,\cK)$ in Definition \ref{ku2def11} plus a suitable choice of partial order $\pr$ on~$I$.

Readers are advised to compare Definition \ref{ku2def12} with `FOOO weak good coordinate systems' in Definition \ref{kuAdef7}. Because in Definition \ref{ku2def12} the $\mu$-coordinate changes $\Phi_{ij}$ go in only one direction, it would be reasonable to suppose that $\Phi_{ij}=[V_{ij},\phi_{ij},\hat\phi_{ij}]$ is represented by a FOOO $\mu$-coordinate change $(V_{ij},\phi_{ij},\hat\phi_{ij})$ for all $i\pr j$ in $I$, and then the comparison with FOOO Kuranishi spaces and FOOO weak good coordinate systems in Appendix \ref{kuA} becomes closer.

\begin{ex} Let $V$ be a manifold, $E\ra V$ a vector bundle, and $s:V\ra E$ a smooth section. Set $X=s^{-1}(0)$, as a closed subset of $V$ with the induced topology. Then $X$ is Hausdorff and second countable, as $V$ is. Define a $\mu$-Kuranishi structure $\cK=\bigl(\{0\},(V_0,E_0,s_0, \psi_0),\Phi_{00}\bigr)$ on $X$ with indexing set $I=\{0\}$, one $\mu$-Kuranishi neighbourhood $(V_0,E_0,s_0,\psi_0)$ with $V_0=V$, $E_0=E$, $s_0=s$ and $\psi_0=\id_X$, and one $\mu$-coordinate change $\Phi_{00}=\id_{(V_0,E_0,s_0,\psi_0)}$. Then $\bX=(X,\cK)$ is a $\mu$-Kuranishi space, with $\vdim\bX=\dim V-\rank E$. We write~$\bS_{V,E,s}=\bX$.
\label{ku2ex2}
\end{ex}

When we are discussing several $\mu$-Kuranishi spaces at once, we need good notation to distinguish $\mu$-Kuranishi neighbourhoods and $\mu$-coordinate changes on the different spaces. One choice we will use often for $\mu$-Kuranishi spaces $\bW,\bX,\bY,\bZ$, based on consecutive letters in the Latin and Greek alphabets, is
\ea
{}\!\!\!\bW&\!=\!(W,\cH),\; \cH\!=\!\bigl(H,\kern -.2em(T_h,\kern -.2em C_h,\kern -.2em q_h,\kern -.2em\vp_h)_{h\in H},\Si_{hh'}\!=\![T_{hh'},\kern .02em\si_{hh'},\kern -.2em\hat\si_{hh'}]_{h,h'\in H}\bigr),\!\!{}
\label{ku2eq13}\\
{}\!\!\!\bX&\!=\!(X,\cI),\; \cI\!=\!\bigl(I,(U_i,D_i,r_i,\chi_i)_{i\in I},\Tau_{ii'}\!=\![U_{ii'},\tau_{ii'},\hat\tau_{ii'}]_{i,i'\in I}\bigr),\!\!{}
\label{ku2eq14}\\
{}\!\!\!\bY&\!=\!(Y,\cJ),\; \cJ\!=\!\bigl(J,(V_j,E_j,s_j,\psi_j)_{j\in J},\Up_{jj'}\!=\![V_{jj'},\up_{jj'},\hat\up_{jj'}]_{j,j'\in J}\bigr),\!\!{}
\label{ku2eq15}\\
{}\!\!\!\bZ&\!=\!(Z,\cK),\; \cK\!=\!\bigl(K,(W_k,F_k,t_k,\om_k)_{k\in K},\!\Phi_{kk'}\!=\![W_{kk'},\phi_{kk'},\hat\phi_{kk'}]_{k,k'\in K}\bigr).\!\!{}
\label{ku2eq16}
\ea

\begin{dfn} Let $\bX=(X,\cI)$ and $\bY=(Y,\cJ)$ be $\mu$-Kuranishi spaces, with notation \eq{ku2eq14}--\eq{ku2eq15}. A {\it morphism\/} $\bs f:\bX\ra\bY$ is $\bs f=\bigl(f,\bs f_{ij,\;i\in I,\; j\in J}\bigr)$, where $f:X\ra Y$ is a continuous map, and $\bs f_{ij}=[U_{ij},f_{ij},\hat f_{ij}]:(U_i,D_i,r_i,\chi_i)\ra (V_j,E_j,s_j,\psi_j)$ is a morphism of $\mu$-Kuranishi neighbourhoods over $f$ for all $i\in I$, $j\in J$ (defined over $S=\Im\chi_i\cap f^{-1}(\Im\psi_j)$, by the convention in Definition \ref{ku2def8}), satisfying the conditions:
\begin{itemize}
\setlength{\itemsep}{0pt}
\setlength{\parsep}{0pt}
\item[(a)] If $i,i'\in I$ and $j\in J$ then in morphisms over $f$ we have
\e
\bs f_{i'j}\ci\Tau_{ii'}=\bs f_{ij},
\label{ku2eq17}
\e
where by our usual convention, \eq{ku2eq17} holds over $S=\Im\chi_i\cap \Im\chi_{i'}\cap f^{-1}(\Im\psi_j)$, and each term in \eq{ku2eq17} is implicitly restricted to $S$. In particular, \eq{ku2eq17} {\it does not determine\/} $\bs f_{ij}$, but only its restriction~$\bs f_{ij}\vert_S$.
\item[(b)] If $i\in I$ and $j,j'\in J$ then interpreted as for \eq{ku2eq17}, we have
\e
\Up_{jj'}\ci \bs f_{ij}=\bs f_{ij'}.
\label{ku2eq18}
\e
\end{itemize}

If $x\in\bX$ (i.e. $x\in X$), we will write $\bs f(x)=f(x)\in\bY$.

When $\bY=\bX$, so that $J=I$, define $\bs\id_\bX=\bigl(\id_X,\Tau_{ij,\;i,j\in I}\bigr)$. Then Definition \ref{ku2def11}(f) implies that (a),(b) hold, so $\bs\id_\bX:\bX\ra\bX$ is a morphism of $\mu$-Kuranishi spaces, which we call the {\it identity morphism}.
\label{ku2def13}
\end{dfn}

In the next theorem, we use the sheaf property of morphisms of $\mu$-Kuranishi neighbourhoods in Theorem \ref{ku2thm1} to construct compositions $\bs g\ci\bs f:\bX\ra\bZ$ of morphisms of $\mu$-Kuranishi spaces $\bs f:\bX\ra\bY$, $\bs g:\bY\ra\bZ$, and hence show that $\mu$-Kuranishi spaces form a category~$\muKur$.

\begin{thm}{\bf(a)} Let\/ $\bX=(X,\cI),\bY=(Y,\cJ),\bZ=(Z,\cK)$ be $\mu$-Kuranishi spaces with notation {\rm\eq{ku2eq14}--\eq{ku2eq16},} and\/ $\bs f:\bX\ra\bY,$ $\bs g:\bY\ra\bZ$ be morphisms, where $\bs f=\bigl(f,\bs f_{ij,\; i\in I,\; j\in J}\bigr),$ $\bs g=\bigl(g,\bs g_{jk,\; j\in J,\; k\in K}\bigr)$. Then there exists a unique morphism $\bs h:\bX\ra\bZ,$ where $\bs h=\bigl(h,\bs h_{ik,\; i\in I,\; k\in K}\bigr)$ such that\/ $h=g\ci f:X\ra Z,$ and for all\/ $i\in I,$ $j\in J,$ $k\in K$ we have
\e
\bs h_{ik}=\bs g_{jk}\ci\bs f_{ij},
\label{ku2eq19}
\e
where by our usual convention \eq{ku2eq19} holds over\/ $\Im\chi_i\cap f^{-1}(\Im\psi_j)\cap h^{-1}(\Im\om_k),$ and so may not determine $\bs h_{ik}$ over\/~$\Im\chi_i\cap h^{-1}(\Im\om_k)$.

We write $\bs g\ci\bs f=\bs h,$ so that\/ $\bs g\ci\bs f:\bX\ra\bZ$ is a morphism of $\mu$-Kuranishi spaces, and call\/ $\bs g\ci\bs f$ the \begin{bfseries}composition\end{bfseries} of\/~$\bs f,\bs g$.
\smallskip

\noindent{\bf(b)} Composition of morphisms is associative, that is, if\/ $\bs e:\bW\ra\bZ$ is another morphism of $\mu$-Kuranishi spaces then $(\bs g\ci\bs f)\ci\bs e=\bs g\ci(\bs f\ci\bs e)$.
\smallskip

\noindent{\bf(c)} Composition is compatible with identities, that is, $\bs f\ci\bs\id_\bX=\bs\id_\bY\ci\bs f=\bs f$ for all morphisms of $\mu$-Kuranishi spaces\/~$\bs f:\bX\ra\bY$.
\smallskip

Thus $\mu$-Kuranishi spaces form a category, which we write as $\muKur$.
\label{ku2thm3}
\end{thm}

\begin{proof} For (a), define $h=g\ci f:X\ra Z$. Let $i\in I$ and $k\in K$, and set 
$S=\Im\chi_i\cap h^{-1}(\Im\om_k)$, so that $S$ is open in $X$. We want to define a morphism $\bs h_{ik}:(U_i,D_i,r_i,\chi_i)\ra (W_k,F_k,t_k,\om_k)$ of $\mu$-Kuranishi neighbourhoods over $(S,h)$. Equation \eq{ku2eq19} means that for each $j\in J$ we must have 
\e
\bs h_{ik}\vert_{S\cap f^{-1}(\Im\psi_j)}
=\bs g_{jk}\ci\bs f_{ij}\vert_{S\cap f^{-1}(\Im\psi_j)}.
\label{ku2eq20}
\e
Since $\{\Im\psi_j:j\in J\}$ is an open cover of $Y$ and $f$ is continuous, $\bigl\{S\cap f^{-1}(\Im\psi_j):j\in J\bigr\}$ is an open cover of $S$. For all $j,j'\in J$ we have
\ea
&\bs g_{jk}\ci\bs f_{ij}\vert_{S\cap f^{-1}(\Im\psi_j)\cap f^{-1}(\Im\psi_{j'})}
=\bs g_{j'k}\ci\Up_{jj'}\ci\bs f_{ij}\vert_{S\cap f^{-1}(\Im\psi_j)\cap f^{-1}(\Im\psi_{j'})}
\nonumber\\
&=\bs g_{j'k}\ci\bs f_{ij'}\vert_{S\cap f^{-1}(\Im\psi_j)\cap f^{-1}(\Im\psi_{j'})},
\label{ku2eq21}
\ea
using \eq{ku2eq17} for $\bs g$ in the first step, and \eq{ku2eq18} for $\bs f$ in the second.

Now the right hand side of \eq{ku2eq20} prescribes values for a morphism over $h$ on the sets of an open cover $\bigl\{S\cap f^{-1}(\Im\psi_j):j\in J\bigr\}$ of $S$. Equation \eq{ku2eq21} shows that these values agree on overlaps $\bigl(S\cap f^{-1}(\Im\psi_j)\bigr)\cap\bigl(S\cap f^{-1}(\Im\psi_{j'})\bigr)$. Therefore the sheaf property Theorem \ref{ku2thm1}(b) shows that there is a unique morphism $\bs h_{ik}$ over $(S,h)$ satisfying \eq{ku2eq20} for all~$j\in J$.

We have now defined $\bs h=\bigl(h,\bs h_{ik,\; i\in I,\; k\in K}\bigr)$. To show $\bs h:\bX\ra\bZ$ is a morphism, we must verify Definition \ref{ku2def13}(a),(b). For (a), suppose $i,i'\in I$, $j\in J$ and $k\in K$. Then we have
\begin{align*}
\bs h_{i'k}&\ci\Tau_{ii'}\vert_{\Im\chi_i\cap\Im\chi_{i'}\cap f^{-1}(\Im\psi_j)\cap h^{-1}(\Im\om_k)}=\bs g_{jk}\ci\bs f_{i'j}\ci\Tau_{ii'}\vert_{\cdots}\\
&=\bs g_{jk}\ci\bs f_{ij}\vert_{\cdots}=\bs h_{ij}\vert_{\Im\chi_i\cap\Im\chi_{i'}\cap f^{-1}(\Im\psi_j)\cap h^{-1}(\Im\om_k)},
\end{align*}
using \eq{ku2eq20} with $i'$ in place of $i$ in the first step, \eq{ku2eq17} for $\bs f$ in the second, and \eq{ku2eq20} in the third. This proves the restriction of \eq{ku2eq17} for $\bs h,i,i',k$ to $\Im\chi_i\cap\Im\chi_{i'}\cap f^{-1}(\Im\psi_j)\cap h^{-1}(\Im\om_k)$, for each~$j\in J$.

Since the $\Im\chi_i\cap\Im\chi_{i'}\cap f^{-1}(\Im\psi_j)\cap h^{-1}(\Im\om_k)$ for $j\in J$ form an open cover of $\Im\chi_i\cap\Im\chi_{i'}\cap h^{-1}(\Im\om_k)$, the sheaf property Theorem \ref{ku2thm1}(b) implies that \eq{ku2eq17} holds for $\bs h,i,i',k$ on the correct domain $\Im\chi_i\cap\Im\chi_{i'}\cap h^{-1}(\Im\om_k)$, yielding Definition \ref{ku2def13}(a) for $\bs h$. Definition \ref{ku2def13}(b) follows by a similar argument, involving \eq{ku2eq18} for $\bs g$. Hence $\bs h:\bX\ra\bZ$ is a morphism, proving part~(a).

For (b), in notation \eq{ku2eq13}--\eq{ku2eq16}, if $h\in H$, $i\in I$, $j\in J$, $k\in K$ we find that
\begin{align*}
[(\bs g&\ci\bs f)\ci\bs e]_{h,k}\vert_{\Im\vp_h\cap e^{-1}(\Im\chi_i)\cap (f\ci e)^{-1}(\Im\psi_j)\cap (g\ci f\ci e)^{-1}(\Im\om_k)}=\bs g_{jk}\ci\bs f_{ij}\ci\bs e_{hi}\\
&=[\bs g\ci(\bs f\ci\bs e)]_{h,k}\vert_{\Im\vp_h\cap e^{-1}(\Im\chi_i)\cap (f\ci e)^{-1}(\Im\psi_j)\cap (g\ci f\ci e)^{-1}(\Im\om_k)},
\end{align*}
where the middle step makes sense without brackets by associativity of composition of morphisms of $\mu$-Kuranishi neighbourhoods over maps. Since $\Im\vp_h\cap e^{-1}(\Im\chi_i)\cap (f\ci e)^{-1}(\Im\psi_j)\cap (g\ci f\ci e)^{-1}(\Im\om_k)$ for all $i\in I$ and $j\in J$ form an open cover of $\Im\vp_h\cap(g\ci f\ci e)^{-1}(\Im\om_k)$, the sheaf property Theorem \ref{ku2thm1}(b) implies that $\bigl[(\bs g\ci\bs f)\ci\bs e\bigr]{}_{h,k}=\bigl[\bs g\ci(\bs f\ci\bs e)\bigr]{}_{h,k}$ over the correct domain $\Im\vp_h\cap(g\ci f\ci e)^{-1}(\Im\om_k)$, so that $(\bs g\ci\bs f)\ci\bs e=\bs g\ci(\bs f\ci\bs e)$, proving (b). 

For (c), let $i,i'\in I$ and $j\in J$. Then we have
\begin{equation*}
[\bs f\ci\bs\id_\bX]_{i,j}\vert_{\Im\chi_i\cap\Im\chi_{i'}\cap f^{-1}(\Im\psi_j)}=\bs f_{i'j}\ci\Tau_{ii'}=\bs f_{ij}\vert_{\Im\chi_i\cap\Im\chi_{i'}\cap f^{-1}(\Im\psi_j)},
\end{equation*}
using \eq{ku2eq19} and the definition of $\bs\id_\bX$ in the first step, and \eq{ku2eq17} in the second. Since $\Im\chi_i\cap\Im\chi_{i'}\cap f^{-1}(\Im\psi_j)$ for all $i'\in I$ form an open cover of $\Im\chi_i\cap f^{-1}(\Im\psi_j)$, the sheaf property implies that $[\bs f\ci\bs\id_\bX]_{i,j}=\bs f_{ij}$ on the correct domain $\Im\chi_i\cap f^{-1}(\Im\psi_j)$, so $\bs f\ci\bs\id_\bX=\bs f$. We prove $\bs\id_\bY\ci\bs f=\bs f$ in the same way. This completes the proof.
\end{proof}

We can regard manifolds as examples of $\mu$-Kuranishi spaces:

\begin{dfn} We will define a functor $F_\Man^\muKur:\Man\ra\muKur$. On objects, if $X$ is a manifold, define a $\mu$-Kuranishi space $F_\Man^\muKur(X)=\bX=(X,\cK)$ with topological space $X$ and $\mu$-Kuranishi structure $\cK=\bigl(\{0\},(V_0,E_0,s_0,\psi_0),\ab\Phi_{00}\bigr)$, with indexing set $I=\{0\}$, one $\mu$-Kuranishi neighbourhood $(V_0,E_0,s_0,\psi_0)$ with $V_0=X$, $E_0\ra V_0$ the zero vector bundle, $s_0=0$, and $\psi_0=\id_X$, and one $\mu$-coordinate change $\Phi_{00}=\id_{(V_0,E_0,s_0,\psi_0)}$. 

On morphisms, if $f:X\ra Y$ is a smooth map of manifolds and $\bX=F_\Man^\muKur(X)$, $\bY=F_\Man^\muKur(Y)$, define a morphism $F_\Man^\muKur(f)=\bs f:\bX\ra\bY$ by $\bs f=(f,\bs f_{00})$, where $\bs f_{00}=[V_{00},f_{00},\hat f_{00}]$ with $V_{00}=X$, $f_{00}=f$, and $\hat f_{00}$ is the zero map on zero vector bundles.

It is now easy to check that $F_\Man^\muKur$ is a functor, which is full and faithful (surjective and injective on morphisms), and thus embeds $\Man$ as a full subcategory of $\muKur$. So we can identify $\Man$ with its image in $\muKur$, and regard manifolds as special examples of $\mu$-Kuranishi spaces, and $\mu$-Kuranishi spaces as generalizations of manifolds. We will usually write manifolds as $X,Y,Z,\ldots,$ rather than in bold $\bX,\bY,\bZ,\ldots,$ even when we regard them as $\mu$-Kuranishi spaces, so we write $F_\Man^\muKur(X)$ as $X$ rather than~$\bX$.

We say that a $\mu$-Kuranishi space $\bX$ {\it is a manifold\/} if $\bX\cong F_\Man^\muKur(X')$ in $\muKur$, for some manifold~$X'$.
\label{ku2def14}
\end{dfn}

\begin{ex} Let $\bX=(X,\cK)$ be a $\mu$-Kuranishi space, with $\cK=\bigl(I,(V_i,\ab E_i,\ab s_i,\ab\psi_i)_{i\in I},\Phi_{ij,\; i,j\in I}\bigr)$, and $Y$ a manifold, regarded as a $\mu$-Kuranishi space. We will work out the definition of morphism $\bs f:\bX\ra Y$ explicitly.

Since $Y$ has only one $\mu$-Kuranishi neighbourhood $(Y,0,0,\id_Y)$, a morphism $\bs f:\bX\ra Y$ is $\bs f=(f,\bs f_{i0,\; i\in I})$, where $f:X\ra Y$ is continuous, and $\bs f_{i0}=[V_{i0},f_{i0},0]$ for $i\in I$ is a $\sim$-equivalence class of triples $(V_{i0},f_{i0},0)$ with $V_{i0}$ an open neighbourhood of $s_i^{-1}(0)$ in $V_i$ and $f_{i0}:V_{i0}\ra Y$ a smooth map such that $f\ci\psi_i=f_{i0}\vert_{s_i^{-1}(0)}$, and $0:E_i\vert_{V_{i0}}\ra 0$ is the map to the zero vector bundle.

In this special case, the definition of $\sim$ in Definitions \ref{ku2def3} simplifies: the second equation of \eq{ku2eq1} is trivial as both sides are zero, and in the first equation, the $\La\cdot s_i$ term exactly cancels the $O(s_i)$ part of $f_{i0}'-f_{i0}$. Thus, two such triples $(V_{i0},f_{i0},0)$, $(V_{i0}',f_{i0}',0)$ have $(V_{i0},f_{i0},0)\sim(V_{i0}',f_{i0}',0)$ if there exists an open neighbourhood $\dot V_{i0}$ of $s_i^{-1}(0)$ in $V_{i0}\cap V_{i0}'$ with 
\begin{equation*}
f_{i0}'\vert_{\dot V_{i0}}=f_{i0}\vert_{\dot V_{i0}}+O(s_i).
\end{equation*}

If $i,j\in I$, the compatibility condition on $\bs f_{i0},\bs f_{j0}$ is
\begin{equation*}
f_{i0}\vert_{V_{i0}\cap \phi_{ij}^{-1}(V_{j0})}=f_{j0}\ci\phi_{ij}\vert_{V_{i0}\cap \phi_{ij}^{-1}(V_{j0})}+O(s_i).
\end{equation*}

\label{ku2ex3}
\end{ex}

\begin{rem} Although they do not define morphisms between Kuranishi spaces, Fukaya, Oh, Ohta and Ono \cite[Def.~6.6]{FuOn}, \cite[Def.~A1.13]{FOOO1}, \cite[Def.~4.6]{FOOO6} do define ({\it strongly\/}) {\it smooth maps\/} $\bs f:\bX\ra Y$ from a Kuranishi space $\bX$ to a manifold $Y$. Their definition is like Example \ref{ku2ex3}, apart from the $O(s_i)$ terms.
\label{ku2rem4}
\end{rem}

\begin{ex} Let $\bX,\bY$ be $\mu$-Kuranishi spaces, with notation \eq{ku2eq14}--\eq{ku2eq15}. We will define a {\it product $\mu$-Kuranishi space\/} $\bX\t\bY=(X\t Y,\cK)$, with
\begin{equation*}
\cK=\bigl(I\t J,(W_{(i,j)},F_{(i,j)},t_{(i,j)},\om_{(i,j)})_{(i,j)\in I\t J},\;\Phi_{(i,j)(i',j'),\; (i,j),(i',j')\in I\t J}\bigr),
\end{equation*}
where for all $(i,j)\in I\t J$ we have $W_{(i,j)}=U_i\t V_j$, $F_{(i,j)}=\pi_{U_i}^*(D_i)\op \pi_{V_j}^*(E_j)$, and $t_{(i,j)}=\pi_{U_i}^*(r_i)\op \pi_{V_j}^*(s_j)$ so that $t_{(i,j)}^{-1}(0)=r_i^{-1}(0)\t s_j^{-1}(0)$, and $\om_{(i,j)}=\chi_i\t\psi_j:r_i^{-1}(0)\t s_j^{-1}(0)\ra X\t Y$. Also, for all $(i,j),(i',j')\in I\t J$, if $(U_{ii'},\tau_{ii'},\hat\tau_{ii'})$ represents $\Tau_{ii'}$ and $(V_{jj'},\up_{jj'},\hat\up_{jj'})$ represents $\Up_{jj'}$ in the $\mu$-Kuranishi structures on $\bX,\bY$, we have $\Phi_{(i,j)(i',j')}=\bigl[U_{ii'}\t V_{jj'},\tau_{ii'}\t\up_{jj'},\pi_{U_{ii'}}^*(\hat\tau_{ii'})\op\pi_{V_{jj'}}^*(\hat\up_{jj'})\bigr]$.

One can check that $\cK$ is a $\mu$-Kuranishi structure on $X\t Y$ with dimension $\vdim(\bX\t\bY)=\vdim\bX+\vdim\bY$, independent of choices. There are natural projection morphisms $\bs\pi_\bX:\bX\t\bY\ra\bX$ and $\bs\pi_\bY:\bX\t\bY\ra\bY$, with
\begin{equation*}
\bs\pi_\bX=\bigl(\pi_X,(\bs\pi_\bX)_{(i,j)i',\; (i,j)\in I\t J,\; i'\in I}\bigr),
\end{equation*}
where $(\bs\pi_\bX)_{(i,j)i'}=\bigl[U_{ii'}\t V_j,\tau_{ii'}\ci\pi_{U_{ii'}},\pi_{U_{ii'}}^*(\hat\tau_{ii'})\op 0\bigr]$ in the notation above, and $\bs\pi_\bY$ is similar. Then $\bX\t\bY,\bs\pi_\bX,\bs\pi_\bY$ are a product in $\muKur$ in the sense of category theory. That is, given any morphisms $\bs f:\bW\ra\bX$, $\bs g:\bW\ra\bY$ in $\muKur$, there is a unique morphism $\bs h:\bW\ra\bX\t\bY$ with $\bs\pi_\bX\ci\bs h=\bs f$ and $\bs\pi_\bY\ci\bs h=\bs g$. We write $\bs h=(\bs f,\bs g)$, and call it the {\it direct product}.

Products of $\mu$-Kuranishi spaces are associative and commutative, up to canonical isomorphism, so we generally omit brackets in products like~$\bX\t\bY\t\bZ$.

\label{ku2ex4}
\end{ex}

\begin{ex} Now that we have a category $\muKur$ of $\mu$-Kuranishi spaces including manifolds as a subcategory, we can define the {\it action of a Lie group $G$ on a $\mu$-Kuranishi space\/} $\bX$ in an essentially trivial way: such an action is a morphism of $\mu$-Kuranishi spaces $\bs\mu:G\t\bX\ra\bX$ satisfying the usual relations
\begin{align*}
\bs\mu\ci(\bs\pi_1\t(\bs\mu\ci\bs\pi_{23}))&=\bs\mu\ci((\mu_G\ci\bs\pi_{12})\t\bs\pi_3):G\t G\t\bX\longra\bX,\\
\bs\mu\ci (1\t\bs\id_\bX)&=\bs\id_\bX,
\end{align*}
where $\bs\pi_1:G\t G\t\bX\ra G$ is projection to the first factor, $\bs\pi_{23}:G\t G\t\bX\ra G\t X$ projection to the second and third factors, and so on, $\mu_G:G\t G\ra G$ is group multiplication, and $1\in G$ the identity, regarded as a smooth map $*\ra G$. Here products of $\mu$-Kuranishi spaces $G\t\bX$, $G\t G\t\bX$ are as in Example~\ref{ku2ex4}.

Fukaya, Oh, Ohta and Ono use finite group actions on Kuranishi spaces in \cite{Fuka}, \cite[\S 7]{FOOO4}, and $T^n$-actions on Kuranishi spaces in~\cite{FOOO2,FOOO3}.
\label{ku2ex5}
\end{ex}

\subsection{\texorpdfstring{$\mu$-Kuranishi neighbourhoods on $\mu$-Kuranishi spaces}{\textmu-Kuranishi neighbourhoods on \textmu-Kuranishi spaces}}
\label{ku24}

At the beginning of differential geometry, one defines manifolds $X$ and smooth maps $f:X\ra Y$ in terms of an atlas $\bigl\{(U_i,\vp_i):i\in I\bigr\}$ of charts on $X$, and transition functions $\vp_{ij}=\vp_j^{-1}\ci\vp_i\vert_{\vp_i^{-1}(\Im\vp_j)}$ between charts $(U_i,\vp_i),(U_j,\vp_j)$. However, one quickly comes to regard actually choosing an atlas on $X$ or working explicitly with atlases as unnatural and inelegant, so we generally suppress them, working with `local coordinates' on $X$ if we really need to reduce things to~$\R^n$.

We now wish to advocate a similar philosophy for working with $\mu$-Kuranishi spaces $\bX=(X,\cK)$, in which, like atlases, actually choosing or working explicitly with $\mu$-Kuranishi structures $\cK=\bigl(I,(V_i,E_i,s_i,\psi_i)_{i\in I},\Phi_{ij,\; i,j\in I}\bigr)$ is regarded as inelegant and to be avoided where possible, and $\bX$ is understood to exist as a geometric space independently of any choices of $I,(V_i,E_i,s_i,\psi_i),\ldots.$ Our analogue of `local coordinates' will be `$\mu$-Kuranishi neighbourhoods on $\mu$-Kuranishi spaces'.

\begin{dfn} Suppose $\bX=(X,\cK)$ is a $\mu$-Kuranishi space, where $\cK=\bigl(I,\ab(V_i,\ab E_i,\ab s_i,\ab\psi_i)_{i\in I},\Phi_{ij,\; i,j\in I}\bigr)$. A {\it $\mu$-Kuranishi neighbourhood on\/} $\bX$ is data $(V_a,E_a,s_a,\psi_a)$ and $\Phi_{ai,\; i\in I}$, where $(V_a,E_a,s_a,\psi_a)$ is a $\mu$-Kuranishi neighbourhood on the topological space $X$ in the sense of Definition \ref{ku2def2}, and $\Phi_{ai}:(V_a,E_a,s_a,\psi_a)\ra(V_i,E_i,s_i,\psi_i)$ is a $\mu$-coordinate change for each $i\in I$ (over $S=\Im\psi_a\cap\Im\psi_i$, as usual), such that for all $i,j\in I$ we have
\e
\Phi_{ij}\ci\Phi_{ai}=\Phi_{aj},
\label{ku2eq22}
\e
where \eq{ku2eq22} holds over $S=\Im\psi_a\cap\Im\psi_i\cap\Im\psi_j$ by our usual convention.

Here the subscript `$a$' in $(V_a,E_a,s_a,\psi_a)$ is just a label used to distinguish $\mu$-Kuranishi neighbourhoods, generally not in $I$. If we omit $a$ we will write `$*$' in place of `$a$' in $\Phi_{ai}$, giving~$\Phi_{*i}:(V,E,s,\psi)\ra(V_i,E_i,s_i,\psi_i)$.

We will usually just say $(V_a,E_a,s_a,\psi_a)$ or $(V,E,s,\psi)$ {\it is a $\mu$-Kuranishi neighbourhood on\/} $\bX$, leaving the data $\Phi_{ai,\; i\in I}$ or $\Phi_{*i,\; i\in I}$ implicit. We call such a $(V,E,s,\psi)$ a {\it global\/} $\mu$-Kuranishi neighbourhood on $\bX$ if~$\Im\psi=X$. 
\label{ku2def15}
\end{dfn}

The next theorem can be proved using the sheaf property Theorem \ref{ku2thm1} by very similar methods to Theorem \ref{ku2thm3}, noting that \eq{ku2eq23}--\eq{ku2eq24} imply that
\begin{equation*}
\Phi_{ab}\vert_{\Im\psi_a\cap\Im\psi_b\cap\Im\psi_i}=\Phi_{bi}^{-1}\ci\Phi_{ai},\quad \bs f_{ab}\vert_{\begin{subarray}{l}\Im\psi_a\cap\Im\psi_i\cap\\
f^{-1}(\Im\psi_b\cap\Im\psi_j)\end{subarray}}\!\!=\Phi_{bj}^{-1}\ci\bs f_{ij}\ci \Tau_{bi},
\end{equation*}
so we leave the proof as an exercise for the reader.

\begin{thm}{\bf(a)} Let\/ $\bX=(X,\cK)$ be a $\mu$-Kuranishi space, where $\cK=\bigl(I,\ab(V_i,\ab E_i,\ab s_i,\ab\psi_i)_{i\in I},\Phi_{ij,\; i,j\in I}\bigr),$ and\/ $(V_a,E_a,s_a,\psi_a),(V_b,E_b,s_b,\psi_b)$ be $\mu$-Kuranishi neighbourhoods on $\bX,$ in the sense of Definition\/ {\rm\ref{ku2def15}}. Then there is a unique $\mu$-coordinate change $\Phi_{ab}:(V_a,E_a,s_a,\psi_a)\ra(V_b,E_b,s_b,\psi_b)$ in the sense of Definition\/ {\rm\ref{ku2def6}} such that for all\/ $i\in I$ we have
\e
\Phi_{bi}\ci\Phi_{ab}=\Phi_{ai},
\label{ku2eq23}
\e
which holds on $\Im\psi_a\cap\Im\psi_b\cap\Im\psi_i$ by our usual convention. We will call\/ $\Phi_{ab}$ the \begin{bfseries}$\mu$-coordinate change between the $\mu$-Kuranishi neighbourhoods $(V_a,E_a,s_a,\psi_a),(V_b,E_b,s_b,\psi_b)$ on the $\mu$-Kuranishi space\end{bfseries}~$\bX$.
\smallskip

\noindent{\bf(b)} Let\/ $\bs f:\bX\ra\bY$ be a morphism of $\mu$-Kuranishi spaces, with notation {\rm\eq{ku2eq14}--\eq{ku2eq15},} and let\/ $(U_a,D_a,r_a,\chi_a),(V_b,E_b,s_b,\psi_b)$ be $\mu$-Kuranishi neighbourhoods on $\bX,\bY$ respectively, in the sense of Definition\/ {\rm\ref{ku2def15}}. Then there is a unique morphism $\bs f_{ab}:(U_a,D_a,r_a,\chi_a)\ra(V_b,E_b,s_b,\psi_b)$ of $\mu$-Kuranishi neighbourhoods over $f,$ such that for all\/ $i\in I$ and\/ $j\in J$ we have
\e
\Phi_{bj}\ci\bs f_{ab}=\bs f_{ij}\ci\Tau_{bi},
\label{ku2eq24}
\e
which holds on $\Im\psi_a\cap\Im\psi_i\cap f^{-1}(\Im\psi_b\cap\Im\psi_j)$ by our usual convention. We will call\/ $\bs f_{ab}$ the \begin{bfseries}morphism of $\mu$-Kuranishi neighbourhoods $(V_a,\ab E_a,\ab s_a,\ab\psi_a),\ab(V_b,E_b,s_b,\psi_b)$ over\end{bfseries}~$\bs f:\bX\ra\bY$.
\label{ku2thm4}
\end{thm}

\begin{rem} Note that we make the (potentially confusing) distinction between {\it $\mu$-Kuranishi neighbourhoods\/ $(V_i,E_i,s_i,\psi_i)$ on a topological space\/} $X$, as in Definition \ref{ku2def2}, and {\it $\mu$-Kuranishi neighbourhoods\/ $(V_a,E_a,s_a,\psi_a)$ on a $\mu$-Kur\-an\-ishi space\/} $\bX=(X,\cK)$, which are as in Definition \ref{ku2def15}, and come equipped with the extra implicit data $\Phi_{ai,\; i\in I}$ giving the compatibility with the $\mu$-Kur\-an\-ishi structure $\cK$ on $X$.

We also distinguish between {\it $\mu$-coordinate changes\/ $\Phi_{ij}:(V_i,E_i,s_i,\psi_i)\ra(V_j,E_j,s_j,\psi_j)$ between $\mu$-Kuranishi neighbourhoods on a topological space\/} $X$, which are as in Definition \ref{ku2def2} and for which there may be many choices or none, and {\it $\mu$-coordinate changes\/ $\Phi_{ab}:(V_a,E_a,s_a,\psi_a)\ra(V_b,E_b,s_b,\psi_b)$ between $\mu$-Kuranishi neighbourhoods on a $\mu$-Kuranishi space\/} $\bX$, which are as in Theorem \ref{ku2thm4}(a), and are unique.

Similarly, we distinguish between {\it morphisms $\bs f_{ij}:(U_i,D_i,r_i,\chi_i)\ra(V_j,\ab E_j,\ab s_j,\ab\psi_j)$ of $\mu$-Kuranishi neighbourhoods over a continuous map of topological spaces\/} $f:X\ra Y$, which are as in Definition \ref{ku2def3} and for which there may be many choices or none, and {\it morphisms $\bs f_{ab}:(U_a,D_a,r_a,\chi_a)\ra(V_b,E_b,s_b,\psi_b)$ of $\mu$-Kuranishi neighbourhoods over a morphism of $\mu$-Kuranishi spaces\/} $\bs f:\bX\ra\bY$, which are as in Theorem \ref{ku2thm4}(b), and are unique.
\label{ku2rem5}
\end{rem}

\begin{prop} Let\/ $\bX=(X,\cK)$ be a $\mu$-Kuranishi space, and\/ $\bigl\{(V_a,\ab E_a,\ab s_a,\ab\psi_a):a\in A\bigr\}$ a family of $\mu$-Kuranishi neighbourhoods on $\bX$ with\/ $X=\bigcup_{a\in A}\Im\psi_a$. For all\/ $a,b\in A,$ let\/ $\Phi_{ab}:(V_a,E_a,s_a,\psi_a)\ra(V_b,E_b,s_b,\psi_b)$ be the $\mu$-coordinate change from Theorem\/ {\rm\ref{ku2thm4}(a)}. Then\/ $\cK'\!=\!\bigl(A,(V_a,E_a,s_a,\psi_a)_{a\in A},\ab\Phi_{ab,\; a,b\in A}\bigr)$ is a $\mu$-Kuranishi structure on $X,$ and\/ $\bX'=(X,\cK')$ is canonically isomorphic to\/ $\bX$ in\/~$\muKur$.
\label{ku2prop}
\end{prop}

\begin{proof} Write $\cK=\bigl(I,(V_i,E_i,s_i,\psi_i)_{i\in I},\Phi_{ij,\; i,j\in I}\bigr)$. Let $\Phi_{ab}$ for all $a,b\in A$ be as in Theorem \ref{ku2thm4}(a). We claim that $\cK'=\bigl(A,(V_a,E_a,s_a,\psi_a)_{a\in A},\ab\Phi_{ab,\; a,b\in A}\bigr)$ is a $\mu$-Kuranishi structure on $X$. Definition \ref{ku2def11}(a)--(d) are immediate. For (e), note that $\Phi_{aa},\id_{(V_a,E_a,s_a,\psi_a)}:(V_a,E_a,s_a,\psi_a)\ra (V_a,E_a,s_a,\psi_a)$ both satisfy the conditions of Theorem \ref{ku2thm4}(a) with $a=b$, so by uniqueness we have $\Phi_{aa}=\id_{(V_a,E_a,s_a,\psi_a)}$. Similarly, for $a,b,c\in A$ we can show that $\Phi_{bc}\ci\Phi_{ab}$ and $\Phi_{ac}$ are $\mu$-coordinate changes $(V_a,E_a,s_a,\psi_a)\ra (V_c,E_c,s_c,\psi_c)$ over $\Im\psi_a\cap\Im\psi_b\cap\Im\psi_c$ satisfying the conditions of Theorem \ref{ku2thm4}(a), so uniqueness gives $\Phi_{bc}\ci\Phi_{ab}=\Phi_{ac}$, proving (f). Hence $\cK'$ is a $\mu$-Kuranishi structure.

To show $\bX,\bX'$ are canonically isomorphic, note that each $(V_a,E_a,s_a,\psi_a)$ comes equipped with implicit extra data $\Phi_{ai,\; i\in I}$. Define morphisms $\bs f:\bX\ra\bX'$ and $\bs g:\bX'\ra\bX$ by $\bs f=\bigl(\id_X,\Phi_{ai,\; a\in A,\; i\in I}\bigr)$ and $\bs g=\bigl(\id_X,\Phi^{-1}_{ai,\; i\in I,\; a\in A}\bigr)$. It is easy to check that $\bs f,\bs g$ are morphisms in $\muKur$ with $\bs g\ci\bs f=\bs\id_\bX$ and $\bs f\ci\bs g=\bs\id_{\bX'}$. So $\bs f,\bs g$ are canonical isomorphisms.
\end{proof}

As the $\mu$-Kuranishi neighbourhoods $(V_i,E_i,s_i,\psi_i)$ in the $\mu$-Kuranishi structure on $\bX$ are $\mu$-Kuranishi neighbourhoods on $\bX$, we deduce:

\begin{cor} Let\/ $\bX=(X,\cK)$ be a $\mu$-Kuranishi space with\/ $\cK=\bigl(I,(V_i,\ab E_i,\ab s_i,\ab\psi_i)_{i\in I},\Phi_{ij,\; i,j\in I}\bigr)$. Suppose $J\subseteq I$ with $\bigcup_{j\in J}\Im\psi_j=X$. Then $\cK'=\bigl(J,(V_i,E_i,s_i,\psi_i)_{i\in J},\Phi_{ij,\; i,j\in J}\bigr)$ is a $\mu$-Kuranishi structure on $X,$ and\/ $\bX'=(X,\cK')$ is canonically isomorphic to $\bX$ in $\muKur$.
\label{ku2cor}
\end{cor}

Thus, adding or subtracting extra $\mu$-Kuranishi neighbourhoods to the $\mu$-Kuranishi structure of $\bX$ leaves $\bX$ unchanged up to canonical isomorphism.

We can now state our:
\smallskip

\noindent{\bf Philosophy for working with $\mu$-Kuranishi spaces:} {\it a good way to think about the `real' geometric structure on $\mu$-Kuranishi spaces is as follows:
\begin{itemize}
\setlength{\itemsep}{0pt}
\setlength{\parsep}{0pt}
\item[{\rm(i)}] Every $\mu$-Kuranishi space $\bX$ has an underlying topological space $X,$ and a large collection of `$\mu$-Kuranishi neighbourhoods' $(V_a,E_a,s_a,\psi_a)$ on $\bX,$ which are $\mu$-Kuranishi neighbourhoods on $X$ in the sense of\/ {\rm\S\ref{ku21},} but with an additional compatibility with the $\mu$-Kuranishi structure on~$\bX$.

We think of\/ $(V_a,E_a,s_a,\psi_a)$ as a choice of `local coordinates' on $\bX$.
\item[{\rm(ii)}] For $\mu$-Kuranishi neighbourhoods $(V_a,E_a,s_a,\psi_a),\!(V_b,E_b,s_b,\psi_b)$ on $\bX,$ we get a natural $\mu$-coordinate change $\Phi_{ab}\!:\!(V_a,E_a,s_a,\psi_a)\!\ra\!(V_b,E_b,s_b,\psi_b)$.
\item[{\rm(iii)}] A morphism of $\mu$-Kuranishi spaces $\bs f:\bX\ra\bY$ has an underlying continuous map $f:X\ra Y$. If\/ $(U_a,D_a,r_a,\chi_a),(V_b,E_b,s_b,\psi_b)$ are $\mu$-Kuranishi neighbourhoods on $\bX,\bY,$ there is a natural morphism $\bs f_{ab}:\ab(U_a,\ab D_a,\ab r_a,\ab\chi_a)\ra(V_b,E_b,s_b,\psi_b)$ over $f$.
\item[{\rm(iv)}] The natural $\mu$-coordinate changes and morphisms in {\rm(ii),(iii)} behave in the obvious functorial ways under compositions and identities.
\item[{\rm(v)}] The family of $\mu$-Kuranishi neighbourhoods on $\bX$ is closed under several natural constructions. For example:
\begin{itemize}
\setlength{\itemsep}{0pt}
\setlength{\parsep}{0pt}
\item[{\rm(a)}] If\/ $(V,E,s,\psi)$ is a $\mu$-Kuranishi neighbourhood on $\bX$ and\/ $V'\subseteq V$ is open then $\bigl(V',E\vert_{V'},s\vert_{V'},\psi\vert_{V'\cap s^{-1}(0)}\bigr)$ is a $\mu$-Kuranishi neighbourhood on $\bX$.
\item[{\rm(b)}] If\/ $(V,E,s,\psi)$ is a $\mu$-Kuranishi neighbourhood on $\bX$ and\/ $\pi:F\ra V$ is a vector bundle then $\bigl(F,\pi^*(E)\op\pi^*(F),\pi^*(s)\op\id_F,\psi\ci\pi\vert_{\cdots}\bigr)$ is a $\mu$-Kuranishi neighbourhood on $\bX$.
\end{itemize}
\item[{\rm(vi)}] The collection of all $\mu$-Kuranishi neighbourhoods $(V_a,E_a,s_a,\psi_a)$ on $\bX$ is `much larger' than a particular atlas $\bigl\{(V_i,E_i,s_i,\psi_i):i\in I\bigr\}$.

There are so many $\mu$-Kuranishi neighbourhoods on $\bX$ that we can often choose them to satisfy extra conditions. Two examples of this\/~{\rm\cite{Joyc12}:} 
\begin{itemize}
\setlength{\itemsep}{0pt}
\setlength{\parsep}{0pt}
\item[{\rm(a)}] For any $\bs f:\bX\ra\bY,$ we can choose families of $\mu$-Kuranishi neighbourhoods $(U_a,D_a,r_a,\chi_a)_{a\in A}$ which cover $\bX$ and\/ $(V_b,E_b,s_b,\psi_b)_{b\in B}$ which cover $\bY,$ such that\/ $\bs f_{ab}=[U_{ab},f_{ab},\hat f_{ab}]$ in {\rm(iii)} is represented by $(U_{ab},\ab f_{ab},\ab\hat f_{ab})$ with\/ $f_{ab}:U_{ab}\ra V_b$ a submersion, for all\/~$a,b$.
\item[{\rm(b)}] If\/ $\bX$ is a compact $\mu$-Kuranishi space, there is a $\mu$-Kuranishi neighbourhood $(V,E,s,\psi)$ on $\bX$ with\/~$\Im\psi=X$.
\end{itemize}
\end{itemize}}

\subsection{Tangent spaces and obstruction spaces}
\label{ku25}

A manifold $X$ has tangent spaces $T_xX$ for $x\in X$, which behave functorially under smooth maps of manifolds $f:X\ra Y$. Similarly, for a $\mu$-Kuranishi space $\bX$ we will define a {\it tangent space\/} $T_x\bX$ and {\it obstruction space\/} $O_x\bX$, which behave functorially under morphisms $\bs f:\bX\ra\bY$.

\begin{dfn} Let $\bX=(X,\cK)$ be a $\mu$-Kuranishi space, with $\cK=\bigl(I,(V_i,\ab E_i,\ab s_i,\ab\psi_i)_{i\in I},\Phi_{ij,\; i,j\in I}\bigr)$, and let $x\in\bX$.

For each $i\in I$ with $x\in\Im\psi_i$, set $v_i=\psi_i^{-1}(x)$, and define finite-dimensional real vector spaces $K_i^x,C_i^x$ by the exact sequence
\e
\xymatrix@C=20pt{ 0 \ar[r] & K_i^x \ar[r] & T_{v_i}V_i \ar[rr]^{\d s_i\vert_{v_i}} && E_i\vert_{v_i} \ar[r] & C_i^x \ar[r] & 0, }
\label{ku2eq25}
\e
where $\d s_i\vert_{v_i}$ is as in \eq{ku2eq8}, so that $K_i^x,C_i^x$ are the kernel and cokernel of $\d s_i\vert_{v_i}$. Also $\dim K_i^x-\dim C_i^x=\dim T_{v_i}V_i-\dim E_i\vert_{v_i}=\dim V_i-\rank E_i=\vdim\bX$.

For $i,j\in I$ with $x\in\Im\psi_i\cap\Im\psi_j$, let $(V_{ij},\phi_{ij},\hat\phi_{ij})$ represent the $\sim$-equivalence class $\Phi_{ij}\!=\![V_{ij},\phi_{ij},\hat\phi_{ij}]$. Then we form a diagram with exact rows
\e
\begin{gathered}
\xymatrix@C=23pt@R=15pt{ 0 \ar[r] & K_i^x \ar[r] \ar@{.>}[d]^{\ka^x_{V_{ij},\phi_{ij},\hat\phi_{ij}}}_\cong & T_{v_i}V_i \ar[rr]_{\d s_i\vert_{v_i}} \ar[d]^{\d\phi_{ij}\vert_{v_i}} && E_i\vert_{v_i} \ar[d]^{\hat\phi_{ij}\vert_{v_i}} \ar[r] & C_i^x \ar@{.>}[d]^{\ga^x_{V_{ij},\phi_{ij},\hat\phi_{ij}}}_\cong \ar[r] & 0 \\
0 \ar[r] & K_j^x \ar[r] & T_{v_j}V_j \ar[rr]^{\d s_j\vert_{v_j}} && E_j\vert_{v_j} \ar[r] & C_j^x \ar[r] & 0.\!\! }
\end{gathered}
\label{ku2eq26}
\e
From Definition \ref{ku2def3}(d) and $s_i(v_i)=0$, the central square of \eq{ku2eq26} commutes, and so by exactness there exist unique linear maps
$\ka^x_{V_{ij},\phi_{ij},\hat\phi_{ij}},\ga^x_{V_{ij},\phi_{ij},\hat\phi_{ij}}$ making the diagram commute. As $\Phi_{ij}=[V_{ij},\phi_{ij},\hat\phi_{ij}]$ is a $\mu$-coordinate change, Theorem \ref{ku2thm2} says \eq{ku2eq8} is exact, so $\ka^x_{V_{ij},\phi_{ij},\hat\phi_{ij}},\ga^x_{V_{ij},\phi_{ij},\hat\phi_{ij}}$ are isomorphisms.

Suppose $(V_{ij}',\phi_{ij}',\hat\phi_{ij}')$ is another representative, so that $(V_{ij}',\phi_{ij}',\hat\phi_{ij}')\sim(V_{ij},\phi_{ij},\hat\phi_{ij})$, and there exist $\dot V_{ij}$ and $\La$ satisfying \eq{ku2eq1}. Restricting the derivative of the first equation and the second equation of \eq{ku2eq1} to $v_i$ gives
\e
\d\phi_{ij}'\vert_{v_i}=\d\phi_{ij}\vert_{v_i}+\La\vert_{v_i}\ci \d s_i\vert_{v_i}\quad\text{and}\quad \hat\phi_{ij}'\vert_{v_i}=\hat\phi_{ij}\vert_{v_i}+\d s_j\vert_{v_j}\ci\La\vert_{v_i}.
\label{ku2eq27}
\e
Combining \eq{ku2eq27} with \eq{ku2eq26} for $(V_{ij},\phi_{ij},\hat\phi_{ij})$ and $(V_{ij}',\phi_{ij}',\hat\phi_{ij}')$, we can show that
$\ka^x_{V_{ij}',\phi_{ij}',\hat\phi_{ij}'}=\ka^x_{V_{ij},\phi_{ij},\hat\phi_{ij}}$ and $\ga^x_{V_{ij}',\phi_{ij}',\hat\phi_{ij}'}=\ga^x_{V_{ij},\phi_{ij},\hat\phi_{ij}}$. Thus $\ka^x_{V_{ij},\phi_{ij},\hat\phi_{ij}},\ab\ga^x_{V_{ij},\phi_{ij},\hat\phi_{ij}}$ depend only on the $\sim$-equivalence class $\Phi_{ij}$, so we write $\ka_{ij}^x=\ka^x_{V_{ij},\phi_{ij},\hat\phi_{ij}}$ and~$\ga_{ij}^x=\ga^x_{V_{ij},\phi_{ij},\hat\phi_{ij}}$.

If $i,j,k\in I$ with $x\in\Im\psi_i\cap\Im\psi_j\cap\Im\psi_k$ then from Definition \ref{ku2def11}(f) and \eq{ku2eq26} for $i,j$, $j,k$ and $i,k$ we can deduce that $\ka_{ik}^x=\ka_{jk}^x\ci\ka_{ij}^x:K_i^x\ra K_k^x$ and $\ga_{ik}^x=\ga_{jk}^x\ci\ga_{ij}^x:C_i^x\ra C_k^x$. To summarize: for any $i\in I$ with $x\in\Im\psi_i$ we get vector spaces $K_i^x,C_i^x$ with 
$\dim K_i^x-\dim C_i^x=\vdim\bX$. For two such $i,j$ we get canonical isomorphisms $\ka_{ij}^x:K_i^x\ra K_j^x$, $\ga_{ij}^x:C_i^x\ra C_j^x$, where three such $i,j,k$ satisfy $\ka_{ik}^x=\ka_{jk}^x\ci\ka_{ij}^x$ and~$\ga_{ik}^x=\ga_{jk}^x\ci\ga_{ij}^x$.

Define the {\it tangent space\/} $T_x\bX$ and {\it obstruction space\/} $O_x\bX$ of $\bX$ at $x$ by
\e
T_x\bX=\ts\coprod_{i\in I:x\in\Im\psi_i}K_i^x/\approx\quad\text{and}\quad
O_x\bX=\ts\coprod_{i\in I:x\in\Im\psi_i}C_i^x/\asymp,
\label{ku2eq28}
\e
where $\approx$ is the equivalence relation $k_i\approx k_j$ if $k_i\in K_i^x$ and $k_j\in K_j^x$ with $\ka_{ij}^x(k_i)=k_j$, and $\asymp$ the equivalence relation $c_i\asymp c_j$ if $c_i\in C_i^x$ and $c_j\in C_j^x$ with $\ga_{ij}^x(c_i)=c_j$. Then $T_x\bX,O_x\bX$ are finite-dimensional real vector spaces with canonical isomorphisms $T_x\bX\cong K_i^x$ and $O_x\bX\cong C_i^x$ for each $i\in I$ with $x\in\Im\psi_i$; the fuss above is just to make the definition of $T_x\bX,O_x\bX$ independent of the choice of $i$. We have
\e
\dim T_x\bX-\dim O_x\bX=\vdim\bX.
\label{ku2eq29}
\e

In contrast to manifolds, $\dim T_x\bX,\dim O_x\bX$ may not be locally constant functions of $x\in\bX$, but only upper semicontinuous, so $T_x\bX,O_x\bX$ are not fibres of vector bundles on~$\bX$.

The dual vector spaces of $T_x\bX,O_x\bX$ will be called the {\it cotangent space}, written $T_x^*\bX$, and the {\it coobstruction space}, written~$O_x^*\bX$.

By \eq{ku2eq25}, for any $i\in I$ with $x\in\Im\psi_i$ we have a canonical exact sequence
\e
\xymatrix@C=20pt{ 0 \ar[r] & T_x\bX \ar[r] & T_{v_i}V_i \ar[rr]^{\d s_i\vert_{v_i}} && E_i\vert_{v_i} \ar[r] & O_x\bX \ar[r] & 0. }
\label{ku2eq30}
\e
More generally, the argument above shows that if $(V_a,E_a,s_a,\psi_a)$ is any $\mu$-Kuranishi neighbourhood on $\bX$ in the sense of \S\ref{ku24} with $x\in\Im\psi_a$, we have a canonical exact sequence analogous to~\eq{ku2eq30}.

Now let $\bs f:\bX\ra\bY$ be a morphism of $\mu$-Kuranishi spaces, with notation \eq{ku2eq14}--\eq{ku2eq15}, and let $x\in\bX$ with $\bs f(x)=y\in\bY$, so we have $T_x\bX,O_x\bX$, $T_y\bY,O_y\bY$. Suppose $i\in I$ with $x\in\Im\chi_i$ and $j\in J$ with $y\in\Im\psi_j$, so we have a morphism $\bs f_{ij}=[U_{ij},f_{ij},\hat f_{ij}]$ in $\bs f$. Choose a representative $(U_{ij},f_{ij},\hat f_{ij})$. As for \eq{ku2eq26}, consider the diagram 
\e
\begin{gathered}
\xymatrix@C=23pt@R=15pt{ 0 \ar[r] & T_x\bX \ar[r] \ar@{.>}[d]^{T_x\bs f} & T_{u_i}U_i \ar[rr]_{\d r_i\vert_{u_i}} \ar[d]^{\d f_{ij}\vert_{u_i}} && D_i\vert_{u_i} \ar[d]^{\hat f_{ij}\vert_{u_i}} \ar[r] & O_x\bX \ar@{.>}[d]^{O_x\bs f} \ar[r] & 0 \\
0 \ar[r] & T_y\bY \ar[r] & T_{v_j}V_j \ar[rr]^{\d s_j\vert_{v_j}} && E_j\vert_{v_j} \ar[r] & O_y\bY \ar[r] & 0,\!\! }
\end{gathered}
\label{ku2eq31}
\e
where the rows are \eq{ku2eq30} for $\bX,x,i$ and $\bY,y,j$ and so are exact. The central square commutes, so there are unique linear maps $T_x\bs f:T_x\bX\ra T_y\bY$ and $O_x\bs f:O_x\bX\ra O_y\bY$ making \eq{ku2eq31} commute.

A similar argument to that for $\ka_{ij}^x,\ga_{ij}^x$ above shows that these $T_x\bs f,O_x\bs f$ are independent of the choice of representative $(U_{ij},f_{ij},\hat f_{ij})$, and so depend only on $\bs f_{ij}$, and are also independent of the choices of $i\in I$ and $j\in J$, so that $T_x\bs f,O_x\bs f$ depend only on $x\in\bX$ and~$\bs f:\bX\ra\bY$.

If $(U_a,D_a,r_a,\chi_a)$ and $(V_b,E_b,s_b,\psi_b)$ are any $\mu$-Kuranishi neighbourhoods on $\bX,\bY$ respectively in the sense of \S\ref{ku24} with $x\in\Im\psi_a$, $y\in\Im\psi_b$, and $\bs f_{ab}=[U_{ab},f_{ab},\hat f_{ab}]$ is the morphism of $\mu$-Kuranishi neighbourhoods over $\bs f$ given by Theorem \ref{ku2thm4}(b), then setting $u_a=\chi_a^{-1}(x)$, $v_b=\psi_b^{-1}(y)$, then the argument of \eq{ku2eq31} shows that the following commutes, with exact rows:
\begin{equation*}
\xymatrix@C=23pt@R=15pt{ 0 \ar[r] & T_x\bX \ar[r] \ar[d]^{T_x\bs f} & T_{u_a}U_a \ar[rr]_{\d r_a\vert_{u_a}} \ar[d]^{\d f_{ab}\vert_{u_a}} && D_a\vert_{u_a} \ar[d]^{\hat f_{ab}\vert_{u_a}} \ar[r] & O_x\bX \ar[d]^{O_x\bs f} \ar[r] & 0 \\
0 \ar[r] & T_y\bY \ar[r] & T_{v_b}V_b \ar[rr]^{\d s_b\vert_{v_b}} && E_b\vert_{v_b} \ar[r] & O_y\bY \ar[r] & 0.\!\! }
\end{equation*}

If $\bs g:\bY\ra\bZ$ is another morphism of $\mu$-Kuranishi spaces and $\bs g(y)=z\in\bZ$ then combining \eq{ku2eq31} for $\bs f,\bs g$ and $\bs g\ci\bs f$, it is easy to see that
\begin{align*}
T_x(\bs g\ci\bs f)&=T_y\bs g\ci T_x\bs f:T_x\bX\longra T_z\bZ,\\
O_x(\bs g\ci\bs f)&=O_y\bs g\ci O_x\bs f:O_x\bX\longra O_z\bZ.
\end{align*}
Also $T_x\bs\id_\bX=\id_{T_x\bX}:T_x\bX\ra T_x\bX$, $O_x\bs\id_\bX=\id_{O_x\bX}:O_x\bX\ra O_x\bX$. So tangent and obstruction spaces are functorial.
\label{ku2def16}
\end{dfn}

\begin{rem} In applications, tangent and obstruction spaces will often have the following interpretation. Suppose a $\mu$-Kuranishi space $\bX$ is the moduli space of solutions of a nonlinear elliptic equation on a compact manifold, written as $\bX\cong\Phi^{-1}(0)$ for $\Phi:\cV\ra\cE$ a Fredholm section of a Banach vector bundle $\cE\ra\cV$ over a Banach manifold $\cV$. Then $\d_x\Phi:T_x\cV\ra\cE_x$ is a linear Fredholm map of Banach spaces for $x\in\bX$, and $T_x\bX\cong\Ker(\d_x\Phi)$, $O_x\bX\cong\Coker(\d_x\Phi)$.
\label{ku2rem6}
\end{rem}

We finish by quoting a few definitions and results from the sequel \cite{Joyc12}, to show tangent and obstruction spaces are very useful when working with $\mu$-Kuranishi spaces. There are two natural notions of submersion, immersion and embedding of $\mu$-Kuranishi spaces, a weak and a strong:

\begin{dfn}[\!\!\cite{Joyc12}] Let $\bs f:\bX\ra\bY$ be a morphism of $\mu$-Kuranishi spaces. 

\smallskip

\noindent{\bf(a)} Call $\bs f$ a {\it w-submersion\/} if $O_x\bs f:O_x\bX\ra O_{\bs f(x)}\bY$ is surjective for all~$x\in\bX$.

\smallskip

\noindent{\bf(b)} Call $\bs f$ a {\it submersion\/} if $T_x\bs f:T_x\bX\ra T_{\bs f(x)}\bY$ is surjective and $O_x\bs f:O_x\bX\ra O_{\bs f(x)}\bY$ is an isomorphism for all $x\in\bX$.
\smallskip

\noindent{\bf(c)} Call $\bs f$ a {\it w-immersion\/} if $T_x\bs f:T_x\bX\ra T_{\bs f(x)}\bY$ is injective for all $x\in\bX$, and a {\it w-embedding\/} if $\bs f$ is a w-immersion and $f:X\ra f(X)$ is a homeomorphism.
\smallskip

\noindent{\bf(d)} Call $\bs f$ an {\it immersion\/} if $T_x\bs f:T_x\bX\ra T_{\bs f(x)}\bY$ is injective and $O_x\bs f:O_x\bX\ra O_{\bs f(x)}\bY$ is surjective for all $x\in\bX$, and an {\it embedding\/} if $\bs f$ is an immersion and $f:X\ra f(X)$ is a homeomorphism.
\label{ku2def17}
\end{dfn}

The natural notion of transverse morphisms of $\mu$-Kuranishi spaces involves obstruction spaces rather than tangent spaces:

\begin{dfn}[\!\!\cite{Joyc12}] Let $\bs g:\bX\ra\bZ$ and $\bs h:\bY\ra\bZ$ be morphisms of $\mu$-Kuranishi spaces. Call $\bs g,\bs h$ {\it d-transverse\/} if for all $x\in\bX$ and $y\in\bY$ with $\bs g(x)=\bs h(y)=z\in\bZ$, the map $O_x\bs g\op O_y\bs h:O_x\bX\op O_y\bY\ra O_z\bZ$ is surjective.

If $\bs g$ or $\bs h$ is a w-submersion, then $\bs g,\bs h$ are automatically d-transverse.
\label{ku2def18}
\end{dfn}

Here are some sample results:

\begin{ex}[\cite{Joyc12}] {\bf(i)} A $\mu$-Kuranishi space $\bX$ is a manifold, in the sense of Definition \ref{ku2def14}, if and only if $O_x\bX=0$ for all~$x\in\bX$.
\smallskip

\noindent{\bf(ii)} A morphism $\bs f:\bX\ra\bY$ in $\muKur$ is \'etale (a local isomorphism) if and only if $T_x\bs f:T_x\bX\ra T_{\bs f(x)}\bY$ and $O_x\bs f:O_x\bX\ra O_{\bs f(x)}\bY$ are isomorphisms for all $x\in\bX$. If also $f:X\ra Y$ is a bijection, then $\bs f$ is an isomorphism.
\smallskip

\noindent{\bf(iii)} Let $\bX$ be a $\mu$-Kuranishi space, $Y$ a manifold, and $\bs f:\bX\hookra Y$ an embedding. Then there exists an open neighbourhood $V$ of $f(X)$ in $Y,$ a vector bundle $E\ra V,$ and a smooth section $s:V\ra E$ with $s^{-1}(0)=f(X)\subseteq V\subseteq Y,$ so we may define $\psi=f^{-1}:s^{-1}(0)\ra X,$ and then $(V,E,s,\psi)$ may be given the structure of a global $\mu$-Kuranishi neighbourhood on $\bX,$ as in \S\ref{ku24}. Thus $\bX\cong\bS_{V,E,s}$ in $\muKur$, for $\bS_{V,E,s}$ as in Example~\ref{ku2ex2}.
\smallskip

\noindent{\bf(iv)} Let $\bs g:\bX\ra\bZ$ and $\bs h:\bY\ra\bZ$ be d-transverse morphisms of $\mu$-Kuranishi spaces. Then we can construct a `fibre product' $\bW=\bX\t_{\bs g,\bZ,\bs h}\bY$ in $\muKur$, with topological space $W=\bigl\{(x,y)\in X\t Y:g(x)=h(y)\in Z\bigr\}$ and virtual dimension $\vdim\bW=\vdim\bX+\vdim\bY-\vdim\bZ$. If $\bs g$ is a submersion and $\bY$ is a manifold, then $\bW$ is a manifold.

Note that this is {\it not\/} a fibre product in the sense of category theory, that is, $\bW$ does not satisfy a universal property in $\muKur$. Section \ref{ku4} constructs a weak 2-category of Kuranishi spaces $\Kur$, and \S\ref{ku47} defines a full 2-subcategory $\mKur\subset\Kur$ of `m-Kuranishi spaces', and shows that $\muKur\simeq\Ho(\mKur)$, where $\Ho(\mKur)$ is the homotopy category of the 2-category $\mKur$. As in \cite{Joyc12}, $\bW$ is in fact a genuine fibre product in the 2-category $\mKur$, and satisfies a universal property involving 2-morphisms there, but this universal property does not descend to $\Ho(\mKur)\simeq\muKur$. Thus, $\muKur$ does not have as good categorical properties as~$\mKur$.
\label{ku2ex6}
\end{ex}

\section[\texorpdfstring{$\mu$-Kuranishi spaces with boundary and corners}{\textmu-Kuranishi spaces with boundary and corners}]{$\mu$-Kuranishi spaces with corners}
\label{ku3}

We now extend \S\ref{ku2} from $\mu$-Kuranishi spaces without boundary, to $\mu$-Kuranishi spaces with boundary or corners. To do this we need a good theory of manifolds with corners. So we begin in \S\ref{ku31}--\S\ref{ku34} with background on manifolds with (g-)corners, boundaries, corners, and (b-)tangent bundles. Then \S\ref{ku35}--\S\ref{ku38} generalize \S\ref{ku31}--\S\ref{ku34} to $\mu$-Kuranishi spaces with corners.

There are several interesting, non-equivalent ways to define a category of manifolds with corners. Our definitions of ($\mu$-)Kuranishi spaces with corners in \S\ref{ku35}--\S\ref{ku38} and \S\ref{ku5} are not very sensitive to the details of these: given any category $\mathop{\bf\ti Man^c}$ of manifolds with corners satisfying a list of basic properties, we get a corresponding category $\mathop{\bf\mu\ti Kur^c}$ of $\mu$-Kuranishi spaces with corners, and a corresponding weak 2-category $\mathop{\bf\ti Kur^c}$ of Kuranishi spaces with corners.

Sections \ref{ku31} discusses categories of manifolds with corners. Our primary definition, written $\Manc$, has morphisms Melrose's `b-maps' \cite{KoMe,Melr1,Melr2,Melr3}, which we just call `smooth maps'. We chose these as they are more general than most other definitions, but preserve a lot of interesting structure. We give subcategories $\Mancst,\Mancin,\Mancis,\Mancsi\subset\Manc$ with restrictions on morphisms. 

Section \ref{ku34} explains the author's notion \cite{Joyc9} of {\it manifolds with generalized corners}, or {\it g-corners}, which form a category $\Mangc$ with $\Manc\subset\Mangc$ as a full subcategory, with subcategories $\Mangcin,\Mangcsi\subset\Mangc$. Each of these categories of manifolds with corners yields a category of $\mu$-Kuranishi spaces $\muKurc$, $\muKurcst$, $\muKurcin$, $\muKurcis$, $\muKurcsi$, $\muKurgc$, $\muKurgcin$, $\muKurgcsi$ in \S\ref{ku35}--\S\ref{ku38}, and similarly for 2-categories of Kuranishi spaces in \S\ref{ku5}. The notation of \S\ref{ku31}--\S\ref{ku34} follows~\cite{Joyc9}.

\subsection{Several categories of manifolds with corners}
\label{ku31}

We now define our main category $\Manc$ of manifolds with corners. The relation of our definitions to others in the literature is explained in Remark~\ref{ku3rem1}.

\begin{dfn} Use the notation $\R^m_k=[0,\iy)^k\t\R^{m-k}$
for $0\le k\le m$, and write points of $\R^m_k$ as $u=(u_1,\ldots,u_m)$ for $u_1,\ldots,u_k\in[0,\iy)$, $u_{k+1},\ldots,u_m\in\R$. Let $U\subseteq\R^m_k$ and $V\subseteq \R^n_l$ be open, and $f=(f_1,\ldots,f_n):U\ra V$ be a continuous map, which implies that $f_j=f_j(u_1,\ldots,u_m)$ maps $U\ra[0,\iy)$ for $j=1,\ldots,l$ and $U\ra\R$ for $j=l+1,\ldots,n$. Then we say:
\begin{itemize}
\setlength{\itemsep}{0pt}
\setlength{\parsep}{0pt}
\item[(a)] $f$ is {\it weakly smooth\/} if all derivatives $\frac{\pd^{a_1+\cdots+a_m}}{\pd u_1^{a_1}\cdots\pd u_m^{a_m}}f_j(u_1,\ldots,u_m):U\ra\R$ exist and are continuous in for all $j=1,\ldots,m$ and $a_1,\ldots,a_m\ge 0$, including one-sided derivatives where $u_i=0$ for $i=1,\ldots,k$.

By Seeley's Extension Theorem, this is equivalent to requiring $f_j$ to extend to a smooth function $f_j':U'\ra\R$ on open neighbourhood $U'$ of $U$ in~$\R^m$.
\item[(b)] $f$ is {\it smooth\/} if it is weakly smooth and every $u=(u_1,\ldots,u_m)\in U$ has an open neighbourhood $\ti U$ in $U$ such that for each $j=1,\ldots,l$, either:
\begin{itemize}
\setlength{\itemsep}{0pt}
\setlength{\parsep}{0pt}
\item[(i)] we may uniquely write $f_j(\ti u_1,\ldots,\ti u_m)=F_j(\ti u_1,\ldots,\ti u_m)\cdot\ti u_1^{a_{1,j}}\cdots\ti u_k^{a_{k,j}}$ for all $(\ti u_1,\ldots,\ti u_m)\in\ti U$, where $F_j:\ti U\ra(0,\iy)$ is weakly smooth and $a_{1,j},\ldots,a_{k,j}\in\N=\{0,1,2,\ldots\}$, with $a_{i,j}=0$ if $u_i\ne 0$; or 
\item[(ii)] $f_j\vert_{\smash{\ti U}}=0$.
\end{itemize}
\item[(c)] $f$ is {\it interior\/} if it is smooth, and case (b)(ii) does not occur.
\item[(d)] $f$ is {\it b-normal\/} if it is interior, and in case (b)(i), for each $i=1,\ldots,k$ we have $a_{i,j}>0$ for at most one $j=1,\ldots,l$.
\item[(e)] $f$ is {\it strongly smooth\/} if it is smooth, and in case (b)(i), for each $j=1,\ldots,l$ we have $a_{i,j}=1$ for at most one $i=1,\ldots,k$, and $a_{i,j}=0$ otherwise. 
\item[(f)] $f$ is {\it simple\/} if it is interior, and in case (b)(i), for each $i=1,\ldots,k$ with $u_i=0$ we have $a_{i,j}=1$ for exactly one $j=1,\ldots,l$ and $a_{i,j}=0$ otherwise, and for all $j=1,\ldots,l$ we have $a_{i,j}=1$ for at most one $i=1,\ldots,k$. 

Simple maps are strongly smooth and b-normal. ($\mu$-)coordinate changes for ($\mu$-)Kuranishi neighbourhoods with corners involve simple maps.

\item[(g)] $f$ is a {\it diffeomorphism\/} if it is a bijection and both $f:U\ra V$ and $f^{-1}:V\ra U$ are weakly smooth.

This implies that $f,f^{-1}$ are also smooth, interior, b-normal, strongly smooth, and simple. Hence, all the different definitions of smooth maps of manifolds with corners we discuss yield the same notion of diffeomorphism. 

\end{itemize}

All seven of these classes of maps $f:U\ra V$ include identities, and are closed under compositions from $f:U\ra V$, $g:V\ra W$ to $g\ci f:U\ra W$. Thus, each of them makes the open subsets $U\subseteq\R^m_k$ for all $m,k$ into a category. Also, all seven classes form sheaves on the domain.
\label{ku3def1}
\end{dfn}

\begin{dfn} Let $X$ be a second countable Hausdorff topological space. An {\it
$m$-dimensional chart on\/} $X$ is a pair $(U,\phi)$, where
$U\subseteq\R^m_k$ is open for some $0\le k\le m$, and $\phi:U\ra X$ is a
homeomorphism with an open set~$\phi(U)\subseteq X$.

Let $(U,\phi),(V,\psi)$ be $m$-dimensional charts on $X$. We call
$(U,\phi)$ and $(V,\psi)$ {\it compatible\/} if
$\psi^{-1}\ci\phi:\phi^{-1}\bigl(\phi(U)\cap\psi(V)\bigr)\ra
\psi^{-1}\bigl(\phi(U)\cap\psi(V)\bigr)$ is a diffeomorphism between open subsets of $\R^m_k,\R^m_l$, in the sense of Definition \ref{ku3def1}(g).

An $m$-{\it dimensional atlas\/} for $X$ is a system
$\{(U_a,\phi_a):a\in A\}$ of pairwise compatible $m$-dimensional
charts on $X$ with $X=\bigcup_{a\in A}\phi_a(U_a)$. We call such an
atlas {\it maximal\/} if it is not a proper subset of any other
atlas. Any atlas $\{(U_a,\phi_a):a\in A\}$ is contained in a unique
maximal atlas, the set of all charts $(U,\phi)$ of this type on $X$
which are compatible with $(U_a,\phi_a)$ for all~$a\in A$.

An $m$-{\it dimensional manifold with corners\/} is a second
countable Hausdorff topological space $X$ equipped with a maximal
$m$-dimensional atlas. Usually we refer to $X$ as the manifold,
leaving the atlas implicit, and by a {\it chart\/ $(U,\phi)$ on\/}
$X$, we mean an element of the maximal atlas.

Now let $X,Y$ be manifolds with corners of dimensions $m,n$, and $f:X\ra Y$ a continuous map. We call $f$ {\it weakly smooth}, or {\it smooth}, or {\it interior}, or {\it b-normal}, or {\it strongly smooth}, or {\it simple}, if whenever $(U,\phi),(V,\psi)$ are charts on $X,Y$ with $U\subseteq\R^m_k$, $V\subseteq\R^n_l$ open, then
\e
\psi^{-1}\ci f\ci\phi:(f\ci\phi)^{-1}(\psi(V))\longra V
\label{ku3eq1}
\e
is weakly smooth, or smooth, or interior, or b-normal, or strongly smooth, or simple, respectively, as maps between open subsets of $\R^m_k,\R^n_l$ in the sense of Definition \ref{ku3def1}. It is sufficient to check this on any collections of charts $(U_a,\phi_a)_{a\in A}$ covering $X$ and $(V_b,\psi_b)_{b\in B}$ covering~$Y$.

We call $f:X\ra Y$ a {\it diffeomorphism\/} if $f$ is a bijection and $f:X\ra Y$, $f^{-1}:Y\ra X$ are weakly smooth. This implies $f,f^{-1}$ are also smooth, interior, strongly smooth, and simple.

These seven classes of (a) weakly smooth maps, (b) smooth maps, (c) interior maps, (d) b-normal maps, (e) strongly smooth maps, (f) simple maps, and (g) diffeomorphisms, of manifolds with corners, all contain identities and are closed under composition, so each makes manifolds with corners into a category. 

In this book, we work with smooth maps of manifolds with corners (as we have defined them), and we write $\Manc$ for the category with objects manifolds with corners $X,Y,$ and morphisms smooth maps $f:X\ra Y$ in the sense above. 

We will also write $\Mancin,\Mancst,\Mancis,\Mancsi$ for the subcategories of $\Manc$ with morphisms interior maps, and strongly smooth maps, and interior strongly smooth maps, and simple maps, respectively.

Write $\cManc$ for the category whose objects are disjoint unions $\coprod_{m=0}^\iy X_m$, where $X_m$ is a manifold with corners of dimension $m$, allowing $X_m=\es$, and whose morphisms are continuous maps $f:\coprod_{m=0}^\iy X_m\ra\coprod_{n=0}^\iy Y_n$, such that
$f\vert_{X_m\cap f^{-1}(Y_n)}:\bigl(X_m\cap f^{-1}(Y_n)\bigr)\ra
Y_n$ is a smooth map of manifolds with corners for all $m,n\ge 0$.
Objects of $\cManc$ will be called {\it manifolds with corners of
mixed dimension}. We regard $\Manc$ as a full subcategory of~$\cManc$.

Alternatively, we can regard $\cManc$ as the category defined exactly as for $\Manc$ above, except that in defining atlases $\{(U_a,\phi_a):a\in A\}$ on $X$, we omit the condition that all charts $(U_a,\phi_a)$ in the atlas must have the same dimension~$\dim U_a=m$.

We will also write $\cMancin,\cMancst,\cMancis,\cMancsi$ for the subcategories of $\cManc$ with the same objects, and morphisms interior, or strongly smooth, or interior strongly smooth, or simple maps, respectively.
\label{ku3def2}
\end{dfn}

\begin{rem} Some references on manifolds with corners are Cerf \cite{Cerf}, Douady \cite{Doua}, Gillam and Molcho \cite[\S 6.7]{GiMo}, Kottke and Melrose \cite{KoMe}, Margalef-Roig and Outerelo Dominguez \cite{MaOu}, Melrose \cite{Melr1,Melr2,Melr3}, Monthubert \cite{Mont}, and the author \cite{Joyc3}, \cite[\S 5]{Joyc8}. Just as objects, without considering morphisms, most authors define manifolds with corners as in Definition \ref{ku3def2}. However, Melrose \cite{KoMe,Melr1,Melr2,Melr3} and authors who follow him impose an extra condition: in \S\ref{ku32} we will define the boundary $\pd X$ of a manifold with corners $X$, with an immersion $i_X:\pd X\ra X$. Melrose requires that $i_X\vert_C:C\ra X$ should be injective for each connected component $C$ of $\pd X$ (such $X$ are sometimes called {\it manifolds with faces\/}).

There is no general agreement in the literature on how to define smooth maps, or morphisms, of manifolds with corners: 
\begin{itemize}
\setlength{\itemsep}{0pt}
\setlength{\parsep}{0pt}
\item[(i)] Our notion of `smooth map' in Definitions \ref{ku3def1} and \ref{ku3def2} is due to Melrose \cite[\S 1.12]{Melr3}, \cite[\S 1]{Melr1}, \cite[\S 1]{KoMe}, who calls them {\it b-maps}. 

Our notation of `interior maps' and `b-normal maps' is also due to Melrose.
\item[(ii)] Monthubert's {\it morphisms of manifolds with corners\/} \cite[Def.~2.8]{Mont} coincide with our strongly smooth b-normal maps. 
\item[(iii)] The author \cite{Joyc3} defined and studied `strongly smooth maps' above (which were just called `smooth maps' in \cite{Joyc3}). 

Strongly smooth maps were also used to define 1-morphisms of d-manifolds with corners in the 2012 version of \cite{Joyc8}. However, the final version of \cite{Joyc8} will have a different definition using smooth maps (i.e.~Melrose's b-maps).
\item[(iv)] Gillam and Molcho's {\it morphisms of manifolds with corners\/} \cite[\S 6.7]{GiMo} coincide with our `interior maps'.
\item[(v)] Most other authors, such as Cerf \cite[\S I.1.2]{Cerf}, define smooth maps of manifolds with corners to be weakly smooth maps, in our notation.
\end{itemize}
\label{ku3rem1}
\end{rem}

\subsection{Boundaries and corners of manifolds with corners}
\label{ku32}

The material of this section broadly follows the author \cite{Joyc3}, \cite[\S 2.2 \& \S 3.4]{Joyc9}.

\begin{dfn} Let $U\subseteq\R^n_l$ be open. For each
$u=(u_1,\ldots,u_n)$ in $U$, define the {\it depth\/} $\depth_Uu$ of
$u$ in $U$ to be the number of $u_1,\ldots,u_l$ which are zero. That
is, $\depth_Uu$ is the number of boundary faces of $U$
containing~$u$.

Let $X$ be an $n$-manifold with corners. For $x\in X$, choose a
chart $(U,\phi)$ on the manifold $X$ with $\phi(u)=x$ for $u\in U$,
and define the {\it depth\/} $\depth_Xx$ of $x$ in $X$ by
$\depth_Xx=\depth_Uu$. This is independent of the choice of
$(U,\phi)$. For each $l=0,\ldots,n$, define the {\it depth\/ $l$
stratum\/} of $X$ to be
\begin{equation*}
S^l(X)=\bigl\{x\in X:\depth_Xx=l\bigr\}.
\end{equation*}
Then $X=\coprod_{l=0}^nS^l(X)$ and $\overline{S^l(X)}=
\bigcup_{k=l}^n S^k(X)$. The {\it interior\/} of $X$ is
$X^\ci=S^0(X)$. Each $S^l(X)$ has the structure of an
$(n-l)$-manifold without boundary.

\label{ku3def3}
\end{dfn}

\begin{dfn} Let $X$ be an $n$-manifold with corners, $x\in X$, and $k=0,1,\ldots,n$. A {\it local $k$-corner component\/ $\ga$ of\/ $X$ at\/} $x$ is a local choice of connected component of $S^k(X)$ near $x$. That is, for
each sufficiently small open neighbourhood $V$ of $x$ in $X$, $\ga$
gives a choice of connected component $W$ of $V\cap S^k(X)$ with
$x\in\overline W$, and any two such choices $V,W$ and $V',W'$ must
be compatible in that~$x\in\overline{(W\cap W')}$.

Let $\depth_Xx=l$. Choose a chart $(U,\phi)$ on $X$ with $(0,\ldots,0)\in U\subseteq\R^n_l$ open and $\phi(0,\ldots,0)=x$. Then we have
\e
\begin{split}
S^k(U)=\coprod_{1\le a_1<a_2<\cdots<a_k\le l}\begin{aligned}[t]
\bigl\{(u_1&,\ldots,u_n)\in U:u_{a_i}=0,\;\> i=1,\ldots,k,\\
&u_j\ne 0,\;\> j\in \{1,\ldots,l\}\sm\{a_1,\ldots,a_k\}\bigr\}.
\end{aligned}
\end{split}
\label{ku3eq2}
\e
For each choice of $a_1,\ldots,a_k$, the subset on the right hand of \eq{ku3eq2} contains $(0,\ldots,0)$ in its closure in $U$, and its intersection with a small ball about $(0,\ldots,0)$ is connected. Thus this subset determines a local $k$-corner component of $U$ at $(0,\ldots,0)$, and hence a local $k$-corner component of $X$ at~$x$. 

Equation \eq{ku3eq2} implies that all local $k$-corner components of $U$ at $(0,\ldots,0)$ and $X$ at $x$ are of this form. Therefore, local $k$-corner components of $U\subseteq\R^n_l$ at $(0,\ldots,0)$ are in 1-1 correspondence with subsets $\{a_1,\ldots,a_k\}\subseteq\{1,\ldots,l\}$ of size $k$, and there are $\binom{\depth_Xx}{k}$ distinct local $k$-corner components of $X$ at~$x$.

When $k=1$, we also call local 1-corner components {\it local boundary components of\/ $X$ at\/} $x$. There are $\depth_Xx$ distinct local boundary components of $X$ at $x$. By considering the local model $\R^n_l$, it is easy to see that there is a natural 1-1 correspondence between local $k$-corner components $\ga$ of $X$ at $x$, and (unordered) sets $\{\be_1,\ldots,\be_k\}$ of $k$ distinct local boundary components $\be_1,\ldots,\be_k$ of $X$ at $x$, such that if $V$ is a sufficiently small open neighbourhood of $x$ in $X$ and $\be_1,\ldots,\be_k$ and $\ga$ give connected components $W_1,\ldots,W_k$ of $V\cap S^1(X)$ and $W'$ of $V\cap S^k(X)$, then $W'\subseteq\bigcap_{i=1}^k\kern .1em\overline{\kern -.1em W}_{\!i}$.

As sets, define the {\it boundary\/} $\pd X$ and {\it $k$-corners\/} $C_k(X)$ for $k=0,1,\ldots,n$ by
\ea
\pd X&=\bigl\{(x,\be):\text{$x\in X$, $\be$ is a local boundary
component of $X$ at $x$}\bigr\},
\label{ku3eq3}\\
C_k(X)&=\bigl\{(x,\ga):\text{$x\in X$, $\ga$ is a local $k$-corner 
component of $X$ at $x$}\bigr\},
\label{ku3eq4}
\ea
so that $\pd X=C_1(X)$. The 1-1 correspondence above shows that
\e
\begin{aligned}
C_k(X)\cong\bigl\{(x,\,&\{\be_1,\ldots,\be_k\}):\text{$x\in X,$
$\be_1,\ldots,\be_k$ are distinct}\\
&\text{local boundary components for $X$ at $x$}\bigr\}.
\end{aligned}
\label{ku3eq5}
\e
Since each $x\in X$ has a unique 0-boundary component, we have~$C_0(X)\cong X$.

If $(U,\phi)$ is a chart on $X$ with $U\subseteq\R^n_l$ open, then
for each $i=1,\ldots,l$ we can define a chart $(U_i,\phi_i)$ on $\pd
X$ by
\e
\begin{split}
&U_i\!=\!\bigl\{(v_1,\ldots,v_{n-1})\!\in\!\R^{n-1}_{l-1}:
(v_1,\ldots,v_{i-1},0,v_i,\ldots,v_{n-1})\!\in\!
U\subseteq\R^n_l\bigr\},\\
&\phi_i:(v_1,\ldots,v_{n-1})\longmapsto\bigl(\phi
(v_1,\ldots,v_{i-1},0,v_i,\ldots,v_{n-1}),\phi_*(\{u_i=0\})\bigr).
\end{split}
\label{ku3eq6}
\e
Similarly, if $0\le k\le l$, then for each $1\le a_1<\cdots<a_k\le l$ we can define a chart $(U_{\{a_1,\ldots,a_k\}},\phi_{\{a_1,\ldots,a_k\}})$ on $C_k(X)$ by
\e
\begin{split}
&U_{\{a_1,\ldots,a_k\}}\!=\!\bigl\{(v_1,\ldots,v_{n-k})\!\in\!\R^{n-k}_{l-k}:
(v_1,\ldots,v_{a_1-1},0,v_{a_1},\ldots,v_{a_2-2},0,\\
&\quad v_{a_2-1},\ldots,v_{a_3-3},0,v_{a_3-2},\ldots,v_{a_k-k},0,v_{a_k-k+1},\ldots,v_{n-k})\!\in\!
U\subseteq\R^n_l\bigr\},\\
&\phi_{\{a_1,\ldots,a_k\}}:(v_1,\ldots,v_{n-k})\longmapsto\bigl( 
\phi(v_1,\ldots,v_{a_1-1},0,v_{a_1},\ldots,v_{a_2-2},0,\\
&\quad v_{a_2-1},\ldots,v_{a_3-3},0,v_{a_3-2},\ldots,v_{a_k-k},0,v_{a_k-k+1},\ldots,v_{n-k}),\\
&\qquad\qquad\qquad\qquad \phi_*(\{u_{a_1}=\cdots=u_{a_k}=0\})\bigr).
\end{split}
\label{ku3eq7}
\e
The families of all such charts on $\pd X$ and $C_k(X)$ are pairwise compatible, and define atlases on $\pd X$ and $C_k(X)$. The corresponding maximal atlases make $\pd X$ into an $(n-1)$-manifold with corners and $C_k(X)$ into an $(n-k)$-manifold with corners, with $\pd X=C_1(X)$ and $C_0(X)\cong X$ as manifolds with corners.

Define the {\it corners\/ $C(X)$ of\/} $X$ by
\e
C(X)=\ts\coprod_{k=0}^{\dim X}C_k(X),
\label{ku3eq8}
\e
considered as an object of $\cManc$ in Definition \ref{ku3def2}, a manifold with corners of mixed dimension.

Define maps $i_X:\pd X\ra X$, $i_X:C_k(X)\ra X$ and $i_X:C(X)\ra X$ by $i_X:(x,\be)\mapsto x$ and $i_X:(x,\ga)\mapsto x$. Since these are locally modelled on the inclusion maps $U_i\hookra U$ and $U_{\{a_1,\ldots,a_k\}}\hookra U$ for $U_i,U_{\{a_1,\ldots,a_k\}}$ as in \eq{ku3eq6}--\eq{ku3eq7}, these maps $i_X$ are smooth (in fact strongly smooth), but not interior.

Note that these maps $i_X$ {\it may not be injective}, since the preimage of $x\in X$ is $\depth_Xx$ points in $\pd X$ and $\binom{\depth_Xx}{k}$ points in $C_k(X)$. So we cannot regard $\pd X$ and $C_k(X)$ as subsets of $X$.

Since $\pd X$ is a manifold with corners, we can iterate the process, and define manifolds with corners $\pd^2X=\pd(\pd X),\pd^3X,\ldots,\pd^nX$. To relate these to the corners $C_k(X)$, note that by considering local models $U\subseteq\R^n_l$, it is easy to see that there is a natural 1-1 correspondence
\e
\begin{split}
&\bigl\{\text{local boundary components of $\pd X$ at $(x,\be)$}\bigr\}\cong\\
&\bigl\{\text{local boundary components $\be'$ of $X$ at $x$ with $\be'\ne\be$}\bigr\}.
\end{split}
\label{ku3eq9}
\e
Using this and induction, we can show that there is a natural identification
\e
\begin{split}
\pd^kX\cong\bigl\{(x,\be_1,\ldots,\be_k):\,&\text{$x\in X,$
$\be_1,\ldots,\be_k$ are distinct}\\
&\text{local boundary components for $X$ at $x$}\bigr\},
\end{split}
\label{ku3eq10}
\e
where under the identifications \eq{ku3eq10}, the map $i_{\pd^{k-1}X}:\pd^kX\ra\pd^{k-1}X$ maps $(x,\be_1,\ldots,\be_k)\mapsto(x,\be_1,\ldots,\be_{k-1})$. From \eq{ku3eq10}, we see that there is a natural, free action of the symmetric group $S_k$ on $\pd^kX$, by permutation of $\be_1,\ldots,\be_k$. The action is by diffeomorphisms, so the quotient $\pd^kX/S_k$ is also a manifold with corners. Dividing by $S_k$ turns the ordered $k$-tuple $\be_1,\ldots,\be_k$ into an unordered set $\{\be_1,\ldots,\be_k\}$. So from \eq{ku3eq5}, we see that there is a natural identification
\e
C_k(X)\cong \pd^kX/S_k,
\label{ku3eq11}
\e
which is a diffeomorphism.

We call $X$ a {\it manifold without boundary\/} if $\pd X=\es$, and
a {\it manifold with boundary\/} if $\pd^2X=\es$. We write $\Man$ and $\Manb$ for the full subcategories of $\Manc$ with objects manifolds without boundary, and manifolds with boundary, so that $\Man\subset\Manb\subset\Manc$. This definition of $\Man$ is equivalent to the usual definition of the category of manifolds.
\label{ku3def4}
\end{dfn}

Next we consider how smooth maps $f:X\ra Y$ of manifolds with corners act on boundaries $\pd X,\pd Y$ and corners $C_k(X),C_l(Y)$. The following lemma is easy to prove from Definition \ref{ku3def1}(b). The analogue is {\it false\/} for weakly smooth maps, so the rest of the section does not work in the weakly smooth case.

\begin{lem} Let\/ $f:X\ra Y$ be a smooth map of manifolds with corners. Then $f$ \begin{bfseries}is compatible with the depth stratifications\end{bfseries} $X=\coprod_{k\ge 0}S^k(X),$ $Y=\coprod_{l\ge 0}S^l(Y)$ in Definition\/ {\rm\ref{ku3def3},} in the sense that if\/ $\es\ne W\subseteq S^k(X)$ is a connected subset for some $k\ge 0,$ then $f(W)\subseteq S^l(Y)$ for some unique $l\ge 0$.
\label{ku3lem1}
\end{lem}

It is {\it not\/} true that general smooth $f:X\ra Y$ induce maps $\pd f:\pd X\ra\pd Y$ or $C_k(f):C_k(Y)\ra C_k(Y)$ (although this does hold for {\it simple\/} maps, as in Proposition \ref{ku3prop1}). For example, if $f:X\ra Y$ is the inclusion $[0,\iy)\hookra\R$ then no map $\pd f:\pd X\ra\pd Y$ exists, as $\pd X\ne\es$ and $\pd Y=\es$. So boundaries and $k$-corners do not give functors on $\Manc$. However, if we work in the enlarged category $\cManc$ of Definition \ref{ku3def2} and consider the full corners $C(X)=\coprod_{k\ge 0}C_k(X)$ in \eq{ku3eq8}, we can define a functor.

\begin{dfn} Let $f:X\ra Y$ be a smooth map of manifolds with corners, and suppose $\ga$ is a local $k$-corner component of $X$ at $x\in X$. For each sufficiently small open neighbourhood $V$ of $x$ in $X$, $\ga$ gives a choice of connected component $W$ of $V\cap S^k(X)$ with $x\in\overline W$, so by Lemma \ref{ku3lem1} $f(W)\subseteq S^l(Y)$ for some $l\ge 0$. As $f$ is continuous, $f(W)$ is connected, and $f(x)\in\ov{f(W)}$. Thus there is a unique $l$-corner component $f_*(\ga)$ of $Y$ at $f(x)$, such that if $\ti V$ is a sufficiently small open neighbourhood of $f(x)$ in $Y$, then the connected component $\ti W$ of $\ti V\cap S^l(Y)$ given by $f_*(\ga)$ has $\ti W\cap f(W)\ne\es$. This $f_*(\ga)$ is independent of the choice of sufficiently small $V,\ti V$, so is well-defined.

Define a map $C(f):C(X)\ra C(Y)$ by $C(f):(x,\ga)\mapsto (f(x),f_*(\ga))$. Given charts $(U,\phi)$ on $X$ and $(V,\psi)$ on $Y$, so that \eq{ku3eq1} gives a smooth map $\psi^{-1}\ci f\ci\phi$, then in the charts $(U_{\{a_1,\ldots,a_k\}},\phi_{\{a_1,\ldots,a_k\}})$ on $C_k(X)$ and $(V_{\{b_1,\ldots,b_l\}},\psi_{\{b_1,\ldots,b_l\}})$ on $C_l(Y)$ defined from $(U,\phi)$ and $(V,\psi)$ in \eq{ku3eq7}, we see that 
\begin{align*}
\psi_{\{b_1,\ldots,b_l\}}^{-1}\ci C(f)\ci\phi_{\{a_1,\ldots,a_k\}}:(f\ci\phi_{\{a_1,\ldots,a_k\}})^{-1}(\psi_{\{b_1,\ldots,b_l\}}(V_{\{b_1,\ldots,b_l\}}))&\\
\longra V_{\{b_1,\ldots,b_l\}}&
\end{align*}
is just the restriction of $\psi^{-1}\ci f\ci\phi$ to a map from a codimension $k$ boundary face of $U$ to a codimension $l$ boundary face of $V$, and so is clearly smooth in the sense of Definition \ref{ku3def1}. Since such charts $(U_{\{a_1,\ldots,a_k\}},\phi_{\{a_1,\ldots,a_k\}})$ and $(V_{\{b_1,\ldots,b_l\}},\psi_{\{b_1,\ldots,b_l\}})$ cover $C_k(X)$ and $C_l(Y)$, it follows that $C(f)$ is smooth (that is, $C(f)$ is a morphism in $\cManc$).

If $g:Y\ra Z$ is another smooth map of manifolds with corners, and $\ga$ is a local $k$-corner component of $X$ at $x$, it is easy to see that $(g\ci f)_*(\ga)=g_*\ci f_*(\ga)$ in local $m$-corner components of $Z$ at $g\ci f(x)$. Therefore $C(g\ci f)=C(g)\ci C(f):C(X)\ra C(Z)$. Clearly $C(\id_X)=\id_{C(X)}:C(X)\ra C(X)$. Hence $C:\Manc\ra\cManc$ is a functor, which we call the {\it corner functor}.
\label{ku3def5}
\end{dfn}

Here is a property of simple maps acting on local $k$-corner components:

\begin{lem} Let\/ $f:X\ra Y$ be a simple map of manifolds with corners, and\/ $x\in X$ with\/ $f(x)=y\in Y$. Then for each $k\ge 0,$ mapping $\ga\mapsto\ga'=f_*(\ga)$ gives a $1$-$1$ correspondence between local\/ $k$-corner components $\ga$ of\/ $X$ at\/ $x$ and local\/ $k$-corner components $\ga'$ of\/ $Y$ at\/ $y,$ for the same $k$.
\label{ku3lem2}
\end{lem}

The following properties of the corner functor are easy to check using the local models in Definition~\ref{ku3def1}.

\begin{prop} Let\/ $f:X\ra Y$ be a smooth map of manifolds with corners. 
\begin{itemize}
\setlength{\itemsep}{0pt}
\setlength{\parsep}{0pt}
\item[{\bf(a)}] $C(f):C(X)\ra C(Y)$ is an interior map of manifolds with corners of mixed dimension, so $C$ is a functor $C:\Manc\ra\cMancin$.
\item[{\bf(b)}] $f$ is interior if and only if\/ $C(f)$ maps $C_0(X)\ra C_0(Y)$.
\item[{\bf(c)}] $f$ is b-normal if and only if\/ $C(f)$ maps $C_k(X)\ra \coprod_{l=0}^kC_l(Y)$ for all\/~$k$.
\item[{\bf(d)}] If\/ $f$ is simple then $C(f)$ maps $C_k(X)\ra C_k(Y)$ for all\/ $k\ge 0,$ and\/ $C_k(f):=C(f)\vert_{C_k(X)}:C_k(X)\ra C_k(Y)$ is also a simple map.
\end{itemize}

Part\/ {\bf(d)} implies that, writing $\Mancsi$ for the category with objects manifolds with corners and morphisms simple maps, we have a \begin{bfseries}boundary functor\end{bfseries} $\pd:\Mancsi\ra\Mancsi$ mapping $X\mapsto\pd X$ on objects and\/ $f\mapsto \pd f:=C(f)\vert_{C_1(X)}:\pd X\ra\pd Y$ on (simple) morphisms $f:X\ra Y,$ and for all\/ $k\ge 0$ a \begin{bfseries}$k$-corner functor\end{bfseries} $C_k:\Mancsi\ra\Mancsi$ mapping $X\mapsto C_k(X)$ on objects and\/ $f\mapsto C_k(f):=C(f)\vert_{C_k(X)}:C_k(X)\ra C_k(Y)$ on morphisms.
\label{ku3prop1}
\end{prop}

\begin{rem} We can think of a smooth map $f:X\ra Y$ of manifolds with corners as comprising both discrete and continuous data. Here by {\it discrete data\/} we mean the following. Let $F:(-\ep,\ep)\t X\ra Y$ be a smooth map for $\ep>0$, and define $f_t:X\ra Y$ by $f_t(x)=F(t,x)$ for $t\in(-\ep,\ep)$ and $x\in X$, so that $f_t$ is smooth and depends smoothly on $t\in(-\ep,\ep)$. Then discrete data is properties of smooth maps $f:X\ra Y$ which is constant in all such families $f_t$, $t\in(-\ep,\ep)$. 

For example, each connected component of $S^k(X)$ is mapped by $f$ to some connected component of some $S^l(Y)$, and these image components in $S^l(Y)$ are discrete data. Also, choosing charts $(U,\phi)$, $(V,\psi)$ on $X,Y$, which of Definition \ref{ku3def1}(b)(i) or (ii) hold for $\psi^{-1}\ci f\ci\phi$ in \eq{ku3eq1}, and the values of $a_{i,j}$ in Definition \ref{ku3def1}(b)(i), are discrete data.

Note that discrete data need not be constant in {\it continuous\/} families $f_t$, $t\in(-\ep,\ep)$ of smooth maps $f_t:X\ra Y$, but only those where $F:(-\ep,\ep)\t X\ra Y$ mapping $F:(t,x)\mapsto f_t(x)$ is smooth. For example, if $X=*$, $Y=[0,\iy)$ and $f_t:*\mapsto t^2$, then $f_0$ maps $S^0(X)\ra S^1(Y)$, but $f_t$ for $t\ne 0$ maps $S^0(X)\ra S^0(Y)$, so $f_0$ and $f_t$ for $t\ne 0$ have different discrete data. This is allowed because $F:(-\ep,\ep)\ra[0,\iy)$, $F:t\mapsto t^2$ is not smooth in the sense of~\S\ref{ku31}.

Prescribing discrete data is both an open and closed condition in the space of all smooth maps $f:X\ra Y$, in a suitable topology. That $f$ be interior, or b-normal, or strongly smooth, or simple, in \S\ref{ku31}, are all discrete conditions on $f$ (though being a diffeomorphism is not). Discrete data is essential in defining the corner functor $C:\Manc\ra\cManc$ above, and similar ideas.

Discrete data for $f:X\ra Y$ is constant along any connected component of $S^k(X)$. So, for example, if $f$ is interior/b-normal/strongly smooth/simple near $x\in S^k(X)\subseteq X$, then $f$ is also interior/b-normal/strongly smooth/simple near $x'$ in $X$ for any $x'$ in the same connected component of $S^k(X)$ as~$x$.

Weakly smooth maps do not contain any nontrivial local discrete data.
\label{ku3rem2}
\end{rem}

\subsection{Tangent bundles and b-tangent bundles}
\label{ku33}

Manifolds with corners $X$ have two notions of tangent bundle with
functorial properties, the ({\it ordinary\/}) {\it tangent bundle\/}
$TX$, the obvious generalization of tangent bundles of manifolds
without boundary, and the {\it b-tangent bundle\/} ${}^bTX$
introduced by Melrose \cite[\S 2.2]{Melr2}, \cite[\S I.10]{Melr3},
\cite[\S 2]{Melr1}. Taking duals gives two notions of cotangent
bundle $T^*X,{}^bT^*X$. 

\begin{dfn} Let $X$ be an $m$-manifold with corners. The {\it tangent
bundle\/} $\pi:TX\ra X$ of $X$ is a natural (unique up to canonical
isomorphism) rank $m$ vector bundle on $X$. Here are two equivalent ways to characterize~$TX$:
\begin{itemize}
\setlength{\itemsep}{0pt}
\setlength{\parsep}{0pt}
\item[(a)] {\bf In coordinate charts:} let $(U,\phi)$ be a chart 
on $X$, with $U\subseteq\R^m_k$ open. Then over $\phi(U)$, $TX$ is the trivial vector bundle with basis of sections $\frac{\pd}{\pd u_1},\ldots,\frac{\pd}{\pd u_m}$, for $(u_1,\ldots,u_m)$ the coordinates on $U$. There is a corresponding chart $(TU,T\phi)$ on $TX$, where $TU=U\t\R^m\subseteq\R^{2m}_k$, such that $(u_1,\ldots,u_m,q_1,\ldots,q_m)\in TU$ represents the vector $q_1\frac{\pd}{\pd u_1}+\cdots+q_m\frac{\pd}{\pd u_m}$ over $(u_1,\ldots,u_m)\in U$ or $\phi(u_1,\ldots,u_m)\in X$. Under change of coordinates $(u_1,\ldots,u_m)\rightsquigarrow(\ti u_1,\ldots,\ti u_m)$ from $(U,\phi)$ to $(\ti U,\ti\phi)$, the corresponding change $(u_1,\ldots,u_m,q_1,\ldots,q_m)\rightsquigarrow(\ti u_1,\ldots,\ti u_m,\ti q_1,\ldots,\ti q_m)$ from $(TU,T\phi)$ to $(T\ti U,T\ti\phi)$ is determined by $\frac{\pd}{\pd u_i}=\sum_{j=1}^m\frac{\pd\ti u_j}{\pd u_i}(u_1,\ldots,u_m)\cdot\frac{\pd}{\pd\ti u_j}$, so that~$\ti q_j=\sum_{i=1}^m\frac{\pd\ti u_j}{\pd u_i}(u_1,\ldots,u_m)q_i$.
\item[(b)] {\bf Intrinsically:} for each $x\in X$, there is a
natural isomorphism
\begin{align*}
T_xX\cong\bigl\{v:\,&\text{$v$ is a linear map $C^\iy(X)\ra\R$
satisfying}\\
&\text{$v(fg)=v(f)g(x)+f(x)v(g)$ for all $f,g\in
C^\iy(X)$}\bigr\},
\end{align*}
where $C^\iy(X)$ is the $\R$-algebra of smooth functions $f:X\ra\R$. Also there is a natural isomorphism of $C^\iy(X)$-modules
\begin{align*}
\Ga(TX)\cong\bigl\{v:\,&\text{$v$ is a linear map $C^\iy(X)\ra
C^\iy(X)$
satisfying}\\
&\text{$v(fg)=v(f)\cdot g+f\cdot v(g)$ for all $f,g\in
C^\iy(X)$}\bigr\}.
\end{align*}
Elements of $\Ga(TX)$ are called {\it vector fields}.
\end{itemize}

Now suppose $f:X\ra Y$ is a weakly smooth map of manifolds with corners. We
will define a natural weakly smooth map $Tf:TX\ra TY$, which is smooth if $f$ is smooth, such that the following commutes:
\begin{equation*}
\xymatrix@C=80pt@R=15pt{ *+[r]{TX} \ar[d]^\pi \ar[r]_{Tf} &
*+[l]{TY} \ar[d]_\pi \\ *+[r]{X} \ar[r]^f & *+[l]{Y.} }
\end{equation*}
Let $(U,\phi)$ and $(V,\psi)$ be coordinate charts on $X,Y$ with $U\subseteq\R^m_k$, $V\subseteq\R^n_l$, with coordinates $(u_1,\ldots,u_m)\in U$ and $(v_1,\ldots,v_n)\in V$, and let $(TU,T\phi)$, $(TV,T\psi)$ be the corresponding charts on $TX,TY$, with coordinates $(u_1,\ab\ldots,\ab u_m,\ab q_1,\ab\ldots,\ab q_m)\in TU$ and $(v_1,\ldots,v_n,r_1,\ldots,r_n)\in TV$. Then \eq{ku3eq1} defines a map $\psi^{-1}\ci f\ci\phi$ between open subsets of $U,V$. Write $\psi^{-1}\ci f\ci\phi=(f_1,\ldots,f_n)$, for $f_j=f_j(u_1,\ldots,u_m)$. Then the corresponding $T\psi^{-1}\ci Tf\ci T\phi$ maps
\begin{align*}
&T\psi^{-1}\ci Tf\ci T\phi:(u_1,\ldots,u_m,q_1,\ldots,q_m)\longmapsto
\bigl(f_1(u_1,\ldots,u_m),\ldots,\\
&f_n(u_1,\ldots,u_m),\ts\sum_{i=1}^m\frac{\pd f_1}{\pd u_i}(u_1,\ldots,u_m)q_i,\ldots,\ts\sum_{i=1}^m\frac{\pd f_n}{\pd u_i}(u_1,\ldots,u_m)q_i\bigr).
\end{align*}

If $g:Y\ra Z$ is smooth then $T(g\ci f)=Tg\ci Tf:TX\ra
TZ$, and $T(\id_X)=\id_{TX}:TX\ra TX$. Thus, the assignment
$X\mapsto TX$, $f\mapsto Tf$ is a functor, the {\it tangent
functor\/} $T:\Manc\ra\Manc$. It restricts to $T:\Mancin\ra\Mancin$. We can also regard $Tf$ as a vector bundle morphism $\d f:TX\ra f^*(TY)$ on~$X$.

The {\it cotangent bundle\/} $T^*X$ of a manifold with corners $X$
is the dual vector bundle of $TX$. Cotangent bundles $T^*X$ are not
functorial in the same way, though we do have vector bundle morphisms $(\d f)^*:f^*(T^*Y)\ra T^*X$ on~$X$.
\label{ku3def6}
\end{dfn}

Here is the parallel definition for b-(co)tangent bundles:

\begin{dfn} Let $X$ be an $m$-manifold with corners. The {\it b-tangent
bundle\/} ${}^bTX\ra X$ of $X$ is a natural (unique up to canonical
isomorphism) rank $m$ vector bundle on $X$. It has a natural morphism $I_X:{}^bTX\ra TX$, which is an isomorphism over the interior $X^\ci$, but not over the boundary strata $S^k(X)$ for $k\ge 1$. Here are two equivalent ways to characterize~${}^bTX,I_X$:
\begin{itemize}
\setlength{\itemsep}{0pt}
\setlength{\parsep}{0pt}
\item[(a)] {\bf In coordinate charts:} let $(U,\phi)$ be a chart 
on $X$, with $U\subseteq\R^m_k$ open. Then over $\phi(U)$, ${}^bTX$ has basis of sections $u_1\frac{\pd}{\pd u_1},\ldots,u_k\frac{\pd}{\pd u_k},\frac{\pd}{\pd u_{k+1}},\ab\ldots,\ab
\frac{\pd}{\pd u_m}$, for $(u_1,\ldots,u_m)$ the coordinates on $U$. There is a corresponding chart $({}^bTU,{}^bT\phi)$ on ${}^bTX$, where ${}^bTU=U\t\R^m\subseteq\R^{2m}_k$, such that $(u_1,\ldots,\ab u_m,\ab s_1,\ab\ldots,s_m)\in {}^bTU$ represents the vector 
\begin{equation*}
\ts s_1u_1\frac{\pd}{\pd u_1}+\cdots+s_ku_k\frac{\pd}{\pd u_k}+s_{k+1}\frac{\pd}{\pd u_{k+1}}+\cdots+s_m\frac{\pd}{\pd u_m}
\end{equation*}
over $(u_1,\ldots,u_m)$ in $U$ or $\phi(u_1,\ldots,u_m)$ in $X$. Under change of coordinates $(u_1,\ldots,u_m)\rightsquigarrow(\ti u_1,\ldots,\ti u_m)$ from $(U,\phi)$ to $(\ti U,\ti\phi)$, the corresponding change $(u_1,\ab\ldots,\ab u_m,\ab s_1,\ab\ldots,\ab s_m)\ab\rightsquigarrow(\ti u_1,\ldots,\ab\ti u_m,\ab\ti s_1,\ab\ldots,\ti s_m)$ from $({}^bTU,{}^bT\phi)$ to $({}^bT\ti U,{}^bT\ti\phi)$ is 
\begin{equation*}
\ti s_j=\begin{cases} \sum_{i=1}^k\ti u_j^{-1}u_i\frac{\pd\ti u_j}{\pd u_i}\,s_i+\sum_{i=k+1}^m\ti u_j^{-1}\frac{\pd\ti u_j}{\pd u_i}\,s_i, & j\le k, \\[4pt]
\sum_{i=1}^ku_i\frac{\pd\ti u_j}{\pd u_i}\,s_i+\sum_{i=k+1}^m\frac{\pd\ti u_j}{\pd u_i}\,s_i, & j>k. \end{cases}
\end{equation*}

The morphism $I_X:{}^bTX\ra TX$ acts in coordinate charts $({}^bTU,{}^bT\phi)$, $(TU,T\phi)$ by
\begin{align*}
&(u_1,\ldots,u_m,s_1,\ldots,s_m)\longmapsto (u_1,\ldots,u_m,q_1,\ldots,q_m)\\
&\qquad\qquad =(u_1,\ldots,u_m,u_1s_1,\ldots,u_ks_k,s_{k+1},\ldots,s_m).
\end{align*}
\item[(b)] {\bf Intrinsically:} there is a natural isomorphism
of $C^\iy(X)$-modules
\e
\begin{split}
\Ga({}^bTX)\cong\bigl\{v\in \Ga(TX): \text{$v\vert_{S^k(X)}$ is
tangent to $S^k(X)$ for all $k$}\bigr\}.
\end{split}
\label{ku3eq12}
\e
Elements of $\Ga({}^bTX)$ are called {\it b-vector fields}.

The morphism $I_X:{}^bTX\ra TX$ induces
$(I_X)_*:\Ga({}^bTX)\ra\Ga(TX)$, which under the isomorphism
\eq{ku3eq12} corresponds to the inclusion of the right hand side
of \eq{ku3eq12} in~$\Ga(TX)$.
\end{itemize}

In Definition \ref{ku3def6}, we defined $Tf:TX\ra TY$ for any smooth
(or even weakly smooth) map $f:X\ra Y$. As in \cite[\S 2]{Melr1},
\cite[\S 1]{KoMe} the analogue for b-tangent bundles works only for
{\it interior\/} maps $f:X\ra Y$. So let $f:X\ra Y$ be an interior
map of manifolds with corners. We will define a natural interior map
${}^bTf:{}^bTX\ra{}^bTY$ so that the following commutes:
\begin{equation*}
\xymatrix@C=30pt@R=15pt{ {{}^bTX} \ar[ddr]_(0.6)\pi \ar[dr]^{I_X}
\ar[rrr]_{{}^bTf} &&&
{{}^bTY} \ar[dr]^{I_Y} \ar[ddr]_(0.6)\pi \\
& {TX} \ar[d]^\pi \ar[rrr]^{Tf} &&&
{TY} \ar[d]^\pi \\
& {X} \ar[rrr]^f &&& {Y.\!} }
\end{equation*}

Let $(U,\phi)$ and $(V,\psi)$ be coordinate charts on $X,Y$ with $U\subseteq\R^m_k$, $V\subseteq\R^n_l$, with coordinates $(u_1,\ldots,u_m)\in U$ and $(v_1,\ldots,v_n)\in V$, and let $({}^bTU,{}^bT\phi)$, $({}^bTV,{}^bT\psi)$ be the corresponding charts on $TX,TY$, with coordinates $(u_1,\ab\ldots,\ab u_m,\ab s_1,\ab\ldots,\ab s_m)\in TU$ and $(v_1,\ldots,v_n,t_1,\ldots,t_n)\in TV$. Then \eq{ku3eq1} defines a map $\psi^{-1}\ci f\ci\phi$ between open subsets of $U,V$. Write $\psi^{-1}\ci f\ci\phi=(f_1,\ldots,f_n)$, for $f_j=f_j(u_1,\ldots,u_m)$. Then the corresponding ${}^bT\psi^{-1}\ci {}^bTf\ci {}^bT\phi$ maps
\e
\begin{split}
&{}^bT\psi^{-1}\!\ci\! {}^bTf\!\ci\! {}^bT\phi:(u_1,\ldots,u_m,s_1,\ldots,s_m)\!\longmapsto\! (v_1,\ldots,v_n,t_1,\ldots,t_n),\!\!\!\!\!\!\!\!{}
\\
&\text{where}\quad v_j=f_j(u_1,\ldots,u_m),\quad j=1\ldots,n,\\
&\text{and}\quad t_j=
\begin{cases} \sum_{i=1}^kf_j^{-1}u_i\frac{\pd f_j}{\pd u_i}\,s_i+\sum_{i=k+1}^mf_j^{-1}\frac{\pd f_j}{\pd u_i}\,s_i, & j\le l, \\[4pt]
\sum_{i=1}^ku_i\frac{\pd f_j}{\pd u_i}\,s_i+\sum_{i=k+1}^m\frac{\pd f_j}{\pd u_i}\,s_i, & j>l. \end{cases}
\end{split}
\label{ku3eq13}
\e

Since $f$ is interior, the functions $f_j^{-1}u_i\frac{\pd f_j}{\pd u_i}$ for $i\le k$, $j\le l$ and $f_j^{-1}\frac{\pd f_j}{\pd u_i}$ for $i>k$, $j\le l$ occurring in \eq{ku3eq13} extend uniquely to
smooth functions of $(u_1,\ldots,u_m)$ where $f_j=0$, which by Definition \ref{ku3def1}(b)(i) is only where $u_i=0$ for certain $i=1,\ldots,k$. If $f$ is not interior, we could have $f_j(u_1,\ldots,u_m)=0$ for all $(u_1,\ldots,u_m)$, and then there are no natural values for $f_j^{-1}u_i\frac{\pd f_j}{\pd u_i}$, $f_j^{-1}\frac{\pd f_j}{\pd u_i}$ (just setting them zero is not functorial under change of coordinates), so we cannot define~${}^bTf$.

If $g:Y\ra Z$ is another interior map then ${}^bT(g\ci f)={}^bTg\ci
{}^bTf:{}^bTX\ra {}^bTZ$, and ${}^bT(\id_X)=\id_{{}^bTX}:{}^bTX\ra
{}^bTX$. Thus, writing $\Mancin$ for the subcategory of $\Manc$ with morphisms interior maps, the assignment $X\mapsto {}^bTX$, $f\mapsto {}^bTf$ is a functor, the {\it b-tangent functor\/} ${}^bT:\Mancin\ra\Mancin$. The maps $I_X:{}^bTX\ra TX$ give a natural transformation $I:{}^bT\ra T$ of functors on~$\Mancin$.

We can also regard ${}^bTf$ as a vector bundle morphism ${}^b\d
f:{}^bTX\ra f^*({}^bTY)$ on $X$. The {\it b-cotangent bundle\/} ${}^bT^*X$ of $X$ is the dual vector bundle of ${}^bTX$. B-cotangent bundles ${}^bT^*X$ are not functorial in the same way, though we do have vector bundle morphisms $({}^b\d f)^*:f^*({}^bT^*Y)\ra {}^bT^*X$ for interior~$f$.
\label{ku3def7}
\end{dfn}

\begin{dfn} The following notation is defined in \cite[\S 2.4]{Joyc9}. Let $X$ be an $m$-manifold with corners, and $x\in S^k(X)$. Then we can choose local coordinates $(x_1,\ldots,x_m)\in\R^m_k$ on $X$ near $x$, with $x=(0,\ldots,0)$, giving an isomorphism
\e
{}^bT_xX\cong \ts\bigl\langle x_1\frac{\pd}{\pd x_1},\ldots,x_k\frac{\pd}{\pd x_k},\frac{\pd}{\pd x_{k+1}},\ldots,\frac{\pd}{\pd x_m}\bigr\rangle.
\label{ku3eq14}
\e
Define $\ti M_xX\subseteq {}^bT_xX$ to be the subset of ${}^bT_xX$ identified by \eq{ku3eq14} with
\begin{equation*}
\ts\bigl\{a_1x_1\frac{\pd}{\pd x_1}+\cdots+a_kx_k\frac{\pd}{\pd x_k}:a_1,\ldots,a_k\in\N\bigr\}\cong\N^k.
\end{equation*}
Then $\ti M_xX$ is a toric monoid in the sense of \S\ref{ku341}, a submonoid of ${}^bT_xX$. If $f:X\ra Y$ is an interior map of manifolds with corners and $x\in X$ with $f(x)=y\in Y$, then ${}^bT_xf:{}^bT_xX\ra{}^bT_yY$ maps $\ti M_xX\ra\ti M_yY$, and we write $\ti M_xf:={}^bT_xf\vert_{\ti M_xX}:\ti M_xX\ra\ti M_yY$, a monoid morphism. Then $f$ is simple in the sense of \S\ref{ku31} if and only if $\ti M_xf$ is an isomorphism for all $x\in X$.
\label{ku3def8}
\end{dfn}

As emphasized by Melrose \cite{KoMe,Melr1,Melr2,Melr3}, we can state a meta-mathematical:

\begin{princ} When working with manifolds with corners, it is very often better to use b-tangent bundles ${}^bTX$ rather than tangent bundles $TX$.
\label{ku3princ}
\end{princ}

For example, if $X$ is a (compact, say) manifold without boundary then we have an infinite-dimensional Lie group $\mathop{\rm Diff}(X)$ of diffeomorphisms of $X$, whose Lie algebra (in some sense) is $C^\iy(TX)$. If $X$ has corners, then we should interpret the Lie algebra of $\mathop{\rm Diff}(X)$ as $C^\iy({}^bTX)$, since diffeomorphisms of $X$ preserve the stratification $X=\coprod_{k\ge 0}S^k(X)$, so infinitesimal diffeomorphisms should be vector fields tangent to each stratum $S^k(X)$, as in \eq{ku3eq12}. We cannot interpret $C^\iy(TX)$ as the Lie algebra of a Lie group in a meaningful way. 

\subsection{Manifolds with generalized corners}
\label{ku34}

The author \cite{Joyc9} studied `manifolds with generalized corners', or `manifolds with g-corners' for short, a generalization of manifolds with corners. We give a brief outline, for more details see~\cite{Joyc9}.

\subsubsection{The definition of manifolds with g-corners}
\label{ku341}

Manifolds with corners are locally modelled upon $\R^m_k=[0,\iy)^k\t\R^{m-k}$. Manifolds with g-corners have a more general set of local model spaces $X_P$ for $P$ a weakly toric monoid, as in \cite[\S 3.1--\S 3.2]{Joyc9}, which we now define.

\begin{dfn} A ({\it commutative\/}) {\it monoid\/} $(P,+,0)$ is a set $P$ with a commutative, associative operation $+:P\t P\ra P$ and an identity element $0\in P$. Monoids are like abelian groups, but without inverses. They form a category $\Mon$. Some examples of monoids are the natural numbers $\N=\{0,1,2,\ldots\}$, the integers $\Z$, any abelian group $G$, and $[0,\iy)=\bigl([0,\iy),\cdot,1\bigr)$.

A monoid $P$ is called {\it weakly toric\/} if for some $m,k\ge 0$ and $c_i^j\in\Z$ for $i=1,\ldots,m$, $j=1,\ldots,k$ we have
\e
P\cong P'=\bigl\{(l_1,\ldots,l_m)\in\Z^m:c_1^jl_1+\cdots+c_m^jl_m\ge 0,\;\> j=1,\ldots,k\bigr\}.
\label{ku3eq15}
\e
The {\it rank\/} of a weakly toric monoid $P$ is $\rank P=\dim_\R(P\ot_\N\R)$. 

For example, $\N^k\t\Z^{m-k}$ is a weakly toric monoid of rank $m$ for $0\le k\le m$, as it is the submonoid of $(l_1,\ldots,l_m)$ in $\Z^m$ satisfying $l_j\ge 0$ for~$j=0,\ldots,k$.

A weakly toric monoid $P$ is called {\it toric\/} if $P'$ in \eq{ku3eq15} has $P'\cap -P'=\{0\}$, where the intersection is in $\Z^m$. For example, $\N^k$ is a toric monoid for~$k\ge 0$.

Let $P$ be a weakly toric monoid. Define $X_P$ to be the set of monoid morphisms $x:P\ra[0,\iy)$, where $\bigl([0,\iy),\cdot,1\bigr)$ is the monoid $[0,\iy)$ with operation multiplication and identity 1. Define the {\it interior\/} $X_P^\ci\subset X_P$ of $X_P$ to be the subset of $x$ with~$x(P)\subseteq(0,\iy)\subset[0,\iy)$.

For each $p\in P$, define a function $\la_p:X_P\ra[0,\iy)$ by $\la_p(x)=x(p)$. Then $\la_{p+q}=\la_p\cdot\la_q$ for $p,q\in P$, and~$\la_0=1$.

Define a topology on $X_P$ to be the weakest topology such that $\la_p:X_P\ra[0,\iy)$ is continuous for all $p\in P$. This makes $X_P$ into a locally compact, Hausdorff topological space, and $X_P^\ci$ is open in $X_P$. If $U\subseteq X_P$ is an open set, define the {\it interior\/} $U^\ci$ of $U$ to be~$U^\ci=U\cap X_P^\ci$.

Choose generators $p_1,\ldots,p_m$ for $P,$ and a generating set of relations for $p_1,\ldots,p_m$ of the form
\begin{equation*}
a_1^jp_1+\cdots+a_m^jp_m=b_1^jp_1+\cdots+b_m^jp_m\quad\text{in $P$ for $j=1,\ldots,k,$}
\end{equation*}
where $a_i^j,b_i^j\in\N$ for $i=1,\ldots,m$ and $j=1,\ldots,k$. Then $\la_{p_1}\t\cdots\t\la_{p_m}:X_P\ra[0,\iy)^m$ is a homeomorphism from $X_P$ to
\begin{equation*}
X_P'=\bigl\{(x_1,\ldots,x_m)\in[0,\iy)^m:x_1^{a_1^j}\cdots x_m^{a_m^j}=x_1^{b_1^j}\cdots x_m^{b_m^j},\; j=1,\ldots,k\bigr\},
\end{equation*}
regarding $X_P'$ as a closed subset of\/ $[0,\iy)^m$ with the induced topology.

Let $U\subseteq X_P$ be open, and $U'=(\la_{p_1}\t\cdots\t\la_{p_m})(U)$ be the corresponding open subset of $X_P'$. We say that a continuous function $f:U\ra\R$ or $f:U\ra[0,\iy)$ is {\it smooth\/} if there exists an open neighbourhood $W$ of $U'$ in $[0,\iy)^m$ and a smooth function $g:W\ra\R$ or $g:W\ra[0,\iy)$ in the sense of manifolds with (ordinary) corners, such that $f=g\ci(\la_{p_1}\t\cdots\t\la_{p_m})$. This definition turns out to be independent of the choice of generators~$p_1,\ldots,p_m$.

Now let $Q$ be another weakly toric monoid, $V\subseteq X_Q$ be open, and $f:U\ra V$ be continuous. We say that $f$ is {\it smooth\/} if $\la_q\ci f:U\ra[0,\iy)$ is smooth in the sense above for all $q\in Q$. We call a smooth map $f:U\ra V$ {\it interior\/} if $f(U^\ci)\subseteq V^\ci$, and a {\it diffeomorphism\/} if $f$ has a smooth inverse~$f^{-1}:V\ra U$.
\label{ku3def9}
\end{dfn}

\begin{ex} Let $P$ be the weakly toric monoid $\N^k\t\Z^{m-k}$ for $0\le k\le m$. Define a map $\Phi:\R^m_k\ra X_P$ by $\Phi(x_1,\ldots,x_m)=\rho$, where $\rho:(l_1,\ldots,l_m)=x_1^{l_1}\cdots x_k^{l_k}e^{l_{k+1}x_{k+1}+\cdots+l_mx_m}$ for all $(l_1,\ldots,l_m)\in P$. It is easy to show \cite[\S 3.2]{Joyc9} that $\Phi$ is a bijection, and a homeomorphism of topological spaces. Furthermore, under this identification $\R^m_k\cong X_P$, the notions of smooth maps $f:U\ra V$ for open $U\subseteq\R^m_k$, $V\subseteq\R^n_l$ in Definition \ref{ku3def1}, and smooth maps $f:U\ra V$ for open $U\subseteq X_P$, $V\subseteq X_Q$ with $P=\N^k\t\Z^{m-k}$, $Q=\N^l\t\Z^{n-l}$ in Definition \ref{ku3def9}, agree. Thus the spaces $X_P$ above generalize the $\R^m_k=[0,\iy)^k\t\R^{m-k}$ in~\S\ref{ku31}.
\label{ku3ex1}
\end{ex}

\begin{dfn} Let $X$ be a second countable Hausdorff topological space. An {\it $m$-dimensional generalized chart}, or {\it g-chart}, on $X$ is a triple $(P,U,\phi)$, where $P$ is a weakly toric monoid with $\rank P=m$, and $P$ is a submonoid of $\Z^n$ for some $n\ge 0$, and $U\subseteq X_P$ is open, for $X_P$ as above, and $\phi:U\ra X$ is a homeomorphism with an open set $\phi(U)$ in~$X$.

Let $(P,U,\phi),(Q,V,\psi)$ be $m$-dimensional g-charts on $X$. We call
$(P,U,\phi)$ and $(Q,V,\psi)$ {\it compatible\/} if
$\psi^{-1}\ci\phi:\phi^{-1}\bigl(\phi(U)\cap\psi(V)\bigr)\ra
\psi^{-1}\bigl(\phi(U)\cap\psi(V)\bigr)$ is a diffeomorphism between
open subsets of $X_P,X_Q$, in the sense of Definition~\ref{ku3def9}.

An $m$-{\it dimensional generalized atlas}, or {\it g-atlas}, on $X$ is a family $\{(P^i,U^i,\phi^i)\!:i\!\in\! I\}$ of pairwise compatible $m$-dimensional g-charts on $X$ with $X\!=\!\bigcup_{i\in I}\phi^i(U^i)$. We call such a g-atlas {\it maximal\/} if it is not a proper subset of any other g-atlas. Any g-atlas $\{(P^i,U^i,\phi^i):i\in I\}$ is contained in a unique maximal g-atlas, the family of all g-charts $(P,U,\phi)$ on $X$ compatible with $(P^i,U^i,\phi^i)$ for all~$i\in I$.

An $m$-{\it dimensional manifold with generalized corners}, or {\it g-corners}, is a second countable Hausdorff topological space $X$ with a maximal $m$-dimensional g-atlas. Usually we refer to $X$ as the manifold, leaving the g-atlas implicit. By a {\it g-chart\/ $(P,U,\phi)$ on\/} $X$, we mean an element of the maximal g-atlas. Write~$\dim X=m$.

Let $X,Y$ be manifolds with g-corners, and $f:X\ra Y$ a continuous map of the underlying topological spaces. We say that $f:X\ra Y$ is {\it smooth\/} (or {\it interior\/}) if for all g-charts $(P,U,\phi)$ on $X$ and $(Q,V,\psi)$ on $Y$, the map
\begin{equation*}
\psi^{-1}\ci f\ci\phi: (f\ci\phi)^{-1}(\psi(V))\longra V
\end{equation*}
is a smooth (or interior) map between the open subsets $(f\ci\phi)^{-1}(\psi(V))\subseteq U\subseteq X_P$ and $V\subseteq X_Q$, in the sense of Definition \ref{ku3def10}. Then, manifolds with g-corners and smooth maps form a category $\Mangc$, and manifolds with g-corners and interior maps form a subcategory~$\Mangcin\subset\Mangc$.

As for $\cManc$ in \S\ref{ku31}, write $\cMangc$ for the category whose objects are disjoint unions $\coprod_{m=0}^\iy X_m$, where $X_m$ is a manifold with g-corners of dimension $m$, and whose morphisms are continuous maps $f:\coprod_{m=0}^\iy X_m\ra\coprod_{n=0}^\iy Y_n$, such that $f\vert_{X_m\cap f^{-1}(Y_n)}:X_m\cap f^{-1}(Y_n)\ra Y_n$ is smooth for all $m,n\ge 0$, and write $\cMangcin\subset\cMangc$ for the subcategory with the same objects, and morphisms $f$ with $f\vert_{X_m\cap f^{-1}(Y_n)}$ interior for all $m,n\ge 0$.

Now let $X$ be an $m$-manifold with (ordinary) corners, in the sense of \S\ref{ku31}, and $(U,\phi)$ be a chart on $X$ with $U\subseteq\R^m_k$ open. As in Example \ref{ku3ex1} we identify $\R^m_k\cong X_P$ for $P=\N^k\t\Z^{m-k}$, and then $(\N^k\t\Z^{m-k},U,\phi)$ is a g-chart on $X$. As in \cite[\S 3.3]{Joyc9}, we can uniquely make $X$ into a manifold with g-corners, which we temporarily write as $\hat X$, such that if $(U,\phi)$ is a chart on $X$ with $U\subseteq\R^m_k$ then $(\N^k\t\Z^{m-k},U,\phi)$ is a g-chart on~$\hat X$. 

If $Y$ is another manifold with corners and $\hat Y$ the corresponding manifold with g-corners, then $f:X\ra Y$ is smooth (or interior) in $\Manc$ if and only if $f:\hat X\ra\hat Y$ is smooth (or interior) in $\Mangc$. Thus, identifying $X$ with $\hat X$ we can regard manifolds with corners as examples of manifolds with g-corners, and $\Manc\subset\Mangc$, $\cManc\subset\cMangc$ as full subcategories. 
\label{ku3def10}
\end{dfn}

\begin{ex} Here is the simplest example of a manifold with g-corners which is not a manifold with corners, taken from \cite[Ex.~3.24]{Joyc9}. Define 
\begin{equation*}
P=\bigl\{(a,b,c)\in \Z^3:a\ge 0,\; b\ge 0,\; a+b\ge c\ge 0\bigr\}.
\end{equation*}
Then $P$ is a weakly toric monoid with rank 3. Write $p_1=(1,0,0)$, $p_2=(0,1,1)$, $p_3=(0,1,0)$, and $p_4=(1,0,1)$. Then $p_1,p_2,p_3,p_4$ are generators for $P$, subject to the single relation $p_1+p_2=p_3+p_4$. As in Definition \ref{ku3def9} we have
\e
X_P\cong X_P'=\bigl\{(x_1,x_2,x_3,x_4)\in[0,\iy)^4:x_1x_2=x_3x_4\bigr\}.
\label{ku3eq16}
\e

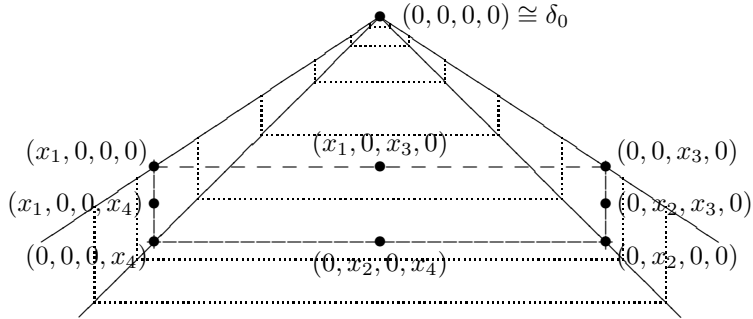
\begin{figure}[htb]
\centerline{$\splinetolerance{.8pt}
\begin{xy}
0;<1mm,0mm>:
,(0,0)*{\bu}
,(14,0)*{(0,0,0,0)\cong\de_0}
,(-30,-20)*{\bu}
,(-39,-18)*{(x_1,0,0,0)}
,(30,-20)*{\bu}
,(39.5,-18)*{(0,0,x_3,0)}
,(-30,-30)*{\bu}
,(-39,-32)*{(0,0,0,x_4)}
,(30,-30)*{\bu}
,(39.5,-32)*{(0,x_2,0,0)}
,(0,-30)*{\bu}
,(0,-33)*{(0,x_2,0,x_4)}
,(0,-20)*{\bu}
,(0,-17)*{(x_1,0,x_3,0)}
,(-30,-25)*{\bu}
,(-40.5,-25)*{(x_1,0,0,x_4)}
,(30,-25)*{\bu}
,(40.5,-25)*{(0,x_2,x_3,0)}
,(0,0);(-45,-30)**\crv{}
?(.8444)="aaa"
?(.85)="bbb"
?(.75)="ccc"
?(.65)="ddd"
?(.55)="eee"
?(.45)="fff"
?(.35)="ggg"
?(.25)="hhh"
?(.15)="iii"
?(.05)="jjj"
,(0,0);(45,-30)**\crv{}
?(.8444)="aaaa"
?(.85)="bbbb"
?(.75)="cccc"
?(.65)="dddd"
?(.55)="eeee"
?(.45)="ffff"
?(.35)="gggg"
?(.25)="hhhh"
?(.15)="iiii"
?(.05)="jjjj"
,(0,0);(-40,-40)**\crv{}
?(.95)="a"
?(.85)="b"
?(.75)="c"
?(.65)="d"
?(.55)="e"
?(.45)="f"
?(.35)="g"
?(.25)="h"
?(.15)="i"
?(.05)="j"
,(0,0);(40,-40)**\crv{}
?(.95)="aa"
?(.85)="bb"
?(.75)="cc"
?(.65)="dd"
?(.55)="ee"
?(.45)="ff"
?(.35)="gg"
?(.25)="hh"
?(.15)="ii"
?(.05)="jj"
,"a";"aa"**@{.}
,"b";"bb"**@{.}
,"c";"cc"**@{.}
,"d";"dd"**@{.}
,"e";"ee"**@{.}
,"f";"ff"**@{.}
,"g";"gg"**@{.}
,"h";"hh"**@{.}
,"i";"ii"**@{.}
,"j";"jj"**@{.}
,"a";"aaa"**@{.}
,"b";"bbb"**@{.}
,"c";"ccc"**@{.}
,"d";"ddd"**@{.}
,"e";"eee"**@{.}
,"f";"fff"**@{.}
,"g";"ggg"**@{.}
,"h";"hhh"**@{.}
,"i";"iii"**@{.}
,"j";"jjj"**@{.}
,"aa";"aaaa"**@{.}
,"bb";"bbbb"**@{.}
,"cc";"cccc"**@{.}
,"dd";"dddd"**@{.}
,"ee";"eeee"**@{.}
,"ff";"ffff"**@{.}
,"gg";"gggg"**@{.}
,"hh";"hhhh"**@{.}
,"ii";"iiii"**@{.}
,"jj";"jjjj"**@{.}
,(-30,-20);(30,-20)**@{--}
,(-30,-20);(-30,-30)**\crv{}
,(-30,-30);(30,-30)**\crv{}
,(30,-30);(30,-20)**\crv{}
\end{xy}$}
\caption{3-manifold with g-corners $X_P'\cong X_P$ in \eq{ku3eq16}}
\label{ku3fig1}
\end{figure}

We sketch $X_P'$ in Figure \ref{ku3fig1}. We can visualize $X_P\cong X_P'$ as a 3-dimensional infinite pyramid on a square base. Using the ideas of \S\ref{ku32}, $X_P'$ has one vertex $(0,0,0,0)$ corresponding to $\de_0\in X_P$ mapping $\de_0:P\ra\bigl([0,\iy),\cdot\bigr)$ with $\de_0(p)=1$ if $p=(0,0,0)$ and $\de_0(p)=0$ otherwise, four 1-dimensional edges of points $(x_1,0,0,0),(0,x_2,0,0), (0,0,x_3,0),(0,0,0,x_4)$, four 2-dimensional faces of points $(x_1,0,x_3,0)$, $(x_1,0,0,x_4)$, $(0,x_2,x_3,0)$, $(0,x_2,0,x_4)$, and an interior $X_P^{\prime\ci}\ab\cong\R^3$ of points $(x_1,x_2,x_3,x_4)$. Then $X_P\sm\{\de_0\}$ is a 3-manifold with corners, but $X_P$ is not a manifold with corners near $\de_0$, as we can see from the non-simplicial face structure.
\label{ku3ex2}
\end{ex}

\begin{rem} Every weakly toric monoid $P$ satisfies $P\cong Q\t\Z^l$ for a unique toric monoid $Q$, and then $X_P\cong X_Q\t X_{\Z^l}\cong X_Q\t\R^l$. Thus, an alternative way to describe the local models for manifolds with g-corners is $X_Q\t\R^l$ for toric monoids $Q$ and $l\ge 0$, where $X_{\N^k}\cong[0,\iy)^k$.

If $Q$ is a toric monoid there is a natural point $\de_0\in X_Q$, the {\it vertex\/} of $X_Q$, acting by $\de_0:Q\ra[0,\iy)$ with $\de_0(q)=1$ if $q=0$ and $\de_0(q)=0$ if $q\in Q\sm\{0\}$.

If $X$ is a manifold with g-corners and $x\in X$ then there is a toric monoid $Q$, unique up to isomorphism, such that $X$ near $x$ is locally modelled on $X_Q\t\R^l$ near $(\de_0,0)$, where~$\rank Q+l=\dim X$.

\label{ku3rem3}
\end{rem}

\subsubsection{Boundaries and corners of manifolds with g-corners}
\label{ku342}

As in \cite[\S 3.4]{Joyc9}, almost all of \S\ref{ku32} extends to manifolds with g-corners: the depth stratification $X=\coprod_{l=0}^{\dim X}S^l(X)$, where $S^l(X)$ is a manifold without boundary of dimension $\dim X-l$ with $S^0(X)=X^\ci$, local $k$-corner components $\ga$, the boundary $\pd X$ and $k$-corners $C_k(X)$ as manifolds with g-corners, the corners $C(X)=\coprod_{k=0}^{\dim X}C_k(X)$ in $\cMangc$, the corner functor $C:\Mangc\ra\cMangc$, and Lemmas \ref{ku3lem1}, \ref{ku3lem2} and Proposition \ref{ku3prop1} (for simple maps in $\Mangc$ defined as in \S\ref{ku343}) all extend to manifolds with g-corners in a straightforward way.

The main difference is that equations \eq{ku3eq5} and \eq{ku3eq9}--\eq{ku3eq11} are false for general manifolds with g-corners $X$. For $k=2$, we replace \eq{ku3eq5} and \eq{ku3eq10} by
\begin{align*}
C_2(X)\cong\bigl\{&(x,\{\be_1,\be_2\}):\text{$x\!\in\! X,$
$\be_1,\be_2$ are distinct local boundary}\\
&\text{components of $X$ at $x$ intersecting in codimension $2$}\bigr\},\\
\pd^2X\cong\bigl\{&(x,\be_1,\be_2):\text{$x\!\in\! X,$
$\be_1,\be_2$ are distinct local boundary}\\
&\text{components of $X$ at $x$ intersecting in codimension $2$}\bigr\},
\end{align*}
and $S_2\cong\Z_2$ acts freely on $\pd^2X$ exchanging $\be_1,\be_2$, with $C_2(X)\cong\pd^2X/S_2$ as in \eq{ku3eq11}. For $k>2$, $S_k$ may not act on $\pd^kX$, so \eq{ku3eq11} does not make sense. There is a natural, surjective, \'etale map $\pd^kX\ra C_k(X)$ for all~$k$.

\subsubsection{B-tangent bundles of manifolds with g-corners}
\label{ku343}

In \cite[\S 3.5]{Joyc9} we discuss how to extend \S\ref{ku33} to manifolds with g-corners.

One big difference between manifolds with corners and with g-corners, is that {\it manifolds with g-corners\/ $X$ do not have well-behaved tangent bundles\/} $TX$. One can define tangent spaces $T_xX$ for $x\in X$, which behave in a functorial way. However, $\dim T_xX$ is not constant, but only upper semicontinuous in $x$, so that $TX=\coprod_{x\in X}T_xX$ is neither a manifold, nor a vector bundle, but only something like a coherent sheaf in algebraic geometry. For instance, in Example \ref{ku3ex2}, $\dim T_xX_P$ is 3 if $x\ne\de_0$, and 4 if~$x=\de_0$.

However, b-tangent bundles ${}^bTX$ and b-cotangent bundles ${}^bT^*X$ are well-defined for manifolds with g-corners $X$ (though the definition is more complicated than Definition \ref{ku3def7}), and have the same functorial properties. For the model spaces $X_P$, we have~${}^bTX_P\cong X_P\t\Hom_\Mon(P,\R)$.

This is another instance of Principle \ref{ku3princ}: if we can write some piece of mathematics about manifolds with corners $X$ using b-(co)tangent bundles ${}^bTX,{}^bT^*X$ rather than (co)tangent bundles $TX,T^*X$, then it is likely to extend immediately to manifolds with g-corners.

Here is the analogue of Definition \ref{ku3def8}, from \cite[\S 3.6]{Joyc9}. Let $X$ be a manifold with g-corners, and $x\in X$. Then as in Remark \ref{ku3rem3}, $X$ near $x$ is locally modelled on $X_Q\t\R^l$ near $(\de_0,0)$ for a toric monoid $Q$, giving an isomorphism
\e
{}^bT_xX\cong \Hom_\Mon(Q,\R)\t\R^l.
\label{ku3eq17}
\e
Define $\ti M_xX$ to be the subset of ${}^bT_xX$ identified with $\Hom_\Mon(Q,\N)\t\{0\}$ by \eq{ku3eq17}. Then $\ti M_xX$ is a toric monoid isomorphic to the dual monoid $Q^\vee$, a submonoid of ${}^bT_xX$, and $\rank\ti M_xX=k$ if $x\in S^k(X)\subseteq X$. Also $X$ is a manifold with (ordinary) corners if and only if $\ti M_xX\cong\N^k$ for all~$x\in X$.

If $f:X\ra Y$ is an interior map of manifolds with g-corners and $x\in X$ with $f(x)=y\in Y$, then ${}^bT_xf:{}^bT_xX\ra{}^bT_yY$ maps $\ti M_xX\ra\ti M_yY$, and we write $\ti M_xf:={}^bT_xf\vert_{\ti M_xX}:\ti M_xX\ra\ti M_yY$, a monoid morphism. We define $f$ to be {\it simple\/} if $\ti M_xf$ is an isomorphism for all $x\in X$. Write $\Mangcsi$ for the subcategory of $\Mangc$ with morphisms simple maps.

We can also define {\it b-normal\/} interior maps of manifolds with g-corners in terms of properties of the $\ti M_xf$. By \cite[\S 4.1]{Joyc9}, $f$ is b-normal if and only if\/ $C(f)$ maps $C_k(X)\ra \coprod_{l=0}^kC_l(Y)$ for all $k$, as in Proposition~\ref{ku3prop1}(c).

We now have a diagram of categories of manifolds with (g-)corners:
\e
\begin{gathered}
\xymatrix@C=50pt@R=15pt{
\Mancsi  \ar[rr]_\subset \ar[d]^\subset &&
\Mangcsi \ar[d]^\subset\\
\Mancis \ar[r]_\subset \ar[d]^\subset & \Mancin \ar[r]_\subset \ar[d]^\subset & \Mangcin \ar[d]^\subset   \\
\Mancst \ar[r]^\subset & \Manc \ar[r]^\subset & \Mangc.\! }
\end{gathered}
\label{ku3eq18}
\e

One reason for considering manifolds with g-corners is that they are specially well behaved under fibre products, as the next result shows \cite[\S 4.3]{Joyc9}:

\begin{thm} Let\/ $g:X\ra Z$ and\/ $h:Y\ra Z$ be interior maps of manifolds with g-corners. Call\/ $g,h$ \begin{bfseries}b-transverse\end{bfseries} if\/ ${}^bT_xg\op {}^bT_yh:{}^bT_xX\op {}^bT_yY\ra{}^bT_zZ$ is a surjective linear map for all\/ $x\in X$ and\/ $y\in Y$ with\/ $g(x)=h(y)=z\in Z$. 

If\/ $g,h$ are b-transverse then the fibre product\/ $X\t_{g,Z,h}Y$ exists in~$\Mangcin$.
\label{ku3thm1}
\end{thm}

The analogue is false for manifolds with (ordinary) corners, unless we impose complicated extra restrictions on $g,h$ over~$\pd^jX,\pd^kY,\pd^lZ$. 

\subsection{\texorpdfstring{$\mu$-Kuranishi spaces with boundary and with corners}{\textmu-Kuranishi spaces with boundary and with corners}}
\label{ku35}

In \S\ref{ku2} we defined $\mu$-Kuranishi spaces without boundary. We now explain how to modify the picture to define {\it $\mu$-Kuranishi spaces with boundary\/} and {\it with corners}. We use the definitions and notation of \S\ref{ku31}--\S\ref{ku33} throughout.

\begin{dfn} Definition \ref{ku2def1} defined our `$O(s)$' and `$O(s^2)$' notation. Parts (i),(ii),(vi),(vii) need no changes when $V,W$ are manifolds with corners. However, if $f,g:V\ra W$ are smooth maps of manifolds with corners, then parts (iii)--(v) defining when $f=g+O(s)$, or $f=g+O(s^2)$, or $f=g+\La\cdot s+O(s^2)$, do need alteration, and we replace them with (iii$)'$--(v$)'$, as follows.

Let $V,W$ be manifolds with corners, $E\ra V$ a vector bundle, $s\in C^\iy(E)$ a smooth section, and $f:V\ra W$ a smooth map. Then:
\begin{itemize}
\setlength{\itemsep}{0pt}
\setlength{\parsep}{0pt}
\item[(iii$)'$] We write $f=g+O(s)$ if:
\begin{itemize}
\setlength{\itemsep}{0pt}
\setlength{\parsep}{0pt}
\item[(A)] Whenever $h:W\ra\R$ is a smooth map, there exists $\al\in C^\iy(E^*)$ such that $h\ci f=h\ci g+\al\cdot s$.
\item[(B)] Suppose $v\in V$ with $s(v)=0$. Then (A) implies that $f(v)=g(v)$ in $W$. Suppose $W'$ is an open neighbourhood of $f(v)$ in $W$, and $H:W'\ra[0,\iy)$ is a smooth map. Then there exists an open neighbourhood $V'$ of $v$ in $f^{-1}(W')\cap g^{-1}(W')$ and $\al\in C^\iy(E^*\vert_{V'})$ such that $H\ci f\vert_{V'}=e^{\al\cdot s}H\ci g\vert_{V'}:V'\ra[0,\iy)$.
\end{itemize}
\item[(iv$)'$] We write $f=g+O(s^2)$ if the obvious analogues of (iii$)'$(A),(B) hold, with $\al\in C^\iy(E^*\ot E^*)$ and $\al\in C^\iy(E^*\ot E^*\vert_{V'})$.
\item[(v$)'$] If $\La\in C^\iy\bigl(E^*\ot f^*({}^bTW)\bigr)$, we write $f=g+\La\cdot s+O(s^2)$ if:
\begin{itemize}
\setlength{\itemsep}{0pt}
\setlength{\parsep}{0pt}
\item[(A)] Whenever $h:W\ra\R$ is a smooth map, there exists $\al\in C^\iy(E^*\ot E^*)$ such that $h\ci f=h\ci g+\La\cdot(s\ot f^*(\d h))+\al\cdot (s\ot s)$.
\item[(B)] Suppose $v\in V$ with $s(v)=0$. Then (A) implies that $f(v)=g(v)$ in $W$. Suppose $W'$ is an open neighbourhood of $f(v)$ in $W$, and $H:W'\ra[0,\iy)$ is smooth. Then there exists an open neighbourhood $V'$ of $v$ in $f^{-1}(W')\cap g^{-1}(W')$ and $\al\in C^\iy(E^*\ot E^*\vert_{V'})$ such that $H\ci f\vert_{V'}=e^{\al\cdot (s\ot s)}H\ci g\vert_{V'}+\La\cdot(s\ot f^*(\d H)):V'\ra\R$.
\end{itemize}
\end{itemize}
\label{ku3def11}
\end{dfn}

\begin{rem}{\bf(a)} In Remark \ref{ku3rem2} we explained that a smooth map $f:V\ra W$ of manifolds with corners comprises both `discrete' and `continuous data'. One consequence of Definition \ref{ku3def11}(iii$)'$--(v$)'$ is that if $f=g+O(s)$, or $f=g+O(s^2)$, or $f=g+\La\cdot s+O(s^2)$, then {\it the discrete data of\/ $f,g:V\ra W$ agree near\/} $s^{-1}(0)\subseteq V$. For example, in (iii$)'$(B), let $(v_1,\ldots,v_m)\in\R^m_k$ be coordinates near $v$ on $V$, with $v=(0,\ldots,0)$. Then by Definition \ref{ku3def1}(b)
\begin{align*}
H\ci f\vert_{V'}&=F(v_1,\ldots,v_m) v_1^{a_1}\cdots v_k^{a_k} &&\text{or} & H\ci f\vert_{V'}&=0 &&\text{near $v$ in $V$,}\\
H\ci g\vert_{V'}&=G(v_1,\ldots,v_m) v_1^{b_1}\cdots v_k^{b_k} &&\text{or} & H\ci g\vert_{V'}&=0 &&\text{near $v$ in $V$,}
\end{align*}
where $F,G>0$ and $a_i,b_i\in\N$. Then $H\ci f\vert_{V'}=e^{\al\cdot s}H\ci g\vert_{V'}$ implies that $H\ci f\vert_{V'}=0$ if and only if $H\ci g\vert_{V'}=0$ near $v$, and if not then $a_i=b_i$ for all~$i$.

\smallskip

\noindent{\bf(b)} We can make (iii$)'$(B),(v$)'$(B) look more like (iii$)'$(A),(v$)'$(A) by setting $h=\log H$. Then $H\ci f\vert_{V'}=e^{\al\cdot s}H\ci g\vert_{V'}$ in (iii$)'$(B) becomes $h\ci f=h\ci g+\al\cdot s$ in (iii$)'$(A) wherever $H>0$.
\smallskip

\noindent{\bf(c)} Note that we used the {\it b-tangent bundle\/} ${}^bTW$ rather than the tangent bundle $TW$ in $\La\in C^\iy\bigl(E^*\ot f^*({}^bTW)\bigr)$ in (v$)'$. As stated in Principle \ref{ku3princ}, b-tangent bundles are very often preferable to tangent bundles, and it is so in this case.
\label{ku3rem4}
\end{rem}

Next we modify Definitions \ref{ku2def2} and \ref{ku2def3} to the corners case.

\begin{dfn} Let $X$ be a topological space. A {\it $\mu$-Kuranishi neighbourhood $(V,E,s,\psi)$ on\/ $X$ with boundary}, or {\it with corners}, is as in Definition \ref{ku2def2}, except that we take $V,E$ to be manifolds with boundary, or with corners, respectively.

\label{ku3def12}
\end{dfn}

\begin{dfn} Let $X,Y$ be topological spaces, $f:X\ra Y$ a continuous map, $(V_i,E_i,s_i,\psi_i)$, $(V_j,E_j,s_j,\psi_j)$ be $\mu$-Kuranishi neighbourhoods with boundary, or with corners, on $X,Y$, and $S\subseteq\Im\psi_i\cap f^{-1}(\Im\psi_j)\subseteq X$ be open. Define {\it triples\/} $(V_{ij},\phi_{ij},\hat\phi_{ij})$, the binary relation $\sim$ on triples involving $\La$, $\sim$-equivalence classes $[V_{ij},\phi_{ij},\hat\phi_{ij}]$ of $(V_{ij},\phi_{ij},\hat\phi_{ij})$, and {\it morphisms\/ $[V_{ij},\phi_{ij},\hat\phi_{ij}]:(V_i,E_i,s_i,\psi_i)\ra (V_j,E_j,s_j,\psi_j)$ of $\mu$-Kuranishi neighbourhoods with boundary, or with corners, over\/} $(S,f)$, as in Definition \ref{ku2def3}, but with the following differences:
\begin{itemize}
\setlength{\itemsep}{0pt}
\setlength{\parsep}{0pt}
\item[(i)] $V_i,V_j,V_{ij},V_{ij}',\dot V_{ij},\ldots$ are now manifolds with corners, and $\phi_{ij},\phi_{ij}',\ldots$ are smooth maps of manifolds with corners, in the sense of~\S\ref{ku31}. 
\item[(ii)] In the definition of triples $(V_{ij},\phi_{ij},\hat\phi_{ij})$ in Definition \ref{ku2def3}(a)--(e), we impose the following extra condition: let $C$ be a connected component of $S^k(V_{ij})$
for $k\ge 1$, and $\ov C$ be the closure of $C$ in $V_{ij}$. Then~$\ov C\cap \psi_i^{-1}(S)\ne\es$.
\item[(iii)] We define $\sim$ with $\La:E_i\vert_{\dot V_{ij}}\ra \phi_{ij}^*({}^bTV_j)\vert_{\dot V_{ij}}$, where we have used the {\it b-tangent bundle\/} ${}^bTV_j$ of \S\ref{ku33} rather than the tangent bundle~$TV_j$.
\item[(iv)] The equation $\phi_{ij}'=\phi_{ij}+\La\cdot s_i+O(s_i^2)$ in \eq{ku2eq1} in the definition of $\sim$ is interpreted as in Definition \ref{ku3def11}(v$)'$. This makes sense by~(iii).
\end{itemize}
\label{ku3def13}
\end{dfn}

\begin{rem} Here is what Definition \ref{ku3def13}(ii) is for.

In Remark \ref{ku3rem2} we explained that a smooth map of manifolds with corners such as $\phi_{ij}:V_{ij}\ra V_j$ is comprised of `discrete data' and `continuous data', and some interesting conditions, such as $\phi_{ij}$ interior, or b-normal, or strongly smooth, or simple, depend only on the discrete data. 

We can show conditions such as $\phi_{ij}$ simple (as in Proposition \ref{ku3prop2} below) hold near $\psi_i^{-1}(S)$ in $V_{ij}$. But if a discrete condition on $\phi_{ij}$ holds at $x\in S^k(V_{ij})\subseteq V_{ij}$, then it also holds everywhere in the connected component $C$ of $S^k(V_{ij})$ containing $x$. Therefore Definition \ref{ku3def13}(ii) implies that if a discrete condition on $\phi_{ij}$ holds near $\psi_i^{-1}(S)$, then it holds on all of~$V_{ij}$. 

If $V_{ij}$ does not satisfy Definition \ref{ku3def13}(ii), we can modify it to do so by deleting any components $C$ of $S^k(V_{ij})$ for $k\ge 1$ with~$\ov C\cap \psi_i^{-1}(S)=\es$.

The restriction to $k\ge 1$ means Definition \ref{ku3def13}(ii) is vacuous in the `without boundary' case, as in~\S\ref{ku21}--\S\ref{ku24}.
\label{ku3rem5}
\end{rem}

For Definitions \ref{ku2def4}--\ref{ku2def8}, we insert `with boundary' or `with corners' throughout. In Definition \ref{ku2def5} we write $\muKur^{\bf b}_S(X)\subset\muKur^{\bf c}_S(X)$ for the categories of $\mu$-Kuranishi neighbourhoods over $S\subseteq X$ with boundary and with corners, and $\GmuKurb\subset\GmuKurc$ for the categories of global $\mu$-Kuranishi neighbourhoods with boundary and with corners. For restrictions $\Phi_{ij}\vert_T$ of morphisms $\Phi_{ij}$ on $S$ to $T\subseteq S$ in Definition \ref{ku2def7}, where $(V_{ij},\phi_{ij},\hat\phi_{ij})$ represents the $\sim_S$-equivalence class $\Phi_{ij}$, note that Definition \ref{ku3def13}(ii) for $(V_{ij},\ab\phi_{ij},\ab\hat\phi_{ij}),S$ does not imply Definition \ref{ku3def13}(ii) for $(V_{ij},\phi_{ij},\hat\phi_{ij}),T$. To deal with this, let $V_{ij}'\subseteq V_{ij}$ be the dense open subset obtained by deleting all connected components $C$ of $S^k(V_{ij})$ for $k\ge 1$ such that $\ov C\cap \psi_i^{-1}(T)=\es$, let $\phi_{ij}'=\phi_{ij}\vert_{\smash{V_{ij}'}}$ and $\hat\phi_{ij}'=\hat\phi_{ij}\vert_{\smash{V_{ij}'}}$, and define $\Phi_{ij}\vert_T$ to be the $\sim_T$-equivalence class of $(V_{ij}',\phi_{ij}',\hat\phi_{ij}')$. Otherwise there are no significant changes. This extends \S\ref{ku21} to the corners case.

\begin{prop} Suppose\/ $\Phi_{ij}:(V_i,E_i,s_i,\psi_i)\ra (V_j,E_j,s_j,\psi_j)$ is a $\mu$-coordinate change with corners over $S\subseteq X,$ that is, $\Phi_{ij}$ is invertible in the category $\muKur^{\bf c}_S(X),$ and let\/ $(V_{ij},\phi_{ij},\hat\phi_{ij})$ be a representative for the $\sim$-equivalence class $\Phi_{ij}$. Then $\phi_{ij}:V_{ij}\ra V_j$ is a \begin{bfseries}simple map\end{bfseries}, in the sense of\/~{\rm\S\ref{ku31}}. 
\label{ku3prop2}
\end{prop}

\begin{proof} Let $\Phi_{ji}:(V_j,E_j,s_j,\psi_j)\ra (V_i,E_i,s_i,\psi_i)$ be the inverse in $\muKur^{\bf c}_S(X)$ of $\Phi_{ij}$, and $(V_{ji},\phi_{ji},\hat\phi_{ji})$ a representative for $\Phi_{ji}$. Then $\Phi_{ji}\ci\Phi_{ij}$ is represented by 
\begin{equation*}
(V_{ii},\phi_{ii},\hat\phi_{ii})=\bigl(\phi_{ij}^{-1}(V_{ji}),\phi_{ji}\ci\phi_{ij}\vert_{V_{ii}},\phi_{ij}\vert_{V_{ii}}^*(\hat\phi_{ji})\ci\hat\phi_{ij}\vert_{V_{ii}}\bigr),
\end{equation*}
so $\Phi_{ji}=\Phi_{ij}^{-1}$ implies that $(V_i,\id_{V_i},\id_{E_i})\sim(V_{ii},\phi_{ii},\hat\phi_{ii})$ and thus by Definition \ref{ku3def13} there exist open $\psi_i^{-1}(S)\subseteq\dot V_{ii}\subseteq V_{ii}$ and $\La:E_i\vert_{\dot V_{ii}}\ra \id_{V_i}^*({}^bTV_i)\vert_{\dot V_{ii}}$ with
\e
\phi_{ji}\ci\phi_{ij}\vert_{\dot V_{ii}}=\id_{V_i}\vert_{\dot V_{ii}}+\La\cdot s_i+O(s_i^2).
\label{ku3eq19}
\e
Similarly we have $\psi_j^{-1}(S)\!\subseteq\!\dot V_{jj}\!\subseteq\!\phi_{ji}^{-1}(V_{ij})$ and $\La':E_j\vert_{\dot V_{jj}}\!\ra\!\id_{V_j}^*({}^bTV_j)\vert_{\dot V_{jj}}$ with
\e
\phi_{ij}\ci\phi_{ji}\vert_{\dot V_{jj}}=\id_{V_j}\vert_{\dot V_{jj}}+\La'\cdot s_j+O(s_j^2).
\label{ku3eq20}
\e

Let $x\in S$ and set $v_i=\psi_i^{-1}(x)\in V_{ij}\subseteq V_i$ and $v_j=\psi_j^{-1}(x)\in V_{ji}\subseteq V_j$, so that $\phi_{ij}(v_i)=v_j$ and $\phi_{ji}(v_j)=v_i$. Choose charts $(U,\chi)$ on $V_{ij}$ with $(0,\ldots,0)\in U\subseteq\R^m_k$ open and $\chi(0,\ldots,0)=v_i$, and $(W,\xi)$ on $V_{ji}$ with $(0,\ldots,0)\in W\subseteq\R^n_l$ open and $\xi(0,\ldots,0)=v_j$. Then \eq{ku3eq1} gives smooth maps $\xi^{-1}\ci\phi_{ij}\ci\chi$, $\chi^{-1}\ci\phi_{ji}\ci\xi$ between open subsets of $\R^m_k,\R^n_l$. Write $\xi^{-1}\ci\phi_{ij}\ci\chi=(f_1,\ldots,f_n)$ for $f_q=f_q(u_1,\ldots,u_m)$ and $\chi^{-1}\ci\phi_{ji}\ci\xi=(g_1,\ldots,g_m)$ for~$g_p=g_p(w_1,\ldots,w_n)$.

Then Definition \ref{ku3def1}(b) implies that near $(0,\ldots,0)\in U$ we have
\e
f_q(u_1,\ldots,u_m)=F_q(u_1,\ldots,u_m)\cdot u_1^{a_{1,q}}\cdots u_k^{a_{k,q}}\quad\text{or}\quad f_q=0,
\label{ku3eq21}
\e
for all $q=1,\ldots,l$, where $F_q>0$ and $a_{p,q}\in\N$, and near $(0,\ldots,0)\in W$ we have
\e
g_p(w_1,\ldots,w_n)=G_p(w_1,\ldots,w_n)\cdot w_1^{b_{1,p}}\cdots w_l^{b_{l,p}}\quad\text{or}\quad g_p=0,
\label{ku3eq22}
\e
for all $p=1,\ldots,k$, where $G_p>0$ and~$b_{q,p}\in\N$.

By Remark \ref{ku3rem4}(a), \eq{ku3eq19}--\eq{ku3eq20} imply that $\phi_{ji}\ci\phi_{ij}$ and $\id_{V_i}$ have the same discrete data near $\psi_i^{-1}(S)$, and $\phi_{ij}\ci\phi_{ji}$ and $\id_{V_j}$ the same discrete data near $\psi_j^{-1}(S)$. But using \eq{ku3eq21}--\eq{ku3eq22} we can compute the discrete data of $\phi_{ji}\ci\phi_{ij}$ near $v_i$ and of $\phi_{ij}\ci\phi_{ji}$ near $v_j$, and compare it to those for $\id_{V_i},\id_{V_j}$. We find that the possibilities $f_q=0$, $g_p=0$ in \eq{ku3eq21}--\eq{ku3eq22} do not occur, and $k=l$, and $\bigl(a_{p,q}\bigr){}_{p,q=1}^k$, $\bigl(b_{q,p}\bigr){}_{q,p=1}^k$ are inverse matrices. Since $a_{p,q},b_{q,p}\in\N$, this is only possible if $\bigl(a_{p,q}\bigr){}_{p,q=1}^k$, $\bigl(b_{q,p}\bigr){}_{q,p=1}^k$ have one 1 in each row and each column, and are otherwise zero. By Definition \ref{ku3def1}(f), this means that $\phi_{ij}$ and $\phi_{ji}$ are simple near $v_i,v_j$. Hence $\phi_{ij}$ is simple near $\psi_i^{-1}(S)$, so by Definition \ref{ku3def13}(ii) and Remark \ref{ku3rem5}, $\phi_{ij}$ is simple.
\end{proof}

Here are the analogues of Theorem \ref{ku2thm1} (the sheaf property for morphisms) and Theorem \ref{ku2thm2} (criteria for invertibility of morphisms), proved in \S\ref{ku63}--\S\ref{ku64}.

\begin{thm} The analogue of Theorem\/ {\rm\ref{ku2thm1}} holds for $\mu$-Kuranishi neighbourhoods with boundary, and with corners.
\label{ku3thm2}
\end{thm}

\begin{thm} Let\/ $\Phi_{ij}=[V_{ij},\phi_{ij},\hat\phi_{ij}]:(V_i,E_i,s_i,\psi_i)\ra (V_j,E_j,s_j,\psi_j)$ be a morphism of $\mu$-Kuranishi neighbourhoods with corners on $X$ over an open subset\/ $S\subseteq X$. Let\/ $x\in S,$ and set\/ $v_i=\psi_i^{-1}(x)\in V_i$ and\/ $v_j=\psi_j^{-1}(x)\in V_j$. Suppose $\phi_{ij}$ is interior near $v_i$ (this is implied by the condition $\phi_{ij}$ simple below). Consider the sequence of finite-dimensional real vector spaces:
\e
\begin{gathered}
\smash{\xymatrix@C=16pt{ 0 \ar[r] & {}^bT_{v_i}V_i \ar[rrr]^(0.39){{}^b\d s_i\vert_{v_i}\op{}^b\d\phi_{ij}\vert_{v_i}} &&& E_i\vert_{v_i} \!\op\!{}^bT_{v_j}V_j 
\ar[rrr]^(0.56){-\hat\phi_{ij}\vert_{v_i}\op {}^b\d s_j\vert_{v_j}} &&& E_j\vert_{v_j} \ar[r] & 0. }}
\end{gathered}
\label{ku3eq23}
\e
Here ${}^b\d s_i\vert_{v_i}$ means the composition of the natural map $I_{V_i}\vert_{v_i}:{}^bT_{v_i}V_i\ra T_{v_i}V_i$ and\/ $\nabla s_i\vert_{v_i}:T_{v_i}V_i\ra E_i\vert_{v_i}$ for any connection $\nabla$ on $E_i\ra V_i,$ and is independent of the choice of\/ $\nabla$ as $s_i(v_i)=0$. Also ${}^b\d\phi_{ij}\vert_{v_i}:{}^bT_{v_i}V_i\ra {}^bT_{v_j}V_j$ is defined as in\/ {\rm\S\ref{ku33},} as $\phi_{ij}$ is interior, and\/ ${}^b\d\phi_{ij}\vert_{v_i},\hat\phi_{ij}\vert_{v_i}$ are independent of the choice of representative $(V_{ij},\phi_{ij},\hat\phi_{ij})$ for the $\sim$-equivalence class $[V_{ij},\phi_{ij},\hat\phi_{ij}]$. Thus \eq{ku3eq23} is well defined, and Definition\/ {\rm\ref{ku2def3}(d)} implies that\/ \eq{ku3eq23} is a complex.

Then $\Phi_{ij}$ is a $\mu$-coordinate change over $S$ if and only if\/ $\phi_{ij}$ is simple and\/ \eq{ku3eq23} is exact for all\/~$x\in S$.
\label{ku3thm3}
\end{thm}

Here the condition $\phi_{ij}$ simple is necessary for $\Phi_{ij}$ to be a $\mu$-coordinate change, by Proposition~\ref{ku3prop2}.

In the corners analogue of Definition \ref{ku2def10}, we require {\it embeddings of manifolds with corners\/} to be simple, by definition. Then the analogue of Lemma \ref{ku2lem1} for FOOO coordinate changes of Kuranishi neighbourhoods with corners can be proved using Theorem \ref{ku3thm3}, noting that $\phi_{ij}:V_{ij}\hookra V_j$ is simple. This extends \S\ref{ku22} to the corners case.

Now in \S\ref{ku21}--\S\ref{ku22} we were doing differential geometry, so in $\mu$-Kuranishi neighbourhoods $(V_i,E_i,s_i,\psi_i)$, it mattered that $V_i$ is a manifold (possibly with corners), $E_i\ra V_i$ a vector bundle, and so on. In \S\ref{ku23}--\S\ref{ku24}, the arguments were of a different character: we treated $\mu$-Kuranishi neighbourhoods $(V_i,E_i,s_i,\psi_i)$ and their morphisms $\Phi_{ij}$ just as objects and morphisms in certain categories $\muKur_S(X)$, without caring what they really are, and the proofs involved categories and sheaves on topological spaces, but no differential geometry.

Because of this, there really are {\it no significant issues\/} in extending \S\ref{ku23}--\S\ref{ku24} to the boundary and corners cases. We just insert `with boundary' or `with corners' throughout, use the definitions of $\mu$-Kuranishi neighbourhoods and their morphisms and $\mu$-coordinate changes explained above, apply Theorem \ref{ku3thm2} in place of Theorem \ref{ku2thm1}, and everything just works, exactly as it did before.

Thus, as in Definitions \ref{ku2def11} and \ref{ku2def13} we can define {\it $\mu$-Kuranishi spaces with boundary\/} and {\it with corners\/} $\bX$ and their morphisms $\bs f:\bX\ra\bY$, and as in Theorem \ref{ku2thm3} we can define {\it composition of morphisms\/} $\bs g\ci\bs f:\bX\ra\bZ$, and show that $\mu$-Kuranishi spaces, with boundary, and with corners form categories $\muKurb,\muKurc$. We can regard $\mu$-Kuranishi spaces (without boundary) and their morphisms, as in \S\ref{ku21}--\S\ref{ku23}, as special examples of $\mu$-Kuranishi spaces with corners in which the $V_i$ in $(V_i,E_i,s_i,\psi_i)$ have $\pd V_i=\es$ for each $i\in I$. Hence, we have inclusions of full subcategories~$\muKur\subset\muKurb\subset\muKurc$.

Also, as in Definition \ref{ku2def14}, we can regard manifolds with boundary or corners as examples of $\mu$-Kuranishi spaces with boundary or corners, and define full and faithful functors $F_\Manb^\muKurb:\Manb\ra\muKurb$ and $F_\Manc^\muKurc:\Manc\ra\muKurc$. As in Definition \ref{ku2def15}, we can define {\it $\mu$-Kuranishi neighbourhoods with boundary or corners\/ $(V_a,E_a,s_a,\psi_a)$ on $\mu$-Kuranishi spaces with boundary or corners\/} $\bX$, and the analogue of Theorem \ref{ku2thm4} holds.

\subsection{\texorpdfstring{Boundaries and corners of $\mu$-Kuranishi spaces with corners}{Boundaries and corners of \textmu-Kuranishi spaces with corners}}
\label{ku36}

Section \ref{ku32} defined the {\it boundary\/} $\pd X$ and $k$-{\it corners\/} $C_k(X)$ of a manifold with corners $X$. The definition involved {\it local boundary components\/} and {\it local\/ $k$-corner components\/} of $X$ at each $x\in X$. We now generalize these ideas to $\mu$-Kuranishi spaces with corners.

\begin{dfn} Let $\bX=(X,\cK)$ be a $\mu$-Kuranishi space with corners, with $\cK=\bigl(I,(V_i,\ab E_i,\ab s_i,\ab\psi_i)_{i\in I},\Phi_{ij,\; i,j\in I}\bigr)$, and fix $x\in\bX$ and $k\ge 0$.

For each $i\in I$ with $x\in\Im\psi_i$, set $v_i=\psi_i^{-1}(x)\in V_i$, and consider the set
\e
\Ga_{x,i}^k=\bigl\{\text{local $k$-corner components of $V_i$ at $v_i$}\bigr\}.
\label{ku3eq24}
\e
If $i,j\in I$ with $x\in\Im\psi_i\cap\Im\psi_j$, and $(V_{ij},\phi_{ij},\hat\phi_{ij})$ is a representative for $\Phi_{ij}=[V_{ij},\phi_{ij},\hat\phi_{ij}]$, then $v_i\in V_{ij}\subseteq V_i$, and $\phi_{ij}:V_{ij}\ra V_j$ is a simple map of manifolds with corners by Proposition \ref{ku3prop2}, with $\phi_{ij}(v_i)=v_j$. Therefore Lemma \ref{ku3lem2} shows that $\ga\mapsto \ga'=(\phi_{ij})_*(\ga)$ gives a 1-1 correspondence $\Ga_{x,i}^k\ra \Ga_{x,j}^k$.

If $(V_{ij}',\phi_{ij}',\hat\phi_{ij}')$ is an alternative representative for $\Phi_{ij}$ then $\phi_{ij}=\phi_{ij}'+\La\cdot s_i+O(s_i^2)$ near $v_i$ for some $\La\in C^\iy(E_i^*\ot\phi_{ij}^*({}^bTV_j))$. As in Remark \ref{ku3rem4}(a), we can show this implies that $(\phi_{ij})_*= (\phi_{ij}')_*$ as maps $\Ga_{x,i}^k\ra \Ga_{x,j}^k$. Thus $(\phi_{ij})_*$ depends only on $\Phi_{ij}$, and we write~$(\Phi_{ij})_*=(\phi_{ij})_*:\Ga_{x,i}^k\ra \Ga_{x,j}^k$.

If $h,i,j\in I$ with $x\in\Im\psi_h\cap\Im\psi_i\cap\Im\psi_j$, then by choosing representatives for $\Phi_{hi},\Phi_{hj},\Phi_{ij}$ including $\phi_{hi},\phi_{hj},\phi_{ij}$ and noting that as $\Phi_{hj}=\Phi_{ij}\ci\Phi_{hi}$ we have $\phi_{hj}=\phi_{ij}\ci\phi_{hi}+\La\cdot s_h+O(s_h^2)$ for some $\La\in C^\iy(E_h^*\ot\phi_{hj}^*({}^bTV_j))$, so that $(\phi_{hj})_*=(\phi_{ij})_*\ci(\phi_{hi})_*:\Ga_{x,h}^k\ra \Ga_{x,j}^k$ as above, we deduce that 
\e
(\Phi_{hj})_*=(\Phi_{ij})_*\ci(\Phi_{hi})_*:\Ga_{x,h}^k\longra \Ga_{x,j}^k.
\label{ku3eq25}
\e
Also $(\Phi_{ii})_*=(\id_{V_i})_*=\id:\Ga_{x,i}^k\longra \Ga_{x,i}^k$.

Define a {\it local\/ $k$-corner component of\/ $\bX$ at\/} $x$ to be an element of the set
\e
\Ga_{x,\bX}^k=\bigl(\ts\coprod_{i\in I:x\in\Im\psi_i}\Ga_{x,i}^k\bigr)\big/\approx,
\label{ku3eq26}
\e
where $\approx$ is the binary relation $\ga_i\approx \ga_j$ if $\ga_i\in \Ga_{x,i}^k$ and $\ga_j\in \Ga_{x,j}^k$ with $(\Phi_{ij})_*(\ga_i)=\ga_j$. The discussion above implies that $\approx$ is an equivalence relation, and that there are canonical bijections $\Ga_{x,i}^k\cong\Ga_{x,\bX}^k$ for each $i\in I$ with $x\in\Im\psi_i$. That is, {\it local\/ $k$-corner components of\/ $\bX$ at\/ $x$ are canonically identified with local\/ $k$-corner components of\/ $V_i$ at\/ $v_i=\psi_i^{-1}(x),$ if\/ $i\in I$ with\/}~$x\in\Im\psi_i$.

When $k=1$, we also call local 1-corner components of $\bX$ at $x$ {\it local boundary components of\/ $\bX$ at\/}~$x$.

As in \S\ref{ku32}, there is a natural 1-1 correspondence between local $k$-corner components $\ga$ of $V_i$ at $v_i$, and (unordered) sets $\{\be_1,\ldots,\be_k\}$ of $k$ distinct local boundary components $\be_1,\ldots,\be_k$ of $V_i$ at $v_i$. These 1-1 correspondences for $i,j\in I$ commute with the $(\Phi_{ij})_*$ above. Therefore there is a natural 1-1 correspondence between local $k$-corner components $\ga$ of $\bX$ at $x$, and sets $\{\be_1,\ldots,\be_k\}$ of $k$ distinct local boundary components $\be_1,\ldots,\be_k$ of $\bX$ at~$x$.

Write $\depth_\bX x$ for the (necessarily finite) number of local boundary components of $\bX$ at $x$. Then the number of local $k$-corner components of $\bX$ at $x$ is~$\bmd{\Ga_{x,\bX}^k}=\binom{\depth_\bX x}{k}$.

We will define a $\mu$-Kuranishi space with corners $C_k(\bX)=\bigl(C_k(X),\cK^{C_k}\bigr)$ for each $k\in\N$ called the $k$-{\it corners of\/} $\bX$, with virtual dimension $\vdim C_k(\bX)=\vdim\bX-k$, and a morphism $\bs i_\bX:C_k(\bX)\ra\bX$, which is not interior for $k>0$. When $k=1$, we also write $\pd\bX=C_1(\bX)$ and $\bs i_\bX:\pd\bX\ra\bX$, and call $\pd\bX$ the {\it boundary\/} of $\bX$. As in \eq{ku3eq4}, the underlying topological space $C_k(X)$ is given as a set by
\e
C_k(X)=\bigl\{(x,\ga):\text{$x\in X$, $\ga$ is a local $k$-corner 
component of $\bX$ at $x$}\bigr\},
\label{ku3eq27}
\e
and the underlying continuous map is $i_\bX:(x,\ga)\mapsto x$. We define the topology on $C_k(X)$ shortly. As in \eq{ku3eq5}, the 1-1 correspondence above shows that
\e
\begin{aligned}
C_k(X)\cong\bigl\{(x,\,&\{\be_1,\ldots,\be_k\}):\text{$x\in X,$
$\be_1,\ldots,\be_k$ are distinct}\\
&\text{local boundary components for $\bX$ at $x$}\bigr\}.
\end{aligned}
\label{ku3eq28}
\e

Define a $\mu$-Kuranishi structure with corners $\cK^{C_k}$ on $C_k(X)$ by
\e
\cK^{C_k}=\bigl(I,(V^{C_k}_i,E_i^{C_k},s_i^{C_k},\psi_i^{C_k})_{i\in I},\Phi_{ij,\; i,j\in I}^{C_k}\bigr),
\label{ku3eq29}
\e
where for each $i\in I$, $V_i^{C_k}=C_k(V_i)$, and $E_i^{C_k}=i_{V_i}^*(E_i)$ for $i_{V_i}:C_k(V_i)\ra V_i$ the projection, and $s_i^{C_k}=i_{V_i}^*(s_i)$, and $\psi_i^{C_k}:(s_i^{C_k})^{-1}(0)\ra C_k(X)$ maps $\psi_i^{C_k}:(v_i,\ga)\mapsto (\psi_i(v_i),[\ga])$, where $\ga\in\Ga_{\psi_i(v_i),i}^k$ is a local $k$-corner component of $V_i$ at $v_i\in s_i^{-1}(0)\subseteq V_i$, and $[\ga]\in\Ga_{\psi_i(v_i),\bX}^k$ the $\approx$-equivalence class of $\ga$. Since $\psi_i$ is injective, and $\ga\mapsto[\ga]$ bijective, we see that $\psi_i^{C_k}$ is injective.

If $i,j\in I$ and $(V_{ij},\phi_{ij},\hat\phi_{ij})$ represents $\Phi_{ij}=[V_{ij},\phi_{ij},\hat\phi_{ij}]$, then we set
\e
\Phi_{ij}^{C_k}=\bigl[V_{ij}^{C_k},\phi_{ij}^{C_k},\hat\phi_{ij}^{C_k}\bigr],
\label{ku3eq30}
\e
with $V_{ij}^{C_k}=C_k(V_{ij})$, and $\phi_{ij}^{C_k}=C_k(\phi_{ij}): V_{ij}^{C_k}=C_k(V_{ij})\ra C_k(V_j)=V_j^{C_k}$, where $\phi_{ij}$ is simple by Proposition \ref{ku3prop2} so that $C_k(\phi_{ij}):C_k(V_{ij})\ra C_k(V_j)$ is defined by Proposition \ref{ku3prop1}, and
\begin{align*}
\hat\phi_{ij}^{C_k}= i_{V_{ij}}^*&(\hat\phi_{ij}):E_i^{C_k}\vert_{V_{ij}^{C_k}}=i_{V_{ij}}^*(E_i)\longra \\
&i_{V_{ij}}^*\ci\phi_{ij}^*(E_j)=C_k(\phi_{ij})^*\ci i_{V_j}^*(E_j)=
(\phi_{ij}^{C_k})^*(E_j^{C_k}),
\end{align*}
where we use~$\phi_{ij}\ci i_{V_{ij}}=i_{V_j}\ci C_k(\phi_{ij}):C_k(V_{ij})\ra V_j$.

It is now easy to see that $\bigl(V_{ij}^{C_k},\phi_{ij}^{C_k},\hat\phi_{ij}^{C_k}\bigr)$ satisfies Definition \ref{ku2def3}(a)--(e), as $(V_{ij},\phi_{ij},\hat\phi_{ij})$ does. If $(V_{ij}',\phi_{ij}',\hat\phi_{ij}')$ is an alternative representative for $\Phi_{ij}$, so that $(V_{ij},\phi_{ij},\hat\phi_{ij})\sim(V_{ij}',\phi_{ij}',\hat\phi_{ij}')$ using $\dot V_{ij}$ and $\La$ in Definition \ref{ku2def3}, where $\La:E_i\vert_{\dot V_{ij}}\ra \phi_{ij}^*({}^bTV_j)\vert_{\dot V_{ij}}$ by Definition \ref{ku3def13}(iii), then using $\dot V_{ij}^{C_k}=C_k(\dot V_{ij})$ and $\La^{C_k}$ determined by the commutative diagram
\begin{equation*}
\xymatrix@!0@C=300pt@R=35pt{ *+[r]{i_{\dot V_{ij}}^*(E_i)\!=\!E_i^{C_k}\vert_{\dot V_{ij}^{C_k}}} \ar[r]_(0.32){\La^{C_k}} \ar[d]^(0.6){i_{\dot V_{ij}}^*(\La)} & *+[l]{(\phi^{C_k}_{ij})^*({}^bTV^{C_k}_j)\vert_{\dot V^{C_k}_{ij}}\!=\!C_k(\phi_{ij})^*({}^bT(C_k(V_j)))} \\
*+[r]{i_{\dot V_{ij}}^*\ci \phi_{ij}^*({}^bTV_j)} \ar@{=}[r] & *+[l]{C_k(\phi_{ij})^*\ci i_{V_j}^*({}^bTV_j),\!\!} \ar[u]^{C_k(\phi_{ij})^*(\Pi)} }
\end{equation*}
where $\Pi:i_{V_j}^*({}^bTV_j)\ra {}^bT(C_k(V_j))$ is the natural projection, we see that $\bigl(V_{ij}^{C_k},\ab\phi_{ij}^{C_k},\ab\hat\phi_{ij}^{C_k}\bigr)\ab\sim \bigl(V_{ij}^{\prime C_k},\phi_{ij}^{\prime C_k},\hat\phi_{ij}^{\prime C_k}\bigr)$. Hence $\Phi_{ij}^{C_k}$ in \eq{ku3eq30} is independent of choices, and is well-defined.
 
Since $\Ga_{x,i}^k\cong\Ga_{x,\bX}^k$ for $x\in\Im\psi_i$, we see that $\Im\psi_i^{C_k}\cong i_\bX^{-1}(\Im\psi_i)\subseteq C_k(X)$. As $X=\bigcup_{i\in I}\Im\psi_i$ by Definition \ref{ku2def11}(d), it follows that $C_k(X)=\bigcup_{i\in I}\Im\psi_i^{C_k}$, proving Definition \ref{ku2def11}(d) for $(C_k(X),\cK^{C_k})$. If $i,j\in I$ then 
\begin{equation*}
\psi_j^{-1}\ci\psi_i=\phi_{ij}\vert_{\psi_i^{-1}(\Im\psi_j)}:\psi_i^{-1}(\Im\psi_j)\longra\psi_j^{-1}(\Im\psi_i)
\end{equation*}
is a homeomorphism from an open set $\psi_i^{-1}(\Im\psi_j)\subseteq s_i^{-1}(0)$ to an open set $\psi_j^{-1}(\Im\psi_i)\subset s_j^{-1}(0)$. Applying the corner functor $C_k$, we deduce that 
\begin{equation*}
(\psi_j^{C_k})^{-1}\!\ci\!\psi_i^{C_k}\!=\!C_k(\phi_{ij})\vert_{(\psi_i^{C_k})^{-1}(\Im\psi_j^{C_k})}:(\psi_i^{C_k})^{-1}(\Im\psi_j^{C_k})\!\longra\!(\psi_j^{C_k})^{-1}(\Im\psi_i^{C_k})
\end{equation*}
is a homeomorphism from an open set $(\psi_i^{C_k})^{-1}(\Im\psi_j^{C_k})\subseteq (s_i^{C_k})^{-1}(0)$ to an open set $(\psi_j^{C_k})^{-1}(\Im\psi_i^{C_k})\subset (s_j^{C_k})^{-1}(0)$. Also $\psi_i^{C_k}$ is injective, as above. From all this, it follows that there is a unique topology on $C_k(X)$ such that $\psi_i^{C_k}$ is a homeomorphism with an open set in $C_k(X)$ for all~$i\in I$.

We have now defined all the structures in $C_k(\bX)=\bigl(C_k(X),\cK^{C_k}\bigr)$. We must verify Definition \ref{ku2def11}(a)--(f). Parts (a),(c),(d) are already proved. For (b), as $\dim V_i-\rank E_i=\vdim\bX$ and $\dim C_k(V_i)=\dim V_i-k$ we have $\dim V_i^{C_k}-\rank E_i^{C_k}=\vdim\bX-k$ for all $i\in I$. Part (e) follows from $C_k(\id_{V_i})=\id_{C_k(V_i)}$, and (f) by applying the corner functor $C_k$ to Definition \ref{ku2def11}(f) for $\bX$. Hence $C_k(\bX)$ is a $\mu$-Kuranishi space with corners with $\vdim C_k(\bX)=\vdim\bX-k$, and so the boundary $\pd\bX=C_1(\bX)$ is a $\mu$-Kuranishi space with corners.

Define a morphism $\bs i_\bX:C_k(\bX)\ra\bX$ by $\bs i_\bX=\bigl(i_\bX,(\bs i_\bX)_{ij,\;i\in I,\; j\in I}\bigr)$, where $i_\bX:(x,\ga)\mapsto x$ as above, and for all $i,j\in I$, if $(V_{ij},\phi_{ij},\hat\phi_{ij})$ represents $\Phi_{ij}=[V_{ij},\phi_{ij},\hat\phi_{ij}]$, then $(\bs i_\bX)_{ij}=\bigl[C_k(V_{ij}),\phi_{ij}\ci i_{V_{ij}}, i_{V_{ij}}^*(\hat\phi_{ij})\bigr]$, where $i_{V_{ij}}:C_k(V_{ij})\ra V_{ij}$. It is easy to show this is well-defined.

We can now take boundaries repeatedly to get $\mu$-Kuranishi spaces with corners $\pd\bX,\pd^2\bX,\ldots,\pd^k\bX,\ldots.$ As for manifolds with corners in \eq{ku3eq9}--\eq{ku3eq11}, we can show that there is a natural identification of sets
\e
\begin{split}
\pd^kX\cong\bigl\{(x,\be_1,\ldots,\be_k):\,&\text{$x\in X,$
$\be_1,\ldots,\be_k$ are distinct}\\
&\text{local boundary components for $\bX$ at $x$}\bigr\},
\end{split}
\label{ku3eq31}
\e
and there is a natural, free action of the symmetric group $S_k$ on $\pd^k\bX$ by isomorphisms of $\mu$-Kuranishi spaces, which on topological spaces acts by permutation of $\be_1,\ldots,\be_k$ in \eq{ku3eq31}, and as one would expect from \eq{ku3eq28} and \eq{ku3eq31}, there is a natural isomorphism of $\mu$-Kuranishi spaces $C_k(\bX)\cong \pd^k\bX/S_k$. Here the quotient $\pd^k\bX/S_k$ exists in $\muKur$ as $S_k$ acts freely by isomorphisms, so the $\mu$-Kuranishi structure can be pushed down locally from $\pd^k\bX$ to $\pd^k\bX/S_k$, and then glued using the sheaf property Theorem~\ref{ku3thm2}.
\label{ku3def14}
\end{dfn}

As for the categories $\Manc\subset\cManc$ in Definition \ref{ku3def2}, we enlarge~$\muKurc$:

\begin{dfn} Define a {\it $\mu$-Kuranishi space with corners $\bX=(X,\cK)$ of mixed dimension\/} with $\cK=\bigl(I,(V_i,E_i,s_i,\psi_i)_{i\in I},\Phi_{ij,\;i,j\in I}\bigr)$ as for $\mu$-Kuranishi spaces with corners in \S\ref{ku23} and \S\ref{ku35}, except that we omit the condition $\dim V_i-\rank E_i=n$ in Definition \ref{ku2def11}(b), allowing $\dim V_i-\rank E_i$ to take any value.

Write $I_m=\{i\in I:\dim V_i-\rank E_i=m\}$ for $m\in\Z$, so that $I=
\coprod_{m\in\Z}I_m$, and $X_m=\bigcup_{i\in I_m}\Im\psi_i$, so that $X_m$ is open and closed in $X$ with $X\!=\!\coprod_{m\in\Z}X_m$, allowing $I_m\!=\!X_m\!=\!\es$. Set $\cK_m\!=\!\bigl(I_m,(V_i,E_i,s_i,\psi_i)_{i\in I_m},\Phi_{ij,\;i,j\in I_m}\bigr)$ for $m\in\Z$. It is then easy to see that $\bX_m=(X_m,\cK_m)$ is a $\mu$-Kuranishi space with corners with $\vdim\bX_m=m$ for $m\in\Z$, and $\bX=\coprod_{m\in\Z}\bX_m$.

Define {\it morphisms $\bs f:\bX\ra\bY$ of $\mu$-Kuranishi spaces with corners of mixed dimension\/} as for $\mu$-Kuranishi spaces with corners in \S\ref{ku23} and \S\ref{ku35}. Then $\mu$-Kuranishi spaces with corners of mixed dimension form a category $\cmuKurc$. Writing $\bX=\coprod_{m\in\Z}\bX_m$ and $\bY=\coprod_{n\in\Z}\bY_n$ with $\vdim\bX_m=m$, $\vdim\bY_n=n$, then $\bs f\vert_{\bX_{mn}}:\bX_{mn}\ra\bY_n$ is a morphism in $\muKurc$ for all $m,n\in\Z$, where $\bX_{mn}:=\bX_m\cap\bs f^{-1}(\bY_n)$ is open and closed in~$\bX_m\subseteq\bX$.

Alternatively, we can regard $\cmuKurc$ as the category with objects $\bX=\coprod_{m\in\Z}\bX_m$ for $\bX_m\in\Kurc$ with $\vdim\bX_m=m$, allowing $\bX_m=\es$, and with morphisms $\bs f:\coprod_{m\in\Z}\bX_m\ra\coprod_{n\in\Z}\bY_n$ consisting of decompositions $\bX_m=\coprod_{n\in\Z}\bX_{mn}$ for $m\in\Z$ with $\bX_{mn}\subseteq\bX_m$ open and closed, and morphisms $\bs f_{mn}:\bX_{mn}\ra\bY_n$ in $\muKurc$ for all $m,n\in\Z$.

Clearly $\muKurc\subset\cmuKurc$ is a full subcategory.
\label{ku3def15}
\end{dfn}

In \S\ref{ku32} we defined the corner functor $C:\Manc\ra\cManc$. Here is the analogue for $\mu$-Kuranishi spaces with corners.

\begin{dfn} We will define a functor $C:\muKurc\ra\cmuKurc$, called the {\it corner functor}. On objects $\bX\!=\!(X,\cK)$ with $\cK\!=\!\bigl(I,(V_i,E_i,s_i,\psi_i)_{i\in I},\Phi_{ij,\;i,j\in I}\bigr)$ in $\muKurc$, define
\begin{equation*}
C(\bX)=\ts\coprod_{k=0}^\iy C_k(\bX),
\end{equation*}
regarded as an object in $\cmuKurc$. Write the indexing set for the $\mu$-Kuranishi structure on $C(\bX)$ as $I\t\N$, where $(i,k)\in I\t\N$ corresponds to the $\mu$-Kuranishi neighbourhood $(V^{C_k}_i,E_i^{C_k},s_i^{C_k},\psi_i^{C_k})$ on $C_k(X)$ in \eq{ku3eq29}.

Let $\bs f:\bX\ra\bY$ be a morphism in $\muKurc$. Use notation \eq{ku2eq14}--\eq{ku2eq15} for $\bX,\bY$, write $\bs f=\bigl(f,\bs f_{ij,\;i\in I,\; j\in J}\bigr)$ as in \S\ref{ku23}, and choose representatives $(U_{ij},f_{ij},\hat f_{ij})$ for $\bs f_{ij}:(U_i,D_i,r_i,\chi_i)\ra (V_j,E_j,s_j,\psi_j)$ for $i\in I$ and $j\in J$.

Let $x\in\bX$ with $y=\bs f(x)\in\bY$. Suppose $\ga$ is a local $k$-corner 
component of $\bX$ at $x$ for $k\ge 0$, as in Definition \ref{ku3def14}, so that $(x,\ga)\in C_k(X)\subseteq C(X)$. Choose $i\in I$, $j\in J$ with $x\in\Im\chi_i\subseteq X$ and $y\in\Im \psi_j\subseteq Y$, and set $u_i=\chi_i^{-1}(x)\in U_{ij}\subseteq U_i$ and $v_j=\psi_j^{-1}(y)\in V_j$. Then $\ga=[\ga_i]$ for some unique local $k$-corner component $\ga_i$ of $U_{ij}$ at $u_i$ by \eq{ku3eq26}, where $[\ga_i]$ is the $\approx$-equivalence class of~$\ga_i$.

Since $f_{ij}:U_{ij}\ra V_j$ is a smooth map in $\Manc$ with $f_{ij}(u_i)=v_j$, Definition \ref{ku3def5} defines a local $l$-corner component $(f_{ij})_*(\ga_i)$ of $V_j$ at $v_j$ for some $l\ge 0$, so $[(f_{ij})_*(\ga_i)]$ is a local $l$-corner component of $\bY$ at $y$ by \eq{ku3eq26}. Using Definition \ref{ku2def13}(a),(b) and the definition of $\approx$ in \eq{ku3eq26}, we can show that $[(f_{ij})_*(\ga_i)]$ is independent of the choice of $i,j$, so $\bs f_*(\ga):=[(f_{ij})_*(\ga_i)]$ is well-defined. 

Define $C(\bs f)=\bigl(C(f),C(\bs f)_{(i,k)(j,l),\;(i,k)\in I\t\N,\; (j,l)\in J\t\N}\bigr)$, where $C(f):C(X)\ab\ra C(Y)$ maps $C(f):(x,\ga)\mapsto (f(x),\bs f_*(\ga))$, and for all $(i,k)\in I\t\N$, $(j,l)\in J\t\N$ we have $C(\bs f)_{(i,k)(j,l)}=\bigl[U_{ij}^{C_{k,l}},f_{ij}^{C_{k,l}},\hat f_{ij}^{C_{k,l}}\bigr]$, where in a similar way to \eq{ku3eq30}, if $(U_{ij},f_{ij},\hat f_{ij})$ is a representative for $\bs f_{ij}$ then $U_{ij}^{C_{k,l}}=C_k(U_{ij})\cap C(f_{ij})^{-1}(C_l(V_j))\subseteq C_k(U_{ij})$, and $f_{ij}^{C_{k,l}}=C(f_{ij})\vert_{U_{ij}^{C_{k,l}}}:U_{ij}^{C_{k,l}}\ra C_l(V_j)$, and $\hat f_{ij}^{C_{k,l}}=i_{U_{ij}}\vert_{U_{ij}^{C_{k,l}}}^*(\hat f_{ij})$. It is now straightforward to check using functoriality of $C$ on $\Manc$ that $C(\bs f)_{(i,k)(j,l)}$ is independent of the choice of $(U_{ij},f_{ij},\hat f_{ij})$, and $C(\bs f):C(\bX)\ra C(\bY)$ is a morphism in $\cmuKurc$. It is also easy to show that if $\bs g:\bY\ra\bZ$ is another morphism in $\muKurc$ then $C(\bs g\ci\bs f)=C(\bs g)\ci C(\bs f)$, and $C(\bs\id_\bX)=\bs\id_{C(\bX)}$. 
So $C:\muKurc\ra\cmuKurc$ is a functor.
\label{ku3def16}
\end{dfn}

Definition \ref{ku3def2} defines when a smooth map $f:X\ra Y$ of manifolds with corners is interior, or b-normal, or strongly smooth, or simple. We can extend these to $\mu$-Kuranishi spaces with corners.

\begin{dfn} Let $\bs f:\bX\ra\bY$ be a morphism of $\mu$-Kuranishi spaces with corners, use notation \eq{ku2eq14}--\eq{ku2eq15} for $\bX,\bY$ and $\bs f=\bigl(f,\bs f_{ij,\;i\in I,\; j\in J}\bigr)$, and let $(U_{ij},f_{ij},\hat f_{ij})$ be a representative of the $\sim$-equivalence class $\bs f_{ij}$ for each $i\in I$ and $j\in J$, so that $U_{ij}$ is an open neighbourhood of $(f\ci\chi_i)^{-1}(\Im\psi_j)$ in $U_i$, and $f_{ij}:U_{ij}\ra V_j$ is a smooth map of manifolds with corners.

We say that $\bs f$ is {\it interior}, or {\it b-normal}, or {\it strongly smooth}, or {\it simple}, if for all $i\in I$ and $j\in J$, the smooth map $f_{ij}:U_{ij}\ra V_j$ is interior, \ldots, simple, respectively. This is independent of the choice of representatives~$(U_{ij},f_{ij},\hat f_{ij})$.

Note that as interior, b-normal, strongly smooth, and simple are {\it discrete\/} conditions on smooth maps of manifolds with corners in the sense of Remark \ref{ku3rem2}, Remarks \ref{ku3rem4}(a), \ref{ku3rem5} and Definition \ref{ku3def13}(ii) for $(U_{ij},f_{ij},\hat f_{ij})$ imply that $f_{ij}$ is interior, \ldots, simple on all of $U_{ij}$ if and only if $f_{ij}$ is interior, \ldots, simple near $(f\ci\chi_i)^{-1}(\Im\psi_j)$ in $U_{ij}$, respectively. 

Also, if $i,i'\in I$, $j,j'\in J$ and $x\in\Im\chi_i\cap\Im\chi_{i'}\cap f^{-1}(\Im\psi_j\cap\Im\psi_{j'})\subseteq X$ then as $f_{i'j'}=\up_{jj'}\ci f_{ij}\ci\tau_{i'i}+O(s_{i'})$ near $\chi_i^{-1}(x)$ with $\tau_{i'i},\up_{jj'}$ simple, one can show that $f_{ij}$ is interior, \ldots, simple near $\chi_i^{-1}(x)$ if and only if $f_{i'j'}$ is interior, \ldots, simple near $\chi_{i'}^{-1}(x)$, respectively. That is, these are really local conditions near each point $x\in\bX$, and are independent of the choice of $i\in I$, $j\in J$ with $x\in\Im\psi_i\cap f^{-1}(\Im\psi_j)$ that we check the conditions for~$\bs f_{ij}=[U_{ij},f_{ij},\hat f_{ij}]$.

The classes of interior, b-normal, strongly smooth, and simple morphisms of $\mu$-Kuranishi spaces with corners are each closed under composition and contain identities, as this holds for the corresponding classes of morphisms of manifolds with corners. Thus, each class defines subcategories of $\muKurc,\cmuKurc$. Write $\muKurcst,\muKurcin,\muKurcis,\muKurcsi$ for the subcategories of $\muKurc$ with strongly smooth morphisms, interior morphisms, interior strongly smooth morphisms, and simple morphisms, respectively. Write $\cmuKurcst$, $\cmuKurcin$, $\cmuKurcis$, $\cmuKurcsi$ for the corresponding subcategories of~$\cmuKurc$.
\label{ku3def17}
\end{dfn}

Proposition \ref{ku3prop1} now implies its analogue for $\mu$-Kuranishi spaces:

\begin{prop} Let\/ $\bs f:\bX\ra\bY$ be a morphism in $\muKurc$. 
\begin{itemize}
\setlength{\itemsep}{0pt}
\setlength{\parsep}{0pt}
\item[{\bf(a)}] $C(\bs f):C(\bX)\ra C(\bY)$ is an interior morphism of $\mu$-Kuranishi spaces with corners of mixed dimension, so $C$ is a functor $C:\muKurc\ra\cmuKurcin$.
\item[{\bf(b)}] $\bs f$ is interior if and only if\/ $C(\bs f)$ maps $C_0(\bX)\ra C_0(\bY)$.
\item[{\bf(c)}] $\bs f$ is b-normal if and only if\/ $C(\bs f)$ maps $C_k(\bX)\ra \coprod_{l=0}^kC_l(\bY)$ for all\/~$k$.
\item[{\bf(d)}] If\/ $\bs f$ is simple then $C(\bs f)$ maps $C_k(\bX)\ra C_k(\bY)$ for all\/ $k\ge 0,$ and\/ $C_k(\bs f):=C(\bs f)\vert_{C_k(\bX)}:C_k(\bX)\ra C_k(\bY)$ is also a simple morphism.
\end{itemize}

Part\/ {\bf(d)} implies that we have a \begin{bfseries}boundary functor\end{bfseries} $\pd:\muKurcsi\!\ra\!\muKurcsi$ mapping $\bX\mapsto\pd\bX$ on objects and\/ $\bs f\mapsto\pd\bs f:=C(\bs f)\vert_{C_1(\bX)}:\pd\bX\ra\pd\bY$ on (simple) morphisms $\bs f:\bX\ra\bY,$ and for all\/ $k\ge 0$ a \begin{bfseries}$k$-corner functor\end{bfseries} $C_k:\muKurcsi\ra\muKurcsi$ mapping $\bX\mapsto C_k(\bX)$ on objects and\/ $\bs f\mapsto C_k(\bs f):=C(\bs f)\vert_{C_k(\bX)}:C_k(\bX)\ra C_k(\bY)$ on morphisms.
\label{ku3prop3}
\end{prop}

\subsection{(B-)tangent spaces and (b-)obstruction spaces}
\label{ku37}

In \S\ref{ku25} we defined tangent spaces $T_x\bX$ and obstruction spaces $O_x\bX$ for $\mu$-Kuranishi spaces. As in \S\ref{ku33}, manifolds with corners $X$ have two notions of tangent bundle $TX$ and ${}^bTX$. So $\mu$-Kuranishi spaces with corners also have two kinds of tangent space $T_x\bX,{}^bT_x\bX$ and obstruction space~$O_x\bX,{}^bO_x\bX$.

\begin{dfn} Let $\bX=(X,\cK)$ be a $\mu$-Kuranishi space with corners, with $\cK=\bigl(I,(V_i,\ab E_i,\ab s_i,\ab\psi_i)_{i\in I},\Phi_{ij,\; i,j\in I}\bigr)$, and let $x\in\bX$. Then the $V_i$ and $V_{ij}\subseteq V_i$ in $\Phi_{ij}=[V_{ij},\phi_{ij},\hat\phi_{ij}]$ are manifolds with corners, and so have tangent bundles $TV_i,TV_{ij}=TV_i\vert_{V_{ij}}$ and b-tangent bundles ${}^bTV_i,{}^bTV_{ij}={}^bTV_i\vert_{V_{ij}}$ as in~\S\ref{ku33}.

Define the {\it tangent space\/} $T_x\bX$ and {\it obstruction space\/} $O_x\bX$ of $\bX$ at $x$ exactly as in Definition \ref{ku2def16}, using tangent bundles $TV_i,TV_j$, with one small change: in the part where we show that $\ka_{ij}^x=\ka^x_{V_{ij},\phi_{ij},\hat\phi_{ij}}$ and $\ga_{ij}^x=\ga^x_{V_{ij},\phi_{ij},\hat\phi_{ij}}$ are independent of the choice of representatives $(V_{ij},\phi_{ij},\hat\phi_{ij})$ for $\Phi_{ij}$, we note that two such representatives $(V_{ij},\phi_{ij},\hat\phi_{ij}),(V_{ij}',\phi_{ij}',\hat\phi_{ij}')$ are related by \eq{ku2eq27} for some $\dot V_{ij},\La$. As in Definition \ref{ku3def13}(iii), in the corners case $\La$ maps $E_i\vert_{\dot V_{ij}}\ra \phi_{ij}^*({}^bTV_j)\vert_{\dot V_{ij}}$. So we replace \eq{ku2eq27} by
\begin{align*}
\d\phi_{ij}'\vert_{v_i}&=\d\phi_{ij}\vert_{v_i}+(\phi_{ij}^*(I_{V_j})\ci\La)\vert_{v_i}\ci \d s_i\vert_{v_i}\;\>\text{and}\\ 
\hat\phi_{ij}'\vert_{v_i}&=\hat\phi_{ij}\vert_{v_i}+\d s_j\vert_{v_j}\ci(\phi_{ij}^*(I_{V_j})\ci\La)\vert_{v_i},
\end{align*}
where $I_{V_j}:{}^bTV_j\ra TV_j$ is as in Definition~\ref{ku3def7}.

As in Definition \ref{ku2def16}, we have 
\e
\dim T_x\bX-\dim O_x\bX=\vdim\bX,
\label{ku3eq32}
\e
and if $(V_a,E_a,s_a,\psi_a)$ is any $\mu$-Kuranishi neighbourhood on $\bX$ in the sense of \S\ref{ku24} and \S\ref{ku35} with $x\in\Im\psi_a$, we have a canonical exact sequence 
\e
\xymatrix@C=20pt{ 0 \ar[r] & T_x\bX \ar[r] & T_{v_a}V_a \ar[rr]^{\d s_a\vert_{v_a}} && E_a\vert_{v_a} \ar[r] & O_x\bX \ar[r] & 0, }
\label{ku3eq33}
\e
where $v_a=\psi_a^{-1}(x)$. The dual vector spaces of $T_x\bX,O_x\bX$ are called the {\it cotangent space\/} $T_x^*\bX$ and {\it coobstruction space\/} $O_x^*\bX$. 

If $\bs f:\bX\ra\bY$ is a morphism of $\mu$-Kuranishi spaces with corners and $x\in\bX$ with $y=\bs f(x)\in\bY$, as in Definition \ref{ku2def16} we define canonical, functorial linear maps $T_x\bs f:T_x\bX\ra T_y\bY$ and $O_x\bs f:O_x\bX\ra O_y\bY$, such that if $(U_a,D_a,r_a,\chi_a)$ and $(V_b,E_b,s_b,\psi_b)$ are $\mu$-Kuranishi neighbourhoods on $\bX,\bY$ respectively with $x\in\Im\psi_a$, $y\in\Im\psi_b$, and $\bs f_{ab}=[U_{ab},f_{ab},\hat f_{ab}]$ is the morphism of $\mu$-Kuranishi neighbourhoods over $\bs f$ given by Theorem \ref{ku2thm4}(b), then setting $u_a=\chi_a^{-1}(x)$, $v_b=\psi_b^{-1}(y)$, the following commutes, with rows \eq{ku3eq33} for $\bX,x$ and~$\bY,y$:
\e
\begin{gathered}
\xymatrix@C=23pt@R=15pt{ 0 \ar[r] & T_x\bX \ar[r] \ar[d]^{T_x\bs f} & T_{u_a}U_a \ar[rr]_{\d r_a\vert_{u_a}} \ar[d]^{\d f_{ab}\vert_{u_a}} && D_a\vert_{u_a} \ar[d]^{\hat f_{ab}\vert_{u_a}} \ar[r] & O_x\bX \ar[d]^{O_x\bs f} \ar[r] & 0 \\
0 \ar[r] & T_y\bY \ar[r] & T_{v_b}V_b \ar[rr]^{\d s_b\vert_{v_b}} && E_b\vert_{v_b} \ar[r] & O_y\bY \ar[r] & 0.\!\! }
\end{gathered}
\label{ku3eq34}
\e

There is also an analogue of the whole of this story for the b-tangent bundles ${}^bTV_i,{}^bTV_j$. Here, instead of \eq{ku2eq25} we start with the exact sequence 
\begin{equation*}
\xymatrix@C=20pt{ 0 \ar[r] & {}^bK_i^x \ar[r] & {}^bT_{v_i}V_i \ar[rr]^{{}^b\d s_i\vert_{v_i}} && {}^bE_i\vert_{v_i} \ar[r] & {}^bC_i^x \ar[r] & 0, }
\end{equation*}
where ${}^b\d s_i\vert_{v_i}=\d s_i\vert_{v_i}\ci I_{V_i}\vert_{v_i}$, and instead of \eq{ku2eq26} we have
\begin{equation*}
\xymatrix@C=23pt@R=15pt{ 0 \ar[r] & {}^bK_i^x \ar[r] \ar@{.>}[d]^{{}^b\ka^x_{V_{ij},\phi_{ij},\hat\phi_{ij}}}_\cong & {}^bT_{v_i}V_i \ar[rr]_(0.6){{}^b\d s_i\vert_{v_i}} \ar[d]^{{}^b\d\phi_{ij}\vert_{v_i}} && E_i\vert_{v_i} \ar[d]^{\hat\phi_{ij}\vert_{v_i}} \ar[r] & {}^bC_i^x \ar@{.>}[d]^{{}^b\ga^x_{V_{ij},\phi_{ij},\hat\phi_{ij}}}_\cong \ar[r] & 0 \\
0 \ar[r] & {}^bK_j^x \ar[r] & {}^bT_{v_j}V_j \ar[rr]^(0.6){{}^b\d s_j\vert_{v_j}} && E_j\vert_{v_j} \ar[r] & {}^bC_j^x \ar[r] & 0,\!\! }
\end{equation*}
where ${}^b\d\phi_{ij}:{}^bTV_{ij}\ra \phi_{ij}^*({}^bTV_j)$ is defined as in \S\ref{ku33} since $\phi_{ij}$ is interior (as it is simple). The rest of Definition \ref{ku2def16} generalizes with the obvious changes.

Thus, we define the {\it b-tangent space\/} ${}^bT_x\bX$ and {\it b-obstruction space\/} ${}^bO_x\bX$ of $\bX$ at $x$, which have
\e
\dim {}^bT_x\bX-\dim {}^bO_x\bX=\vdim\bX.
\label{ku3eq35}
\e
The dual vector spaces are the {\it b-cotangent space\/} ${}^bT_x^*\bX$ and {\it b-coobstruction space\/} ${}^bO_x^*\bX$. If $(V_a,E_a,s_a,\psi_a)$ is any $\mu$-Kuranishi neighbourhood on $\bX$ with $x\in\Im\psi_a$ and $v_a=\psi_a^{-1}(x)$, we have a canonical exact sequence 
\e
\xymatrix@C=20pt{ 0 \ar[r] & {}^bT_x\bX \ar[r] & {}^bT_{v_a}V_a \ar[rr]^{{}^b\d s_a\vert_{v_a}} && E_a\vert_{v_a} \ar[r] & {}^bO_x\bX \ar[r] & 0. }
\label{ku3eq36}
\e

The morphisms $I_{V_i}:{}^bTV_i\ra TV_i$ in Definition \ref{ku3def7} induce canonical linear maps $I^T_{\bX,x}:{}^bT_x\bX\ra T_x\bX$ and $I^O_{\bX,x}:{}^bO_x\bX\ra O_x\bX$, such that if $(V_a,E_a,s_a,\psi_a)$ is any $\mu$-Kuranishi neighbourhood on $\bX$ in the sense of \S\ref{ku24} with $x\in\Im\psi_a$ and $v_a=\psi_a^{-1}(x)$, we have a canonical commutative diagram
\e
\begin{gathered}
\xymatrix@C=20pt@R=13pt{ 0 \ar[r] & {}^bT_x\bX \ar[d]^{I^T_{\bX,x}} \ar[r] & {}^bT_{v_a}V_a \ar[rr]_{{}^b\d s_a\vert_{v_a}} \ar[d]^{I_{V_a}\vert_{v_a}} && E_a\vert_{v_a} \ar[d]^{\id} \ar[r] & {}^bO_x\bX \ar[d]^{I^O_{\bX,x}} \ar[r] & 0
\\
0 \ar[r] & T_x\bX \ar[r] & T_{v_a}V_a \ar[rr]^{\d s_a\vert_{v_a}} && E_a\vert_{v_a} \ar[r] & O_x\bX \ar[r] & 0,\!\!{} }
\end{gathered}
\label{ku3eq37}
\e
where the rows are \eq{ku3eq33} and \eq{ku3eq36}.

To define analogues of $T_x\bs f:T_x\bX\ra T_y\bY$ and $O_x\bs f:O_x\bX\ra O_y\bY$ for ${}^bT_x\bX,{}^bT_y\bY,{}^bO_x\bX,{}^bO_y\bY$, we must replace $\d f_{ij}\vert_{u_i}$ in \eq{ku2eq31} by ${}^b\d f_{ij}\vert_{u_i}$, where as in \S\ref{ku33} ${}^bf_{ij}:{}^bTV_{ij}\ra f_{ij}^*({}^bTV_j)$ is defined only if $f_{ij}$ is an interior map. Therefore we must restrict to $\bs f:\bX\ra\bY$ interior, in the sense of Definition \ref{ku3def17}. If $\bs f:\bX\ra\bY$ is an interior morphism of $\mu$-Kuranishi spaces with corners and $x\in\bX$ with $y=\bs f(x)\in\bY$, we define canonical, functorial linear maps ${}^bT_x\bs f:{}^bT_x\bX\ra {}^bT_y\bY$ and ${}^bO_x\bs f:{}^bO_x\bX\ra {}^bO_y\bY$, such that if $(U_a,D_a,r_a,\chi_a)$ and $(V_b,E_b,s_b,\psi_b)$ are $\mu$-Kuranishi neighbourhoods on $\bX,\bY$ respectively with $x\in\Im\psi_a$, $y\in\Im\psi_b$, and $\bs f_{ab}=[U_{ab},f_{ab},\hat f_{ab}]$ is the morphism of $\mu$-Kuranishi neighbourhoods over $\bs f$ given by Theorem \ref{ku2thm4}(b), then setting $u_a=\chi_a^{-1}(x)$, $v_b=\psi_b^{-1}(y)$, the following commutes, with rows \eq{ku3eq36} for $\bX,x$ and~$\bY,y$:
\e
\begin{gathered}
\xymatrix@C=23pt@R=15pt{ 0 \ar[r] & {}^bT_x\bX \ar[r] \ar[d]^{{}^bT_x\bs f} & {}^bT_{u_a}U_a \ar[rr]_(0.6){{}^b\d r_a\vert_{u_a}} \ar[d]^{{}^b\d f_{ab}\vert_{u_a}} && D_a\vert_{u_a} \ar[d]^{\hat f_{ab}\vert_{u_a}} \ar[r] & {}^bO_x\bX \ar[d]^{{}^bO_x\bs f} \ar[r] & 0 \\
0 \ar[r] & {}^bT_y\bY \ar[r] & {}^bT_{v_b}V_b \ar[rr]^(0.6){{}^b\d s_b\vert_{v_b}} && E_b\vert_{v_b} \ar[r] & {}^bO_y\bY \ar[r] & 0.\!\! }
\end{gathered}
\label{ku3eq38}
\e

Combining \eq{ku3eq34}, \eq{ku3eq37} for $\bX,x,\bY,y$, and \eq{ku3eq38}, we see that $T_x\bs f\ci I_{\bX,x}^T=I_{\bY,y}^T\ci{}^bT_x\bs f:{}^bT_x\bX\ra T_y\bY$ and~$O_x\bs f\ci I_{\bX,x}^O=I_{\bY,y}^O\ci{}^bO_x\bs f:{}^bO_x\bX\ra O_y\bY$.

\label{ku3def18}
\end{dfn}

As in \S\ref{ku25}, these $T_x\bX,O_x\bX,T_x\bs f,{}^bT_x\bX,{}^bO_x\bX,{}^bT_x\bs f$ are very useful for defining interesting properties of morphisms of $\mu$-Kuranishi spaces with corners. See \cite{Joyc12} for more details.

\subsection{\texorpdfstring{$\mu$-Kuranishi spaces with generalized corners}{\textmu-Kuranishi spaces with generalized corners}}
\label{ku38}

In \S\ref{ku34} we discussed the author's theory of {\it manifolds with generalized corners}, or {\it g-corners\/} \cite{Joyc9}, and we saw that most of \S\ref{ku31}--\S\ref{ku33} extends to manifolds with g-corners, including all the essentials for $\mu$-Kuranishi spaces with corners.

Therefore we can go through \S\ref{ku35}--\S\ref{ku37} replacing manifolds with corners by manifolds with g-corners throughout. Mostly this just works, but there are a few places where the extension to g-corners fails because of a corresponding failure for manifolds with (g-)corners explained in \S\ref{ku34}, and other places where we prove results using local coordinates in $[0,\iy)^k\t\R^{m-k}$ on manifolds with corners, and the same results hold for g-corners, but need changes to the proofs. 

In the following we give the new notation we use for $\mu$-Kuranishi spaces with g-corners, and draw attention to the parts of \S\ref{ku35}--\S\ref{ku37} and \S\ref{ku63}--\S\ref{ku64} that either do not extend to g-corners, or need changes to their proofs.
\begin{itemize}
\setlength{\itemsep}{0pt}
\setlength{\parsep}{0pt}
\item[(a)] In Definition \ref{ku3def12} we define {\it $\mu$-Kuranishi neighbourhoods $(V,E,s,\psi)$ on\/ $X$ with g-corners\/} as before, but taking $V,E$ to be manifolds with g-corners.

{\it Morphisms\/} and {\it $\mu$-coordinate changes\/} of $\mu$-Kuranishi neighbourhoods with g-corners over $S\subseteq X$, and the category $\muKur_S^{\rm gc}(X)$, are as in \S\ref{ku35}, but with g-corners.
\item[(b)] Proposition \ref{ku3prop2} holds for $\mu$-Kuranishi neighbourhoods with g-corners, but the proof requires modification. We use the definition of simple maps of manifolds with g-corners in \S\ref{ku343}. Rather than showing that $\bigl(a_{p,q}\bigr)$, $\bigl(b_{q,p}\bigr)$ are inverse matrices, we show $\ti M_{v_i}\phi_{ij}:\ti M_{v_i}V_i\ra \ti M_{v_j}V_j$ and $\ti M_{v_j}\phi_{ji}:\ti M_{v_j}V_j\ra\ti M_{v_i}V_i$ are inverse monoid morphisms, using the notation of \S\ref{ku343}, so that $\ti M_{v_i}\phi_{ij}$ is an isomorphism, and $\phi_{ij}$ is simple near~$v_i$.
\item[(c)] Theorem \ref{ku3thm2} holds for $\mu$-Kuranishi neighbourhoods with g-corners. Its proof in \S\ref{ku63} can be adapted to the g-corners case, using `b-metrics' as in \cite[\S 4.5.3]{Joyc9}, although the explicit coordinate expressions \eq{ku6eq43}--\eq{ku6eq45} need to be rewritten.
\item[(d)] Theorem \ref{ku3thm3} holds for $\mu$-Kuranishi neighbourhoods with g-corners, and the proof in \S\ref{ku64} can be rewritten for the g-corners case, using ideas in \cite[\S 4.2 \& \S 5.1]{Joyc9} similar to the proofs of~\cite[Th.~4.10 \& Cor.~4.14]{Joyc9}.
\item[(e)] As in the end of \S\ref{ku35}, we define a category $\muKurgc$ of {\it $\mu$-Kuranishi spaces with g-corners\/} $\bX$. Regarding $\Manc\subset\Mangc$ as a full subcategory as in \S\ref{ku341}, we see that Kuranishi spaces with corners are special examples of Kuranishi spaces with g-corners, and $\muKurc\subset\muKurgc$ is a full subcategory. We define a full and faithful functor $F_\Mangc^\muKurgc:\Mangc\hookra\muKurgc$, so we can also regard manifolds with g-corners as examples of Kuranishi spaces with g-corners. As in \S\ref{ku24}, we define {\it $\mu$-Kuranishi neighbourhoods with g-corners\/ $(V_a,E_a,s_a,\psi_a)$ on $\mu$-Kuranishi spaces with g-corners\/} $\bX$, and the analogue of Theorem \ref{ku2thm4} holds.
\item[(f)] Almost everything in \S\ref{ku36} extends to $\mu$-Kuranishi spaces with g-corners, with only trivial modifications. The exception is equation \eq{ku3eq28}, which is false for $\mu$-Kuranishi spaces with corners, since the analogue for manifolds with g-corners is false as in~\S\ref{ku342}.

As in Definition \ref{ku3def15}, we define {\it $\mu$-Kuranishi spaces with g-corners\/ $\bX$ of mixed dimension}, which form a category $\cmuKurgc$. The {\it corner functor\/} $C:\muKurgc\ra\cmuKurgc$ is then defined as in Definition \ref{ku3def16}. Following Definition \ref{ku3def17}, we define {\it interior}, {\it b-normal\/} and {\it simple\/} morphisms of $\mu$-Kuranishi spaces with g-corners, and write $\muKurgcin,\cmuKurgcin$ and $\muKurgcsi,\cmuKurgcsi$ for the subcategories of $\muKurgc,\cmuKurgc$ with interior, and simple, morphisms. Proposition \ref{ku3prop3} holds for g-corners.
\item[(g)] In \S\ref{ku37}, {\it we do not define tangent spaces\/ $T_x\bX$ and obstruction spaces\/ $O_x\bX$ for $\mu$-Kuranishi spaces with g-corners}, since as in \S\ref{ku343} we do not define tangent bundles $TX$ for manifolds with g-corners, but only b-tangent bundles ${}^bTX$. The material on {\it b-tangent spaces\/} ${}^bT_x\bX$ and {\it b-obstruction spaces\/} ${}^bO_x\bX$ in \S\ref{ku37}, and the linear maps ${}^bT_x\bs f:{}^bT_x\bX\ra{}^bT_y\bY$, ${}^bO_x\bs f:{}^bO_x\bX\ra{}^bO_y\bY$ for interior $\bs f:\bX\ra\bY$, extends immediately to $\mu$-Kuranishi spaces with g-corners~$\bX,\bY$.
\end{itemize}

Corresponding to the diagram of categories of manifolds with (g-)corners \eq{ku3eq18}, we have a diagram of categories of $\mu$-Kuranishi spaces with (g-)corners: 
\begin{equation*}
\xymatrix@C=50pt@R=15pt{
\muKurcsi  \ar[rr]_\subset \ar[d]^\subset &&
\muKurgcsi \ar[d]^\subset\\
\muKurcis \ar[r]_\subset \ar[d]^\subset & \muKurcin \ar[r]_\subset \ar[d]^\subset & \muKurgcin \ar[d]^\subset   \\
\muKurcst \ar[r]^\subset & \muKurc \ar[r]^\subset & \muKurgc.\! }
\end{equation*}

\section{The weak 2-category of Kuranishi spaces}
\label{ku4}

We now generalize \S\ref{ku2} from the `manifold' ($\mu$-) version to the `orbifold' version of Kuranishi spaces, including finite groups $\Ga$ in our Kuranishi neighbourhoods $(V,E,\Ga,s,\psi)$ on $X$, so that $X$ is locally modelled on $s^{-1}(0)/\Ga$. At the same time, we pass from categories to (weak) 2-categories, so that instead of morphisms of $\mu$-Kuranishi neighbourhoods we have 1- and 2-morphisms of Kuranishi neighbourhoods, and our Kuranishi spaces form a weak 2-category $\Kur$. 

The proofs of Theorems \ref{ku4thm1}, \ref{ku4thm2}, \ref{ku4thm8} and \ref{ku4thm9} are deferred until \S\ref{ku71}--\S\ref{ku74}. For an introduction to 2-categories, see Appendix~\ref{kuB}.

\subsection[Kuranishi neighbourhoods, 1-morphisms, and 2-morphisms]{Kuranishi neighbourhoods and their 1-morphisms \\ and 2-morphisms}
\label{ku41}

Here is our notion of Kuranishi neighbourhood, similar to those used in \cite{Fuka,FOOO1,FOOO2,FOOO3,FOOO4, FOOO5,FOOO6,FOOO7,FOOO8,McWe3,Yang1,Yang2,Yang3}, as explained in Appendix \ref{kuA}. The differences are in whether $\Ga$ must act effectively on $V$, and whether the vector bundle $E$ is assumed to be trivial.

\begin{dfn} Let $X$ be a topological space. A {\it Kuranishi neighbourhood\/} on $X$ is a quintuple $(V,E,\Ga,s,\psi)$ such that:
\begin{itemize}
\setlength{\itemsep}{0pt}
\setlength{\parsep}{0pt}
\item[(a)] $V$ is a smooth manifold (without boundary, in this section).
\item[(b)] $\pi:E\ra V$ is a real vector bundle over $V$, called the {\it obstruction bundle}.
\item[(c)] $\Ga$ is a finite group with a smooth action on $V$ (not necessarily effective), and a compatible action on $E$ preserving the vector bundle structure.
\item[(d)] $s:V\ra E$ is a $\Ga$-equivariant smooth section of $E$, called the {\it Kuranishi section}.
\item[(e)] $\psi$ is a homeomorphism from $s^{-1}(0)/\Ga$ to an open subset $\Im\psi=\bigl\{\psi(\Ga v):v\in s^{-1}(0)\bigr\}$ in $X$, called the {\it footprint\/} of~$(V,E,\Ga,s,\psi)$.
\end{itemize}
We will write $\bar\psi:s^{-1}(0)\ra\Im\psi\subseteq X$ for the composition of $\psi$ with the projection $s^{-1}(0)\ra s^{-1}(0)/\Ga$.
\label{ku4def1}
\end{dfn}

The next two crucial definitions, of 1- and 2-morphisms of Kuranishi neighbourhoods, correspond in \S\ref{ku2} to the first and second halves of Definition \ref{ku2def3}: 1-morphisms correspond to triples $(V_{ij},\phi_{ij},\hat\phi_{ij})$ in \S\ref{ku2}, and 2-morphisms to equivalences $(V_{ij},\phi_{ij},\hat\phi_{ij})\sim(V_{ij}',\phi_{ij}',\hat\phi_{ij}')$ in \S\ref{ku2}, so that morphisms $\Phi_{ij}=[V_{ij},\phi_{ij},\hat\phi_{ij}]$ in \S\ref{ku2} correspond to 2-isomorphism classes of 1-morphisms below.

\begin{dfn} Let $X,Y$ be topological spaces, $f:X\ra Y$ a continuous map, $(V_i,E_i,\Ga_i,s_i,\psi_i)$, $(V_j,E_j,\Ga_j,s_j,\psi_j)$ be Kuranishi neighbourhoods on $X,Y$ respectively, and $S\subseteq\Im\psi_i\cap f^{-1}(\Im\psi_j)\subseteq X$ be an open set. A 1-{\it morphism $\Phi_{ij}=(P_{ij},\pi_{ij},\phi_{ij},\hat\phi_{ij}):(V_i,E_i,\Ga_i,s_i,\psi_i)\ra (V_j,E_j,\Ga_j,s_j,\psi_j)$ of Kuranishi neighbourhoods over\/} $(S,f)$ is a quadruple $(P_{ij},\pi_{ij},\phi_{ij},\hat\phi_{ij})$ satisfying 
\begin{itemize}
\setlength{\itemsep}{0pt}
\setlength{\parsep}{0pt}
\item[(a)] $P_{ij}$ is a manifold, with commuting smooth actions of $\Ga_i,\Ga_j$ (that is, with a smooth action of $\Ga_i\t\Ga_j$), with the $\Ga_j$-action free.
\item[(b)] $\pi_{ij}:P_{ij}\ra V_i$ is a smooth map which is $\Ga_i$-equivariant, $\Ga_j$-invariant, and \'etale (a local diffeomorphism). The image $V_{ij}:=\pi_{ij}(P_{ij})$ is a $\Ga_i$-invariant open neighbourhood of $\bar\psi_i^{-1}(S)$ in $V_i$, and the fibres $\pi_{ij}^{-1}(v)$ of $\pi_{ij}$ for $v\in V_{ij}$ are $\Ga_j$-orbits, so that $\pi_{ij}:P_{ij}\ra V_{ij}$ is a principal $\Ga_j$-bundle.

We do not require $\bar\psi_i^{-1}(S)=V_{ij}\cap s_i^{-1}(0)$, only that~$\bar\psi_i^{-1}(S)\subseteq V_{ij}\cap s_i^{-1}(0)$.
\item[(c)] $\phi_{ij}:P_{ij}\ra V_j$ is a $\Ga_i$-invariant and $\Ga_j$-equivariant smooth map, that is, $\phi_{ij}(\ga_i\cdot p)=\phi_{ij}(p)$, $\phi_{ij}(\ga_j\cdot p)=\ga_j\cdot\phi_{ij}(p)$ for all $\ga_i\in\Ga_i$, $\ga_j\in\Ga_j$, $p\in P_{ij}$.
\item[(d)] $\hat\phi_{ij}:\pi_{ij}^*(E_i)\ra\phi_{ij}^*(E_j)$ is a $\Ga_i$- and $\Ga_j$-equivariant morphism of vector bundles on $P_{ij}$, where the $\Ga_i,\Ga_j$-actions are induced by the given $\Ga_i$-action and the trivial $\Ga_j$-action on $E_i$, and vice versa for $E_j$.
\item[(e)] $\hat\phi_{ij}(\pi_{ij}^*(s_i))=\phi_{ij}^*(s_j)+O(\pi_{ij}^*(s_i)^2)$, in the sense of Definition~\ref{ku2def1}.
\item[(f)] $f\ci\bar\psi_i\ci\pi_{ij}=\bar\psi_j\ci\phi_{ij}$ on~$\pi_{ij}^{-1}(s_i^{-1}(0))\subseteq P_{ij}$.
\end{itemize}

If $Y=X$ and $f=\id_X$ then we call $\Phi_{ij}$ a 1-{\it morphism of Kuranishi neighbourhoods over\/} $S$, or just a 1-{\it morphism over\/}~$S$.
\label{ku4def2}
\end{dfn}

\begin{dfn} Suppose $X,Y$ are topological spaces, $f:X\ra Y$ is a continuous map, $(V_i,E_i,\Ga_i,s_i,\psi_i)$, $(V_j,E_j,\Ga_j,s_j,\psi_j)$ are Kuranishi neighbourhoods on $X,Y$ respectively, $S\subseteq\Im\psi_i\cap f^{-1}(\Im\psi_j)\subseteq X$ is open, and $\Phi_{ij},\Phi_{ij}':(V_i,E_i,\Ga_i,s_i,\psi_i)\ra (V_j,\ab E_j,\ab\Ga_j,\ab s_j,\ab\psi_j)$ are two 1-morphisms over $(S,f)$, with $\Phi_{ij}=(P_{ij},\pi_{ij},\phi_{ij},\hat\phi_{ij})$ and~$\Phi_{ij}'=(P_{ij}',\pi_{ij}',\phi_{ij}',\hat\phi_{ij}')$.

Consider triples $(\dot P_{ij},\la_{ij},\hat\la_{ij})$ satisfying:
\begin{itemize}
\setlength{\itemsep}{0pt}
\setlength{\parsep}{0pt}
\item[(a)] $\dot P_{ij}$ is a $\Ga_i$- and $\Ga_j$-invariant open neighbourhood of $\pi_{ij}^{-1}(\bar\psi_i^{-1}(S))$ in $P_{ij}$.
\item[(b)] $\la_{ij}:\dot P_{ij}\ra P_{ij}'$ is a $\Ga_i$- and $\Ga_j$-equivariant smooth map with $\pi_{ij}'\ci\la_{ij}=\pi_{ij}\vert_{\dot P_{ij}}$. This implies that $\la_{ij}$ is an isomorphism of principal $\Ga_j$-bundles over $\dot V_{ij}:=\pi_{ij}(\dot P_{ij})$, so $\la_{ij}$ is a diffeomorphism with a $\Ga_i$- and $\Ga_j$-invariant open set $\la_{ij}(\dot P_{ij})$ in~$P_{ij}'$.
\item[(c)] $\hat\la_{ij}:\pi_{ij}^*(E_i)\vert_{\dot P_{ij}}\ra \phi_{ij}^*(TV_j)\vert_{\dot P_{ij}}$ is a $\Ga_i$- and $\Ga_j$-invariant smooth morphism of vector bundles on $\dot P_{ij}$, satisfying
\e
\begin{split}
\phi_{ij}'\ci\la_{ij}&=\phi_{ij}\vert_{\dot P_{ij}}+\hat\la_{ij}\cdot \pi_{ij}^*(s_i)+O\bigl(\pi_{ij}^*(s_i)^2\bigr)\;\>\text{and}\\ 
\la_{ij}^*(\hat\phi_{ij}')&=\hat\phi_{ij}\vert_{\dot P_{ij}}+\hat\la_{ij}\cdot \phi_{ij}^*(\d s_j)+O\bigl(\pi_{ij}^*(s_i)\bigr)\;\> \text{on $\dot P_{ij}$,}
\end{split}
\label{ku4eq1}
\e
in the notation of Definition \ref{ku2def1}, interpreting $\d s_j$ as in \eq{ku2eq1}.
\end{itemize}

Define a binary relation $\approx$ on such triples by $(\dot P_{ij},\la_{ij},\hat\la_{ij})\approx(\dot P_{ij}',\la_{ij}',\hat\la_{ij}')$ if there exists an open neighbourhood $\ddot P_{ij}$ of $\pi_{ij}^{-1}(\bar\psi_i^{-1}(S))$ in $\dot P_{ij}\cap \dot P_{ij}'$ with
\e
\la_{ij}\vert_{\ddot P_{ij}}=\la_{ij}'\vert_{\ddot P_{ij}}\quad\text{and}\quad \hat\la_{ij}\vert_{\ddot P_{ij}}=\hat\la_{ij}'\vert_{\ddot P_{ij}}+O\bigl(\pi_{ij}^*(s_i)\bigr)\quad\text{on $\ddot P_{ij}$.}
\label{ku4eq2}
\e
Then $\approx$ is an equivalence relation. We also write $\approx_S$ for $\approx$ if we wish to stress the open set $S$. Write $[\dot P_{ij},\la_{ij},\hat\la_{ij}]$ for the $\approx$-equivalence class of $(\dot P_{ij},\la_{ij},\hat\la_{ij})$. We say that $[\dot P_{ij},\la_{ij},\hat\la_{ij}]:\Phi_{ij}\Ra\Phi_{ij}'$ is a 2-{\it morphism of\/ $1$-morphisms of Kuranishi neighbourhoods on\/ $X$ over\/} $(S,f)$, or just a 2-{\it morphism over\/} $(S,f)$. We often write~$\La_{ij}=[\dot P_{ij},\la_{ij},\hat\la_{ij}]$.

If $Y=X$ and $f=\id_X$ then we call $\La_{ij}$ a 2-{\it morphism of Kuranishi neighbourhoods over\/} $S$, or just a 2-{\it morphism over\/}~$S$.
\label{ku4def3}
\end{dfn}

\begin{rem} If $\Ga_i=\Ga_j=\{1\}$ then morphisms $[V_{ij},\phi_{ij},\hat\phi_{ij}]:(V_i,E_i,s_i,\psi_i)\ra (V_j,E_j,s_j,\psi_j)$ in Definition \ref{ku2def3} correspond naturally to 2-isomorphism classes of 1-morphisms $\Phi_{ij}\!=\!(V_{ij},\id_{V_{ij}},\phi_{ij},\hat\phi_{ij}):(V_i,E_i,\{1\},s_i,\psi_i)\!\ra\!(V_j,E_j,\{1\},s_j,\psi_j)$ in Definitions \ref{ku4def2} and~\ref{ku4def3}.
\label{ku4rem1}
\end{rem}

In Definition \ref{ku2def5} we defined a category $\muKur_S(X)$ of $\mu$-Kuranishi neighbourhoods over $S\subseteq X$. The analogue in this section will be a weak 2-category $\Kur_S(X)$ of Kuranishi neighbourhoods over $S\subseteq X$. We have defined the 1-morphisms and 2-morphisms. Next we define the remaining structures of a weak 2-category: composition of 1-morphisms, horizontal and vertical composition of 2-morphisms, identity 1- and 2-morphisms, and coherence 2-isomorphisms $\al_{g,f,e},\be_f,\ga_f$ as in \eq{kuBeq6} and~\eq{kuBeq9}.

\begin{dfn} Let $X,Y,Z$ be topological spaces, $f:X\ra Y$, $g:Y\ra Z$ be continuous maps, $(V_i,E_i,\Ga_i,s_i,\psi_i),(V_j,E_j,\Ga_j,s_j,\psi_j),(V_k,E_k,\Ga_k,s_k,\psi_k)$ be Kuranishi neighbourhoods on $X,Y,Z$ respectively, and $T\subseteq \Im\psi_j\cap g^{-1}(\Im\psi_k)\ab\subseteq Y$ and $S\subseteq\Im\psi_i\cap f^{-1}(T)\subseteq X$ be open. Suppose $\Phi_{ij}=(P_{ij},\pi_{ij},\phi_{ij},\hat\phi_{ij}):(V_i,E_i,\Ga_i,s_i,\psi_i)\ra (V_j,E_j,\Ga_j,s_j,\psi_j)$ is a 1-morphism of Kuranishi neighbourhoods over $(S,f)$, and $\Phi_{jk}=(P_{jk},\pi_{jk},\phi_{jk},\hat\phi_{jk}):(V_j,E_j,\Ga_j,s_j,\psi_j)\ra (V_k,E_k,\ab\Ga_k,\ab s_k,\ab\psi_k)$ is a 1-morphism of Kuranishi neighbourhoods over~$(T,g)$.

Consider the diagram of manifolds and smooth maps:
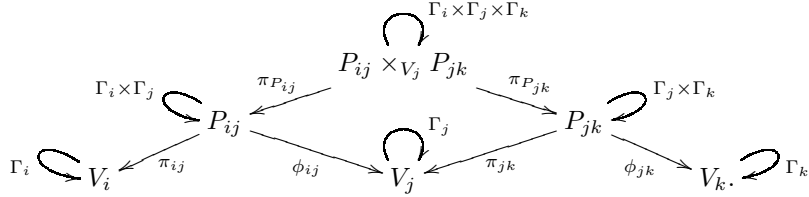
\begin{equation*}
\xymatrix@C=30pt@R=6pt{ 
&& P_{ij}\t_{V_j}P_{jk} \ar@(ul,ur)[]^(0.7){\Ga_i\t\Ga_j\t\Ga_k} \ar[dl]_(0.65){\pi_{P_{ij}}} \ar[dr]^(0.65){\pi_{P_{jk}}} \\
& P_{ij} \ar@(ul,l)[]_(0.5){\Ga_i\t\Ga_j} \ar[dl]^{\pi_{ij}} \ar[dr]_{\phi_{ij}} && P_{jk} \ar@(ur,r)[]^(0.5){\Ga_j\t\Ga_k} \ar[dl]^{\pi_{jk}} \ar[dr]_{\phi_{jk}} \\
V_i \ar@(ul,l)[]_(0.6){\Ga_i} && V_j \ar@(ul,ur)[]^(0.7){\Ga_j} && V_k. \ar@(ur,r)[]^(0.6){\Ga_k} }
\end{equation*}
Here the fibre product $P_{ij}\t_{V_j}P_{jk}$ is transverse, and so exists, as $\pi_{jk}$ is \'etale. We have shown the actions of various combinations of $\Ga_i,\Ga_j,\Ga_k$ on each space. In fact $\Ga_i\t\Ga_j\t\Ga_k$ acts on the whole diagram, with all maps equivariant, but we have omitted the trivial actions (for instance, $\Ga_j,\Ga_k$ act trivially on~$V_i$).

As $\Ga_j$ acts freely on $P_{ij}$, it also acts freely on $P_{ij}\t_{V_j}P_{jk}$, so $P_{ik}:=(P_{ij}\t_{V_j}P_{jk})/\Ga_j$ is a manifold, with projection $\Pi:P_{ij}\t_{V_j}P_{jk}\ra P_{ik}$. The commuting actions of $\Ga_i,\Ga_k$ on $P_{ij}\t_{V_j}P_{jk}$ descend to commuting actions of $\Ga_i,\Ga_k$ on $P_{ik}$ such that $\Pi$ is $\Ga_i$- and $\Ga_k$-equivariant. As $\pi_{ij}\ci\pi_{P_{ij}}:P_{ij}\t_{V_j}P_{jk}\ra V_i$ and $\phi_{jk}\ci\pi_{P_{jk}}:P_{ij}\t_{V_j}P_{jk}\ra V_k$ are $\Ga_j$-invariant, they factor through $\Pi$, so there are unique smooth maps $\pi_{ik}:P_{ik}\ra V_i$ and $\phi_{ik}:P_{ik}\ra V_k$ such that $\pi_{ij}\ci\pi_{P_{ij}}=\pi_{ik}\ci\Pi$ and~$\phi_{jk}\ci\pi_{P_{jk}}=\phi_{ik}\ci\Pi$. 

Consider the diagram of vector bundles on $P_{ij}\t_{V_j}P_{jk}$:
\begin{equation*}
\xymatrix@C=15pt@R=13pt{
*+[r]{\Pi^*\ci\pi_{ik}^*(E_i)} \ar@{.>}[rrrrrrr]_{\Pi^*(\hat\phi_{ik})} \ar@{=}[d] &&&&&&& *+[l]{\Pi^*\ci \phi_{ik}^*(E_k)} \ar@{=}[d] \\
*+[r]{\pi_{P_{ij}}^*\!\ci\!\pi_{ij}^*(E_i)} \ar[rrr]^(0.6){\pi_{P_{ij}}^*(\hat\phi_{ij})} 
&&& \pi_{P_{ij}}^*\!\ci\!\phi_{ij}^*(E_j) \ar@{=}[r] &
\pi_{P_{jk}}^*\!\ci\!\pi_{jk}^*(E_j) \ar[rrr]^(0.4){\pi_{P_{jk}}^*(\hat\phi_{jk})} &&&
*+[l]{\pi_{P_{jk}}^*\!\ci\!\phi_{jk}^*(E_k).\!\!} }
\end{equation*}
There is a unique morphism on the top line making the diagram commute. As $\hat\phi_{ij},\hat\phi_{jk}$ are $\Ga_j$-equivariant, this is $\Ga_j$-equivariant, so it is the pullback under $\Pi^*$ of a unique morphism $\hat\phi_{ik}:\pi_{ik}^*(E_i)\ra \phi_{ik}^*(E_k)$, as shown. It is now easy to check that $(P_{ik},\pi_{ik},\phi_{ik},\hat\phi_{ik})$ satisfies Definition \ref{ku4def2}(a)--(f), and is a 1-morphism $\Phi_{ik}=(P_{ik},\pi_{ik},\phi_{ik},\hat\phi_{ik}):(V_i,E_i,\Ga_i,s_i,\psi_i)\ra(V_k,E_k,\Ga_k,s_k,\psi_k)$ over $(S,g\ci f)$. We write $\Phi_{jk}\ci\Phi_{ij}=\Phi_{ik}$, and call it the {\it composition of\/ $1$-morphisms}. 

If we have three such 1-morphisms $\Phi_{ij},\Phi_{jk},\Phi_{kl}$, define
\e
\la_{ijkl}:\bigl[P_{ij}\t_{V_j}\bigl((P_{jk}\t_{V_k}P_{kl})/\Ga_k\bigr)\bigr]/\Ga_j\ra\bigl[\bigl((P_{ij}\t_{V_j}P_{jk})/\Ga_j\bigr)\t_{V_k}P_{kl}\bigr]/\Ga_k
\label{ku4eq3}
\e
to be the natural identification. Then we have a 2-isomorphism
\e
\begin{split}
\bs\al_{\Phi_{kl},\Phi_{jk},\Phi_{ij}}:=\bigl[[P_{ij}\t_{V_j}((P_{jk}&\t_{V_k}P_{kl})/\Ga_k)]/\Ga_j,\la_{ijkl},0\bigr]:\\
&(\Phi_{kl}\ci\Phi_{jk})\ci\Phi_{ij}\Longra\Phi_{kl}\ci(\Phi_{jk}\ci\Phi_{ij}).
\end{split}
\label{ku4eq4}
\e
That is, composition of $1$-morphisms is associative up to canonical 2-iso\-mor\-ph\-ism, as for weak 2-categories in~\S\ref{kuB1}.

Let $(V_i,E_i,\Ga_i,s_i,\psi_i)$ be a Kuranishi neighbourhood on $X$, and $S\subseteq\Im\psi_i$ be open. We will define the {\it identity\/ $1$-morphism\/}
\e
\id_{(V_i,E_i,\Ga_i,s_i,\psi_i)}\!=\!(P_{ii},\pi_{ii},\phi_{ii},\hat\phi_{ii})\!:\!(V_i,E_i,\Ga_i,s_i,\psi_i)\!\ra\! (V_i,E_i,\Ga_i,s_i,\psi_i).
\label{ku4eq5}
\e
Since $P_{ii}$ must have two different actions of $\Ga_i$, for clarity we write $\Ga_i^1=\Ga_i^2=\Ga_i$, where $\Ga_i^1$ and $\Ga_i^2$ mean the copies of $\Ga_i$ acting on the domain and target of the 1-morphism in \eq{ku4eq5}, respectively.
 
Define $P_{ii}=V_i\t\Ga_i$, and let $\Ga_i^1$ act on $P_{ii}$ by $\ga^1:(v,\ga)\mapsto (\ga^1\cdot v,\ga(\ga^1)^{-1})$ and $\Ga_i^2$ act on $P_{ii}$ by $\ga^2:(v,\ga)\mapsto(v,\ga^2\ga)$. Define $\pi_{ii},\phi_{ii}:P_{ii}\ra V_i$ by $\pi_{ii}:(v,\ga)\mapsto v$ and $\phi_{ii}:(v,\ga)\mapsto\ga\cdot v$. Then $\pi_{ii}$ is $\Ga_i^1$-equivariant and $\Ga_i^2$-invariant, and is a $\Ga_i^2$-principal bundle, and $\phi_{ii}$ is $\Ga_i^1$-invariant and $\Ga_i^2$-equivariant.

At $(v,\ga)\in P_{ii}$, the morphism $\hat\phi_{ii}:\pi_{ii}^*(E_i)\ra\phi_{ii}^*(E_i)$ must map $E_i\vert_v\ra E_i\vert_{\ga\cdot v}$. We have such a map, the lift of the $\ga$-action on $V_i$ to $E_i$. So we define $\hat\phi_{ii}$ on $V_i\t\{\ga\}\subseteq P_{ii}$ to be the lift to $E_i$ of the $\ga$-action on $V_i$, for each $\ga\in\Ga$. It is now easy to check that $(P_{ii},\pi_{ii},\phi_{ii},\hat\phi_{ii})$ satisfies Definition \ref{ku4def2}(a)--(f), so \eq{ku4eq5} is a 1-morphism over $S$.

For $\Phi_{ij}:(V_i,E_i,\Ga_i,s_i,\psi_i)\ra (V_j,E_j,\Ga_j,s_j,\psi_j)$ as above, define 
\begin{align*}
\mu_{ij}&:((V_i\t\Ga_i)\t_{V_i}P_{ij})/\Ga_i\longra P_{ij},\\
\nu_{ij}&:(P_{ij}\t_{V_j}(V_j\t\Ga_j))/\Ga_j\longra P_{ij},
\end{align*}
to be the natural identifications. Then we have 2-isomorphisms
\e
\begin{split}
\bs\be_{\Phi_{ij}}:=\bigl[((V_i\t\Ga_i)\t_{V_i}P_{ij})/\Ga_i,\mu_{ij},0\bigr]&:\Phi_{ij}\ci\id_{(V_i,E_i,\Ga_i,s_i,\psi_i)}\Longra\Phi_{ij},\\
\bs\ga_{\Phi_{ij}}:=\bigl[(P_{ij}\t_{V_j}(V_j\t\Ga_j))/\Ga_j,\nu_{ij},0\bigr]&:\id_{(V_j,E_j,\Ga_j,s_j,\psi_j)}\ci\Phi_{ij}\Longra\Phi_{ij},
\end{split}
\label{ku4eq6}
\e
so identity 1-morphisms behave as they should up to canonical 2-isomorphism, as for weak 2-categories in~\S\ref{kuB1}.
\label{ku4def4}
\end{dfn}

\begin{dfn} Let $X,Y$ be topological spaces, $f:X\ra Y$ be continuous, $(V_i,E_i,\Ga_i,s_i,\psi_i)$, $(V_j,E_j,\Ga_j,s_j,\psi_j)$ be Kuranishi neighbourhoods on $X,Y$, $S\subseteq \Im\psi_i\cap f^{-1}(\Im\psi_j)\subseteq X$ be open, and $\Phi_{ij},\Phi_{ij}',\Phi_{ij}'':(V_i,\ab E_i,\ab\Ga_i,\ab s_i,\ab\psi_i)\ra (V_j,E_j,\Ga_j,s_j,\psi_j)$ be 1-morphisms over $(S,f)$ with $\Phi_{ij}=(P_{ij},\pi_{ij},\phi_{ij},\hat\phi_{ij})$, $\Phi_{ij}'=(P_{ij}',\pi_{ij}',\phi_{ij}',\hat\phi_{ij}')$, $\Phi_{ij}''=(P_{ij}'',\pi_{ij}'',\phi_{ij}'',\hat\phi_{ij}'')$. Suppose $\La_{ij}=[\dot P_{ij},\la_{ij},\hat\la_{ij}]:\Phi_{ij}\Ra\Phi_{ij}'$ and $\La_{ij}'=[\dot P_{ij}',\la_{ij}',\hat\la_{ij}']:\Phi_{ij}'\Ra\Phi_{ij}''$ are 2-morphisms over $(S,f)$. We will define the {\it vertical composition of\/ $2$-morphisms over\/} $(S,f)$, written
\begin{equation*}
\La_{ij}'\od\La_{ij}=[\dot P_{ij}',\la_{ij}',\hat\la_{ij}']\od [\dot P_{ij},\la_{ij},\hat\la_{ij}]:\Phi_{ij}\Longra\Phi_{ij}''.
\end{equation*}

Choose representatives $(\dot P_{ij},\la_{ij},\hat\la_{ij}),(\dot P_{ij}',\la_{ij}',\hat\la_{ij}')$ in the $\approx$-equivalence cla\-sses $[\dot P_{ij},\la_{ij},\hat\la_{ij}],[\dot P_{ij}',\la_{ij}',\hat\la_{ij}']$. Define $\dot P_{ij}''=\la_{ij}^{-1}(\dot P_{ij})\subseteq \dot P_{ij}\subseteq P_{ij}$, and $\la_{ij}''=\la_{ij}'\ci\la_{ij}\vert_{\dot P_{ij}''}$. Consider the morphism of vector bundles
\begin{equation*}
\la_{ij}^*(\hat\la_{ij}'):\pi_{ij}^*(E_i)\vert_{\dot P_{ij}''}\!=\!\la_{ij}^*\ci\pi_{ij}^{\prime *}(E_i)\vert_{\dot P_{ij}''}\!\longra\!\la_{ij}^*\ci\phi_{ij}^{\prime *}(TV_j)\!=\!(\phi_{ij}'\ci\la_{ij})^*(TV_j)\vert_{\dot P_{ij}''}.
\end{equation*}
Since $\phi_{ij}'\ci\la_{ij}\vert_{\dot P_{ij}''}=\phi_{ij}\vert_{\dot P_{ij}''}+O(\pi_{ij}^*(s_i))$ by \eq{ku4eq1}, the discussion after Definition \ref{ku2def1}(vi) shows that there exists $\check\la'_{ij}:\pi_{ij}^*(E_i)\vert_{\dot P_{ij}''}\ra \phi_{ij}^*(TV_j)\vert_{\dot P_{ij}''}$ with
\e
\check\la_{ij}'=\la_{ij}\vert_{\dot P_{ij}''}^*(\hat\la_{ij}')+O(\pi_{ij}^*(s_i)),
\label{ku4eq7}
\e
as in Definition \ref{ku2def1}(vi), and $\check\la_{ij}'$ is unique up to $O(\pi_{ij}^*(s_i))$. By averaging over the $\Ga_i\t\Ga_j$-action we can suppose $\check\la_{ij}'$ is $\Ga_i$- and $\Ga_j$-equivariant, as $\hat\la_{ij}'$ is.

Define $\hat\la_{ij}'':\pi_{ij}^*(E_i)\vert_{\dot P_{ij}''}\ra \phi_{ij}^*(TV_j)\vert_{\dot P_{ij}''}$ by $\hat\la_{ij}''=\hat\la_{ij}\vert_{\dot P_{ij}''}+\check\la_{ij}'$. It is now easy to see that $(\dot P_{ij}'',\la_{ij}'',\hat\la_{ij}'')$ satisfies Definition \ref{ku4def3}(a)--(c) for $\Phi_{ij},\Phi_{ij}''$, using \eq{ku4eq1} for $\hat\la_{ij},\hat\la_{ij}'$ and \eq{ku4eq7} to prove \eq{ku4eq1} for $\hat\la_{ij}''$. Hence $\La_{ij}''=[\dot P_{ij}'',\la_{ij}'',\hat\la_{ij}'']:\Phi_{ij}\Ra\Phi_{ij}''$ is a 2-morphism over $(S,f)$. It is independent of choices. We define $[\dot P_{ij}',\la_{ij}',\hat\la_{ij}']\od [\dot P_{ij},\la_{ij},\hat\la_{ij}]=[\dot P_{ij}'',\la_{ij}'',\hat\la_{ij}'']$, or~$\La_{ij}'\od\La_{ij}=\La_{ij}''$.

For vertical composition of three 2-morphisms, at the level of representatives $(\dot P_{ij},\la_{ij},\hat\la_{ij})$, the definitions of $\dot P_{ij}'',\la_{ij}''$ above are strictly associative, and the definition of $\hat\la_{ij}''$ is associative up to $O(\pi_{ij}^*(s_i))$. So passing to $\approx$-equivalence classes, vertical composition of 2-morphisms is associative.

For a 1-morphism $\Phi_{ij}=(P_{ij},\pi_{ij},\phi_{ij},\hat\phi_{ij})$, define the {\it identity\/ $2$-morphism\/}
\begin{equation*}
\id_{\Phi_{ij}}=[P_{ij},\id_{P_{ij}},0]:\Phi_{ij}\Longra\Phi_{ij}.
\end{equation*}
Clearly, it acts as the identity under vertical composition of 2-morphisms.

Let $\La_{ij}:\Phi_{ij}\Ra\Phi_{ij}'$ be a 2-morphism over $(S,f)$, and choose a representative $(\dot P_{ij},\la_{ij},\hat\la_{ij})$ for $\La_{ij}=[\dot P_{ij},\la_{ij},\hat\la_{ij}]$. Define $\dot P_{ij}'=\la_{ij}(\dot P_{ij})$, so that $\dot P_{ij}'\subseteq P_{ij}'$ is open and $\la_{ij}:\dot P_{ij}\ra\dot P_{ij}'$ is a diffeomorphism. Set $\la_{ij}'=\la_{ij}^{-1}:\dot P_{ij}'\ra\dot P_{ij}\subseteq P_{ij}$. Then $\dot P_{ij}'$ is $\Ga_i$- and $\Ga_j$-invariant, and $\la_{ij}'$ is $\Ga_i$- and $\Ga_j$-equivariant. 

Now $\phi_{ij}'=\phi_{ij}\ci\la_{ij}'+O(\pi_{ij}^{\prime *}(s_i))$, so the discussion after Definition \ref{ku2def1}(vi) shows that there exists $\hat\la_{ij}':\pi_{ij}^{\prime *}(E_i)\vert_{\dot P_{ij}'}\ra\phi_{ij}^{\prime *}(TV_j)\vert_{\dot P_{ij}'}$ with $\hat\la_{ij}'=-\la_{ij}^{\prime *}(\hat\la_{ij})+O(\pi_{ij}^{\prime *}(s_i))$, as in Definition \ref{ku2def1}(vi). Since $\hat\la_{ij}$ is $\Ga_i,\Ga_j$-equivariant, by averaging $\hat\la_{ij}'$ over the $\Ga_i\t\Ga_j$-action we can suppose $\hat\la_{ij}'$ is $\Ga_i,\Ga_j$-equivariant. It is not difficult to show that $(\dot P_{ij}',\la_{ij}',\hat\la_{ij}')$ satisfies Definition \ref{ku4def3}(a)--(c), so that $\La_{ij}'=[\dot P_{ij}',\la_{ij}',\hat\la_{ij}']:\Phi_{ij}'\Ra\Phi_{ij}$ is a 2-morphism over $(S,f)$. One can check that $\La_{ij}'$ is a two-sided inverse $\La_{ij}^{-1}$ for $\La_{ij}$ under vertical composition. Thus, {\it all\/ $2$-morphisms over $(S,f)$ are invertible under vertical composition, that is, they are $2$-isomorphisms}.

If $f:X\ra Y$ is continuous, $(V_i,E_i,\Ga_i,s_i,\psi_i),(V_j,E_j,\Ga_j,s_j,\psi_j)$ are Kuranishi neighbourhoods on $X,Y$, and $S\subseteq\Im\psi_i\cap f^{-1}(\Im\psi_j)\subseteq X$ is open, write $\bHom_{S,f}\bigl((V_i,E_i,\Ga_i,s_i,\psi_i),(V_j,E_j,\Ga_j,s_j,\psi_j)\bigr)$ for the groupoid with objects 1-morphisms $\Phi_{ij}:(V_i,E_i,\Ga_i,s_i,\psi_i)\ab\ra (V_j,\ab E_j,\ab\Ga_j,\ab s_j,\ab\psi_j)$ over $(S,f)$, and morphisms 2-morphisms $\La_{ij}:\Phi_{ij}\Ra\Phi_{ij}'$ over~$(S,f)$.

If $X\!=\!Y$ and $f\!=\!\id_X$, we write $\bHom_S\bigl((V_i,E_i,\Ga_i,s_i,\psi_i),(V_j,E_j,\Ga_j,s_j,\psi_j)\bigr)$ in place of $\bHom_{S,f}\bigl((V_i,E_i,\Ga_i,s_i,\psi_i),(V_j,E_j,\Ga_j,s_j,\psi_j)\bigr)$.
\label{ku4def5}
\end{dfn}

\begin{dfn} Let $X,Y,Z$ be topological spaces, $f:X\ra Y$, $g:Y\ra Z$ be continuous maps, $(V_i,E_i,\Ga_i,s_i,\psi_i),(V_j,E_j,\Ga_j,s_j,\psi_j),(V_k,E_k,\Ga_k,s_k,\psi_k)$ be Kuranishi neighbourhoods on $X,Y,Z$, and $T\subseteq \Im\psi_j\cap g^{-1}(\Im\psi_k)\ab\subseteq Y$ and $S\subseteq\Im\psi_i\cap f^{-1}(T)\subseteq X$ be open. Suppose $\Phi_{ij},\Phi_{ij}':(V_i,E_i,\Ga_i,s_i,\psi_i)\ra (V_j,E_j,\Ga_j,s_j,\psi_j)$ are 1-morphisms of Kuranishi neighbourhoods over $(S,f)$, and $\La_{ij}:\Phi_{ij}\Ra\Phi_{ij}'$ is a 2-morphism over $(S,f)$, and $\Phi_{jk},\Phi_{jk}':(V_j,E_j,\Ga_j,\ab s_j,\ab\psi_j)\ab\ra (V_k,E_k,\ab\Ga_k,\ab s_k,\ab\psi_k)$ are 1-morphisms of Kuranishi neighbourhoods over $(T,g)$, and $\La_{jk}:\Phi_{jk}\Ra\Phi_{jk}'$ is a 2-morphism over~$(T,g)$.

We will define the {\it horizontal composition of\/ $2$-morphisms}, written
\e
\La_{jk}*\La_{ij}:\Phi_{jk}\ci\Phi_{ij}\Longra\Phi_{jk}'\ci\Phi_{ij}'\qquad\text{over $(S,g\ci f)$.}
\label{ku4eq8}
\e
Use our usual notation for $\Phi_{ij},\ldots,\La_{jk}$, and write $(P_{ik},\pi_{ik},\phi_{ik},\hat\phi_{ik})=\Phi_{jk}\ci\Phi_{ij}$, $(P_{ik}',\pi_{ik}',\phi_{ik}',\hat\phi_{ik}')=\Phi_{jk}'\ci\Phi_{ij}'$, as in Definition \ref{ku4def4}. Choose representatives $(\dot P_{ij},\la_{ij},\hat\la_{ij})$, $(\dot P_{jk},\la_{jk},\hat\la_{jk})$ for $\La_{ij}=[\dot P_{ij},\la_{ij},\hat\la_{ij}]$ and~$\La_{jk}=[\dot P_{jk},\la_{jk},\hat\la_{jk}]$. 

Then $P_{ik}=(P_{ij}\t_{V_j}P_{jk})/\Ga_j$, and $\dot P_{ij}\subseteq P_{ij}$, $\dot P_{jk}\subseteq P_{jk}$ are open and $\Ga_j$-invariant, so $\dot P_{ij}\t_{V_j}\dot P_{jk}$ is open and $\Ga_j$-invariant in $P_{ij}\t_{V_j}P_{jk}$. Define $\dot P_{ik}=(\dot P_{ij}\t_{V_j}\dot P_{jk})/\Ga_j$, as an open subset of $P_{ik}$. It is $\Ga_i$- and $\Ga_k$-invariant, as $\dot P_{ij}$, $\dot P_{jk}$ are $\Ga_i$- and $\Ga_k$-invariant, respectively.

The maps $\la_{ij}:\dot P_{ij}\ra P_{ij}'$, $\la_{jk}:\dot P_{jk}\ra P_{jk}'$ satisfy $\phi_{ij}'\ci\la_{ij}=\phi_{ij}\vert_{\dot P_{ij}}:\dot P_{ij}\ra V_j$ and $\pi_{jk}'\ci\la_{jk}=\pi_{jk}\vert_{\dot P_{jk}}:\dot P_{jk}\ra V_j$. Hence by properties of fibre products they induce a unique smooth map $\ti\la_{ik}:\dot P_{ij}\t_{\phi_{ij},V_j,\pi_{jk}}\dot P_{jk}\ra P_{ij}'\t_{\phi_{ij}',V_j,\pi_{jk}'}P_{jk}'$ with $\pi_{P_{ij}'}\ci\ti\la_{ik}=\la_{ij}\ci\pi_{\dot P_{ij}}$ and $\pi_{P_{jk}'}\ci\ti\la_{ik}=\la_{jk}\ci\pi_{\dot P_{jk}}$. As everything is $\Ga_j$-equivariant, $\ti\la_{ik}$ descends to the quotients by $\Ga_j$. Thus we obtain a unique smooth map
\begin{equation*}
\la_{ik}:\dot P_{ik}=(\dot P_{ij}\t_{V_j}\dot P_{jk})/\Ga_j\longra 
(P_{ij}'\t_{V_j}P_{jk}')/\Ga_j=P_{ik}'
\end{equation*}
with $\la_{ik}\ci\Pi=\Pi'\ci\ti\la_{ik}$, for $\Pi:\dot P_{ij}\t_{V_j}\dot P_{jk}\ra (\dot P_{ij}\t_{V_j}\dot P_{jk})/\Ga_j$, $\Pi':P_{ij}'\t_{V_j}P_{jk}'\ra (P_{ij}'\t_{V_j}P_{jk}')/\Ga_j$ the projections.

Define a morphism of vector bundles on $\dot P_{ij}\t_{V_j}\dot P_{jk}$
\begin{gather*}
\check\la_{ik}:\Pi^*\!\ci\!\pi_{ik}^*(E_i)\!=\!(\pi_{ij}\!\ci\!\pi_{\dot P_{ij}})^*(E_i)\!\longra\! (\phi_{jk}\!\ci\!\pi_{\dot P_{jk}})^*(TV_k)\!=\!\Pi^*\!\ci\!\phi_{ik}^*(TV_k)\\
\text{by}\;\>
\check\la_{ik}=\pi_{\dot P_{jk}}^*(\d\phi_{jk}\ci (\d\pi_{jk})^{-1})\ci\pi_{\dot P_{ij}}^*(\hat\la_{ij})+
\pi_{\dot P_{jk}}^*(\hat\la_{jk})\ci \pi_{\dot P_{ij}}^*(\hat\phi_{ij}),
\end{gather*}
where the morphisms are given in the diagram
\begin{equation*}
\xymatrix@C=80pt@R=2pt{
*+[r]{(\pi_{ij}\!\ci\!\pi_{\dot P_{ij}})^*(E_i)} 
\ar[dd]^{\pi_{\dot P_{ij}}^*(\hat\la_{ij})}
\ar[r]_(0.6){\pi_{\dot P_{ij}}^*(\hat\phi_{ij})} & (\phi_{ij}\!\ci\!\pi_{\dot P_{ij}})^*(E_j) \ar@{=}[r] & *+[l]{(\pi_{jk}\!\ci\!\pi_{\dot P_{jk}})^*(E_j)} \ar[ddd]_{\pi_{\dot P_{jk}}^*(\hat\la_{jk})} 
\\ \\
*+[r]{(\phi_{ij}\!\ci\!\pi_{\dot P_{ij}})^*(TV_j)} \ar@{=}[d]
\\
*+[r]{(\pi_{jk}\!\ci\!\pi_{\dot P_{jk}})^*(TV_j)} 
\ar@<1ex>[r]^(0.65){\pi_{\dot P_{jk}}^*((\d\pi_{jk})^{-1})} & \pi_{\dot P_{jk}}^*(T\dot P_{jk}) \ar[r]^(0.35){\pi_{\dot P_{jk}}^*(\d\phi_{jk})} 
\ar[l]^(0.35){\pi_{\dot P_{jk}}^*(\d\pi_{jk})} & 
*+[l]{(\phi_{jk}\!\ci\!\pi_{\dot P_{jk}})^*(TV_k).\!\!} }
\end{equation*}
Here $\d\pi_{jk}:T\dot P_{jk}\ra \pi_{jk}^*(TV_j)$ is invertible as $\pi_{jk}$ is \'etale. As all the ingredients are $\Ga_i,\Ga_j,\Ga_k$-invariant or equivariant, $\check\la_{ik}$ is $\Ga_j$-invariant, and so descends to $\dot P_{ik}=(\dot P_{ij}\t_{V_j}\dot P_{jk})/\Ga_j$. That is, there is a unique morphism $\hat\la_{ik}:\pi_{ik}\vert_{\dot P_{ik}}^*(E_i)\ra \phi_{ik}\vert_{\dot P_{ik}}^*(TV_k)$ of vector bundles on $\dot P_{ik}$ with $\Pi^*(\hat\la_{ik})=\check\la_{ik}$. As $\check\la_{ik}$ is $\Ga_i$- and $\Ga_k$-equivariant, so is $\hat\la_{ik}$.

One can now check that $(\dot P_{ik},\la_{ik},\hat\la_{ik})$ satisfies Definition \ref{ku4def3}(a)--(c), where \eq{ku4eq1} for $\hat\la_{ik}$ follows from adding the pullbacks to $\dot P_{ij}\t_{V_j}\dot P_{jk}$ of \eq{ku4eq1} for $\hat\la_{ij},\hat\la_{jk}$, so $\La_{ik}=[\dot P_{ik},\la_{ik},\hat\la_{ik}]$ is a 2-morphism as in \eq{ku4eq8}, which is independent of choices of $(\dot P_{ij},\la_{ij},\hat\la_{ij})$, $(\dot P_{jk},\la_{jk},\hat\la_{jk})$. We define $\La_{jk}*\La_{ij}=\La_{ik}$ in~\eq{ku4eq8}.
\label{ku4def6}
\end{dfn}

We have now defined all the structures of a weak 2-category: objects (Kuranishi neighbourhoods), 1- and 2-morphisms, their three kinds of composition, two kinds of identities, and the coherence 2-isomorphisms \eq{ku4eq4}, \eq{ku4eq6} for associativity and identities. However, our 1- and 2-morphisms are defined over an open set $S\subseteq X$, which is not part of the usual 2-category structure. There are two ways to make a genuine weak 2-category: either (a) work on a fixed topological space $X$ and open $S\subseteq X$ with 1- and 2-morphisms over $f=\id_X:X\ra X$ only; or (b) allow $X,f$ to vary but require $S=X$ throughout.

\begin{dfn} Let $X$ be a topological space and $S\subseteq X$ be open. Define the {\it weak\/ $2$-category $\Kur_S(X)$ of Kuranishi neighbourhoods over\/} $S$ to have objects Kuranishi neighbourhoods $(V_i,E_i,\Ga_i,s_i,\psi_i)$ on $X$ with $S\subseteq\Im\psi_i$, and 1- and 2-morphisms $\Phi_{ij},\La_{ij}$ to be 1- and 2-morphisms of Kuranishi neighbourhoods over $S$ (i.e. with $f=\id_X:X\ra X$), and compositions etc. are as above. 

Define the {\it weak\/ $2$-category $\GKur$ of global Kuranishi neighbourhoods\/} by:
\begin{itemize}
\setlength{\itemsep}{0pt}
\setlength{\parsep}{0pt}
\item Objects $\bigl(X,(V,E,\Ga,s,\psi)\bigr)$ in $\GKur$ are pairs of a topological space $X$ and a Kuranishi neighbourhood $(V,E,\Ga,s,\psi)$ on $X$ with~$\Im\psi=X$.
\item 1-morphisms $(f,\Phi):\bigl(X,(V,E,\Ga,s,\psi)\bigr)\ra\bigl(Y,(W,F,\De,t,\chi)\bigr)$ in $\GKur$ are pairs of a continuous map $f:X\ra Y$ and a 1-morphism $\Phi:(V,\ab E,\ab\Ga,\ab s,\ab\psi)\ab\ra(W,F,\De,t,\chi)$ of Kuranishi neighbourhoods over $(X,f)$.
\item For 1-morphisms $(f,\Phi),(g,\Psi):\bigl(X,(V,E,\Ga,s,\psi)\bigr)\ra\bigl(Y,(W,F,\De,t,\chi)\bigr),$ the 2-morphisms $\La:(f,\Phi)\Ra(g,\Psi)$ in $\GKur$ exist only if $f=g,$ and are 2-morphisms $\La:\Phi\Ra\Psi$ of Kuranishi neighbourhoods over~$(X,f)$.
\end{itemize}
Compositions, identities, etc. in $\GKur$ are as above. 

For both $\Kur_S(X)$ and $\GKur$, verifying the remaining 2-category axioms in \S\ref{kuB1} that we have not already proved above is tedious but straightforward.

All 2-morphisms in $\Kur_S(X)$, $\GKur$ are 2-isomorphisms, that is, $\Kur_S(X)$, $\GKur$ are $(2,1)$-categories.
\label{ku4def7}
\end{dfn}

The next three definitions are analogues of Definitions \ref{ku2def6}--\ref{ku2def8}.

\begin{dfn} Recall from Appendix \ref{kuB} that an {\it equivalence\/} in a 2-category $\bs\cC$ is a 1-morphism $f:A\ra B$ in $\bs\cC$ such that there exist a 1-morphism $g:B\ra A$ (called a {\it quasi-inverse\/}) and 2-isomorphisms $\eta:g\ci f\Ra\id_A$ and $\ze:f\ci g\Ra\id_B$. A 1-morphism $f:A\ra B$ is an equivalence if and only if $[f]:A\ra B$ is an isomorphism (is invertible) in the homotopy category~$\Ho(\bs\cC)$.

A 1-morphism $\Phi_{ij}:(V_i,E_i,\Ga_i,s_i,\psi_i)\ra (V_j,E_j,\Ga_j,s_j,\psi_j)$ on $X$ over $S$ is a {\it coordinate change over\/} $S$ if $\Phi_{ij}$ is an equivalence in the 2-category~$\Kur_S(X)$.
\label{ku4def8}
\end{dfn}

\begin{dfn} Let $T\subseteq S\subseteq X$ be open. Define the {\it restriction\/ $2$-functor\/} $\vert_T:\Kur_S(X)\ra\Kur_T(X)$ to map objects $(V_i,E_i,\Ga_i,s_i,\psi_i)$ to exactly the same objects, and 1-morphisms $\Phi_{ij}$ to exactly the same 1-morphisms but regarded as 1-morphisms over $T$, and 2-morphisms $\La_{ij}=[\dot P_{ij},\la_{ij},\hat\la_{ij}]$ over $S$ to $\La_{ij}\vert_T=[\dot P_{ij},\la_{ij},\hat\la_{ij}]\vert_T$, where $[\dot P_{ij},\la_{ij},\hat\la_{ij}]\vert_T$ is the $\approx_T$-equivalence class of any representative $(\dot P_{ij},\ab\la_{ij},\ab\hat\la_{ij})$ for the $\approx_S$-equivalence class~$[\dot P_{ij},\la_{ij},\hat\la_{ij}]$. 

Then $\vert_T:\Kur_S(X)\ra\Kur_T(X)$ commutes with all the structure, so it is actually a `strict 2-functor of weak 2-categories', but as this is not a very well-behaved notion, we regard $\vert_T$ as a weak 2-functor as in \S\ref{kuB2} for which the additional 2-isomorphisms $F_{g,f}$ are identities.

If $U\subseteq T\subseteq S\subseteq X$ are open then $\vert_U\ci\vert_T=\vert_U:\Kur_S(X)\ra\Kur_U(X)$.

Also $\vert_S$ gives a functor $\vert_T:\bHom_S\bigl((V_i,E_i,\Ga_i,s_i,\psi_i),(V_j,E_j,\Ga_j,s_j,\psi_j)\bigr)\ra \bHom_T\bigl((V_i,E_i,\Ga_i,s_i,\psi_i),(V_j,E_j,\Ga_j,s_j,\psi_j)\bigr)$ when $T\subseteq S\subseteq \Im\psi_i\cap\Im\psi_j$, in the notation of Definition~\ref{ku4def5}.
\label{ku4def9}
\end{dfn}

\begin{dfn} So far we have discussed 1- and 2-morphisms of Kuranishi neighbourhoods, and coordinate changes, {\it over a specified open set\/} $S\subseteq X$, or over $(S,f)$. We now make the convention that {\it when we do not specify a domain\/ $S$ for a $1$-morphism, $2$-morphism, or coordinate change, the domain should be as large as possible}. For example, if we say that $\Phi_{ij}:(V_i,E_i,\Ga_i,s_i,\psi_i)\ra (V_j,E_j,\Ga_j,s_j,\psi_j)$ is a 1-morphism (or a 1-morphism over $f:X\ra Y$) without specifying $S$, we mean that $S=\Im\psi_i\cap\Im\psi_j$ (or~$S=\Im\psi_i\cap f^{-1}(\Im\psi_j)$).

Similarly, if we write a formula involving several 2-morphisms (possibly defined on different domains), without specifying the domain $S$, we make the convention {\it that the domain where the formula holds should be as large as possible}. That is, the domain $S$ is taken to be the intersection of the domains of each 2-morphism in the formula, and we implicitly restrict each morphism in the formula to $S$ as in Definition \ref{ku4def9}, to make it make sense.
\label{ku4def10}
\end{dfn}

\begin{rem}{\bf(i)} Our coordinate changes in Definition \ref{ku4def8} are closely related to coordinate changes between Kuranishi neighbourhoods in the theory of Fukaya, Oh, Ohta and Ono \cite{Fuka, FOOO1,FOOO2,FOOO3,FOOO4,FOOO5,FOOO6,FOOO7,FOOO8,FuOn}, as described in Appendix \ref{kuA}. We explain the connection in \S\ref{kuA1}. One of the most important innovations in our theory is to introduce the notion of 2-morphism between coordinate changes.
\smallskip

\noindent{\bf(ii)} Our 1-morphisms of Kuranishi neighbourhoods involve $\smash{V_i \,{\buildrel\pi_{ij}\over\longleftarrow}\, P_{ij}\,{\buildrel\phi_{ij}\over\longra}\,V_j}$ with $\pi_{ij}$ a $\Ga_i$-equivariant principal $\Ga_j$-bundle, and $\phi_{ij}$ $\Ga_i$-invariant and $\Ga_j$-equivariant. As in \S\ref{ku45}, this is a known way of writing 1-morphisms of orbifolds $[V_i/\Ga_i]\ra[V_j/\Ga_j]$, called {\it Hilsum--Skandalis morphisms}. So the data $P_{ij},\pi_{ij},\phi_{ij}$ in $\Phi_{ij}=(P_{ij},\pi_{ij},\phi_{ij},\hat\phi_{ij})$ is very natural from the orbifold point of view.

\smallskip

\noindent{\bf(iii)} In the definition of 2-morphisms $\La_{ij}=[\dot P_{ij},\la_{ij},\hat\la_{ij}]$ in Definition \ref{ku4def3}, by restricting to arbitrarily small open neighbourhoods $\dot P_{ij}$ of $\pi_{ij}^{-1}(\bar\psi_i^{-1}(S))$ in $P_{ij}$ and then taking equivalence classes, we are in effect taking {\it germs\/} about $\bar\psi_i^{-1}(S)$ in $V_i$, or germs about $\pi_{ij}^{-1}(\bar\psi_i^{-1}(S))$ in $P_{ij}$. Fukaya--Ono's first definition of Kuranishi space \cite[\S 5]{FuOn} involved germs of Kuranishi neighbourhoods at points. We take germs at larger subsets $\bar\psi_i^{-1}(S)$ in 2-morphisms.
\label{ku4rem2}
\end{rem}

\subsection{Properties of 1- and 2-morphisms}
\label{ku42}

We now generalize \S\ref{ku22} to the orbifold/2-category case. The 2-category analogue of sheaves in Definition \ref{ku2def9} are {\it stacks on a topological space\/}:

\begin{dfn} Let $X$ be a topological space. A {\it prestack\/} (or {\it prestack in groupoids}, or 2-{\it presheaf\/}) $\bcE$ on $X$, consists of the data of a groupoid $\bcE(S)$ for every open set $S\subseteq X$, and a functor $\rho_{ST}:\bcE(S)\ra\bcE(T)$ called the {\it restriction map\/} for every inclusion $T\subseteq S\subseteq X$ of open sets, and a natural isomorphism of functors $\eta_{STU}:\rho_{TU}\ci\rho_{ST}\Ra \rho_{SU}$ for all inclusions $U\subseteq T\subseteq S\subseteq X$ of open sets, satisfying the conditions that:
\begin{itemize}
\setlength{\itemsep}{0pt}
\setlength{\parsep}{0pt}
\item[(i)] $\rho_{SS}=\id_{\bcE(S)}:\bcE(S)\ra\bcE(S)$ for all open
$S\subseteq X$, and $\eta_{SST}=\eta_{STT}=\id_{\rho_{ST}}$ for all open $T\subseteq S\subseteq X$; and
\item[(ii)] $\eta_{SUV}\od(\id_{\rho_{UV}}*\eta_{STU})=\eta_{STV}\od(\eta_{TUV}*\id_{\rho_{ST}}):\rho_{UV}\ci\rho_{TU}\ci\rho_{ST}\Longra\rho_{SV}$ for all open~$V\subseteq U\subseteq T\subseteq S\subseteq X$.
\end{itemize}

A prestack $\bcE$ on $X$ is called a {\it stack\/} (or {\it stack in groupoids}, or 2-{\it sheaf\/}) on $X$ if whenever $S\subseteq X$ is open and $\{T_i:i\in I\}$ is an open cover of $S$, then:
\begin{itemize}
\setlength{\itemsep}{0pt}
\setlength{\parsep}{0pt}
\item[(iii)] If $\al,\be:A\ra B$ are morphisms in $\bcE(S)$ and $\rho_{ST_i}(\al)=\rho_{ST_i}(\be):\rho_{ST_i}(A)\ra \rho_{ST_i}(B)$ in $\bcE(T_i)$ for all $i\in I$, then $\al=\be$.
\item[(iv)] If $A,B$ are objects of $\bcE(S)$ and $\al_i:\rho_{ST_i}(A)\ra \rho_{ST_i}(B)$ are morphisms in $\bcE(T_i)$ for all $i\in I$ with
\begin{align*}
&\eta_{ST_i(T_i\cap T_j)}(B)\ci\rho_{T_i(T_i\cap T_j)}(\al_i)\ci\eta_{ST_i(T_i\cap T_j)}(A)^{-1}\\
&\quad=\eta_{ST_j(T_i\cap T_j)}(B)\ci\rho_{T_j(T_i\cap T_j)}(\al_j)\ci\eta_{ST_j(T_i\cap T_j)}(A)^{-1}
\end{align*}
in $\bcE(T_i\cap T_j)$ for all $i,j\in I$, then there exists $\al:A\ra B$ in $\bcE(S)$ (necessarily unique by (iii)) with $\rho_{ST_i}(\al)=\al_i$ for all~$i\in I$.
\item[(v)] If $A_i\in\bcE(T_i)$ for $i\in I$ and $\al_{ij}:\rho_{T_i(T_i\cap T_j)}(A_i)\ra \rho_{T_j(T_i\cap T_j)}(A_j)$ are morphisms in $\bcE(T_i\cap T_j)$ for all $i,j\in I$ satisfying
\end{itemize}
\begin{align*}
& \\[-27pt]
&\eta_{T_k(T_j\cap T_k)(T_i\cap T_j\cap T_k)}(A_k)\!\ci\!\rho_{(T_j\cap T_k)(T_i\cap T_j\cap T_k)}(\al_{jk})\!\ci\!\eta_{T_j(T_j\cap T_k)(T_i\cap T_j\cap T_k)}(A_j)^{-1}\\
&\ci\eta_{T_j(T_i\cap T_j)(T_i\cap T_j\cap T_k)}(A_j)\!\ci\! \rho_{(T_i\cap T_j)(T_i\cap T_j\cap T_k)}(\al_{ij})\!\ci\!\eta_{T_i(T_i\cap T_j)(T_i\cap T_j\cap T_k)}(A_i)^{-1}\\
&=\!\eta_{T_k(T_i\cap T_k)(T_i\cap T_j\cap T_k)}(A_k)\!\ci\!\rho_{(T_i\cap T_k)(T_i\cap T_j\cap T_k)}(\al_{ik})\!\ci\!\eta_{T_i(T_i\cap T_k)(T_i\cap T_j\cap T_k)}(A_i)^{-1}
\end{align*}
\begin{itemize}
\setlength{\itemsep}{0pt}
\setlength{\parsep}{0pt}
\item[]for all $i,j,k\in I$, then there exist an object $A$ in $\bcE(S)$ and morphisms $\be_i:A_i\ra\rho_{ST_i}(A)$ for $i\in I$ such that for all $i,j\in I$ we have
\begin{equation*}
\eta_{ST_i(T_i\cap T_j)}(A)\ci\rho_{T_i(T_i\cap T_j)}(\be_i)=\eta_{ST_j(T_i\cap T_j)}(A)\ci\rho_{T_j(T_i\cap T_j)}(\be_j)\ci\al_{ij}.
\end{equation*}

If $\ti A,\ti\be_i$ are alternative choices then (iii),(iv) imply there is a unique isomorphism $\ga:A\ra\ti A$ in $\bcE(S)$ with $\rho_{ST_i}(\ga)=\ti\be_i\ci\be_i^{-1}$ for all~$i\in I$.
\end{itemize}
\label{ku4def11}
\end{dfn}

In the examples of stacks on topological spaces that will be important to us, we will have $\rho_{TU}\ci\rho_{ST}=\rho_{SU}$ and $\eta_{STU}=\id_{\rho_{SU}}$ for all open $U\subseteq T\subseteq S\subseteq X$. So (ii) is automatic, and all the $\eta_{\cdots}(\cdots)$ terms in (iv),(v) can be omitted.

Here is the analogue of Theorem \ref{ku2thm1}, proved in \S\ref{ku71}, which is very important in our theory. We will call Theorem \ref{ku4thm1} the {\it stack property}. We will use it in \S\ref{ku43} to construct compositions of 1- and 2-morphisms of Kuranishi spaces, and in \S\ref{ku44} to prove good behaviour of Kuranishi neighbourhoods on Kuranishi spaces. The first part of (a) is a special case of (b) when~$f=\id_X:X\ra X$.

\begin{thm}{\bf(a)} Let\/ $(V_i,E_i,\Ga_i,s_i,\psi_i),(V_j,E_j,\Ga_j,s_j,\psi_j)$ be Kuranishi nei\-ghbourhoods on a topological space $X$. For each open $S\subseteq\Im\psi_i\cap\Im\psi_j,$ set
\begin{align*}
\bcHom\bigl((V_i&,E_i,\Ga_i,s_i,\psi_i),(V_j,E_j,\Ga_j,s_j,\psi_j)\bigr)(S)\\
&=\bHom_S\bigl((V_i,E_i,\Ga_i,s_i,\psi_i),(V_j,E_j,\Ga_j,s_j,\psi_j)\bigr),
\end{align*}
as in Definition\/ {\rm\ref{ku4def5},} for all open $T\subseteq S\subseteq \Im\psi_i\cap\Im\psi_j$ define a functor
\begin{align*}
\rho_{ST}:\,&\bcHom\bigl((V_i,E_i,\Ga_i,s_i,\psi_i),(V_j,E_j,\Ga_j,s_j,\psi_j)\bigr)(S)\longra\\
&\bcHom\bigl((V_i,E_i,\Ga_i,s_i,\psi_i),(V_j,E_j,\Ga_j,s_j,\psi_j)\bigr)(T)\end{align*}
by $\rho_{ST}=\vert_T,$ as in Definition\/ {\rm\ref{ku4def9},} and for all open $U\subseteq T\subseteq S\subseteq \Im\psi_i\cap\Im\psi_j$ take the obvious isomorphism $\eta_{STU}=\id_{\rho_{SU}}:\rho_{TU}\ci\rho_{ST}\Ra \rho_{SU}$.
Then $\bcHom\bigl((V_i,E_i,\Ga_i,s_i,\psi_i),(V_j,E_j,\Ga_j,s_j,\psi_j)\bigr)$ is a stack on $\Im\psi_i\cap\Im\psi_j$. 

Coordinate changes $(V_i,E_i,\Ga_i,s_i,\psi_i)\!\ra\!(V_j,E_j,\Ga_j,s_j,\psi_j)$ also form a stack\/ $\bcEqu\bigl((V_i,E_i,\Ga_i,s_i,\psi_i),(V_j,E_j,\Ga_j,s_j,\psi_j)\bigr)$ on $\Im\psi_i\cap\Im\psi_j,$ which is a substack of\/~$\bcHom\bigl((V_i,E_i,\Ga_i,s_i,\psi_i),(V_j,E_j,\Ga_j,s_j,\psi_j)\bigr)$.
\smallskip

\noindent{\bf(b)} Let\/ $f:X\!\ra\! Y$ be continuous, and\/ $(U_i,D_i,\Be_i,r_i,\chi_i),(V_j,E_j,\ab\Ga_j,\ab s_j,\ab\psi_j)$ be Kuranishi neighbourhoods on $X,Y,$ respectively. Then $1$- and\/ $2$-morphisms from $(U_i,D_i,\Be_i,r_i,\chi_i)$ to $(V_j,E_j,\Ga_j,s_j,\psi_j)$ over $f$ form a stack\/ $\bcHom_f\bigl((U_i,\ab D_i,\ab\Be_i,\ab r_i,\ab\chi_i),\ab (V_j,E_j,\Ga_j,s_j,\psi_j)\bigr)$ on\/~$\Im\chi_i\cap f^{-1}(\Im\psi_j)\subseteq X$.
\label{ku4thm1}
\end{thm}

\begin{rem} Setting $\Ga_i=\Be_i=\Ga_j=\{1\}$ and taking 2-isomorphism classes in Theorem \ref{ku4thm1} does not imply Theorem \ref{ku2thm1}, since taking isomorphism classes in a stack only gives a presheaf of sets, not a sheaf. In fact, if $\Ga_j\ne\{1\}$ then taking 2-isomorphism classes in Theorem \ref{ku4thm1} generally yields a presheaf but not a sheaf. This is why the `$\mu$-Kuranishi' construction in \S\ref{ku2}--\S\ref{ku3} does not work in the orbifold case.
\label{ku4rem3}
\end{rem}

Here is the analogue of Theorem \ref{ku2thm2}, proved in \S\ref{ku72}.

\begin{thm} Let\/ $\Phi_{ij}=(P_{ij},\pi_{ij},\phi_{ij},\hat\phi_{ij}):(V_i,E_i,\Ga_i,s_i,\psi_i)\ra (V_j,E_j,\ab\Ga_j,\ab s_j,\ab\psi_j)$ be a $1$-morphism of Kuranishi neighbourhoods over\/ $S\subseteq X$. Let\/ $p\in\pi_{ij}^{-1}(\bar\psi_i^{-1}(S))\subseteq P_{ij},$ and set\/ $v_i=\pi_{ij}(p)\in V_i$ and\/ $v_j=\phi_{ij}(p)\in V_j$. As in\/ {\rm\eq{ku2eq8},} consider the sequence of real vector spaces:
\e
\xymatrix@C=12.5pt{ 0 \ar[r] & T_{v_i}V_i \ar[rrrrr]^(0.43){\d s_i\vert_{v_i}\op(\d\phi_{ij}\vert_p\ci\d\pi_{ij}\vert_p^{-1})} &&&&& E_i\vert_{v_i} \!\op\!T_{v_j}V_j 
\ar[rrr]^(0.56){-\hat\phi_{ij}\vert_p\op \d s_j\vert_{v_j}} &&& E_j\vert_{v_j} \ar[r] & 0. }\!\!\!\!{}
\label{ku4eq9}
\e
Here $\d\pi_{ij}\vert_p:T_pP_{ij}\ra T_{v_i}V_i$ is invertible as $\pi_{ij}$ is \'etale. Definition\/ {\rm\ref{ku4def2}(e)} implies that\/ \eq{ku4eq9} is a complex. Also consider the morphism of finite groups
\e
\begin{split}
&\rho_p:\bigl\{(\ga_i,\ga_j)\in\Ga_i\t\Ga_j:(\ga_i,\ga_j)\cdot p=p\bigr\}
\longra\bigl\{\ga_j\in\Ga_j:\ga_j\cdot v_j=v_j\bigr\},\\
&\rho_p:(\ga_i,\ga_j)\longmapsto \ga_j.
\end{split}
\label{ku4eq10}
\e

Then $\Phi_{ij}$ is a coordinate change over $S,$ in the sense of Definition\/ {\rm\ref{ku4def8},} if and only if\/ \eq{ku4eq9} is exact and\/ \eq{ku4eq10} is an isomorphism for all\/~$p\in\pi_{ij}^{-1}(\bar\psi_i^{-1}(S))$.

\label{ku4thm2}
\end{thm}

\subsection{The definition of Kuranishi space}
\label{ku43}

We can now at last give one of the main definitions of the book:

\begin{dfn} Let $X$ be a Hausdorff, second countable topological space (not necessarily compact), and $n\in\Z$. A {\it Kuranishi structure\/ $\cK$ on $X$ of virtual dimension\/} $n$ is data $\cK=\bigl(I,(V_i,E_i,\Ga_i,s_i,\psi_i)_{i\in I}$, $\Phi_{ij,\;i,j\in I}$, $\La_{ijk,\; i,j,k\in I}\bigr)$, where:
\begin{itemize}
\setlength{\itemsep}{0pt}
\setlength{\parsep}{0pt}
\item[(a)] $I$ is an indexing set (not necessarily finite).
\item[(b)] $(V_i,E_i,\Ga_i,s_i,\psi_i)$ is a Kuranishi neighbourhood on $X$ for each $i\in I$, with~$\dim V_i-\rank E_i=n$.
\item[(c)] $\Phi_{ij}=(P_{ij},\pi_{ij},\phi_{ij},\hat\phi_{ij}):(V_i,E_i,\ab\Ga_i,\ab s_i,\ab\psi_i)\ab\ra (V_j,E_j,\Ga_j,s_j,\psi_j)$ is a coordinate change for all $i,j\in I$ (as usual, defined over $S=\Im\psi_i\cap\Im\psi_j$).
\item[(d)] $\La_{ijk}=[\dot P_{ijk},\la_{ijk},\hat\la_{ijk}]:\Phi_{jk}\ci\Phi_{ij}\Ra\Phi_{ik}$ is a 2-morphism for all $i,j,k\in I$ (as usual, defined over $S=\Im\psi_i\cap\Im\psi_j\cap\Im\psi_k$).
\item[(e)] $\bigcup_{i\in I}\Im\psi_i=X$. 
\item[(f)] $\Phi_{ii}=\id_{(V_i,E_i,\Ga_i,s_i,\psi_i)}$ for all $i\in I$.
\item[(g)] $\La_{iij}=\bs\be_{\Phi_{ij}}$ and $\La_{ijj}=\bs\ga_{\Phi_{ij}}$ for all $i,j\in I$, for $\bs\be_{\Phi_{ij}},\bs\ga_{\Phi_{ij}}$ as in \eq{ku4eq6}.
\item[(h)] The following diagram of 2-morphisms over $S=\Im\psi_i\ab\cap\ab\Im\psi_j\ab\cap\ab\Im\psi_k\ab\cap\ab\Im\psi_l$ commutes for all $i,j,k,l\in I$, for $\bs\al_{\Phi_{kl},\Phi_{jk},\Phi_{ij}}$ as in \eq{ku4eq4}:
\begin{equation*}
\xymatrix@C=90pt@R=15pt{
*+[r]{(\Phi_{kl}\ci\Phi_{jk})\ci\Phi_{ij}} \ar@{=>}[d]^{\bs\al_{\Phi_{kl},\Phi_{jk},\Phi_{ij}}}
\ar@{=>}[rr]_(0.53){\La_{jkl}*\id_{\Phi_{ij}}} && *+[l]{\Phi_{jl}\ci\Phi_{ij}} \ar@{=>}[d]_{\La_{ijl}}  
\\
*+[r]{\Phi_{kl}\ci(\Phi_{jk}\ci\Phi_{ij})} \ar@{=>}[r]^(0.65){\id_{\Phi_{kl}}*\La_{ijk}}
& \Phi_{kl}\ci\Phi_{ik} \ar@{=>}[r]^{\La_{ikl} } & *+[l]{\Phi_{il}.\!} }
\end{equation*}
\end{itemize}
We call $\bX=(X,\cK)$ a {\it Kuranishi space}, of {\it virtual dimension\/}~$\vdim\bX=n$.

When we write $x\in\bX$, we mean that $x\in X$.
\label{ku4def12}
\end{dfn}

\begin{rem} Our basic assumption on the topological space $X$ of a Kuranishi space $\bX=(X,\cK)$ is that $X$ should be {\it Hausdorff and second countable}, following the usual topological assumptions on manifolds, and the definitions of d-manifolds and d-orbifolds \cite{Joyc6,Joyc7,Joyc8}. Here is how this relates to other conditions.

Since $X$ can be covered by open sets $\Im\psi_i\cong s_i^{-1}(0)/\Ga$, it is automatically {\it locally compact}, {\it locally second countable}, and {\it regular}. Hausdorff, second countable, and locally compact imply {\it paracompact}. Hausdorff, second countable, and regular imply {\it metrizable}. Compact and locally second countable, imply second countable. Metrizable implies Hausdorff.

Thus, if $\bX=(X,\cK)$ is a Kuranishi space in our sense, then $X$ is also Hausdorff, second countable, locally compact, regular, paracompact, and metrizable. Paracompactness is very useful.

The usual topological assumption in previous papers on Kuranishi spaces \cite{FOOO1,FOOO6,FuOn,McWe1,McWe2,McWe3,Yang1,Yang2,Yang3} is that $X$ is {\it compact and metrizable}. Since $X$ is automatically locally second countable as it can be covered by Kuranishi neighbourhoods, this implies that $X$ is Hausdorff and second countable.
\label{ku4rem4}
\end{rem}

Here is the analogue of Example~\ref{ku2ex2}.

\begin{ex} Let $V$ be a manifold, $E\ra V$ a vector bundle, $\Ga$ a finite group with a smooth action on $V$ and a compatible action on $E$ preserving the vector bundle structure, and $s:V\ra E$ a $\Ga$-equivariant smooth section. Set $X=s^{-1}(0)/\Ga$, with the quotient topology induced from the closed subset $s^{-1}(0)\subseteq V$. Then $X$ is Hausdorff and second countable, as $V$ is and $\Ga$ is finite. 

Define a Kuranishi structure $\cK=\bigl(\{0\},(V_0,E_0,\Ga_0,s_0,\psi_0),\Phi_{00},\La_{000}\bigr)$ on $X$ with indexing set $I=\{0\}$, one Kuranishi neighbourhood $(V_0,E_0,\Ga_0,s_0,\psi_0)$ with $V_0=V$, $E_0=E$, $\Ga_0=\Ga$, $s_0=s$ and $\psi_0=\id_X$, one coordinate change $\Phi_{00}=\id_{(V_0,E_0,\Ga_0,s_0,\psi_0)}$, and one 2-morphism $\La_{000}=\id_{\Phi_{00}}$. Then $\bX=(X,\cK)$ is a Kuranishi space, with $\vdim\bX=\dim V-\rank E$. We write~$\bS_{V,E,\Ga,s}=\bX$.
\label{ku4ex1}
\end{ex}

We will need notation to distinguish Kuranishi neighbourhoods, coordinate changes, and 2-morphisms on different Kuranishi spaces. As for \eq{ku2eq13}--\eq{ku2eq16}, we will often use the following notation for Kuranishi spaces $\bW,\bX,\bY,\bZ$:
\ea
\bW&=(W,\cH),& \cH&=\bigl(H,(T_h,C_h,\Al_i,q_h,\vp_h)_{h\in H},\; \Si_{hh'}=(O_{hh'},\pi_{hh'},\si_{hh'},
\nonumber\\
&& {}\hskip -60pt \hat\si_{hh'})_{h,h'\in H}&,\; \Io_{hh'h''}=[\dot O_{hh'h''},\io_{hh'h''},\hat\io_{hh'h''}]_{h,h',h''\in H}\bigr),
\label{ku4eq11}\\
\bX&=(X,\cI),& \cI&=\bigl(I,(U_i,D_i,\Be_i,r_i,\chi_i)_{i\in I},\;
\Tau_{ii'}=(P_{ii'},\pi_{ii'},\tau_{ii'},
\nonumber\\
&&{}\hskip -60pt \hat\tau_{ii'})_{i,i'\in I}&,\; \Ka_{ii'i''}=[\dot P_{ii'i''},\ka_{ii'i''},\hat\ka_{ii'i''}]_{i,i',i''\in I}\bigr),
\label{ku4eq12}\\
\bY&=(Y,\cJ),& \cJ&=\bigl(J,(V_j,E_j,\Ga_j,s_j,\psi_j)_{j\in J},\; \Up_{jj'}=(Q_{jj'},\pi_{jj'},\up_{jj'},
\nonumber\\
&&{}\hskip -60pt \hat\up_{jj'})_{j,j'\in J}&,\; \La_{jj'j''}=[\dot Q_{jj'j''},\la_{jj'j''},\hat\la_{jj'j''}]_{j,j',j''\in J}\bigr),
\label{ku4eq13}\\
\bZ&=(Z,\cK),& \cK&=\bigl(K,(W_k,F_k,\De_k,t_k,\om_k)_{k\in K},\; \Phi_{kk'}=(R_{kk'},\pi_{kk'},\phi_{kk'},
\nonumber\\
&&{}\hskip -60pt \hat\phi_{kk'})_{k,k'\in K}&,\; \Mu_{kk'k''}=[\dot R_{kk'k''},\mu_{kk'k''},\hat\mu_{kk'k''}]_{k,k',k''\in K}\bigr).
\label{ku4eq14}
\ea

The rest of the section will make Kuranishi spaces into a weak 2-category, as in \S\ref{kuB1}. We first define 1- and 2-morphisms of Kuranishi spaces. Note a possible confusion: we will be defining 1-{\it morphisms of Kuranishi spaces\/} $\bs f,\bs g:\bX\ra\bY$ and 2-{\it morphisms of Kuranishi spaces\/} $\bs\eta:\bs f\Ra\bs g$, but these will be built out of 1-{\it morphisms of Kuranishi neighbourhoods\/} $\bs f_{ij},\bs g_{ij}:(U_i,D_i,\Be_i,r_i,\chi_i)\ra (V_j,E_j,\Ga_j,s_j,\psi_j)$ and 2-{\it morphisms of Kuranishi neighbourhoods\/} $\bs\eta_{ij}:\bs f_{ij}\Ra\bs g_{ij}$ in the sense of \S\ref{ku41}, so `1-morphism' and `2-morphism' can mean two things.

\begin{dfn} Let $\bX=(X,\cI)$ and $\bY=(Y,\cJ)$ be Kuranishi spaces, with notation \eq{ku4eq12}--\eq{ku4eq13}. A 1-{\it morphism of Kuranishi spaces\/} $\bs f:\bX\ra\bY$ is data
\e
\bs f=\bigl(f,\bs f_{ij,\;i\in I,\; j\in J},\; \bs F_{ii',\;i,i'\in I}^{j,\; j\in J},\; \bs F_{i,\;i\in I}^{jj',\; j,j'\in J}\bigr),
\label{ku4eq15}
\e
satisfying the conditions:
\begin{itemize}
\setlength{\itemsep}{0pt}
\setlength{\parsep}{0pt}
\item[(a)] $f:X\ra Y$ is a continuous map.
\item[(b)] $\bs f_{ij}=(P_{ij},\pi_{ij},f_{ij},\hat f_{ij}):(U_i,D_i,\Be_i,r_i,\chi_i)\ra (V_j,E_j,\Ga_j,s_j,\psi_j)$ is a 1-morphism of Kuranishi neighbourhoods over $f$ for all $i\in I$, $j\in J$ (defined over $S=\Im\chi_i\cap f^{-1}(\Im\psi_j)$, as usual).
\item[(c)] $\bs F_{ii'}^j=[\dot P_{ii'}^j,F_{ii'}^j,\hat F_{ii'}^j]:\bs f_{i'j}\ci\Tau_{ii'}\Ra \bs f_{ij}$ is a 2-morphism over $f$ for all $i,i'\in I$ and $j\in J$ (defined over $S=\Im\chi_i\cap\Im\chi_{i'}\cap f^{-1}(\Im\psi_j)$).
\item[(d)] $\bs F_i^{jj'}=[\dot P_i^{jj'},F_i^{jj'},\hat F_i^{jj'}]:\Up_{jj'}\ci\bs f_{ij}\Ra \bs f_{ij'}$ is a 2-morphism over $f$ for all $i\in I$ and $j,j'\in J$ (defined over $S=\Im\chi_i\cap f^{-1}(\Im\psi_j\cap\Im\psi_{j'})$).
\item[(e)] $\bs F_{ii}^j=\bs\be_{\bs f_{ij}}$ and $\bs F_i^{jj}=\bs\ga_{\bs f_{ij}}$ for all $i\in I$, $j\in J$, for $\bs\be_{\bs f_{ij}},\bs\ga_{\bs f_{ij}}$ as in~\eq{ku4eq6}.
\item[(f)] The following commutes for all $i,i',i''\in I$ and $j\in J$:
\begin{equation*}
\xymatrix@C=91pt@R=15pt{
*+[r]{(\bs f_{i''j}\ci\Tau_{i'i''})\ci\Tau_{ii'}} \ar@{=>}[d]^{\bs\al_{\bs f_{i''j},\Tau_{i'i''},\Tau_{ii'}}}
\ar@{=>}[rr]_(0.53){\bs F_{i'i''}^j*\id_{\Tau_{ii'}}} && *+[l]{\bs f_{i'j}\ci\Tau_{ii'}} \ar@{=>}[d]_{\bs F_{ii'}^j}  
\\
*+[r]{\bs f_{i''j}\ci(\Tau_{i'i''}\ci\Tau_{ii'})} \ar@{=>}[r]^(0.65){\id_{\bs f_{i''j}}*\Ka_{ii'i''}}
& \bs f_{i''j}\ci \Tau_{ii''} \ar@{=>}[r]^{\bs F_{ii''}^j} & *+[l]{\bs f_{ij}.\!} }
\end{equation*}
\item[(g)] The following commutes for all $i,i'\in I$ and $j,j'\in J$:
\begin{equation*}
\xymatrix@C=91pt@R=15pt{
*+[r]{(\Up_{jj'}\ci\bs f_{i'j})\ci\Tau_{ii'}} \ar@{=>}[d]^{\bs\al_{\Up_{jj'},\bs f_{i'j},\Tau_{ii'}}}
\ar@{=>}[rr]_(0.53){\bs F_{i'}^{jj'} *\id_{\Tau_{ii'}}} && *+[l]{\bs f_{i'j'}\ci\Tau_{ii'}} \ar@{=>}[d]_{\bs F_{ii'}^{j'}}  
\\
*+[r]{\Up_{jj'}\ci(\bs f_{i'j}\ci\Tau_{ii'})} \ar@{=>}[r]^(0.65){\id_{\Up_{jj'}}*\bs F_{ii'}^j}
& \Up_{jj'}\ci\bs f_{ij} \ar@{=>}[r]^{\bs F_i^{jj'}} & *+[l]{\bs f_{ij'}.\!} }
\end{equation*}
\item[(h)] The following commutes for all $i\in I$ and $j,j',j''\in J$:
\begin{equation*}
\xymatrix@C=89pt@R=15pt{
*+[r]{(\Up_{j'j''}\ci\Up_{jj'})\ci\bs f_{ij}} \ar@{=>}[d]^{\bs\al_{\Up_{j'j''},\Up_{jj'},\bs f_{ij}}}
\ar@{=>}[rr]_(0.53){\La_{jj'j''}*\id_{\bs f_{ij}}} && *+[l]{\Up_{jj''}\ci\bs f_{ij}} \ar@{=>}[d]_{\bs F_i^{jj''}}  
\\
*+[r]{\Up_{j'j''}\ci(\Up_{jj'}\ci\bs f_{ij})} \ar@{=>}[r]^(0.68){\id_{\Up_{j'j''}}*\bs F_i^{jj'}}
& \Up_{j'j''}\ci\bs f_{ij'} \ar@{=>}[r]^{\bs F_i^{j'j''}} & *+[l]{\bs f_{ij''}.\!} }
\end{equation*}
\end{itemize}

If $x\in\bX$ (i.e. $x\in X$), we will write $\bs f(x)=f(x)\in\bY$.

When $\bY=\bX$, define the {\it identity\/ $1$-morphism\/} $\bs\id_\bX:\bX\ra\bX$ by
\e
\bs\id_\bX=\bigl(\id_X,\Tau_{ij,\; i,j\in I},\; \Ka_{ii'j,\;i,i'\in I}^{\;\;\;\;\;\;\;\; j\in I},\; \Ka_{ijj',\;i\in I}^{\;\;\; j,j'\in I}\bigr),
\label{ku4eq16}
\e
Then Definition \ref{ku4def12}(h) implies that (f)--(h) above hold.
\label{ku4def13}
\end{dfn}

\begin{dfn} Let $\bX=(X,\cI)$ and $\bY=(Y,\cJ)$ be Kuranishi spaces, with notation as in \eq{ku4eq12}--\eq{ku4eq13}, and $\bs f,\bs g:\bX\ra\bY$ be 1-morphisms. Suppose the continuous maps $f,g:X\ra Y$ in $\bs f,\bs g$ satisfy $f=g$. A 2-{\it morphism of Kuranishi spaces\/} $\bs\eta:\bs f\Ra\bs g$ is data $\bs\eta=\bigl(\bs\eta_{ij,\; i\in I,\; j\in J}\bigr)$, where $\bs\eta_{ij}=[\dot P_{ij},\eta_{ij},\hat\eta_{ij}]:\bs f_{ij}\Ra\bs g_{ij}$ is a 2-morphism of Kuranishi neighbourhoods over $f=g$ (defined over $S=\Im\chi_i\cap f^{-1}(\Im\psi_j)$, as usual), satisfying the conditions:
\begin{itemize}
\setlength{\itemsep}{0pt}
\setlength{\parsep}{0pt}
\item[(a)] $\bs G_{ii'}^j\od(\bs\eta_{i'j}*\id_{\Tau_{ii'}})=\bs\eta_{ij}\od\bs F_{ii'}^j:\bs f_{i'j}\ci\Tau_{ii'}\Ra\bs g_{ij}$ for all $i,i'\in I$, $j\in J$. 
\item[(b)] $\bs G_i^{jj'}\od(\id_{\Up_{jj'}}*\bs\eta_{ij})=\bs\eta_{ij'}\od\bs F_i^{jj'}:\Up_{jj'}\ci\bs f_{ij}\Ra\bs g_{ij'}$ for all $i\in I$, $j,j'\in J$.\end{itemize}
Note that by definition, 2-morphisms $\bs\eta:\bs f\Ra\bs g$ only exist if $f=g$.

If $\bs f=\bs g$, the {\it identity\/ $2$-morphism\/} is $\bs\id_{\bs f}=\bigl(\id_{\bs f_{ij},\; i\in I,\; j\in J}\bigr):\bs f\Ra\bs f$.
\label{ku4def14}
\end{dfn}

Next we will define composition of 1-morphisms. For morphisms of $\mu$-Kuranishi spaces $\bs f:\bX\ra\bY$, $\bs g:\bY\ra\bZ$ in \S\ref{ku2}, Theorem \ref{ku2thm3} used the sheaf property of morphisms of $\mu$-Kuranishi neighbourhoods in Theorem \ref{ku2thm1} to construct the composition $\bs g\ci\bs f:\bX\ra\bZ$, and $\bs g\ci\bs f$ was unique. In the 2-category Kuranishi space version, things are more complicated: we must use the stack property in Theorem \ref{ku4thm1} to construct compositions of 1-morphisms $\bs g\ci\bs f:\bX\ra\bZ$, and $\bs g\ci\bs f$ is only unique up to 2-isomorphism.

In the next proposition, part (a) constructs candidates $\bs h$ for $\bs g\ci\bs f$, part (b) shows such $\bs h$ are unique up to canonical 2-isomorphism, and part (c) that $\bs g$ and $\bs f$ are allowed candidates for $\bs g\ci\bs\id_\bY,$ $\bs\id_\bY\ci\bs f$ respectively.

\begin{prop}{\bf(a)} Let\/ $\bX=(X,\cI),\bY=(Y,\cJ),\bZ=(Z,\cK)$ be Kuranishi spaces with notation {\rm\eq{ku4eq12}--\eq{ku4eq14},} and\/ $\bs f:\bX\ra\bY,$ $\bs g:\bY\ra\bZ$ be $1$-morphisms, with\/ $\bs f=\bigl(f,\bs f_{ij},\bs F_{ii'}^j,\bs F_i^{jj'}\bigr),$ $\bs g=\bigl(g,\bs g_{jk},\bs G_{jj'}^k,\bs G_j^{kk'}\bigr)$. Then there exists a $1$-morphism $\bs h:\bX\ra\bZ$ with\/ $\bs h=\bigl(h,\bs h_{ik},\bs H_{ii'}^k,\bs H_i^{kk'}\bigr),$ such that\/ $h=g\ci f:X\ra Z,$ and for all\/ $i\in I,$ $j\in J,$ $k\in K$ we have $2$-morphisms over $h$
\e
\Th_{ijk}:\bs g_{jk}\ci\bs f_{ij}\Longra\bs h_{ik},
\label{ku4eq17}
\e
where as usual\/ \eq{ku4eq17} holds over\/ $S=\Im\chi_i\cap f^{-1}(\Im\psi_j)\cap h^{-1}(\Im\om_k),$ and for all\/ $i,i'\in I,$ $j,j'\in J,$ $k,k'\in K$ the following commute:
\ea
\begin{gathered}
\xymatrix@C=85pt@R=15pt{
*+[r]{(\bs g_{jk}\ci\bs f_{i'j})\ci\Tau_{ii'}} \ar@{=>}[d]^{\bs\al_{\bs g_{jk},\bs f_{i'j},\Tau_{ii'}}}
\ar@{=>}[rr]_(0.53){\Th_{i'jk}*\id_{\Tau_{ii'}}} && *+[l]{\bs h_{i'k}\ci\Tau_{ii'}} \ar@{=>}[d]_{\bs H_{ii'}^k}  
\\
*+[r]{\bs g_{jk}\ci(\bs f_{i'j}\ci\Tau_{ii'})} \ar@{=>}[r]^(0.65){\id_{\bs g_{jk}}*\bs F_{ii'}^j}
& \bs g_{jk}\ci\bs f_{ij} \ar@{=>}[r]^{\Th_{ijk}} & *+[l]{\bs h_{ik},\!} }\end{gathered}
\label{ku4eq18}\\
\begin{gathered}
\xymatrix@C=85pt@R=15pt{
*+[r]{(\bs g_{j'k}\ci\Up_{jj'})\ci\bs f_{ij}} \ar@{=>}[d]^{\bs\al_{\bs g_{j'k},\Up_{jj'},\bs f_{ij}}}
\ar@{=>}[rr]_(0.53){\bs G_{jj'}^k*\id_{\bs f_{ij}}} && *+[l]{\bs g_{jk}\ci\bs f_{ij}} \ar@{=>}[d]_{\Th_{ijk}}  
\\
*+[r]{\bs g_{j'k}\ci(\Up_{jj'}\ci\bs f_{ij})} \ar@{=>}[r]^(0.65){\id_{\bs g_{j'k}}*\bs F_i^{jj'}}
& \bs g_{j'k}\ci\bs f_{ij'} \ar@{=>}[r]^{\Th_{ij'k}} & *+[l]{\bs h_{ik},\!} }
\end{gathered}
\label{ku4eq19}\\
\begin{gathered}
\xymatrix@C=85pt@R=15pt{
*+[r]{(\Phi_{kk'}\ci\bs g_{jk})\ci\bs f_{ij}} \ar@{=>}[d]^{\bs\al_{\Phi_{kk'},\bs g_{jk},\bs f_{ij}}}
\ar@{=>}[rr]_(0.53){\bs G_j^{kk'}*\id_{\bs f_{ij}}} && *+[l]{\bs g_{jk'}\ci\bs f_{ij}} \ar@{=>}[d]_{\Th_{ijk'}}  
\\
*+[r]{\Phi_{kk'}\ci(\bs g_{jk}\ci\bs f_{ij})} \ar@{=>}[r]^(0.65){\id_{\Phi_{kk'}}*\Th_{ijk}}
& \Phi_{kk'}\ci\bs h_{ik} \ar@{=>}[r]^{\bs H_i^{kk'}} & *+[l]{\bs h_{ik'}.\!} }
\end{gathered}
\label{ku4eq20}
\ea

\noindent{\bf(b)} If\/ $\bs{\ti h}=\bigl(h,\bs{\ti h}_{ik},\bs{\ti H}{}_{ii'}^k,\bs{\ti H}{}_i^{kk'}\bigr),\ti\Th_{ijk}$ are alternative choices for $\bs h,\Th_{ijk}$ in {\bf(a)\rm,} then there is a unique $2$-morphism of Kuranishi spaces $\bs\eta=(\bs\eta_{ik}):\bs h\Ra\bs{\ti h}$ satisfying $\bs\eta_{ik}\od\Th_{ijk}=\ti\Th_{ijk}:\bs g_{jk}\ci\bs f_{ij}\Ra\bs{\ti h}_{ik}$ for all\/ $i\in I,$ $j\in J,$ $k\in K$.
\smallskip

\noindent{\bf(c)} If\/ $\bX=\bY$ and\/ $\bs f=\bs\id_\bY$ in {\bf(a)\rm,} so that\/ $I=J,$ then a possible choice for $\bs h,\Th_{ijk}$ in {\bf(a)} is $\bs h=\bs g$ and\/~$\Th_{ijk}=\bs G_{ij}^k$.

Similarly, if\/ $\bZ=\bY$ and\/ $\bs g=\bs\id_\bY$ in {\bf(a)\rm,} so that\/ $K=J,$ then a possible choice for $\bs h,\Th_{ijk}$ in {\bf(a)} is $\bs h=\bs f$ and\/~$\Th_{ijk}=\bs F_i^{jk}$.
\label{ku4prop1}
\end{prop}

\begin{proof} For (a), define $h=g\ci f:X\ra Z$. Let $i\in I$ and $k\in K$, and set 
$S=\Im\chi_i\cap h^{-1}(\Im\om_k)$, so that $S$ is open in $X$. We want to choose a 1-morphism $\bs h_{ik}:(U_i,D_i,\De_i,r_i,\chi_i)\ra (W_k,F_k,\De_k,t_k,\om_k)$ of Kuranishi neighbourhoods over $(S,h)$. Since $\{\Im\psi_j:j\in J\}$ is an open cover of $Y$ and $f$ is continuous, $\bigl\{S\cap f^{-1}(\Im\psi_j):j\in J\bigr\}$ is an open cover of $S$. For all $j,j'\in J$ we have a 2-morphism over $S\cap f^{-1}(\Im\psi_j\cap\Im\psi_{j'}),h$
\e
\begin{split}
&(\id_{\bs g_{j'k}}*\bs F_i^{jj'})\od(\bs G_{jj'}^k*\id_{\bs f_{ij}})^{-1}:\\
&\bs g_{jk}\ci\bs f_{ij}\vert_{S\cap f^{-1}(\Im\psi_j\cap\Im\psi_{j'})}
\Longra\bs g_{j'k}\ci\bs f_{ij'}\vert_{S\cap f^{-1}(\Im\psi_j\cap\Im\psi_{j'})}.
\end{split}
\label{ku4eq21}
\e
For $j,j',j''\in J$, consider the diagram of 2-morphisms over $S\cap f^{-1}(\Im\psi_j\cap\Im\psi_{j'}\cap\Im\psi_{j''}),h$
\e
\begin{gathered}
\text{\begin{small}$\displaystyle\xymatrix@!0@C=89pt@R=25pt{
&&& {\bs g_{jk}\!\ci\!\bs f_{ij}}
\\
\\
&&& {(\bs g_{j'k}\!\ci\!\Up_{jj'})\!\ci\!\bs f_{ij}} \ar[uu]_{\bs G_{jj'}^k*\id_{\bs f_{ij}}} \ar[dd]^{\bs\al_{\bs g_{j'k},\Up_{jj'},\bs f_{ij}}} 
\\
*+[r]{(\bs g_{j''k}\!\ci\!\Up_{jj''})\!\ci\!\bs f_{ij}} \ar[uuurrr]^{\bs G_{jj''}^k*\id_{\bs f_{ij}}\,\,\,\,} \ar[dddddd]^{\bs\al_{\bs g_{j''k},\Up_{jj''},\bs f_{ij}}} && {((\bs g_{j''k}\!\ci\!\Up_{j'j''})\!\ci\!\Up_{jj'})\!\ci\!\bs f_{ij}} \ar[ddl]^{\bs\al_{\cdots}} \ar[ur]^(0.3){(\bs G_{j'j''}^k*\id_{\Up_{jj'}})*\id_{\bs f_{ij}}} \ar[ddd]^(0.25){\bs\al_{\cdots}}
\\
&&& {\bs g_{j'k}\!\ci\!(\Up_{jj'}\!\ci\!\bs f_{ij})} \ar[dd]^{\id_{\bs g_{j'k}}*\bs F_i^{jj'}} 
\\
& (\bs g_{j''k}\!\ci\!(\Up_{j'j''}\!\ci\!\Up_{jj'}))\!\ci\!\bs f_{ij} \ar[dd]^{\bs\al_{\cdots}} \ar[uul]_(0.4){(\id_{\bs g_{j''k}}*\La_{jj'j''})*\id_{\bs f_{ij}}} 
\\
&& {(\bs g_{j''k}\!\ci\!\Up_{j'j''})\!\ci\!(\Up_{jj'}\!\ci\!\bs f_{ij})} 
\ar[ddd]^(0.75){\bs\al_{\cdots}}  \ar[ddr]_{(\id_{\bs g_{j''k}}*\id_{\Up_{j'j''}})*\bs F_i^{jj'}} \ar[uur]^{\bs G_{j'j''}^k*(\id_{\Up_{jj'}}*\id_{\bs f_{ij}})} & {\bs g_{j'k}\!\ci\!\bs f_{ij'}} 
\\
& \bs g_{j''k}\!\ci\!((\Up_{j'j''}\!\ci\!\Up_{jj'})\!\ci\!\bs f_{ij}) \ar[ddr]^{\bs\al_{\cdots}} \ar[ddl]^(0.4){\id_{\bs g_{j''k}}*(\La_{jj'j''}*\id_{\bs f_{ij}})}
\\
&&& {(\bs g_{j''k}\!\ci\!\Up_{j'j''})\!\ci\! \bs f_{ij'}} \ar[uu]_{\bs G_{j'j''}^k*\id_{\bs f_{ij'}}} \ar[dd]^{\bs\al_{\bs g_{j''k},\Up_{j'j''},\bs f_{ij'}}} 
\\
*+[r]{\bs g_{j''k}\!\ci\!(\Up_{jj''}\!\ci\!\bs f_{ij})} \ar[dddrrr]_{\id_{\bs g_{j''k}}*\bs F_i^{jj''}\,\,\,\,} && {\bs g_{j''k}\!\ci\!(\Up_{j'j''}\!\ci\!(\Up_{jj'}\!\ci\!\bs f_{ij}))} \ar[dr]_(0.3){\id_{\bs g_{j''k}}*(\id_{\Up_{j'j''}}*\bs F_i^{jj'})} 
\\
&&& {\bs g_{j''k}\!\ci\!(\Up_{j'j''}\!\ci\! \bs f_{ij'})}  \ar[dd]^{\id_{\bs g_{j''k}}*\bs F_i^{j'j''}} 
\\
\\
&&& {\bs g_{j''k}\!\ci\!\bs f_{ij''}.\!\!} }\!\!\!\!\!\!\!\!\!\!\!\!\!\!\!\!\!\!\!\!\!\!\!\!\!{}$\end{small}}
\end{gathered}
\label{ku4eq22}
\e

Here the top pentagon of \eq{ku4eq22} commutes by Definition \ref{ku4def13}(f) for $\bs g$ composed with $\id_{\bs f_{ij}}$, the bottom pentagon commutes by Definition \ref{ku4def13}(h) for $\bs f$ composed with $\id_{\bs g_{j''k}}$, and the rest commutes by properties of weak 2-categories. Thus \eq{ku4eq22} commutes. This implies that
\e
\begin{split}
\bigl(&(\id_{\bs g_{j''k}}*\bs F_i^{j'j''})\od\bs\al_{\bs g_{j''k},\Up_{j'j''},\bs f_{ij'}}\od(\bs G_{j'j''}^k*\id_{\bs f_{ij'}})^{-1}\bigr)\od{}\\
\bigl(&(\id_{\bs g_{j'k}}*\bs F_i^{jj'})\od\bs\al_{\bs g_{j'k},\Up_{jj'},\bs f_{ij}}\od(\bs G_{jj'}^k*\id_{\bs f_{ij}})^{-1}\bigr)\\
&=(\id_{\bs g_{j''k}}*\bs F_i^{jj''})\od\bs\al_{\bs g_{j''k},\Up_{jj''},\bs f_{ij}}\od(\bs G_{jj''}^k*\id_{\bs f_{ij}})^{-1}.
\end{split}
\label{ku4eq23}
\e

Now Theorem \ref{ku4thm1}(b) says that 1- and 2-morphisms from $(U_i,D_i,\Be_i,r_i,\chi_i)$ to $(W_k,F_k,\De_k,t_k,\om_k)$ over $h$ form a stack on $S$, so applying Definition \ref{ku4def11}(v) to the open cover $\bigl\{S\cap f^{-1}(\Im\psi_j):j\in J\bigr\}$ of $S$ with $\bs g_{jk}\ci\bs f_{ij}$ in place of $A_j$, \eq{ku4eq21} in place of $\al_{jj'}$, and \eq{ku4eq23}, shows that there exist a 1-morphism $\bs h_{ik}:(U_i,D_i,\Be_i,r_i,\chi_i)\ra(W_k,F_k,\De_k,t_k,\om_k)$ over $(S,h)$, and 2-morphisms
\begin{equation*}
\Th_{ijk}:\bs g_{jk}\ci\bs f_{ij}\vert_{S\cap f^{-1}(\Im\psi_j)}\Longra \bs h_{ik}\vert_{S\cap f^{-1}(\Im\psi_j)}
\end{equation*}
for all $j\in J$, satisfying for all $j,j'\in J$
\e
\begin{split}
&\Th_{ijk}\vert_{S\cap f^{-1}(\Im\psi_j\cap\Im\psi_{j'})}\\
&\quad=\Th_{ij'k}\od(\id_{\bs g_{j'k}}*\bs F_i^{jj'})\od\bs\al_{\bs g_{j'k},\Up_{jj'},\bs f_{ij}}\od(\bs G_{jj'}^k*\id_{\bs f_{ij}})^{-1}.
\end{split}
\label{ku4eq24}
\e
Observe that \eq{ku4eq24} is equivalent to equation \eq{ku4eq19} in the proposition.

So far we have chosen the data $h,\bs h_{ik}$ for all $i,k$ in $\bs h=\bigl(h,\bs h_{ik},\bs H_{ii'}^k,\bs H_i^{kk'}\bigr)$, where $\bs h_{ik}$ involved an arbitrary choice. To define $\bs H_{ii'}^k$ for $i,i'\in I$ and $k\in K$, note that for each $j\in J$, equation \eq{ku4eq18} of the proposition implies that
\e
\begin{split}
\bs H_{ii'}^k&\vert_{\Im\chi_i\cap\Im\chi_{i'}\cap f^{-1}(\Im\psi_j)\cap h^{-1}(\Im\om_k)}\\
&=\Th_{ijk}\od(\id_{\bs g_{jk}}*\bs F_{ii'}^j)\od\bs\al_{\bs g_{jk},\bs f_{i'j},\Tau_{ii'}}\od\bigl(\Th_{i'jk}*\id_{\Tau_{ii'}}\bigr){}^{-1}.
\end{split}
\label{ku4eq25}
\e
Using \eq{ku4eq24} for $i,i'$ and a similar commutative diagram to \eq{ku4eq22}, we can show that the prescribed values \eq{ku4eq25} for $j,j'\in J$ agree when restricted to $\Im\chi_i\cap\Im\chi_{i'}\cap f^{-1}(\Im\psi_j\cap\Im\psi_{j'})\cap h^{-1}(\Im\om_k)$. Therefore the stack property Theorem \ref{ku4thm1}(b) and Definition \ref{ku4def11}(iii),(iv) show that there is a unique 2-morphism $\bs H_{ii'}^k:\bs h_{i'k}\ci\Tau_{ii'}\Ra \bs h_{ik}$ over $h$ satisfying \eq{ku4eq25} for all $j\in J$, or equivalently, satisfying \eq{ku4eq18} for all $j\in J$. Similarly, there is a unique 2-morphism $\bs H_i^{kk'}:\Phi_{kk'}\ci\bs h_{ik}\Ra\bs h_{ik'}$ over $h$ satisfying \eq{ku4eq20} for all~$j\in J$.

We now claim that $\bs h=\bigl(h,\bs h_{ik},\bs H_{ii'}^k,\bs H_i^{kk'}\bigr)$ is a 1-morphism $\bs h:\bX\ra\bZ$. It remains to show Definition \ref{ku4def13}(f)--(h) hold for $\bs h$. To prove this, we first fix $j\in J$ and prove the restrictions of (f)--(h) to the intersections of their domains with $f^{-1}(\Im\psi_j)$. For instance, for part (f), for $i,i',i''\in I$ and $k\in K$ we have
\begin{align*}
&\bigl(\bs H_{ii''}^k\od(\id_{\bs h_{i''k}}*\Ka_{ii'i''})\od\bs\al_{\bs h_{i''k},\Tau_{i'i''},\Tau_{ii'}}\bigr)\vert_{\Im\chi_i\cap\cdots\cap h^{-1}(\Im\om_k)}\\
&=\bigl[\Th_{ijk}\od(\id_{\bs g_{jk}}*\bs F_{ii''}^j)\od\bs\al_{\bs g_{jk},\bs f_{i''j},\Tau_{ii''}}\od\bigl(\Th_{i''jk}*\id_{\Tau_{ii''}}\bigr){}^{-1}\bigr]\\ 
&\qquad\od(\id_{\bs h_{i''k}}*\Ka_{ii'i''})\od\bs\al_{\bs h_{i''k},\Tau_{i'i''},\Tau_{ii'}}\\
&=\Th_{ijk}\od\bigl(\id_{\bs g_{jk}}*(\bs F_{ii''}^j\od(\id_{\bs f_{ij''}}*\Ka_{ii'i''})\od\bs\al_{\bs f_{i''j},\Tau_{i'i''},\Tau_{ii'}})\bigr)\\
&\quad \od\bs\al_{\bs g_{jk},\bs f_{i''j}\ci\Tau_{i'i''},\Tau_{ii'}}\od(\bs\al_{\bs g_{jk},\bs f_{i''j},\Tau_{i'i''}}*\id_{\Tau_{ii'}})
\od\bigl((\Th_{i''jk}^{-1}*\id_{\Tau_{i'i''}})*\id_{\Tau_{ii'}}\bigr)\\
&=\Th_{ijk}\od\bigl(\id_{\bs g_{jk}}*(\bs F_{ii'}^j\od(\bs F_{i'i''}^j*\id_{\Tau_{ii'}}))\bigr)\\
&\quad \od\bs\al_{\bs g_{jk},\bs f_{i''j}\ci\Tau_{i'i''},\Tau_{ii'}}\od(\bs\al_{\bs g_{jk},\bs f_{i''j},\Tau_{i'i''}}*\id_{\Tau_{ii'}})
\od\bigl((\Th_{i''jk}^{-1}*\id_{\Tau_{i'i''}})*\id_{\Tau_{ii'}}\bigr)\\
&=\bigl[\Th_{ijk}\od(\id_{\bs g_{jk}}*\bs F_{ii'}^j)\od\bs\al_{\bs g_{jk},\bs f_{i'j},\Tau_{ii'}}\od\bigl(\Th_{i'jk}*\id_{\Tau_{ii'}}\bigr){}^{-1}\bigr]\\
&\quad\od\bigl(\bigl[
\bigl(\Th_{i'jk}\od(\id_{\bs g_{jk}}*\bs F_{i'i''}^j)\bigr)\od\bs\al_{\bs g_{jk},\bs f_{i''j},\Tau_{i'i''}}\od\bigl(\Th_{i''jk}*\id_{\Tau_{i'i''}}\bigr){}^{-1}\bigr]*\id_{\Tau_{ii'}}\bigr)\\
&=\bigl(\bs H_{ii'}^k\od(\bs H_{i'i''}^k*\id_{\Tau_{ii'}})\bigr)\vert_{\Im\chi_i\cap\Im\chi_{i'}\cap\Im\chi_{i''}\cap f^{-1}(\Im\psi_j)\cap h^{-1}(\Im\om_k)},
\end{align*}
using \eq{ku4eq25} in the first and fifth steps, Definition \ref{ku4def13}(f) for $\bs f$ in the third, and properties of weak 2-categories. Then we use the stack property Theorem \ref{ku4thm1}(b) and Definition \ref{ku4def11}(iii) to deduce that as Definition \ref{ku4def13}(f)--(h) for $\bs h$ hold on the sets of an open cover, they hold globally. Therefore $\bs h:\bX\ra\bZ$ is a 1-morphism of Kuranishi spaces satisfying \eq{ku4eq18}--\eq{ku4eq20}, proving~(a).

For (b), if $\bs{\ti h},\ti\Th_{ijk}$ are alternatives, then $\bs h_{ik},\bs{\ti h}_{ik}$ are alternative solutions to the application of Theorem \ref{ku4thm1}(b) and Definition \ref{ku4def11}(v) above, for all $i\in I$ and $k\in K$. Thus, the last part of Definition \ref{ku4def11}(v) implies that there is a unique 2-morphism $\bs\eta_{ik}:\bs h_{ik}\Ra\bs{\ti h}_{ik}$ over $h$ such that for all $j\in J$ we have
\e
\bs\eta_{ik}\vert_{\Im\chi_i\cap f^{-1}(\Im\psi_j)\cap h^{-1}(\Im\om_k)}=\ti\Th_{ijk}\od\Th_{ijk}^{-1}.
\label{ku4eq26}
\e
This implies that $\bs\eta_{ik}\od\Th_{ijk}=\ti\Th_{ijk}$, as in (b). For each $j\in J$ we have
\begin{align*}
\bigl(&\bs{\ti H}{}_{ii'}^k\od(\bs\eta_{i'k}*\id_{\Tau_{ii'}})\bigr)\vert_{\Im\chi_i\cap\Im\chi_{i'}\cap f^{-1}(\Im\psi_j)\cap h^{-1}(\Im\om_k)}\\
&=\bigl[\ti\Th_{ijk}\od(\id_{\bs g_{jk}}*\bs F_{ii'}^j)\od\bs\al_{\bs g_{jk},\bs f_{i'j},\Tau_{ii'}}\od\bigl(\ti\Th_{i'jk}*\id_{\Tau_{ii'}}\bigr){}^{-1}\bigr]\\
&\qquad\od\bigl[\bigl(\ti\Th_{i'jk}\od\Th_{i'jk}^{-1}\bigr)*\id_{\Tau_{ii'}}\bigr]\\
&=\bigl[\ti\Th_{ijk}\od\Th_{ijk}^{-1}\bigr]\od\bigl[\Th_{ijk}\od(\id_{\bs g_{jk}}*\bs F_{ii'}^j)\od\bs\al_{\bs g_{jk},\bs f_{i'j},\Tau_{ii'}}\od\bigl(\Th_{i'jk}*\id_{\Tau_{ii'}}\bigr){}^{-1}\bigr]\\
&=\bigl(\bs\eta_{ik}\od\bs H_{ii'}^k\bigr)\vert_{\Im\chi_i\cap\Im\chi_{i'}\cap f^{-1}(\Im\psi_j)\cap h^{-1}(\Im\om_k)},
\end{align*}
using \eq{ku4eq25} and \eq{ku4eq26} in the first and third steps. So by Definition \ref{ku4def11}(iii) we deduce that $\bs{\ti H}{}_{ii'}^k\od(\bs\eta_{i'k}*\id_{\Tau_{ii'}})=\bs\eta_{ik}\od\bs H_{ii'}^k$, which is Definition \ref{ku4def14}(a) for $\bs\eta=(\bs\eta_{ik}):\bs h\Ra\bs{\ti h}$. Similarly Definition \ref{ku4def14}(b) holds, so $\bs\eta:\bs h\Ra\bs{\ti h}$ is a 2-morphism of Kuranishi spaces. This proves (b). Part (c) is immediate, using Definition \ref{ku4def13}(f)--(h) for $\bs f,\bs g$ to prove \eq{ku4eq18}--\eq{ku4eq20} hold for the given choices of $\bs h$ and $\Th_{ijk}$. This completes the proof of Proposition~\ref{ku4prop1}.
\end{proof}

Proposition \ref{ku4prop1}(a) gives possible values $\bs h$ for the composition $\bs g\ci\bs f:\bX\ra\bZ$. Since there is no distinguished choice, we choose $\bs g\ci\bs f$ arbitrarily.

\begin{dfn} For all pairs of 1-morphisms of Kuranishi spaces $\bs f:\bX\ra\bY$ and $\bs g:\bY\ra\bZ$, use the Axiom of Global Choice (see Remark \ref{ku4rem5}) to choose possible values of $\bs h:\bX\ra\bZ$ and $\Th_{ijk}$ in Proposition \ref{ku4prop1}(a), and write $\bs g\ci\bs f=\bs h$, and for $i\in I$, $j\in J$, $k\in K$
\begin{equation*}
\Th_{ijk}^{\bs g,\bs f}=\Th_{ijk}:\bs g_{jk}\ci\bs f_{ij}\Longra(\bs g\ci\bs f)_{ik}.
\end{equation*}
We call $\bs g\ci\bs f$ the {\it composition of\/ $1$-morphisms of Kuranishi spaces}.

For general $\bs f,\bs g$ we make these choices arbitrarily. However, if $\bX=\bY$ and $\bs f=\bs\id_\bY$ then we choose $\bs g\ci\bs\id_\bY=\bs g$ and $\Th_{jj'k}^{\bs g,\bs\id_\bY}=\bs G_{jj'}^k$, and if $\bZ=\bY$ and $\bs g=\bs\id_\bY$ then we choose $\bs\id_\bY\ci\bs f=\bs f$ and $\Th_{ijj'}^{\bs\id_\bY,\bs f}=\bs F_i^{jj'}$. This is allowed by Proposition~\ref{ku4prop1}(c).

The definition of a weak 2-category in Appendix \ref{kuB} includes 2-isomorphisms $\bs\be_{\bs f}:\bs f\ci\bs\id_\bX\Ra\bs f$ and $\bs\ga_{\bs f}:\bs\id_\bY\ci\bs f\Ra\bs f$ in \eq{kuBeq9}, since one does not require $\bs f\ci\bs\id_\bX=\bs f$ and $\bs\id_\bY\ci\bs f=\bs f$ in a general weak 2-category. We define
\e
\bs\be_{\bs f}=\bs\id_{\bs f}:\bs f\ci\bs\id_\bX\Longra\bs f,\quad
\bs\ga_{\bs f}=\bs\id_{\bs f}:\bs\id_\bY\ci\bs f\Longra\bs f.
\label{ku4eq27}
\e

\label{ku4def15}
\end{dfn}

\begin{rem} As in Shulman \cite[\S 7]{Shul} or Herrlick and Strecker \cite[\S 1.2]{HeSt}, the {\it Axiom of Global Choice}, or {\it Axiom of Choice for classes}, used in Definition \ref{ku4def15}, is a strong form of the Axiom of Choice. 

As in Jech \cite{Jech}, in Set Theory one distinguishes between sets, and `classes', which are like sets but may be larger. We are not allowed to consider things like `the set of all sets', or `the set of all manifolds', as this would lead to paradoxes such as `the set of all sets which are not members of themselves'. Instead sets, manifolds, \ldots\ form classes, upon which more restrictive operations are allowed.

The Axiom of Choice says that if $\{S_i:i\in I\}$ is a family of nonempty sets, with $I$ a set, then we can simultaneously choose an element $s_i\in S_i$ for all $i\in I$. The Axiom of Global Choice says the same thing, but allowing $I$ (and possibly also the $S_i$) to be classes rather than sets. As in \cite[\S 7]{Shul}, the Axiom of Global Choice follows from the axioms of von Neumann--Bernays--G\"odel Set Theory.

The Axiom of Global Choice is used, implicitly or explicitly, in the proofs of important results in category theory in their most general form, for example, Adjoint Functor Theorems, or that every category has a skeleton, or that every weak 2-category can be strictified.

We need to use the Axiom of Global Choice above because we make an arbitrary choice of $\bs g\ci\bs f$ for all $\bs f:\bX\ra\bY$ and $\bs g:\bY\ra\bZ$ in $\Kur$, and as we have defined things, the collection of all such $(\bs f,\bs g)$ is a proper class, not a set.

We could avoid this by arranging our foundations differently. For example, if we require that all topological spaces and manifolds in this book are subsets (or subspaces, or submanifolds) of some fixed $\R^\iy$, then the collection of all $(\bs f,\bs g)$ would be a set, and the usual Axiom of Choice would suffice. 

If we did not make arbitrary choices of compositions $\bs g\ci\bs f$ at all, then $\Kur$ would not be a weak 2-category in Theorem \ref{ku4thm3} below, since for 1-morphisms $\bs f:\bX\ra\bY$ and $\bs g:\bY\ra\bZ$ in $\Kur$ we would not be given a unique composition $\bs g\ci\bs f:\bX\ra\bZ$, but only a nonempty family of possible choices for $\bs g\ci\bs f$, which are all 2-isomorphic. Such structures appear in the theory of {\it quasi-categories}, as in Boardman and Vogt \cite{BoVo} or Joyal \cite{Joya}, which are a form of $\iy$-category, and $\Kur$ would be an example of a 3-coskeletal quasi-category. 

If we only wanted to define the homotopy category $\Ho(\Kur)$, there would be no need for the Axiom of Global Choice, as to define composition $[\bs g]\ci[\bs f]=[\bs g\ci\bs f]$ in $\Ho(\Kur)$ we only need to know the 2-isomorphism class of $\bs g\ci\bs f$ in~$\Kur$.
\label{ku4rem5}
\end{rem}

Since composition of 1-morphisms $\bs g\ci\bs f$ is natural only up to canonical 2-isomorphism, as in Proposition \ref{ku4prop1}(b), composition is associative only up to canonical 2-isomorphism. Note that the 2-isomorphisms $\bs\al_{\bs g,\bs f,\bs e}$ in \eq{ku4eq28} are part of the definition of a weak 2-category in Appendix \ref{kuB}, as in~\eq{kuBeq6}.

\begin{prop} Let\/ $\bs e:\bW\ra\bX,$ $\bs f:\bX\ra\bY,$ $\bs g:\bY\ra\bZ$ be $1$-morphisms of Kuranishi spaces, and define composition of\/ $1$-morphisms as in Definition\/ {\rm\ref{ku4def15}}. Then using notation\/ {\rm\eq{ku4eq11}--\eq{ku4eq14},} there is a unique $2$-morphism
\e
\bs\al_{\bs g,\bs f,\bs e}:(\bs g\ci\bs f)\ci\bs e\Longra\bs g\ci(\bs f\ci\bs e)
\label{ku4eq28}
\e
with the property that for all\/ $h\in H,$ $i\in I,$ $j\in J$ and\/ $k\in K$ we have
\e
(\bs\al_{\bs g,\bs f,\bs e})_{hk}\!\od\!\Th^{\bs g\ci\bs f,\bs e}_{hik}\!\od\! (\Th_{ijk}^{\bs g,\bs f}*\id_{\bs e_{hi}})\!=\!
\Th_{hjk}^{\bs g,\bs f\ci\bs e}\!\od\!(\id_{\bs g_{jk}}*\Th_{hij}^{\bs f,\bs e})\!\od\!\bs\al_{\bs g_{jk},\bs f_{ij},\bs e_{hi}}.
\label{ku4eq29}
\e

\label{ku4prop2}
\end{prop}

\begin{proof} The proof uses similar ideas to that of Proposition \ref{ku4prop1}, so we will be brief. Note that for $h\in H,$ $i\in I,$ $j\in J,$ $k\in K$, equation \eq{ku4eq29} implies that
\e
\begin{split}
&(\bs\al_{\bs g,\bs f,\bs e})_{hk}\vert_{\Im\vp_h\cap e^{-1}(\Im\chi_i)\cap (f\ci e)^{-1}(\Im\psi_j)\cap (g\ci f\ci e)^{-1}(\Im\om_k)}\\
&=\Th_{hjk}^{\bs g,\bs f\ci\bs e}\!\od\!(\id_{\bs g_{jk}}*\Th_{hij}^{\bs f,\bs e})\!\od\!\bs\al_{\bs g_{jk},\bs f_{ij},\bs e_{hi}}\!\od\!(\Th_{ijk}^{\bs g,\bs f}*\id_{\bs e_{hi}})^{-1}\!\od\!(\Th^{\bs g\ci\bs f,\bs e}_{hik})^{-1}.
\end{split}
\label{ku4eq30}
\e
We show that for $i'\in I$, $j'\in J$, the right hand sides of \eq{ku4eq30} for $h,i,j,k$ and for $h,i',j',k$ agree on the overlap of their domains, using the properties \eq{ku4eq18}--\eq{ku4eq20} of the $\Th_{ijk}^{\bs g,\bs f}$. Then we use the stack property Theorem \ref{ku4thm1}(b) and Definition \ref{ku4def11}(iii),(iv) to deduce that there is a unique 2-morphism $(\bs\al_{\bs g,\bs f,\bs e})_{hk}$ satisfying \eq{ku4eq30} for all $i\in I$, $j\in J$. 

We prove the restriction of Definition \ref{ku4def14}(a),(b) for $\bs\al_{\bs g,\bs f,\bs e}=((\bs\al_{\bs g,\bs f,\bs e})_{hk})$ to the intersection of their domains with $e^{-1}(\Im\chi_i)\cap (f\ci e)^{-1}(\Im\psi_j)$, for all $i\in I$ and $j\in J$, using \eq{ku4eq30} and properties of the $\Th_{ijk}^{\bs g,\bs f}$. Since these intersections form an open cover of the domains, Theorem \ref{ku4thm1}(b) and Definition \ref{ku4def11}(iii) imply that Definition \ref{ku4def14}(a),(b) for $\bs\al_{\bs g,\bs f,\bs e}$ hold on the correct domains, so $\bs\al_{\bs g,\bs f,\bs e}$ is a 2-morphism, as in \eq{ku4eq28}. Uniqueness follows from uniqueness of $(\bs\al_{\bs g,\bs f,\bs e})_{hk}$ above. This completes the proof.
\end{proof}

We define vertical and horizontal composition of 2-morphisms:

\begin{dfn} Let $\bs f,\bs g,\bs h:\bX\ra\bY$ be 1-morphisms of Kuranishi spaces, using notation \eq{ku4eq12}--\eq{ku4eq13}, and $\bs\eta\!=\!(\bs\eta_{ij}):\bs f\!\Ra\!\bs g$, $\bs\ze\!=\!(\bs\ze_{ij}):\bs g\!\Ra\!\bs h$ be 2-mor\-phisms. Define the {\it vertical composition of\/ $2$-morphisms\/} $\bs\ze\!\od\!\bs\eta:\bs f\!\Ra\!\bs h$ by
\e
\bs\ze\od\bs\eta=\bigl(\bs\ze_{ij}\od\bs\eta_{ij},\; i\in I,\; j\in J\bigr).
\label{ku4eq31}
\e
To see that $\bs\ze\od\bs\eta$ satisfies Definition \ref{ku4def14}(a),(b), for (a) note that for all $i,i'\in I$ and $j\in J$, by Definition \ref{ku4def14}(a) for $\bs\eta,\bs\ze$ we have
\begin{align*}
\bs H_{ii'}^j&\od((\bs\ze_{i'j}\od\bs\eta_{i'j})*\id_{\Tau_{ii'}})
=\bs H_{ii'}^j\od(\bs\ze_{i'j}*\id_{\Tau_{ii'}})\od(\bs\eta_{i'j}*\id_{\Tau_{ii'}})\\
&=\bs\ze_{ij}\od\bs G_{ii'}^j\od(\bs\eta_{i'j}*\id_{\Tau_{ii'}})
=(\bs\ze_{ij}\od\bs\eta_{ij})\od\bs F_{ii'}^j,
\end{align*}
and Definition \ref{ku4def14}(b) for $\bs\ze\od\bs\eta$ is proved similarly.

Clearly, vertical composition of 2-morphisms of Kuranishi spaces is associative, $(\bs\th\od\bs\ze)\od\bs\eta=\bs\th\od(\bs\ze\od\bs\eta)$, since vertical composition of 2-morphisms of Kuranishi neighbourhoods is associative.

If $\bs g=\bs h$ and $\bs\ze=\bs\id_{\bs g}$ then $\bs\id_{\bs g}\od\bs\eta=(\id_{\bs g_{ij}}\od\bs\eta_{ij})=(\bs\eta_{ij})=\bs\eta$, and similarly $\bs\ze\od\bs\id_{\bs g}=\bs\ze$, so identity 2-morphisms behave as expected under~$\od$.

If $\bs\eta=(\bs\eta_{ij}):\bs f\Ra\bs g$ is a 2-morphism of Kuranishi spaces, then as 2-morphisms $\bs\eta_{ij}$ of Kuranishi neighbourhoods are invertible, we may define $\bs\eta^{-1}=(\bs\eta_{ij}^{-1}):\bs g\Ra\bs f$. It is easy to check that $\bs\eta^{-1}$ is a 2-morphism, and $\bs\eta^{-1}\od\bs\eta=\bs\id_{\bs f}$, $\bs\eta\od\bs\eta^{-1}=\bs\id_{\bs g}$. Thus, all 2-morphisms of Kuranishi spaces are 2-isomorphisms.

\label{ku4def16}
\end{dfn}

\begin{dfn} Let $\bs e,\bs f:\bX\ra\bY$ and $\bs g,\bs h:\bY\ra\bZ$ be 1-morphisms of Kuranishi spaces, using notation \eq{ku4eq12}--\eq{ku4eq14}, and $\bs\eta=(\bs\eta_{ij}):\bs e\Ra\bs f$, $\bs\ze=(\bs\ze_{jk}):\bs g\Ra\bs h$ be 2-morphisms. We claim there is a unique 2-morphism $\bs\th=(\bs\th_{ik}):\bs g\ci\bs e\Ra\bs h\ci\bs f$, such that for all $i\in I$, $j\in J$, $k\in K$, we have
\e
\bs\th_{ik}\vert_{\Im\chi_i\cap e^{-1}(\Im\psi_j)\cap (g\ci e)^{-1}(\Im\om_k)}=
\Th_{ijk}^{\bs h,\bs f}\od(\bs\ze_{jk}*\bs\eta_{ij})\od(\Th_{ijk}^{\bs g,\bs e})^{-1}.
\label{ku4eq32}
\e

To prove this, suppose $j,j'\in J$, and consider the diagram of 2-morphisms over $\Im\chi_i\cap e^{-1}(\Im\psi_j\cap\Im\psi_{j'})\cap (g\ci e)^{-1}(\Im\om_k)$:
\e
\begin{gathered}
\xymatrix@!0@C=70pt@R=15pt{
& \bs g_{jk}\ci\bs e_{ij} \ar[rr]_{\bs\ze_{jk}*\bs\eta_{ij}} && \bs h_{jk}\ci\bs f_{ij} \ar@/^1pc/[dddr]^(0.55){\Th_{ijk}^{\bs h,\bs f}}
\\
\\
& (\bs g_{j'k}\ci\Up_{jj'})\ci\bs e_{ij} \ar[rr]^{(\bs\ze_{j'k}*\id_{\Up_{jj'}})*\bs\eta_{ij}} \ar@<3ex>[uu]_{\bs G_{jj'}^k*\id_{\bs e_{ij}}}  \ar@<-3ex>[dd]^{\bs\al_{\bs g_{j'k},\Up_{jj'},\bs e_{ij}}} && (\bs h_{j'k}\ci\Up_{jj'})\ci\bs f_{ij} \ar@<-3ex>[uu]^{\bs H_{jj'}^k*\id_{\bs f_{ij}}} \ar@<3ex>[dd]_{\bs\al_{\bs h_{j'k},\Up_{jj'},\bs f_{ij}}} \\
*+[r]{(\bs g\ci\bs e)_{ik}} \ar@/^1pc/[uuur]^(0.45){(\Th_{ijk}^{\bs g,\bs e})^{-1}}
 \ar@/_1pc/[dddr]_(0.45){(\Th_{ij'k}^{\bs g,\bs e})^{-1}} &&&& *+[l]{(\bs h\ci\bs f)_{ik}} \\
& \bs g_{j'k}\ci(\Up_{jj'}\ci\bs e_{ij}) \ar[rr]^{\bs\ze_{j'k}*(\id_{\Up_{jj'}}*\bs\eta_{ij})}  \ar@<-3ex>[dd]^{\id_{\bs g_{j'k}}*\bs E_i^{jj'}} && \bs h_{j'k}\ci(\Up_{jj'}\ci\bs f_{ij})  \ar@<3ex>[dd]_{\id_{\bs h_{j'k}}*\bs F_i^{jj'}} \\
\\
& \bs g_{j'k}\ci\bs e_{ij'} \ar[rr]^{\bs\ze_{j'k}*\bs\eta_{ij'}} && \bs h_{j'k}\ci\bs f_{ij'} \ar@/_1pc/[uuur]_(0.55){\Th_{ij'k}^{\bs h,\bs f}} }
\end{gathered}
\label{ku4eq33}
\e
Here the left and right polygons commute by \eq{ku4eq19}, the top and bottom rectangles commute by Definition \ref{ku4def14}(a),(b) for $\bs\ze,\bs\eta$, and the central rectangle commutes by properties of weak 2-categories. Hence \eq{ku4eq33} commutes.

The two routes round the outside of \eq{ku4eq33} imply that the prescribed values \eq{ku4eq32} for $\bs\th_{ik}$ agree on overlaps between open sets for $j,j'$. As the $\Im\chi_i\cap e^{-1}(\Im\psi_j)\cap (g\ci e)^{-1}(\Im\om_k)$ for $j\in J$ form an open cover of the correct domain $\Im\chi_i\cap (g\ci e)^{-1}(\Im\om_k)$, by Theorem \ref{ku4thm1}(b) and Definition \ref{ku4def11}(iii),(iv), there is a unique 2-morphism $\bs\th_{ik}:(\bs g\ci\bs e)_{ik}\Ra(\bs h\ci\bs f)_{ik}$ satisfying \eq{ku4eq32} for all~$j\in J$.

To show $\bs\th=(\bs\th_{ik}):\bs g\ci\bs e\Ra\bs h\ci\bs f$ is a 2-morphism, we must verify Definition \ref{ku4def14}(a),(b) for $\bs\th$. We do this by first showing that (a),(b) hold on the intersections of their domains with $e^{-1}(\Im\psi_j)$ for $j\in J$ using \eq{ku4eq18}, \eq{ku4eq20}, \eq{ku4eq32}, and Definition \ref{ku4def14} for $\bs\eta,\bs\ze$, and then use Theorem \ref{ku4thm1}(b) and Definition \ref{ku4def11}(iii) to deduce that Definition \ref{ku4def14}(a),(b) for $\bs\th$ hold on their whole domains. So $\bs\th$ is a 2-morphism of Kuranishi spaces.

Define the {\it horizontal composition of\/ $2$-morphisms\/} $\bs\ze*\bs\eta:\bs g\ci\bs e\Ra\bs h\ci\bs f$ to be $\bs\ze*\bs\eta=\bs\th$. By \eq{ku4eq32}, for all $i\in I$, $j\in J$, $k\in K$ we have
\e
(\bs\ze*\bs\eta)_{ik}\od\Th_{ijk}^{\bs g,\bs e}=
\Th_{ijk}^{\bs h,\bs f}\od(\bs\ze_{jk}*\bs\eta_{ij}),
\label{ku4eq34}
\e
and this characterizes $\bs\ze*\bs\eta$ uniquely.
\label{ku4def17}
\end{dfn}

We have now defined all the structures of a {\it weak\/ $2$-category of Kuranishi spaces\/} $\Kur$, as in Appendix \ref{kuB}: objects $\bX,\bY$, 1-morphisms $\bs f,\bs g:\bX\ra\bY$, 2-morphisms $\bs\eta:\bs f\Ra\bs g$, identity 1- and 2-morphisms, composition of 1-morphisms, vertical and horizontal composition of 2-morphisms, 2-iso\-mor\-ph\-isms $\bs\al_{\bs g,\bs f,\bs e}$ in \eq{ku4eq28} for associativity of 1-morphisms, and $\bs\be_{\bs f},\bs\ga_{\bs f}$ in \eq{ku4eq27} for identity 1-morphisms. To show that $\Kur$ is a weak 2-category, it remains only to prove the 2-morphism identities \eq{kuBeq5}, \eq{kuBeq7}, \eq{kuBeq8}, \eq{kuBeq10} and \eq{kuBeq11}. Of these, \eq{kuBeq10}--\eq{kuBeq11} are easy as $\bs\be_{\bs f}=\bs\ga_{\bs f}=\bs\id_{\bs f}$, and we leave them as an exercise. The next three propositions prove \eq{kuBeq5}, \eq{kuBeq7} and \eq{kuBeq8} hold.

\begin{prop} Let\/ $\bs f,\bs{\dot f},\bs{\ddot f}:\bX\ra\bY,$ $\bs g,\bs{\dot g},\bs{\ddot g}:\bY\ra\bZ$ be $1$-morphisms of Kuranishi spaces, and\/ $\bs\eta:\bs f\Ra\bs{\dot f},$ $\bs{\dot\eta}:\bs{\dot f}\Ra\bs{\ddot f},$ $\bs\ze:\bs g\Ra\bs{\dot g},$ $\bs{\dot\ze}:\bs{\dot g}\Ra\bs{\ddot g}$ be $2$-morphisms. Then
\e
(\bs{\dot\ze}\od\bs\ze)*(\bs{\dot\eta}\od\bs\eta)=(\bs{\dot\ze}*\bs{\dot\eta})\od(\bs\ze*\bs\eta):\bs g\ci\bs f\Longra\bs{\ddot g}\ci\bs{\ddot f}.
\label{ku4eq35}
\e

\label{ku4prop3}
\end{prop}

\begin{proof} Use notation \eq{ku4eq12}--\eq{ku4eq14} for $\bX,\bY,\bZ$. For $i\in I$, $j\in J$, $k\in K$ we have
\begin{align*}
\bigl[&(\bs{\dot\ze}\od\bs\ze)*(\bs{\dot\eta}\od\bs\eta)\bigr]{}_{ik}\big\vert_{\Im\chi_i\cap f^{-1}(\Im\psi_j)\cap (g\ci f)^{-1}(\Im\om_k)}\\
&=\Th_{ijk}^{\bs{\ddot g},\bs{\ddot f}}\od\bigl((\bs{\dot\ze}_{jk}\od\bs\ze_{jk})*(\bs{\dot\eta}_{ij}\od\bs\eta_{ij})\bigr)\od(\Th_{ijk}^{\bs g,\bs f})^{-1}\\
&=\Th_{ijk}^{\bs{\ddot g},\bs{\ddot f}}\od\bigl((\bs{\dot\ze}_{jk}*\bs{\dot\eta}_{ij})\od(\bs\ze_{jk}*\bs\eta_{ij})\bigr)\od(\Th_{ijk}^{\bs g,\bs f})^{-1}\\
&=\bigl[\Th_{ijk}^{\bs{\ddot g},\bs{\ddot f}}\od(\bs{\dot\ze}_{jk}*\bs{\dot\eta}_{ij})\od(\Th_{ijk}^{\bs{\dot g},\bs{\dot f}})^{-1}\bigr]\od
\bigl[\Th_{ijk}^{\bs{\dot g},\bs{\dot f}}\od(\bs\ze_{jk}*\bs\eta_{ij})\od(\Th_{ijk}^{\bs g,\bs f})^{-1}\bigr]\\
&=\bigl[(\bs{\dot\ze}*\bs{\dot\eta})\od(\bs\ze*\bs\eta)\bigr]{}_{ik}\vert_{\Im\chi_i\cap f^{-1}(\Im\psi_j)\cap (g\ci f)^{-1}(\Im\om_k)},
\end{align*}
using \eq{ku4eq31} and \eq{ku4eq34} in the first and fourth steps, and compatibility of vertical and horizontal composition for 2-morphisms of Kuranishi neighbourhoods in the second. Since the $\Im\chi_i\cap f^{-1}(\Im\psi_j)\cap (g\ci f)^{-1}(\Im\om_k)$ for all $j\in J$ form an open cover of the domain $\Im\chi_i\cap (g\ci f)^{-1}(\Im\om_k)$, Theorem \ref{ku4thm1}(b) and Definition \ref{ku4def11}(iii) imply that $\bigl[(\bs{\dot\ze}\od\bs\ze)*(\bs{\dot\eta}\od\bs\eta)\bigr]{}_{ik}=\bigl[(\bs{\dot\ze}*\bs{\dot\eta})\od(\bs\ze*\bs\eta)\bigr]{}_{ik}$. As this holds for all $i\in I$ and $k\in K$, equation \eq{ku4eq35} follows.
\end{proof}

\begin{prop} Suppose $\bs e,\bs{\dot e}:\bW\ra\bX,$ $\bs f,\bs{\dot f}:\bX\ra\bY,$ $\bs g,\bs{\dot g}:\bY\ra\bZ$ are $1$-morphisms of Kuranishi spaces, and\/ $\bs\ep:\bs e\Ra\bs{\dot e},$ $\bs\eta:\bs f\Ra\bs{\dot f},$ $\bs\ze:\bs g\Ra\bs{\dot g}$ are $2$-morphisms. Then the following diagram of\/ $2$-morphisms commutes:
\e
\begin{gathered}
\xymatrix@C=150pt@R=14pt{ *+[r]{(\bs g\ci\bs f)\ci\bs e} \ar@{=>}[r]_{\bs\al_{\bs g,\bs f,\bs e}} \ar@{=>}[d]^{(\bs\ze*\bs\eta)*\bs\ep} & *+[l]{\bs g\ci(\bs f\ci\bs e)} \ar@{=>}[d]_{\bs\ze*(\bs\eta*\bs\ep)} \\
*+[r]{(\bs{\dot g}\ci\bs{\dot f})\ci\bs{\dot e}} \ar@{=>}[r]^{\bs\al_{\bs{\dot g},\bs{\dot f},\bs{\dot e}}} & *+[l]{\bs{\dot g}\ci(\bs{\dot f}\ci\bs{\dot e}).\!\!} }
\end{gathered}
\label{ku4eq36}
\e

\label{ku4prop4}
\end{prop}

\begin{proof} Use notation \eq{ku4eq11}--\eq{ku4eq14} for $\bW,\bX,\bY,\bZ$. For $h\in H$, $i\in I$, $j\in J$, $k\in K$ we have
\begin{align*}
\bigl[&(\bs\ze*(\bs\eta*\bs\ep))\od\bs\al_{\bs g,\bs f,\bs e}\bigr]{}_{hk}\big\vert{}_{\Im\vp_h\cap e^{-1}(\Im\chi_i)\cap (f\ci e)^{-1}(\Im\psi_j)\cap (g\ci f\ci e)^{-1}(\Im\om_k)}\\
&=\bigl[\Th_{hjk}^{\bs{\dot g},\bs{\dot f}\ci\bs{\dot e}}\od \bigl[\bs\ze_{jk}*\bigl(\Th_{hij}^{\bs{\dot f},\bs{\dot e}}\od(\bs\eta_{ij}*\bs\ep_{hi})\od(\Th_{hij}^{\bs f,\bs e})^{-1}\bigr)\bigr]\od (\Th_{hjk}^{\bs g,\bs f\ci\bs e})^{-1}\bigr]\\
&\quad \od \bigl[\Th_{hjk}^{\bs g,\bs f\ci\bs e}\od(\id_{\bs g_{jk}}*\Th_{hij}^{\bs f,\bs e})\od\bs\al_{\bs g_{jk},\bs f_{ij},\bs e_{hi}}\od(\Th_{ijk}^{\bs g,\bs f}*\id_{\bs e_{hi}})^{-1}
\od(\Th^{\bs g\ci\bs f,\bs e}_{hik})^{-1}\bigr]\\
&=\Th_{hjk}^{\bs{\dot g},\bs{\dot f}\ci\bs{\dot e}}\od 
(\id_{\bs{\dot g}_{jk}}\!*\!\Th_{hij}^{\bs{\dot f},\bs{\dot e}})\od
(\bs\ze_{jk}*(\bs\eta_{ij}*\bs\ep_{hi}))\\
&\quad\od\bs\al_{\bs g_{jk},\bs f_{ij},\bs e_{hi}}
\od(\Th_{ijk}^{\bs g,\bs f}*\id_{\bs e_{hi}})^{-1}
\od(\Th^{\bs g\ci\bs f,\bs e}_{hik})^{-1}\\
&=\bigl[\Th_{hjk}^{\bs{\dot g},\bs{\dot f}\ci\bs{\dot e}}\od(\id_{\bs{\dot g}_{jk}}*\Th_{hij}^{\bs{\dot f},\bs{\dot e}})\od\bs\al_{\bs{\dot g}_{jk},\bs{\dot f}_{ij},\bs{\dot e}_{hi}}\od(\Th_{ijk}^{\bs{\dot g},\bs{\dot f}}*\id_{\bs{\dot e}_{hi}})^{-1}
\od(\Th^{\bs{\dot g}\ci\bs{\dot f},\bs{\dot e}}_{hik})^{-1}\bigr]\\
&\quad \od 
\bigl[\Th_{hik}^{\bs{\dot g}\ci\bs{\dot f},\bs{\dot e}}\od
\bigl[\bigl(\Th_{ijk}^{\bs{\dot g},\bs{\dot f}}\od(\bs\ze_{jk}*\bs\eta_{ij})\od(\Th_{ijk}^{\bs g,\bs f})^{-1}\bigr)*\bs\ep_{hi}\bigr]
\od (\Th_{hik}^{\bs g\ci\bs f,\bs e})^{-1}\bigr]\\
&=\bigl[\bs\al_{\bs{\dot g},\bs{\dot f},\bs{\dot e}}\od((\bs\ze*\bs\eta)*\bs\ep)\bigr]{}_{hk}\big\vert{}_{\Im\vp_h\cap e^{-1}(\Im\chi_i)\cap (f\ci e)^{-1}(\Im\psi_j)\cap (g\ci f\ci e)^{-1}(\Im\om_k)},
\end{align*}
using \eq{ku4eq29} and \eq{ku4eq34} in the first and fourth steps, and properties of weak 2-categories in the second and third. This proves the restriction of the `$hk$' component of \eq{ku4eq36} to $\Im\vp_h\cap e^{-1}(\Im\chi_i)\cap (f\ci e)^{-1}(\Im\psi_j)\cap (g\ci f\ci e)^{-1}(\Im\om_k)$ commutes. Since these subsets for all $i,j$ form an open cover of the domain, Theorem \ref{ku4thm1}(b) and Definition \ref{ku4def11}(iii) imply that the `$hk$' component of \eq{ku4eq36} commutes for all $h\in H$, $k\in K$, so \eq{ku4eq36} commutes.
\end{proof}

\begin{prop} Let\/ $\bs d:\bV\ra\bW,$ $\bs e:\bW\ra\bX,$ $\bs f:\bX\ra\bY,$ $\bs g:\bY\ra\bZ$ be $1$-morphisms of Kuranishi spaces. Then in $2$-morphisms we have
\e
\begin{split}
\bs\al_{\bs g,\bs f,\bs e\ci\bs d}&\od\bs\al_{\bs g\ci\bs f,\bs e,\bs d}=(\bs\id_{\bs g}*\bs\al_{\bs f,\bs e,\bs d})\od\bs\al_{\bs g,\bs f\ci\bs e,\bs d}\od(\bs\al_{\bs g,\bs f,\bs e}*\bs\id_{\bs d}):\\
&((\bs g\ci\bs f)\ci\bs e)\ci\bs d\Longra\bs g\ci(\bs f\ci(\bs e\ci\bs d)).
\end{split}
\label{ku4eq37}
\e

\label{ku4prop5}
\end{prop}

\begin{proof} Use notation \eq{ku4eq11}--\eq{ku4eq14} for $\bW,\bX,\bY,\bZ$, and take $G$ to be the indexing set for $\bV$. Then for $g\in G$, $h\in H$, $i\in I$, $j\in J$, $k\in K$, on $\Im\up_g\cap\ab
d^{-1}(\Im\vp_h)\ab\cap (e\ci d)^{-1}(\Im\chi_i)\cap (f\ci e\ci d)^{-1}(\Im\psi_j)\cap (g\ci f\ci e\ci d)^{-1}(\Im\om_k)$ we have
\begin{align*}
&\bigl[(\bs\al_{\bs g,\bs f,\bs e\ci\bs d})\od(\bs\al_{\bs g\ci\bs f,\bs e,\bs d})\bigr]{}_{gk}\big\vert{}_{{\ldots}}\\
&=\bigl\{\Th_{gjk}^{\bs g,\bs f\ci(\bs e\ci\bs d)}\od(\id_{\bs g_{jk}}*\Th_{gij}^{\bs f,\bs e\ci\bs d})\od\al_{\bs g_{jk},\bs f_{ij},(\bs e\ci\bs d)_{gi}}\od(\Th_{ijk}^{\bs g,\bs f}*\id_{(\bs e\ci\bs d)_{gi}})^{-1}\\
&\quad\od(\Th^{\bs g\ci\bs f,\bs e\ci\bs d}_{gik})^{-1}\bigr\}\od
\bigl\{\Th_{gik}^{\bs g\ci\bs f,\bs e\ci\bs d}\od(\id_{(\bs g\ci\bs f)_{ik}}*\Th_{ghi}^{\bs e,\bs d})\od\al_{(\bs g\ci\bs f)_{ik},\bs e_{hi},d_{gh}}\\
&\quad\od(\Th_{hik}^{\bs g\ci\bs f,\bs e}*\id_{\bs d_{gh}})^{-1}\od(\Th^{(\bs g\ci\bs f)\ci\bs e,\bs d}_{ghk})^{-1}\bigr\}\\
&=\Th_{gjk}^{\bs g,\bs f\ci(\bs e\ci\bs d)}\od(\id_{\bs g_{jk}}*\Th_{gij}^{\bs f,\bs e\ci\bs d})\od\al_{\bs g_{jk},\bs f_{ij},(\bs e\ci\bs d)_{gi}}
\od((\Th_{ijk}^{\bs g,\bs f})^{-1}*\Th_{ghi}^{\bs e,\bs d})\\
&\quad\od\al_{(\bs g\ci\bs f)_{ik},\bs e_{hi},\bs d_{gh}}
\od(\Th_{hik}^{\bs g\ci\bs f,\bs e}*\id_{\bs d_{gh}})^{-1}\od(\Th^{(\bs g\ci\bs f)\ci\bs e,\bs d}_{ghk})^{-1}\\
&=\bigl\{\Th_{gjk}^{\bs g,\bs f\ci(\bs e\ci\bs d)}\od\bigl(\id_{\bs g_{jk}}*\bigl[ 
\Th_{gij}^{\bs f,\bs e\ci\bs d}\od(\id_{\bs f_{ij}}*\Th_{ghi}^{\bs e,\bs d})\od\al_{\bs f_{ij},\bs e_{hi},\bs d_{gh}}\\
&\quad\od(\Th_{hij}^{\bs f,\bs e}*\id_{\bs d_{gh}})^{-1}\od(\Th^{\bs f\ci\bs e,\bs d}_{ghj})^{-1}\bigr]\bigr)\od(\Th_{gjk}^{\bs g,(\bs f\ci\bs e)\ci\bs d})^{-1}\bigr\}\od
\bigl\{
\Th_{gjk}^{\bs g,(\bs f\ci\bs e)\ci\bs d}\\
&\quad\od(\id_{\bs g_{jk}}*\Th_{ghj}^{\bs f\ci\bs e,\bs d})\od\al_{\bs g_{jk},(\bs f\ci\bs e)_{hj},\bs d_{gh}}
\od(\Th_{hjk}^{\bs g,\bs f\ci\bs e}*\id_{\bs d_{gh}})^{-1}\od(\Th^{\bs g\ci(\bs f\ci\bs e),\bs d}_{ghk})^{-1}\bigr\}\\
&\quad\od
\bigl\{\Th^{\bs g\ci(\bs f\ci\bs e),\bs d}_{ghk})\od\bigl(\bigl[
\Th_{hjk}^{\bs g,\bs f\ci\bs e}\od(\id_{\bs g_{jk}}*\Th_{hij}^{\bs f,\bs e})
\od\al_{\bs g_{jk},\bs f_{ij},\bs e_{hi}}\od(\Th_{ijk}^{\bs g,\bs f}*\id_{\bs e_{hi}})^{-1}\\
&\quad\od(\Th^{\bs g\ci\bs f,\bs e}_{hik})^{-1}
\bigr]*\id_{\bs d_{gh}}\bigr)\od(\Th^{(\bs g\ci\bs f)\ci\bs e,\bs d}_{ghk})^{-1}\bigr\}\\
&=\bigl[(\bs\id_{\bs g}*\bs\al_{\bs f,\bs e,\bs d})\od\bs\al_{\bs g,\bs f\ci\bs e,\bs d}\od(\bs\al_{\bs g,\bs f,\bs e}*\bs\id_{\bs d})\bigr]{}_{gk}\big\vert{}_{\ldots},
\end{align*}
using \eq{ku4eq29} and \eq{ku4eq34} in the first and fourth steps, and properties of weak 2-categories in the second and third. This proves the restriction of the `$gk$' component of \eq{ku4eq37} to the subset $\Im\up_g\cap d^{-1}(\Im\vp_h)\cap (e\ci d)^{-1}(\Im\chi_i)\cap (f\ci e\ci d)^{-1}(\Im\psi_j)\cap (g\ci f\ci e\ci d)^{-1}(\Im\om_k)$. Since these subsets for all $h,i,j$ form an open cover of the domain, Theorem \ref{ku4thm1}(b) and Definition \ref{ku4def11}(iii) imply that the `$gk$' component of \eq{ku4eq37} commutes for all $g\in G$ and $k\in K$, so \eq{ku4eq37} commutes.
\end{proof}

We summarize the work of this section in the following:

\begin{thm} The definitions and propositions above define a \begin{bfseries}weak\/ $2$-cat\-eg\-ory of Kuranishi spaces\end{bfseries}~$\Kur$.
\label{ku4thm3}
\end{thm}

\begin{rem}{\bf(a)} We proved in \S\ref{ku41} that Kuranishi neighbourhoods over $S\subseteq X$ form a weak 2-category $\Kur_S(X)$, and now we have shown that Kuranishi spaces also form a weak 2-category $\Kur$. But morally, $\Kur_S(X)$ is closer to being a strict 2-category. In $\Kur_S(X)$ there is a natural notion of composition of 1-morphisms $\Phi_{jk}\ci\Phi_{ij}$, but it just fails to be strictly associative, as the canonical isomorphism of fibre products $\la_{ijkl}$ in \eq{ku4eq3} is not the identity. The analogue $\mKur_S(X)$ for m-Kuranishi spaces in \S\ref{ku47} is a strict 2-category.

In $\Kur$, there is no natural notion of composition of 1-morphisms $\bs g\ci\bs f$, so as in Definition \ref{ku4def15} we have to choose $\bs g\ci\bs f$ using the Axiom of Global Choice, and composition of 1-morphisms in $\Kur$ is far from being strictly associative.
\smallskip

\noindent{\bf(b)} The 2-category $\GKur$ of global Kuranishi neighbourhoods in Definition \ref{ku4def7} is equivalent to the full 2-subcategory of $\Kur$ with objects $(X,\cK)\in\Kur$ whose Kuranishi structure $\cK$ contains only one Kuranishi neighbourhood.
\label{ku4rem6}
\end{rem}

In a similar way to Definition \ref{ku2def14}, we can define a (weak 2-)functor from manifolds to Kuranishi spaces. Weak 2-functors are defined in \S\ref{kuB2}. Since $\Kur$ is a weak 2-category, no other kind of functor to $\Kur$ makes sense.

\begin{dfn} We will define a weak 2-functor $F_\Man^\Kur:\Man\ra\Kur$. If $X$ is a manifold, define a Kuranishi space $F_\Man^\Kur(X)=\bX=(X,\cK)$ with topological space $X$ and Kuranishi structure $\cK=\bigl(\{0\},(V_0,E_0,\Ga_0,s_0,\psi_0),\ab\Phi_{00},\ab\La_{000}\bigr)$, with indexing set $I=\{0\}$, one Kuranishi neighbourhood $(V_0,E_0,\Ga_0,s_0,\psi_0)$ with $V_0=X$, $E_0\ra V_0$ the zero vector bundle, $\Ga_0=\{1\}$, $s_0=0$, and $\psi_0=\id_X$, one coordinate change $\Phi_{00}=\id_{(V_0,E_0,\Ga_0,s_0,\psi_0)}$, and one 2-morphism~$\La_{000}=\id_{\Phi_{00}}$.

On (1-)morphisms, if $f:X\ra Y$ is a smooth map of manifolds and $\bX=F_\Man^\Kur(X)$, $\bY=F_\Man^\Kur(Y)$, define a 1-morphism $F_\Man^\Kur(f)=\bs f:\bX\ra\bY$ by $\bs f=(f,\bs f_{00},\bs F_{00}^0,\bs F_0^{00})$, where $\bs f_{00}=(P_{00},\pi_{00},f_{00},\hat f_{00})$ with $P_{00}=X$, $\pi_{00}=\id_X$, $f_{00}=f$, and $\hat f_{00}$ is the zero map on zero vector bundles, and~$\bs F_{00}^0=\bs F_0^{00}=\id_{\bs f_{00}}$.

On 2-morphisms, regarding $\Man$ as a 2-category, the only 2-morphisms are identity morphisms $\id_f:f\Ra f$ for (1-)morphisms $f:X\ra Y$ in $\Man$. We define~$F_\Man^\Kur(\id_f)=\id_{F_\Man^\Kur(f)}$.

Suppose $f:X\ra Y$, $g:Y\ra Z$ are (1-)morphisms in $\Man$, and write $\bX,\bY,\bZ,\bs f,\bs g$ for the images of $X,Y,Z,f,g$ under $F_\Man^\Kur$. As in \eq{ku4eq12}--\eq{ku4eq14}, write $(U_0,D_0,\Be_0,r_0,\chi_0),(V_0,E_0,\Ga_0,s_0,\psi_0),(W_0,F_0,\De_0,t_0,\om_0)$ for the Kuranishi neighbourhoods on $\bX,\bY,\bZ$. Let $\bs h=\bs g\ci\bs f$ be the composition in $\Kur$ from Definition \ref{ku4def15}, where $\bs h=\bigl(g\ci f,\bs h_{00},\bs H_{00}^0,\bs H_0^{00}\bigr)$ with $\bs h_{00}=(Q_{00},\pi_{00},h_{00},\hat h_{00})$. Then $\pi_{00}:Q_{00}\ra U_{00}$ is a principal $\De_0$-bundle over an open neighbourhood $U_{00}$ of $r_0^{-1}(0)$ in $U_0=X$. Since $r_0=0$ this forces $U_{00}=X$, so $\pi_{00}:Q_{00}\ra X$ is a diffeomorphism as $\De_0=\{1\}$. Also $h_{00}=g\ci f\ci\pi_{00}:Q_{00}\ra W_0=Z$ and $\hat h_{00}=0$, and $\bs H_{00}^0,\bs H_0^{00}$ are identity 2-morphisms. Then
\begin{equation*}
(F_\Man^\Kur)_{g,f}:=\bigl([Q_{00},\pi_{00},0]_{00}\bigr):F_\Man^\Kur(g)\ci F_\Man^\Kur(f)\Longra F_\Man^\Kur(g\ci f)
\end{equation*}
is a 2-morphism in $\Kur$. For any object $X$ in $\Man$, define
\begin{equation*}
(F_\Man^\Kur)_X:=\bs\id_{\bs\id_\bX}:F_\Man^\Kur(\id_X)=\bs\id_\bX\Longra\bs\id_\bX=\bs\id_{F_\Man^\Kur(X)}.
\end{equation*}

We have now defined all the data of a weak 2-functor $F_\Man^\Kur:\Man\ra\Kur$ in Definition \ref{kuBdef2}. It is easy to check that $F_\Man^\Kur$ is a weak 2-functor, which is full and faithful. Applying $F_\Man^\Kur$, we can regard manifolds $X$ as special examples of Kuranishi spaces. We will usually write manifolds as $X,Y,Z,\ldots,$ rather than in bold $\bX,\bY,\bZ,\ldots,$ even when we regard them as Kuranishi spaces, so we write $F_\Man^\Kur(X)$ as $X$ rather than~$\bX$.

We say that a Kuranishi space $\bX$ {\it is a manifold\/} if $\bX\simeq F_\Man^\Kur(X')$ in $\Kur$, for some manifold~$X'$.
\label{ku4def18}
\end{dfn}

As in Example \ref{ku2ex4}, if $\bX,\bY$ are Kuranishi spaces with notation \eq{ku4eq12}--\eq{ku4eq13}, we can define an explicit {\it product Kuranishi space\/} $\bX\t\bY=(X\t Y,\cK)$, where $\cK$ has indexing set $I\t J$ and Kuranishi neighbourhoods
\begin{align*}
&(W_{(i,j)},F_{(i,j)},\De_{(i,j)},t_{(i,j)},\om_{(i,j)})=\\
&\bigl(U_i\t V_j,\pi_{U_i}^*(D_i)\op \pi_{V_j}^*(E_j),\Be_i\t\Ga_j,\pi_{U_i}^*(r_i)\op \pi_{V_j}^*(s_j),\chi_i\t\psi_j\bigr),
\end{align*}
with $\vdim(\bX\t\bY)=\vdim\bX+\vdim\bY$, and projection 1-morphisms $\bs\pi_\bX:\bX\t\bY\ra\bX$, $\bs\pi_\bY:\bX\t\bY\ra\bY$. It will follow from results in \cite{Joyc12} that these satisfy the universal property of products in a weak 2-category.

\subsection{Kuranishi neighbourhoods on Kuranishi spaces}
\label{ku44}

Here is the analogue of Definition \ref{ku2def15} in~\S\ref{ku24}.

\begin{dfn} Suppose $\bX=(X,\cK)$ is a Kuranishi space, where $\cK=\bigl(I,(V_i,E_i,\Ga_i,s_i,\psi_i)_{i\in I}$, $\Phi_{ij,\;i,j\in I}$, $\La_{ijk,\; i,j,k\in I}\bigr)$. A {\it Kuranishi neighbourhood on the Kuranishi space\/} $\bX$ is data $(V_a,E_a,\Ga_a,s_a,\psi_a)$, $\Phi_{ai,\; i\in I}$ and $\La_{aij,\; i,j\in I}$ where $(V_a,E_a,\Ga_a,s_a,\psi_a)$ is a Kuranishi neighbourhood on the topological space $X$ in the sense of Definition \ref{ku4def1}, and $\Phi_{ai}:(V_a,E_a,\Ga_a,s_a,\psi_a)\ra(V_i,E_i,\Ga_i,s_i,\psi_i)$ is a coordinate change for each $i\in I$ (over $S=\Im\psi_a\cap\Im\psi_i$, as usual) as in Definition \ref{ku4def8}, and $\La_{aij}:\Phi_{ij}\ci\Phi_{ai}\Ra\Phi_{aj}$ is a 2-morphism (over $S=\Im\psi_a\cap\Im\psi_i\cap\Im\psi_j$, as usual) as in Definition \ref{ku4def3} for all $i,j\in I$, such that $\La_{aii}=\id_{\Phi_{ai}}$ for all $i\in I$, and as in Definition \ref{ku4def12}(h), for all $i,j,k\in I$ we have
\e
\begin{split}
\La_{ajk}\od(\id_{\Phi_{jk}}*\La_{aij})&\od\bs\al_{\Phi_{jk},\Phi_{ij},\Phi_{ai}}=\La_{aik}\od(\La_{ijk}*\id_{\Phi_{ai}}):\\
&(\Phi_{jk}\ci\Phi_{ij})\ci\Phi_{ai}\Longra\Phi_{ak},
\end{split}
\label{ku4eq38}
\e
where \eq{ku4eq38} holds over $S=\Im\psi_a\cap\Im\psi_i\cap\Im\psi_j\cap\Im\psi_k$ by our usual convention.

Here the subscript `$a$' in $(V_a,E_a,\Ga_a,s_a,\psi_a)$ is just a label used to distinguish Kuranishi neighbourhoods, generally not in $I$. If we omit $a$ we will write `$*$' in place of `$a$' in $\Phi_{ai},\La_{aij}$, giving $\Phi_{*i}:(V,E,\Ga,s,\psi)\ra(V_i,E_i,\Ga_i,s_i,\psi_i)$ and~$\La_{*ij}:\Phi_{ij}\ci\Phi_{*i}\Ra\Phi_{*j}$.

We will usually just say $(V_a,E_a,\Ga_a,s_a,\psi_a)$ or $(V,E,\Ga,s,\psi)$ {\it is a Kuranishi neighbourhood on\/} $\bX$, leaving the data $\Phi_{ai},\La_{aij}$ or $\Phi_{*i},\La_{*ij}$ implicit. We call such a $(V,E,\Ga,s,\psi)$ a {\it global\/} Kuranishi neighbourhood on $\bX$ if~$\Im\psi=X$. 
\label{ku4def19}
\end{dfn}

\begin{dfn} Using the same notation, suppose $(V_a,E_a,\Ga_a,s_a,\psi_a), \Phi_{ai,\; i\in I}$, $\La_{aij,\; i,j\in I}$ and $(V_b,E_b,\Ga_b,s_b,\psi_b)$, $\Phi_{bi,\; i\in I}$, $\La_{bij,\; i,j\in I}$ are Kuranishi neighbourhoods on $\bX$, and $S\subseteq \Im\psi_a\cap\Im\psi_b$ is open. A {\it coordinate change from $(V_a,E_a,\Ga_a,s_a,\psi_a)$ to $(V_b,E_b,\Ga_b,s_b,\psi_b)$ over $S$ on the Kuranishi space\/} $\bX$ is data $\Phi_{ab}$, $\La_{abi,\; i\in I}$, where $\Phi_{ab}:(V_a,E_a,\Ga_a,s_a,\psi_a)\ab\ra (V_b,E_b,\Ga_b,s_b,\psi_b)$ is a coordinate change over $S$ as in Definition \ref{ku4def8}, and $\La_{abi}:\Phi_{bi}\ci\Phi_{ab}\Ra\Phi_{ai}$ is a 2-morphism over $S\cap\Im\psi_i$ as in Definition \ref{ku4def3} for each $i\in I$, such that for all $i,j\in I$ we have
\e
\begin{split}
\La_{aij}\od(\id_{\Phi_{ij}}*\La_{abi})&\od\bs\al_{\Phi_{ij},\Phi_{bi},\Phi_{ab}}=\La_{abj}\od(\La_{bij}*\id_{\Phi_{ab}}):\\
&(\Phi_{ij}\ci\Phi_{bi})\ci\Phi_{ab}\Longra\Phi_{aj},
\end{split}
\label{ku4eq39}
\e
where \eq{ku4eq39} holds over $S\cap\Im\psi_i\cap\Im\psi_j$.

We will usually just say that $\Phi_{ab}:(V_a,E_a,\Ga_a,s_a,\psi_a)\ra (V_b,E_b,\Ga_b,s_b,\psi_b)$ {\it is a coordinate change over\/ $S$ on\/} $\bX$, leaving the data $\La_{abi,\; i\in I}$ implicit.

If we do not specify $S$, we mean that $S$ is as large as possible, that is,~$S=\Im\psi_a\cap\Im\psi_b$.

Suppose $\Phi_{ab}:(V_a,E_a,\Ga_a,s_a,\psi_a)\ra (V_b,E_b,\Ga_b,s_b,\psi_b)$, $\La_{abi,\; i\in I}$ and $\Phi_{bc}:(V_b,E_b,\Ga_b,s_b,\psi_b)\ra (V_c,E_c,\Ga_c,s_c,\psi_c)$, $\La_{bci,\; i\in I}$ are such coordinate changes over $S\subseteq\Im\psi_a\cap\Im\psi_b\cap\Im\psi_c$. Define $\Phi_{ac}=\Phi_{bc}\ci\Phi_{ab}:(V_a,E_a,\Ga_a,s_a,\psi_a)\ra (V_c,E_c,\Ga_c,s_c,\psi_c)$ and $\La_{aci}=\La_{abi}\od(\La_{bci}*\id_{\Phi_{ab}})\od\bs\al_{\Phi_{ci},\Phi_{bc},\Phi_{ab}}^{-1}:\Phi_{ci}\ci\Phi_{ac}\Ra\Phi_{ai}$ for all $i\in I$. It is easy to show that $\Phi_{ac}=\Phi_{bc}\ci\Phi_{ab}$, $\La_{aci,\; i\in I}$ is a coordinate change from $(V_a,E_a,\Ga_a,s_a,\psi_a)$ to $(V_c,E_c,\Ga_c,s_c,\psi_c)$ over $S$ on $\bX$. We call this {\it composition of coordinate changes}.
\label{ku4def20}
\end{dfn}

\begin{dfn} Let $\bs f:\bX\ra\bY$ be a 1-morphism of Kuranishi spaces, and use notation \eq{ku4eq12}--\eq{ku4eq13} for $\bX,\bY$, and \eq{ku4eq15} for $\bs f$. Suppose $(U_a,D_a,\Be_a,r_a,\chi_a)$, $\Tau_{ai,\; i\in I}$, $\Ka_{aii',\; i,i'\in I}$ is a Kuranishi neighbourhood on $\bX$, and $(V_b,E_b,\Ga_b,s_b,\psi_b)$, $\Up_{bj,\; j\in J}$, $\La_{bjj',\; j,j'\in J}$ a Kuranishi neighbourhood on $\bY$, as in Definition \ref{ku4def19}. Let $S\subseteq \Im\chi_a\cap f^{-1}(\Im\psi_b)$ be open. A 1-{\it morphism from $(U_a,D_a,\Be_a,r_a,\chi_a)$ to $(V_b,E_b,\Ga_b,s_b,\psi_b)$ over\/ $(S,\bs f)$ on the Kuranishi spaces\/} $\bX,\bY$ is data $\bs f_{ab}$, $\bs F_{ai,\; i\in I}^{bj,\; j\in J}$, where $\bs f_{ab}:(U_a,D_a,\Be_a,r_a,\chi_a)\ra (V_b,E_b,\Ga_b,s_b,\psi_b)$ is a 1-morphism of Kuranishi neighbourhoods over $(S,f)$ in the sense of Definition \ref{ku4def2}, and $\bs F_{ai}^{bj}:\Up_{bj}\ci\bs f_{ab}\Ra \bs f_{ij}\ci\Tau_{ai}$ is a 2-morphism over $S\cap\Im\chi_i\cap f^{-1}(\Im\psi_j),f$ as in Definition \ref{ku4def3} for all $i\in I$, $j\in J$, such that for all $i,i'\in I$, $j,j'\in J$ we have
\begin{align*}
(\bs F_{ai}^{bj})^{-1}\od(\bs F_{ii'}^j*\id_{\Tau_{ai}})&=(\bs F_{ai'}^{bj})^{-1}\od(\id_{\bs f_{i'j}}*\Ka_{aii'})\od\bs\al_{\bs f_{i'j},\Tau_{ii'},\Tau_{ai}}:\\
&(\bs f_{i'j}\ci\Tau_{ii'})\ci\Tau_{ai}\Longra\Up_{bj}\ci\bs f_{ab},\\
\bs F_{ai}^{bj'}\od(\La_{bjj'}*\id_{\bs f_{ab}})&=(\bs F_i^{jj'}*\id_{\Tau_{ai}})\od(\id_{\Up_{jj'}}*\bs F_{ai}^{bj})\od\bs\al_{\Up_{jj'},\Up_{bj},\bs f_{ab}}:\\
&(\Up_{jj'}\ci\Up_{bj})\ci\bs f_{ab}\Longra\bs f_{ij'}\ci\Tau_{ai}.
\end{align*}

We will usually just say that $\bs f_{ab}:(U_a,D_a,\Be_a,r_a,\chi_a)\ra (V_b,E_b,\Ga_b,s_b,\psi_b)$ {\it is a $1$-morphism of Kuranishi neighbourhoods over $(S,\bs f)$ on\/} $\bX,\bY$, leaving the data $\bs F_{ai,\; i\in I}^{bj,\; j\in J}$ implicit.

Suppose $\bs g:\bY\ra\bZ$ is another 1-morphism of Kuranishi spaces, using notation \eq{ku4eq14} for $\bZ$, and $(W_c,F_c,\De_c,t_c,\om_c)$ is a Kuranishi neighbourhood on $\bY$, and $T\subseteq \Im\psi_b\cap g^{-1}(\Im\om_c)$, $S\subseteq \Im\chi_a\cap f^{-1}(T)$ are open, $\bs f_{ab}:(U_a,D_a,\Be_a,r_a,\chi_a)\ra (V_b,E_b,\Ga_b,s_b,\psi_b)$ is a 1-morphism of Kuranishi neighbourhoods over $(S,\bs f)$ on $\bX,\bY$, and $\bs g_{bc}:(V_b,E_b,\Ga_b,s_b,\psi_b)\ra (W_c,\ab F_c,\ab\De_c,\ab t_c,\ab\om_c)$ is a 1-morphism of Kuranishi neighbourhoods over $(T,\bs g)$ on $\bY,\bZ$.

Define $\bs h=\bs g\ci\bs f:\bX\ra\bZ$, so that Definition \ref{ku4def15} gives 2-morphisms 
\begin{equation*}
\Th_{ijk}^{\bs g,\bs f}:\bs g_{jk}\ci\bs f_{ij}\Longra \bs h_{ik}
\end{equation*}
for all $i\in I$, $j\in J$ and $k\in K$. Set $\bs h_{ac}=\bs g_{bc}\ci\bs f_{ab}:(U_a,D_a,\Be_a,r_a,\chi_a)\ra(W_c,F_c,\De_c,t_c,\om_c)$. Using the stack property Theorem \ref{ku4thm1}(b), one can show that for all $i\in I$, $k\in K$ there is a unique 2-morphism $\bs H_{ai}^{ck}:\Phi_{ck}\ci\bs h_{ac}\Ra \bs h_{ik}\ci\Tau_{ai}$ over $S\cap\Im\chi_i\cap h^{-1}(\Im\om_k),h$, such that for all $j\in J$ we have
\begin{align*}
&\bs H_{ai}^{ck}\vert_{S\cap\Im\chi_i\cap f^{-1}(\Im\psi_j)\cap h^{-1}(\Im\om_k)}=(\Th_{ijk}^{\bs g,\bs f}*\id_{\Tau_{ai}})\od\bs\al_{\bs g_{jk},\bs f_{ij},\Tau_{ai}}^{-1}\\
&\qquad\od(\id_{\bs g_{jk}}*\bs F_{ai}^{bj})\od\bs\al_{\bs g_{jk},\Up_{bj},\bs f_{ab}}\od(\bs G_{bj}^{ck}*\id_{\bs f_{ab}})\od\bs\al_{\Phi_{ck},\bs g_{bc},\bs f_{ab}}^{-1}.
\end{align*}
It is then easy to prove that $\bs h_{ac}=\bs g_{bc}\ci\bs f_{ab}$, $\bs H_{ai,\; i\in I}^{ck,\; k\in K}$ is a 1-morphism from $(U_a,D_a,\Be_a,r_a,\chi_a)$ to $(W_c,F_c,\De_c,t_c,\om_c)$ over $(S,\bs h)$ on $\bX,\bZ$. We call this {\it composition of\/ $1$-morphisms}.
\label{ku4def21}
\end{dfn}

As for Theorem \ref{ku2thm4}, the next theorem can be proved using the stack property Theorem \ref{ku4thm1} by very similar methods to Propositions \ref{ku4prop1}, \ref{ku4prop2}, \ref{ku4prop3}, \ref{ku4prop4} and \ref{ku4prop5}, so we leave the proof as an exercise for the reader.

\begin{thm}{\bf(a)} Let\/ $\bX=(X,\cK)$ be a Kuranishi space, where $\cK=\bigl(I,\ab(V_i,\ab E_i,\ab\Ga_i,\ab s_i,\ab\psi_i)_{i\in I},$ $\Phi_{ij}$, $\La_{ijk}\bigr),$ and\/ $(V_a,E_a,\Ga_a,s_a,\psi_a),$ $(V_b,E_b,\Ga_b,s_b,\psi_b)$ be Kuranishi neighbourhoods on $\bX,$ in the sense of Definition\/ {\rm\ref{ku4def19},} and\/ $S\subseteq\Im\psi_a\cap\Im\psi_b$ be open. Then there exists a coordinate change $\Phi_{ab}:(V_a,E_a,\Ga_a,s_a,\psi_a)\ab\ra (V_b,E_b,\Ga_b,s_b,\psi_b),\La_{abi,\; i\in I}$ over $S$ on $\bX,$ in the sense of Definition\/ {\rm\ref{ku4def20}}. If\/ $\Phi_{ab},\ti\Phi_{ab}$ are two such coordinate changes, there is a unique\/ $2$-morphism $\Xi_{ab}:\Phi_{ab}\Ra\ti\Phi_{ab}$ over $S$ as in Definition\/ {\rm\ref{ku4def3},} such that for all\/ $i\in I$ we have
\e
\La_{abi}=\ti\La_{abi}\od(\id_{\Phi_{bi}}*\Xi_{ab}):\Phi_{bi}\ci\Phi_{ab}\Longra\Phi_{ai},
\label{ku4eq40}
\e
which holds over $S\cap\Im\psi_i$ by our usual convention.
\smallskip

\noindent{\bf(b)} Let\/ $\bs f:\bX\ra\bY$ be a $1$-morphism of Kuranishi spaces, and use notation {\rm\eq{ku4eq12}, \eq{ku4eq13}, \eq{ku4eq15}}. Let\/ $(U_a,D_a,\Be_a,r_a,\chi_a),$ $(V_b,E_b,\Ga_b,s_b,\psi_b)$ be Kuranishi neighbourhoods on $\bX,\bY$ respectively in the sense of Definition\/ {\rm\ref{ku4def19},} and let\/ $S\subseteq \Im\chi_a\cap f^{-1}(\Im\psi_b)$ be open. Then there exists a $1$-morphism $\bs f_{ab}:(U_a,D_a,\Be_a,r_a,\chi_a)\ra (V_b,E_b,\Ga_b,s_b,\psi_b)$ of Kuranishi neighbourhoods over\/ $(S,\bs f)$ on\/ $\bX,\bY,$ in the sense of Definition\/ {\rm\ref{ku4def21}}.

If\/ $\bs f_{ab},\bs{\ti f}_{ab}$ are two such\/ $1$-morphisms, there is a unique\/ $2$-morphism $\Xi_{ab}:\bs f_{ab}\Ra\bs{\ti f}_{ab}$ over $(S,f)$ as in Definition\/ {\rm\ref{ku4def3},} such that for\/ $i\in I,$ $j\in J$ we have
\begin{equation*}
\bs F_{ai}^{bj}=\bs{\ti F}{}_{ai}^{bj}\od(\id_{\Up_{bj}}*\Xi_{ab}):\Up_{bj}\ci\bs f_{ab}\Longra \bs f_{ij}\ci\Tau_{ai},
\end{equation*}
which holds over\/ $S\cap\Im\chi_i\cap f^{-1}(\Im\psi_j)$ by our usual convention.
\label{ku4thm4}
\end{thm}

\begin{rem} Note that we make the (potentially confusing) distinction between {\it Kuranishi neighbourhoods\/ $(V_i,E_i,\Ga_i,s_i,\psi_i)$ on a topological space\/} $X$, as in Definition \ref{ku4def1}, and {\it Kuranishi neighbourhoods\/ $(V_a,E_a,\Ga_a,s_a,\psi_a)$ on a Kuranishi space\/} $\bX=(X,\cK)$, which are as in Definition \ref{ku4def19}, and come equipped with the extra implicit data $\Phi_{ai,\; i\in I}$, $\La_{aij,\; i,j\in I}$ giving the compatibility with the Kuranishi structure $\cK$ on~$X$. 

We also distinguish between {\it coordinate changes\/ $\Phi_{ij}:(V_i,E_i,\Ga_i,s_i,\psi_i)\ra(V_j,E_j,\Ga_j,s_j,\psi_j)$ between Kuranishi neighbourhoods on a topological space\/} $X$, which are as in Definition \ref{ku4def8} and for which there may be many choices or none, and {\it coordinate changes\/ $\Phi_{ab}:(V_a,E_a,s_a,\psi_a)\ra(V_b,E_b,s_b,\psi_b)$ between Kuranishi neighbourhoods on a Kuranishi space\/} $\bX$, which are as in Definition \ref{ku4def20}, and come equipped with implicit extra data $\La_{abi,\; i\in I}$, and which by Theorem \ref{ku4thm4}(a) always exist, and are unique up to unique 2-isomorphism.

Similarly, we distinguish between 1-{\it morphisms $\bs f_{ij}:(U_i,D_i,\Be_i,r_i,\chi_i)\ra(V_j,E_j,\Ga_j,s_j,\psi_j)$ of Kuranishi neighbourhoods over a continuous map of topological spaces\/} $f:X\ra Y$, which are as in Definition \ref{ku4def2} and for which there may be many choices or none, and 1-{\it morphisms $\bs f_{ab}:(U_a,D_a,\Be_a,r_a,\chi_a)\ra(V_b,E_b,\Ga_b,s_b,\psi_b)$ of Kuranishi neighbourhoods over a $1$-morphism of Kuranishi spaces\/} $\bs f:\bX\ra\bY$, which are as Definition \ref{ku4def21}, and come equipped with implicit extra data $\bs F_{ai,\; i\in I}^{bj,\; j\in J}$, and which by Theorem \ref{ku4thm4}(b) always exist, and are unique up to unique 2-isomorphism.
\label{ku4rem7}
\end{rem}

Here are the analogues of Proposition \ref{ku2prop} and Corollary \ref{ku2cor}. The proofs are similar but a little more complicated, and we leave them as an exercise.

\begin{prop} Let\/ $\bX=(X,\cK)$ be a Kuranishi space, and\/ $\bigl\{(V_a,\ab E_a,\ab\Ga_a,\ab s_a,\ab\psi_a):a\in A\bigr\}$ a family of Kuranishi neighbourhoods on $\bX$ with\/ $X=\bigcup_{a\in A}\Im\psi_a$. For all\/ $a,b\in A,$ let\/ $\Phi_{ab}:(V_a,E_a,\Ga_a,s_a,\psi_a)\ra(V_b,E_b,\Ga_b,s_b,\psi_b)$ be a coordinate change over $S=\Im\psi_a\cap\Im\psi_b$ on $\bX$ given by Theorem\/ {\rm\ref{ku4thm4}(a),} which is unique up to $2$-isomorphism. For all\/ $a,b,c\in A,$ both\/ $\Phi_{bc}\ci\Phi_{ab}\vert_S$ and\/ $\Phi_{ac}\vert_S$ are coordinate changes $(V_a,E_a,\Ga_a,s_a,\psi_a)\ra(V_c,E_c,\Ga_c,s_c,\psi_c)$ over $S=\Im\psi_a\cap\Im\psi_b\cap\Im\psi_c$ on $\bX,$ so Theorem\/ {\rm\ref{ku4thm4}(a)} gives a unique $2$-morphism $\La_{abc}:\Phi_{bc}\ci\Phi_{ab}\vert_S\Ra \Phi_{ac}\vert_S$. Then\/ $\cK'=\bigl(A,(V_a,E_a,\Ga_a,s_a,\psi_a)_{a\in A},\ab\Phi_{ab,\; a,b\in A},\ab\La_{abc,\; a,b,c\in A}\bigr)$ is a Kuranishi structure on $X,$ and\/ $\bX'=(X,\cK')$ is equivalent to\/ $\bX$ in\/ $\Kur,$ by an equivalence $\bs i:\bX\ra\bX'$ canonical up to $2$-isomorphism.
\label{ku4prop6}
\end{prop}

\begin{cor} Let\/ $\bX=(X,\cK)$ be a Kuranishi space with\/ $\cK=\bigl(I,(V_i,\ab E_i,\ab\Ga_i,\ab s_i,\ab\psi_i)_{i\in I},\Phi_{ij,\; i,j\in I},\La_{ijk,\;i,j,k\in I}\bigr)$. Suppose $J\subseteq I$ with $\bigcup_{j\in J}\Im\psi_j=X$. Then $\cK'=\bigl(J,(V_i,E_i,\Ga_i,s_i,\psi_i)_{i\in J},\Phi_{ij,\; i,j\in J},\La_{ijk,\;i,j,k\in J}\bigr)$ is a Kuranishi structure on $X,$ and\/ $\bX'=(X,\cK')$ is equivalent to $\bX$ in $\Kur$.
\label{ku4cor1}
\end{cor}

Just as in \S\ref{ku24}, we can now state our philosophy for working with Kuranishi spaces, but we will not explain it in detail.

\subsection{Orbifolds and Kuranishi spaces}
\label{ku45}

The next remark reviews different definitions of orbifolds in the literature.

\begin{rem} Orbifolds are generalizations of manifolds locally modelled on $\R^n/G$, for $G$ a finite group acting linearly on $\R^n$. They were introduced by Satake \cite{Sata}, who called them `V-manifolds'. Later they were studied by Thurston \cite[Ch.~13]{Thur} who gave them the name `orbifold'.

As for Kuranishi spaces, defining orbifolds $\fX,\fY$ and smooth maps $\ff:\fX\ra\fY$ was initially problematic, and early definitions of ordinary categories of orbifolds \cite{Sata,Thur} had some bad differential-geometric behaviour  (e.g. for some definitions, one cannot define pullbacks $\ff^*(\mathfrak E)$ of orbifold vector bundles $\mathfrak E\ra\fY$). It is now generally agreed that it is best to define orbifolds to be a 2-category. See Lerman \cite{Lerm} for a good overview of ways to define orbifolds.

There are three main definitions of ordinary categories of orbifolds:
\begin{itemize}
\setlength{\itemsep}{0pt}
\setlength{\parsep}{0pt}
\item[(a)] Satake \cite{Sata} and Thurston \cite{Thur} defined an orbifold $\fX$ to be a Hausdorff topological space $X$ with an atlas $\bigl\{(V_i,\Ga_i,\psi_i):i\in I\bigr\}$ of orbifold charts $(V_i,\Ga_i,\psi_i)$, where $V_i$ is a manifold, $\Ga_i$ a finite group acting smoothly (and locally effectively) on $V_i$, and $\psi_i:V_i/\Ga_i\ra X$ a homeomorphism with an open set in $X$, and pairs of charts $(V_i,\Ga_i,\phi_i),(V_j,\Ga_j,\phi_j)$ satisfy compatibility conditions on their overlaps in $X$. Smooth maps $\ff:\fX\ra\fY$ between orbifolds are continuous maps $f:X\ra Y$ of the underlying spaces, which lift locally to smooth maps on the charts, giving a category $\Orb_{\rm ST}$.
\item[(b)] Chen and Ruan \cite[\S 4]{ChRu} defined orbifolds $\fX$ in a similar way to \cite{Sata,Thur}, but using germs of orbifold charts $(V_p,\Ga_p,\psi_p)$ for $p\in X$. Their morphisms $\ff:\fX\ra\fY$ are called {\it good maps}, giving a category $\Orb_{\rm CR}$. 
\item[(c)] Moerdijk and Pronk \cite{Moer,MoPr} defined a category of orbifolds $\Orb_{\rm MP}$ as {\it proper \'etale Lie groupoids\/} in $\Man$. Their definition of smooth map $\ff:\fX\ra\fY$, called {\it strong maps\/} \cite[\S 5]{MoPr} is complicated: it is an equivalence class of diagrams $\smash{\fX\,{\buildrel\phi\over \longleftarrow}\,\fX'\,{\buildrel\psi\over\longra}\,\fY}$, where $\fX'$ is a third orbifold, and $\phi,\psi$ are morphisms of groupoids with $\phi$ an equivalence (loosely, a diffeomorphism).
\end{itemize}
A book on orbifolds in the sense of \cite{ChRu,Moer,MoPr} is Adem, Leida and Ruan~\cite{ALR}.

There are four main definitions of 2-categories of orbifolds: 
\begin{itemize}
\setlength{\itemsep}{0pt}
\setlength{\parsep}{0pt}
\item[(i)] Pronk \cite{Pron} defines a strict 2-category $\bf LieGpd$ of Lie groupoids in $\Man$ as in (c), with the obvious 1-morphisms of groupoids, and localizes by a class of weak equivalences $\cW$ to get a weak 2-category $\Orb_{\rm Pr}={\bf LieGpd}[\cW^{-1}]$.
\item[(ii)] Lerman \cite[\S 3.3]{Lerm} defines a weak 2-category $\Orb_{\rm Le}$ of Lie groupoids in $\Man$ as in (c), with a non-obvious notion of 1-morphism called `Hilsum--Skandalis morphisms' involving `bibundles', and does not need to localize.

Henriques and Metzler \cite{HeMe} also use Hilsum--Skandalis morphisms.
\item[(iii)] Behrend and Xu \cite[\S 2]{BeXu}, Lerman \cite[\S 4]{Lerm} and Metzler \cite[\S 3.5]{Metz} define a strict 2-category of orbifolds $\Orb_{\rm ManSta}$ as a class of Deligne--Mumford stacks on the site $(\Man,{\cal J}_\Man)$ of manifolds with Grothendieck topology ${\cal J}_\Man$ coming from open covers.
\item[(iv)] The author \cite{Joyc4} defines a strict 2-category of orbifolds $\Orb_{C^\iy{\rm Sta}}$ as a class of Deligne--Mumford stacks on the site $(\CSch,{\cal J}_\CSch)$ of $C^\iy$-schemes.
\end{itemize}

As in Behrend and Xu \cite[\S 2.6]{BeXu}, Lerman \cite{Lerm}, Pronk \cite{Pron}, and the author \cite[Th.~9.26]{Joyc4}, approaches (i)--(iv) give equivalent weak 2-categories $\Orb_{\rm Pr},\ab\Orb_{\rm Le},\ab\Orb_{\rm ManSta},\ab\Orb_{C^\iy{\rm Sta}}$. As they are equivalent, the differences between them are not of mathematical importance, but more a matter of convenience or taste. Properties of localization also imply that $\Orb_{\rm MP}\simeq \Ho(\Orb_{\rm Pr})$. Thus, all of (c) and (i)--(iv) are equivalent at the level of homotopy categories.
\label{ku4rem8}
\end{rem}

We now give a fifth definition of a weak 2-category of orbifolds, essentially as a full 2-subcategory $\OrbKur\subset\Kur$, and we will show that $\OrbKur$ is equivalent to $\Orb_{\rm Pr},\ab\Orb_{\rm Le},\ab\Orb_{\rm ManSta},\ab\Orb_{C^\iy{\rm Sta}}$ in (i)--(iv) above. This provides a convenient way to relate orbifolds and Kuranishi spaces. Fukaya et al.\ \cite[\S 9]{FOOO6} and McDuff \cite{McDu} also define (effective) orbifolds as special examples of their notions of Kuranishi space/Kuranishi atlas.

The basic idea is that orbifolds $\fX$ in $\OrbKur$ are just Kuranishi spaces $\bX=(X,\cK)$ with $\cK=\bigl(I,(V_i,E_i,\Ga_i,s_i,\psi_i)_{i\in I}$, $\Phi_{ij}=(P_{ij},\pi_{ij},\phi_{ij},\hat\phi_{ij})_{i,j\in I}$, $\La_{ijk}=[\dot P_{ijk},\la_{ijk},\hat\la_{ijk}]_{i,j,k\in I}\bigr)$, for which the obstruction bundles $E_i\ra V_i$ are zero for all $i\in I$, so that the sections $s_i$ are also zero. This allows us to simplify the notation a lot. Equations in \S\ref{ku41} involving error terms $O\bigl(\pi_{ij}^*(s_i)\bigr)$ or $O\bigl(\pi_{ij}^*(s_i)^2\bigr)$ become exact, as $s_i=0$.

As $E_i,s_i$ are zero we can take `orbifold charts' to be $(V_i,\Ga_i,\psi_i)$. As $\hat\phi_{ij}=0$ we can take coordinate changes to be $\Phi_{ij}=(P_{ij},\pi_{ij},\phi_{ij})$, and we can also take $V_{ij}=\pi_{ij}(P_{ij})$ to be equal to $\bar\psi_i^{-1}(S)$, rather than just an open neighbourhood of $\bar\psi_i^{-1}(S)$ in $V_i$, since $\bar\psi_i^{-1}(S)$ is open in $V_i$ when $s_i=0$. For 2-morphisms $\La_{ij}=[\dot P_{ij},\la_{ij},\hat\la_{ij}]:\Phi_{ij}\Ra\Phi_{ij}'$ in \S\ref{ku41}, we have $\hat\la_{ij}=0$, and we are forced to take $\dot P_{ij}=P_{ij}$, and the equivalence relation $\approx$ in Definition \ref{ku4def3} becomes trivial, so we can take 2-morphisms to be just $\la_{ij}$. Here are the analogues of Definitions \ref{ku4def1}--\ref{ku4def3} and \ref{ku4def4}--\ref{ku4def8} with these simplifications made:

\begin{dfn} Let $X$ be a topological space. An {\it orbifold chart\/} on $X$ is a triple $(V,\Ga,\psi)$, where $V$ is a smooth manifold, $\Ga$ a finite group with a smooth action on $V$, and $\psi$ a homeomorphism from $V/\Ga$ to an open subset $\Im\psi$ in $X$. We write $\bar\psi:V\ra X$ for the composition of $\psi$ with the projection $V\ra V/\Ga$.

We call an orbifold chart $(V,\Ga,\psi)$ {\it effective\/} if the action of $\Ga$ on $V$ is locally effective, that is, no nonempty open set $U\subseteq V$ is fixed by $1\ne\ga\in\Ga$.
\label{ku4def22}
\end{dfn}

\begin{dfn} Let $(V_i,\Ga_i,\psi_i),(V_j,\Ga_j,\psi_j)$ be orbifold charts on a topological space $X$, and $S\subseteq\Im\psi_i\cap\Im\psi_j$ be open. A {\it coordinate change\/ $\Phi_{ij}:(V_i,\Ga_i,\psi_i)\ra (V_j,\Ga_j,\psi_j)$ over\/} $S$ is a triple $\Phi_{ij}=(P_{ij},\pi_{ij},\phi_{ij})$ satisfying:
\begin{itemize}
\setlength{\itemsep}{0pt}
\setlength{\parsep}{0pt}
\item[(a)] $P_{ij}$ is a manifold, with commuting smooth, free actions of $\Ga_i,\Ga_j$.
\item[(b)] $\pi_{ij}:P_{ij}\ra V_i$ is a smooth map which is $\Ga_i$-equivariant, $\Ga_j$-invariant, and \'etale, with $\pi_{ij}(P_{ij})=\bar\psi_i^{-1}(S)$. The fibres $\pi_{ij}^{-1}(v)$ of $\pi_{ij}$ for $v\in\bar\psi_i^{-1}(S)$ are $\Ga_j$-orbits, so that $\pi_{ij}:P_{ij}\ra\bar\psi_i^{-1}(S)$ is a principal $\Ga_j$-bundle.
\item[(c)] $\phi_{ij}:P_{ij}\ra V_j$ is a smooth map which is $\Ga_i$-invariant, $\Ga_j$-equivariant, and \'etale, with $\phi_{ij}(P_{ij})=\bar\psi_j^{-1}(S)$. The fibres $\phi_{ij}^{-1}(v)$ of $\phi_{ij}$ for $v\in \bar\psi_j^{-1}(S)$ are $\Ga_i$-orbits, so that $\phi_{ij}:P_{ij}\ra\bar\psi_j^{-1}(S)$ is a principal $\Ga_i$-bundle.
\item[(d)] $\bar\psi_i\ci\pi_{ij}=\bar\psi_j\ci\phi_{ij}:P_{ij}\ra X$.
\end{itemize}

\label{ku4def23}
\end{dfn}

\begin{dfn} Let $\Phi_{ij},\Phi_{ij}':(V_i,\Ga_i,\psi_i)\ra(V_j,\Ga_j,\psi_j)$ be coordinate changes of orbifold charts over $S\subseteq\Im\psi_i\cap\Im\psi_j\subseteq X$, where $\Phi_{ij}=(P_{ij},\pi_{ij},\phi_{ij})$ and $\Phi_{ij}'=(P_{ij}',\pi_{ij}',\phi_{ij}')$. A 2-{\it morphism\/} $\la_{ij}:\Phi_{ij}\Ra\Phi_{ij}'$ is a $\Ga_i$- and $\Ga_j$-equivariant diffeomorphism $\la_{ij}:P_{ij}\ra P_{ij}'$ with $\pi_{ij}'\ci\la_{ij}=\pi_{ij}$ and $\phi_{ij}'\ci\la_{ij}=\phi_{ij}$. That is, 2-morphisms are just isomorphisms of coordinate changes preserving all the structure, in the most obvious way.
\label{ku4def24}
\end{dfn}

\begin{dfn} Suppose $\Phi_{ij}=(P_{ij},\pi_{ij},\phi_{ij}):(V_i,\Ga_i,\psi_i)\ra (V_j,\Ga_j,\psi_j)$ and $\Phi_{jk}=(P_{jk},\pi_{jk},\phi_{jk}):(V_j,\Ga_j,\psi_j)\ra (V_k,\Ga_k,\psi_k)$ are coordinate changes of orbifold charts over $S\subseteq\Im\psi_i\cap\Im\psi_j\cap\Im\psi_k\subseteq X$. As in Definition \ref{ku4def4}, $\Ga_j$ acts freely on the transverse fibre product $P_{ij}\t_{V_j}P_{jk}$, so $P_{ik}=(P_{ij}\t_{V_j}P_{jk})/\Ga_j$ is a manifold, with projection $\Pi:P_{ij}\t_{V_j}P_{jk}\ra P_{ik}$. There are unique smooth maps $\pi_{ik}:P_{ik}\ra V_i$ and $\phi_{ik}:P_{ik}\ra V_k$ such that $\pi_{ij}\ci\pi_{P_{ij}}=\pi_{ik}\ci\Pi$ and~$\phi_{jk}\ci\pi_{P_{jk}}=\phi_{ik}\ci\Pi$. Then $\Phi_{ik}=(P_{ik},\pi_{ik},\phi_{ik})$ is a coordinate change $(V_i,\Ga_i,\psi_i)\ra (V_k,\Ga_k,\psi_k)$ over $S$. We write $\Phi_{jk}\ci\Phi_{ij}=\Phi_{ik}$, and call it the {\it composition of coordinate changes}.

If $S\subseteq \Im\psi_i\subseteq X$ is open, we define the {\it identity coordinate change over\/}~$S$
\begin{equation*}
\id_{(V_i,\Ga_i,\psi_i)}=(V_i\t\Ga_i,\pi_{ii},\phi_{ii}):(V_i,\Ga_i,\psi_i)\longra (V_i,\Ga_i,\psi_i),
\end{equation*}
where $\pi_{ii},\phi_{ii}:V_i\t\Ga_i\ra V_i$ map $\pi_{ii}:(v,\ga)\mapsto v$ and $\phi_{ii}:(v,\ga)\mapsto\ga\cdot v$. 

As in \eq{ku4eq4} and \eq{ku4eq6}, there are canonical 2-morphisms
\begin{gather*}
\bs\al_{\Phi_{kl},\Phi_{jk},\Phi_{ij}}:(\Phi_{kl}\ci\Phi_{jk})\ci\Phi_{ij}\Longra\Phi_{kl}\ci(\Phi_{jk}\ci\Phi_{ij}),\\
\bs\be_{\Phi_{ij}}:\Phi_{ij}\ci\id_{(V_i,\Ga_i,\psi_i)}\Longra\Phi_{ij},\quad
\bs\ga_{\Phi_{ij}}:\id_{(V_j,\Ga_j,\psi_j)}\ci\Phi_{ij}\Longra\Phi_{ij}.
\end{gather*}

Given coordinate changes $\Phi_{ij},\Phi_{ij}',\Phi_{ij}'':(V_i,\Ga_i,\psi_i)\ra(V_j,\Ga_j,\psi_j)$ and 2-morphisms $\la_{ij}:\Phi_{ij}\Ra\Phi_{ij}'$, $\la_{ij}':\Phi_{ij}'\Ra\Phi_{ij}''$ over $S$, the {\it vertical composition\/} $\la_{ij}'\od\la_{ij}:\Phi_{ij}\Ra\Phi_{ij}''$ is just the composition~$\la_{ij}'\od\la_{ij}=\la_{ij}'\ci\la_{ij}:P_{ij}\ra P_{ij}''$.

Given coordinate changes $\Phi_{ij},\Phi_{ij}':(V_i,\Ga_i,\psi_i)\ra(V_j,\Ga_j,\psi_j)$, $\Phi_{jk},\Phi_{jk}':(V_j,\Ga_j,\psi_j)\ra(V_k,\Ga_k,\psi_k)$ and 2-morphisms $\la_{ij}:\Phi_{ij}\Ra\Phi_{ij}'$, $\la_{jk}:\Phi_{jk}\Ra\Phi_{jk}'$ over $S$, write $\la_{jk}\t_{V_j}\la_{ij}:P_{ij}\t_{V_j}P_{jk}\ra P_{ij}'\t_{V_j}P_{jk}'$ for the induced diffeomorphism of fibre products. It is $\Ga_j$-equivariant, and so induces a unique diffeomorphism $\la_{jk}*\la_{ij}:P_{ik}=(P_{ij}\t_{V_j}P_{jk})/\Ga_j\ra (P_{ij}'\t_{V_j}P_{jk}')/\Ga_j=P_{ik}'$. Then $\la_{jk}*\la_{ij}:\Phi_{jk}\ci\Phi_{ij}\Ra \Phi_{jk}'\ci\Phi_{ij}'$ is a 2-morphism, {\it horizontal composition}.

The {\it identity\/ $2$-morphism\/} $\id_{\Phi_{ij}}:\Phi_{ij}\Ra\Phi_{ij}$ is $\id_{\Phi_{ij}}=\id_{P_{ij}}:P_{ij}\ra P_{ij}$.

As for $\Kur_S(X)$ in \S\ref{ku41}, if $S\subseteq X$ is open we have now defined a weak 2-category $\mathop{\rm Orb}_S(X)$, with objects orbifold charts $(V_i,\Ga_i,\psi_i)$ on $X$ with $S\subseteq\Im\psi_i$, 1-morphisms coordinate changes $\Phi_{ij},\Phi_{ij}':(V_i,\Ga_i,\psi_i)\ra(V_j,\Ga_j,\psi_j)$, and 2-morphisms $\la_{ij}:\Phi_{ij}\Ra\Phi_{ij}'$ as above.
\label{ku4def25}
\end{dfn}

We can now follow \S\ref{ku41}--\S\ref{ku43} from Definition \ref{ku4def9} until Theorem \ref{ku4thm3}, taking the $E_i,s_i,\hat\phi_{ijk},\hat\la_{ijk}$ to be zero throughout. This gives:

\begin{thm} We can define a weak\/ $2$-category $\OrbKur$ of \begin{bfseries}Kuranishi orbifolds\end{bfseries}, or just \begin{bfseries}orbifolds\end{bfseries}. Objects of\/ $\OrbKur$ are $\fX=(X,\O)$ for $X$ a Hausdorff, second countable topological space and\/ $\O=\bigl(I,(V_i,\Ga_i,\psi_i)_{i\in I},$ $\Phi_{ij,\;i,j\in I},$ $\la_{ijk,\; i,j,k\in I}\bigr)$ an \begin{bfseries}orbifold structure on $X$ of dimension\end{bfseries} $n\in\N,$ defined as in {\rm\S\ref{ku43}} but using orbifold charts, coordinate changes and $2$-morphisms as above.

There is a natural full and faithful\/ $2$-functor $F_\OrbKur^\Kur:\OrbKur\hookra\Kur$ embedding $\OrbKur$ as a full\/ $2$-subcategory of\/ $\Kur,$ which on objects maps $F_\OrbKur^\Kur:(X,\O)\mapsto(X,\cK),$ where for $\O$ as above, $\cK=\bigl(I,(V_i,0,\Ga_i,0,\psi_i)_{i\in I},$ $(P_{ij},\pi_{ij},\phi_{ij},0)_{ij,\;i,j\in I},$ $[P_{ijk},\la_{ijk},0]_{i,j,k\in I}\bigr)$ is the Kuranishi structure obtained by taking all the obstruction bundle data $E_i,s_i,\hat\phi_{ijk},\hat\la_{ijk}$ to be zero.

We say that a Kuranishi space $\bX$ \begin{bfseries}is an orbifold\end{bfseries} if\/ $\bX\simeq F_\OrbKur^\Kur(\fX)$ in $\Kur$ for some~$\fX\in\OrbKur$.
\label{ku4thm5}
\end{thm}

We relate this to previous definitions of (2-)categories of orbifolds.

\begin{thm} The $2$-category of Kuranishi orbifolds $\OrbKur$ in Theorem\/ {\rm\ref{ku4thm5}} is equivalent as a weak\/ $2$-category to the $2$-categories of orbifolds $\Orb_{\rm Pr},$ $\Orb_{\rm Le},$ $\Orb_{\rm ManSta},$ $\Orb_{C^\iy{\rm Sta}}$ in {\rm\cite{BeXu,Joyc4,Lerm,Metz,Pron}} described in Remark\/ {\rm\ref{ku4rem8}}. Also there is an equivalence of categories $\Ho(\OrbKur)\simeq\Orb_{\rm MP},$ for $\Orb_{\rm MP}$ the category of orbifolds from  Moerdijk and Pronk\/ {\rm\cite{Moer,MoPr}}.
\label{ku4thm6}
\end{thm}

\begin{proof} Use the notation of Remark \ref{ku4rem8}. We will define a full and faithful weak 2-functor $F_\OrbKur^{\Orb_{\rm Le}}:\OrbKur\ra\Orb_{\rm Le}$, which is an equivalence of 2-categories. Given an orbifold $\fX=(X,\O)$ in our sense with $\O=\bigl(I,(V_i,\Ga_i,\psi_i)_{i\in I}$, $\Phi_{ij}=(P_{ij},\pi_{ij},\phi_{ij})_{i,j\in I}$, $\la_{ijk,\; i,j,k\in I}\bigr)$, we define a natural proper \'etale Lie groupoid $[V\rra U]=(U,V,s,t,u,i,m)$ in $\Man$ (that is, a groupoid-orbifold in the sense of \cite{Moer,MoPr,Pron} and \cite[\S 3.3]{Lerm}, as in Remark \ref{ku4rem8}(c),(i),(ii)) with $U=\coprod_{i\in I}V_i$, and $V=\coprod_{i,j\in I}P_{ij}$, and $s,t:V\ra U$ given by $s=\coprod_{i,j\in I}\pi_{ij}$ and $t=\coprod_{i,j\in I}\phi_{ij}$, where the data $\la_{ijk,\; i,j,k\in I}$ gives the multiplication map $m:V\t_UV\ra V$. We define~$F_\OrbKur^{\Orb_{\rm Le}}(\fX)=[V\rra U]$. 

By working through the definitions, it turns out that Lerman's definitions of 1- and 2-morphisms in $\Orb_{\rm Le}$ in terms of `bibundles', when applied to groupoids $[V\rra U]$ of the form $F_\OrbKur^{\Orb_{\rm Le}}(\fX)$, reduce exactly to 1- and 2-morphisms in $\OrbKur$ as above. Thus, the definition of $F_\OrbKur^{\Orb_{\rm Le}}$ on 1- and 2-morphisms, and that $F_\OrbKur^{\Orb_{\rm Le}}$ is full and faithful, are immediate. The rest of the weak 2-functor data and conditions are straightforward. To show $F_\OrbKur^{\Orb_{\rm Le}}$ is an equivalence, we need to show that every groupoid-orbifold $[V\rra U]$ is equivalent in $\Orb_{\rm Le}$ to $F_\OrbKur^{\Orb_{\rm Le}}(\fX)$ for some $\fX$ in $\OrbKur$. This can be done as in Moerdijk and Pronk~\cite[Proof of Th.~4.1]{MoPr}.

The discussion in Remark \ref{ku4rem8} now shows that our $\OrbKur$ is equivalent as a weak 2-category to $\Orb_{\rm Pr},\ab\Orb_{\rm Le},\ab\Orb_{\rm ManSta},\ab\Orb_{C^\iy{\rm Sta}}$, and also that $\Ho(\OrbKur)\simeq \Orb_{\rm MP}$ as categories.
\end{proof}

\subsection{Isotropy groups, and tangent and obstruction spaces}
\label{ku46}

In \S\ref{ku25}, for a $\mu$-Kuranishi space $\bX$, we defined a tangent space $T_x\bX$ and an obstruction space $O_x\bX$ for each $x\in\bX$, which were unique up to canonical isomorphism and behaved functorially under morphisms of $\mu$-Kuranishi spaces. We now generalize these ideas to Kuranishi spaces~$\bX$.

\begin{dfn} Let $\bX=(X,\cK)$ be a Kuranishi space, with $\cK=\bigl(I,(V_i,\ab E_i,\ab\Ga_i,\ab s_i,\ab\psi_i)_{i\in I}$, $\Phi_{ij,\;i,j\in I}$, $\La_{ijk,\; i,j,k\in I}\bigr)$, and let $x\in\bX$.

Choose an arbitrary $i\in I$ with $x\in\Im\psi_i$, and choose $v_i\in s_i^{-1}(0)\subseteq V_i$ with $\bar\psi_i(v_i)=x$. Our definitions of $G_x\bX,T_x\bX,O_x\bX$ will depend on these choices, and we explain shortly to what extent. 

Define a finite group $G_x\bX$ called the {\it isotropy group of $\bX$ at\/} $x$ by
\e
G_x\bX=\bigl\{\ga\in\Ga_i:\ga\cdot v_i=v_i\bigr\}=\Stab_{\Ga_i}(v_i),
\label{ku4eq41}
\e
as a subgroup of $\Ga_i$. Define finite-dimensional real vector spaces $T_x\bX$, the {\it tangent space of\/ $\bX$ at\/} $x$, and $O_x\bX$, the {\it obstruction space of\/ $\bX$ at\/} $x$, to be the kernel and cokernel of $\d s_i\vert_{v_i}$, so that as in \eq{ku2eq30} they fit into an exact sequence
\e
\xymatrix@C=20pt{ 0 \ar[r] & T_x\bX \ar[r] & T_{v_i}V_i \ar[rr]^{\d s_i\vert_{v_i}} && E_i\vert_{v_i} \ar[r] & O_x\bX \ar[r] & 0. }
\label{ku4eq42}
\e
The actions of $\Ga_i$ on $V_i,E_i$ induce linear actions of $G_x\bX$ on $T_x\bX,O_x\bX$, by the commutative diagram for each $\ga\in G_x\bX$:
\begin{equation*}
\xymatrix@C=23pt@R=15pt{ 0 \ar[r] & T_x\bX \ar[r] \ar@{.>}[d]^{\ga\cdot} & T_{v_i}V_i \ar[rr]_{\d s_i\vert_{v_i}} \ar[d]^{\d(\ga\cdot)} && E_i\vert_{v_i} \ar[d]^{\ga\cdot} \ar[r] & O_x\bX \ar@{.>}[d]^{\ga\cdot} \ar[r] & 0 \\
0 \ar[r] & T_x\bX \ar[r] & T_{v_i}V_i \ar[rr]^{\d s_i\vert_{v_i}} && E_i\vert_{v_i} \ar[r] & O_x\bX \ar[r] & 0.\!\! }
\end{equation*}
This makes $T_x\bX,O_x\bX$ into representations of $G_x\bX$. Definition \ref{ku4def12}(b) yields
\e
\dim T_x\bX-\dim O_x\bX=\vdim\bX.
\label{ku4eq43}
\e

The dual vector spaces of $T_x\bX,O_x\bX$ will be called the {\it cotangent space}, written $T_x^*\bX$, and the {\it coobstruction space}, written~$O_x^*\bX$.

We now explain how the triple $(G_x\bX,T_x\bX,O_x\bX)$ depends on the choice of $i,v_i$. Let $j,v_j$ be alternative choices, giving another triple $(G_x'\bX,T_x'\bX,O_x'\bX)$. Then we have a coordinate change $\Phi_{ij}=(P_{ij},\pi_{ij},\phi_{ij},\hat\phi_{ij})$. Consider the set
\e
S_x=\bigl\{p\in P_{ij}:\pi_{ij}(p)=v_i,\;\> \phi_{ij}=v_j\bigr\}.
\label{ku4eq44}
\e
Using Theorem \ref{ku4thm2}, we find that $S_x$ is invariant under the commuting actions of $G_x\bX=\Stab_{\Ga_i}(v_i)\subseteq\Ga_i$ and $G_x'\bX=\Stab_{\Ga_j}(v_j)\subseteq\Ga_j$ on $P_{ij}$ induced by the $\Ga_i,\Ga_j$-actions on $P_{ij}$, and $G_x\bX,G_x'\bX$ each act freely and transitively on~$S_x$.

Pick $p\in S_x$. Define isomorphisms $I^G_x:G_x\bX\ra G_x'\bX$, $I^T_x:T_x\bX\ra T_x'\bX$, $I^O_x:O_x\bX\ra O_x'\bX$ by $I^G_x(\ga)=\ga'$ if $\ga\cdot p=(\ga')^{-1}\cdot p$ in $S_x$, using the free transitive actions of $G_x\bX,G_x'\bX$ on $S_x$, and the commutative diagram
\begin{equation*}
\xymatrix@C=23pt@R=18pt{ 0 \ar[r] & T_x\bX \ar[r] \ar@{.>}[d]^{I^T_x} & T_{v_i}V_i \ar[rr]_{\d s_i\vert_{v_i}} \ar[d]^(0.49){\d\phi_{ij}\vert_p\ci(\d\pi_{ij}\vert_p)^{-1}} && E_i\vert_{v_i} \ar[d]^{\hat\phi_{ij}\vert_p} \ar[r] & O_x\bX \ar@{.>}[d]^{I^O_x} \ar[r] & 0 \\
0 \ar[r] & T_x'\bX \ar[r] & T_{v_j}V_j \ar[rr]^{\d s_j\vert_{v_j}} && E_j\vert_{v_j} \ar[r] & O_x'\bX \ar[r] & 0.\!\! }
\end{equation*}
Then $I^T_x,I^O_x$ are $I_G^x$-equivariant isomorphisms.

Suppose we instead picked $\ti p\in S_x$, yielding isomorphisms $\ti I{}^G_x,\ti I{}^T_x,\ti I{}^O_x$. Since $G_x'\bX$ acts freely transitively on $S_x$, there is a unique $\de\in G_x'\bX$ with $\de\cdot p=\ti p$. Then we can show that $\ti I^G_x(\ga)=\de I^G_x(\ga)\de^{-1}$, $\ti I^T_x(v)=\de\cdot I^T_x(v)$ and $\ti I^O_x(w)=\de\cdot I^O_x(w)$ for all $\ga\in G_x\bX$, $v\in T_x\bX$, and~$w\in O_x\bX$. 

If $k,v_k$ is a third choice for $i,v_i$, yielding a triple $(G_x''\bX,T_x''\bX,O_x''\bX)$, then as above by picking points $p\in S_x$ we can define isomorphisms of triples
\begin{align*}
(I^G_x,I^T_x,I^O_x)&:(G_x\bX,T_x\bX,O_x\bX)\longra(G_x'\bX,T_x'\bX,O_x'\bX),\\
(\dot I{}^G_x,\dot I{}^T_x,\dot I{}^O_x)&:(G_x'\bX,T_x'\bX,O_x'\bX)\longra(G_x''\bX,T_x''\bX,O_x''\bX),\\
(\ddot I{}^G_x,\ddot I{}^T_x,\ddot I{}^O_x)&:(G_x\bX,T_x\bX,O_x\bX)\longra(G_x''\bX,T_x''\bX,O_x''\bX).
\end{align*}
We can show that $(\dot I{}^G_x,\dot I{}^T_x,\dot I{}^O_x)\ci (I^G_x,I^T_x,I^O_x)$ and $(\ddot I{}^G_x,\ddot I{}^T_x,\ddot I{}^O_x)$ differ by the action of some canonical $\de\in G_x''\bX$, as for $(I^G_x,I^T_x,I^O_x),(\ti I{}^G_x,\ti I{}^T_x,\ti I{}^O_x)$ above.

To summarize: the triple $(G_x\bX,T_x\bX,O_x\bX)$ is independent of the choice of $i,v_i$ up to isomorphism, but not up to canonical isomorphism. There are isomorphisms $(I^G_x,I^T_x,I^O_x):(G_x\bX,T_x\bX,O_x\bX)\ra(G_x'\bX,T_x'\bX,O_x'\bX)$ between any two choices for $(G_x\bX,T_x\bX,O_x\bX)$, which are canonical up to conjugation by an element of $G_x'\bX$, and behave as expected under composition.
\label{ku4def26}
\end{dfn}

We discuss functoriality of the $G_x\bX,T_x\bX,O_x\bX$ under 1- and 2-morphisms.

\begin{dfn} Let $\bs f:\bX\ra\bY$ be a 1-morphism of Kuranishi spaces, with notation \eq{ku4eq12}--\eq{ku4eq13}, and let $x\in\bX$ with $\bs f(x)=y\in\bY$. Then Definition \ref{ku4def26} gives $G_x\bX,T_x\bX,O_x\bX$, defined using $i\in I$ and $u_i\in U_i$ with $\bar\chi_i(u_i)=x$, and $G_y\bY,T_y\bY,O_y\bY$, defined using $j\in J$ and $v_j\in V_j$ with $\bar\chi_j(v_j)=y$. In $\bs f$ we have a 1-morphism $\bs f_{ij}=(P_{ij},\pi_{ij},f_{ij},\hat f_{ij})$ over $f$. As in \eq{ku4eq44}, define
\e
S_{x,\bs f}=\bigl\{p\in P_{ij}:\pi_{ij}(p)=u_i,\;\> f_{ij}(p)=v_j\bigr\}.
\label{ku4eq45}
\e
As for $S_x$ in Definition \ref{ku4def26}, we find that $S_{x,\bs f}$ is invariant under the commuting actions of $G_x\bX=\Stab_{\Be_i}(u_i)\subseteq\Be_i$ and $G_y\bY=\Stab_{\Ga_j}(v_j)\subseteq\Ga_j$ on $P_{ij}$ induced by the $\Be_i,\Ga_j$-actions on $P_{ij}$. But this time, $G_y\bY$ acts freely and transitively on $S_{x,\bs f}$, but the action of $G_x\bX$ need not be free or transitive. Choose an arbitrary point $p_0\in S_{x,\bs f}$. Our definitions of $G_x\bs f,T_x\bs f,O_x\bs f$ will depend on this choice. 

Define morphisms $G_x\bs f:G_x\bX\ra G_y\bY$, $T_x\bs f:T_x\bX\ra T_y\bY$, $O_x\bs f:O_x\bX\ra O_y\bY$ by $G_x\bs f(\ga)=\ga'$ if $\ga\cdot p_0=(\ga')^{-1}\cdot p_0$ in $S_{x,\bs f}$, using the actions of $G_x\bX,G_y\bY$ on $S_{x,\bs f}$ with $G_y\bY$ free and transitive, and the commutative diagram
\e
\begin{gathered}
\xymatrix@C=23pt@R=19pt{ 0 \ar[r] & T_x\bX \ar[r] \ar@{.>}[d]^{T_x\bs f} & T_{u_i}U_i \ar[rr]_{\d r_i\vert_{u_i}} \ar[d]^(0.49){\d f_{ij}\vert_{p_0}\ci(\d\pi_{ij}\vert_{p_0})^{-1}} && D_i\vert_{u_i} \ar[d]^{\hat f_{ij}\vert_{p_0}} \ar[r] & O_x\bX \ar@{.>}[d]^{O_x\bs f} \ar[r] & 0 \\
0 \ar[r] & T_y\bY \ar[r] & T_{v_j}V_j \ar[rr]^{\d s_j\vert_{v_j}} && E_j\vert_{v_j} \ar[r] & O_y\bY \ar[r] & 0.\!\! }
\end{gathered}
\label{ku4eq46}
\e
Then $T_x\bs f,O_x\bs f$ are $G_x\bs f$-equivariant linear maps.

If $p_0'\in S_{x,\bs f}$ is an alternative choice for $p_0$, yielding $G'_x\bs f,T'_x\bs f,O'_x\bs f$, there is a unique $\de\in G_y\bY$ with $\de\cdot p_0=p_0'$, and then $G'_x\bs f(\ga)=\de (G_x\bs f(\ga))\de^{-1}$, $T'_x\bs f(v)=\de\cdot T_x\bs f(v)$, $O'_x\bs f(w)=\de\cdot O_x\bs f(w)$ for all $\ga\in G_x\bX$, $v\in T_x\bX$, and $w\in O_x\bX$. That is, the triple $(G_x\bs f,T_x\bs f,O_x\bs f)$ is canonical up to conjugation by an element of~$G_y\bY$.

Continuing with the same notation, suppose $\bs g:\bX\ra\bY$ is another 1-morphism and $\bs\eta:\bs f\Ra\bs g$ a 2-morphism in $\Kur$. Then above we define $G_x\bs g,T_x\bs g,O_x\bs g$ by choosing an arbitrary point $q_0\in S_{x,\bs g}$, where
\begin{equation*}
S_{x,\bs g}=\bigl\{q\in Q_{ij}:\pi_{ij}(q)=u_i,\;\> g_{ij}(q)=v_j\bigr\},
\end{equation*}
with $\bs g_{ij}=(Q_{ij},\pi_{ij},g_{ij},\hat g_{ij})$ in $\bs g$. In $\bs\eta$ we have $\bs\eta_{ij}=[\dot P_{ij},\la_{ij},\hat\la_{ij}]$ represented by $(\dot P_{ij},\la_{ij},\hat\la_{ij})$, where $\dot P_{ij}\subseteq P_{ij}$ and $\la_{ij}:\dot P_{ij}\ra Q_{ij}$. From the definitions we find that $S_{x,\bs f}\subseteq\dot P_{ij}$, and $\la_{ij}\vert_{S_{x,\bs f}}:S_{x,\bs f}\ra S_{x,\bs g}$ is a bijection. Since $G_y\bY$ acts freely and transitively on $S_{x,\bs g}$, there is a unique element $G_x\bs\eta\in G_y\bY$ with $G_x\bs\eta\cdot \la_{ij}(p_0)=q_0$. One can now check that
\begin{align*}
G_x\bs g(\ga)&=(G_x\bs\eta)(G_x\bs f(\ga))(G_x\bs\eta)^{-1},\quad T_x\bs g(v)=G_x\bs\eta\cdot T_x\bs f(v),\quad\text{and} \\ 
O_x\bs g(w)&=G_x\bs\eta\cdot O_x\bs f(w)\quad\text{for all $\ga\in G_x\bX$, $v\in T_x\bX$, and $w\in O_x\bX$.} 
\end{align*}
That is, $(G_x\bs g,T_x\bs g,O_x\bs g)$ is conjugate to $(G_x\bs f,T_x\bs f,O_x\bs f)$ under $G_x\bs\eta\in G_y\bY$, the same indeterminacy as in the definition of~$(G_x\bs f,T_x\bs f,O_x\bs f)$.

Suppose instead that $\bs g:\bY\ra\bZ$ is another 1-morphism of Kuranishi spaces and $\bs g(y)=z\in\bZ$. Then in a similar way we can show there is a canonical element $G_{x,\bs g,\bs f}\in G_z\bZ$ such that for all $\ga\in G_x\bX$, $v\in T_x\bX$, $w\in O_x\bX$ we have
\begin{align*}
G_x(\bs g\ci\bs f)(\ga)&=(G_{x,\bs g,\bs f})((G_y\bs g\ci G_x\bs f)(\ga))(G_{x,\bs g,\bs f})^{-1},\\ 
T_x(\bs g\ci\bs f)(v)&=G_{x,\bs g,\bs f}\cdot (T_y\bs g\ci T_x\bs f)(v), \\ 
O_x(\bs g\ci\bs f)(w)&=G_{x,\bs g,\bs f}\cdot (O_y\bs g\ci O_x\bs f)(w). 
\end{align*}
That is, $(G_x(\bs g\ci\bs f),T_x(\bs g\ci\bs f),O_x(\bs g\ci\bs f))$ is conjugate to $(G_y\bs g,T_y\bs g,O_y\bs g)\ci (G_x\bs f,T_x\bs f,O_x\bs f)$ under~$G_{x,\bs g,\bs f}\in G_z\bZ$.

Since 2-morphisms $\bs\eta:\bs f\Ra\bs g$ relate triples $(G_x\bs f,T_x\bs f,O_x\bs f)$ and $(G_x\bs g,\ab T_x\bs g,\ab O_x\bs g)$ by isomorphisms, if $\bs f:\bX\ra\bY$ is an equivalence in $\Kur$ then $G_x\bs f,T_x\bs f,O_x\bs f$ are isomorphisms for all~$x\in\bX$.
\label{ku4def27}
\end{dfn}

\begin{rem}{\bf(a)} The definitions of $G_x\bX,\ab T_x\bX,\ab O_x\bX,G_x\bs f,T_x\bs f,O_x\bs f$ above depend on arbitrary choices. We could use the Axiom of (Global) Choice as in Remark \ref{ku4rem5} to choose particular values for $G_x\bX,\ldots,O_x\bs f$ for all $\bX,x,\bs f$. But this is not really necessary, we can just bear the non-uniqueness in mind when working with them. All the definitions we make using $G_x\bX,\ldots,O_x\bs f$ will be independent of the arbitrary choices in Definitions \ref{ku4def26} and~\ref{ku4def27}.
\smallskip

\noindent{\bf(b)} If we do choose particular values for $G_x\bX,\ldots,O_x\bs f$, then we can express the naturality and functoriality of the $G_x\bX,\ldots,O_x\bs f$ succinctly as follows: we have defined a weak 2-functor $\Kur_{\bs *}\ra \mathop{\bf GVect}^2$ as in \S\ref{kuB2}, where $\Kur_{\bs *}$ is the weak 2-category of `pointed Kuranishi spaces' with objects $(\bX,x)$ for $\bX\in\Kur$ and $x\in\bX$, and $\mathop{\bf GVect}^2$ is the strict 2-category with objects triples $(G,T,O)$ for $G$ a finite group and $T,O$ finite-dimensional real $G$-representations, and 1-morphisms $(\rho,L,M):(G,T,O)\ra (G',T',O')$ where $\rho:G\ra G'$ is a group morphism and $L:T\ra T'$, $M:O\ra O'$ are $\rho$-equivariant linear maps, and 2-morphisms $\de:(\rho,L,M)\Ra(\ti\rho,\ti L,\ti M)$ of 1-morphisms $(\rho,L,M),(\ti\rho,\ti L,\ti M):(G,T,O)\ra (G',T',O')$ being elements $\de\in G'$ with $\ti\rho(\ga)=\de\rho(\ga)\de^{-1}$, $\ti L(v)=\de\cdot L(v)$, $\ti M(w)=\de\cdot M(w)$ for all $\ga\in G$, $v\in T$ and~$w\in O$.
\label{ku4rem9}
\end{rem}

As in \S\ref{ku25}, we quote from the sequel~\cite{Joyc12}.

\begin{dfn}[\!\!\cite{Joyc12}] Let $\bs f:\bX\ra\bY$ be a 1-morphism of Kuranishi spaces. 
\smallskip

\noindent{\bf(a)} Call $\bs f$ {\it representable\/} if $G_x\bs f:G_x\bX\ra G_{\bs f(x)}\bY$ is injective for all~$x\in\bX$.
\smallskip

\noindent{\bf(b)} Call $\bs f$ a {\it w-submersion\/} if $O_x\bs f:O_x\bX\ra O_{\bs f(x)}\bY$ is surjective for all~$x\in\bX$.

\smallskip

\noindent{\bf(c)} Call $\bs f$ a {\it submersion\/} if $T_x\bs f:T_x\bX\ra T_{\bs f(x)}\bY$ is surjective and $O_x\bs f:O_x\bX\ra O_{\bs f(x)}\bY$ is an isomorphism for all $x\in\bX$.
\smallskip

\noindent{\bf(d)} Call $\bs f$ a {\it w-immersion\/} if it is representable and $T_x\bs f:T_x\bX\ra T_{\bs f(x)}\bY$ is injective for all $x\in\bX$. Call $\bs f$ an {\it immersion\/} if it is a w-immersion and $O_x\bs f:O_x\bX\ra O_{\bs f(x)}\bY$ is surjective for all $x\in\bX$.
\smallskip

\noindent{\bf(e)} Call $\bs f$ a {\it w-embedding\/} (or an {\it embedding\/}) if it is a w-immersion (or an immersion, respectively), and $G_x\bs f:G_x\bX\ra G_{\bs f(x)}\bY$ is an isomorphism for all $x\in\bX$, and $f:X\ra f(X)$ is a homeomorphism.
\label{ku4def28}
\end{dfn}

\begin{dfn}[\!\!\cite{Joyc12}] Let $\bs g:\bX\ra\bZ$ and $\bs h:\bY\ra\bZ$ be 1-morphisms of Kuranishi spaces. Call $\bs g,\bs h$ {\it d-transverse\/} if for all $x\in\bX$, $y\in\bY$ with $\bs g(x)=\bs h(y)=z$ in $\bZ$, and all $\ga\in G_z\bZ$, then $O_x\bs g\op (\ga\cdot O_y\bs h):O_x\bX\op O_y\bY\ra O_z\bZ$ is surjective.

If $\bs g$ or $\bs h$ is a w-submersion, then $\bs g,\bs h$ are automatically d-transverse.
\label{ku4def29}
\end{dfn}

Here are some sample results:

\begin{ex}[\!\!\cite{Joyc12}] {\bf(i)} A Kuranishi space $\bX$ is an orbifold, in the sense of Theorem \ref{ku4thm5}, if and only if $O_x\bX=0$ for all $x\in\bX$. Also $\bX$ is a manifold, as in Definition \ref{ku4def18}, if and only if $G_x\bX=\{1\}$ and $O_x\bX=0$ for all~$x\in\bX$.
\smallskip

\noindent{\bf(ii)} A representable 1-morphism $\bs f:\bX\ra\bY$ in $\Kur$ is \'etale if and only if $T_x\bs f:T_x\bX\ra T_{\bs f(x)}\bY$ and $O_x\bs f:O_x\bX\ra O_{\bs f(x)}\bY$ are isomorphisms for all $x\in\bX$. Also $\bs f$ is an equivalence in $\Kur$ if and only if $\bs f$ is \'etale, and $G_x\bs f:G_x\bX\ra G_{\bs f(x)}\bY$ is an isomorphism for all $x\in\bX$, and $f:X\ra Y$ is a bijection.

\smallskip

\noindent{\bf(iii)} Let $\bs g:\bX\ra\bZ$ and $\bs h:\bY\ra\bZ$ be d-transverse 1-morphisms of Kuranishi spaces. Then the fibre product $\bW=\bX\t_{\bs g,\bZ,\bs h}\bY$ exists in the weak 2-category $\Kur$, in the sense of \S\ref{kuB3}, with $\vdim\bW=\vdim\bX+\vdim\bY-\vdim\bZ$. As sets we have
\begin{align*}
W\cong \bigl\{(x,y,C):\,&x\in \bX,\;\> y\in \bY,\;\> \bs g(x)=\bs h(y)=z\in Z,\\
&C\in G_x\bs g(G_x\bX) \big\backslash G_z\bZ \big/ G_y\bs h(G_y\bY)\bigr\}.
\end{align*}

If $\bs g$ is a submersion and $\bY$ is an orbifold, then $\bW$ is an orbifold.
\label{ku4ex2}
\end{ex}

\subsection{\texorpdfstring{m-Kuranishi spaces and $\mu$-Kuranishi spaces}{m-Kuranishi spaces and \textmu-Kuranishi spaces}}
\label{ku47}

In \S\ref{ku45} we defined a weak 2-category of orbifolds $\OrbKur$, essentially as the full 2-subcategory $\OrbKur\subset\Kur$ of Kuranishi spaces $\bX=(X,\cK)$ for which all the obstruction bundles $E_i$ in the Kuranishi neighbourhoods $(V_i,E_i,\Ga_i,s_i,\psi_i)$ in $\cK$ are zero. This allowed us to simplify our notation: we took orbifold charts to be $(V_i,\Ga_i,\psi_i)$ rather than $(V_i,E_i,\Ga_i,s_i,\psi_i)$, and coordinate changes to be $(P_{ij},\pi_{ij},\phi_{ij})$ rather than $(P_{ij},\pi_{ij},\phi_{ij},\hat\phi_{ij})$, and 2-morphisms to be $\la_{ij}$ rather than $[\dot P_{ij},\la_{ij},\hat\la_{ij}]$, as in the orbifold case $E_i,s_i,\hat\phi_{ij},\hat\la_{ij}$ are zero and $\dot P_{ij}=P_{ij}$.

In a similar way, we now define a weak 2-category of {\it m-Kuranishi spaces\/} $\mKur$, a kind of derived manifold, essentially as the full 2-subcategory $\mKur\subset\Kur$ of Kuranishi spaces $\bX$ with groups $\Ga_i=\{1\}$ for all $i\in I$. This allows us to simplify our notation. We take `m-Kuranishi neighbourhoods' to be $(V_i,E_i,s_i,\psi_i)$ rather than $(V_i,E_i,\Ga_i,s_i,\psi_i)$. In coordinate changes, $\pi_{ij}:P_{ij}\ra V_{ij}$ in \S\ref{ku41} is a diffeomorphism, since it is a principal $\Ga_j$-bundle for $\Ga_j=\{1\}$, so we can replace $P_{ij}$ by $V_{ij}$ and take `m-coordinate changes' to be $(V_{ij},\phi_{ij},\hat\phi_{ij})$, and 2-morphisms to be $[\dot V_{ij},\hat\la_{ij}]$ for open $\dot V_{ij}\subseteq V_{ij}\subseteq V_i$. The analogue $\mKur_S(X)$ of $\Kur_S(X)$ in \S\ref{ku41} is then a {\it strict\/} 2-category, as the canonical identification $\la_{ijkl}$ in \eq{ku4eq3} is replaced by the identity map.

Here are the simplified analogues of Definitions \ref{ku4def1}--\ref{ku4def3} and \ref{ku4def4}--\ref{ku4def8}. Note the strong similarity to $\mu$-Kuranishi neighbourhoods and morphisms in~\S\ref{ku21}.

\begin{dfn} Let $X$ be a topological space. An {\it m-Kuranishi neighbourhood\/ $(V,E,s,\psi)$ on\/} $X$ is a $\mu$-Kuranishi neighbourhood on $X$ from Definition~\ref{ku2def2}.

Let $(V_i,E_i,s_i,\psi_i)$, $(V_j,E_j,s_j,\psi_j)$ be m-Kuranishi neighbourhoods on a topological space $X$, and $S\subseteq\Im\psi_i\cap\Im\psi_j$ an open set. A 1-{\it morphism $\Phi_{ij}:(V_i,E_i,s_i,\psi_i)\ra (V_j,E_j,s_j,\psi_j)$ of m-Kuranishi neighbourhoods over\/} $S$ is a triple $\Phi_{ij}=(V_{ij},\phi_{ij},\hat\phi_{ij})$ satisfying:
\begin{itemize}
\setlength{\itemsep}{0pt}
\setlength{\parsep}{0pt}
\item[(a)] $V_{ij}$ is an open neighbourhood of $\psi_i^{-1}(S)$ in $V_i$. We do not require that $V_{ij}\cap s_i^{-1}(0)=\psi_i^{-1}(S)$, only that~$\psi_i^{-1}(S)\subseteq V_{ij}\cap s_i^{-1}(0)\subseteq V_{ij}$.
\item[(b)] $\phi_{ij}:V_{ij}\ra V_j$ is a smooth map.
\item[(c)] $\hat\phi_{ij}:E_i\vert_{V_{ij}}\ra\phi_{ij}^*(E_j)$ is a morphism of vector bundles on $V_{ij}$.
\item[(d)] $\hat\phi_{ij}(s_i\vert_{V_{ij}})=\phi_{ij}^*(s_j)+O(s_i^2)$.
\item[(e)] $\psi_i=\psi_j\ci\phi_{ij}$ on $s_i^{-1}(0)\cap V_{ij}$.
\end{itemize}

The {\it identity $1$-morphism\/} is~$\id_{(V_i,E_i,s_i,\psi_i)}=(V_i,\id_{V_i},\id_{E_i})$. 
\label{ku4def30}
\end{dfn}

\begin{dfn} Let $\Phi_{ij},\Phi_{ij}'\!:\!(V_i,E_i,s_i,\psi_i)\!\ra\!(V_j,E_j,s_j,\psi_j)$ be 1-morphisms of m-Kuranishi neighbourhoods over $S\subseteq\Im\psi_i\cap\Im\psi_j\subseteq X$, where $\Phi_{ij}=(V_{ij},\phi_{ij},\hat\phi_{ij})$ and~$\Phi_{ij}'=(V_{ij}',\phi_{ij}',\hat\phi_{ij}')$. Consider pairs $(\dot V_{ij},\hat\la_{ij})$ satisfying:
\begin{itemize}
\setlength{\itemsep}{0pt}
\setlength{\parsep}{0pt}
\item[(a)] $\dot V_{ij}$ is an open neighbourhood of $\psi_i^{-1}(S)$ in $V_{ij}\cap V_{ij}'$.
\item[(b)] $\hat\la_{ij}:E_i\vert_{\dot V_{ij}}\ra \phi_{ij}^*(TV_j)\vert_{\dot V_{ij}}$ is a vector bundle morphism on $\dot V_{ij}$ with
\end{itemize}
\e
\phi_{ij}'=\phi_{ij}+\hat\la_{ij}\cdot s_i+O(s_i^2)\;\>\text{and}\;\> \hat\phi_{ij}'=\hat\phi_{ij}+\hat\la_{ij}\cdot \phi_{ij}^*(\d s_j)+O(s_i)\;\> \text{on $\dot V_{ij}$.}
\label{ku4eq47}
\e

Define a binary relation $\approx$ on such pairs by $(\dot V_{ij},\hat\la_{ij})\approx(\dot V_{ij}',\hat\la_{ij}')$ if there exists an open neighbourhood $\ddot V_{ij}$ of $\psi_i^{-1}(S)$ in $\dot V_{ij}\cap \dot V_{ij}'$ with
\e
\hat\la_{ij}\vert_{\ddot V_{ij}}=\hat\la_{ij}'\vert_{\ddot V_{ij}}+O(s_i)\quad\text{on $\ddot V_{ij}$.}
\label{ku4eq48}
\e
Clearly $\approx$ is an equivalence relation. Write $[\dot V_{ij},\hat\la_{ij}]$ for the $\approx$-equivalence class of $(\dot V_{ij},\hat\la_{ij})$. We say that $[\dot V_{ij},\hat\la_{ij}]:\Phi_{ij}\Ra\Phi_{ij}'$ is a 2-{\it morphism of\/ $1$-morphisms of m-Kuranishi neighbourhoods on\/ $X$ over\/} $S$, or just a 2-{\it morphism over\/} $S$. We often write~$\La_{ij}=[\dot V_{ij},\hat\la_{ij}]$.

The {\it identity\/ $2$-morphism\/} of $\Phi_{ij}$ is~$\id_{\Phi_{ij}}=[V_{ij},0]:\Phi_{ij}\Ra\Phi_{ij}$. 
\label{ku4def31}
\end{dfn}

\begin{dfn} Suppose $\Phi_{ij}\!=\!(V_{ij},\phi_{ij},\hat\phi_{ij}):(V_i,E_i,s_i,\psi_i)\!\ra\! (V_j,E_j,s_j,\psi_j)$ and $\Phi_{jk}=(V_{jk},\phi_{jk},\hat\phi_{jk}):(V_j,E_j,s_j,\psi_j)\ra (V_k,E_k,\ab s_k,\ab\psi_k)$ are 1-morphisms of m-Kuranishi neighbourhoods over~$S\subseteq X$. Define the {\it composition\/} to be $\Phi_{jk}\ci\Phi_{ij}=(V_{ik},\phi_{ik},\hat\phi_{ik})$, where $V_{ik}=\phi_{ij}^{-1}(V_{jk})\subseteq V_{ij}\subseteq V_i$, and $\phi_{ik}:V_{ij}\ra V_k$ is $\phi_{ik}=\phi_{jk}\ci\phi_{ij}\vert_{V_{ik}}$, and $\hat\phi_{ik}:E_i\vert_{V_{ik}}\ra\phi_{ik}^*(E_k)$ is~$\hat\phi_{ik}=\phi_{ij}\vert_{V_{ik}}^*(\hat\phi_{jk})\ci\hat\phi_{ij}\vert_{V_{ik}}$. 

It is easy to check that $\Phi_{jk}\ci\Phi_{ij}:(V_i,E_i,s_i,\psi_i)\ra (V_k,E_k,s_k,\psi_k)$ is a 1-morphism of m-Kuranishi neighbourhoods over $S$. Also, composition of 1-morphisms is clearly {\it strictly associative}. This is a contrast with 1-morphisms of Kuranishi neighbourhoods in \S\ref{ku41}, where we had a 2-morphism $\bs\al_{\Phi_{kl},\Phi_{jk},\Phi_{ij}}:(\Phi_{kl}\ci\Phi_{jk})\ci\Phi_{ij}\Ra\Phi_{kl}\ci(\Phi_{jk}\ci\Phi_{ij})$ in \eq{ku4eq4}, generally not the identity.

Clearly $\Phi_{ij}\ci\id_{(V_i,E_i,s_i,\psi_i)}=\id_{(V_j,E_j,s_j,\psi_j)}\ci\Phi_{ij}=\Phi_{ij}$ for 1-morphisms $\Phi_{ij}:(V_i,E_i,s_i,\psi_i)\ra (V_j,E_j,s_j,\psi_j)$. This is a contrast with 1-morphisms of Kuranishi neighbourhoods in \S\ref{ku41}, where we had 2-morphisms $\bs\be_{\Phi_{ij}},\bs\ga_{\Phi_{ij}}$ in \eq{ku4eq6}, generally not the identities.
\label{ku4def32}
\end{dfn}

\begin{dfn} Let $\Phi_{ij},\Phi_{ij}',\Phi_{ij}'':(V_i,E_i,s_i,\psi_i)\ra (V_j,E_j,s_j,\psi_j)$ be 1-mor\-phisms of m-Kuranishi neighbourhoods over $S\subseteq X$ with $\Phi_{ij}=(V_{ij},\phi_{ij},\hat\phi_{ij})$, $\Phi_{ij}'=(V_{ij}',\phi_{ij}',\hat\phi_{ij}')$, $\Phi_{ij}''=(V_{ij}'',\phi_{ij}'',\hat\phi_{ij}'')$, and $\La_{ij}=[\dot V_{ij},\hat\la_{ij}]:\Phi_{ij}\Ra\Phi_{ij}'$ and $\La_{ij}'=[\dot V_{ij}',\hat\la_{ij}']:\Phi_{ij}'\Ra\Phi_{ij}''$ be 2-morphisms over $S$. Choose representatives $(\dot V_{ij},\hat\la_{ij}),(\dot V_{ij}',\hat\la_{ij}')$ in the $\approx$-equivalence classes $\La_{ij},\La_{ij}'$. Define $\dot V_{ij}''=\dot V_{ij}\cap \dot V_{ij}'\subseteq V_i$. Since $\phi_{ij}'\vert_{\dot V_{ij}''}=\phi_{ij}\vert_{\dot V_{ij}''}+O(s_i)$ by \eq{ku4eq47}, the discussion after Definition \ref{ku2def1}(vi) shows that there exists $\check\la'_{ij}:E_i\vert_{\dot V_{ij}''}\ra \phi_{ij}^*(TV_j)\vert_{\dot V_{ij}''}$ with
\begin{equation*}
\check\la_{ij}'=\hat\la_{ij}'\vert_{\dot V_{ij}''}+O(s_i).
\end{equation*}

Define $\hat\la_{ij}'':E_i\vert_{\dot V_{ij}''}\ra \phi_{ij}^*(TV_j)\vert_{\dot V_{ij}''}$ by $\hat\la_{ij}''=\hat\la_{ij}\vert_{\dot V_{ij}''}+\check\la_{ij}'$. Then $(\dot V_{ij}'',\hat\la_{ij}'')$ satisfies Definition \ref{ku4def31}(a),(b) for $\Phi_{ij},\Phi_{ij}''$. Hence $\La_{ij}''=[\dot V_{ij}'',\hat\la_{ij}'']:\Phi_{ij}\Ra\Phi_{ij}''$ is a 2-morphism over $S$. It is independent of choices. We define $\La_{ij}'\od\La_{ij}=\La_{ij}''$, and call this the {\it vertical composition of\/ $2$-morphisms over\/}~$S$.
\label{ku4def33}
\end{dfn}

\begin{dfn} Let $\Phi_{ij},\Phi_{ij}':(V_i,E_i,s_i,\psi_i)\ra (V_j,E_j,s_j,\psi_j)$ and $\Phi_{jk},\Phi_{jk}':(V_j,E_j,s_j,\ab\psi_j)\ab\ra (V_k,E_k,\ab s_k,\ab\psi_k)$ be 1-morphisms of m-Kuranishi neighbourhoods over $S\subseteq X$, and $\La_{ij}:\Phi_{ij}\Ra\Phi_{ij}'$, $\La_{jk}:\Phi_{jk}\Ra\Phi_{jk}'$ be 2-morphisms over $S$. 
Use our usual notation for $\Phi_{ij},\ldots,\La_{jk}$, and write $(V_{ik},\phi_{ik},\hat\phi_{ik})=\Phi_{jk}\ci\Phi_{ij}$, $(V_{ik}',\phi_{ik}',\hat\phi_{ik}')=\Phi_{jk}'\ci\Phi_{ij}'$, as in Definition \ref{ku4def32}. Choose representatives $(\dot V_{ij},\hat\la_{ij})$, $(\dot V_{jk},\hat\la_{jk})$ for~$\La_{ij},\La_{jk}$. 

Set $\dot V_{ik}=\dot V_{ij}\cap\phi_{ij}^{-1}(\dot V_{jk})\subseteq V_i$. Define a morphism of vector bundles on $\dot V_{ik}$
\begin{gather*}
\hat\la_{ik}:E_i\vert_{\dot V_{ik}}\longra \phi_{ik}^*(TV_k)\vert_{\dot V_{ik}}\\
\text{by}\;\>
\hat\la_{ik}=\phi_{ij}\vert_{\dot V_{ik}}^*(\d\phi_{jk})\ci(\hat\la_{ij})+
\phi_{ij}\vert_{\dot V_{ik}}^*(\hat\la_{jk})\ci \hat\phi_{ij}\vert_{\dot V_{ik}}.
\end{gather*}

One can now check that $(\dot V_{ik},\hat\la_{ik})$ satisfies  Definition \ref{ku4def31}(a),(b) for $\Phi_{ik},\Phi_{ik}'$, so $\La_{ik}=[\dot V_{ik},\hat\la_{ik}]$ is a 2-morphism, which is independent of choices. We define {\it horizontal composition of\/ $2$-morphisms over\/} $S$ to be $\La_{jk}*\La_{ij}=\La_{ik}$.

We have now defined all the structures of a strict 2-category: objects (m-Kuranishi neighbourhoods), 1- and 2-morphisms, their three kinds of composition, and two kinds of identities. As for $\Kur_S(X)$ in \S\ref{ku41}, the remaining 2-category axioms in \S\ref{kuB1} hold. Thus, we have defined the {\it strict\/ $2$-category $\mKur_S(X)$ of m-Kuranishi neighbourhoods over\/} $S$ to have objects m-Kuranishi neighbourhoods $(V_i,E_i,s_i,\psi_i)$ on $X$ with $S\subseteq\Im\psi_i$, and 1- and 2-morphisms $\Phi_{ij},\La_{ij}$ as above. All 2-morphisms in $\mKur_S(X)$ are 2-isomorphisms, that is, $\mKur_S(X)$ is a (2,1)-category.

We define a 1-morphism $\Phi_{ij}:(V_i,E_i,s_i,\psi_i)\ra (V_j,E_j,s_j,\psi_j)$ in $\mKur_S(X)$ to be an {\it m-coordinate change over\/} $S$ if it is an equivalence in~$\mKur_S(X)$.
\label{ku4def34}
\end{dfn}

We can now follow \S\ref{ku41}--\S\ref{ku43} from Definition \ref{ku4def9} until Theorem \ref{ku4thm3}, taking $\Ga_i=\{1\}$ throughout. This gives:

\begin{thm} We can define a weak\/ $2$-category $\mKur$ of \begin{bfseries}m-Kuranishi spaces\end{bfseries}. Objects of\/ $\mKur$ are $\fX=(X,\cK)$ where $X$ is a Hausdorff, second countable topological space and\/ $\cK=\bigl(I,$ $(V_i,E_i,s_i,\psi_i)_{i\in I},$ $(V_{ij},\phi_{ij},\hat\phi_{ij})_{i,j\in I},$ $[\dot V_{ijk},\hat\la_{ijk}]_{i,j,k\in I}\bigr)$ is an \begin{bfseries}m-Kuranishi structure on $X$ of virtual dimension\end{bfseries} $n\in\Z,$ defined as in {\rm\S\ref{ku43}} but using m-Kuranishi neighbourhoods, m-coordinate changes and $2$-morphisms as above.

There is a natural full and faithful weak\/ $2$-functor $F_\mKur^\Kur:\mKur\ra\Kur$ embedding $\mKur$ as a full\/ $2$-subcategory of\/ $\Kur,$ which on objects maps $F_\mKur^\Kur:(X,\cK)\mapsto(X,\cK'),$ where for $\cK$ as above, $\cK'=\bigl(I,(V_i,E_i,\{1\},s_i,\psi_i)_{i\in I},$ $(V_{ij},\ab\id_{V_{ij}},\ab\phi_{ij},\ab\hat\phi_{ij})_{i,j\in I},$ $[\dot V_{ijk},\id_{\dot V_{ijk}},\hat\la_{ijl}]_{i,j,k\in I}\bigr)_{i,j,k\in I}\bigr)$ is the Kuranishi structure obtained by taking all finite groups $\Ga_i$ to be~$\{1\}$.
\label{ku4thm7}
\end{thm}

We can now explain how Kuranishi spaces relate to $\mu$-Kuranishi spaces in \S\ref{ku2}, and to Kuranishi spaces with trivial isotropy groups in~\S\ref{ku46}.

\begin{dfn} We will define a functor $F_\mKur^\muKur:\Ho(\mKur)\ra\muKur$, where $\Ho(\mKur)$ is the {\it homotopy category\/} of the weak 2-category $\mKur$, that is, the category with objects $\bX,\bY$ objects of $\mKur$, and morphisms $[\bs f]:\bX\ra\bY$ are 2-isomorphism classes $[\bs f]$ of 1-morphisms $\bs f:\bX\ra\bY$ in~$\mKur$.

Let $\bX=(X,\cK)$ be an object of $\mKur$, with $\cK=\bigl(I,$ $(V_i,E_i,s_i,\psi_i)_{i\in I},$ $(V_{ij},\phi_{ij},\hat\phi_{ij})_{i,j\in I},$ $[\dot V_{ijk},\hat\la_{ijk}]_{i,j,k\in I}\bigr)$. Then $(V_i,E_i,s_i,\psi_i)$ is a $\mu$-Kuranishi neighbourhood on $X$ for each $i\in I$, and $(V_{ij},\phi_{ij},\hat\phi_{ij})$ satisfies Definition \ref{ku2def3}(a)--(e) over $S=\Im\psi_i\cap\Im\psi_j$ by Definition \ref{ku4def30}(a)--(e), so taking the $\sim$-equivalence class $\Phi_{ij}'=[V_{ij},\phi_{ij},\hat\phi_{ij}]$ as in Definition \ref{ku2def3} gives a $\mu$-coordinate change $\Phi_{ij}':(V_i,E_i,s_i,\psi_i)\ra(V_j,E_j,s_j,\psi_j)$. Write $\cK'=\bigl(I,(V_i,E_i,s_i,\psi_i)_{i\in I}$, $\Phi_{ij,\;i,j\in I}'\bigr)$ and $\bX'=(X,\cK')$. Then Definition \ref{ku2def11}(d)--(f) follow from Definition \ref{ku4def12}(e),\ab (f),\ab (h), so $\bX'$ is a $\mu$-Kuranishi space. Define~$F_\mKur^\muKur(\bX)=\bX'$.

Next let $\bs f:\bX\ra\bY$ be a 1-morphism in $\mKur$, and set $\bX'=F_\mKur^\muKur(\bX)$ and $\bY'=F_\mKur^\muKur(\bY)$. Writing $\bs f=\bigl(f,\bs f_{ij,\;i\in I,\; j\in J}$, $\bs F_{ii',\;i,i'\in I}^{j,\; j\in J}$, $\bs F_{i,\;i\in I}^{jj',\; j,j'\in J}\bigr)$ as in \eq{ku4eq15}, with $\bs f_{ij}=(U_{ij},f_{ij},\hat f_{ij})$ for $i\in I$ and $j\in J$, we find that
\e
\bs f'=\bigl(f,[U_{ij},f_{ij},\hat f_{ij}]_{ij,\;i\in I,\; j\in J}\bigr):\bX'\ra\bY'
\label{ku4eq49}
\e
is a morphism in~$\muKur$.

Now suppose $\bs g:\bX\ra\bY$ is another 1-morphism and $\bs\eta:\bs f\Ra\bs g$ a 2-morphism in $\mKur$, so that $f=g:X\ra Y$, and let $\bs g$ include $\bs g_{ij}=(\ti U_{ij},g_{ij},\hat g_{ij})$ for $i\in I$ and $j\in J$. Then $\bs\eta_{ij}:(U_{ij},f_{ij},\hat f_{ij})\Ra(\ti U_{ij},g_{ij},\hat g_{ij})$ is a 2-morphism of m-Kuranishi neighbourhoods over $(S,f)$ for $S=\Im\chi_i\cap f^{-1}(\Im\psi_j)$, so comparing Definitions \ref{ku2def4} and \ref{ku4def31} shows that $[U_{ij},f_{ij},\hat f_{ij}]=[\ti U_{ij},g_{ij},\hat g_{ij}]$ in morphisms of $\mu$-Kuranishi neighbourhoods over $(S,f)$. Therefore $\bs f'$ in \eq{ku4eq49} is independent of the choice of representative $\bs f$ for the morphism $[\bs f]:\bX\ra\bY$ in $\Ho(\mKur)$. Define $F_\mKur^\muKur([\bs f])=\bs f'$. Comparing definitions of composition of (1-)morphisms and identity (1-)morphisms in \S\ref{ku23} and \S\ref{ku43}, we see that $F_\mKur^\muKur:\Ho(\mKur)\ra\muKur$ is a functor.

Define $\KurtrG$ to be the full 2-subcategory of $\Kur$ with objects $\bX$ with $G_x\bX=\{1\}$ for all $x\in\bX$. Observe that the 2-functor $F_\mKur^\Kur:\mKur\hookra\Kur$ in Theorem \ref{ku4thm7} maps to $\KurtrG\subset\Kur$, since if $\bX=(X,\cK)$ is a Kuranishi space with all groups $\Ga_i=\{1\}$ in $\cK$ then $G_x\bX=\{1\}$ for all $x\in\bX$, as $G_x\bX\subseteq\Ga_i$ for some $i\in I$ by \eq{ku4eq41}. 
\label{ku4def35}
\end{dfn}

The next (quite difficult) theorem will be proved in~\S\ref{ku73}. 

\begin{thm}{\bf(a)} The functor $F_\mKur^\muKur:\Ho(\mKur)\ra\muKur$ in Definition\/ {\rm\ref{ku4def35}} is an equivalence of categories.
\smallskip

\noindent{\bf(b)} The weak\/ $2$-functor $F_\mKur^\Kur:\mKur\ra\KurtrG$ from Theorem\/ {\rm\ref{ku4thm7}} is an equivalence of weak\/ $2$-categories.
\smallskip

Combining {\bf(a)\rm,\bf(b)} gives an equivalence of categories\/~$\Ho(\KurtrG)\!\simeq\!\muKur$.

\label{ku4thm8}
\end{thm}

The 2-category structure on $\mKur$ or $\KurtrG$ encodes information which is forgotten by $\muKur\simeq\Ho(\mKur)\simeq\Ho(\KurtrG)$, but which is sometimes important. An example in which the 2-category structure matters is d-transverse `fibre products' $\bX\t_{\bs g,\bZ,\bs h}\bY$, as discussed in Example \ref{ku2ex6}(iv) and Example \ref{ku4ex2}(iii): these do not satisfy a universal property in the category $\muKur$, so are not fibre products in $\muKur$ in the sense of category theory, but they do satisfy a universal property (see \S\ref{kuB3}) in the weak 2-categories $\mKur,\KurtrG$ or $\Kur$, and so are genuine 2-category fibre products there.

This is similar to the question of whether orbifolds should be a considered to be a category $\Ho(\Orb)$ or a 2-category $\Orb$, as discussed in \S\ref{ku45}. However, the category $\muKur$ has one good property not shared by $\Ho(\Orb)$: morphisms $\bs f:\bX\ra\bY$ in $\muKur$ form a sheaf on $\bX$ (that is, given an open cover $\{\bU_i:i\in I\}$ of $\bX$ and morphisms $\bs f_i:\bU_i\ra\bY$ in $\muKur$ with $\bs f_i\vert_{\bU_i\cap\bU_j}=\bs f_j\vert_{\bU_i\cap\bU_j}$, there is a unique $\bs f:\bX\ra\bY$ with $\bs f\vert_{\bU_i}=\bs f_i$ for all $i\in I$), but morphisms $[\ff]:\fX\ra\fY$ in $\Ho(\Orb)$ do not form a sheaf if $\fY$ has nontrivial isotropy groups.

\subsection{Relation to other definitions of Kuranishi space}
\label{ku48}

Different notions of Kuranishi space, and related terms such as `Kuranishi structure', `Kuranishi atlas', and `good coordinate system', can be found in the literature in the work of Fukaya, Oh, Ohta and Ono \cite{Fuka,FOOO1,FOOO2,FOOO3, FOOO4,FOOO5,FOOO6,FOOO7,FOOO8,FuOn}, McDuff and Wehrheim \cite{McWe2,McWe3}, Yang \cite{Yang1,Yang2,Yang3}, and others. They are reviewed in Appendix~\ref{kuA}.

Our next definition of {\it fair coordinate system\/} is a kind of `least common denominator' for these other notions of Kuranishi space and related terms. As we show in Examples \ref{ku4ex3}--\ref{ku4ex6}, Kuranishi spaces in the sense of \S\ref{ku43}, and Fukaya--Oh--Ohta--Ono Kuranishi spaces \cite[\S A1]{FOOO1}, and topological spaces with Fukaya--Oh--Ohta--Ono weak good coordinate systems \cite[\S A1]{FOOO1}, and topological spaces with McDuff--Wehrheim weak Kuranishi atlases \cite{McWe3}, all carry natural fair coordinate systems. 

Theorem \ref{ku4thm9} shows that a topological space $X$ with a fair coordinate system $\cF$ can be made into a Kuranishi space $\bX$, uniquely up to equivalence in $\Kur$. This maps the spaces considered in \cite{Fuka,FOOO1,FOOO2,FOOO3, FOOO4,FOOO5,FOOO6,FOOO7,FOOO8,FuOn,McWe2,McWe3,Yang1,Yang2,Yang3} to our Kuranishi spaces.

\begin{dfn} Let $X$ be a Hausdorff, second countable topological space. A {\it fair coordinate system\/ $\cF$ on\/ $X,$ of virtual dimension\/} $n\in\Z$, is data $\cF=\bigl(A,(V_a,E_a,\Ga_a,s_a,\psi_a)_{a\in A}$, $S_{ab},\Phi_{ab,\;a,b\in A}$, $S_{abc},\La_{abc,\; a,b,c\in A}\bigr)$, where:
\begin{itemize}
\setlength{\itemsep}{0pt}
\setlength{\parsep}{0pt}
\item[(a)] $A$ is an indexing set (not necessarily finite).
\item[(b)] $(V_a,E_a,\Ga_a,s_a,\psi_a)$ is a Kuranishi neighbourhood on $X$ for each $a\in A$, with~$\dim V_a-\rank E_a=n$.
\item[(c)] $S_{ab}\subseteq \Im\psi_a\cap\Im\psi_b$ is an open set for all $a,b\in A$. (We can have $S_{ab}=\es$.)
\item[(d)] $\Phi_{ab}=(P_{ab},\pi_{ab},\phi_{ab},\hat\phi_{ab}):(V_a,E_a,\ab\Ga_a,\ab s_a,\ab\psi_a)\ra (V_b,E_b,\Ga_b,s_b,\psi_b)$ is a coordinate change over $S_{ab}$, for all $a,b\in A$.
\item[(e)] $S_{abc}\subseteq S_{ab}\cap S_{ac}\cap S_{bc}\subseteq\Im\psi_a\cap\Im\psi_b\cap\Im\psi_c$ is an open set for all $a,b,c\in A$. (We can have $S_{abc}=\es$.)
\item[(f)] $\La_{abc}=[\dot P_{abc},\la_{abc},\hat\la_{abc}]:\Phi_{bc}\ci\Phi_{ab}\Ra\Phi_{ac}$ is a 2-morphism for all $a,b,c\in A$, defined over $S_{abc}$. 
\item[(g)] $\bigcup_{a\in A}\Im\psi_a=X$. 
\item[(h)] $S_{aa}=\Im\psi_a$ and $\Phi_{aa}=\id_{(V_a,E_a,\Ga_a,s_a,\psi_a)}$ for all $a\in A$.
\item[(i)] $S_{aab}=S_{abb}=S_{ab}$ and $\La_{aab}=\bs\be_{\Phi_{ab}}$, $\La_{abb}=\bs\ga_{\Phi_{ab}}$ for all $a,b\in A$. 
\item[(j)] The following diagram of 2-morphisms over $S_{abc}\cap S_{abd}\cap S_{acd}\cap S_{bcd}$ commutes for all $a,b,c,d\in A$:
\begin{equation*}
\xymatrix@C=90pt@R=15pt{
*+[r]{(\Phi_{cd}\ci\Phi_{bc})\ci\Phi_{ab}} \ar@{=>}[d]^{\bs\al_{\Phi_{cd},\Phi_{bc},\Phi_{ab}}}
\ar@{=>}[rr]_(0.53){\La_{bcd}*\id_{\Phi_{ab}}} && *+[l]{\Phi_{bd}\ci\Phi_{ab}} \ar@{=>}[d]_{\La_{abd}}  
\\
*+[r]{\Phi_{cd}\ci(\Phi_{bc}\ci\Phi_{ab})} \ar@{=>}[r]^(0.65){\id_{\Phi_{cd}}*\La_{abc}}
& \Phi_{cd}\ci\Phi_{ac} \ar@{=>}[r]^{\La_{acd} } & *+[l]{\Phi_{ad}.\!} }
\end{equation*}
\end{itemize}
Also, {\it either\/} condition (k) {\it or\/} condition (k$)'$ below hold, or both, where:
\begin{itemize}
\setlength{\itemsep}{0pt}
\setlength{\parsep}{0pt}
\item[(k)] Suppose $B\subseteq A$ is finite and nonempty, and $x\in\bigcap_{b\in B}\Im\psi_b\subseteq X$. Then there exists $a\in A$ such that $x\in S_{ab}$ for all $b\in B$, and if $b,c\in B$ with $x\in S_{bc}$ then~$x\in S_{abc}$.
\end{itemize}
And: 
\begin{itemize}
\setlength{\itemsep}{0pt}
\setlength{\parsep}{0pt}
\item[(k$)'$] Suppose $B\subseteq A$ is finite and nonempty, and $x\in\bigcap_{b\in B}\Im\psi_b\subseteq X$. Then there exists $d\in A$ such that $x\in S_{bd}$ for all $b\in B$, and if $b,c\in B$ with $x\in S_{bc}$ then~$x\in S_{bcd}$.
\end{itemize}
\label{ku4def36}
\end{dfn}

The name `fair coordinate system' is intended to suggest something like the `good coordinate systems' in Appendix \ref{kuA}, but not as strong. Here (k),(k$)'$ are somewhat arbitrary, and one could substitute other conditions instead. What we are trying to achieve by these conditions on the $S_{ab},S_{abc}$ is roughly that: 
\begin{itemize}
\setlength{\itemsep}{0pt}
\setlength{\parsep}{0pt}
\item[(A)] If $x\in\Im\psi_b\cap\Im\psi_c$, one can map $(V_b,E_b,\Ga_b,s_b,\psi_b)\!\ra\!(V_c,E_c,\Ga_c,s_c,\psi_c)$ near $x$ by a finite chain of coordinate changes $\Phi_{ij}$ and their (quasi)inverses $\Phi_{ji}^{-1}$ --- for (k) by $\Phi_{ac}\ci\Phi_{ab}^{-1}$, and for (k$)'$ by $\Phi_{cd}^{-1}\ci\Phi_{bd}$.
\item[(B)] Any two such chains of $\Phi_{ij},\Phi_{ji}^{-1}$ near $x$ are canonically 2-isomorphic near $x$ using combinations of the 2-isomorphisms $\La_{ijk}$ and their inverses.
\end{itemize}
We chose (k),(k$)'$ as they hold in our examples, and there is a nice method to prove Theorem \ref{ku4thm9} using (k) or (k$)'$.

\begin{ex} Let $\bX=(X,\cK)$ be a Kuranishi space in the sense of \S\ref{ku43}, with $\cK=\bigl(I,(V_i,E_i,\Ga_i,s_i,\psi_i)_{i\in I}$, $\Phi_{ij,\;i,j\in I}$, $\La_{ijk,\; i,j,k\in I}\bigr)$. Set $S_{ij}=\Im\psi_i\cap\Im\psi_j$ for all $i,j\in I$, and $S_{ijk}=\Im\psi_i\cap\Im\psi_j\cap\Im\psi_k$ for all $i,j,k\in I$. Then $\cF=\bigl(I,(V_i,E_i,\Ga_i,s_i,\psi_i)_{i\in I}$, $S_{ij},\Phi_{ij,\;i,j\in I}$, $S_{ijk},\La_{ijk,\; i,j,k\in I}\bigr)$ is a fair coordinate system on $X$. Here Definition \ref{ku4def36}(a)--(j) are immediate from Definition \ref{ku4def12}(a)--(h), and both of Definition \ref{ku4def36}(k),(k$)'$ hold, where we can take $a\in B$ arbitrary in (k) and $d\in B$ arbitrary in (k$)'$.
\label{ku4ex3}
\end{ex}

\begin{ex} Using the notation of \S\ref{kuA1}, suppose $\bX=(X,\cK)$ is a FOOO Kuranishi space without boundary with $\vdim\bX=n$, in the sense of Definition \ref{kuAdef3}. Then $\cK$ gives a FOOO Kuranishi neighbourhood $(V_p,E_p,\Ga_p,s_p,\psi_p)$ for each $p\in X$, and for all $p,q\in X$ with $q\in\Im\psi_p$ it gives a FOOO coordinate change $\Phi_{qp}=(V_{qp},h_{qp},\vp_{qp},\hat\vp_{qp}):(V_q,E_q,\Ga_q,s_q,\psi_q)\ra(V_p,E_p,\Ga_p,s_p,\psi_p)$ defined on an open neighbourhood $S_{qp}$ of $q$ in $\Im\psi_q\cap\Im\psi_p$, and for all $p,q,r\in X$ with $q\in\Im\psi_p$ and $r\in S_{qp}$, Definition \ref{kuAdef3}(b) gives unique group elements $\ga_{rqp}^\al\in\Ga_p$ which relate $\Phi_{qp}\ci\Phi_{rq}$ to $\Phi_{rp}$ on~$S_{rqp}:=S_{qp}\cap S_{rp}\cap S_{rq}$. 

We will define a fair coordinate system $\cF$ on $X$. Take the indexing set $A$ to be $A=X$, and for each $p\in A$, let the Kuranishi neighbourhood $(V_p,E_p,\Ga_p,s_p,\psi_p)$ be as in $\cK$, regarded as a Kuranishi neighbourhood in the sense of \S\ref{ku41} as in Example \ref{kuAex1}. If $p\ne q\in A$ with $q\in\Im\psi_p$, define $S_{qp}\subseteq\Im\psi_q\cap\Im\psi_p$ to be the domain of the FOOO coordinate change $\Phi_{qp}$ in $\cK$. Define $\ti\Phi_{qp}:(V_q,E_q,\Ga_q,s_q,\psi_q)\ra(V_p,E_p,\Ga_p,s_p,\psi_p)$ to be the coordinate change over $S_{qp}$ in the sense of \S\ref{ku41} associated to the FOOO coordinate change $\Phi_{qp}$ in Example \ref{kuAex2}. Define $S_{pp}=\Im\psi_p$ and $\ti\Phi_{pp}=\id_{(V_p,E_p,\Ga_p,s_p,\psi_p)}$ for all $p\in A$. If $p\ne q\in A$ and $q\notin \Im\psi_p$, define $S_{qp}=\es$ and $\ti\Phi_{qp}=(\es,\es,\es,\es)$.

If $p\ne q\ne r\in A$ with $q\in\Im\psi_p$ and $r\in S_{qp}$, set $S_{rqp}=S_{qp}\cap S_{rp}\cap S_{rq}$, and define $\La_{rqp}:\ti\Phi_{qp}\ci\ti\Phi_{rq}\Ra\ti\Phi_{rp}$ to be the 2-morphism over $S_{rqp}$ defined in Example \ref{kuAex3}(ii) using the group elements $\ga_{rqp}^\al\in\Ga_p$ in Definition \ref{kuAdef3}(b). If $p\ne q\ne r\in A$ with $q\notin\Im\psi_p$ or $r\notin S_{qp}$, define $S_{rqp}=\es$ and $\La_{rqp}=[\es,\es,\es]$. Define $S_{qpp}=S_{qqp}=S_{qp}$ and $\La_{qqp}=\bs\be_{\ti\Phi_{qp}}$, $\La_{qpp}=\bs\ga_{\ti\Phi_{qp}}$ for all $p,q\in A$. This defines all the data in $\cF=\bigl(A,(V_p,E_p,\Ga_p,s_p,\psi_p)_{p\in A}$, $S_{qp},\ti\Phi_{qp,\;q,p\in A}$, $S_{rqp},\La_{rqp,\; r,q,p\in A}\bigr)$. We will show $\cF$ satisfies Definition~\ref{ku4def36}(a)--(k).

Parts (a)--(i) are immediate. For (j), if $p\ne q\ne r\ne s\in X$ with $q\in\Im\psi_p$ and $r\in S_{qp}$ and $s\in S_{rq}\cap S_{rp}$ then Definition \ref{kuAdef3}(b) gives elements $\ga_{rqp}^\al,\ga_{sqp}^{\al'},\ga_{srp}^{\al''}\in\Ga_p$ and $\ga_{srq}^{\al'''}\in\Ga_q$ satisfying \eq{kuAeq3}. Using \eq{kuAeq3} four times we see that
\e
\ga_{rqp}^{\al}\ga_{srp}^{\al''}\cdot\vp_{sp}
=\vp_{qp}\ci\vp_{rq}\ci\vp_{sr}=h_{qp}(\ga_{srq}^{\al'''})\ga_{sqp}^{\al''}\cdot\vp_{sp},
\label{ku4eq50}
\e
where \eq{ku4eq50} holds on the domain
\e
\begin{split}
&\vp_{sr}^{-1}\bigl((\vp_{rq}^{-1}(V_{qp})\cap V_{rq} \cap V_{rp})^\al\bigr)\cap(\vp_{sq}^{-1}(V_{qp})\cap V_{sq} \cap V_{sp})^{\al'}\cap \\
&(\vp_{sr}^{-1}(V_{rp})\cap V_{sr} \cap V_{sp})^{\al''}\cap
(\vp_{sr}^{-1}(V_{rq})\cap V_{sr} \cap V_{sq})^{\al'''}.
\end{split}
\label{ku4eq51}
\e
If \eq{ku4eq51} is nonempty, the argument of Example \ref{kuAex3}(iii) implies that $\ga_{rqp}^{\al}\ga_{srp}^{\al''}=h_{qp}(\ga_{srq}^{\al'''})\ga_{sqp}^{\al''}$. This is the condition required to verify $\La_{srp}\od(\La_{rqp}*\id_{\ti\Phi_{sr}})=\La_{sqp}\od(\id_{\ti\Phi_{qp}}*\La_{srq})\od\bs\al_{\ti\Phi_{qp},\ti\Phi_{rq},\ti\Phi_{sr}}$ on the component of $S_{srq}\cap S_{srp}\cap S_{sqp}\cap S_{rqp}$ corresponding to the connected components $\al,\al',\al'',\al'''$. 

This proves Definition \ref{ku4def36}(j) in this case. If $p=q$ then (j) becomes
\e
\begin{split}
\La_{srq}\od(\bs\ga_{\ti\Phi_{rq}}*\id_{\ti\Phi_{sr}})=\bs\ga_{\ti\Phi_{sq}}\od(\id_{\id_{(V_q,E_q,\Ga_q,s_q,\psi_q)}}*\La_{srq})&\\
{}\od\bs\al_{\id_{(V_q,E_q,\Ga_q,s_q,\psi_q)},\ti\Phi_{rq},\ti\Phi_{sr}}&,
\end{split}
\label{ku4eq52}
\e
which holds trivially, and the cases $q=r$, $r=s$ are similar. In the remaining cases one of $S_{srq},S_{srp},S_{sqp},S_{rqp}$ is empty, so (j) is vacuous. Thus (j) holds.

For (k), suppose $B\subseteq A$ is finite and nonempty, and $x\in\bigcap_{p\in B}\Im\psi_p\subseteq X$. Then $x\in S_{xp}$ for all $p\in B$, since $S_{xp}$ is an open neighbourhood of $x$ in $\Im\psi_x\cap\Im\psi_p$, and $x\in S_{xqp}$ for all $q,p\in B$ with $x\in S_{qp}$, since $S_{xqp}=S_{qp}\cap S_{xp}\cap S_{xq}$ in this case and $x\in S_{xp}$, $x\in S_{xq}$. Thus (k) holds with $a=x$, and $\cF$ is a fair coordinate system on~$X$.

If instead $X$ has a DY Kuranishi structure $\cK$ without boundary in the sense of Definition \ref{kuAdef12} (a minor variation on FOOO Kuranishi structures), exactly the same construction yields a fair coordinate system $\cF$ on~$X$.
\label{ku4ex4}
\end{ex}

\begin{ex} Let $\bigl((I,\pr),(V_i,E_i,\Ga_i,s_i,\psi_i)_{i\in I},\Phi_{ij,\; i\pr j}\bigr)$ be a FOOO weak good coordinate system without boundary of virtual dimension $n\in\Z$ on a compact, metrizable topological space $X$, in the sense of Definition~\ref{kuAdef7}. 

We will define a fair coordinate system $\cF$ on $X$. Take the indexing set $A$ to be $I$, and the (FOOO) Kuranishi neighbourhoods $(V_i,E_i,\Ga_i,s_i,\psi_i)$ for $i\in I$ to be as given. If $i\ne j\in I$ with $i\pr j$, define $S_{ij}=\Im\psi_i\cap\Im\psi_j$, and $\ti\Phi_{ij}:(V_i,E_i,\Ga_i,s_i,\psi_i)\ra(V_j,E_j,\Ga_j,s_j,\psi_j)$ to be the coordinate change over $S_{ij}$ in the sense of \S\ref{ku41} associated to the FOOO coordinate change $\Phi_{ij}$ in Example \ref{kuAex2}. Define $S_{ii}=\Im\psi_i$ and $\ti\Phi_{ii}=\id_{(V_i,E_i,\Ga_i,s_i,\psi_i)}$ for all $i\in I$. If $i\ne j\in I$ and $i\not\pr j$, define $S_{ij}=\es$ and~$\ti\Phi_{ij}=(\es,\es,\es,\es)$.

If $i\ne j\ne k\in I$ with $i\pr j\pr k$, set $S_{ijk}=\Im\psi_i\cap\Im\psi_j\cap\Im\psi_k$, and define $\La_{ijk}:\ti\Phi_{jk}\ci\ti\Phi_{ij}\Ra\ti\Phi_{ik}$ to be the 2-morphism over $S_{ijk}$ defined in Example \ref{kuAex3}(ii) using the unique group element $\ga_{ijk}\in\Ga_k$ in
Definition \ref{kuAdef7}(b). If $i\ne j\ne k\in I$ with $i\not\pr j$ or $j\not\pr k$, define $S_{ijk}=\es$ and $\La_{ijk}=[\es,\es,\es]$. Set $S_{iij}=S_{ijj}=\Im\psi_i\cap\Im\psi_j$ and $\La_{iij}=\bs\be_{\ti\Phi_{ij}}$, $\La_{ijj}=\bs\ga_{\ti\Phi_{ij}}$ for all $i,j\in I$. This defines all the data in $\cF=\bigl(I,(V_i,E_i,\Ga_i,s_i,\psi_i)_{i\in I}$, $S_{ij},\ti\Phi_{ij,\;i,j\in I}$, $S_{ijk},\La_{ijk,\; i,j,k\in I}\bigr)$. We shall show $\cF$ satisfies Definition~\ref{ku4def36}(a)--(k).

Parts (a)--(i) are immediate. For (j), if $i\ne j\ne k\ne l$ in $I$ with $i\pr j\pr k\pr l$ and $\Im\psi_i\cap\Im\psi_j\cap\Im\psi_k\cap\Im\psi_l\ne\es$ then the argument of \eq{ku4eq50}--\eq{ku4eq51} shows that $\ga_{jkl}\ga_{ijl}=h_{kl}(\ga_{ijk})\ga_{ikl}$, and so $\La_{ijl}\od(\La_{jkl}*\id_{\ti\Phi_{ij}})=\La_{ikl}\od(\id_{\ti\Phi_{kl}}*\La_{ijk})\od\bs\al_{\ti\Phi_{kl},\ti\Phi_{jk},\ti\Phi_{ij}}$ as we want. The cases $i=j$, $j=k$, $k=l$ hold as for \eq{ku4eq52}, and in the remaining cases one of $S_{ijk},S_{ijl},S_{ikl},S_{jkl}$ is empty, so (j) is vacuous. Thus (j) holds.

For (k) or (k$)'$, suppose $\es\ne B\subseteq I$ is finite and $x\in\bigcap_{b\in B}\Im\psi_b$. Then for all $b\ne c\in B$ we have $x\in\Im\psi_b\cap\Im\psi_c\ne\es$, so $b\pr c$ or $c\pr b$ by Definition \ref{kuAdef7}(a). Thus the partial order $\pr$ restricted to $B$ is a total order, and we may uniquely write $B=\{b_1,b_2,\ldots,b_m\}$ with $b_1\pr b_2\pr\cdots\pr b_m$. It is now easy to check that (k) holds with $a=b_1$, and also (k$)'$ holds with $d=b_m$. Therefore $\cF$ is a fair coordinate system on~$X$.
\label{ku4ex5}
\end{ex}

\begin{ex} Let $\bigl(A,I,(V_B,E_B,\Ga_B,s_B,\psi_B)_{B\in I},\Phi_{BC,\;B,C\in I,\; B\subsetneq C}\bigr)$ be an MW weak Kuranishi atlas without boundary of virtual dimension $n\in\Z$ on a compact, metrizable topological space $X$, in the sense of Definition~\ref{kuAdef11}.

We will define a fair coordinate system $\cF$ on $X$. Take the indexing set to be $I$ (not $A$) and the Kuranishi neighbourhoods $(V_B,E_B,\Ga_B,s_B,\psi_B)$ for $B\in I$ to be as given. If $B,C\in I$ with $B\subsetneq C$, define $S_{BC}=\Im\psi_B\cap\Im\psi_C$, and $\ti\Phi_{BC}:(V_B,E_B,\Ga_B,s_B,\psi_B)\ra(V_C,E_C,\Ga_C,s_C,\psi_C)$ to be the coordinate change over $S_{BC}$ in the sense of \S\ref{ku41} associated to the MW coordinate change $\Phi_{BC}$ in Example \ref{kuAex4}. Define $S_{BB}=\Im\psi_B$ and $\ti\Phi_{BB}=\id_{(V_B,E_B,\Ga_B,s_B,\psi_B)}$ for all $B\in I$. If $B\not\subseteq C$ in $I$, define $S_{BC}=\es$ and~$\ti\Phi_{BC}=(\es,\es,\es,\es)$.

If $B\subsetneq C\subsetneq D$ in $I$ then Definition \ref{kuAdef11}(b)--(d) say essentially that $\Phi_{CD}\ci\Phi_{BC}=\Phi_{BD}$ on the intersection of their domains. Example \ref{kuAex5} defines a canonical 2-isomorphism $\La_{BCD}:\ti\Phi_{CD}\ci\ti\Phi_{BC}\Ra\ti\Phi_{BD}$ on~$S_{BCD}:=\Im\psi_B\cap\Im\psi_C\cap\Im\psi_D$.

If $B\ne C\ne D\in I$ with $B\not\subset C$ or $C\not\subset D$, define $S_{BCD}=\es$ and $\La_{BCD}=[\es,\es,\es]$. Set $S_{BBC}=S_{BCC}=S_{BC}$ and $\La_{BBC}=\bs\be_{\ti\Phi_{BC}}$, $\La_{BCC}=\bs\ga_{\ti\Phi_{BC}}$ for all $B,C\in I$. This defines all the data in $\cF=\bigl(I,(V_B,E_B,\Ga_B,s_B,\psi_B)_{B\in I}$, $S_{BC},\ti\Phi_{BC,\;B,C\in I}$, $S_{BCD},\La_{BCD,\; B,C,D\in I}\bigr)$. We shall show $\cF$ satisfies Definition \ref{ku4def36}(a)--(j),(k$)'$.
 
Parts (a)--(i) are immediate. For (j), if $B\subsetneq C\subsetneq D\subsetneq E$ in $I$ then Definition \ref{kuAdef11}(b)--(d) basically imply that
\begin{equation*}
\Phi_{DE}\ci(\Phi_{CD}\ci\Phi_{BC})=\Phi_{BE}=(\Phi_{DE}\ci\Phi_{CD})\ci\Phi_{BC}
\end{equation*}
holds on the intersection of their domains, and from this we easily see that $\La_{BDE}\od(\La_{CDE}*\id_{\ti\Phi_{BC}})=\La_{BDE}\od(\id_{\ti\Phi_{DE}}*\La_{BCD})\od\bs\al_{\ti\Phi_{DE},\ti\Phi_{CD},\ti\Phi_{BC}}$, as we want. The remaining cases follow as in Examples \ref{ku4ex4} and \ref{ku4ex5}. Thus (j) holds.

For (k$)'$, suppose $\es\ne J\subseteq I$ is finite and $x\in\bigcap_{B\in J}\Im\psi_B\subseteq X$. Then Definition \ref{kuAdef11}(a) says that $D=\bigcup_{B\in J}B$ lies in $I$, and $x\in\bigcap_{B\in J}\Im\psi_B\ab\subseteq\Im\psi_D$. For any $B\in J$ we have $B\subseteq D$, so $S_{BD}=\Im\psi_B\cap\Im\psi_D\ni x$. If $B,C\in J$ with $x\in S_{BC}$ then $B\subseteq C$, as otherwise $S_{BC}=\es$, so $B\subseteq C\subseteq D$ and $S_{BCD}=\Im\psi_B\cap\Im\psi_C\cap\Im\psi_D\ni x$. Therefore (k$)'$ holds with $d=D$, and $\cF$ is a fair coordinate system on~$X$.
\label{ku4ex6}
\end{ex}

Here is the main result of this section, which will be proved in \S\ref{ku74}. When we say $(V_a,E_a,\Ga_a,s_a,\psi_a)$ `may be given the structure of a Kuranishi neighbourhood on the Kuranishi space $\bX$', we mean that as in \S\ref{ku44}, we can choose implicit extra data $\Phi_{ai,\; i\in I}$, $\La_{aij,\; i,j\in I}$ relating $(V_a,E_a,\Ga_a,s_a,\psi_a)$ to the Kuranishi structure $\cK$ on $X$, and similarly, by `$\Phi_{ab}$ may be given the structure of a coordinate change over $S_{ab}$ on the Kuranishi space $\bX$', we mean that we can choose implicit extra data $\La_{abi,\; i\in I}$ relating $\Phi_{ab}$ to~$\cK$.

\begin{thm} Suppose\/ $\cF=\bigl(A,(V_a,E_a,\Ga_a,s_a,\psi_a)_{a\in A},$ $S_{ab},\Phi_{ab,\;a,b\in A},$ $S_{abc},\ab\La_{abc,\; a,b,c\in A}\bigr)$ is a fair coordinate system of virtual dimension\/ $n\in\Z$ on a Hausdorff, second countable topological space $X,$ in the sense of Definition\/ {\rm\ref{ku4def36}}. Then we may make $X$ into a Kuranishi space $\bX=(X,\cK)$ in the sense of\/ {\rm\S\ref{ku43}} with\/ $\vdim\bX=n,$ such that\/ $(V_a,E_a,\Ga_a,s_a,\psi_a)$ may be given the structure of a Kuranishi neighbourhood on the Kuranishi space $\bX$ in the sense of\/ {\rm\S\ref{ku44}} for all\/ $a\in A,$ and\/ $\Phi_{ab}:(V_a,E_a,\Ga_a,s_a,\psi_a)\ra(V_b,E_b,\Ga_b,s_b,\psi_b)$ may be given the structure of a coordinate change over $S_{ab}$ on the Kuranishi space $\bX$ in the sense of\/ {\rm\S\ref{ku44}} for all\/ $a,b\in A,$ and\/ $\La_{abc}:\Phi_{bc}\ci\Phi_{ab}\Ra\Phi_{ac}$ is the unique $2$-morphism over $S_{abc}$ given by Theorem\/ {\rm\ref{ku4thm4}(a)} for all\/ $a,b,c\in A$. This $\bX$ is unique up to equivalence in the $2$-category~$\Kur$.
\label{ku4thm9}
\end{thm}

The next proposition follows easily from Corollary \ref{ku4cor1} and Theorem~\ref{ku4thm9}.

\begin{prop} Let\/ $\cF=\bigl(A,(V_a,E_a,\Ga_a,s_a,\psi_a)_{a\in A},$ $S_{ab},\Phi_{ab,\;a,b\in A},$ $S_{abc},\ab\La_{abc,\; a,b,c\in A}\bigr)$ be a fair coordinate system on $X$. Suppose $\ti A\subseteq A$ with\/ $\bigcup_{a\in\ti A}\Im\psi_a\ab =X,$ and in Definition\/ {\rm\ref{ku4def36}(k),(k$)'$,} if\/ $B\subseteq\ti A\subseteq A$ then we can choose $a\in\ti A$ in {\rm(k)} and\/ $d\in\ti A$ in {\rm(k$)'$}. Then\/ $\ti{\cal F}=\bigl(\ti A,(V_a,E_a,\Ga_a,s_a,\psi_a)_{a\in\ti A},$ $S_{ab},\Phi_{ab,\;a,b\in\ti A},$ $S_{abc},\ab\La_{abc,\; a,b,c\in\ti A}\bigr)$ is also a fair coordinate system on $X$. Let\/ $\bX=(X,\cK)$ and\/ $\bs{\ti X}=(X,\ti{\cal K})$ be the Kuranishi spaces constructed from $\cF,\ti{\cal F}$ in Theorem\/ {\rm\ref{ku4thm9},} respectively. Then $\bX,\bs{\ti X}$ are canonically equivalent in the $2$-category $\Kur$.
\label{ku4prop7}
\end{prop}

We can now combine Theorem \ref{ku4thm9} with Examples \ref{ku4ex3}--\ref{ku4ex6}. For Example \ref{ku4ex3}, examining the proof in \S\ref{ku74} shows that a possible choice for the new Kuranishi structure $\cK'$ constructed from $\cF$ in Theorem \ref{ku4thm9} is for $\cK'$ to be equal to the old Kuranishi structure $\cK$ used to define $\cF$ in Example \ref{ku4ex3}. Thus, starting with a Kuranishi space $\bX$ and applying Example \ref{ku4ex3} to get $\cF$ and then Theorem \ref{ku4thm9} yields a Kuranishi space $\bX'$ equivalent to $\bX$ in $\Kur$, and we can choose $\bX'=\bX$. From Examples \ref{ku4ex4}--\ref{ku4ex6} we deduce:

\begin{thm} Suppose $\bX=(X,\cK)$ is a FOOO Kuranishi space without boundary, as in Definition\/ {\rm\ref{kuAdef3}}. Then we can construct a Kuranishi space $\bX'=(X,\cK')$ in the sense of\/ {\rm\S\ref{ku43}} with\/ $\vdim\bX'=\vdim\bX,$ with the same topological space $X,$ and\/ $\bX'$ is natural up to equivalence in the $2$-category\/~$\Kur$.\label{ku4thm10}
\end{thm}

\begin{thm} Suppose $X$ is a compact, metrizable topological space with a FOOO weak good coordinate system\/ $\bigl((I,\pr),(V_i,E_i,\Ga_i,s_i,\psi_i)_{i\in I},\Phi_{ij,\; i\pr j}\bigr)$ without boundary, of virtual dimension $n\in\Z,$ in the sense of Definition\/ {\rm\ref{kuAdef7}}. Then we can make $X$ into a Kuranishi space $\bX=(X,\cK)$ in the sense of\/ {\rm\S\ref{ku43}} with\/ $\vdim\bX=n,$ and\/ $\bX$ is natural up to equivalence in the $2$-category\/~$\Kur$.
\label{ku4thm11}
\end{thm}

\begin{thm} Suppose $X$ is a compact, metrizable topological space with an MW weak Kuranishi atlas\/ $\cK$ without boundary, of virtual dimension $n\in\Z,$ in the sense of Definition\/ {\rm\ref{kuAdef11}}. Then we can make $X$ into a Kuranishi space $\bX'=(X,\cK')$ in the sense of\/ {\rm\S\ref{ku43}} with\/ $\vdim\bX'=n,$ and\/ $\bX'$ is natural up to equivalence in the $2$-category $\Kur$. Commensurate MW weak Kuranishi atlases $\cK,\ti\cK$ on $\bX$ yield equivalent Kuranishi spaces\/~$\bX{}',\ti\bX{}'$.
\label{ku4thm12}
\end{thm}

\begin{proof} The first part is immediate from Example \ref{ku4ex6} and Theorem \ref{ku4thm9}. For the second part, note that as in Definition \ref{kuAdef11}, if $\cK,\ti\cK$ are commensurate then they are linked by a diagram of MW weak Kuranishi atlases
\e
\begin{gathered}
\xymatrix@!0@C=34pt@R=30pt{ *+[r]{\cK=\cK_0} \ar[dr] && \cK_1 \ar[dl] \ar[dr] && \cdots \ar[dl] \ar[dr] && \cK_{m-1} \ar[dl] \ar[dr] && *+[l]{\cK_m=\ti\cK} \ar[dl] \\
& \hat\cK_1 && \hat\cK_1 & \cdots & \hat\cK_{m-1} && \hat\cK_m, }
\end{gathered}
\label{ku4eq53}
\e
where each arrow is an inclusion of MW weak Kuranishi atlases.

By Proposition \ref{ku4prop7}, the construction of the first part applied to MW weak Kuranishi atlases $\cK,\hat\cK$ with $\cK\subseteq\hat\cK$ yields equivalent Kuranishi spaces, so \eq{ku4eq53} induces a corresponding diagram of equivalences in $\Kur$, and thus $\bX{}',\ti\bX{}'$ are equivalent in $\Kur$.
\end{proof}

\begin{thm} Suppose $X$ is a compact, metrizable topological space with a DY Kuranishi structure $\cK$ without boundary in the sense of Definition\/ {\rm\ref{kuAdef12}}. Then we can construct a Kuranishi space $\bX'=(X,\cK')$ in the sense of\/ {\rm\S\ref{ku43}} with\/ $\vdim\bX'=\vdim\bX,$ with the same topological space $X,$ and\/ $\bX'$ is natural up to equivalence in the $2$-category $\Kur$. R-equivalent DY Kuranishi structures $\cK,\ti\cK$ on $\bX$ yield equivalent Kuranishi spaces\/~$\bX{}',\ti\bX{}'$.
\label{ku4thm13}
\end{thm}

\begin{proof} The first part is immediate from Example \ref{ku4ex4} and Theorem \ref{ku4thm9}. For the second part, note that as in Definition \ref{kuAdef13}, if $\cK,\ti\cK$ are R-equivalent then there is a diagram of embeddings of DY Kuranishi structures on $X$:
\e
\xymatrix@C=30pt@R=15pt{\cK  & \cK_1 \ar[l]_\sim \ar@{=>}[r] & \cK_2 & \cK_3 \ar@{=>}[l] \ar[r]^\sim & \ti\cK. }
\label{ku4eq54}
\e

If $\ep:\cK_1\ra\cK_2$ is an embedding of DY Kuranishi structures, then following Example \ref{ku4ex4} we can define three fair coordinate systems $\cF_1,\cF_2,\cF_{12}$ on $X$, where $\cF_1,\cF_2$ come from $\cK_1,\cK_2$, and $\cF_{12}$ contains the Kuranishi neighbourhoods from $\cK_1$ and $\cK_2$, and the coordinate changes from $\cK_1,\cK_2$ and $\ep$, so that $\cF_{12}$ contains $\cF_1$ and $\cF_2$. Theorem \ref{ku4thm9} then gives Kuranishi structures $\cK_1',\cK_2',\cK_{12}'$ on $X$. Since $\cF_1\subset\cF_{12}$, $\cF_2\subset\cF_{12}$, by Proposition \ref{ku4prop7} we have equivalences $(X,\cK_1')\ra(X,\cK_{12}')$, $(X,\cK_2')\ra(X,\cK_{12}')$ in $\Kur$, and hence an equivalence $(X,\cK_1')\ra(X,\cK_2')$ in $\Kur$. Therefore \eq{ku4eq54} induces a corresponding diagram of equivalences in $\Kur$, and thus $\bX{}',\ti\bX{}'$ are equivalent in~$\Kur$.
\end{proof}

Combining Theorem \ref{ku4thm13} with Yang's Theorem \ref{kuAthm3}, \cite[Th.~3.1.7]{Yang1} gives:

\begin{thm} Suppose we are given a `polyfold Fredholm structure'\/ $\cP$ on a compact metrizable topological space $X,$ that is, we write $X$ as the zeroes of a Fredholm section ${\mathfrak s}:\fV\ra\fE$ of a strong polyfold vector bundle $\fE\ra\fV$ over a polyfold\/ $\fV,$ where $\mathfrak s$ has constant Fredholm index $n\in\Z$. Then we can make $X$ into a Kuranishi space $\bX=(X,\cK)$ in the sense of\/ {\rm\S\ref{ku43}} with\/ $\vdim\bX=n,$ and\/ $\bX$ is natural up to equivalence in the $2$-category\/~$\Kur$.
\label{ku4thm14}
\end{thm}

Theorems \ref{ku4thm10}--\ref{ku4thm14} are important, as they show that the geometric structures on moduli spaces considered by Fukaya, Oh, Ohta and Ono \cite{Fuka,FOOO1,FOOO2,FOOO3,FOOO4,FOOO5,FOOO6,FOOO7,FOOO8}, McDuff and Wehrheim \cite{McWe2,McWe3}, Yang \cite{Yang1,Yang2,Yang3}, and Hofer, Wysocki and Zehnder \cite{Hofe,HWZ1,HWZ2,HWZ3,HWZ4,HWZ5,HWZ6}, can all be transformed to Kuranishi spaces in our sense. Thus, large parts of the symplectic geometry literature can now be interpreted in our framework. 

\section{Kuranishi spaces with boundary and corners}
\label{ku5}

Section \ref{ku2} defined the category of $\mu$-Kuranishi spaces $\muKur$, and \S\ref{ku3} explained how to modify this to define categories of $\mu$-Kuranishi spaces with boundary $\muKurb$, with corners $\muKurc$, and with g-corners $\muKurgc$. Section \ref{ku4} defined the weak 2-category of Kuranishi spaces $\Kur$. We now explain how to modify this to define weak 2-categories of {\it Kuranishi spaces with boundary\/} $\Kurb$, and {\it with corners\/} $\Kurc$, and {\it with g-corners\/}~$\Kurgc$.

Much of this is straightforward: we just apply the changes described in \S\ref{ku35}--\S\ref{ku38} to \S\ref{ku4} rather than \S\ref{ku2}. So we will be brief, except where new issues arise. We assume the material of \S\ref{ku31}--\S\ref{ku34} on manifolds with (g-)corners throughout. Sections \ref{ku51}--\ref{ku56} are the analogues of \S\ref{ku35}, \S\ref{ku36}, \S\ref{ku37}, \S\ref{ku47}, \S\ref{ku48}, and~\S\ref{ku38}.

\subsection{Defining Kuranishi spaces with boundary and corners}
\label{ku51}

We now extend \S\ref{ku41}--\S\ref{ku44} to the boundary and corners case.

\begin{dfn} Let $X$ be a topological space. A {\it Kuranishi neighbourhood $(V,E,\Ga,s,\psi)$ on\/ $X$ with boundary}, or {\it with corners}, is as in Definition \ref{ku4def1}, except that we take $V$ to be a manifold with boundary, or with corners, respectively.
\label{ku5def1}
\end{dfn}

\begin{dfn} Let $X,Y$ be topological spaces, $f:X\ra Y$ a continuous map, $(V_i,E_i,\Ga_i,s_i,\psi_i)$, $(V_j,E_j,\Ga_j,s_j,\psi_j)$ be Kuranishi neighbourhoods with boundary, or with corners, on $X,Y$ respectively, and $S\subseteq\Im\psi_i\cap f^{-1}(\Im\psi_j)\subseteq X$ be an open set. A 1-{\it morphism $\Phi_{ij}=(P_{ij},\pi_{ij},\phi_{ij},\hat\phi_{ij}):(V_i,E_i,\Ga_i,s_i,\psi_i)\ra (V_j,E_j,\Ga_j,s_j,\psi_j)$ of Kuranishi neighbourhoods with boundary or corners over\/} $(S,f)$ is as in Definition \ref{ku4def2}(a)--(f), except that in (a) we take $P_{ij}$ to be a manifold with boundary or corners, in (e) the notation $O(\pi_{ij}^*(s_i)^2)$ is interpreted as in Definition \ref{ku3def11} (in fact, all $O(s),O(s^2)$ notation in this chapter is interpreted as in Definition \ref{ku3def11}), and we impose the following extra condition (g), as in Definition~\ref{ku3def13}(ii):
\begin{itemize}
\setlength{\itemsep}{0pt}
\setlength{\parsep}{0pt}
\item[(g)] Let $C$ be a connected component of $S^k(P_{ij})$
for $k\ge 1$, and $\ov C$ be the closure of $C$ in $P_{ij}$. Then~$\ov C\cap \pi_{ij}^{-1}\bigl(\bar\psi_i^{-1}(S)\bigr)\ne\es$.
\end{itemize}

We impose (g) for the same reasons as in Remark \ref{ku3rem5}. Note that in the corners case, $\pi_{ij}:P_{ij}\ra V_i$ being \'etale (a local diffeomorphism) in Definition \ref{ku4def2}(b) implies that $\pi_{ij}$ is simple, in the sense of~\S\ref{ku31}.
\label{ku5def2}
\end{dfn}

\begin{dfn} Suppose $X,Y$ are topological spaces, $f:X\ra Y$ is continuous, $(V_i,E_i,\Ga_i,s_i,\psi_i)$, $(V_j,E_j,\Ga_j,s_j,\psi_j)$ are Kuranishi neighbourhoods with boundary or corners on $X,Y$, $S\subseteq\Im\psi_i\cap f^{-1}(\Im\psi_j)\subseteq X$ is open, and $\Phi_{ij},\Phi_{ij}':(V_i,E_i,\Ga_i,s_i,\psi_i)\ra (V_j,\ab E_j,\ab\Ga_j,\ab s_j,\ab\psi_j)$ are 1-morphisms over $(S,f)$, with $\Phi_{ij}=(P_{ij},\pi_{ij},\phi_{ij},\hat\phi_{ij})$ and $\Phi_{ij}'=(P_{ij}',\pi_{ij}',\phi_{ij}',\hat\phi_{ij}')$. Define {\it triples\/} $(\dot P_{ij},\la_{ij},\hat\la_{ij})$, the equivalence relation $\approx$ on triples, the $\approx$-equivalence class $[\dot P_{ij},\la_{ij},\hat\la_{ij}]$ of $(\dot P_{ij},\la_{ij},\hat\la_{ij})$, and 2-{\it morphisms\/ $\La_{ij}=[\dot P_{ij},\la_{ij},\hat\la_{ij}]:\Phi_{ij}\Ra\Phi_{ij}'$ over\/} $(S,f)$ as in Definition \ref{ku4def3}, but with the following differences, as in Definition~\ref{ku3def13}:
\begin{itemize}
\setlength{\itemsep}{0pt}
\setlength{\parsep}{0pt}
\item[(i)] $V_i,V_j,P_{ij},P_{ij}',\dot P_{ij},\ldots$ are now manifolds with corners, and $\pi_{ij},\phi_{ij},\ab\pi_{ij}',\ab\phi_{ij}',\ab\la_{ij},\ldots$ are smooth maps of manifolds with corners, in the sense of~\S\ref{ku31}. 
\item[(ii)] $\hat\la_{ij}$ maps $\pi_{ij}^*(E_i)\vert_{\dot P_{ij}}\ra \phi_{ij}^*({}^bTV_j)\vert_{\dot P_{ij}}$, where we have used the {\it b-tangent bundle\/} ${}^bTV_j$ of \S\ref{ku33} rather than the tangent bundle~$TV_j$.
\item[(iii)] The equation $\phi_{ij}'\ci\la_{ij}=\phi_{ij}\vert_{\dot P_{ij}}+\hat\la_{ij}\cdot \pi_{ij}^*(s_i)+O\bigl(\pi_{ij}^*(s_i)^2\bigr)$ in \eq{ku4eq1} is interpreted as in Definition \ref{ku3def11}(v$)'$. This makes sense by~(ii).
\end{itemize}
\label{ku5def3}
\end{dfn}

For Definitions \ref{ku4def4}--\ref{ku4def10}, we insert `with boundary' or `with corners' throughout. In Definition \ref{ku4def7} we write $\Kur^{\bf b}_S(X)\subset\Kur^{\bf c}_S(X)$ for the weak 2-categories of Kuranishi neighbourhoods over $S\subseteq X$ with boundary and with corners, and $\GKurb\subset\GKurc$ for the weak 2-categories of global Kuranishi neighbourhoods with boundary and with corners. 

For restrictions $\Phi_{ij}\vert_T$ of 1-morphisms $\Phi_{ij}=(P_{ij},\pi_{ij},\phi_{ij},\hat\phi_{ij})$ on $S$ to $T\subseteq S$ in Definitions \ref{ku4def9}, note that Definition \ref{ku5def2}(g) for $P_{ij},\pi_{ij},S$ does not imply Definition \ref{ku5def2}(g) for $P_{ij},\pi_{ij},T$. To deal with this, let $P_{ij}'\subseteq P_{ij}$ be the dense open subset obtained by deleting all connected components $C$ of $S^k(P_{ij})$ for $k\ge 1$ such that $\ov C\cap \pi_{ij}^{-1}\bigl(\bar\psi_i^{-1}(T)\bigr)=\es$, let $\pi_{ij}',\phi_{ij}',\hat\phi_{ij}'$ be the restrictions of $\pi_{ij},\phi_{ij},\hat\phi_{ij}$ to $P_{ij}'$, and define $\Phi_{ij}\vert_T=(P_{ij}',\pi_{ij}',\phi_{ij}',\hat\phi_{ij}')$. Otherwise there are no significant changes. This extends \S\ref{ku41} to the corners case.

As for Proposition \ref{ku3prop2}, using Definition \ref{ku5def2}(g), we may prove:

\begin{prop} Suppose\/ $\Phi_{ij}=(P_{ij},\pi_{ij},\phi_{ij},\hat\phi_{ij}):(V_i,E_i,\Ga_i,s_i,\psi_i)\ra (V_j,\ab E_j,\ab\Ga_j,\ab s_j,\ab\psi_j)$ is a coordinate change with corners over $S\subseteq X,$ that is, $\Phi_{ij}$ is an equivalence in the $2$-category $\Kur^{\bf c}_S(X)$. Then $\phi_{ij}:P_{ij}\ra V_j$ is a \begin{bfseries}simple map\end{bfseries}, in the sense of\/ {\rm\S\ref{ku31}}. (Also $\pi_{ij}:P_{ij}\ra V_i$ is simple, as it is \'etale.)
\label{ku5prop1}
\end{prop}

Here are the analogues of Theorems \ref{ku4thm1} and \ref{ku4thm2}, proved in \S\ref{ku75} and~\S\ref{ku76}.

\begin{thm} The analogue of Theorem\/ {\rm\ref{ku4thm1}} holds for Kuranishi neighbourhoods with boundary, and with corners.
\label{ku5thm1}
\end{thm}

\begin{thm} Let\/ $\Phi_{ij}=(P_{ij},\pi_{ij},\phi_{ij},\hat\phi_{ij}):(V_i,E_i,\Ga_i,s_i,\psi_i)\ra (V_j,\ab E_j,\ab\Ga_j,\ab s_j,\ab\psi_j)$ be a $1$-morphism of Kuranishi neighbourhoods with corners over\/ $S\subseteq X$. Let\/ $p\in\pi_{ij}^{-1}(\bar\psi_i^{-1}(S))\subseteq P_{ij},$ and set\/ $v_i=\pi_{ij}(p)\in V_i$ and\/ $v_j=\phi_{ij}(p)\in V_j$.

Suppose $\phi_{ij}$ is interior near $p$ (this is implied by the condition $\phi_{ij}$ simple below). Consider the sequence of finite-dimensional real vector spaces:
\e
\xymatrix@C=12.5pt{ 0 \ar[r] & {}^bT_{v_i}V_i \ar[rrrrr]^(0.43){{}^b\d s_i\vert_{v_i}\op({}^b\d\phi_{ij}\vert_p\ci{}^b\d\pi_{ij}\vert_p^{-1})} &&&&& E_i\vert_{v_i} \!\op\!{}^bT_{v_j}V_j 
\ar[rrr]^(0.56){-\hat\phi_{ij}\vert_p\op {}^b\d s_j\vert_{v_j}} &&& E_j\vert_{v_j} \ar[r] & 0. }\!\!\!\!{}
\label{ku5eq1}
\e
Here ${}^b\d s_i\vert_{v_i}$ means the composition of the natural map ${}^bT_{v_i}V_i\ra T_{v_i}V_i$ and $\nabla s_i\vert_{v_i}:T_{v_i}V_i\ra E_i\vert_{v_i}$ for any connection $\nabla$ on $E_i\ra V_i,$ and is independent of the choice of\/ $\nabla$ as $s_i(v_i)=0$. Also ${}^b\d\pi_{ij}\vert_p:{}^bT_pP_{ij}\ra {}^bT_{v_i}V_i$ and\/ ${}^b\d\phi_{ij}\vert_p:{}^bT_pP_{ij}\ra {}^bT_{v_j}V_j$ are defined as in\/ {\rm\S\ref{ku33},} as $\pi_{ij},\phi_{ij}$ are interior near $p,$ and\/ ${}^b\d\pi_{ij}\vert_p$ is invertible as $\pi_{ij}$ is \'etale. Definition\/ {\rm\ref{ku4def2}(e)} implies that\/ \eq{ku5eq1} is a complex. Also consider the morphism of finite groups
\e
\begin{split}
&\rho_p:\bigl\{(\ga_i,\ga_j)\in\Ga_i\t\Ga_j:(\ga_i,\ga_j)\cdot p=p\bigr\}
\longra\bigl\{\ga_j\in\Ga_j:\ga_j\cdot v_j=v_j\bigr\},\\
&\rho_p:(\ga_i,\ga_j)\longmapsto \ga_j.
\end{split}
\label{ku5eq2}
\e

Then $\Phi_{ij}$ is a coordinate change over $S$ if and only if\/ $\phi_{ij}$ is simple, and for all\/ $p\in\pi_{ij}^{-1}(\bar\psi_i^{-1}(S))$ equation\/
\eq{ku5eq1} is exact and\/ \eq{ku5eq2} is an isomorphism.
\label{ku5thm2}
\end{thm}

Here the condition $\phi_{ij}$ simple is necessary for $\Phi_{ij}$ to be a coordinate change, by Proposition~\ref{ku5prop1}.

Now as we explained for $\mu$-Kuranishi spaces in \S\ref{ku35}, in \S\ref{ku41}--\S\ref{ku42} we were doing differential geometry, so in Kuranishi neighbourhoods $(V_i,E_i,\Ga_i,s_i,\psi_i)$, it mattered that $V_i$ is a manifold (possibly with corners), $E_i\ra V_i$ a vector bundle, and so on. In \S\ref{ku43}--\S\ref{ku44}, the arguments were of a different character: we treated Kuranishi neighbourhoods $(V_i,E_i,\Ga_i,s_i,\psi_i)$ and their 1-morphisms $\Phi_{ij}$ and 2-morphisms $\La_{ij}$ just as objects and 1- and 2-morphisms in certain 2-categories $\Kur_S(X)$, without caring what they really are, and the proofs involved 2-categories and stacks on topological spaces, but no differential geometry.

Because of this, there really are {\it no significant issues\/} in extending \S\ref{ku43}--\S\ref{ku44} to the boundary and corners cases. We just insert `with boundary' or `with corners' throughout, use the definitions of Kuranishi neighbourhoods with corners and their 1- and 2-morphisms and coordinate changes explained above, apply Theorem \ref{ku5thm1} in place of Theorem \ref{ku4thm1}, and everything just works, exactly as it did before.

Thus, as in Definitions \ref{ku4def12}, \ref{ku4def13} and \ref{ku4def14} we can define {\it Kuranishi spaces with boundary\/} and {\it with corners\/} $\bX$ and their 1-morphisms $\bs f:\bX\ra\bY$ and 2-morphisms $\bs\eta:\bs f\Ra\bs g$, as in Definitions \ref{ku4def15}, \ref{ku4def16} and \ref{ku4def17} and Propositions \ref{ku4prop1}, \ref{ku4prop2}, \ref{ku4prop3}, \ref{ku4prop4} and \ref{ku4prop5} we can define the remaining structures of a weak 2-category, and so show as in Theorem \ref{ku4thm3} that Kuranishi spaces with boundary, and with corners, form weak 2-categories $\Kurb,\Kurc$.

We can regard Kuranishi spaces (without boundary) and their 1- and 2-morphisms, as in \S\ref{ku41}--\S\ref{ku43}, as special examples of Kuranishi spaces with corners in which the $V_i$ in $(V_i,E_i,\Ga_i,s_i,\psi_i)$ have $\pd V_i=\es$ for each $i\in I$. Hence, we have inclusions of full 2-subcategories~$\Kur\subset\Kurb\subset\Kurc$.

Also, as in Definition \ref{ku4def18}, we can regard manifolds with boundary or corners as examples of Kuranishi spaces with boundary or corners, and define weak 2-functors $F_\Manb^\Kurb:\Manb\ra\Kurb$ and $F_\Manc^\Kurc:\Manc\ra\Kurc$. As in Definitions \ref{ku4def19}-\ref{ku4def21}, we can define {\it Kuranishi neighbourhoods with boundary or corners\/ $(V_a,E_a,\Ga_a,s_a,\psi_a)$ on Kuranishi spaces with boundary or corners\/} $\bX$, and their coordinate changes and 1-morphisms, and the analogue of Theorem \ref{ku4thm4} holds.

Following the method of \S\ref{ku45} closely, we can define weak 2-categories of {\it orbifolds with boundary\/} $\Orbb$ and {\it orbifolds with corners\/} $\Orbc$, and full and faithful weak 2-functors $F_\Orbb^\Kurb:\Orbb\ra\Kurb$ and $F_\Orbc^\Kurc:\Orbc\ra\Kurc$, but we will not discuss this. The only foundational work on orbifolds with corners known to the author is \cite[\S 8]{Joyc8}, which uses $C^\iy$-stacks.

\subsection[Boundaries and corners of Kuranishi spaces with corners]{Boundaries and corners of Kuranishi spaces with \\ corners}
\label{ku52}

Section \ref{ku32} defined the boundary $\pd X$ and $k$-corners $C_k(X)$ of a manifold with corners $X$, and \S\ref{ku36} generalized these to $\mu$-Kuranishi spaces with corners. We now extend them to Kuranishi spaces with corners. 

\begin{dfn} Let $\bX=(X,\cK)$ be a Kuranishi space with corners, with $\cK=\bigl(I,(V_i,E_i,\Ga_i,s_i,\psi_i)_{i\in I}$, $\Phi_{ij,\; i,j\in I}$, $\La_{ijk,\; i,j,k\in I}\bigr)$, and fix $x\in\bX$ and $k\ge 0$. Generalizing \eq{ku3eq24}, for each $i\in I$ with $x\in\Im\bar\psi_i$ define
\begin{align*}
\Ga_{x,i}^k=\bigl\{(v_i,\ga):\,& v_i\in\bar\psi_i^{-1}(\{x\})\subseteq V_i,\\
&\text{$\ga$ is a local $k$-corner component of $V_i$ at $v_i$}\bigr\}/\Ga_i.
\end{align*}

Suppose $i,j\in I$ with $x\in\Im\bar\psi_i\cap\Im\bar\psi_j$. We have $\Phi_{ij}=(P_{ij},\pi_{ij},\phi_{ij},\hat\phi_{ij})$, where $\pi_{ij}:P_{ij}\ra V_i$, $\phi_{ij}:P_{ij}\ra V_j$, and $\Ga_i\t\Ga_j$ acts on $P_{ij}$, so we may define
\e
\begin{split}
\Ga_{x,ij}^k=\bigl\{&(p_{ij},\ga):p_{ij}\in\pi_{ij}^{-1}(\bar\psi_i^{-1}(\{x\}))\subseteq P_{ij},\\
&\text{$\ga$ is a local $k$-corner component of $P_{ij}$ at $p_{ij}$}\bigr\}/(\Ga_i\t\Ga_j).
\end{split}
\label{ku5eq3}
\e
Define projections $\Pi_{ij}^i:\Ga_{x,ij}^k\ra\Ga_{x,i}^k$ and $\Pi_{ij}^j:\Ga_{x,ij}^k\ra\Ga_{x,j}^k$ by
\e
\begin{split}
\Pi_{ij}^i&:(\Ga_i\t\Ga_j)\cdot (p_{ij},\ga)\longmapsto \Ga_i\cdot \bigl(\pi_{ij}(p_{ij}),(\pi_{ij})_*(\ga)\bigr),\\
\Pi_{ij}^j&:(\Ga_i\t\Ga_j)\cdot (p_{ij},\ga)\longmapsto \Ga_j\cdot \bigl(\phi_{ij}(p_{ij}),(\phi_{ij})_*(\ga)\bigr).
\end{split}
\label{ku5eq4}
\e
These are well-defined as $\pi_{ij}$ is $\Ga_j$-invariant and $\Ga_i$-equivariant, and $\phi_{ij}$ is $\Ga_i$-invariant and $\Ga_j$-invariant, and $\pi_{ij},\phi_{ij}$ are simple and so map local $k$-corner components $\ga$ to local $k$-corner components $(\pi_{ij})_*(\ga),(\phi_{ij})_*(\ga)$.

Using Definition \ref{ku4def3}(b), equation \eq{ku5eq2} an isomorphism by Theorem \ref{ku5thm2}, and $\pi_{ij},\phi_{ij}$ simple, we see that $\Pi_{ij}^i,\Pi_{ij}^j$ in \eq{ku5eq4} are bijections. Thus, in a similar way to Definition \ref{ku3def14}, we may define a bijection
\begin{equation*}
(\Phi_{ij})_*:\Pi_{ij}^j\ci(\Pi_{ij}^i)^{-1}:\Ga_{x,i}^k\longra \Ga_{x,j}^k.
\end{equation*}

If $h,i,j\in I$ with $x\in\Im\bar\psi_h\!\cap\!\Im\bar\psi_i\!\cap\!\Im\bar\psi_j$, using Definition \ref{ku4def4} we see that
\begin{equation*}
(\Phi_{ij}\ci\Phi_{hi})_*=(\Phi_{ij})_*\ci(\Phi_{hi})_*
:\Ga_{x,h}^k\longra \Ga_{x,j}^k.
\end{equation*}
And using $\La_{hij}=[\dot P_{hij},\la_{hij},\hat\la_{hij}]:\Phi_{ij}\ci\Phi_{hi}\Ra\Phi_{hj}$, we find that
\begin{equation*}
(\Phi_{ij}\ci\Phi_{hi})_*=(\Phi_{hk})_*:\Ga_{x,h}^k\longra\Ga_{x,j}^k,
\end{equation*}
since $\la_{hij}$ identifies the sets \eq{ku5eq3} for $\Phi_{ij}\ci\Phi_{hi}$ and $\Phi_{hj}$ compatibly with the projections to $\Ga_{x,h}^k,\Ga_{x,j}^k$. Thus as for \eq{ku3eq25} we have
\begin{equation*}
(\Phi_{hj})_*=(\Phi_{ij})_*\ci(\Phi_{hi})_*:\Ga_{x,h}^k\longra \Ga_{x,j}^k.
\end{equation*}
Also $(\Phi_{ii})_*=\id:\Ga_{x,i}^k\longra \Ga_{x,i}^k$.

Define a {\it local\/ $k$-corner component of\/ $\bX$ at\/} $x$ to be an element of the set
\begin{equation*}
\Ga_{x,\bX}^k=\bigl(\ts\coprod_{i\in I:x\in\Im\psi_i}\Ga_{x,i}^k\bigr)\big/\approx,
\end{equation*}
as in \eq{ku3eq26}, where $\approx$ is the binary relation $\ga_i\approx \ga_j$ if $\ga_i\in \Ga_{x,i}^k$ and $\ga_i\in \Ga_{x,i}^k$ with $(\Phi_{ij})_*(\ga_i)=\ga_j$. The discussion above implies that $\approx$ is an equivalence relation, and that there are canonical bijections $\Ga_{x,i}^k\cong\Ga_{x,\bX}^k$ for each $i\in I$ with $x\in\Im\bar\psi_i$. When $k=1$, we also call local 1-corner components of $\bX$ at $x$ {\it local boundary components of\/ $\bX$ at\/}~$x$.

For each $k=0,1,\ldots$ we will define a Kuranishi space with corners $C_k(\bX)=\bigl(C_k(X),\cK^{C_k}\bigr)$ called the $k$-{\it corners of\/} $\bX$, with virtual dimension $\vdim C_k(\bX)=\vdim\bX-k$, and a morphism $\bs i_\bX:C_k(\bX)\ra\bX$. When $k=1$, we also write $\pd\bX=C_1(\bX)$ and $\bs i_\bX:\pd\bX\ra\bX$, and call $\pd\bX$ the {\it boundary\/} of $\bX$. As in \eq{ku3eq27}, the underlying topological space $C_k(X)$ is given as a set by
\begin{equation*}
C_k(X)=\bigl\{(x,\ga):\text{$x\in X$, $\ga$ is a local $k$-corner 
component of $\bX$ at $x$}\bigr\}.
\end{equation*}
Note that the analogues of \eq{ku3eq5} and \eq{ku3eq28} may not hold in the orbifold case. The topology on $C_k(X)$ is determined uniquely by requiring $\psi_i^{C_k}$ in \eq{ku5eq5} below to be a homeomorphism with an open set in $C_k(X)$ for all~$i\in I$.

Define a Kuranishi structure with corners $\cK^{C_k}$ on $C_k(X)$ by
\e
\cK^{C_k}=\bigl(I,(V_i^{C_k},E_i^{C_k},\Ga_i^{C_k},s_i^{C_k},\psi_i^{C_k})_{i\in I},\;\Phi_{ij,\; i,j\in I}^{C_k},\;\La_{hij,\; h,i,j\in I}^{C_k}\bigr),
\label{ku5eq5}
\e
where for each $i\in I$ we have $V_i^{C_k}=C_k(V_i)$, $E_i^{C_k}=i_{V_i}^*(E_i)$ for $i_{V_i}:C_k(V_i)\ra V_i$ the projection, $\Ga_i^{C_k}=\Ga_i$, $s_i^{C_k}=i_{V_i}^*(s_i)$, and $\psi_i^{C_k}:(s_i^{C_k})^{-1}(0)/\Ga_i^{C_k}\ra C_k(X)$ maps $\psi_i^{C_k}:\Ga_i^{C_k}\cdot(v_i,\ga)\mapsto[\Ga_i\cdot(v_i,\ga)]$, writing $[\Ga_i\cdot(v_i,\ga)]\in\Ga_{x,\bX}^k$ for the $\approx$-equivalence class of $\Ga_i\cdot(v_i,\ga)\in \Ga_{x,i}^k$, with $x=\bar\psi_i(v_i)$. Since the map $\Ga_{x,i}^k\ra\Ga_{x,\bX}^k$ taking $\Ga_i\cdot(v_i,\ga)\mapsto[\Ga_i\cdot(v_i,\ga)]$ is bijective, we see that $\psi_i^{C_k}$ is injective.

If $i,j\in I$, define $\Phi_{ij}^{C_k}=\bigl(P_{ij}^{C_k},\pi_{ij}^{C_k},\phi_{ij}^{C_k},\hat\phi_{ij}^{C_k}\bigr)$, where $P_{ij}^{C_k}=C_k(P_{ij})$, and $\pi_{ij}^{C_k}=C_k(\pi_{ij}):P_{ij}^{C_k}=C_k(P_{ij})\ra C_k(V_i)=V_i^{C_k}$, and $\phi_{ij}^{C_k}=C_k(\phi_{ij}):P_{ij}^{C_k}=C_k(P_{ij})\ra C_k(V_j)=V_j^{C_k}$ (both defined by Proposition \ref{ku3prop1} as $\pi_{ij},\phi_{ij}$ are simple by Proposition \ref{ku5prop1}), and
\begin{align*}
\hat\phi_{ij}^{C_k}=i_{P_{ij}}^*&(\hat\phi_{ij}):(\pi_{ij}^{C_k})^*(E_i^{C_k})=C_k(\pi_{ij})^*\ci i_{V_i}^*(E_i)=i_{P_{ij}}^*\ci \pi_{ij}^*(E_i) \\
&\longra i_{P_{ij}}^*\ci\phi_{ij}^*(E_j)=C_k(\phi_{ij})^*\ci i_{V_j}^*(E_j)=(\phi_{ij}^{C_k})^*(E_j^{C_k}),
\end{align*}
where we use $\pi_{ij}\ci i_{P_{ij}}=i_{V_i}\ci C_k(\pi_{ij})$ and~$\phi_{ij}\ci i_{P_{ij}}=i_{V_j}\ci C_k(\phi_{ij})$.

If $h,i,j\in I$, write $\Phi_{ij}\ci\Phi_{hi}=(P_{hij},\pi_{hij},\phi_{hij},\hat\phi_{hij})$ and $\Phi_{ij}^{C_k}\ci\Phi_{hi}^{C_k}=(P_{hij}^{C_k},\pi_{hij}^{C_k},\phi_{hij}^{C_k},\hat\phi_{hij}^{C_k})$ so that $P_{hij}^{C_k}=C_k(P_{hij})$, $\pi_{hij}^{C_k}=C_k(\pi_{hij})$ and $\phi_{hij}^{C_k}=C_k(\phi_{hij})$. Let $(\dot P_{hij},\la_{hij},\hat\la_{hij})$ represent $\La_{hij}=[\dot P_{hij},\la_{hij},\hat\la_{hij}]:\Phi_{ij}\ci\Phi_{hi}\Ra\Phi_{hk}$, so that $\dot P_{hij}\subseteq P_{hij}$ is open. Define $\La_{hij}^{C_k}=[\dot P_{hij}^{C_k},\la_{hij}^{C_k},\hat\la_{hij}^{C_k}]:\Phi_{ij}^{C_k}\ci\Phi_{hi}^{C_k}\Ra\Phi_{hk}^{C_k}$, where $\dot P_{hij}^{C_k}=C_k(\dot P_{hij})$, $\la_{hij}^{C_k}=C_k(\la_{hij})$, and $\hat\la_{hij}^{C_k}$ is determined by the commutative diagram
\e
\begin{gathered}
{}\!\!\!\!\!\!\!\!\xymatrix@!0@C=289pt@R=55pt{ *+[r]{\,\,\,\,\begin{subarray}{l}\ts 
(\pi_{hij}^{C_k})^*(E_h^{C_k})\vert_{\dot P_{hij}^{C_k}}= \\
\ts C_k(\pi_{hij})^*\!\ci\! i_{V_h}^*(E_h)\vert_{\dot P_{hij}^{C_k}}\!=\!i_{\dot P_{hij}}^*\!\!\ci\!\pi_{hij}^*(E_h) \end{subarray}\!\!\!\!\!\!\!\!\!\!\!\!\!\!\!\!\!\! } \ar@<2.3ex>[r]^(0.55){\hat\la_{hij}^{C_k}} \ar[d]^(0.6){i_{\dot P_{hij}}^*(\hat\la_{hij})} & *+[l]{\begin{subarray}{l} \ts (\phi^{C_k}_{hij})^*({}^bTV^{C_k}_j)\vert_{\dot P_{hij}^{C_k}}= \\ \ts C_k(\phi_{hij})^*({}^bT(C_k(V_j)))\vert_{\dot P_{hij}^{C_k}} \end{subarray}\,\,\,\,} \\
*+[r]{i_{\dot P_{hij}}^*\ci \phi_{hij}^*({}^bTV_j)} \ar@{=}[r] & *+[l]{C_k(\phi_{hij})^*\ci i_{V_j}^*({}^bTV_j),\!\!} \ar[u]^(0.4){C_k(\phi_{hij})^*(\Pi)} }\!\!\!\!\!\!\!\!\!{}
\end{gathered}
\label{ku5eq6}
\e

where $\Pi:i_{V_j}^*({}^bTV_j)\ra {}^bT(C_k(V_j))$ is the natural projection. One can show that $\La_{hij}^{C_k}$ is well-defined, and independent of the choice of $(\dot P_{hij},\la_{hij},\hat\la_{hij})$.

This defines all the data in $\cK^{C_k}$ in \eq{ku5eq5}. As in \S\ref{ku36}, it is now straightforward to check that $\cK^{C_k}$ is a Kuranishi structure with corners on $C_k(X)$ (with its induced topology), so that $C_k(\bX)=\bigl(C_k(X),\cK^{C_k}\bigr)$ is a Kuranishi space with corners with $\vdim C_k(\bX)=\vdim\bX-k$, and $\pd\bX=C_1(\bX)$ when~$k=1$.

Define a 1-morphism $\bs i_\bX:C_k(\bX)\ra\bX$ by 
\begin{equation*}
\bs i_\bX=\bigl(i_\bX,(\bs i_\bX)_{ij,\;i\in I,\; j\in I},\; (\bs I_\bX)_{ii',\;i,i'\in I}^{j,\; j\in I},\; (\bs I_\bX)_{i,\;i\in I}^{jj',\; j,j'\in I}\bigr),
\end{equation*}
where $i_\bX:(x,\ga)\mapsto x$, and $(\bs i_\bX)_{ij}=\bigl(C_k(P_{ij}),C_k(\pi_{ij}),\phi_{ij}\ci i_{P_{ij}},\hat\phi_{ij}\ci i_{P_{ij}}^*\bigr)$, and $(\bs I_\bX)_{ii'}^j=\bigl[C_k(\dot P_{ii'j}),C_k(\la_{ii'j}),\hat\la_{ii'}^{C_k\,j}\bigr]$, $(\bs I_\bX)_i^{jj'}=\bigl[C_k(\dot P_{ijj'}),C_k(\la_{ijj'}),\hat\la_i^{C_k\,jj}\bigr]$, where $(\dot P_{ii'j},\la_{ii'j},\hat\la_{ii'j})$ represents $\La_{ii'j}$, and $(\dot P_{ijj'},\la_{ijj'},\hat\la_{ijj'})$ represents $\La_{ijj'}$, and $\hat\la_{ii'}^{C_k\,j},\hat\la_i^{C_k\,jj}$ are defined using $\hat\la_{ii'j},\hat\la_{ijj'}$ in a similar way to \eq{ku5eq6}. It is not difficult to show that $(\bs I_\bX)_{ii'}^j,(\bs I_\bX)_i^{jj'}$ are independent of the choice of representatives for $\La_{ii'j},\La_{ijj'}$, and $\bs i_\bX:C_k(\bX)\ra\bX$ is a 1-morphism in~$\Kur$.

\label{ku5def4}
\end{dfn}

Here are the analogues of Definitions \ref{ku3def15}, \ref{ku3def16} and \ref{ku3def17}:

\begin{dfn} Define a {\it Kuranishi space with corners $\bX=(X,\cK)$ of mixed dimension\/} with $\cK=\bigl(I,(V_i,E_i,\Ga_i,s_i,\psi_i)_{i\in I}$, $\Phi_{ij,\; i,j\in I}$, $\La_{ijk,\; i,j,k\in I}\bigr)$ as for Kuranishi spaces with corners in \S\ref{ku43} and \S\ref{ku51}, but omitting the condition $\dim V_i-\rank E_i=n$ in Definition \ref{ku4def12}(b), allowing $\dim V_i-\rank E_i$ to take any value.

Write $I_m=\{i\in I:\dim V_i-\rank E_i=m\}$ for $m\in\Z$, so that $I=
\coprod_{m\in\Z}I_m$, and $X_m=\bigcup_{i\in I_m}\Im\psi_i$, so that $X_m$ is open and closed in $X$ with $X\!=\!\coprod_{m\in\Z}X_m$, allowing $I_m\!=\!X_m\!=\!\es$. Set $\cK_m\!=\!\bigl(I_m,(V_i,E_i,\Ga_i,s_i,\psi_i)_{i\in I_m}$, $\Phi_{ij,\; i,j\in I_m}$, $\La_{ijk,\; i,j,k\in I_m}\bigr)$ for $m\in\Z$. Then $\bX_m=(X_m,\cK_m)$ is a Kuranishi space with corners with $\vdim\bX_m=m$ for $m\in\Z$, and $\bX=\coprod_{m\in\Z}\bX_m$.

Define 1-{\it morphisms\/} $\bs f:\bX\ra\bY$ and 2-{\it morphisms\/} $\bs\eta:\bs f\Ra\bs g$ of Kuranishi spaces with corners of mixed dimension, and all the other structures of a weak 2-category, as for Kuranishi spaces with corners in \S\ref{ku43} and \S\ref{ku51}, without change. Then Kuranishi spaces with corners of mixed dimension form a weak 2-category $\cKurc$, and $\Kurc\subset\cKurc$ is a full 2-subcategory.
\label{ku5def5}
\end{dfn}

\begin{dfn} We will define a weak 2-functor $C:\Kurc\ra\cKurc$, called the {\it corner functor}. On objects $\bX\!=\!(X,\cK)$ with $\cK=\bigl(I,(V_i,E_i,\Ga_i,s_i,\psi_i)_{i\in I}$, $\Phi_{ij,\; i,j\in I}$, $\La_{ijk,\; i,j,k\in I}\bigr)$ in $\Kurc$, define
\begin{equation*}
C(\bX)=\ts\coprod_{k=0}^\iy C_k(\bX),
\end{equation*}
regarded as an object in $\cKurc$. Write the indexing set for the $\mu$-Kuranishi structure on $C(\bX)$ as $I\t\N$, where $(i,k)\in I\t\N$ corresponds to the $\mu$-Kuranishi neighbourhood $(V^{C_k}_i,E_i^{C_k},\Ga_i^{C_k},s_i^{C_k},\psi_i^{C_k})$ on $C_k(X)$ in \eq{ku5eq5}.

In Definition \ref{ku3def5}, given a smooth map $f:X\ra Y$ of manifolds with corners, a point $x\in X$ with $f(x)\in Y$, and a local $k$-corner component $\ga$ of $X$ at $x$, we defined a local $l$-corner component $f_*(\ga)$ of $Y$ at $y$ for some $l\ge 0$. We will now define the analogue for 1-morphisms of Kuranishi spaces with corners.

Let $\bs f:\bX\ra\bY$ be a 1-morphism in $\Kurc$. Use notation \eq{ku4eq12}--\eq{ku4eq13} for $\bX,\bY$ and \eq{ku4eq15} for $\bs f$, where $\bs f_{ij}=(P_{ij},\pi_{ij},f_{ij},\hat f_{ij})$, and $(\dot P_{ii'}^j,F_{ii'}^j,\hat F_{ii'}^j)$, $(\dot P_i^{jj'},F_i^{jj'},\hat F_i^{jj'})$ are representatives for $\bs F_{ii'}^j,\bs F_i^{jj'}$, as in Definition \ref{ku4def13}.

Let $x\in\bX$ with $y=\bs f(x)\in\bY$. Choose $i\in I$, $j\in J$ with $x\in\Im\bar\chi_i\subseteq X$ and $y\in\Im\bar\psi_j\subseteq Y$. Then by Definition \ref{ku5def4} we have $\Ga_{x,\bX}^k\cong\Ga_{x,i}^k$ and $\Ga_{y,\bY}^l\cong\Ga_{y,j}^l$ for $k,l\ge 0$. As for \eq{ku5eq3}--\eq{ku5eq4}, for $k\ge 0$ we define
\begin{align*}
\Ga_{x,\bs f,ij}^k=\bigl\{&(p_{ij},\ga):p_{ij}\in\pi_{ij}^{-1}(\bar\chi_i^{-1}(\{x\}))\subseteq P_{ij},\\
&\text{$\ga$ is a local $k$-corner component of $P_{ij}$ at $p_{ij}$}\bigr\}/(\Be_i\t\Ga_j).
\end{align*}
Define maps $\Pi_{ij}^i:\Ga_{x,\bs f,ij}^k\ra\Ga_{x,i}^k$ and $\Pi_{ij}^j:\Ga_{x,\bs f,ij}^k\ra\coprod_{l\ge 0}\Ga_{y,j}^l$ by
\begin{align*}
\Pi_{ij}^i&:(\Be_i\t\Ga_j)\cdot (p_{ij},\ga)\longmapsto \Be_i\cdot \bigl(\pi_{ij}(p_{ij}),(\pi_{ij})_*(\ga)\bigr),\\
\Pi_{ij}^j&:(\Be_i\t\Ga_j)\cdot (p_{ij},\ga)\longmapsto \Ga_j\cdot \bigl(f_{ij}(p_{ij}),(f_{ij})_*(\ga)\bigr).
\end{align*}
As in Definition \ref{ku5def4}, $\Pi_{ij}^i$ is a bijection, since $\pi_{ij}:P_{ij}\ra U_i$ is a principal $\Ga_j$-bundle over an open neighbourhood of $\bar\chi_i^{-1}(x)$. However, as $\bs f_{ij}$ need not be a coordinate change, $f_{ij}$ need not be simple, so $(f_{ij})_*(\ga)$ can be a local $l$-corner component of $V_j$ for any $l\ge 0$, and $\Pi_{ij}^j$ need not be a bijection.

Define a map $\Ga_{x,\bX}^k\ra\coprod_{l\ge 0}\Ga_{y,\bY}^l$, written $\ga\mapsto\bs f_*(\ga)$, to be the composition
\begin{equation*}
\smash{\xymatrix@C=37pt{\Ga_{x,\bX}^k \ar[r]^\cong & \Ga_{x,i}^k \ar[r]^(0.45){(\Pi_{ij}^i)^{-1}} & \Ga_{x,\bs f,ij}^k \ar[r]^(0.45){\Pi_{ij}^j} & \coprod_{l\ge 0}\Ga_{y,j}^l \ar[r]^(0.45)\cong & \coprod_{l\ge 0}\Ga_{y,\bY}^l. }}
\end{equation*}
The arguments of Definition \ref{ku5def4} imply this is independent of the choice of $i,j$. Now define a map of topological spaces $C(f):C(X)\ra C(Y)$ by $C(f):(x,\ga)\mapsto (\bs f(x),\bs f_*(\ga))$. Since the topologies on $C(X),C(Y)$ are induced by their Kuranishi neighbourhoods, and we will give 1-morphisms $C(\bs f)_{(i,k)(j,l)}$ of these Kuranishi neighbourhoods below compatible with $C(f)$, this $C(f)$ is continuous. 

For all $(i,k)\in I\t\N$ and $(j,l)\in J\t\N$, define a 1-morphism of Kuranishi neighbourhoods $C(\bs f)_{(i,k)(j,l)}:(U^{C_k}_i,D_i^{C_k},\Be_i^{C_k},r_i^{C_k},\chi_i^{C_k})\ra(V^{C_l}_j,E_j^{C_l},\ab\Ga_j^{C_l},\ab s_j^{C_l},\ab\psi_j^{C_l})$ over $C(f)$ by 
\begin{equation*}
C(\bs f)_{(i,k)(j,l)}=\bigl(C_k(P_{ij})\cap C(f_{ij})^{-1}(C_l(V_j)),C(\pi_{ij})\vert_{\cdots},C(f_{ij})\vert_{\cdots},i_{P_{ij}}^*(\hat f_{ij})\vert_{\cdots}\bigr).
\end{equation*}
Define 2-morphisms $C(\bs F)_{(i,k)(i',k)}^{(j,l)}:C(\bs f)_{(i',k)(j,l)}\ci\Tau_{ii}^{C_k}\Ra C(\bs f)_{(i,k)(j,l)}$ and $C(\bs F)_{(i,k)}^{(j,l)(j',l)}:\Up_{jj'}^{C_l}\ci C(\bs f)_{(i,k)(j,l)}\Ra C(\bs f)_{(i,k)(j',l)}$ over $C(f)$ by
\begin{align*}
C(\bs F)_{(i,k)(i',k)}^{(j,l)}=\bigl[C_k(\dot P_{ii'}^j)\cap C(f_{ii'j})^{-1}(C_l(V_j)),C(F_{ii'}^j)\vert_{\cdots}&,\\
C(f_{ii'j})^*(\Pi)\ci i_{\dot P_{ii'}^j}^*(\hat F_{ii'}^j)\vert_{\cdots}& \bigr],\\
C(\bs F)_{(i,k)}^{(j,l)(j',l)}=\bigl[C_k(\dot P_i^{jj'})\cap C(f_{ijj'})^{-1}(C_l(V_{j'})),C(F_i^{jj'})\vert_{\cdots}&,\\
C(f_{ijj'})^*(\Pi)\ci i_{\dot P_i^{jj'}}^*(\hat F_i^{jj'})\vert_{\cdots}& \bigr].
\end{align*}
The 2-morphisms $\bs F_{(i,k)(i',k')}^{(j,l)}$ for $k\ne k'$ are trivial as the footprints of $(U^{C_k}_i,\ab D_i^{C_k},\ab\Be_i^{C_k},\ab r_i^{C_k},\ab\chi_i^{C_k})$ and $(U^{C_{k'}}_{i'},D_{i'}^{C_{k'}},\Be_{i'}^{C_{k'}},r_{i'}^{C_{k'}},\chi_{i'}^{C_{k'}})$ do not intersect in $C(\bX)$, and similarly for $\bs F_{(i,k)}^{(j,l)(j',l')}$ for $l\ne l'$. Using the functoriality of $C:\Manc\ra\cManc$, one can now show that
\begin{align*}
C(\bs f)=\bigl(&C(f),C(\bs f)_{(i,k)(j,l),\;(i,k)\in I\t\N,\; (j,l)\in J\t\N},\\ 
&C(\bs F)_{(i,k)(i',k'),\;(i,k),(i',k')\in I\t\N}^{(j,l),\; (j,l)\in J\t\N},\; C(\bs F)_{(i,k),\;(i,k)\in I\t\N}^{(j,l)(j',l'),\; (j,l),(j',l')\in J\t\N}\bigr)
\end{align*}
is a 1-morphism $C(\bs f):C(\bX)\ra C(\bY)$ in~$\cKurc$.

Similarly, if $\bs f,\bs g:\bX\ra\bY$ are 1-morphisms and $\bs\eta:\bs f\Ra\bs g$ a 2-morphism in $\Kurc$, one can show that $C(f)=C(g):C(X)\ra C(Y)$ on topological spaces, and define a natural 2-morphism $C(\bs\eta):C(\bs f)\Ra C(\bs g)$ in $\cKurc$. One can also define the remaining structures of a weak 2-functor $C:\Kurc\ra\cKurc$, which we call the {\it corner $2$-functor}. Verifying this is tedious but straightforward, and we leave the details to the reader.
\label{ku5def6}
\end{dfn}

\begin{dfn} Let $\bs f:\bX\ra\bY$ be a 1-morphism of Kuranishi spaces with corners. Use notation \eq{ku4eq12}--\eq{ku4eq13} for $\bX,\bY$ and \eq{ku4eq15} for $\bs f$, where $\bs f_{ij}=(P_{ij},\pi_{ij},f_{ij},\hat f_{ij})$. We say that $\bs f$ is {\it interior}, or {\it b-normal}, or {\it strongly smooth}, or {\it simple}, if for all $i\in I$ and $j\in J$, the smooth map $f_{ij}:P_{ij}\ra V_j$ of manifolds with corners is interior, \ldots, simple, respectively.

Note that as interior, b-normal, strongly smooth, and simple are {\it discrete\/} conditions on smooth maps of manifolds with corners in the sense of Remark \ref{ku3rem2}, Remarks \ref{ku3rem4}(a), \ref{ku3rem5} and Definition \ref{ku5def2}(g) imply that $f_{ij}$ is interior, \ldots, simple on all of $P_{ij}$ if and only if $f_{ij}$ is interior, \ldots, simple near each point $p_{ij}$ in $\pi_{ij}^{-1}(\bar\chi_i^{-1}(\Im\psi_j))\subseteq P_{ij}$. Furthermore, these are really local conditions near each $x\in\bX$, and are independent of the choice of $i\in I$ with $x\in\Im\bar\chi_i$, $j\in J$ with $\bs f(x)\in\Im\bar\psi_j$ and $p_{ij}\in\pi_{ij}^{-1}(\bar\chi_i^{-1}(x))$ that we check the conditions for.

The classes of interior, b-normal, strongly smooth, and simple 1-morphisms of Kuranishi spaces with corners are each closed under composition and contain identities, as this holds for the corresponding classes of morphisms of manifolds with corners. Thus, each class defines 2-subcategories of $\Kurc,\cKurc$. Write $\Kurcin,\cKurcin$ and $\Kurcsi,\cKurcsi$ for the 2-subcategories of $\Kurc,\ab\cKurc$ with interior 1-morphisms, and with simple 1-morphisms, respectively.
\label{ku5def7}
\end{dfn}

As for Propositions \ref{ku3prop1} and \ref{ku3prop3}, we now have:

\begin{prop} Let\/ $\bs f:\bX\ra\bY$ be a $1$-morphism in $\Kurc$. 
\begin{itemize}
\setlength{\itemsep}{0pt}
\setlength{\parsep}{0pt}
\item[{\bf(a)}] $C(\bs f):C(\bX)\ra C(\bY)$ is an interior\/ $1$-morphism of Kuranishi spaces with corners of mixed dimension, so $C$ is a weak\/ $2$-functor $C:\Kurc\ra\cKurcin$.
\item[{\bf(b)}] $\bs f$ is interior if and only if\/ $C(\bs f)$ maps $C_0(\bX)\ra C_0(\bY)$.
\item[{\bf(c)}] $\bs f$ is b-normal if and only if\/ $C(\bs f)$ maps $C_k(\bX)\ra \coprod_{l=0}^kC_l(\bY)$ for all\/~$k$.
\item[{\bf(d)}] If\/ $\bs f$ is simple then $C(\bs f)$ maps $C_k(\bX)\ra C_k(\bY)$ for all\/ $k\ge 0,$ and\/ $C_k(\bs f):=C(\bs f)\vert_{C_k(\bX)}:C_k(\bX)\ra C_k(\bY)$ is also a simple $1$-morphism.
\end{itemize}

Part\/ {\bf(d)} implies that we have a \begin{bfseries}boundary $2$-functor\end{bfseries} $\pd:\Kurcsi\!\ra\!\Kurcsi$ mapping $\bX\mapsto\pd\bX$ on objects, $\bs f\mapsto\pd\bs f:=C(\bs f)\vert_{C_1(\bX)}$ on (simple) $1$-morphisms $\bs f:\bX\ra\bY$ and\/ $\bs\eta\mapsto \pd\bs\eta:=C(\bs\eta)\vert_{C_1(\bX)}$ on $2$-morphisms $\bs\eta:\bs f\Ra\bs g$ of\/ $\bs f,\bs g:\bX\ra\bY$. Similarly, for all\/ $k\ge 0$ we have a \begin{bfseries}$k$-corner functor\end{bfseries} $C_k:\muKurcsi\ra\muKurcsi$ mapping $\bX\mapsto C_k(\bX)$ on objects, $\bs f\mapsto C_k(\bs f):=C(\bs f)\vert_{C_k(\bX)}$ on $1$-morphisms, and\/ $\bs\eta\mapsto C_k(\bs\eta):=C(\bs\eta)\vert_{C_k(\bX)}$ on $2$-morphisms.
\label{ku5prop2}
\end{prop}

\subsection[Isotropy groups, and (b-)tangent and (b-)obstruction spaces]{Isotropy groups, and (b-)tangent and (b-)obstruction \\ spaces}
\label{ku53}

Section \ref{ku25} we defined tangent spaces $T_x\bX$ and obstruction spaces $O_x\bX$ of a $\mu$-Kuranishi space $\bX$, and \S\ref{ku37} extended this to (b-)tangent spaces $T_x\bX,{}^bT_x\bX$ and (b-)obstruction spaces $O_x\bX,{}^bO_x\bX$ of a $\mu$-Kuranishi space with corners $\bX$. Section \ref{ku46} defined isotropy groups $G_x\bX$, tangent spaces $T_x\bX$ and obstruction spaces $O_x\bX$ of a Kuranishi space~$\bX$.

For Kuranishi spaces with corners $\bX$, we define isotropy groups $G_x\bX$, \hbox{(b-)}\ab tan\-gent spaces $T_x\bX,{}^bT_x\bX$ and (b-)obstruction spaces $O_x\bX,{}^bO_x\bX$, with canonicality and functoriality properties as in \S\ref{ku46}. The ideas of \S\ref{ku25}, \S\ref{ku37} and \S\ref{ku46} are combined in such an obvious way that we will give few details.

\begin{dfn} Let $\bX=(X,\cK)$ be a Kuranishi space with corners, with $\cK=\bigl(I,(V_i,\ab E_i,\ab\Ga_i,\ab s_i,\ab\psi_i)_{i\in I}$, $\Phi_{ij,\;i,j\in I}$, $\La_{ijk,\; i,j,k\in I}\bigr)$, and let $x\in\bX$.

Choose an arbitrary $i\in I$ with $x\in\Im\psi_i$, and choose $v_i\in s_i^{-1}(0)\subseteq V_i$ with $\bar\psi_i(v_i)=x$. Our definitions of $G_x\bX,T_x\bX,O_x\bX,{}^bT_x\bX,{}^bO_x\bX$ will depend on these choices, and we explain shortly to what extent.

As in \eq{ku4eq41}, define a finite subgroup $G_x\bX$ of $\Ga_i$ by
\begin{equation*}
G_x\bX=\bigl\{\ga\in\Ga_i:\ga\cdot v_i=v_i\bigr\}=\Stab_{\Ga_i}(v_i).
\end{equation*}
As in \eq{ku2eq30}, \eq{ku3eq33}, \eq{ku3eq36}, \eq{ku3eq37} and \eq{ku4eq42}, define finite-dimensional real vector spaces $T_x\bX,O_x\bX,{}^bT_x\bX,{}^bO_x\bX$ and linear maps $I^T_{\bX,x}:{}^bT_x\bX\ra T_x\bX$ and $I^O_{\bX,x}:{}^bO_x\bX\ra O_x\bX$ by the commutative diagram with exact rows
\begin{equation*}
\xymatrix@C=20pt@R=13pt{ 0 \ar[r] & {}^bT_x\bX \ar[d]^{I^T_{\bX,x}} \ar[r] & {}^bT_{v_i}V_i \ar[rr]_{{}^b\d s_i\vert_{v_i}} \ar[d]^{I_{V_i}\vert_{v_i}} && E_i\vert_{v_i} \ar[d]^{\id} \ar[r] & {}^bO_x\bX \ar[d]^{I^O_{\bX,x}} \ar[r] & 0 \\
0 \ar[r] & T_x\bX \ar[r] & T_{v_i}V_i \ar[rr]^{\d s_i\vert_{v_i}} && E_i\vert_{v_i} \ar[r] & O_x\bX \ar[r] & 0.\!\!{} }
\end{equation*}
As in equations \eq{ku2eq29}, \eq{ku3eq32}, \eq{ku3eq35} and \eq{ku4eq43}, we have
\begin{equation*}
\dim T_x\bX-\dim O_x\bX=\dim {}^bT_x\bX-\dim {}^bO_x\bX=\vdim\bX.
\end{equation*}

As in \S\ref{ku46}, the quintuple $(G_x\bX,T_x\bX,O_x\bX,{}^bT_x\bX,{}^bO_x\bX)$ is independent of the choice of $i,v_i$ up to isomorphism, but not up to canonical isomorphism. If $(G_x'\bX,T_x'\bX,O_x'\bX,{}^bT_x'\bX,{}^bO_x'\bX)$ is an alternative choice, there are natural choices of isomorphisms $(I^G_x,I^T_x,I^O_x,{}^bI^T_x,{}^bI^O_x):(G_x\bX,T_x\bX,O_x\bX,\ab{}^bT_x\bX,\ab{}^bO_x\bX)\ab\ra\ab(G_x'\bX,T_x'\bX,O_x'\bX,{}^bT_x'\bX,{}^bO_x'\bX)$, and two such isomorphisms differ by conjugation by an element of $G_x'\bX$.

Let $\bs f:\bX\ra\bY$ be a 1-morphism of Kuranishi spaces with corners, and let $x\in\bX$ with $\bs f(x)=y\in\bY$. Then as in Definition \ref{ku4def27}, by choosing an arbitrary point $p_0$ in the $S_{x,\bs f}$ of \eq{ku4eq45}, we define morphisms $G_x\bs f:G_x\bX\ra G_y\bY$, $T_x\bs f:T_x\bX\ra T_y\bY$, $O_x\bs f:O_x\bX\ra O_y\bY$. An alternative choice of $p_0$ yields a triple $(G_x'\bs f,T_x'\bs f,O_x'\bs f)$ conjugate to $(G_x\bs f,T_x\bs f,O_x\bs f)$ by an element of~$G_y\bY$.

As in \S\ref{ku37}, to define analogues of $T_x\bs f,O_x\bs f$ for ${}^bT_x\bX,{}^bT_y\bY,{}^bO_x\bX,{}^bO_y\bY$, we must assume $\bs f:\bX\ra\bY$ is interior, as in Definition \ref{ku5def7}. Then, using the same point $p_0$ used to define $T_x\bs f,O_x\bs f$, as in \eq{ku3eq38} and \eq{ku4eq46} we obtain $G_x\bs f$-equivariant linear maps ${}^bT_x\bs f:{}^bT_x\bX\ra {}^bT_y\bY$ and ${}^bO_x\bs f:{}^bO_x\bX\ra {}^bO_y\bY$ making the following diagram commute:
\begin{equation*}
\xymatrix@C=23pt@R=15pt{ 0 \ar[r] & {}^bT_x\bX \ar[r] \ar[d]^{{}^bT_x\bs f} & {}^bT_{u_i}U_i \ar[rr]_(0.6){{}^b\d r_i\vert_{u_i}} \ar[d]^{{}^b\d f_{ij}\vert_{u_i}} && D_i\vert_{u_i} \ar[d]^{\hat f_{ij}\vert_{u_i}} \ar[r] & {}^bO_x\bX \ar[d]^{{}^bO_x\bs f} \ar[r] & 0 \\
0 \ar[r] & {}^bT_y\bY \ar[r] & {}^bT_{v_j}V_j \ar[rr]^(0.6){{}^b\d s_j\vert_{v_j}} && E_j\vert_{v_j} \ar[r] & {}^bO_y\bY \ar[r] & 0.\!\! }
\end{equation*}
The quintuple $(G_x\bs f,T_x\bs f,O_x\bs f,{}^bT_x\bs f,{}^bO_x\bs f)$ is independent of choices up to conjugation by an element of $G_y\bY$.

As in Definition \ref{ku3def18} we have $T_x\bs f\ci I_{\bX,x}^T=I_{\bY,y}^T\ci{}^bT_x\bs f:{}^bT_x\bX\ra T_y\bY$ and $O_x\bs f\ci I_{\bX,x}^O=I_{\bY,y}^O\ci{}^bO_x\bs f:{}^bO_x\bX\ra O_y\bY$.

As in \S\ref{ku46}, if $\bs g:\bX\ra\bY$ is another 1-morphism and $\bs\eta:\bs f\Ra\bs g$ a 2-morphism in $\Kurc$, then there is a canonical element $G_x\bs\eta\in G_y\bY$ such that $(G_x\bs f,T_x\bs f,O_x\bs f)$ and $(G_x\bs g,T_x\bs g,O_x\bs g)$ are conjugate under $G_x\bs\eta$ for general $\bs f,\bs g$, and $(G_x\bs f,T_x\bs f,O_x\bs f,{}^bT_x\bs f,{}^bO_x\bs f)$ and $(G_x\bs g,T_x\bs g,O_x\bs g,{}^bT_x\bs g,{}^bO_x\bs g)$ are conjugate under $G_x\bs\eta$ for interior~$\bs f,\bs g$.

As in \S\ref{ku46}, if instead $\bs g:\bY\ra\bZ$ is another 1-morphism in $\Kurc$ and $\bs g(y)=z\in\bZ$, there is a canonical $G_{x,\bs g,\bs f}\in G_z\bZ$ with $(G_x(\bs g\ci\bs f),T_x(\bs g\ci\bs f),O_x(\bs g\ci\bs f))$ conjugate to $(G_y\bs g,T_y\bs g,O_y\bs g)\ci (G_x\bs f,T_x\bs f,O_x\bs f)$ under $G_{x,\bs g,\bs f}$ for general $\bs f,\bs g$, and $(G_x(\bs g\ci\bs f),T_x(\bs g\ci\bs f),O_x(\bs g\ci\bs f),{}^bT_x(\bs g\ci\bs f),{}^bO_x(\bs g\ci\bs f))$ conjugate to $(G_y\bs g,T_y\bs g,O_y\bs g,{}^bT_y\bs g,{}^bO_y\bs g)\ci (G_x\bs f,T_x\bs f,O_x\bs f,{}^bT_x\bs f,{}^bO_x\bs f)$ under $G_{x,\bs g,\bs f}$ for interior~$\bs f,\bs g$.

If $\bs f:\bX\ra\bY$ is an equivalence in $\Kurc$ then $G_x\bs f,T_x\bs f,O_x\bs f,{}^bT_x\bs f,{}^bO_x\bs f$ are isomorphisms for all~$x\in\bX$.
\label{ku5def8}
\end{dfn}

\subsection[\texorpdfstring{Relating Kuranishi spaces and $\mu$-Kuranishi spaces with corners}{Relating Kuranishi spaces and \textmu-Kuranishi spaces with corners}]{Relating Kuranishi spaces and $\mu$-Kuranishi spaces \\ with corners}
\label{ku54}

The material of \S\ref{ku47} extends to the corners case essentially without change. As in Theorem \ref{ku4thm7} we define a weak 2-category $\mKurc$ of {\it m-Kuranishi spaces with corners}, which are basically Kuranishi spaces with corners $\bX=(X,\cK)$ with all groups $\Ga_i=\{1\}$ in $\cK$. There is a natural full and faithful 2-functor $F_\mKurc^\Kurc:\mKurc\ra\Kurc$. As in Definition \ref{ku4def35} we define a functor $F_\mKurc^\muKurc:\Ho(\mKurc)\ra\muKurc$, and a full 2-subcategory $\KurtrGc\subset\Kurc$ with objects $\bX\in\Kurc$ with $G_x\bX=\{1\}$ for all $x\in\bX$. As in Theorem \ref{ku4thm8} we show that $F_\mKurc^\muKurc:\Ho(\mKurc)\ra\muKurc$ is an equivalence of categories, and $F_\mKurc^\Kurc:\mKurc\ra\KurtrGc$ is an equivalence of weak 2-categories, so that we also have~$\Ho(\KurtrGc)\simeq\muKurc$.

\subsection{Relation to other definitions of Kuranishi space with corners}
\label{ku55}

Fukaya, Oh, Ohta and Ono\cite{Fuka,FOOO1,FOOO2,FOOO3, FOOO4,FOOO5,FOOO6,FOOO7,FOOO8} use (FOOO) Kuranishi spaces with corners in their work on Lagrangian Floer cohomology, since moduli spaces $\oM$ of $J$-holomorphic curves $\Si$ in a symplectic manifold $(M,\om)$ with boundary $\pd\Si$ in a Lagrangian $L\subset M$ have boundaries $\pd\oM$ where the curve $\Si$ has boundary nodes. So we should relate our notion of Kuranishi space with corners to theirs.

Most of the material of \S\ref{ku48} extends to the corners case with no significant changes: we just insert `with boundary' or `with corners' throughout, apply Theorem \ref{ku5thm1} in place of Theorem \ref{ku4thm1}, and everything just works. In particular, as in Definition \ref{ku4def36} we define {\it fair coordinate systems with corners}. As in Example \ref{ku4ex3}, we can show that a Kuranishi space with corners $\bX$, in the sense of \S\ref{ku51}, has a natural fair coordinate system with corners.

As in Example \ref{ku4ex4}, if $\bX=(X,\cK)$ is a FOOO Kuranishi space with corners, we can define a natural fair coordinate system with corners on $X$. Here as in Remark \ref{kuArem1}(b), it is important that the definition of embedding $\vp_{qp}:V_{qp}\hookra V_p$ of manifolds with corners used in FOOO coordinate changes $\Phi_{qp}=(V_{qp},h_{qp},\vp_{qp},\hat\vp_{qp})$ includes the condition that $\vp_{qp}$ be simple. This is needed when applying Theorem \ref{ku5thm2} to deduce that the 1-morphism $\ti\Phi_{qp}$ of Kuranishi neighbourhoods with corners induced from $\Phi_{qp}$ is a coordinate change.

As in Examples \ref{ku4ex5} and \ref{ku4ex6}, if $X$ is a compact, metrizable topological space equipped with a FOOO weak good coordinate system with corners, or an MW weak Kuranishi atlas with corners, where again as for Example \ref{ku4ex4} the embeddings $\vp_{ij}:V_{ij}\hookra V_j$ and $\ti V_{BC}\hookra V_C$ must be simple, then $X$ has a natural fair coordinate system with corners.

Theorem \ref{ku4thm9} and its proof in \S\ref{ku74} extend immediately to the corners case, and the analogues of Theorems \ref{ku4thm10}--\ref{ku4thm14} follow.

\subsection{Kuranishi spaces with generalized corners}
\label{ku56}

In \S\ref{ku34} we summarized the author's theory of {\it manifolds with generalized corners}, or {\it g-corners\/} \cite{Joyc9}, and we saw that most of \S\ref{ku31}--\S\ref{ku33} extends to manifolds with g-corners. Section \ref{ku38} explained that the theory of $\mu$-Kuranishi spaces with corners in \S\ref{ku35}--\S\ref{ku37} can be extended, with minor changes, to a theory of {\it $\mu$-Kuranishi spaces with g-corners}, by replacing manifolds with corners by manifolds with g-corners throughout.

In the same way, the material on Kuranishi spaces with corners in \S\ref{ku51}--\S\ref{ku55} can be extended, with minor changes, to a theory of Kuranishi spaces with g-corners. All the issues which come up in doing this already arise in extending $\mu$-Kuranishi spaces with corners to g-corners, and were discussed in~\S\ref{ku38}. 

Thus, in \S\ref{ku51} we define {\it Kuranishi neighbourhoods $(V,E,\Ga,s,\psi)$ on $X$ with g-corners\/} by taking $V,E$ to be manifolds with g-corners, and define their 1- and 2-morphisms and coordinate changes, and Proposition \ref{ku5prop1} and Theorems \ref{ku5thm1} and \ref{ku5thm2} hold with g-corners. So we define a weak 2-category $\Kurgc$ of {\it Kuranishi spaces with g-corners\/} $\bX$, containing $\Kurc\subset\Kurgc$ as a full 2-subcategory.

In \S\ref{ku52} we define {\it boundaries\/} $\pd\bX$, $k$-{\it corners\/} $C_k(\bX)$, {\it Kuranishi spaces with g-corners of mixed dimension\/} $\cKurgc$, and the {\it corner\/ $2$-functor\/} $C:\Kurgc\ra\cKurgc$. We define {\it interior}, {\it b-normal\/} and {\it simple\/} 1-morphisms in $\Kurgc$, and write $\Kurgcin,\cKurgcin,\Kurgcsi,\cKurgcsi$ for the 2-subcategories of $\Kurgc,\cKurgc$ with interior, and simple, 1-morphisms. Proposition \ref{ku5prop2} holds with g-corners.

In \S\ref{ku53} we define {\it isotropy groups\/} $G_x\bX$, {\it b-tangent spaces\/} ${}^bT_x\bX$, and {\it b-obstruction spaces\/} ${}^bO_x\bX$ for Kuranishi spaces with g-corners $\bX$, and we define $G_x\bs f:G_x\bX\ra G_y\bY$ for all 1-morphisms $\bs f:\bX\ra\bY$ in $\Kurgc$ and ${}^bT_x\bs f:{}^bT_x\bX\ra{}^bT_y\bY$, ${}^bO_x\bs f:{}^bO_x\bX\ra{}^bO_y\bY$ for interior $\bs f:\bX\ra\bY$. We do {\it not\/} define $T_x\bX,O_x\bX,T_x\bs f,O_x\bs f$. 

In \S\ref{ku54} we define a weak 2-category $\mKurgc$ of {\it m-Kuranishi spaces with g-corners}, and a full 2-subcategory $\KurtrGgc\subset\Kurgc$ of Kuranishi spaces with g-corners $\bX$ with trivial isotropy groups $G_x\bX$, and we define equivalences of (2-)categories $F_\mKurgc^\muKurgc:\Ho(\mKurgc)\ra\muKurgc$ and $F_\mKurgc^\KurtrGgc:\mKurgc\ra\KurtrGgc$, so that~$\Ho(\KurtrGgc)\simeq\muKurgc$.

In \S\ref{ku55} we can define {\it fair coordinate systems with g-corners\/} as in Definition \ref{ku4def36}, and Theorem \ref{ku4thm9} holds with g-corners, but the author knows of no definitions of Kuranishi space with g-corners in the literature to compare to~ours.

Kuranishi spaces with g-corners have better behaviour under fibre products than Kuranishi spaces with corners. For example, in \cite{Joyc12} we will prove the following theorem, related to Theorem~\ref{ku3thm1}:

\begin{thm} Suppose $\bX,\bY$ are Kuranishi spaces with g-corners, $Z$ is a manifold with g-corners, and\/ $\bs g:\bX\ra Z,$ $\bs h:\bY\ra Z$ are interior $1$-morphisms in $\Kurgc$. Then the fibre product\/ $\bW=\bX\t_{\bs g,Z,\bs h}\bY$ exists in the\/ $2$-category\/ $\Kurgcin$ of Kuranishi spaces with g-corners and interior\/ $1$-morphisms, with virtual dimension\/~$\vdim\bW=\vdim\bX+\vdim\bY-\dim Z$.
\label{ku5thm3}
\end{thm}

The analogue is false for Kuranishi spaces with (ordinary) corners, unless we impose complicated extra restrictions on $\bs g,\bs h$ over~$\pd^j\bX,\pd^k\bY,\pd^lZ$. 

Kuranishi spaces with g-corners will be important in future applications in symplectic geometry that the author is planning, for two reasons. Firstly, the author would like to develop an approach to moduli spaces of $J$-holomorphic curves using `representable 2-functors', modelled on Grothendieck's representable functors in algebraic geometry. It turns out that even if the moduli space is a Kuranishi space with (ordinary) corners, the definition of the moduli 2-functor near curves with boundary nodes involves fibre products which exist in $\Kurgc$, but not in $\Kurc$. So we need $\Kurgc$ to define moduli spaces using this method.

Secondly, some kinds of moduli spaces of $J$-holomorphic curves should actually have g-corners rather than ordinary corners, in particular the moduli spaces of `pseudoholomorphic quilts' of Ma'u, Wehrheim and Woodward \cite{Mau,MaWo,WeWo1,WeWo2,WeWo3}, which are used to define actions of Lagrangian correspondences on Lagrangian Floer cohomology and Fukaya categories. 

Ma'u and Woodward \cite{MaWo} define moduli spaces $\oM_{n,1}$ of `stable $n$-marked quilted discs'. As in \cite[\S 6]{MaWo}, for $n\ge 4$ these are not ordinary manifolds with corners, but have an exotic corner structure; in the language of \cite{Joyc9}, the $\oM_{n,1}$ are manifolds with g-corners. As in \cite[Ex.~6.3]{MaWo}, the first exotic example $\oM_{4,1}$ has a point locally modelled on $X_P$ near $\de_0$ in Example \ref{ku3ex2}. 

More generally, if one omits the simplifying monotonicity and genericity assumptions in \cite{Mau,WeWo1,WeWo2,WeWo3}, the moduli spaces of marked quilted $J$-holomorphic discs discussed in \cite{Mau,WeWo1,WeWo2,WeWo3} should be Kuranishi spaces with g-corners (though we do not claim to prove this), just as moduli spaces of marked $J$-holomorphic discs in Fukaya et al. \cite{FOOO1} are Kuranishi spaces with (ordinary) corners.

\section{Proofs of main theorems in \S\ref{ku2} and \S\ref{ku3}}
\label{ku6}

\subsection{Proof of Theorem \ref{ku2thm1}}
\label{ku61}

The proofs will involve partitions of unity, so we begin by defining them.

\begin{dfn} Let $V$ be a manifold (possibly with boundary or corners), and $\{V^a:a\in A\}$ an open cover of $V$. A {\it partition of unity $\{\eta^a:a\in A\}$ on $V$ subordinate to\/} $\{V^a:a\in A\}$ is a collection of smooth functions $\eta^a:V\ra\R$ with $\eta^a(V)\subseteq[0,1]$ for all $a\in A$ satisfying:
\begin{itemize}
\setlength{\itemsep}{0pt}
\setlength{\parsep}{0pt}
\item[(a)] Define the {\it open support\/} of $\eta^a$ to be $\supp^\ci\eta^a=\{v\in V:\eta^a(v)\ne 0\}$.

Define the {\it support\/} of $\eta^a$ to be $\supp\eta^a=\ov{\supp^\ci\eta^a}=\ov{\{v\in V:\eta^a(v)\ne 0\}}$, where `$\,\ov{\phantom{m}}\,$' denotes the closure in $V$. Then $\supp\eta^a\subseteq V^a$ for all~$a\in A$.
\item[(b)] Each $v\in V$ has an open neighbourhood $U_v$ in $V$ with $U_v\cap\supp\eta^a=\es$ for all but finitely many $a\in A$.
\item[(c)] $\sum_{a\in A}\eta^a=1$. Here the sum makes sense by (b), as near each $v\in V$ there are only finitely many nonzero terms (i.e.\ the sum is {\it locally finite\/}).
\end{itemize}

It is well known that such partitions of unity always exist. (Note that our manifolds $V$ are paracompact, by definition, which is needed for local finiteness.)

\label{ku6def}
\end{dfn}

\subsubsection{Theorem \ref{ku2thm1}(a): $\cHom\bigl((V_i,E_i,s_i,\psi_i),(V_j,E_j,s_j,\psi_j)\bigr)$ is a sheaf}
\label{ku611}

Let $X$ be a topological space, $(V_i,E_i,s_i,\psi_i),$ $(V_j,E_j,s_j,\psi_j)$ be $\mu$-Kuranishi neighbourhoods on $X$, and $\cHom\bigl((V_i,E_i,s_i,\psi_i),(V_j,E_j,s_j,\psi_j)\bigr)$ be as in Theorem \ref{ku2thm1}(a). We must show that $\cHom\bigl((V_i,E_i,s_i,\psi_i),(V_j,E_j,s_j,\psi_j)\bigr)$ is a sheaf of sets on $\Im\psi_i\cap\Im\psi_j$, that is, that it satisfies Definition \ref{ku2def9}(i)--(iv). Parts (i),(ii) are immediate from the definition of restriction $\vert_T$ in Definition \ref{ku2def7}.

For (iii), suppose $S\subseteq\Im\psi_i\cap\Im\psi_j$ is open, $\Phi_{ij},\Phi_{ij}':(V_i,E_i,s_i,\psi_i)\ra(V_j,E_j,s_j,\psi_j)$ are morphisms over $S$, and $\{T^a:a\in A\}$ is an open cover of $S$ such that $\Phi_{ij}\vert_{T^a}=\Phi_{ij}'\vert_{T^a}$ for all $a\in A$. We must show that $\Phi_{ij}=\Phi_{ij}'$. Let $(V_{ij},\phi_{ij},\hat\phi_{ij})$ and $(V_{ij}',\phi_{ij}',\hat\phi_{ij}')$ represent $\Phi_{ij},\Phi_{ij}'$, as in Definition \ref{ku2def3}. Then $\Phi_{ij}\vert_{T^a}=\Phi_{ij}'\vert_{T^a}$ means that for each $a\in A$ there exists an open neighbourhood $\dot V_{ij}^a$ of $\psi_i^{-1}(T^a)$ in $V_{ij}\cap V_{ij}'$ and a smooth morphism $\smash{\La^a:E_i\vert_{\dot V_{ij}^a}\ra \phi_{ij}^*(TV_j)\vert_{\dot V_{ij}^a}}$ of vector bundles on $\dot V_{ij}^a$ which by \eq{ku2eq1} satisfies 
\e
\phi_{ij}'=\phi_{ij}+\La^a\cdot s_i+O(s_i^2)\;\>\text{and}\;\> \hat\phi_{ij}'=\hat\phi_{ij}+\La^a\cdot \phi_{ij}^*(\d s_j)+O(s_i)\;\> \text{on $\dot V_{ij}^a$.}
\label{ku6eq1}
\e

Set $\dot V_{ij}=\bigcup_{a\in A}\dot V_{ij}^a$. Then $\dot V_{ij}$ is an open neighbourhood of $\psi_i^{-1}(S)$ in $V_{ij}\cap V_{ij}'$, since $S=\bigcup_{a\in A}T^a$. Since $\{\dot V_{ij}^a:a\in A\}$ is an open cover of $\dot V_{ij}$, as in Definition \ref{ku6def} we can choose a partition of unity $\{\eta^a:a\in A\}$ subordinate to $\{\dot V_{ij}^a:a\in A\}$. Define $\La:E_i\vert_{\dot V_{ij}}\ra\phi_{ij}^*(TV_j)\vert_{\dot V_{ij}}$ by
\e
\La=\ts\sum_{a\in A}\eta^a\La^a,
\label{ku6eq2}
\e
taking $\eta^a\La^a=0$ on $\dot V_{ij}\sm\dot V_{ij}^a$. Then multiplying \eq{ku6eq1} by $\eta^a$ and summing over $a\in A$ implies that \eq{ku2eq1} holds. 

In more detail, by Definition \ref{ku2def1}(v) the first equation of \eq{ku6eq1} means that if $h:V_j\ra\R$ is smooth, there exists $\al^a\in C^\iy(E_i^*\ot E_i^*\vert_{\dot V_{ij}^a})$ such that
\e
h\ci\phi_{ij}'=h\ci\phi_{ij}+\La^a\cdot(s_i\ot \phi_{ij}^*(\d h))+\al^a\cdot (s_i\ot s_i)\qquad\text{on $\dot V_{ij}^a$,}
\label{ku6eq3}
\e
and we multiply \eq{ku6eq3} by $\eta^a$ and sum over $a\in A$ to deduce the analogue of \eq{ku6eq3} for $\dot V_{ij},\La$ with $\al=\sum_{a\in A}\eta^a\al^a$, and the first equation of \eq{ku2eq1}. Therefore $(V_{ij},\phi_{ij},\hat\phi_{ij})\sim(V_{ij}',\phi_{ij}',\hat\phi_{ij}')$, and so $\Phi_{ij}=\Phi_{ij}'$. This proves Definition~\ref{ku2def9}(iii).

For (iv), suppose $S\subseteq\Im\psi_i\cap\Im\psi_j$ is open, $\{T^a:a\in A\}$ is an open cover of $S$, and $\Phi_{ij}^a:(V_i,E_i,s_i,\psi_i)\ra(V_j,E_j,s_j,\psi_j)$ is a morphism over $T^a$ for all $a\in A$, such that $\Phi_{ij}^a\vert_{T^a\cap T^b}=\Phi_{ij}^b\vert_{T^a\cap T^b}$ for all $a,b\in A$. Let $(V_{ij}^a,\phi_{ij}^a,\hat\phi_{ij}^a)$ represent $\Phi_{ij}^a$ for each $a\in A$. Making $V_{ij}^a\subseteq V_i$ smaller if necessary, we can suppose that $V_{ij}^a\cap s_i^{-1}(0)=\psi_i^{-1}(T^a)$. Then $\Phi_{ij}^a\vert_{T^a\cap T^b}=\Phi_{ij}^b\vert_{T^a\cap T^b}$ means that there exists an open neighbourhood $\dot V_{ij}^{ab}$ of $\psi_i^{-1}(T^a\cap T^b)$ in $V_{ij}^a\cap V_{ij}^b$ and a smooth morphism $\La^{ab}:E_i\vert_{\dot V_{ij}^{ab}}\ra (\phi_{ij}^a)^*(TV_j)\vert_{\dot V_{ij}^{ab}}$ of vector bundles on $\dot V_{ij}^{ab}$ satisfying 
\e
\phi_{ij}^b=\phi_{ij}^a+\La^{ab}\cdot s_i+O(s_i^2)\;\>\text{and}\;\> \hat\phi_{ij}^b=\hat\phi_{ij}^a+\La^{ab}\cdot (\phi_{ij}^a)^*(\d s_j)+O(s_i)\;\> \text{on $\dot V_{ij}^{ab}$.}
\label{ku6eq4}
\e

Since $\bigl\{\dot V_{ij}^{ab},(V_{ij}^a\cap V_{ij}^b)\sm\psi_i^{-1}(T^a\cap T^b)\bigr\}$ is an open cover of $V_{ij}^a\cap V_{ij}^b$, we can choose a subordinate partition of unity $\{\eta,1-\eta\}$. That is, $\eta:V_{ij}^a\cap V_{ij}^b\ra\R$ is smooth with $\supp\eta\subseteq\dot V_{ij}^{ab}$, and $\eta=1$ on an open neighbourhood of $\psi_i^{-1}(T^a\cap T^b)$ in $V_{ij}^a\cap V_{ij}^b$. Set $\ti\La^{ab}=\eta\La^{ab}$, which is defined on $V_{ij}^a\cap V_{ij}^b$. As the $O(s_i^2),O(s_i)$ conditions are trivial away from $s_i^{-1}(0)=\psi_i^{-1}(T^a\cap T^b)$ in $V_{ij}^a\cap V_{ij}^b$, equation \eq{ku6eq4} holds on $V_{ij}^a\cap V_{ij}^b$ with $\ti\La^{ab}$ in place of $\La^{ab}$. Thus, replacing $\La^{ab}$ by $\ti\La^{ab}$, we may take~$\dot V_{ij}^{ab}=V_{ij}^a\cap V_{ij}^b$.

We must construct a morphism $\Phi_{ij}=[V_{ij},\phi_{ij},\hat\phi_{ij}]:(V_i,E_i,s_i,\psi_i)\ra(V_j,\ab E_j,\ab s_j,\ab\psi_j)$ over $S$ such that $\Phi_{ij}\vert_{T^a}=\Phi_{ij}^a$ for all $a\in A$. Choose a partition of unity $\{\eta^a:a\in A\}$ on $\bigcup_{a\in A}V_{ij}^a\subseteq V_i$ subordinate to $\{V_{ij}^a:a\in A\}$. Let $V_{ij}$ be an open neighbourhood of $\psi_i^{-1}(S)$ in $\bigcup_{a\in A}V_{ij}^a\subseteq V_i$, to be chosen later. We want to define a smooth map $\phi_{ij}:V_{ij}\ra V_j$ by combining the maps $\phi_{ij}^a:V_{ij}^a\ra V_j$ using $\{\eta^a:a\in A\}$. Roughly speaking, we would like to define
\e
\phi_{ij}=\ts\sum_{a\in A}\bigl(\eta^a\cdot\phi^a_{ij}\bigr)\vert_{V_{ij}}.
\label{ku6eq5}
\e
However, \eq{ku6eq5} does not make sense, as $V_j$ is a manifold, not a vector space.

The next lemma will allow us to make sense of \eq{ku6eq5}, see~\eq{ku6eq10}.

\begin{lem} For each\/ $n=1,2,\ldots,$ write $\De_{n-1}$ for the $(n-1)$-simplex
\begin{equation*}
\De_{n-1}=\bigl\{(t^1,\ldots,t^n)\in[0,1]^n:t^1+\cdots+t^n=1\bigr\}.
\end{equation*}
Then we may choose $W_n\subseteq V_j^n\t\De_{n-1}$ and a smooth map $\Xi_n:W_n\ra V_j$ for all\/ $n\ge 1,$ such that:
\begin{itemize}
\setlength{\itemsep}{0pt}
\setlength{\parsep}{0pt}
\item[{\bf(i)}] $W_n$ is an open neighbourhood of\/ $\bigl\{(v,\ldots,v)\!:\!v\!\in\! V_j\bigr\}\!\t\!\De_{n-1}$ in\/~$V_j^n\!\t\!\De_{n-1}$.
\item[{\bf(ii)}] Let\/ $(x_1,\ldots,x_m)$ be local coordinates on an open set $U\subseteq V_j$. Writing points in these coordinates, we have
\ea
&\Xi_n\bigl((x_1^1,\ldots,x_m^1),\ldots,(x_1^n,\ldots,x_m^n),t^1,\ldots,t^n\bigr)=
\label{ku6eq6}\\
&\bigl(t^1x_1^1+\cdots+t^nx_1^n,\ldots,t^1x_m^1+\cdots+t^nx_m^n\bigr)+O\bigl(\ts\sum_{a,b,l}t^at^b(x^a_l-x^b_l)^2\bigr).
\nonumber
\ea
Here `$F=O(G)$' in \eq{ku6eq6} for functions $F,G$ means that $\md{F}\le C\cdot G$ for some continuous function $C:W_n\cap (U^n\t\De_{n-1})\ra[0,\iy),$ so that\/ $C$ is bounded on compact subsets of\/~$W_n\cap (U^n\t\De_{n-1})$.
\item[{\bf(iii)}] For all permutations $\si\in S_n,$ we have $(v^1,\ldots,v^n,t^1,\ldots,t^n)\in W_n$ if and only if\/ $(v^{\si(1)},\ldots,v^{\si(n)},t^{\si(1)},\ldots,t^{\si(n)})\in W_n,$ and then
\begin{equation*}
\Xi_n(v^1,\ab\ldots,\ab v^n,\ab t^1,\ab\ldots,\ab t^n)\ab=\Xi_n(v^{\si(1)},\ldots,v^{\si(n)},t^{\si(1)},\ldots,t^{\si(n)}).
\end{equation*}
\item[{\bf(iv)}] Whenever $(v^1,\ldots,v^n,t^1,\ldots,t^n)\in W_n$ and\/ $(v^1,\ldots,v^n,v^{n+1},t^1,\ldots,t^n,0)\ab\in W_{n+1}$ we have
\begin{equation*}
\Xi_n(v^1,\ldots,v^n,t^1,\ldots,t^n)=\Xi_{n+1}(v^1,\ldots,v^n,v^{n+1},t^1,\ldots,t^n,0).
\end{equation*}
\end{itemize}
\label{ku6lem}
\end{lem}

\begin{proof} We give two proofs, (A) and (B). Method (A) is simpler, but method (B) will extend more cleanly to $V_j$ a manifold with corners in~\S\ref{ku63}.
\smallskip

\noindent{\bf Method (A).} Choose an embedding $i:V_j\hookra\R^N$ for $N\gg 0$ ($N=2m+1$ will do, where $m=\dim V_j$). By the existence of tubular neighbourhoods, we can choose an open neighbourhood $T$ of $i(V_j)$ in $\R^N$ and a smooth map $\pi:T\ra V_j$ with $\pi\ci i=\id_{V_j}$. For all $n\ge 1$, define $W_n\subseteq V_j^n\t\De_{n-1}$ and $\Xi_n:W_n\ra V_j$ by
\ea
&W_n=\bigl\{(v^1,\ldots,v^n,t^1,\ldots,t^n)\in V_j^n\t\De_{n-1}:
t^1i(v^1)+\cdots+t^ni(v^n)\in T\bigr\},
\nonumber\\
&\text{and}\quad\Xi_n:(v^1,\ldots,v^n,t^1,\ldots,t^n)\longmapsto\pi\bigl(t^1i(v^1)+\cdots+t^ni(v^n)\bigr).
\label{ku6eq7}
\ea
Parts (i) and (iii)--(iv) are immediate. To prove (ii), observe that \eq{ku6eq6} is equivalent to prescribing $\Xi_n$ and its derivative $\d\Xi_n$ at $\bigl\{(v,\ldots,v):v\in V_j\bigr\}\t\De_{n-1}$ in $W_n$, where $\Xi_n(v,\ldots,v,t^1,\ldots,t^n)=v$, and one can easily check this $\Xi_n$ has the desired properties. 
\smallskip

\noindent{\bf Method (B).} Fix a Riemannian metric $g$ on $V_j$. Then for some continuous function $r:V_j\ra(0,\iy)$, for each $v\in V_j$ there is a smooth {\it exponential map\/} $\exp_v:B_{v,r(v)}\ra V_j$, where $B_{v,r(v)}$ is the open ball about 0 of radius $r(v)$ in $T_vV_j$, with the Euclidean metric $g\vert_v$, such that for each $\al\in B_{v,r(v)}$ there is a unique geodesic interval $\ga:[0,1]\ra V_j$ in $(V_j,g)$ of length $\md{\al}$, with $\ga(0)=v$, $\frac{\d\ga}{\d t}(0)=\al$, and $\ga(1)=\exp_v(\al)$. If $r(v)$ is smaller than the injectivity radius of $g$ at $v$ then $\exp_v$ is a diffeomorphism with a small open ball about $v$ in $V_j$.

Consider the submanifold $Y_n$ of $\bigop^nTV_j\t\De_{n-1}$ given by
\e
\begin{split}
Y_n=\bigl\{(v,\al^1,\ldots,\al^n,t^1&,\ldots,t^n):v\in V_j, \;\> \al^1,\ldots,\al^n\in B_{v,r(v)},\\ 
&(t^1,\ldots,t^n)\in\De_{n-1},\;\> t^1\al^1+\cdots+t^n\al^n=0\bigr\},
\end{split}
\label{ku6eq8}
\e
and smooth maps $\Pi:Y_n\ra V_j^n\t\De_{n-1}$ and $\pi:Y_n\ra V_j$ defined by
\e
\begin{split}
\Pi&:(v,\al^1,\ldots,\al^n,t^1,\ldots,t^n)\longmapsto\bigl(\exp_v(\al^1),\ldots,\exp_v(\al^n),t^1,\ldots,t^n),\\
\pi&:(v,\al^1,\ldots,\al^n,t^1,\ldots,t^n)\longmapsto v.
\end{split}
\label{ku6eq9}
\e
Computing the derivative of $\Pi$ at $(v,0,\ldots,0,t^1,\ldots,t^n)$, we see that $\Pi$ is a diffeomorphism near $\al^1=\cdots=\al^n=0$. Therefore, making the function $r:V_j\ra(0,\iy)$ smaller if necessary, we can suppose $\Pi$ is a diffeomorphism with its image $W_n:=\Im\Pi$, which is an open neighbourhood of $\bigl\{(v,\ldots,v):v\in V_j\bigr\}\t\De_{n-1}$ in $V_j^n\t\De_{n-1}$. Define $\Xi_n:W_n\ra V_j$ by $\Xi_n=\pi\ci\Pi^{-1}$. One can now check that (i)--(iv) hold.
\end{proof}

As above, let $V_{ij}$ be an open neighbourhood of $\psi_i^{-1}(S)$ in $\bigcup_{a\in A}V_{ij}^a\subseteq V_i$, to be chosen shortly. Define $\phi_{ij}:V_{ij}\ra V_j$ by 
\e
\begin{split}
&\phi_{ij}(v)=\Xi_n\bigl(\phi_{ij}^{a_1}(v),\ldots,\phi_{ij}^{a_n}(v),\eta^{a_1}(v),\ldots,\eta^{a_n}(v)\bigr),\\
&\text{where}\;\> \{a_1,\ldots,a_n\}=\bigl\{a\in A:v\in\supp\eta^a\bigr\}.
\end{split}
\label{ku6eq10}
\e
Here $\bigl\{a\in A:v\in\supp\eta^a\bigr\}$ is finite and nonempty by Definition \ref{ku6def}(b),(c), so we may write $\bigl\{a\in A:v\in\supp\eta^a\bigr\}=\{a_1,\ldots,a_n\}$ for $a_1,\ldots,a_n\in A$ distinct, uniquely up to the order of $a_1,\ldots,a_n$. By Lemma \ref{ku6lem}(iii), the r.h.s.\ of \eq{ku6eq10} is independent of this order. Thus $\phi_{ij}$ is well-defined provided $\bigl(\phi_{ij}^{a_1}(v),\ldots,\phi_{ij}^{a_n}(v),\eta^{a_1}(v),\ldots,\eta^{a_n}(v)\bigr)\in W_n$ for each $v\in V_{ij}$. 

Let us impose this condition at $v\in V_{ij}$ not just for $\{a_1,\ldots,a_n\}$ in \eq{ku6eq10}, but also for every finite nonempty subset of $\{a_1,\ldots,a_n\}$. This then becomes an open condition on $v\in V_{ij}$, since small deformations of $v$ in $V_{ij}$ can only reduce the finite collection of subsets $\{a_1,\ldots,a_n\}\subseteq A$ for which we require $\bigl(\phi_{ij}^{a_1}(v),\ldots,\phi_{ij}^{a_n}(v),\eta^{a_1}(v),\ldots,\eta^{a_n}(v)\bigr)\in W_n$. If $v\in\psi_i^{-1}(S)$ this condition holds by Lemma \ref{ku6lem}(i), since then $\phi_{ij}^{a_k}(v)=\psi_j^{-1}\ci\psi_i(v)$ for $k=1,\ldots,n$. Hence if $V_{ij}$ is a small enough open neighbourhood of $\psi_i^{-1}(S)$ in $\bigcup_{a\in A}V_{ij}^a\subseteq V_i$, then $\phi_{ij}$ is well-defined.

We claim that $\phi_{ij}:V_{ij}\ra V_j$ is smooth. Fix $v\in V_{ij}$, and let $\{a_1,\ldots,a_n\}$ be as in \eq{ku6eq10}. By Definition \ref{ku6def}(b), there is an open neighbourhood $U_v$ of $v$ in $V_{ij}$ such that there are only finitely many $a\in A\sm\{a_1,\ldots,a_n\}$ with $U_v\cap\supp\eta^a\ne\es$. Since $\supp\eta^a$ is closed with $v\notin\supp\eta^a$, by replacing $U_v$ by the complement of these $\supp\eta^a$ in $U_v$, we obtain an open neighbourhood $U_v$ of $v$ such that $U_v\cap\supp\eta^a=\es$ for all $a\in A\sm\{a_1,\ldots,a_n\}$. Making $U_v$ smaller, we can also suppose that for every subset $\{a_{k_1},\ldots,a_{k_{n'}}\}\subseteq\{a_1,\ldots,a_n\}$ of size $n'\ge 1$ and every $v'\in U_v$ we have $\bigl(\phi_{ij}^{a_{k_1}}(v),\ldots,\phi_{ij}^{a_{k_{n'}}}(v),\eta^{a_{k_1}}(v),\ldots,\eta^{a_{k_{n'}}}(v)\bigr)\in W_{n'}$, since this holds at $v$ by definition of $V_{ij}$, and is an open condition.

Thus, for all $v'\in U_v$, the subset $\{a'_1,\ldots,a'_{n'}\}$ in \eq{ku6eq10} corresponding to $v'$ is a subset of $\{a_1,\ldots,a_n\}$. For fixed $v'\in U_v$, we can reorder $a_1,\ldots,a_n$ (which does not change $\phi_{ij}$ by Lemma \ref{ku6lem}(iii)) so that the subset in \eq{ku6eq10} corresponding to $v'$ is $\{a_1,\ldots,a_{n'}\}\subseteq\{a_1,\ldots,a_n\}$, for $1\le n'\le n$. Then $\eta^{a_k}(v')=0$ for $n'<k\le n$. Therefore Lemma \ref{ku6lem}(iv) and induction over $n',\ldots,n$ gives
\begin{align*}
\phi_{ij}(v')&=\Xi_{n'}\bigl(\phi_{ij}^{a_1}(v),\ldots,\phi_{ij}^{a_{n'}}(v),\eta^{a_1}(v),\ldots,\eta^{a_{n'}}(v)\bigr)\\
&=\Xi_n\bigl(\phi_{ij}^{a_1}(v),\ldots,\phi_{ij}^{a_n}(v),\eta^{a_1}(v),\ldots,\eta^{a_{n'}}(v),0,\ldots,0\bigr)\\
&=\Xi_n\bigl(\phi_{ij}^{a_1}(v),\ldots,\phi_{ij}^{a_n}(v),\eta^{a_1}(v),\ldots,\eta^{a_n}(v)\bigr).
\end{align*}
That is, for $v'$ in the open neighbourhood $U_v$, we can define $\phi_{ij}(v')$ as in \eq{ku6eq10} but using the fixed subset $\{a_1,\ldots,a_n\}\subseteq A$, rather than using a subset $\{a'_1,\ldots,a'_{n'}\}$ depending on $v'$. So $\phi_{ij}$ is smooth on $U_v$, since $\phi_{ij}^{a_k},\eta^{a_k}$ and $\Xi_n$ are all smooth. Hence $\phi_{ij}:V_{ij}\ra V_j$ is smooth, as this holds for all~$v\in V_{ij}$.

Fix $a\in A$. Then for each $b\in A$ we have a morphism $\La^{ab}$ on $V_{ij}^a\cap V_{ij}^b$ satisfying \eq{ku6eq4} on $V_{ij}^a\cap V_{ij}^b$. As for \eq{ku6eq2}, define a morphism $\La^a:E_i\vert_{V_{ij}^a}\ra(\phi_{ij}^a)^*(TV_j)$ of vector bundles on $V_{ij}^a$ by
\e
\La^a=\ts\sum_{b\in A}\eta^b\vert_{V_{ij}^a}\cdot\La^{ab}.
\label{ku6eq11}
\e 
As in \eq{ku6eq3}, the first equation of \eq{ku6eq4} means that if $h:V_j\ra\R$ is smooth, there exists $\al^{ab}\in C^\iy(E_i^*\ot E_i^*\vert_{V_{ij}^a\cap V_{ij}^b})$ such that
\e
h\ci\phi_{ij}^b=h\ci\phi_{ij}^a+\La^{ab}\cdot(s_i\ot (\phi_{ij}^a)^*(\d h))+\al^{ab}\cdot (s_i\ot s_i)\quad\text{on $V_{ij}^a\cap V_{ij}^b$.}
\label{ku6eq12}
\e
Multiplying \eq{ku6eq12} by $\eta^b$, regarding it as an equation on $V_{ij}^a\cap V_{ij}$, summing over $b\in A$, and using Lemma \ref{ku6lem}(ii) and \eq{ku6eq11}, we deduce that
\e
\phi_{ij}=\phi_{ij}^a+\La^a\cdot s_i+O(s_i^2)\quad\text{on $V_{ij}^a\cap V_{ij}$.}
\label{ku6eq13}
\e
Here using Lemma \ref{ku6lem}(ii) we can show that $h\ci\phi_{ij}=\sum_{b\in A}\eta^b\cdot(h\ci\phi_{ij}^b)+O(s_i^2)$.

Equation \eq{ku6eq13} implies that $\phi_{ij}=\phi_{ij}^a+O(s_i)$ on $V_{ij}^a\cap V_{ij}$. Thus as in Definition \ref{ku2def1} there exists a morphism $\hat\phi_{ij}^{\prime a}:E_i\vert_{V_{ij}^a\cap V_{ij}}\ra \phi_{ij}\vert_{V_{ij}^a\cap V_{ij}}^*(E_j)$ with $\hat\phi_{ij}^{\prime a}=\hat\phi_{ij}^a+O(s_i)$ in the sense of Definition \ref{ku2def1}(vi), and $\hat\phi_{ij}^{\prime a}$ is unique up to $O(s_i)$. Define a morphism $\hat\phi_{ij}:E_i\vert_{V_{ij}}\ra\phi_{ij}\vert_{V_{ij}}^*(E_j)$ by
\e
\hat\phi_{ij}=\ts\sum_{a\in A}\eta^a\vert_{V_{ij}}\cdot \hat\phi_{ij}^{\prime a}.
\label{ku6eq14}
\e
Now for each $a\in A$, on $V_{ij}\cap V_{ij}^a$ we have
\e
\begin{split}
\hat\phi_{ij}&=\ts\sum_{b\in A}\eta^b\vert_{V_{ij}}\cdot \hat\phi_{ij}^{\prime b}\\
&=\ts\sum_{b\in A}\eta^b\vert_{V_{ij}}\cdot \bigl(
\hat\phi_{ij}^a+\La^{ab}\cdot (\phi_{ij}^a)^*(\d s_j)+O(s_i)\bigr)\\
&=\hat\phi_{ij}^a+\La^a\cdot (\phi_{ij}^a)^*(\d s_j)+O(s_i),
\end{split}
\label{ku6eq15}
\e
using \eq{ku6eq14} in the first step, the second equation of \eq{ku6eq4} in the second, and $\sum_{b\in A}\eta^b=1$ and \eq{ku6eq11} in the third.

Observe that \eq{ku6eq13} and \eq{ku6eq15} imply that $(V_{ij},\phi_{ij},\hat\phi_{ij})\sim(V_{ij}^a,\phi_{ij}^a,\hat\phi_{ij}^a)$ over $T^a$ in the sense of Definition \ref{ku2def3}. Since $(V_{ij}^a,\phi_{ij}^a,\hat\phi_{ij}^a)$ satisfies Definition \ref{ku2def3}(a)--(e) over $T^a$, this implies that $(V_{ij},\phi_{ij},\hat\phi_{ij})$ satisfies Definition \ref{ku2def3}(a)--(e) over $T^a$ for each $a\in A$, and as $\bigcup_{a\in A}T^a=S$, we see that $(V_{ij},\phi_{ij},\hat\phi_{ij})$ satisfies Definition \ref{ku2def3}(a)--(e) over $S$. Hence $\Phi_{ij}:=[V_{ij},\phi_{ij},\hat\phi_{ij}]:(V_i,E_i,s_i,\psi_i)\ra(V_j,E_j,s_j,\psi_j)$ is a morphism over $S$, and \eq{ku6eq15}, \eq{ku6eq15} imply that $\Phi_{ij}\vert_{T^a}=\Phi_{ij}^a$ for all $a\in A$. This completes Definition \ref{ku2def9}(iv). So $\cHom((V_i,E_i,s_i,\psi_i),(V_j,E_j,s_j,\psi_j))$ is a sheaf on $\Im\psi_i\cap\Im\psi_j$, proving the first part of Theorem~\ref{ku2thm1}(a).

\subsubsection{Theorem \ref{ku2thm1}(a): $\cIso\bigl((V_i,E_i,s_i,\psi_i),(V_j,E_j,s_j,\psi_j)\bigr)$ is a sheaf}
\label{ku612}

The second part, that $\mu$-coordinate changes $\cIso((V_i,\ab E_i,\ab s_i,\ab\psi_i),(V_j,\ab E_j,\ab s_j,\ab\psi_j))$ are a subsheaf of $\cHom((V_i,E_i,s_i,\psi_i),(V_j,E_j,s_j,\psi_j))$, is a purely formal consequence of the first part. Definition \ref{ku2def9}(i)--(iii) for $\cIso\bigl((V_i,E_i,s_i,\psi_i),\ab(V_j,\ab E_j,\ab s_j,\ab\psi_j)\bigr)$ are immediate. To prove (iv), we must show that in the last part of the proof in \S\ref{ku611}, if the $\Phi_{ij}^a$ for $a\in A$ are $\mu$-coordinate changes, then the $\Phi_{ij}$ we construct with $\Phi_{ij}\vert_{T^a}=\Phi_{ij}^a$ is also a $\mu$-coordinate change. 

To see this, note that $(\Phi_{ij}^a)^{-1}:(V_j,E_j,s_j,\psi_j)\ra(V_i,E_i,s_i,\psi_i)$ exists on $T^a$ as $\Phi_{ij}^a$ is a $\mu$-coordinate change, and $(\Phi_{ij}^a)^{-1}\vert_{T^a\cap T^b}=(\Phi_{ij}^b)^{-1}\vert_{T^a\cap T^b}$ as $\Phi_{ij}^a\vert_{T^a\cap T^b}=\Phi_{ij}^b\vert_{T^a\cap T^b}$, so the proof above gives $\Phi_{ji}:(V_j,E_j,s_j,\psi_j)\ra(V_i,\ab E_i,\ab s_i,\ab\psi_i)$ on $S$ with $\Phi_{ji}\vert_{T^a}=(\Phi_{ij}^a)^{-1}$ for all $a\in A$. Then we have
\begin{align*}
(\Phi_{ji}\ci\Phi_{ij})\vert_{T^a}&=(\Phi_{ij}^a)^{-1}\ci\Phi_{ij}^a\vert_{T^a}=\id_{(V_i,E_i,s_i,\psi_i)}\vert_{T^a},\\
(\Phi_{ij}\ci\Phi_{ji})\vert_{T^a}&=\Phi_{ij}^a\ci(\Phi_{ij}^a)^{-1}\vert_{T^a}=\id_{(V_j,E_j,s_j,\psi_j)}\vert_{T^a},
\end{align*}
so Definition \ref{ku2def9}(iii) for $\cHom((V_i,E_i,s_i,\psi_i),(V_j,E_j,s_j,\psi_j))$ yields $\Phi_{ji}\ci\Phi_{ij}=\id_{(V_i,E_i,s_i,\psi_i)}$ and $\Phi_{ij}\ci\Phi_{ji}=\id_{(V_j,E_j,s_j,\psi_j)}$. Therefore $\Phi_{ji}=\Phi_{ij}^{-1}$, and $\Phi_{ij}$ is a $\mu$-coordinate change. This completes Theorem~\ref{ku2thm1}(a). 

\subsubsection{Theorem \ref{ku2thm1}(b): $\cHom_f\bigl((U_i,\ab D_i,\ab r_i,\ab\chi_i),\ab (V_j,E_j,s_j,\psi_j)\bigr)$ is a sheaf}
\label{ku613}

The proof of Theorem \ref{ku2thm1}(b) is essentially the same as for the first part of Theorem \ref{ku2thm1}(a) in \S\ref{ku611}, but inserting `over $f$' throughout.

\subsection{Proof of Theorem \ref{ku2thm2}}
\label{ku62}

Theorem \ref{ku2thm2} gives a criterion for a morphism of $\mu$-Kuranishi neighbourhoods $\Phi_{ij}:(V_i,E_i,s_i,\psi_i)\ra (V_j,E_j,s_j,\psi_j)$ to be a coordinate change.

\subsubsection{The `only if' part of Theorem \ref{ku2thm2}}
\label{ku621}

Suppose $\Phi_{ij}:(V_i,E_i,s_i,\psi_i)\ra (V_j,E_j,s_j,\psi_j)$ is a $\mu$-coordinate change over $S\subseteq X$, and let $(V_{ij},\phi_{ij},\hat\phi_{ij})$ represent $\Phi_{ij}$. Let $x\in S,$ and set $v_i=\psi_i^{-1}(x)\in V_i$ and $v_j=\psi_j^{-1}(x)\in V_j$, so that $\phi_{ij}(v_i)=v_j$. We must show that \eq{ku2eq8} is exact.

As $\Phi_{ij}$ is a $\mu$-coordinate change, it has an inverse $\Phi_{ji}=\Phi_{ij}^{-1}$ over $S$, represented by $(V_{ji},\phi_{ji},\hat\phi_{ji})$, say. By Definition \ref{ku2def4}, 
\begin{equation*}
\bigl(\phi_{ij}^{-1}(V_{ji}),\phi_{ji}\ci\phi_{ij}\vert_{\phi_{ij}^{-1}(V_{ji})},\phi_{ij}\vert_{\phi_{ij}^{-1}(V_{ji})}^*(\hat\phi_{ji})\ci\hat\phi_{ij}\vert_{\phi_{ij}^{-1}(V_{ji})}\bigr)
\end{equation*}
represents $\Phi_{ji}\ci\Phi_{ij}=\id_{(V_i,E_i,s_i,\psi_i)}=[V_i,\id_{V_i},\id_{E_i}]$, so by Definition \ref{ku2def3}, there exists an open neighbourhood $\dot V_{ii}$ of $\psi_i^{-1}(S)$ in $\phi_{ij}^{-1}(V_{ji})\subseteq V_i$ and a vector bundle morphism $\La:E_i\vert_{\dot V_{ii}}\ra TV_i\vert_{\dot V_{ii}}$ satisfying
\e
\begin{split}
\phi_{ji}\ci\phi_{ij}&=\id_{V_i}+\La\cdot s_i+O(s_i^2)\;\>\text{and}\\
\phi_{ij}^*(\hat\phi_{ji})\ci\hat\phi_{ij}&=\id_{E_i}+\La\cdot\d s_i+O(s_i)\;\> \text{on $\dot V_{ii}$.}
\end{split}
\label{ku6eq16}
\e
Restricting the derivative of the first equation of \eq{ku6eq16}, and the second equation of \eq{ku6eq16}, to $v_i$ and using $s_i(v_i)=0$, yields
\e
\begin{split}
\d\phi_{ji}\vert_{v_j}\ci\d\phi_{ij}\vert_{v_i}&=\id_{T_{v_i}V_i}+\La\vert_{v_i}\ci\d s_i\vert_{v_i},\\
\hat\phi_{ji}\vert_{v_j}\ci\hat\phi_{ij}\vert_{v_i}&=\id_{E_i\vert_{v_i}}+\d s_i\vert_{v_i}\ci\La\vert_{v_i}.
\end{split}
\label{ku6eq17}
\e

Similarly, from $\Phi_{ij}\!\ci\!\Phi_{ji}\!=\!\id_{(V_j,E_j,s_j,\psi_j)}$ we obtain $\Mu:E_j\vert_{\dot V_{jj}}\!\ra\! TV_j\vert_{\dot V_{jj}}$ with
\e
\begin{split}
\d\phi_{ij}\vert_{v_i}\ci\d\phi_{ji}\vert_{v_j}&=\id_{T_{v_j}V_j}+\Mu\vert_{v_j}\ci\d s_j\vert_{v_j},\\
\hat\phi_{ij}\vert_{v_i}\ci\hat\phi_{ji}\vert_{v_j}&=\id_{E_j\vert_{v_j}}+\d s_j\vert_{v_j}\ci\Mu\vert_{v_j}.
\end{split}
\label{ku6eq18}
\e
Differentiating Definition \ref{ku2def3}(d) for $\Phi_{ij}$ yields 
\e
\hat\phi_{ji}\vert_{v_j}\ci\d s_j\vert_{v_j}=\d s_i\vert_{v_i}\ci\d\phi_{ji}\vert_{v_j}.
\label{ku6eq19}
\e
By a similar proof to Proposition \ref{kuBprop} we can also choose $\La,\Mu$ to satisfy
\e
\d\phi_{ij}\vert_{v_i}\ci\La\vert_{v_i}=\Mu\vert_{v_j}\ci\hat\phi_{ij}\vert_{v_i},
\label{ku6eq20}
\e
although this is not really needed for the proof.

Now consider the diagram of vector spaces
\e
\xymatrix@C=17pt{ 0 \ar@<.5ex>[r] & T_{v_i}V_i \ar@<.5ex>[rrr]^(0.39){\d s_i\vert_{v_i}\op\d\phi_{ij}\vert_{v_i}} \ar@<.5ex>[l] &&& E_i\vert_{v_i} \!\op\!T_{v_j}V_j 
\ar@<.5ex>[rrr]^(0.56){-\hat\phi_{ij}\vert_{v_i}\op \d s_j\vert_{v_j}} \ar@<.5ex>[lll]^(0.61){-\La\vert_{v_i}\op\d\phi_{ji}\vert_{v_j}} &&& E_j\vert_{v_j} \ar@<.5ex>[lll]^(0.44){-\hat\phi_{ji}\vert_{v_j}\op -\Mu\vert_{v_j}} \ar@<.5ex>[r] & 0. \ar@<.5ex>[l] }
\label{ku6eq21}
\e
The rightward morphisms are \eq{ku2eq8}. Equation \eq{ku6eq21} is of the form
\begin{equation*}
\xymatrix@C=25pt{ 0 \ar@<.5ex>[r] & E \ar@<.5ex>[rr]^\al \ar@<.5ex>[l] && F \ar@<.5ex>[rr]^\be \ar@<.5ex>[ll]^\ga && G \ar@<.5ex>[ll]^\de \ar@<.5ex>[r] & 0, \ar@<.5ex>[l] }
\end{equation*}
where $\ga\ci\al=\id_E$, $\al\ci\ga+\de\ci\be=\id_F$, $\be\ci\de=\id_G$ by \eq{ku6eq17}--\eq{ku6eq20}. Thus \eq{ku2eq8} is exact, by elementary linear algebra, proving the `only if' part of Theorem~\ref{ku2thm2}.

\subsubsection{The `if' part of Theorem \ref{ku2thm2}}
\label{ku622}

Let $\Phi_{ij}:(V_i,E_i,s_i,\psi_i)\ra(V_j,E_j,s_j,\psi_j)$ be a morphism of $\mu$-Kuranishi neighbourhoods over $S\subseteq X$, represented by $(V_{ij},\phi_{ij},\hat\phi_{ij})$, and suppose that \eq{ku2eq8} is exact for all $x\in S$. We must show $\Phi_{ij}$ is a $\mu$-coordinate change. Fix $x\in X$, and set $v_i=\psi_i^{-1}(x)\in V_i$ and $v_j=\psi_j^{-1}(x)\in V_j$. We will first show that $\Phi_{ij}\vert_{S^x}$ is invertible over $S^x$ for a small open neighbourhood $S^x$ of $x$ in $X$. The `if' part will then follow from the sheaf property of~$\cIso\bigl((V_i,E_i,s_i,\psi_i),(V_j,E_j,s_j,\psi_j)\bigr)$.

As in \eq{ku2eq26}, we form a commutative diagram with exact rows
\e
\begin{gathered}
\xymatrix@C=23pt@R=15pt{ 0 \ar[r] & K_i^x \ar[r] \ar[d]^{\ka^x_{ij}}_\cong & T_{v_i}V_i \ar[rr]_{\d s_i\vert_{v_i}} \ar[d]^{\d\phi_{ij}\vert_{v_i}} && E_i\vert_{v_i} \ar[d]^{\hat\phi_{ij}\vert_{v_i}} \ar[r] & C_i^x \ar[d]^{\ga^x_{ij}}_\cong \ar[r] & 0 \\
0 \ar[r] & K_j^x \ar[r] & T_{v_j}V_j \ar[rr]^{\d s_j\vert_{v_j}} && E_j\vert_{v_j} \ar[r] & C_j^x \ar[r] & 0,\!\! }
\end{gathered}
\label{ku6eq22}
\e
where $K_i^x,C_i^x$ and $K_j^x,C_j^x$ are the kernel and cokernel of $\d s_i\vert_{v_i}$ and $\d s_j\vert_{v_j}$. By elementary linear algebra, \eq{ku2eq8} being exact is equivalent to $\ka^x_{ij}:K_i^x\ra K_j^x$ and $\ga^x_{ij}:C_i^x\ra C_j^x$ both being isomorphisms.

Choose a small open neighbourhood $V_{ij}^x$ of $v_i$ in $V_{ij}\subseteq V_i$, and a vector subbundle $F_i$ of $E_i\vert_{V_{ij}^x}$ such that the projection $F_i\vert_x\ra C_i^x$ is an isomorphism. Several times in what follows we will make the open neighbourhood $\smash{V_{ij}^x}$ smaller. Write $E_i^x=E_i\vert_{V_{ij}^x}$, $s_i^x=s_i\vert_{V_{ij}^x}$, $\phi_{ij}^x=\phi_{ij}\vert_{V_{ij}^x}$ and~$\hat\phi_{ij}^x=\hat\phi_{ij}\vert_{V_{ij}^x}$. 

Now $E_i^x/F_i$ is a vector bundle on $V_{ij}^x$ and $s_i^x$ projects to a section $\ti u_i=s_i^x+F_i\in C^\iy(E_i^x/F_i)$, where $\ti u_i(v_i)=0$ and $\d\ti u_i\vert_{v_i}:T_{v_i}V_i\ra (E_i^x/F_i)\vert_{v_i}$ is surjective by \eq{ku6eq22}. Making $V_{ij}^x$ smaller, we can suppose $\d\ti u_i\vert_v:T_vV_i\ra (E_i^x/F_i)\vert_v$ is surjective for all $v\in V_{ij}^x$ with $\ti u_i(v)=0$. Then $W_i:=\bigl\{v\in V_{ij}^x:\ti u_i(v)=0\bigr\}$ is a submanifold of $V_{ij}^x$ containing $v_i$. Write $E_i'=E_i\vert_{W_i}$, $F_i'=F_i\vert_{W_i}$, and $t_i'=s_i\vert_{W_i}\in C^\iy(F_i')$. Write $S^x=\psi_i((s_i^x)^{-1}(0))$, so that $S^x$ is an open neighbourhood of $x$ in $S\subseteq X$. Then~$\psi_i^{-1}(S^x)=(t_i')^{-1}(0)\subseteq W_i$.

Choose a connection $\nabla^{E_i^x}$ on $E_i^x\ra V_{ij}^x$ which preserves the subbundle $F_i\subseteq E_i^x$. Choose a splitting $TV_i\vert_{W_i}=TW_i\op\nu_i$, where $\nu_i$ is the normal bundle of $W_i$ in $V_{ij}^x$. By definition of $W_i$, $\nabla^{E_i^x}s_i^x\vert_{\nu_i}:\nu_i\ra E_i'$ is an injective morphism of vector bundles on $W_i$, which projects to an isomorphism with $E_i'/F_i'$. Define $G_i'=(\nabla^{E_i^x}s_i^x)(\nu_i)$, a vector subbundle of $E_i'$. Then $E_i'=F_i'\op G_i'$. Choose a vector subbundle $G_i$ of $E_i^x$ such that $E_i^x=F_i\op G_i$ and $G_i\vert_{W_i}=G_i'$. Write $s_i^x=t_i\op u_i$ for $t_i\in C^\iy(F_i)$ and $u_i\in C^\iy(G_i)$. Write $\nabla^{F_i},\nabla^{G_i}$ for the connections on $F_i,G_i\ra V_{ij}^x$ induced from $\nabla^{E_i}$. Then the projection $G_i\ra E_i^x/F_i$ is an isomorphism mapping $u_i\mapsto\ti u_i$, so that $W_i=\bigl\{v\in V_{ij}^x:u_i(v)=0\bigr\}$, and $\nabla^{G_i}u_i\vert_{\nu_i}=\d u_i\vert_{\nu_i}:\nu_i\ra G_i'$ is an isomorphism. As $\d s_i\vert_{v_i}=0\op \d u_i\vert_{v_i}$ we have
\begin{equation*}
K_i^x=\Ker\d s_i\vert_{v_i}=\Ker\d u_i\vert_{v_i}=T_{v_i}W_i.
\end{equation*}

We see from \eq{ku6eq22} that $\d\phi_{ij}\vert_{v_i}:K_i^x\ra T_{v_j}V_j$ is injective. Since $K_i^x=T_{v_i}W_i$, this means that $\phi_{ij}\vert_{W_i}:W_i\ra V_j$ is an embedding near $v_i$, so making $V_{ij}^x,W_i$ smaller, we can suppose $\phi_{ij}\vert_{W_i}:W_i\ra V_j$ is an embedding. Thus $W_j:=\phi_{ij}(W_i)$ is a submanifold of $V_j$ containing $v_j$, and $\xi_{ij}:=\phi_{ij}\vert_{W_i}:W_i\ra W_j$ is a diffeomorphism, with inverse~$\xi_{ji}:=\xi_{ij}^{-1}:W_j\ra W_i$.

From \eq{ku6eq22} and $F_i\vert_{v_i}\ra C_i^x$ an isomorphism we see that $\hat\phi_{ij}\vert_{v_i}:F_i'\vert_{v_i}=F_i\vert_{v_i}\ra E_j\vert_{v_j}$ is injective. Thus making $V_{ij}^x,W_i,W_j$ smaller, we can suppose $\hat\phi_{ij}\vert_{F'_i}:F'_i\ra \xi_{ij}^*(E_j)$ is an injective morphism of vector bundles on $W_i$. Define a vector subbundle $F_j'$ of $E_j\vert_{W_j}$ by $F_j'=\xi_{ji}^*(\hat\phi_{ij}(F'_i))$. Then $\hat\xi_{ij}:=\hat\phi_{ij}\vert_{F'_i}:F'_i\ra \xi_{ij}^*(F'_j)$ is an isomorphism of vector bundles on $W_i$. So we may define $\hat\xi_{ji}:F'_j\ra \xi_{ji}^*(F'_i)$ by $\hat\xi_{ji}=\xi_{ji}^*(\hat\xi_{ij}^{-1})$, an isomorphism of vector bundles on~$W_j$.

Choose a small open neighbourhood $V_{ji}^x$ of $W_j$ (and hence $v_j$) in $V_j$, with $W_j$ closed in $V_{ji}^x$, and a vector subbundle $F_j$ of $E_j^x:=E_j\vert_{V_{ji}^x}$ with $F_j\vert_{W_j}=F_j'$. Write $s_j^x=s_j\vert_{V_{ji}^x}$. Since $\psi_i^{-1}(S^x)\subseteq W_i$, we have $\psi_j^{-1}(S^x)\subseteq W_j\subseteq V_{ji}^x$. As $\psi_j^{-1}(S^x)$ is open in $(s_j^x)^{-1}(0)$ and $W_j$ is closed in $V_{ji}^x$, making $V_{ji}^x$ smaller we can suppose that $(s_j^x)^{-1}(0)=\psi_j^{-1}(S^x)\subseteq W_j$.

Choose a connection $\nabla^{E_j^x}$ on $E_j^x\ra V_{ji}^x$ which preserves the subbundle $F_j\subseteq E_j^x$. Choose a splitting $TV_j\vert_{W_j}=TW_j\op\nu_j$, where $\nu_j$ is the normal bundle of $W_j$ in $V_{ji}^x$. Then $\nabla^{E_j^x}s_j^x\vert_{\nu_j}:\nu_j\ra E_j'$ is a morphism of vector bundles on $W_j$, which projects to an isomorphism on $E_j'/F_j'$ at $v_j$. So making $V_{ij}^x,W_i,V_{ji}^x,W_j$ smaller, we can suppose $\nabla^{E_j^x}s_j^x\vert_{\nu_j}:\nu_j\ra E_j'$ is an injective morphism of vector bundles on $W_j$, projecting to an isomorphism on $E_j'/F_j'$. Define $G_j'=(\nabla^{E_j^x}s_j^x)(\nu_j)$, a vector subbundle of $E_j'$. Then $E_j'=F_j'\op G_j'$. Write $s_j\vert_{W_j}=t_j'\op u_j'$ for $t_j'\in C^\iy(F_j')$, $u_j'\in C^\iy(G_j')$.

Choose a vector subbundle $G_j$ of $E_j^x$ such that $E_j^x=F_j\op G_j$ and $G_j\vert_{W_j}=G_j'$. Write $s_j^x=t_j\op u_j$ for $t_j\in C^\iy(F_j)$, $u_j\in C^\iy(G_j)$, and $\nabla^{F_j'},\nabla^{G_j'},\nabla^{F_j},\nabla^{G_j}$ for the connections on $F_j',G_j',F_j,G_j$ induced by $\nabla^{E_j}$. Making $V_{ji}^x$ smaller if necessary, choose a smooth map $\phi_{ji}^x:V_{ji}^x\ra V_i$ such that $\phi_{ji}^x\vert_{W_j}=\xi_{ji}:W_j\ra W_i\subseteq V_i$, and $\phi_{ji}^x$ is otherwise arbitrary. 

Define vector bundle morphisms $\al,\be,\ga,\de,\ep$ on $W_i$ by
\ea
\d\phi_{ij}^x\vert_{W_i}&=\begin{pmatrix} \d\xi_{ij} & \al \\ 0 & \be \end{pmatrix}: \begin{matrix} TW_i \\ \op\,\nu_i \end{matrix} \longra 
\begin{matrix} \xi_{ij}^*(TW_j) \\ \op\,\xi_{ij}^*(\nu_j), \end{matrix}
\label{ku6eq23}\\
\hat\phi_{ij}^x\vert_{W_i}&=\begin{pmatrix} \hat\xi_{ij} & \ga \\ 0 & \de \end{pmatrix}: \begin{matrix} F_i' \\ \op\, G_i' \end{matrix} \longra 
\begin{matrix} \xi_{ij}^*(F_j') \\ \op\,\xi_{ij}^*(G'_j), \end{matrix}
\label{ku6eq24}\\
\nabla^{E_i^x}s_i\vert_{W_i}&=\begin{pmatrix} \nabla^{F_i'}t_i' & 0 \\ 0 & \ep \end{pmatrix}: \begin{matrix} TW_i \\ \op\,\nu_i \end{matrix} \longra 
\begin{matrix} F'_i \\ \op\,G_i'. \end{matrix}
\label{ku6eq25}
\ea
Here by construction $\d\xi_{ij}:TW_i\ra\xi_{ij}^*(TW_j)$, $\hat\xi_{ij}:F_i'\ra\xi_{ij}^*(F_j')$ and $\ep:\nu_i\ra G_i'$ are isomorphisms. Define vector bundle morphisms $\eta,\ze,\th$ on $W_j$ by
\ea
\d\phi_{ji}^x\vert_{W_j}&=\begin{pmatrix} \d\xi_{ji} & \eta \\ 0 & \ze \end{pmatrix}: \begin{matrix} TW_j \\ \op\,\nu_j \end{matrix} \longra 
\begin{matrix} \xi_{ji}^*(TW_i) \\ \op\,\xi_{ji}^*(\nu_i), \end{matrix}
\label{ku6eq26}\\
\nabla^{E_j^x}s_j\vert_{W_j}&=\begin{pmatrix} \nabla^{F_j'}t_j' & 0 \\ \nabla^{G_j'}u_j' & \th \end{pmatrix}: \begin{matrix} TW_j \\ \op\,\nu_j \end{matrix} \longra 
\begin{matrix} F'_j \\ \op\,G_j'. \end{matrix}
\label{ku6eq27}
\ea
Here by construction $\d\xi_{ji}:TW_j\ra\xi_{ji}^*(TW_i)$ is an isomorphism, with inverse $\xi_{ji}^*(\d\xi_{ij})$, and $\th:\nu_j\ra G_j'$ is an isomorphism.

Definition \ref{ku2def3}(d) for $\Phi_{ij}$ gives $\hat\phi_{ij}(s_i\vert_{V_{ij}})=\phi_{ij}^*(s_j)+O(s_i^2)$. Restricting this to $W_i$ and using $s_i\vert_{W_i}=t_i'\op 0$, $s_j\vert_{W_j}=t_j'\op u_j'$ and \eq{ku6eq23}--\eq{ku6eq24} gives
\e
\hat\xi_{ij}(t_i')=\xi_{ij}^*(t_j')+O((t_i')^2) \quad\text{and}\quad
0=\xi_{ij}^*(u_j')+O((t_i')^2).
\label{ku6eq28}
\e
Since $\xi_{ji}=\xi_{ij}^{-1}$ and $\hat\xi_{ji}=\xi_{ji}^*(\hat\xi_{ij})^{-1}$, applying $\xi_{ji}^*$ to \eq{ku6eq28} implies that
\e
\hat\xi_{ji}(t_j')=\xi_{ji}^*(t_i')+O((t_j')^2) \quad\text{and}\quad
u_j'=O((t_j')^2).
\label{ku6eq29}
\e

Differentiating $\hat\phi_{ij}(s_i\vert_{V_{ij}})=\phi_{ij}^*(s_j)+O(s_i^2)$ and restricting to $W_i$ yields
\e
\hat\phi_{ij}^x\vert_{W_i}\!\ci\! \nabla^{E_i^x}s_i\vert_{W_i}\!=\!\xi_{ij}^*(\nabla^{E_j^x}s_j\vert_{W_j})\!\ci\! \d\phi_{ij}^x\vert_{W_i}\!+\!O(\nabla^{E_i^x}s_i\vert_{W_i}\cdot t_i')\!+\!O((t_i')^2).
\label{ku6eq30}
\e
Substituting in \eq{ku6eq23}, \eq{ku6eq24}, \eq{ku6eq25} and \eq{ku6eq27} yields
\ea
\hat\xi_{ij}\ci\nabla^{F_i'}t_i'&=\xi_{ij}^*(\nabla^{F_j'}t_j')\ci\d\xi_{ij}\!+\!O(\nabla^{E_i^x}s_i\vert_{W_i}\cdot t_i')\!+\!O((t_i')^2), 
\label{ku6eq31}\\ 
\ga\ci\ep&=\xi_{ij}^*(\nabla^{F_j'}t_j')\ci\al\!+\!O(\nabla^{E_i^x}s_i\vert_{W_i}\cdot t_i')\!+\!O((t_i')^2), 
\label{ku6eq32}\\
0&=\xi_{ij}^*(\nabla^{G_j'}u_j')\ci\d\xi_{ij}\!+\!O(\nabla^{E_i^x}s_i\vert_{W_i}\cdot t_i')\!+\!O((t_i')^2), 
\label{ku6eq33}\\
\de\ci\ep&=\xi_{ij}^*(\nabla^{G_j'}u_j')\ci\al\!+\!\xi_{ij}^*(\th)\ci\be\!+\!O(\nabla^{E_i^x}s_i\vert_{W_i}\cdot t_i')\!+\!O((t_i')^2).
\label{ku6eq34}
\ea

Choose a vector bundle morphism $\hat\phi_{ji}^x:E_j\ra(\phi_{ji}^x)^*(E_i)$ on $V_{ji}^x$ satisfying
\e
\hat\phi_{ji}^x\vert_{W_j}=\begin{pmatrix} \hat\xi_{ji} & \xi_{ji}^*(\nabla^{F_i'}t_i')\ci\eta\ci\th^{-1} \\ 0 & \xi_{ji}^*(\ep)\ci\ze\ci\th^{-1} \end{pmatrix}: \begin{matrix} F_j' \\ \op\, G_j' \end{matrix} \longra 
\begin{matrix} \xi_{ij}^*(F_j') \\ \op\,\xi_{ij}^*(G'_j), \end{matrix}
\label{ku6eq35}
\e
where $\hat\phi_{ji}^x$ is arbitrary except on $W_j$. We claim that
\e
\hat\phi_{ji}^x(s_j\vert_{V_{ji}^x})=(\phi_{ji}^x)^*(s_i)+O(s_j^2)\quad\text{on $V_{ji}^x$.}
\label{ku6eq36}
\e
To see this, first note that the restriction of \eq{ku6eq36} to $W_j$ is
\e
\begin{pmatrix} \d\xi_{ji} & \eta \\ 0 & \ze \end{pmatrix}\begin{pmatrix} t_j' \\ u_j' \end{pmatrix} = \begin{pmatrix} \xi_{ji}^*(t_i') \\ 0 \end{pmatrix} +O\bigl((t_j'\op u_j')^2\bigr),
\label{ku6eq37}
\e
which holds by \eq{ku6eq29}. As for \eq{ku6eq30}, the derivative of \eq{ku6eq36} restricted to $W_j$ is
\e
\hat\phi_{ji}^x\vert_{W_j}\!\ci\! \nabla^{E_j^x}s_j\vert_{W_j}\!=\!\xi_{ji}^*(\nabla^{E_i^x}s_i\vert_{W_i})\ci \d\phi_{ji}^x\vert_{W_j}\!+\!O(\nabla^{E_j^x}s_j\vert_{W_j}\!\cdot t_j')\!+\!O((t_j')^2),
\label{ku6eq38}
\e
which holds by \eq{ku6eq25}, \eq{ku6eq26}, \eq{ku6eq27}, \eq{ku6eq31}, \eq{ku6eq33}, and \eq{ku6eq35}. Write $I_{W_j}\subset C^\iy(V_j^x)$ for the ideal of smooth functions on $V_j^x$ vanishing on $W_j$, and $I_{s_j^x}\subset C^\iy(V_j^x)$ the ideal generated by $s_j^x$. Equation \eq{ku6eq37} implies that \eq{ku6eq36} holds up to $I_{W_j}$, the ideal of functions on $V_j^x$ vanishing on $W_j$. Together with this, \eq{ku6eq38} implies that \eq{ku6eq36} holds up to $I_{W_j}^2$. But $I_{W_j}\subseteq I_{s_j^x}$, so $I_{W_j}^2$ is contained in the $O(s_j^2)$ in \eq{ku6eq36}, and thus \eq{ku6eq36} holds. 

Equation \eq{ku6eq36} implies that $(V_{ji}^x,\phi_{ji}^x,\hat\phi_{ji}^x)$ satisfies Definition \ref{ku2def3}(a)--(e) over $S^x$, so that $\Phi_{ji}^x:=[V_{ji}^x,\phi_{ji}^x,\hat\phi_{ji}^x]:(V_j,E_j,s_j,\psi_j)\ra(V_i,E_i,s_i,\psi_i)$ is a morphism of $\mu$-Kuranishi neighbourhoods over~$x\in S^x\subseteq X$. 

Choose vector bundle morphisms $\La_{ii}:E_i^x\ra TV_{ij}^x$ on $V_{ij}^x$ and $\Mu_{jj}:E_j^x\ra TV_{ji}^x$ on $V_{ji}^x$ satisfying
\ea
\La_{ii}\vert_{W_i}&=\begin{pmatrix} 0 & \xi_{ij}^*(\d\xi_{ji})\ci\al\ci\ep^{-1}+\xi_{ij}^*(\eta)\ci\be\ci\ep^{-1} \\ 0 & \xi_{ij}^*(\ze)\ci\be\ci\ep^{-1}-\ep^{-1} \end{pmatrix}: \begin{matrix} F_i' \\ \op\, G_i' \end{matrix} \longra 
\begin{matrix} TW_i \\ \op\,\nu_i, \end{matrix}
\label{ku6eq39}\\
\Mu_{jj}\vert_{W_j}&=\begin{pmatrix} 0 & \xi_{ji}^*(\d\xi_{ij})\ci\eta\ci\th^{-1}+\xi_{ji}^*(\al)\ci\ze\ci\th^{-1} \\ 0 & \xi_{ji}^*(\be)\ci\ze\ci\th^{-1}-\th^{-1} \end{pmatrix}: \begin{matrix} F_j' \\ \op\, G_j' \end{matrix} \longra 
\begin{matrix} TW_j \\ \op\,\nu_j, \end{matrix}
\label{ku6eq40}
\ea
where $\La_{ii},\Mu_{jj}$ are arbitrary away from $W_i,W_j$. In a similar way to the proof of \eq{ku6eq36}, using \eq{ku6eq23}--\eq{ku6eq35} and \eq{ku6eq39}--\eq{ku6eq40} we can show that
\ea
\begin{split}
\phi_{ji}^x\ci\phi_{ij}^x&=\id_{V_{ij}^x}+\La_{ii}\cdot s_i+O(s_i^2)\;\>\text{and}\\
(\phi_{ij}^x)^*(\hat\phi_{ji}^x)\!\ci\!\hat\phi_{ij}^x&=\id_{E_i^x}\!+\!\La_{ii}\!\cdot\!\d s_i\!+\!O(s_i)\;\> \text{on $\dot V_{ii}\!=\!V_{ij}^x\!\cap\!(\phi_{ij}^x)^{-1}(V_{ji}^x)$,}
\end{split}
\label{ku6eq41}\\
\begin{split}
\phi_{ij}^x\ci\phi_{ji}^x&=\id_{V_{ji}^x}+\Mu_{jj}\cdot s_j+O(s_j^2)\;\>\text{and}\\
(\phi_{ji}^x)^*(\hat\phi_{ij}^x)\!\ci\!\hat\phi_{ji}^x&=\id_{E_j^x}\!+\!\Mu_{jj}\!\cdot\!\d s_j\!+\!O(s_j)\;\> \text{on $\dot V_{jj}\!=\!V_{ji}^x\!\cap\!(\phi_{ji}^x)^{-1}(V_{ij}^x)$.}
\end{split}
\label{ku6eq42}
\ea

As in \eq{ku6eq16}, equations \eq{ku6eq41}--\eq{ku6eq42} give $\Phi_{ji}^x\ci\Phi_{ij}\vert_{S^x}=\id_{(V_i,E_i,s_i,\psi_i)}\vert_{S^x}$ and $\Phi_{ij}\vert_{S^x}\ci\Phi_{ji}^x=\id_{(V_j,E_j,s_j,\psi_j)}\vert_{S^x}$. So $\Phi_{ij}\vert_{S^x}$ is invertible, and a $\mu$-coordinate change on $S^x$. We have shown that every $x\in S$ has an open neighbourhood $S^x$ in $S$ such that $\Phi_{ij}\vert_{S^x}$ is a $\mu$-coordinate change on $S^x$. Thus by the sheaf property of $\cIso\bigl((V_i,E_i,s_i,\psi_i),(V_j,E_j,s_j,\psi_j)\bigr)$ in Theorem \ref{ku2thm1}(a), proved in \S\ref{ku612}, $\Phi_{ij}$ is a $\mu$-coordinate change on $S$. This completes the `if' part, and so the whole of Theorem~\ref{ku2thm2}.

\subsection{Proof of Theorem \ref{ku3thm2}}
\label{ku63}

Next we explain how to generalize the proof in \S\ref{ku61} to $(V_i,\ab E_i,\ab s_i,\ab\psi_i),\ab(V_j,\ab E_j,\ab s_j,\ab\psi_j)$ $\mu$-Kuranishi neighbourhoods with corners, and so prove Theorem \ref{ku3thm2}. All of $V_i,V_j,V_{ij},V_{ij}',\dot V_{ij}^a,\ldots$ now become manifolds with corners, and we replace $TV_j$ by ${}^bTV_j$ in the definition of $\La,\La^a,\La^{ab},\ldots,$ as in~\S\ref{ku35}. 

Apart from these, the only nontrivial change is the proof in \S\ref{ku611} that $\cHom((V_i,E_i,s_i,\psi_i),(V_j,E_j,s_j,\psi_j))$ satisfies Definition \ref{ku2def9}(iv), where we construct $\phi_{ij}:V_{ij}\ra V_j$ by combining the maps $\phi_{ij}^a:V_{ij}^a\ra V_j$ using $\{\eta^a:a\in A\}$, as in \eq{ku6eq5}, \eq{ku6eq10}, and Lemma \ref{ku6lem}. The problem is that Lemma \ref{ku6lem} in the form above is {\it false\/} if $\pd V_j\ne\es$. (To see this, take $V_j=[0,\iy)$, and try and define $\Xi_2:W_2\ra[0,\iy)$ near $(0,0)\t\De_1$ in $W_2$ satisfying Lemma \ref{ku6lem}(ii),(iii).) Even if Lemma \ref{ku6lem} were true with corners, \eq{ku6eq6} is not what we need to prove \eq{ku6eq13} and \eq{ku6eq15}, given the new interpretation of $O(s_i),O(s_i^2)$ notation in Definition~\ref{ku3def11}.

Consider what happens if we try to take $V_j$ a manifold with corners in the proof of Lemma \ref{ku6lem}. The first proof, method (A), immediately runs into trouble when $\pd V_j\ne\es$: if we take $T$ an open neighbourhood of $i(V_j)$ in $\R^N$ then we cannot define smooth $\pi:T\ra V_j$ near $i(\pd V_j)$ with $\pi\ci i=\id$, and if we let $T\subset\R^N$ be an $N$-manifold with corners then $T$ is not open in $\R^N$, so $W_n$ in \eq{ku6eq7} is not open in $V_j^n\t\De_{n-1}$ (nor a manifold with corners, in general).

The second proof, method (B), can be adapted to the corners case, and in doing so we find out how to rewrite the statement of Lemma \ref{ku6lem}. The key, as in Principle \ref{ku3princ}, is to use the b-(co)tangent bundles ${}^bTV_j,{}^bT^*V_j$ in place of $TV_j,T^*V_j$ everywhere we can. 

On a manifold without boundary $V_j$, a Riemannian metric $g$ is a smooth section of $\Sym^2T^*V_j$ giving a positive definite, nondegenerate quadratic form on $TV_j$. When $V_j$ has corners, we instead take $g$ to be a smooth section of $\Sym^2({}^bT^*V_j)$ giving a positive definite, nondegenerate quadratic form on ${}^bTV_j$. This is what Melrose calls a {\it b-metric\/} \cite{Melr2,Melr3}. An example of a b-metric on $\R^n_k=[0,\iy)^k\t\R^{n-k}$ is $x_1^{-2}\d x_1^2+\cdots+x_k^{-2}\d x_k^2+\d x_{k+1}^2+\cdots+\d x_n^2$; note that it has poles on the boundary faces $x_i=0$ for $i=1,\ldots,k$, in conventional coordinates. The restriction of $g$ to the interior $V_j^\ci$ is an ordinary Riemannian metric, and with respect to this, the boundary $\pd V_j$ is at infinite distance from any point in~$V_j^\ci$.

Suppose $v\in S^k(V_j)\subseteq V_j$, for $S^k(V_j)$ the depth $k$ stratum of $V_j$ as in \S\ref{ku32}. Then $I_{V_j}\vert_v:{}^bT_vV_j\ra T_vV_j$ has image $T_vS^k(V_j)\subseteq T_vV_j$ of dimension $m-k$, where $m=\dim V_j$, and $g\vert_v$ induces a Euclidean metric on $T_vS^k(V_j)$ (though not on $T_vV_j$). Write $B_{v,r(v)}$ for the open ball about 0 of radius $r(v)>0$ in $T_vS^k(V_j)$. Then for a small continuous function $r:V_j\ra(0,\iy)$ we can again define the exponential map $\exp_v:B_{v,r(v)}\ra S^k(V_j)\subseteq V_j$, and $\exp_v$ is a diffeomorphism with a small open ball about $v$ in $S^k(V_j)$ (not $V_j$) for small $r(v)$.

We now define $Y_n\subseteq\bigop^nTV_j\t\De_{n-1}$ and $\Pi:Y_n\ra V_j^n\t\De_{n-1}$, $\pi:Y_n\ra V_j$ as in \eq{ku6eq8}--\eq{ku6eq9}. In general $Y_n$ is no longer a manifold, since the dimensions of the $B_{v,r(v)}$ vary with $v$. However, it is still true that $\Pi$ is injective provided $r:V_j\ra(0,\iy)$ is small enough, so we can again define $W_n=\Im\Pi$ and $\Xi_n:W_n\ra V_j$ by $\Xi_n=\pi\ci\Pi^{-1}$, as in method (B) in the proof of Lemma~\ref{ku6lem}.

Lemma \ref{ku6lem}(iii)--(iv) hold for these $W_n,\Xi_n$ as before, but Lemma \ref{ku6lem}(i),(ii) do not: for (i), $W_n$ is generally not open near $(v,\ldots,v)\t\De_{n-1}$ for $v\in S^k(V_j)\subset V_j$ with $k\ge 1$, and in (ii), equation \eq{ku6eq6} is false near $x^i_k=0$ for local coordinates $x_k$ in $(x_1,\ldots,x_m)$ taking values in $[0,\iy)$ rather than $\R$. By considering local models $\R^m_k$ for $V_j$, we find that Lemma \ref{ku6lem} holds in the corners case if we replace parts (i),(ii) by (i$)'$,(ii$)'$, where:
\begin{itemize}
\setlength{\itemsep}{0pt}
\setlength{\parsep}{0pt}
\item[(i$)'$] Let $v\in S^k(V_j)\subseteq V_j$, and choose local coordinates $(x_1,\ldots,x_m)$ on $V_j$ near $v$ with $v=(0,\ldots,0)$, where $x_1,\ldots,x_k\in[0,\iy)$ and $x_{k+1},\ldots,x_m\in\R$. Then there exists small $\ep>0$ such that writing points $v^a=(x_1^a,\ldots,x_m^a)$ near $v$ in coordinates in $V_j$, for all $n\ge 1$ we have
\e
\begin{split}
&\bigl\{\bigl((x_1^1,\ldots,x_m^1),\ldots,(x_1^n,\ldots,x_m^n),t^1,\ldots,t^n\bigr):x_1^a,\ldots,x_k^a\in[0,\iy),\\
&x_{k+1}^a,\ldots,x_m^a\in\R,\;\>
\md{x_l^a-x_l^b}<\ep\md{x_l^a},\;\> \md{x_{l'}^a}<\ep,\;\> a,b=1,\ldots,n,\\ &
l=1,\ldots,k, \;\> l'=1,\ldots,m,\;\>(t^1,\ldots,t^n)\in\De_{n-1}\bigr\}
\subseteq W_n.
\end{split}
\label{ku6eq43}
\e
\item[(ii$)'$] Let\/ $(x_1,\ldots,x_m)$ be local coordinates on an open set $U\subseteq V_j$, where $x_1,\ldots,x_k\in[0,\iy)$ and $x_{k+1},\ldots,x_m\in\R$. Then writing
\begin{align*}
\Xi_n\bigl((x_1^1,\ldots,x_m^1),\ldots,(x_1^n,\ldots,x_m^n),t^1,\ldots,t^n\bigr)&=(y_1,\ldots,y_m),\\
\bigl(t^1x_1^1+\cdots+t^nx_1^n,\ldots,t^1x_m^1+\cdots+t^nx_m^n\bigr)&=(z_1,\ldots,z_m),
\end{align*}
we have
\end{itemize}
\ea
&\md{y_{l'}-z_{l'}}=
\label{ku6eq44}\\
&\begin{cases} \md{y_{l'}}\cdot O\raisebox{-1.5pt}{$\displaystyle\biggl[$}\ts\sum\limits_{a,b=1}^n\sum\limits_{l=1}^kt^at^b\md{x^a_l}^{-2}(x^a_l\!-\!x^b_l)^2\!+\!\!\sum\limits_{a,b=1}^n\sum\limits_{l=k+1}^mt^at^b(x^a_l\!-\!x^b_l)^2\raisebox{-1.5pt}{$\displaystyle\biggr]$}, & l'\!\le\!k, \\
O\raisebox{-1.5pt}{$\displaystyle\biggl[$}\ts\sum\limits_{a,b=1}^n\sum\limits_{l=1}^kt^at^b\md{x^a_l}^{-2}(x^a_l\!-\!x^b_l)^2+\sum\limits_{a,b=1}^n\sum\limits_{l=k+1}^mt^at^b(x^a_l\!-\!x^b_l)^2\raisebox{-1.5pt}{$\displaystyle\biggr]$}, & l'\!>\!k.\end{cases}
\nonumber
\ea

Note that $\md{x_l^a-x_l^b}<\ep\md{x_l^a}$ in \eq{ku6eq43} is not an open condition, so the l.h.s.\ of \eq{ku6eq43} is not an open neighbourhood of $(v,\ldots,v)\t\De_{n-1}$. Also note the $\md{y_{l'}}$ and $\md{x^a_l}^{-2}$ factors in the $O[\cdots]$ term in \eq{ku6eq44}, which are not present in~\eq{ku6eq6}.

Now consider the remainder of the proof of Definition \ref{ku2def9}(iv) for $\cHom((V_i,\ab E_i,\ab s_i,\ab\psi_i),(V_j,E_j,s_j,\psi_j))$, from before \eq{ku6eq10} to after \eq{ku6eq15}. This all works in the corners case, with Lemma \ref{ku6lem} modified as above. The main points are these:
\begin{itemize}
\setlength{\itemsep}{0pt}
\setlength{\parsep}{0pt}
\item[(a)] We can still choose an open neighbourhood $V_{ij}$ of $\psi_i^{-1}(S)$ in $\bigcup_{a\in A}V_{ij}^a\subseteq V_i$ with $\bigl(\phi_{ij}^{a_1}(v),\ldots,\phi_{ij}^{a_n}(v),\eta^{a_1}(v),\ldots,\eta^{a_n}(v)\bigr)\in W_n$ for each $v\in V_{ij}$, despite the fact that $W_n$ may not be open near $(v,\ldots,v)\t\De_{n-1}$. 

This is because in the first equation of \eq{ku6eq4}, the fact that $\La^{ab}$ maps to ${}^bTV_j$ rather than $TV_j$ means that if $\phi_{ij}^a(v)=(x_1^a,\ldots,x_m^a)$ and $\phi_{ij}^b(v)=(x_1^b,\ldots,x_m^b)$ in local coordinates $(x_1,\ldots,x_m)$ as in (i$)'$, then the estimates we get from \eq{ku6eq4} are of the form
\e
\md{x_l^a-x_l^b}=\begin{cases} \md{x_l^a}\cdot O(s_i), & l=1,\ldots,k,\\
O(s_i), & l=k+1,\ldots,m,
\end{cases}
\label{ku6eq45}
\e
and the extra factor $\md{x_l^a}$ in the top line of \eq{ku6eq45} compensates for the $\md{x_l^a}$ in $\md{x_l^a-x_l^b}<\ep\md{x_l^a}$ in \eq{ku6eq43}. 
\item[(b)] The proof of \eq{ku6eq13} used \eq{ku6eq6}, which is now replaced by \eq{ku6eq44}. But we can still prove \eq{ku6eq13}, for the same reason as in (a), that is, the extra $\md{x^a_l}^{-2}$ in \eq{ku6eq44} are compensated for by the $\md{x_l^a}$ in \eq{ku6eq45}, squared.
\end{itemize}
This completes the proof of Theorem~\ref{ku3thm2}.

\subsection{Proof of Theorem \ref{ku3thm3}}
\label{ku64}

\subsubsection{The `only if' part of Theorem \ref{ku3thm3}}
\label{ku641}

Let $\Phi_{ij}:(V_i,E_i,s_i,\psi_i)\ra (V_j,E_j,s_j,\psi_j)$ be a $\mu$-coordinate change with corners over $S\subseteq X$, and let $(V_{ij},\phi_{ij},\hat\phi_{ij})$ represent $\Phi_{ij}$. Then Proposition \ref{ku3prop2} says that $\phi_{ij}:V_{ij}\ra V_j$ is simple. Let $x\in S,$ and set $v_i=\psi_i^{-1}(x)\in V_i$ and $v_j=\psi_j^{-1}(x)\in V_j$, so that $\phi_{ij}(v_i)=v_j$. We prove \eq{ku3eq23} is exact as for exactness of \eq{ku2eq8} in \S\ref{ku621}, but replacing $TV_i,TV_j,\d s_i,\d s_j,\d\phi_{ij}$ by ${}^bTV_i,\ab{}^bTV_j,\ab{}^b\d s_i,\ab{}^b\d s_j,\ab{}^b\d\phi_{ij}$ throughout. This proves the `only if' part of Theorem~\ref{ku3thm3}.

\subsubsection{The `if' part of Theorem \ref{ku3thm3}}
\label{ku642}

Let $\Phi_{ij}:(V_i,E_i,s_i,\psi_i)\ra(V_j,E_j,s_j,\psi_j)$ be a morphism of $\mu$-Kuranishi neighbourhoods with corners over $S\subseteq X$, represented by $(V_{ij},\phi_{ij},\hat\phi_{ij})$, and suppose that $\phi_{ij}$ is simple and \eq{ku3eq23} is exact for all $x\in S$.

We prove that for each $x\in S$, there exists an open neighbourhood $S^x$ of $x$ in $S$ such that $\Phi_{ij}\vert_{S^x}$ has an inverse $\Phi_{ji}^x$ over $S^x$. The `if' part will then follow from the sheaf property of $\cIso\bigl((V_i,E_i,s_i,\psi_i),(V_j,E_j,s_j,\psi_j)\bigr)$. The proof follows \S\ref{ku622}, with the following differences. We replace $TV_i,TV_j,\ab TW_i,\ab TW_j,\ab\d s_i,\ab\d s_j,\ab\d\phi_{ij},\ab\ldots$ by ${}^bTV_i,\ab{}^bTV_j,{}^bTW_i,{}^bTW_j,\ab{}^b\d s_i,\ab{}^b\d s_j,\ab{}^b\d\phi_{ij},\ldots$ throughout.

Apart from this, the one new issue is at the point in the paragraph before \eq{ku6eq23} when we choose a smooth map $\phi_{ji}^x:V_{ji}^x\ra V_i$, such that $\phi_{ji}^x\vert_{W_j}=\xi_{ji}:W_j\ra W_i\subseteq V_i$. As $V_{ji}^x,V_j,W_i,W_j$ are now manifolds with corners, we must justify that it is possible to choose $\phi_{ji}^x$, since this is not obvious.

With the notation of \S\ref{ku622}, write $m=\dim V_i$, $n=\dim V_j$, and let $v_i\in S^k(V_{ij})$ for $k=0,\ldots,m$, in the sense of \S\ref{ku32}. Since $\phi_{ij}:V_{ij}\ra V_j$ is simple, by assumption, and $\phi_{ij}(v_i)=v_j$, it follows that $v_j\in S^k(V_j)$. Thus $V_{ij}^x$ is locally modelled on $\R^m_k=[0,\iy)^k\t\R^{m-k}$ near $v_i$, and $V_{ji}^x$ locally modelled on $\R^n_k=[0,\iy)^k\t\R^{n-k}$ near~$v_j$.

Choose coordinates $(y_1,\ldots,y_n)\in\R^n_k$ on $V_j$ near $v_j$, with $v_j=(0,\ldots,0)$. Making $V_{ij}^x,W_i,V_{ji}^x,W_j$ smaller if necessary, we can suppose $V_{ji}^x$ is contained in the domain of $(y_1,\ldots,y_n)$. Then $y_1,\ldots,y_k$ are `boundary defining functions' for the $k$ local boundary components of $V_j$ at $v_j$. Since $\phi_{ij}$ is simple with $\phi_{ij}(v_i)=v_j$, it follows that $y_1\ci\phi_{ij},\ldots,y_k\ci\phi_{ij}$ are boundary defining functions for $V_{ij}$ at $v_i$. Therefore we can choose coordinates $(x_1,\ldots,x_m)\in\R^m_k$ on $V_{ij}$ near $v_i$, with $v_i=(0,\ldots,0)$, such that $\phi_{ij}$ maps the domain of $(x_1,\ldots,x_m)$ to the domain of $(y_1,\ldots,y_n)$, and 
$x_i=y_i\ci\phi_{ij}$ for $i=1,\ldots,k$. Making $V_{ij}^x,W_i,V_{ji}^x,W_j$ smaller if necessary, we can suppose $V_{ij}^x$ is contained in the domain of~$(x_1,\ldots,x_m)$.

The map $\phi_{ij}^x:V_{ij}^x\ra V_j$ may be written in coordinates as
\begin{align*}
\phi_{ij}^x:(x_1&,\ldots,x_m)\longmapsto(y_1,\ldots,y_n)\\
&=\bigl(x_1,\ldots,x_k,f_{k+1}(x_1,\ldots,x_m),\ldots,f_n(x_1,\ldots,x_m)\bigr).
\end{align*}
We can now choose $\phi_{ji}^x:V_{ji}^x\ra V_i$ to have the form
\begin{align*}
\phi_{ji}^x:(y_1&,\ldots,y_n)\longmapsto(x_1,\ldots,x_m)\\
&=\bigl(y_1,\ldots,y_k,g_{k+1}(y_1,\ldots,y_n),\ldots,g_m(y_1,\ldots,y_n)\bigr),
\end{align*}
for $g_{k+1},\ldots,g_m$ smooth functions of $y_1,\ldots,y_n$. This is consistent with $\phi_{ji}^x\vert_{W_j}=\xi_{ji}:W_j\ra W_i\subseteq V_i$, and yields a simple map of manifolds with corners. So it is possible to choose $\phi_{ji}^x$ with the properties we require. The rest of the proof in \S\ref{ku622} works without difficulty. This completes the proof of Theorem~\ref{ku3thm3}.

\section{Proofs of main theorems in \S\ref{ku4} and \S\ref{ku5}}
\label{ku7}

\subsection{Proof of Theorem \ref{ku4thm1}}
\label{ku71}

\subsubsection{\!\!\!Theorem \ref{ku4thm1}(a):$\,\bcHom((V_i,E_i,\Ga_i,s_i,\psi_i),(V_j,E_j,\Ga_j,s_j,\psi_j))$ a stack}
\label{ku711}

Let $X$ be a topological space, $(V_i,E_i,\Ga_i,s_i,\psi_i),$ $(V_j,E_j,\Ga_j,s_j,\psi_j)$ be Kuranishi neighbourhoods on $X$, and $\bcHom\bigl((V_i,E_i,\Ga_i,s_i,\psi_i),(V_j,E_j,\Ga_j,s_j,\psi_j)\bigr)$ be as in Theorem \ref{ku4thm1}(a). We must show $\bcHom\bigl((V_i,E_i,\Ga_i,s_i,\psi_i),(V_j,E_j,\Ga_j,s_j,\psi_j)\bigr)$ is a stack on $\Im\psi_i\cap\Im\psi_j$, that is, that it satisfies Definition \ref{ku4def11}(i)--(v). Parts (i),(ii) are immediate from the definition of restriction $\vert_T$ in Definition~\ref{ku4def9}.

\subsubsection*{Definition \ref{ku4def11}(iii) for $\bcHom((V_i,E_i,\Ga_i,s_i,\psi_i),(V_j,E_j,\Ga_j,s_j,\psi_j))$.}
For (iii), let $S\subseteq\Im\psi_i\cap\Im\psi_j$ be open, $\Phi_{ij}=(P_{ij},\pi_{ij},\phi_{ij},\hat\phi_{ij})$ and $\Phi_{ij}'=(P_{ij}',\pi_{ij}',\phi_{ij}',\hat\phi_{ij}')$ be 1-morphisms $(V_i,E_i,\Ga_i,s_i,\psi_i)\ra(V_j,E_j,\Ga_j,s_j,\psi_j)$ over $S$, and $\La_{ij},\La_{ij}':\Phi_{ij}\Ra\Phi_{ij}'$ be 2-morphisms over $S$. Suppose $\{T^a:a\in A\}$ is an open cover of $S$, such that $\La_{ij}\vert_{T^a}=\La_{ij}'\vert_{T^a}$ for all $a\in A$. Choose representatives $(\dot P_{ij},\la_{ij},\hat\la_{ij}),(\dot P_{ij}',\la_{ij}',\hat\la_{ij}')$ for $\La_{ij},\La_{ij}'$. Then $\La_{ij}\vert_{T^a}=\La_{ij}'\vert_{T^a}$ means as in \eq{ku4eq2} that there exists an open neighbourhood $\ddot P_{ij}^a$ of $\pi_{ij}^{-1}(\bar\psi_i^{-1}(T^a))$ in $\dot P_{ij}\cap \dot P_{ij}'$ with
\e
\la_{ij}\vert_{\ddot P_{ij}^a}=\la_{ij}'\vert_{\ddot P_{ij}^a}\quad\text{and}\quad \hat\la_{ij}\vert_{\ddot P_{ij}^a}=\hat\la_{ij}'\vert_{\ddot P_{ij}^a}+O\bigl(\pi_{ij}^*(s_i)\bigr)\quad\text{on $\ddot P_{ij}^a$.}
\label{ku7eq1}
\e
Set $\ddot P_{ij}=\bigcup_{a\in A}\ddot P_{ij}^a$, an open neighbourhood of $\pi_{ij}^{-1}(\bar\psi_i^{-1}(S))$ in $\dot P_{ij}\cap \dot P_{ij}'$. Then \eq{ku7eq1} for all $a\in A$ implies \eq{ku4eq2} on $\ddot P_{ij}$, so $\La_{ij}=\La_{ij}'$. This proves Definition \ref{ku4def11}(iii) for~$\bcHom\bigl((V_i,E_i,\Ga_i,s_i,\psi_i),(V_j,E_j,\Ga_j,s_j,\psi_j)\bigr)$.

\subsubsection*{Definition \ref{ku4def11}(iv) for $\bcHom((V_i,E_i,\Ga_i,s_i,\psi_i),(V_j,E_j,\Ga_j,s_j,\psi_j))$.}
For (iv), suppose $S,\Phi_{ij},\Phi_{ij}'$ are as above, $\{T^a:a\in A\}$ is an open cover of $S$, and $\La_{ij}^a:\Phi_{ij}\vert_{T^a}\Ra\Phi_{ij}'\vert_{T^a}$ are 2-morphisms over $T^a$ for $a\in A$ with $\La_{ij}^a\vert_{T^a\cap T^b}=\La_{ij}^b\vert_{T^a\cap T^b}$ for all $a,b\in A$. Choose representatives $(\dot P_{ij}^a,\la_{ij}^a,\hat\la_{ij}^a)$ for $\La_{ij}^a$ for $a\in A$, and making $\dot P_{ij}^a$ smaller if necessary, suppose that $\dot P_{ij}^a\cap\pi_{ij}^{-1}(s_i^{-1}(0))=\pi_{ij}^{-1}(\bar\psi_i^{-1}(T^a))$. Then $\La_{ij}^a\vert_{T^a\cap T^b}=\La_{ij}^b\vert_{T^a\cap T^b}$ means there exists an open neighbourhood $\ddot P_{ij}^{ab}$ of $\pi_{ij}^{-1}(\bar\psi_i^{-1}(T^a\cap T^b))$ in $\dot P_{ij}^a\cap \dot P_{ij}^b$ with
\e
\la_{ij}^a\vert_{\ddot P_{ij}^{ab}}=\la_{ij}^b\vert_{\ddot P_{ij}^{ab}}\quad\text{and}\quad \hat\la_{ij}^a\vert_{\ddot P_{ij}^{ab}}=\hat\la_{ij}^b\vert_{\ddot P_{ij}^{ab}}+O\bigl(\pi_{ij}^*(s_i)\bigr)\quad\text{on $\ddot P_{ij}^{ab}$.}
\label{ku7eq2}
\e
By the argument after \eq{ku6eq4}, we can suppose the second equation of \eq{ku7eq2} holds on~$\dot P_{ij}^a\cap \dot P_{ij}^b$.

Choose a partition of unity $\{\eta^a:a\in A\}$ on $\bigcup_{a\in A}\dot P_{ij}^a\subseteq P_{ij}$ subordinate to the open cover $\{\dot P_{ij}^a:a\in A\}$. By averaging the $\eta^a$ over the $\Ga_i\t\Ga_j$-action on $P_{ij}$, we suppose each $\eta^a$ is $\Ga_i$- and $\Ga_j$-invariant. Consider the subset $\dot P_{ij}\subseteq P_{ij}$ given by
\e
\begin{split}
\dot P_{ij}=\bigl\{\ts p\in \bigcup_{a\in A}\dot P_{ij}^a:\text{if $a,b\in A$ with $p\in\supp\eta^a\cap\supp\eta^b$}&\\
\text{then $\la_{ij}^a(p)=\la_{ij}^b(p)$}&\bigr\}.
\end{split}
\label{ku7eq3}
\e
We claim that $\dot P_{ij}$ is open in $P_{ij}$. To see this, note that $\dot P_{ij}$ is the complement in the open set $\bigcup_{a\in A}\dot P_{ij}^a\subseteq P_{ij}$ of the sets $S^{a,b}$ for all $a,b\in A$, where
\begin{equation*}
S^{a,b}=\bigl\{p\in\supp\eta^a\cap\supp\eta^b:\la_{ij}^a(p)\ne\la_{ij}^b(p)\bigr\}.
\end{equation*}
Now $\la_{ij}^a,\la_{ij}^b:\dot P_{ij}^a\cap\dot P_{ij}^b\ra P_{ij}'$ are smooth with $\pi_{ij}'\ci\la_{ij}^a=\pi_{ij}'\ci\la_{ij}^b$, where $\pi_{ij}':P_{ij}'\ra V_i$ is a principal $\Ga_j$-bundle over $V_{ij}'\subseteq V_i$. Thus the condition $\la_{ij}^a\ne\la_{ij}^b$ is open and closed in $\dot P_{ij}^a\cap\dot P_{ij}^b$, so $S^{a,b}$ is open and closed in $\supp\eta^a\cap\supp\eta^b$, and closed in $\bigcup_{a\in A}\dot P_{ij}^a$. Definition \ref{ku6def}(b) implies that only finitely many $S^{a,b}$ intersect a small open neighbourhood of each $p\in P_{ij}$. Therefore $\dot P_{ij}$ is open.

Next we claim that $\dot P_{ij}$ contains $\pi_{ij}^{-1}(\bar\psi_i^{-1}(S))$. Let $p\in\pi_{ij}^{-1}(\bar\psi_i^{-1}(S))$. Then $p\in \pi_{ij}^{-1}(\bar\psi_i^{-1}(T^{a'}))$ for some $a'\in A$ as $\bigcup_{a'\in A}T^{a'}=S$, so $p\in\dot P_{ij}^{a'}\subseteq\bigcup_{a\in A}\dot P_{ij}^a$. Suppose $a,b\in A$ with $p\in\supp\eta^a\cap\supp\eta^b$. Then $p\in\dot P_{ij}^a\cap\pi_{ij}^{-1}(s_i^{-1}(0))=\pi_{ij}^{-1}(\bar\psi_i^{-1}(T^a))$ as $\supp\eta^a\subseteq\dot P_{ij}^a$, and similarly $p\in\pi_{ij}^{-1}(\bar\psi_i^{-1}(T^b))$, so $p\in\pi_{ij}^{-1}(\bar\psi_i^{-1}(T^a\cap T^b))\subseteq\ddot P_{ij}^{ab}$, and the first equation of \eq{ku7eq2} gives $\la_{ij}^a(p)=\la_{ij}^b(p)$. Hence $p\in\dot P_{ij}$, proving the claim.

Define $\la_{ij}:\dot P_{ij}\ra P_{ij}'$ by
\e
\la_{ij}(p)=\la_{ij}^a(p)\quad\text{if $a\in A$ with $p\in\supp\eta^a$.}
\label{ku7eq4}
\e
This is well-defined by \eq{ku7eq3} as $\dot P_{ij}\subseteq\bigcup_{a\in A}\supp\eta^a$. As $\dot P_{ij}$ is covered by the open sets $\dot P_{ij}\cap\supp^\ci\eta^a$ for $a\in A$, and $\la_{ij}=\la_{ij}^a$ on $\dot P_{ij}\cap\supp^\ci\eta^a$ with $\la_{ij}^a$ smooth and \'etale, we see that $\la_{ij}$ is smooth and \'etale.

Define a vector bundle morphism $\hat\la_{ij}:\pi_{ij}^*(E_i)\vert_{\dot P_{ij}}\ra\phi_{ij}^*(TV_j)\vert_{\dot P_{ij}}$ by
\e
\hat\la_{ij}=\ts\sum_{a\in A}\eta^a\vert_{\dot P_{ij}}\cdot \hat\la_{ij}^a,
\label{ku7eq5}
\e
as in \eq{ku6eq11} and \eq{ku6eq14}, where $\hat\la_{ij}^a$ is only defined on $\dot P_{ij}\cap \dot P_{ij}^a$, but $\eta^a\cdot\hat\la_{ij}^a$ is well-defined and smooth on $\dot P_{ij}$, being zero outside~$\dot P_{ij}^a$.

For each $a\in A$, define $\ddot P_{ij}^a=\bigl\{p\in\dot P_{ij}\cap \dot P_{ij}^a:\la_{ij}(p)=\la_{ij}^a(p)\bigr\}$. As above this is open and closed in $\dot P_{ij}\cap\dot P_{ij}^a$ and so open in $\dot P_{ij}\cap \dot P_{ij}^a$, and contains $\pi_{ij}^{-1}(\bar\psi_i^{-1}(T^a))$, and by definition 
\e
\la_{ij}\vert_{\ddot P_{ij}^a}=\la_{ij}^a\vert_{\ddot P_{ij}^a}.
\label{ku7eq6}
\e
As for \eq{ku6eq15}, using \eq{ku7eq5} in the first step and the second equation of \eq{ku7eq2} (which holds on $\dot P_{ij}^a\cap \dot P_{ij}^b$) in the second, we have
\e
\begin{split}
\hat\la_{ij}\vert_{\ddot P_{ij}^a}&=\ts\sum_{b\in B}\eta^b\vert_{\ddot P_{ij}^a}\cdot \hat\la_{ij}^b=\ts\sum_{b\in B}\eta^b\vert_{\ddot P_{ij}^a}\cdot \bigl(\hat\la_{ij}^a+O(\pi_{ij}^*(s_i))\bigr)\\
&=\hat\la_{ij}^a\vert_{\ddot P_{ij}^a}+O(\pi_{ij}^*(s_i)).
\end{split}
\label{ku7eq7}
\e

We now claim that $(\dot P_{ij},\la_{ij},\hat\la_{ij})$ satisfies Definition \ref{ku4def3}(a)--(c) over $S$. The $\Ga_i,\Ga_j$-equivariance of $\dot P_{ij},\la_{ij},\hat\la_{ij}$ follows as the ingredients from which they are defined are $\Ga_i,\Ga_j$-equivariant. Equation \eq{ku4eq1} for $\dot P_{ij},\la_{ij},\hat\la_{ij}$ on $\dot P_{ij}\cap \dot P_{ij}^a$ follows from \eq{ku4eq1} for $\dot P_{ij}^a,\la_{ij}^a,\hat\la_{ij}^a$, equation \eq{ku7eq7}, and $\la_{ij}=\la_{ij}^a$ on $\ddot P_{ij}^a$, and the rest of (a)--(c) are already proved. Therefore $\La_{ij}:=[\dot P_{ij},\la_{ij},\hat\la_{ij}]$ is a 2-morphism $\Phi_{ij}\Ra\Phi_{ij}'$ over $S$. Equations \eq{ku7eq6}--\eq{ku7eq7} imply that $(\dot P_{ij},\la_{ij},\hat\la_{ij})\approx_{T^a}(\dot P_{ij}^a,\la_{ij}^a,\hat\la_{ij}^a)$ in the sense of Definition \ref{ku4def3}, so $\La_{ij}\vert_{T^a}=\La_{ij}^a$, for all $a\in A$. This proves Definition \ref{ku4def11}(iv) for~$\bcHom\bigl((V_i,E_i,\Ga_i,s_i,\psi_i),(V_j,E_j,\Ga_j,s_j,\psi_j)\bigr)$.

\subsubsection*{Definition \ref{ku4def11}(v) for $\bcHom((V_i,E_i,\Ga_i,s_i,\psi_i),(V_j,E_j,\Ga_j,s_j,\psi_j))$.}
For (v), let $S\subseteq\Im\psi_i\cap\Im\psi_j$ be open, $\{T^a:a\in A\}$ an open cover of $S$, $\Phi_{ij}^a=(P_{ij}^a,\pi_{ij}^a,\phi_{ij}^a,\hat\phi_{ij}^a):(V_i,E_i,\Ga_i,s_i,\psi_i)\ra(V_j,E_j,\Ga_j,s_j,\psi_j)$ a 1-morphism over $T^a$ for $a\in A$, and $\La_{ij}^{ab}:\Phi_{ij}^a\vert_{T^a\cap T^b}\Ra\Phi_{ij}^b\vert_{T^a\cap T^b}$ a 2-morphism over $T^a\cap T^b$ for all $a,b\in A$ such that $\La_{ij}^{bc}\od\La_{ij}^{ab}=\La_{ij}^{ac}$ over $T^a\cap T^b\cap T^c$ for all $a,b,c\in A$. Choose representatives $(\dot P_{ij}^{ab},\la_{ij}^{ab},\hat\la_{ij}^{ab})$ for $\La_{ij}^{ab}$ for $a,b\in A$, so that \eq{ku4eq1} gives
\e
\begin{split}
\phi_{ij}^b\ci\la_{ij}^{ab}&=\phi_{ij}^a\vert_{\dot P_{ij}^{ab}}+\hat\la_{ij}^{ab}\cdot (\pi_{ij}^a)^*(s_i)+O\bigl((\pi_{ij}^a)^*(s_i)^2\bigr)\;\>\text{and}\\ 
(\la_{ij}^{ab})^*(\hat\phi_{ij}^b)&=\hat\phi_{ij}^a\vert_{\dot P_{ij}^{ab}}+\hat\la_{ij}^{ab}\cdot (\phi_{ij}^a)^*(\d s_j)+O\bigl((\pi_{ij}^a)^*(s_i)\bigr)\;\> \text{on $\dot P_{ij}^{ab}$.}
\end{split}
\label{ku7eq8}
\e 

From \S\ref{ku41}, $\La_{ij}^{bc}\od\La_{ij}^{ab}=\La_{ij}^{ac}$ means there exists an open neighbourhood $\ddot P_{ij}^{abc}$ of $(\pi_{ij}^a)^{-1}(\bar\psi_i^{-1}(T^a\cap T^b\cap T^c))$ in $(\la_{ij}^{ab})^{-1}(\dot P_{ij}^{bc})\cap\dot P_{ij}^{ac}\subseteq P_{ij}^a$, such that
\e
\begin{split}
\la_{ij}^{bc}\ci\la_{ij}^{ab}\vert_{\ddot P_{ij}^{abc}}&=\la_{ij}^{ac}\vert_{\ddot P_{ij}^{abc}}\quad\text{and}\\ 
\hat\la_{ij}^{ab}\vert_{\ddot P_{ij}^{abc}}+\hat\la_{ij}^{\prime bc}\vert_{\ddot P_{ij}^{abc}}&=\hat\la_{ij}^{ac}\vert_{\ddot P_{ij}^{abc}}+O\bigl((\pi_{ij}^a)^*(s_i)\bigr)\quad\text{on $\ddot P_{ij}^{abc}$,}
\end{split}
\label{ku7eq9}
\e
where $\hat\la_{ij}^{\prime bc}:(\pi_{ij}^a)^*(E_i)\vert_{(\la_{ij}^{ab})^{-1}(\dot P_{ij}^{bc})}\ra(\phi_{ij}^a)^*(TV_j)\vert_{(\la_{ij}^{ab})^{-1}(\dot P_{ij}^{bc})}$ satisfies $\hat\la_{ij}^{\prime bc}=\la_{ij}^{ab}\vert_{(\la_{ij}^{ab})^{-1}(\dot P_{ij}^{bc})}^*(\hat\la_{ij}^{bc})+O((\pi_{ij}^a)^*(s_i))$, and $\hat\la_{ij}^{\prime bc}$ is unique up to $O((\pi_{ij}^a)^*(s_i))$.

Set $V_{ij}^a=\pi_{ij}^a(P_{ij}^a)$, so that $V_{ij}^a$ is an open neighbourhood of $\bar\psi_i^{-1}(T^a)$ in $V_i$ for $a\in A$. Choose a partition of unity $\{\eta^a:a\in A\}$ on $\bigcup_{a\in A}V_{ij}^a\subseteq V_i$ subordinate to the open cover $\{V_{ij}^a:a\in A\}$. In a similar way to \eq{ku7eq3}, define
\ea
V_{ij}=\bigl\{\ts v\in \bigcup_{a\in A}V_{ij}^a:\,&\text{if $a,b,c\in A$ with $v\in\supp\eta^a\cap\supp\eta^b\cap\supp\eta^c$}
\nonumber\\
&\text{then $\la_{ij}^{bc}\ci\la_{ij}^{ab}=\la_{ij}^{ac}$ on $(\pi_{ij}^a)^{-1}(v)$}\bigr\}.
\label{ku7eq10}
\ea
As for the argument between \eq{ku7eq3} and \eq{ku7eq4}, $V_{ij}$ is an open neighbourhood of $\bar\psi_i^{-1}(S)$ in $V_i$, and is $\Ga_i$-invariant as all the ingredients in \eq{ku7eq10} are. 

Define a set $P_{ij}$ by
\begin{equation*}
P_{ij}=\bigl(\ts\coprod_{a\in A}(\pi_{ij}^a)^{-1}(V_{ij}\cap\supp^\ci\eta^a)\bigr)\big/\sim,
\end{equation*}
where $(\pi_{ij}^a)^{-1}(V_{ij}\cap\supp^\ci\eta^a)\subseteq P_{ij}^a$ is open, and $\sim$ is the binary relation on $\coprod_{a\in A}(\pi_{ij}^a)^{-1}(V_{ij}\cap\supp^\ci\eta^a)$ given by $p^a\sim p^b$ if $p^a\in(\pi_{ij}^a)^{-1}(V_{ij}\cap\supp^\ci\eta^a)$ and $p^b\in(\pi_{ij}^b)^{-1}(V_{ij}\cap\supp^\ci\eta^b)$ for $a,b\in A$ with $p^b=\la_{ij}^{ab}(p^a)$. This is an equivalence relation by \eq{ku7eq10}. Write $[p^a]$ for the $\sim$-equivalence class of $p^a$. 

Define $\pi_{ij}:P_{ij}\ra V_{ij}\subseteq V_i$ by $\pi_{ij}:[p^a]\mapsto\pi_{ij}^a(p^a)$ for $p^a\in (\pi_{ij}^a)^{-1}(V_{ij}\cap\supp^\ci\eta^a)$. This is well-defined as if $[p^a]=[p^b]$ then $p^a\sim p^b$, so $p^b=\la_{ij}^{ab}(p^a)$, and $\pi_{ij}^a(p^a)=\pi_{ij}^b(p^b)$ as $\pi_{ij}^b\ci\la_{ij}^{ab}=\pi_{ij}^a$ by Definition \ref{ku4def3}(b). The $\Ga_i\t\Ga_j$-actions on $(\pi_{ij}^a)^{-1}(V_{ij}\cap\supp^\ci\eta^a)\subseteq P_{ij}^a$ induce a $\Ga_i\t\Ga_j$-action on $P_{ij}$, and $\pi_{ij}$ is $\Ga_i$-equivariant and~$\Ga_j$-invariant. 

Then $\pi_{ij}:P_{ij}\ra V_{ij}$ is a principal $\Ga_j$-bundle, because it is built by gluing the principal $\Ga_j$-bundles $\pi_{ij}^a:(\pi_{ij}^a)^{-1}(V_{ij}\cap\supp^\ci\eta^a)\ra V_{ij}\cap\supp^\ci\eta^a$ by the isomorphisms $\la_{ij}^{ab}$ on overlaps $V_{ij}\cap\supp^\ci\eta^a\cap\supp^\ci\eta^b$, where the isomorphisms $\la_{ij}^{ab}$ compose correctly by the definition \eq{ku7eq10} of $V_{ij}$. Therefore there is a unique manifold structure on $P_{ij}$ such that the inclusion $(\pi_{ij}^a)^{-1}(V_{ij}\cap\supp^\ci\eta^a)\hookra P_{ij}$ mapping $p^a\mapsto[p^a]$ is a diffeomorphism with an open set for each~$a\in A$.

For each $a\in A$, define $\dot P_{ij}^a\subseteq P_{ij}$ by
\e
\begin{split}
\dot P_{ij}^a=\pi_{ij}^{-1}\bigl(\bigl\{v\in V_{ij}\cap V_{ij}^a:\,&\text{if $b,c\in A$ with $v\in\supp\eta^b\cap\supp\eta^c$}\\
&\text{then $\la_{ij}^{ba}\ci\la_{ij}^{cb}=\la_{ij}^{ca}$ on $(\pi_{ij}^c)^{-1}(v)$}\bigr\}\bigr).
\end{split}
\label{ku7eq11}
\e
As for $V_{ij}$, using the first line of \eq{ku7eq9} we find that $\dot P_{ij}^a$ is a $\Ga_i$- and $\Ga_j$-invariant open neighbourhood of $(\pi_{ij})^{-1}(\bar\psi_i^{-1}(T^a))$ in $P_{ij}$. Also $\{\dot P_{ij}^a:a\in A\}$ is an open cover of $P_{ij}$, since \eq{ku7eq10}--\eq{ku7eq11} imply that $\pi_{ij}^{-1}(V_{ij}\cap\supp^\ci\eta^a)\subseteq\dot P_{ij}^a$, and the $\pi_{ij}^{-1}(V_{ij}\cap\supp^\ci\eta^a)$ cover $P_{ij}$. Define a map $\la_{ij}^a:\dot P_{ij}^a\ra P_{ij}^a$ by $\la_{ij}^a([p^b])=\la_{ij}^{ba}(p^b)$ if $b\in A$ and $p^b\in(\pi_{ij}^b)^{-1}(V_{ij}\cap\supp^\ci\eta^b)$ with $[p^b]\in\dot P_{ij}^a$. This is well-defined by \eq{ku7eq11}, and is a $\Ga_i$- and $\Ga_j$-equivariant smooth map with $\pi_{ij}^a\ci\la_{ij}^a=\pi_{ij}\vert_{\dot P_{ij}^a}$ since $\la_{ij}^{ba}$ is $\Ga_i$- and $\Ga_j$-equivariant and smooth with~$\pi_{ij}^a\ci\la_{ij}^{ba}=\pi_{ij}^b\vert_{\dot P_{ij}^{ba}}$.

To summarize our progress so far: we have defined a $\Ga_i$-invariant open neighbourhood $V_{ij}$ of $\bar\psi_i^{-1}(S)$ in $V_i$, a principal $\Ga_j$-bundle $\pi_{ij}:P_{ij}\ra V_{ij}$ with compatible $\Ga_i$-action, a $\Ga_i$- and $\Ga_j$-invariant open neighbourhood $\dot P_{ij}^a$ of $(\pi_{ij})^{-1}(\bar\psi_i^{-1}(T^a))$ in $P_{ij}$ for $a\in A$ such that $\{\dot P_{ij}^a:a\in A\}$ is an open cover of $P_{ij}$, and a $\Ga_i$- and $\Ga_j$-equivariant smooth map $\la_{ij}^a:\dot P_{ij}^a\ra P_{ij}^a$ with $\pi_{ij}^a\ci\la_{ij}^a=\pi_{ij}\vert_{\dot P_{ij}^a}$ for all~$a\in A$. 

We have smooth maps $\phi_{ij}^a\ci\la_{ij}^a:\dot P_{ij}^a\ra V_j$ and vector bundle morphisms $(\la_{ij}^a)^*(\hat\phi_{ij}^a):\pi_{ij}^*(E_i)\vert_{\dot P_{ij}^a}\ra (\phi_{ij}^a\ci\la_{ij}^a)^*(TV_j)$ for all $a\in A$, such that for all $a,b\in A$, applying $\ci\la_{ij}^a$ and $(\la_{ij}^a)^*$ to the equations of \eq{ku7eq8} gives
\begin{align*}
\phi_{ij}^b\ci\la_{ij}^b&=\phi_{ij}^a\ci\la_{ij}^a+(\la_{ij}^a)^*(\hat\la_{ij}^{ab})\cdot \pi_{ij}^*(s_i)+O\bigl(\pi_{ij}^*(s_i)^2\bigr)\;\>\text{and}\\ 
(\la_{ij}^b)^*(\hat\phi_{ij}^b)&=(\la_{ij}^a)^*(\hat\phi_{ij}^a)+(\la_{ij}^a)^*(\hat\la_{ij}^{ab})\cdot (\phi_{ij}^a\ci\la_{ij}^a)^*(\d s_j)+O\bigl(\pi_{ij}^*(s_i)\bigr),
\end{align*}
which hold on $\dot P_{ij}^a\cap\dot P_{ij}^b$ by the argument after~\eq{ku6eq4}.

We now use the method in \S\ref{ku61} between \eq{ku6eq5} and \eq{ku6eq11} to define a smooth map $\phi_{ij}:P_{ij}\ra V_j$ by combining the $\phi_{ij}^a\ci\la_{ij}^a:\dot P_{ij}^a\ra V_j$ using the partition of unity $\{\pi_{ij}^*(\eta^a):a\in A\}$. As in \eq{ku6eq5}, heuristically we want to write
\e
\phi_{ij}=\ts\sum_{a\in A}\pi_{ij}^*(\eta^a)\cdot(\phi_{ij}^a\ci\la_{ij}^a),
\label{ku7eq12}
\e
but the actual definition of $\phi_{ij}$ is an analogue of \eq{ku6eq10}, involving $W_n,\Xi_n$ from Lemma \ref{ku6lem}. We may take these $W_n,\Xi_n$ to be $\Ga_j$-equivariant, so that $\phi_{ij}$ is $\Ga_j$-equivariant, and also $\Ga_i$-invariant as all the ingredients are. 

As for \eq{ku6eq11}, define a morphism $\hat\la_{ij}^{\prime a}:\pi_{ij}^*(E_i)\vert_{\dot P_{ij}^a}\ra(\phi_{ij}^a\ci\la_{ij}^a)^*(TV_j)$ of vector bundles on $\dot P_{ij}^a$ for $a\in A$ by
\e
\hat\la_{ij}^{\prime a}=\ts\sum_{b\in A}\pi_{ij}^*(\eta^b)\vert_{\dot P_{ij}^a}\cdot(\la_{ij}^a)^*(\hat\la_{ij}^{ab}).
\label{ku7eq13}
\e 
Then as for \eq{ku6eq13} we deduce that
\e
\phi_{ij}=\phi_{ij}^a\ci\la_{ij}^a+\hat\la_{ij}^{\prime a}\cdot \pi_{ij}^*(s_i)+O(\pi_{ij}^*(s_i)^2)\quad\text{on $\dot P_{ij}^a$.}
\label{ku7eq14}
\e
Therefore $\phi_{ij}=\phi_{ij}^a\ci\la_{ij}^a+O(\pi_{ij}^*(s_i))$, so as in Definition \ref{ku2def1} we can choose $\hat\la_{ij}^a:\pi_{ij}^*(E_i)\vert_{\dot P_{ij}^a}\ra\phi_{ij}^*(TV_j)\vert_{\dot P_{ij}^a}$ with $\hat\la_{ij}^a=-\hat\la_{ij}^{\prime a}+O(\pi_{ij}^*(s_i))$, and $\hat\la_{ij}^a$ is unique up to $O(\pi_{ij}^*(s_i))$. Then \eq{ku7eq14} is equivalent to
\e
\phi_{ij}^a\ci\la_{ij}^a=\phi_{ij}+\hat\la_{ij}^a\cdot \pi_{ij}^*(s_i)+O(\pi_{ij}^*(s_i)^2)\quad\text{on $\dot P_{ij}^a$.}
\label{ku7eq15}
\e

Similarly, we choose $\Ga_i$- and $\Ga_j$-equivariant $\hat\phi_{ij}^{\prime a}:\pi_{ij}^*(E_i)\vert_{\dot P_{ij}^a}\ra \phi_{ij}^*(TV_j)\vert_{\dot P_{ij}^a}$ with $\hat\phi_{ij}^{\prime a}=(\la_{ij}^a)^*(\hat\phi_{ij}^a)+O(\pi_{ij}^*(s_i))$, uniquely up to $O(\pi_{ij}^*(s_i))$. As in \eq{ku6eq14}, define a $\Ga_i$- and $\Ga_j$-equivariant morphism $\hat\phi_{ij}:\pi_{ij}^*(E_i)\ra\phi_{ij}^*(E_j)$ by
\begin{equation*}
\hat\phi_{ij}=\ts\sum_{a\in A}\pi_{ij}^*(\eta^a)\cdot \hat\phi_{ij}^{\prime a}.
\end{equation*}
Then as for \eq{ku6eq15}, for each $a\in A$ we have
\e
(\la_{ij}^a)^*(\hat\phi_{ij}^a)=\hat\phi_{ij}+\hat\la_{ij}^a\cdot\phi_{ij}^*(\d s_j)+O(\pi_{ij}^*(s_i)).
\label{ku7eq16}
\e

We have already proved $\Phi_{ij}:=(P_{ij},\pi_{ij},\phi_{ij},\hat\phi_{ij})$ satisfies Definition \ref{ku4def2}(a)--(d). Parts (e),(f) hold on $\dot P_{ij}^a\subseteq P_{ij}$ by \eq{ku7eq15}, \eq{ku7eq16} and Definition \ref{ku4def2}(e),(f) for $\Phi_{ij}^a$, for each $a\in A$, so they hold on $\bigcup_{a\in A}\dot P_{ij}^a=P_{ij}$. Thus $\Phi_{ij}:(V_i,E_i,\Ga_i,s_i,\psi_i)\ab\ra(V_j,E_j,\Ga_j,s_j,\psi_j)$ is a 1-morphism over $S$. 

Equations \eq{ku7eq15} and \eq{ku7eq16} imply that $\La_{ij}^a:=[\dot P_{ij}^a,\la_{ij}^a,\hat\la_{ij}^a]$ is a 2-morphism $\Phi_{ij}\vert_{T^a}\Ra\Phi_{ij}^a$ over $T^a$. For $a,b\in A$, set $\ddot P_{ij}^{ab}=\dot P_{ij}^a\cap\dot P_{ij}^b\cap \pi_{ij}^{-1}(V_{ij}^{ab})$, an open neighbourhood of $\pi_{ij}^{-1}(\bar\psi_i^{-1}(T^a\cap T^b))$ in $P_{ij}$. Then from \eq{ku7eq9}, \eq{ku7eq11}, \eq{ku7eq13} and the definitions of $\la_{ij}^a,\hat\la_{ij}^a$ we deduce that
\e
\begin{split}
\la_{ij}^{ab}\ci\la_{ij}^a\vert_{\ddot P_{ij}^{ab}}&=\la_{ij}^b\vert_{\ddot P_{ij}^{ab}}\quad\text{and}\\ 
\hat\la_{ij}^a\vert_{\ddot P_{ij}^{ab}}+\hat\la_{ij}^{\prime ab}\vert_{\ddot P_{ij}^{ab}}&=\hat\la_{ij}^b\vert_{\ddot P_{ij}^{ab}}+O\bigl(\pi_{ij}^*(s_i)\bigr)\quad\text{on $\ddot P_{ij}^{ab}$,}
\end{split}
\label{ku7eq17}
\e
where $\hat\la_{ij}^{\prime ab}:\pi_{ij}^*(E_i)\ra\phi_{ij}^*(TV_j)$ on $(\la_{ij}^a)^{-1}(\dot P_{ij}^{ab})$ is unique up to $O(\pi_{ij}^*(s_i))$ and satisfies $\hat\la_{ij}^{\prime ab}=\la_{ij}^a\vert_{(\la_{ij}^a)^{-1}(\dot P_{ij}^{ab})}^*(\hat\la_{ij}^{ab})+O(\pi_{ij}^*(s_i))$. As in \eq{ku7eq9}, equation \eq{ku7eq17} implies that $\La_{ij}^b\vert_{T^a\cap T^b}=\La_{ij}^{ab}\od\La_{ij}^a\vert_{T^a\cap T^b}$. This proves Definition \ref{ku4def11}(v), showing that $\bcHom\bigl((V_i,E_i,\Ga_i,s_i,\psi_i),(V_j,E_j,\Ga_j,s_j,\psi_j)\bigr)$ is a stack on $\Im\psi_i\cap\Im\psi_j$, and completes the first part of Theorem~\ref{ku4thm1}(a).

\subsubsection{Theorem \ref{ku4thm1}(a): $\bcEqu((V_i,E_i,\Ga_i,s_i,\psi_i),(V_j,E_j,\Ga_j,s_j,\psi_j))$ is a \ab substack of $\bcHom\bigl((V_i,E_i,\Ga_i,s_i,\psi_i),(V_j,E_j,\Ga_j,s_j,\psi_j)\bigr)$}
\label{ku712}

In this subsection, we will by an abuse of notation treat the weak 2-category $\Kur_S(X)$ defined in \S\ref{ku41} as if it were a strict 2-category. That is, we will pretend the 2-morphisms $\bs\al_{\Phi_{kl},\Phi_{jk},\Phi_{ij}},\bs\be_{\Phi_{ij}},\bs\ga_{\Phi_{ij}}$ in \eq{ku4eq4} and \eq{ku4eq6} are identities or omit them, and we will omit brackets in compositions of 1-morphisms such as $\Phi_{kl}\ci\Phi_{jk}\ci\Phi_{ij}$. This is permissible as every weak 2-category can be strictified. We do it because otherwise diagrams such as Figure \ref{ku7fig1} would become too big.

The second part of Theorem \ref{ku4thm1}(a) is similar to \S\ref{ku612}. Definition \ref{ku4def11}(i)--(iv) for $\bcEqu((V_i,E_i,\Ga_i,s_i,\psi_i),(V_j,E_j,\Ga_j,s_j,\psi_j))$ are immediate. For (v), we must show that in the last part of the proof in \S\ref{ku711}, if the $\Phi_{ij}^a$ are coordinate changes over $T^a$ (i.e.\ equivalences in $\Kur_{T^a}(X)$), then the $\Phi_{ij}$ we construct with 2-morphisms $\La_{ij}^a:\Phi_{ij}\vert_{T^a}\Ra\Phi_{ij}^a$ for $a\in A$ is a coordinate change over~$S$.

Let $S,\{T^a:a\in A\},\Phi_{ij}^a,\La_{ij}^{ab},\Phi_{ij},\La_{ij}^a$ be as in the last part of \S\ref{ku711}, but with the $\Phi_{ij}^a$ coordinate changes. Since $\Phi_{ij}^a$ is an equivalence in $\Kur_{T^a}(X)$, we may choose a coordinate change $\Phi_{ji}^a:(V_j,E_j,\Ga_j,s_j,\psi_j)\ra (V_i,E_i,\Ga_i,s_i,\psi_i)$ over $T^a$ and 2-morphisms $\Io_i^a:\Phi_{ji}^a\ci\Phi_{ij}^a\Ra \id_{(V_i,E_i,\Ga_i,s_i,\psi_i)}$ and $\Ka_j^a:\Phi_{ij}^a\ci\Phi_{ji}^a\Ra\id_{(V_j,E_j,\Ga_j,s_j,\psi_j)}$ for all $a\in A$. By Proposition \ref{kuBprop} we can suppose these satisfy
\e
\id_{\Phi_{ij}^a}*\Io_i^a=\Ka_j^a*\id_{\Phi_{ij}^a}\quad\text{and}\quad \id_{\Phi_{ji}^a}*\Ka_j^a=\Io_i^a*\id_{\Phi_{ji}^a}.
\label{ku7eq18}
\e

Define 2-morphisms $\Mu_{ji}^{ab}:\Phi_{ji}^a\vert_{T^a\cap T^b}\Ra\Phi_{ji}^b\vert_{T^a\cap T^b}$ over $T^a\cap T^b$ for all $a,b\in A$ to be the vertical composition
\e
\xymatrix@C=18pt{
\Phi_{ji}^a\vert_{T^a\cap T^b} \ar@{=>}[r]^{\raisebox{10pt}{$\st\id_{\Phi_{ji}^a}*(\Ka_j^b)^{-1}$}} & \Phi_{ji}^a \!\ci\! \Phi_{ij}^b\!\ci\!\Phi_{ji}^b \ar@{=>}[rr]^{\raisebox{10pt}{$\st\id_{\Phi_{ji}^a}*(\La_{ij}^{ab})^{-1}*\id_{\Phi_{ji}^b}$}} && \Phi_{ji}^a\!\ci\!\Phi_{ij}^a\!\ci\!\Phi_{ji}^b \ar@{=>}[r]^{\raisebox{10pt}{$\st\Io_i^a*\id_{\Phi_{ji}^b}$}} & \Phi_{ji}^b\vert_{T^a\cap T^b}. }
\label{ku7eq19}
\e
For $a,b,c\in A$, consider the diagram Figure \ref{ku7fig1} of 2-morphisms over $T^a\cap T^b\cap T^c$. The three outer quadrilaterals commute by the definition \eq{ku7eq19} of $\Mu_{ji}^{ab}$. Eight inner quadrilaterals commute by compatibility of horizontal and vertical composition, a 2-gon commutes by \eq{ku7eq18}, and a triangle commutes as $\La_{ij}^{bc}\od\La_{ij}^{ab}=\La_{ij}^{ac}$. Hence Figure \ref{ku7fig1} commutes, which shows that $\Mu_{ji}^{bc}\od\Mu_{ji}^{ab}=\Mu_{ji}^{ac}$ over $T^a\cap T^b\cap T^c$ for all~$a,b,c\in A$. 

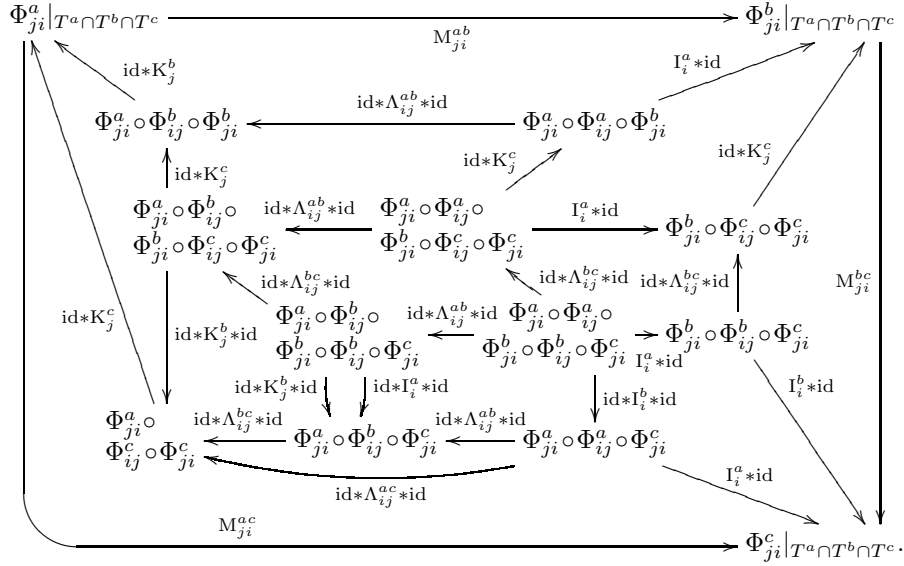
\begin{figure}[htb]
{$\displaystyle\xymatrix@!0@C=54pt@R=40pt{
*+[r]{\Phi_{ji}^a\vert_{T^a\cap T^b\cap T^c}} \ar[rrrrrr]_{\Mu_{ji}^{ab}} 
\ar`d/20pt[dddddr][dddddrrrrrr]^(0.2){\Mu_{ji}^{ac}} 
 &&&&&& *+[l]{\Phi_{ji}^b\vert_{T^a\cap T^b\cap T^c}} \ar[ddddd]_{\Mu_{ji}^{bc}} 
\\ 
& {\Phi_{ji}^a \!\ci\! \Phi_{ij}^b\!\ci\!\Phi_{ji}^b} \ar[ul]_(0.3){\id*\Ka_j^b} &&& {\Phi_{ji}^a\!\ci\!\Phi_{ij}^a\!\ci\!\Phi_{ji}^b} \ar[lll]_(0.45){\id*\La_{ij}^{ab}*\id} \ar[urr]^(0.4){\Io_i^a*\id}
\\
& *+[r]{\begin{subarray}{l}\ts \Phi_{ji}^a\!\ci\!\Phi_{ij}^b\ci \\ 
\ts\Phi_{ji}^b\!\ci\!\Phi_{ij}^c\!\ci\!\Phi_{ji}^c \end{subarray}} \ar[dd]^{\id*\Ka_j^b*\id} \ar[u]_{\id*\Ka_j^c} 
&&
{\begin{subarray}{l}\ts \Phi_{ji}^a\!\ci\!\Phi_{ij}^a\ci \\ 
\ts\Phi_{ji}^b\!\ci\!\Phi_{ij}^c\!\ci\!\Phi_{ji}^c \end{subarray}} \ar[ll]_{\id*\La_{ij}^{ab}*\id} \ar[rr]^{\Io_i^a*\id} \ar[ur]^(0.45){\id*\Ka_j^c}
&& {\Phi_{ji}^b\!\ci\!\Phi_{ij}^c\!\ci\!\Phi_{ji}^c} \ar[uur]^(0.3){\id*\Ka_j^c}
\\
&& *+[r]{\begin{subarray}{l}\ts \Phi_{ji}^a\!\ci\!\Phi_{ij}^b\ci \\ 
\ts\Phi_{ji}^b\!\ci\!\Phi_{ij}^b\!\ci\!\Phi_{ji}^c \end{subarray}} \ar[ul]_(0.3){\id*\La_{ij}^{bc}*\id} \ar@<2.8ex>@/_.5pc/[d]_{\id*\Ka_j^b*\id} \ar@<3.5ex>@/^.5pc/[d]^{\id*\Io_i^a*\id}
&&
*+[l]{\begin{subarray}{l}\ts {}\;\>\;\>\Phi_{ji}^a\!\ci\!\Phi_{ij}^a\ci \\ 
\ts\Phi_{ji}^b\!\ci\!\Phi_{ij}^b\!\ci\!\Phi_{ji}^c \end{subarray}} \ar[ll]_{\id*\La_{ij}^{ab}*\id} \ar[r]_(0.45){\raisebox{-9pt}{$\st\Io_i^a*\id$}} \ar[d]^(0.6){\id*\Io_i^b*\id} \ar[ul]_(0.35){\id*\La_{ij}^{bc}*\id} & {\Phi_{ji}^b\!\ci\!\Phi_{ij}^b\!\ci\!\Phi_{ji}^c} \ar[u]^{\id*\La_{ij}^{bc}*\id} \ar[ddr]^(0.3){\Io_i^b*\id}
\\
& *+[l]{\begin{subarray}{l}\ts \Phi_{ji}^a\ci\\ \ts\Phi_{ij}^c\!\ci\!\Phi_{ji}^c\end{subarray}} \ar[uuuul]^(0.3){\id*\Ka_j^c}
&
*+[r]{\Phi_{ji}^a\!\ci\!\Phi_{ij}^b\!\ci\!\Phi_{ji}^c} \ar[l]_{\id*\La_{ij}^{bc}*\id} && {\Phi_{ji}^a\!\ci\!\Phi_{ij}^a\!\ci\!\Phi_{ji}^c} \ar[ll]_(0.4){\id*\La_{ij}^{ab}*\id} \ar@<.5ex>@/^1pc/[lll]^{\id*\La_{ij}^{ac}*\id} \ar[drr]^{\Io_i^a*\id} &
\\
&&&&&& *+[l]{\Phi_{ji}^c\vert_{T^a\cap T^b\cap T^c}.\!\!} }$}
\caption{Proof that $\Mu_{ji}^{bc}\od\Mu_{ji}^{ab}=\Mu_{ji}^{ac}$}
\label{ku7fig1}
\end{figure}

Thus by Definition \ref{ku4def11}(v) for $\bcHom((V_j,E_j,\Ga_j,s_j,\psi_j),(V_i,E_i,\Ga_i,s_i,\psi_i))$, proved in \S\ref{ku711}, there exists a 1-morphism $\Phi_{ji}:(V_j,E_j,\Ga_j,s_j,\psi_j)\ra (V_i,\ab E_i,\ab\Ga_i,\ab s_i,\ab\psi_i)$ over $S$ and 2-morphisms $\Mu_{ji}^a:\Phi_{ji}\vert_{T^a}\Ra\Phi_{ji}^a$ over $T^a$ for $a\in A$, such that $\Mu_{ji}^b\vert_{T^a\cap T^b}=\Mu_{ji}^{ab}\od\Mu_{ji}^a\vert_{T^a\cap T^b}$ over $T^a\cap T^b$ for all $a,b\in A$.

For each $a\in A$, define a 2-morphism $\Nu_i^a:(\Phi_{ji}\ci\Phi_{ij})\vert_{T^a}\Ra\id_{(V_i,E_i,\Ga_i,s_i,\psi_i)}\vert_{T^a}$ by the vertical composition
\e
\smash{\xymatrix@C=35pt{
(\Phi_{ji}\ci\Phi_{ij})\vert_{T^a} \ar@{=>}[rr]^{\Mu_{ji}^a*\La_{ij}^a} && \Phi_{ji}^a\ci\Phi_{ij}^a \ar@{=>}[r]^(0.4){\Io_i^a} & \id_{(V_i,E_i,\Ga_i,s_i,\psi_i)}\vert_{T^a}. }}
\label{ku7eq20}
\e
Then the following diagram commutes by \eq{ku7eq18}, $\La_{ij}^b=\La_{ij}^{ab}\od\La_{ij}^a$, $\Mu_{ji}^b=\Mu_{ji}^{ab}\od\Mu_{ji}^a$, the definitions of $\Mu_{ji}^{ab},\Nu_i^a$ in \eq{ku7eq19} and \eq{ku7eq20}, and compatibility of horizontal and vertical composition:
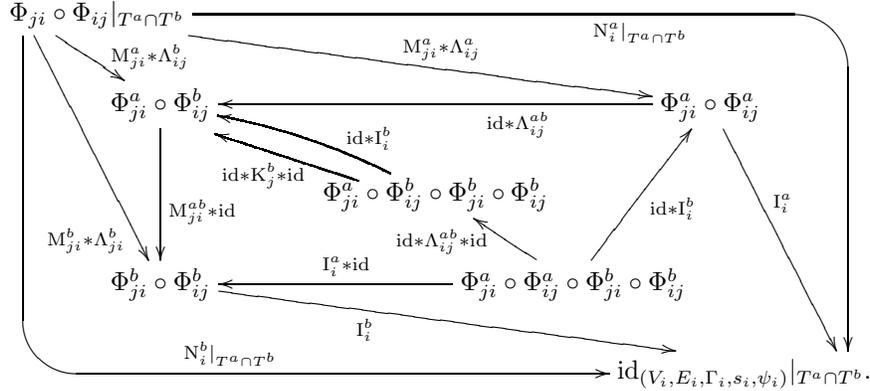
\begin{equation*}
\xymatrix@!0@C=52pt@R=34pt{
*+[r]{\Phi_{ji}\ci\Phi_{ij}\vert_{T^a\cap T^b}} \ar`d/20pt[ddddr]
[ddddrrrrrr]^(0.2){\Nu_i^b\vert_{T^a\cap T^b}}
\ar`r/20pt[rrrrrrd]_(0.8){\Nu_i^a\vert_{T^a\cap T^b}}[rrrrrrdddd] 
\ar[rrrrrd]^(0.6){\Mu_{ji}^a*\La_{ij}^a} \ar[dddr]_(0.8){\Mu_{ji}^b*\La_{ji}^b} \ar[dr]^(0.7){\Mu_{ji}^a*\La_{ij}^b}
&&&&&& 
\\ 
& {\Phi_{ji}^a\ci\Phi_{ij}^b} \ar[dd]^(0.6){\Mu_{ji}^{ab}*\id} &&&&
{\Phi_{ji}^a\ci\Phi_{ij}^a} \ar[dddr]^(0.4){\Io_i^a} \ar[llll]^(0.35){\id*\La_{ij}^{ab}}
&
\\
&&& {\Phi_{ji}^a\ci\Phi_{ij}^b\ci\Phi_{ji}^b\ci\Phi_{ij}^b} 
\ar@<.3ex>@/_.5pc/[ull]_(0.3){\id*\Io_i^b} \ar@<1ex>@/.5pc/[ull]^{\id*\Ka_j^b*\id} &&&
\\
& {\Phi_{ji}^b\ci\Phi_{ij}^b} \ar[drrrrr]_(0.3){\Io_i^b}
&&& {\Phi_{ji}^a\ci\Phi_{ij}^a\ci\Phi_{ji}^b\ci\Phi_{ij}^b} \ar[lll]_(0.55){\Io_i^a*\id} \ar[uur]_{\id*\Io_i^b} \ar[ul]^(0.7){\id*\La_{ij}^{ab}*\id} &&
\\
&&&&&& *+[l]{\id_{(V_i,E_i,\Ga_i,s_i,\psi_i)}\vert_{T^a\cap T^b}.\!\!} }
\end{equation*}
Hence $\Nu_i^a\vert_{T^a\cap T^b}=\Nu_i^b\vert_{T^a\cap T^b}$ for all $a,b\in A$. Therefore by Definition \ref{ku4def11}(iv) for $\bcHom((V_i,E_i,\Ga_i,s_i,\psi_i),(V_i,E_i,\Ga_i,s_i,\psi_i))$, proved in \S\ref{ku711}, there exists a unique 2-morphism $\Nu_i:\Phi_{ji}\ci\Phi_{ij}\Ra\id_{(V_i,E_i,\Ga_i,s_i,\psi_i)}$ over $S$ with $\Nu_i\vert_{T^a}=\Nu_i^a$ for all $a\in A$. Similarly we may construct a 2-morphism ${\rm O}_j:\Phi_{ij}\ci\Phi_{ji}\Ra\id_{(V_j,E_j,\Ga_j,s_j,\psi_j)}$. These $\Phi_{ji},\Nu_i,{\rm O}_j$ imply that $\Phi_{ij}$ is an equivalence in $\Kur_S(X)$, and so a coordinate change over $S$. 

This proves Definition \ref{ku4def11}(v) for $\bcEqu((V_i,E_i,\Ga_i,s_i,\psi_i),(V_j,E_j,\Ga_j,s_j,\psi_j))$, which is thus a substack of $\bcHom((V_i,E_i,\Ga_i,s_i,\psi_i),(V_j,E_j,\Ga_j,s_j,\psi_j))$, completing the proof of Theorem~\ref{ku4thm1}(a).

\subsubsection{Theorem \ref{ku4thm1}(b): $\bcHom_f\bigl((U_i,D_i,\Be_i,r_i,\chi_i),(V_j,E_j,\Ga_j,s_j,\psi_j)\bigr)$ is a stack}
\label{ku713}

The proof of Theorem \ref{ku4thm1}(b) is essentially the same as for the first part of Theorem \ref{ku4thm1}(a) in \S\ref{ku711}, but inserting `over $f$' throughout.

\subsection{Proof of Theorem \ref{ku4thm2}}
\label{ku72}

\subsubsection{The `only if' part of Theorem \ref{ku4thm2}}
\label{ku721}

Let $\Phi_{ij}=(P_{ij},\pi_{ij},\phi_{ij},\hat\phi_{ij}):(V_i,E_i,\Ga_i,s_i,\psi_i)\ra (V_j,E_j,\Ga_j,s_j,\psi_j)$ be a coordinate change of Kuranishi neighbourhoods over $S\subseteq X$. Let $p\in\pi_{ij}^{-1}(\bar\psi_i^{-1}(S))\subseteq P_{ij},$ and set $v_i=\pi_{ij}(p)\in V_i$, $v_j=\phi_{ij}(p)\in V_j$, and $x=\bar\psi_i(v_i)=\bar\psi_j(v_j)\in S$. We must show that \eq{ku4eq9} is exact and \eq{ku4eq10} is an isomorphism.

As $\Phi_{ij}$ is a coordinate change over $S$, it is an equivalence in $\Kur_S(X)$, so there exists a 1-morphism $\Phi_{ji}=(P_{ji},\pi_{ji},\phi_{ji},\hat\phi_{ji}):(V_j,E_j,\Ga_j,s_j,\psi_j)\ra (V_i,E_i,\Ga_i,s_i,\psi_i)$ and 2-morphisms $\La_{ii}:\id_{(V_i,E_i,\Ga_i,s_i,\psi_i)}\Ra\Phi_{ji}\ci\Phi_{ij}$, $\Mu_{jj}:\id_{(V_j,E_j,\Ga_j,s_j,\psi_j)}\Ra\Phi_{ij}\ci\Phi_{ji}$ over $S$. By Proposition \ref{kuBprop} we choose these with
\e
\begin{split}
\Mu_{jj}*\id_{\Phi_{ij}}&=(\id_{\Phi_{ij}}*\La_{ii})\od\al_{\Phi_{ij},\Phi_{ji},\Phi_{ij}},\\
\La_{ii}*\id_{\Phi_{ji}}&=(\id_{\Phi_{ji}}*\Mu_{jj})\od\al_{\Phi_{ji},\Phi_{ij},\Phi_{ji}}.
\end{split}
\label{ku7eq21}
\e

Choose representatives $(\dot P_{ii},\la_{ii},\hat\la_{ii})$ and $(\dot P_{jj},\mu_{jj},\hat\mu_{jj})$ for $\La_{ii},\Mu_{jj}$. Define $V_{ij}=\pi_{ij}(P_{ij})$, $V_{ji}=\pi_{ji}(P_{ji})$, $\dot V_{ii}=\pi_{ii}(\dot P_{ii})$, $\dot V_{jj}=\pi_{jj}(\dot P_{jj})$, so that $v_i\in\dot V_{ii}\subseteq V_{ij}\subseteq V_i$ and $v_j\in\dot V_{jj}\subseteq V_{ji}\subseteq V_j$ are open, $\pi_{ii}:\dot P_{ii}\ra\dot V_{ii}$, $\pi_{ji}:P_{ji}\ra V_{ji}$ are principal $\Ga_i$-bundles, and $\pi_{jj}:\dot P_{jj}\ra\dot V_{jj}$, $\pi_{ij}:P_{ij}\ra V_{ij}$ are principal $\Ga_j$-bundles. By definition $\dot P_{ii}=\dot V_{ii}\t\Ga_i$, and $\la_{ii}:\dot P_{ii}\ra (P_{ij}\t_{\phi_{ij},V_j,\pi_{ji}}P_{ji})/\Ga_j$ is an isomorphism of principal $\Ga_i$-bundles over $\dot V_{ii}$. From this we can show that there is a unique point $q\in P_{ji}$ with $\pi_{ji}(q)=v_j\in V_j$, $\phi_{ji}(q)=v_i\in V_i$, and
\e
\la_{ii}(v_i,1)=\Ga_j\cdot (p,q).
\label{ku7eq22}
\e
From \eq{ku7eq21} we can then show that
\e
\mu_{jj}(v_j,1)=\Ga_i\cdot (q,p).
\label{ku7eq23}
\e

Define subgroups $\De_i\subseteq\Ga_i$ and $\De_j\subseteq\Ga_j$ by
\e
\begin{split}
\De_i=\Stab_{\Ga_i}(v_i)&=\bigl\{\ga_i\in\Ga_i:\ga_i\cdot v_i=v_i\bigr\},\\
\De_j=\Stab_{\Ga_j}(v_j)&=\bigl\{\ga_j\in\Ga_j:\ga_j\cdot v_j=v_j\bigr\}.
\end{split}
\label{ku7eq24}
\e
Choose a connected, simply-connected, $\De_i$-invariant open neighbourhood $V_{ij}^{v_i}$ of $v_i$ in $V_{ij}\subseteq V_i$, and a connected, simply-connected, $\De_j$-invariant open neighbourhood $V_{ji}^{v_j}$ of $v_j$ in $V_{ji}\subseteq V_j$. As $\pi_{ij}:P_{ij}\ra V_{ij}$ is a principal $\Ga_j$-bundle, $v_i\in V_{ij}^{v_i}\subseteq V_{ij}$ with $V_{ij}^{v_i}$ connected and simply-connected, and $p\in P_{ij}$ with $\pi_{ij}(p)=v_i$, there is a unique diffeomorphism
\e
\begin{split}
&P_{ij}^{v_i}:=\pi_{ij}^{-1}(V_{ij}^{v_i})\cong V_{ij}^{v_i}\t\Ga_j, \quad\text{with $p\cong (v_i,1)$,} \\
&\text{$\pi_{ij}:(v,\ga_j)\mapsto v$, and $\Ga_j$-action $\ga:(v,\ga_j)\longmapsto (v,\ga\ga_j)$.}
\end{split}
\label{ku7eq25}
\e
Since $V_{ij}^{v_i}$ is $\De_i$-invariant and $\pi_{ij}$ is $\Ga_i$-equivariant, $P_{ij}^{v_i}$ is also $\De_i$-invariant. As the $\Ga_i$- and $\Ga_j$-actions on $P_{ij}$ commute and $V_{ij}^{v_i}$ is connected, under the identification \eq{ku7eq25}, the $\De_i$-action on $V_{ij}^{v_i}\t\Ga_j$ must be of the form
\e
\de_i:(v,\ga_j)\longmapsto \bigl(\de_i\cdot v,\ga_j\rho(\de_i)^{-1}\bigr),\quad \de_i\in\De_i, \;\> v\in V_{ij}^{v_i},\;\> \ga_j\in\Ga_j,
\label{ku7eq26}
\e
for some map $\rho:\De_i\ra\Ga_j$, which is a group morphism as \eq{ku7eq26} is a group action. As $\phi_{ij}:P_{ij}\ra V_j$ is $\Ga_i$-invariant with $\phi_{ij}(v_i,1)=v_j$, we see from \eq{ku7eq25}--\eq{ku7eq26} that $\rho(\de_i)\cdot v_j=v_j$, so $\rho$ is a group morphism $\De_i\ra\De_j$.

Define a smooth map $\phi_{ij}^{v_i}:V_{ij}^{v_i}\ra V_j$ by $\phi_{ij}^{v_i}:v\mapsto \phi_{ij}(v,1)$, using the identification \eq{ku7eq25}. Since $\phi_{ij}$ is $\Ga_j$-equivariant, it follows that under the identification \eq{ku7eq25} we have
\e
\phi_{ij}\vert_{P_{ij}^{v_i}}:(v,\ga_j)\longmapsto \ga_j\cdot \phi_{ij}^{v_i}(v).
\label{ku7eq27}
\e
As $\phi_{ij}$ is $\Ga_i$-invariant, so that $\phi_{ij}\vert_{P_{ij}^{v_i}}$ is $\De_i$-invariant, equations \eq{ku7eq26}--\eq{ku7eq27} imply that $\phi_{ij}^{v_i}:V_{ij}^{v_i}\ra V_j$ is equivariant under $\rho:\De_i\ra\De_j$. 

Define $\hat\phi_{ij}^{v_i}:E_i\vert_{V_{ij}^{v_i}}\ra(\phi_{ij}^{v_i})^*(TV_j)$ by $\hat\phi_{ij}^{v_i}\vert_v=\hat\phi_{ij}\vert_{(v,1)}$, using \eq{ku7eq25}. Then $\hat\phi_{ij}^{v_i}$ is equivariant under $\rho:\De_i\ra\De_j$. Definition \ref{ku4def2}(e) for $\Phi_{ij}$ implies that $\hat\phi_{ij}^{v_i}(s_i\vert_{V_{ij}^{v_i}})=(\phi_{ij}^{v_i})^*(s_j)+O(s_i^2)$, as in Definition~\ref{ku2def3}(d).

Similarly, we get a unique diffeomorphism
\e
\begin{split}
&P_{ji}^{v_j}:=\pi_{ji}^{-1}(V_{ji}^{v_j})\cong V_{ji}^{v_j}\t\Ga_i, \quad\text{with $q\cong (v_j,1)$,} \\
&\text{$\pi_{ji}:(v,\ga_i)\mapsto v$, $\Ga_i$-action $\ga:(v,\ga_i)\longmapsto (v,\ga\ga_i)$,}\\
&\text{and $\De_j$-action $\de_j:(v,\ga_i)\longmapsto \bigl(\de_j\cdot v,\ga_i\si(\de_j)^{-1}\bigr)$,}
\end{split}
\label{ku7eq28}
\e
for $\si:\De_j\ra\Ga_i$ a group morphism, and defining $\phi_{ji}^{v_j}:V_{ji}^{v_j}\ra V_i$ and $\hat\phi_{ji}^{v_j}:E_j\vert_{V_{ji}^{v_j}}\ra(\phi_{ji}^{v_j})^*(TV_i)$ by $\phi_{ji}^{v_j}:v\mapsto \phi_{ji}(v,1)$ and $\hat\phi_{ji}^{v_j}\vert_v=\hat\phi_{ji}\vert_{(v,1)}$, using the identification \eq{ku7eq28}, then $\phi_{ji}^{v_j},\hat\phi_{ji}^{v_j}$ are $\si$-equivariant.

Consider $\La_{ii}=[\dot P_{ii},\la_{ii},\hat\la_{ii}]:\id_{(V_i,E_i,\Ga_i,s_i,\psi_i)}\Ra\Phi_{ji}\ci\Phi_{ij}$ near $v_i$ in this notation. We have $\la_{ii}:\dot P_{ii}=\dot V_{ii}\t\Ga_i\ra (P_{ij}\t_{\phi_{ij},V_j,\pi_{ji}}P_{ji})/\Ga_j$, an isomorphism of principal $\Ga_i$-bundles over $\dot V_{ii}$. Choose a connected open neighbourhood $\dot V_{ii}^{v_i}$ of $v_i$ in $\dot V_{ii}\cap V_{ij}^{v_i}\cap(\phi_{ij}^{v_i})^{-1}(V_{ji}^{v_j})$, and set $\dot P_{ii}^{v_i}=\dot V_{ii}^{v_i}\t\Ga_i\subseteq \dot P_{ii}$. Over $\dot V_{ii}^{v_i}$, the target $(P_{ij}\t_{\phi_{ij},V_j,\pi_{ji}}P_{ji})/\Ga_j$ has a natural trivialization by \eq{ku7eq25} and \eq{ku7eq28}, and $\la_{ii}(v_i,1)=(v_i,1)$ in this trivialization by \eq{ku7eq22}. Thus as $\dot V_{ii}^{v_i}$ is connected, $\la_{ii}$ must map $(v,\ga_i)\mapsto(v,\ga_i)$ on $\dot P_{ii}^{v_i}$. That is, for $(v,\ga_i)\in\dot P_{ii}^{v_i}$ we have
\e
\la_{ii}:(v,\ga_i)\longmapsto
\Ga_j\cdot \bigl[(v,1),(\phi_{ij}^{v_i}(v),\ga_i)\bigr]\in(P_{ij}\t_{\phi_{ij},V_j,\pi_{ji}}P_{ji})/\Ga_j,
\label{ku7eq29}
\e
using \eq{ku7eq25} and \eq{ku7eq28} to identify $(v,1),(\phi_{ij}^{v_i}(v),\ga_i)$ with points of $P_{ij},P_{ji}$.

Now $\la_{ii}$ is equivariant under two commuting actions of $\Ga_i$. The first acts on both sides of \eq{ku7eq29} by $\ga:\ga_i\mapsto\ga\ga_i$. Under the second, $\dot P_{ii}^{v_i}$ is only invariant under $\De_i\subseteq\Ga_i$, and $\de\in\De_i$ acts on the left of \eq{ku7eq29} by $\de:(v,\ga_i)\mapsto(\de\cdot v,\ga_i\de^{-1})$ and on the right by
\begin{align*}
\de:
\Ga_j\cdot \bigl[(v,1),(\phi_{ij}^{v_i}(v),\ga_i)\bigr]
&\longmapsto
\Ga_j\cdot \bigl[(\de\cdot v,\rho(\de)^{-1}),(\phi_{ij}^{v_i}(v),\ga_i)\bigr]\\
&=\Ga_j\cdot \bigl[(\de\cdot v,1),(\rho(\de)\cdot\phi_{ij}^{v_i}(v),\ga_i(\si\ci\rho(\de))^{-1})\bigr],
\end{align*}
where in the second step we use the $\Ga_j$-actions in \eq{ku7eq25} and \eq{ku7eq28} to multiply both points by $\rho(\de)$. Thus $\ga_{ii}$ being $\De_i$-equivariant means that
\e
\si\ci\rho(\de)=\de\quad\text{for all $\de\in\De_i$.}
\label{ku7eq30}
\e

We have a vector bundle morphism $\hat\la_{ii}:\pi_{ii}^*(E_i)\vert_{\dot P_{ii}}\ra \phi_{ii}^*(TV_i)\vert_{\dot P_{ii}}$ over $\dot P_{ii}=\dot V_{ii}\t\Ga_i$, for $\pi_{ii},\phi_{ii}$ as in \eq{ku4eq5}. Define a vector bundle morphism $\hat\la_{ii}^{v_i}:E_i\vert_{\dot V_{ii}^{v_i}}\ra TV_i\vert_{\dot V_{ii}^{v_i}}$ by $\hat\la_{ii}^{v_i}\vert_v=\hat\la_{ii}\vert_{(v,1)}$. Equation \eq{ku4eq1} for $\La_{ii}$ on $\dot P_{ii}$, restricted to $\dot V_{ii}^{v_i}\t\{1\}\subseteq \dot P_{ii}^{v_i}$, becomes
\e
\begin{aligned}
\phi_{ji}^{v_j}\ci\phi_{ij}^{v_i}\vert_{\dot V_{ii}^{v_i}}&=\id_{V_i}\vert_{\dot V_{ii}^{v_i}}+\hat\la_{ii}^{v_i}\cdot s_i+O(s_i^2)&&\text{and}\\ 
(\phi_{ij}^{v_i})^*(\hat\phi_{ji}^{v_j})\ci\hat\phi_{ij}^{v_i}\vert_{\dot V_{ii}^{v_i}}&=\id_{E_i}\vert_{\dot V_{ii}^{v_i}}+\hat\la_{ii}^{v_i}\cdot\d s_i+O(s_i)&&\text{on $\dot V_{ii}^{v_i}$.}
\end{aligned}
\label{ku7eq31}
\e

In the same way, from $\Mu_{jj}$ rather than $\La_{ii}$, using \eq{ku7eq23} we show that
\e
\rho\ci\si(\de)=\de\quad\text{for all $\de\in\De_j$,}
\label{ku7eq32}
\e
and we choose a connected open neighbourhood $\dot V_{jj}^{v_j}$ of $v_j$ in $\dot V_{jj}\cap V_{ji}^{v_j}\cap(\phi_{ji}^{v_i})^{-1}(V_{ij}^{v_i})$, and construct $\hat\mu_{jj}^{v_j}:E_j\vert_{\dot V_{jj}^{v_j}}\ra TV_j\vert_{\dot V_{jj}^{v_j}}$ satisfying
\e
\begin{split}
\phi_{ij}^{v_i}\ci\phi_{ji}^{v_j}\vert_{\dot V_{jj}^{v_j}}&=\id_{V_j}\vert_{\dot V_{jj}^{v_j}}+\hat\mu_{jj}^{v_j}\cdot s_j+O(s_j^2)\;\>\text{and}\\ 
(\phi_{ji}^{v_j})^*(\hat\phi_{ij}^{v_i})\ci\hat\phi_{ji}^{v_j}\vert_{\dot V_{jj}^{v_j}}&=\id_{E_j}\vert_{\dot V_{jj}^{v_j}}+\hat\mu_{jj}^{v_j}\cdot \d s_j+O(s_j)\;\> \text{on $\dot V_{jj}^{v_j}$.}
\end{split}
\label{ku7eq33}
\e

By equations \eq{ku7eq30} and \eq{ku7eq32}, $\rho:\De_i\ra\De_j$ and $\si:\De_j\ra\De_i$ are inverse, and so isomorphisms. But \eq{ku4eq10} is the projection
\e
\bigl\{(\de,\rho(\de)):\de\in\De_i\bigr\}\longra\De_j,\quad (\de,\rho(\de))\longmapsto\rho(\de),
\label{ku7eq34}
\e
so $\rho$ an isomorphism implies that \eq{ku4eq10} is an isomorphism, as we have to prove.

As for \eq{ku6eq17}, restricting the derivatives of the first equations of \eq{ku7eq31}, \eq{ku7eq33}, and the second equations of \eq{ku7eq31}, \eq{ku7eq33}, to $v_i,v_j$ yields
\ea
\begin{split}
\d\phi_{ji}^{v_j}\vert_{v_j}\ci\d\phi_{ij}^{v_i}\vert_{v_i}&=\id_{T_{v_i}V_i}+\hat\la_{ii}^{v_i}\vert_{v_i}\ci\d s_i\vert_{v_i},\\
\hat\phi_{ji}^{v_j}\vert_{v_j}\ci\hat\phi_{ij}^{v_i}\vert_{v_i}&=\id_{E_i\vert_{v_i}}+\d s_i\vert_{v_i}\ci\hat\la_{ii}^{v_i}\vert_{v_i},
\end{split}
\label{ku7eq35}\\
\begin{split}
\d\phi_{ij}^{v_i}\vert_{v_i}\ci\d\phi_{ij}^{v_j}\vert_{v_j}&=\id_{T_{v_j}V_j}+\hat\mu_{jj}^{v_i}\vert_{v_j}\ci\d s_j\vert_{v_j},\\
\hat\phi_{ij}^{v_i}\vert_{v_i}\ci\hat\phi_{ji}^{v_j}\vert_{v_j}&=\id_{E_j\vert_{v_j}}+\d s_j\vert_{v_j}\ci\hat\mu_{jj}^{v_j}\vert_{v_j}.
\end{split}
\label{ku7eq36}
\ea
Now $\d\phi_{ij}^{v_i}\vert_{v_i}=\d\phi_{ij}\vert_p\ci\d\pi_{ij}\vert_p^{-1}$ and $\hat\phi_{ij}^{v_i}\vert_{v_i}=\hat\phi_{ij}\vert_p$, so \eq{ku4eq9} may be rewritten as
\e
\begin{gathered}
\smash{\xymatrix@C=17pt{ 0 \ar[r] & T_{v_i}V_i \ar[rrr]^(0.39){\d s_i\vert_{v_i}\op\d\phi_{ij}^{v_i}\vert_{v_i}} &&& E_i\vert_{v_i} \!\op\!T_{v_j}V_j 
\ar[rrr]^(0.56){-\hat\phi_{ij}^{v_i}\vert_{v_i}\op \d s_j\vert_{v_j}} &&& E_j\vert_{v_j} \ar[r] & 0. }}
\end{gathered}
\label{ku7eq37}
\e
Therefore the proof in \S\ref{ku621} that \eq{ku2eq8} is exact implies that \eq{ku4eq9} is exact, using \eq{ku7eq35}--\eq{ku7eq36} in place of \eq{ku6eq17}--\eq{ku6eq18}. This proves the `only if' part of Theorem~\ref{ku4thm2}.

\subsubsection{The `if' part of Theorem \ref{ku4thm2}}
\label{ku722}

Let $\Phi_{ij}=(P_{ij},\pi_{ij},\phi_{ij},\hat\phi_{ij}):(V_i,E_i,\Ga_i,s_i,\psi_i)\ra (V_j,E_j,\Ga_j,s_j,\psi_j)$ be a 1-morphism of Kuranishi neighbourhoods over $S\subseteq X$. Suppose that for all $p\in\pi_{ij}^{-1}(\bar\psi_i^{-1}(S))\subseteq P_{ij}$ with $v_i=\pi_{ij}(p)\in V_i$, $v_j=\phi_{ij}(p)\in V_j$ and $x=\bar\psi_i(v_i)=\bar\psi_j(v_j)\in S$, equation \eq{ku4eq9} is exact and \eq{ku4eq10} is an isomorphism. We must show $\Phi_{ij}$ is a coordinate change over~$S$.

Fix $p\in\pi_{ij}^{-1}(\bar\psi_i^{-1}(S))\subseteq P_{ij}$, and set $v_i=\pi_{ij}(p)\in V_i$, $v_j=\phi_{ij}(p)\in V_j$, and $x=\bar\psi_i(v_i)=\bar\psi_j(v_j)\in S$. We can find such $p,v_i,v_j$ for any $x\in S$. In a similar way to \S\ref{ku622}, we will first show that there exists an open neighbourhood $S^x$ of $x$ in $S$, a 1-morphism $\Phi_{ji}^x:(V_j,E_j,\Ga_j,s_j,\psi_j)\ra(V_i,E_i,\Ga_i,s_i,\psi_i)$ over $S^x$, and 2-morphisms $\La_{ii}^x:\id_{(V_i,E_i,\Ga_i,s_i,\psi_i)}\Ra\Phi_{ji}^x\ci\Phi_{ij}\vert_{S^x}$, $\Mu_{jj}:\id_{(V_j,E_j,\Ga_j,s_j,\psi_j)}\Ra\Phi_{ij}\vert_{S^x}\ci\Phi_{ji}^x$ over $S^x$, so that $\Phi_{ij}\vert_{S^x}$ is a coordinate change over $S^x$. The `if' part will then follow from the stack property of~$\bcEqu\bigl((V_i,E_i,\Ga_i,s_i,\psi_i),(V_j,E_j,\Ga_j,s_j,\psi_j)\bigr)$.

As far as possible we will use the notation of \S\ref{ku721}. Define $\De_i\subseteq\Ga_i$ and $\De_j\subseteq\Ga_j$ by \eq{ku7eq24}, write $V_{ij}=\pi_{ij}(P_{ij})$, choose a connected, simply-connected, $\De_i$-invariant open neighbourhood $V_{ij}^{v_i}$ of $v_i$ in $V_{ij}\subseteq V_i$ such that $V_{ij}^{v_i}\cap\ga_i(V_{ij}^{v_i})=\es$ for all $\ga_i\in\Ga_i\sm\De_i$, identify $P_{ij}^{v_i}=\pi_{ij}^{-1}(V_{ij}^{v_i})\cong V_{ij}^{v_i}\t\Ga_j$ as in \eq{ku7eq25}, define a group morphism $\rho:\De_i\ra\De_j$ as in \eq{ku7eq26}, a $\rho$-equivariant smooth map $\phi_{ij}^{v_i}:V_{ij}^{v_i}\ra V_j$ by $\phi_{ij}^{v_i}:v\mapsto \phi_{ij}(v,1)$, and a $\rho$-equivariant vector bundle morphism $\hat\phi_{ij}^{v_i}:E_i\vert_{V_{ij}^{v_i}}\ra(\phi_{ij}^{v_i})^*(TV_j)$ by $\hat\phi_{ij}^{v_i}\vert_v=\hat\phi_{ij}\vert_{(v,1)}$, so that
$\hat\phi_{ij}^{v_i}(s_i\vert_{V_{ij}^{v_i}})=(\phi_{ij}^{v_i})^*(s_j)+O(s_i^2)$, all as in~\S\ref{ku721}.

By assumption \eq{ku4eq9} is exact, so \eq{ku7eq37} is exact, and \eq{ku4eq10} is an isomorphism, so \eq{ku7eq34} is an isomorphism, and $\rho:\De_i\ra\De_j$ is an isomorphism. Define~$\si=\rho^{-1}:\De_j\ra\De_i$.

We are now in almost exactly the same the situation as \S\ref{ku622}, but with extra group equivariance. That is, the triple $(V_{ij}^{v_i},\phi_{ij}^{v_j},\hat\phi_{ij}^{v_i})$ above corresponds to $(V_{ij}^x,\phi_{ij}^x,\hat\phi_{ij}^x)$ in \S\ref{ku622}, with $\hat\phi_{ij}^{v_i}(s_i\vert_{V_{ij}^{v_i}})=(\phi_{ij}^{v_i})^*(s_j)+O(s_i^2)$ and \eq{ku7eq37} exact corresponding to Definition \ref{ku2def3}(d) and \eq{ku2eq8} exact. The main difference is that we now have finite groups $\De_i$ and $\De_j$ acting on $V_{ij}^{v_i},E_i\vert_{V_{ij}^{v_i}}$ and $V_j,E_j$, and an isomorphism $\rho:\De_i\ra\De_j$ such that $\phi_{ij}^{v_j},\hat\phi_{ij}^{v_i}$ are $\rho$-equivariant.

We can now follow the argument of \S\ref{ku622}, but including equivariance under $\De_i,\De_j,\rho,\si$. Thus, possibly making $V_{ij}^{v_i}$ smaller, we construct a $\De_j$-invariant open neighbourhood $V_{ji}^{v_j}$ of $v_j$ in $V_j$, a $\si$-equivariant smooth map $\phi_{ji}^{v_j}:V_{ji}^{v_j}\ra V_i$ such that the following maps are inverse
\e
\xymatrix@C=100pt{
V_{ij}^{v_i}\cap s_i^{-1}(0) \ar@<.5ex>[r]^{ \phi_{ij}^{v_i}\vert_{s_i^{-1}(0)}} & V_{ji}^{v_j}\cap s_j^{-1}(0), \ar@<.5ex>[l]^{ \phi_{ji}^{v_j}\vert_{s_j^{-1}(0)}} }
\label{ku7eq38}
\e
a $\si$-equivariant vector bundle morphism $\hat\phi_{ji}^{v_j}:E_j\vert_{V_{ji}^{v_j}}\ra(\phi_{ji}^{v_j})^*(TV_i)$ with 
\e
\hat\phi_{ji}^{v_j}(s_j\vert_{V_{ji}^{v_j}})=(\phi_{ji}^{v_j})^*(s_i)+O(s_j^2)\quad\text{on $V_{ji}^{v_j}$,}
\label{ku7eq39}
\e
as in \eq{ku6eq36}, and a $\De_i$-equivariant vector bundle morphism $\hat\la_{ii}^{v_i}:E_i\vert_{\dot V_{ii}^{v_i}}\ra TV_i\vert_{\dot V_{ii}^{v_i}}$ on $\dot V_{ii}^{v_i}:=V_{ij}^{v_i}\cap(\phi_{ij}^{v_i})^{-1}(V_{ji}^{v_i})$ satisfying, as in \eq{ku6eq41} and \eq{ku7eq31}
\e
\begin{split}
\phi_{ji}^{v_j}\ci\phi_{ij}^{v_i}\vert_{\dot V_{ii}^{v_i}}&=\id_{V_i}\vert_{\dot V_{ii}^{v_i}}+\hat\la_{ii}^{v_i}\cdot s_i+O(s_i^2)\;\>\text{and}\\ 
(\phi_{ij}^{v_i})^*(\hat\phi_{ji}^{v_j})\ci\hat\phi_{ij}^{v_i}\vert_{\dot V_{ii}^{v_i}}&=\id_{E_i}\vert_{\dot V_{ii}^{v_i}}+\hat\la_{ii}^{v_i}\cdot\d s_i+O(s_i)\;\> \text{on $\dot V_{ii}^{v_i}$,}
\end{split}
\label{ku7eq40}
\e
and a $\De_j$-equivariant vector bundle morphism $\hat\mu_{jj}^{v_j}:E_j\vert_{\dot V_{jj}^{v_j}}\ra TV_j\vert_{\dot V_{jj}^{v_j}}$ on $\dot V_{jj}^{v_i}:=V_{ji}^{v_j}\cap(\phi_{ji}^{v_j})^{-1}(V_{ij}^{v_i})$ satisfying, as in \eq{ku6eq42} and \eq{ku7eq33}
\begin{align*}
\phi_{ij}^{v_i}\ci\phi_{ji}^{v_j}\vert_{\dot V_{jj}^{v_j}}&=\id_{V_j}\vert_{\dot V_{jj}^{v_j}}+\hat\mu_{jj}^{v_j}\cdot s_j+O(s_j^2)\;\>\text{and}\\ 
(\phi_{ji}^{v_j})^*(\hat\phi_{ij}^{v_i})\ci\hat\phi_{ji}^{v_j}\vert_{\dot V_{jj}^{v_j}}&=\id_{E_j}\vert_{\dot V_{jj}^{v_j}}+\hat\mu_{jj}^{v_j}\cdot \d s_j+O(s_j)\;\> \text{on $\dot V_{jj}^{v_j}$.}
\end{align*}
Including group invariance/equivariance in the proofs of \S\ref{ku622} is easy.

As $\De_j=\Stab_{\Ga_j}(v_j)$, making $V_{ji}^{v_j}$ smaller if necessary we may assume that $V_{ji}^{v_j}\cap\ga_j(V_{ji}^{v_j})=\es$ for all $\ga_j\in\Ga_j\sm\De_j$. Define a manifold $P_{ji}^x$ by
\e
\begin{split}
&P_{ji}^x=\bigl(V_{ji}^x\t\Ga_i\t\Ga_j\bigr)\big/\De_j,\quad\text{where}\\
&\text{$\De_j$ acts freely by}\quad
\de:(v,\ga_i,\ga_j)\longmapsto (\de\cdot v,\ga_i\si(\de)^{-1},\ga_j\de^{-1}).
\end{split}
\label{ku7eq41}
\e
Define commuting $\Ga_i$- and $\Ga_j$-actions on $P_{ji}^x$ by
\e
\begin{split}
\ga_i'&:\De_j\cdot (v,\ga_i,\ga_j)\longmapsto\De_j\cdot (v,\ga_i'\ga_i,\ga_j),\\
\ga_j'&:\De_j\cdot (v,\ga_i,\ga_j)\longmapsto\De_j\cdot (v,\ga_i,\ga_j'\ga_j),
\end{split}
\label{ku7eq42}
\e
which are well-defined as the $\Ga_i,\Ga_j$-actions on $V_{ji}^x\t\Ga_i\t\Ga_j$ commute with the $\De_j$-action. Define smooth maps $\pi_{ji}^x:P_{ji}^x\ra V_j$ and $\phi_{ji}^x:P_{ji}^x\ra V_i$ by
\e
\begin{split}
\pi_{ji}^x&:\De_j\cdot (v,\ga_i,\ga_j)\longmapsto \ga_j\cdot v,\\
\phi_{ji}^x&:\De_j\cdot (v,\ga_i,\ga_j)\longmapsto \ga_i\cdot \phi_{ji}^{v_j}(v),
\end{split}
\label{ku7eq43}
\e
which are well-defined as the maps $(v,\ga_i,\ga_j)\longmapsto \ga_j\cdot v$, $(v,\ga_i,\ga_j)\longmapsto \ga_i\cdot \phi_{ij}^{v_i}(v)$ are invariant under the $\De_j$-action on $V_{ji}^x\t\Ga_i\t\Ga_j$, as $\phi_{ji}^{v_j}$ is $\si$-equivariant.

Define $\hat\phi_{ji}^x:(\pi_{ji}^x)^*(E_j)\ra(\phi_{ji}^x)^*(E_i)$ by the commutative diagram
\begin{equation*}
\xymatrix@C=40pt@R=15pt{
*+[r]{(\pi_{ji}^x)^*(E_j)\vert_{\De_j\cdot (v,\ga_i,\ga_j)}} \ar[rrr]_{\hat\phi_{ji}^x\vert_{\De_j\cdot (v,\ga_i,\ga_j)}} \ar@<2ex>@{=}[d] &&& *+[l]{(\phi_{ji}^x)^*(E_i)\vert_{\De_j\cdot (v,\ga_i,\ga_j)}} \ar@<-2ex>@{=}[d] \\
*+[r]{ E_j\vert_{\ga_j\cdot v}} \ar[r]^{\ga_j^{-1}\cdot} & E_j\vert_v \ar[r]^{\hat\phi_{ji}^{v_j}\vert_v} & E_i\vert_{\phi_{ji}^{v_j}(v)} \ar[r]^(0.4){\ga_i\cdot} & *+[l]{E_i\vert_{\ga_i\cdot\phi_{ji}^{v_j}(v)}.\!} }
\end{equation*}
Define an open neighbourhood $S^x$ of $x$ in $S\subseteq X$ by
\e
S^x\!=\!\bar\psi_i\bigl(V_{ij}^{v_i}\!\cap\! s_i^{-1}(0)\bigr)\!=\!\bar\psi_j(V_{ji}^{v_j}\!\cap\! s_j^{-1}(0)\bigr)\!=\!\psi_j\bigl(((\Im\pi_{ij}^x)\!\cap\! s_j^{-1}(0))/\Ga_j\bigr),
\label{ku7eq44}
\e
where the second equality holds by \eq{ku7eq38} and the third by \eq{ku7eq41} and \eq{ku7eq43}. We now claim that $\Phi_{ji}^x:=(P_{ji}^x,\pi_{ji}^x,\phi_{ji}^x,\hat\phi_{ji}^x):(V_j,E_j,\Ga_j,s_j,\psi_j)\ra (V_i,E_i,\Ga_i,s_i,\psi_i)$ is a 1-morphism over $S^x$. To check this, note that of Definition \ref{ku4def2}(a)--(f), part (e) follows from \eq{ku7eq39}, and the rest are straightforward.

As above $\dot V_{ii}^{v_i}=V_{ij}^{v_i}\cap(\phi_{ij}^{v_i})^{-1}(V_{ji}^{v_i})$ is a $\De_i$-invariant open neighbourhood of $V_{ij}^{v_i}\cap s_i^{-1}(0)$ in $V_i$ on which $\hat\la_{ii}^{v_i}$ is defined, and $\dot V_{ii}^{v_i}\cap\ga_i(\dot V_{ii}^{v_i})=\es$ for $\ga_i\in\Ga_i\sm\De_i$ as this holds for $V_{ij}^{v_i}$. Define $\dot V_{ii}^x=\bigcup_{\ga_i\in\Ga_i}\ga_i(\dot V_{ii}^{v_i})$. Then $\dot V_{ii}^x$ is a $\Ga_i$-invariant open neighbourhood of $\bar\psi_i^{-1}(S^x)$, by \eq{ku7eq44}, and we can think of $\dot V_{ii}^x$ as $\md{\Ga_i}/\md{\De_i}$ disjoint copies of $\dot V_{ii}^{v_i}$. Set $\dot P_{ii}^x=\dot V_{ii}^x\t\Ga_i\subseteq V_i\t\Ga_i$. As in \eq{ku7eq29}, define $\la_{ii}^x:\dot P_{ii}^x\ra (P_{ij}\t_{\phi_{ij},V_j,\pi_{ji}^x}P_{ji}^x)/\Ga_j$ by
\begin{equation*}
\la_{ii}^x:(v_i,\ga_i)\longmapsto
\Ga_j\cdot \bigl[\ga_i'\cdot(v,1),\De_j\cdot(\phi_{ij}^{v_i}(v),\ga_i\ga_i',1)\bigr],
\end{equation*}
if $v_i=\ga_i'\cdot v$ for $v\in V_{ij}^{v_i}$ and $\ga_i'\in\Ga_i$, using \eq{ku7eq25} to identify $(v,1)\in V_{ij}^{v_i}\t\Ga_i$ with a point of $P_{ij}^{v_i}\subseteq P_{ij}$. Here the decomposition $v_i=\ga_i'\cdot v$ is unique up to $v_i=\ga_i'\cdot v=\ti\ga_i'\cdot\ti v$ with $\ti v=\de\cdot v$ and $\ti\ga_i'=\ga_i'\de^{-1}$ for $\de\in\De_i$. But
\begin{align*}
\Ga_j\cdot &\bigl[\ti\ga_i'\cdot(\ti v,1),\De_j\cdot(\phi_{ij}^{v_i}(\ti v),\ga_i\ti\ga_i',1)\bigr]\\
&=\Ga_j\cdot \bigl[\ga_i'\de^{-1}\cdot(\de\cdot v,1),\De_j\cdot(\phi_{ij}^{v_i}(\de\cdot v),\ga_i\ga_i'\de^{-1},1)\bigr]\\
&=\Ga_j\cdot \bigl[\ga_i'\cdot(\de^{-1}\de\cdot v,\rho(\de)),\De_j\cdot(\rho(\de)\phi_{ij}^{v_i}(v),\ga_i\ga_i'\de^{-1},1)\bigr]\\
&=\Ga_j\cdot \bigl[\ga_i'\cdot(v,\rho(\de)),\De_j\cdot(\phi_{ij}^{v_i}(v),\ga_i\ga_i'\de^{-1}\si(\rho(\de)),\rho(\de))\bigr]\\
&=\Ga_j\cdot \bigl[\ga_i'\cdot(v,1),\De_j\cdot(\phi_{ij}^{v_i}(v),\ga_i\ga_i',1)\bigr],
\end{align*}
using $\phi_{ij}^{v_i}$ $\rho$-equivariant and \eq{ku7eq26} in the second step, the action by $\rho(\de)^{-1}\in\De_j$ in \eq{ku7eq41} in the third, and $\si=\rho^{-1}$ and the actions of $\rho(\de)^{-1}\in\Ga_j$ in \eq{ku7eq25} and \eq{ku7eq42} in the fourth. So $\la_{ii}^x$ is well-defined.

Define a vector bundle morphism $\hat\la_{ii}^x:\pi_{ii}^*(E_i)\vert_{\dot P_{ii}^x}\ra \phi_{ii}^*(TV_i)\vert_{\dot P_{ii}^x}$ over $\dot P_{ii}^x$, for $\pi_{ii},\phi_{ii}$ as in \eq{ku4eq5}, by the commutative diagram for $(v,\ga_i)\in \dot P_{ii}^x=\dot V_{ii}^x\t\Ga_i$
\begin{equation*}
\xymatrix@C=40pt@R=15pt{
*+[r]{\pi_{ii}^*(E_i)\vert_{(v_i,\ga_i)}} \ar[rrr]_{\hat\la_{ii}^x\vert_{(v_i,\ga_i)}} \ar@<2ex>@{=}[d] &&& *+[l]{\phi_{ii}^*(TV_i)\vert_{(v_i,\ga_i)}} \ar@<-2ex>@{=}[d] \\
*+[r]{ E_i\vert_{v_i}} \ar[r]^{(\ga_i')^{-1}\cdot} & E_i\vert_v \ar[r]^{\hat\la_{ii}^{v_i}\vert_v} & T_vV_i \ar[r]^(0.4){\ga_i\ga_i'\cdot} & *+[l]{T_{\ga_i\cdot v_i}V_i} }
\end{equation*}
if $v_i=\ga_i'\cdot v$ for $v\in V_{ij}^{v_i}$ and $\ga_i'\in\Ga_i$. Again, the decomposition $v_i=\ga_i\cdot v$ is only unique up to $v_i=\ga_i'\cdot v=\ti\ga_i'\cdot\ti v$ with $\ti v=\de\cdot v$ and $\ti\ga_i'=\ga_i'\de^{-1}$ for $\de\in\De_i$, but $\hat\la_{ii}^x\vert_{(v_i,\ga_i)}$ is independent of this change by $\de$ as $\hat\la_{ii}^{v_i}$ is $\De_i$-equivariant. 

We now see that $(\dot P_{ii}^x,\la_{ii}^x,\hat\la_{ii}^x)$ satisfies the conditions Definition \ref{ku4def3}(a)--(c) for $\La_{ii}^x=[\dot P_{ii}^x,\la_{ii}^x,\hat\la_{ii}^x]:\id_{(V_i,E_i,\Ga_i,s_i,\psi_i)}\Ra\Phi_{ji}^x\ci\Phi_{ij}\vert_{S^x}$ to be a 2-morphism over $S^x$, since \eq{ku4eq1} follows from \eq{ku7eq40}, and the rest of (a)--(c) are straightforward.

In the same way, if we define $\dot V_{jj}^x=\bigcup_{\ga_j\in\Ga_j}\ga_j(\dot V_{jj}^{v_j})$, and $\dot P_{jj}^x=\dot V_{jj}^x\t\Ga_j\subseteq V_j\t\Ga_j$, and $\mu_{jj}^x:\dot P_{jj}^x\ra (P_{ji}^x\t_{\phi_{ji}^x,V_i,\pi_{ij}}P_{ij})/\Ga_i$ by
\begin{equation*}
\mu_{jj}^x:(v_j,\ga_j)\longmapsto
\Ga_i\cdot \bigl[\De_j\cdot(v,1,\ga_j'),(\phi_{ji}^{v_i}(v),\ga_j\ga_j')\bigr],
\end{equation*}
if $v_j=\ga_j'\cdot v$ for $v\in\dot V_{jj}^{v_j}$ and $\ga_j'\in\Ga_j$, and $\hat\mu_{jj}^x:\pi_{jj}^*(E_j)\vert_{\dot P_{jj}^x}\ra \phi_{jj}^*(TV_j)\vert_{\dot P_{jj}^x}$ by the commutative diagram
\begin{equation*}
\xymatrix@C=40pt@R=15pt{
*+[r]{\pi_{jj}^*(E_j)\vert_{(v_j,\ga_j)}} \ar[rrr]_{\hat\mu_{jj}^x\vert_{(v_j,\ga_j)}} \ar@<2ex>@{=}[d] &&& *+[l]{\phi_{jj}^*(TV_j)\vert_{(v_j,\ga_j)}} \ar@<-2ex>@{=}[d] \\
*+[r]{E_j\vert_{v_j}} \ar[r]^{(\ga_j')^{-1}\cdot} & E_j\vert_v \ar[r]^{\hat\mu_{jj}^{v_j}\vert_v} & T_vV_j \ar[r]^(0.4){\ga_j\ga_j'\cdot} & *+[l]{T_{\ga_j\cdot v_j}V_j} }
\end{equation*}
if $v_j=\ga_j'\cdot v$ for $v\in\dot V_{jj}^{v_j}$ and $\ga_j'\in\Ga_j$, then $\mu_{jj}^x,\hat\mu_{jj}^x$ are well-defined and $\Mu_{jj}^x=[\dot P_{jj}^x,\mu_{jj}^x,\hat\mu_{jj}^x]:\id_{(V_j,E_j,\Ga_j,s_j,\psi_j)}\Ra\Phi_{ij}\vert_{S^x}\ci\Phi_{ji}^x$ is a 2-morphism over~$S^x$.

For each $x\in S$, we have constructed an open neighbourhood $S^x$ of $x$ in $S$, a 1-morphism $\Phi_{ji}^x:(V_j,E_j,\Ga_j,s_j,\psi_j)\ra (V_i,E_i,\Ga_i,s_i,\psi_i)$ over $S^x$, and 2-morphisms $\La_{ii}^x:\id_{(V_i,E_i,\Ga_i,s_i,\psi_i)}\Ra\Phi_{ji}^x\ci\Phi_{ij}\vert_{S^x}$ and $\Mu_{jj}^x:\id_{(V_j,E_j,\Ga_j,s_j,\psi_j)}\Ra\Phi_{ij}\vert_{S^x}\ci\Phi_{ji}^x$ over $S^x$. Hence $\Phi_{ij}\vert_{S^x}$ is an equivalence in $\Kur_{S^x}(X)$, and a coordinate change over $S^x$. Thus by the stack property of $\bcEqu((V_i,E_i,\Ga_i,s_i,\psi_i),\ab(V_j,\ab E_j,\ab\Ga_j,\ab s_j,\ab\psi_j))$ in Theorem \ref{ku4thm1}(a), proved in \S\ref{ku712}, $\Phi_{ij}$ is a coordinate change on $S$. This completes the `if' part, and so the whole of Theorem~\ref{ku4thm2}.

\subsection{Proof of Theorem \ref{ku4thm8}}
\label{ku73}

\subsubsection{Theorem \ref{ku4thm8}(a): $F_\mKur^\muKur:\Ho(\mKur)\ra\muKur$ an equivalence}
\label{ku731}

Use the notation of Definition \ref{ku4def35}. To show $F_\mKur^\muKur:\Ho(\mKur)\ra\muKur$ is an equivalence of categories, we have to prove three things: that $F_\mKur^\muKur$ is faithful (injective on morphisms), and full (surjective on morphisms), and surjective on isomorphism classes of objects.

The proofs of these will involve gluing together 2-morphisms of m-Kuranishi neighbourhoods using families of partitions of unity, so we begin by showing that partitions of unity with the properties we need exist. 

\subsubsection*{A lemma on partitions of unity on $\bX$ in $\muKur$} 

Let $\bX=(X,\cI)$ be a $\mu$-Kuranishi space, with $\cI=\bigl(I,(U_i,D_i,r_i,\chi_i)_{i\in I}$, $\Tau_{ij}=[U_{ij},\tau_{ij},\hat\tau_{ij}]_{i,j\in I}\bigr)$, as in \eq{ku2eq14}. Then $\bigl\{\Im\chi_i:i\in I\bigr\}$ is an open cover of $X$, with $\chi_i:r_i^{-1}(0)\ra \Im\chi_i$ a homeomorphism for each $i\in I$.

Roughly speaking, we want to define a smooth partition of unity $\{\eta_i:i\in I\}$ on $X$ subordinate to $\bigl\{\Im\chi_i:i\in I\bigr\}$, so that $\eta_i:X\ra\R$ is smooth with $\eta_i(X)\subseteq[0,1]$ and $\sum_{i\in I}\eta_i=1$, as in Definition \ref{ku6def}. However, $X$ is not a manifold, so na\"\i vely `$\eta_i:X\ra\R$ is smooth' does not make sense.

In fact we will not work with `smooth functions' $\eta_i$ on $X$ directly, apart from in the proof of Lemma \ref{ku7lem1}. Instead, for each $i\in I$ we want a partition of unity $\{\eta_{ij}:j\in I\}$ on $U_i$, such that $\eta_{ij}\vert_{r_i^{-1}(0)}=\eta_j\ci\chi_i$ for each $j\in I$. Then $\eta_{ij}:U_i\ra\R$ is a smooth function on the manifold $U_i$ in the usual sense. The fact that $\eta_{ik}:U_i\ra\R$ and $\eta_{jk}:U_j\ra\R$ both come from the same $\eta_k:X\ra\R$ is expressed in the condition $\eta_{ik}=\eta_{jk}\ci\tau_{ij}+O(r_i)$ on $U_{ij}\subseteq U_i$ for all $i,j\in I$. So our result Lemma \ref{ku7lem1} is stated using only smooth functions on manifolds.

But to prove Lemma \ref{ku7lem1}, it is convenient to first choose a `smooth partition of unity' $\{\eta_i:i\in I\}$ on $X$ subordinate to $\bigl\{\Im\chi_i:i\in I\bigr\}$, so that $\{\eta_j\ci\chi_i:j\in I\}$ is a partition of unity on $r_i^{-1}(0)\subseteq U_i$, and then extend this from $r_i^{-1}(0)$ to $U_i$. To do this we have to interpret $X$ and $r_i^{-1}(0)$ as some kind of `smooth space'. We do this using $C^\iy$-{\it schemes\/} and $C^\iy$-{\it algebraic geometry}, as in \cite{Joyc4,Joyc5}, which are the foundation of the author's theory of d-manifolds and d-orbifolds in \cite{Joyc6,Joyc7,Joyc8}. Lemma \ref{ku7lem1} is the only result in this book that uses $C^\iy$-algebraic geometry.

\begin{lem} Let\/ $\bX=(X,\cI)$ be a $\mu$-Kuranishi space, with notation \eq{ku2eq14} for $\cI,$ and let\/ $(U_{ij},\tau_{ij},\hat\tau_{ij})$ represent $\Tau_{ij}$ for $i,j\in I,$ with\/ $(U_{ii},\tau_{ii},\hat\tau_{ij})=(U_i,\id_{U_i},\id_{D_i})$. Then for all\/ $i\in I$ we can choose a partition of unity $\{\eta_{ij}:j\in I\}$ on $U_i$ subordinate to the open cover $\{U_{ij}:j\in I\}$ of\/ $U_i,$ as in Definition\/ {\rm\ref{ku6def},} such that for all\/ $i,j,k\in I$ we have
\e
\eta_{ik}\vert_{U_{ij}}=\eta_{jk}\ci\tau_{ij}+O(r_i)\quad\text{on $U_{ij}\subseteq U_i$.}
\label{ku7eq45}
\e 

\label{ku7lem1}
\end{lem}

\begin{proof} We use notation and results on $C^\iy$-schemes and $C^\iy$-algebraic geometry from \cite{Joyc4,Joyc5}, in which $C^\iy$-schemes are written $\uX=(X,\O_X)$ for $X$ a topological space and $\O_X$ a sheaf of $C^\iy$-rings on $X$, satisfying certain conditions.

For each $i\in I$, the manifold $U_i$ naturally becomes a $C^\iy$-scheme $\uU_i$, and $r_i^{-1}(0)\subseteq U_i$ becomes the closed $C^\iy$-subscheme $\ur_i^{-1}(0)$ in $\uU_i$ defined by the equation $r_i=0$. If $i,j\in I$ and $(U_{ij},\tau_{ij},\hat\tau_{ij})$ represents $\Tau_{ij}$, then $\hat\tau_{ij}(r_i\vert_{U_{ij}})=\tau_{ij}^*(r_j)+O(r_i^2)$ on $U_{ij}$ by Definition \ref{ku2def3}(d). This implies that $\utau_{ij}:\uU_{ij}\ra\uU_j$ restricts to an isomorphism of $C^\iy$-schemes
\e
\utau_{ij}\vert_{\uU_{ij}\cap \ur_i^{-1}(0)}:\uU_{ij}\cap \ur_i^{-1}(0)\ra \uU_{ji}\cap\ur_j^{-1}(0).
\label{ku7eq46}
\e

We now have a topological space $X$, an open cover $\{\Im\chi_i:i\in I\}$ on $X$, $C^\iy$-schemes $\ur_i^{-1}(0)$ with underlying topological spaces $r_i^{-1}(0)$ and homeomorphisms $\chi_i:r_i^{-1}(0)\ra \Im\chi_i\subseteq X$ for all $i\in I$, and isomorphisms of $C^\iy$-schemes \eq{ku7eq46} lifting the homeomorphisms $\chi_j^{-1}\ci\chi_i:U_{ij}\cap r_i^{-1}(0)\ra U_{ji}\cap r_j^{-1}(0)$ over double overlaps $\Im\chi_i\cap\Im\chi_j\subseteq X$. From $\Tau_{jk}\ci\Tau_{ij}=\Tau_{ik}$ in Definition \ref{ku2def11}(f), we deduce that the isomorphisms \eq{ku7eq46} have the obvious composition property $\utau_{jk}\vert_{\cdots}\ci\utau_{ij}\vert_{\cdots}=\utau_{ik}\vert_{\cdots}$ over triple overlaps~$\Im\chi_i\cap\Im\chi_j\cap\Im\chi_k\subseteq X$.

Standard results on schemes (actually, just the fact that sheaves of $C^\iy$-rings on $X$ form a stack on $X$) imply that $X$ may be made into a $C^\iy$-scheme $\uX$, uniquely up to unique isomorphism, and the homeomorphisms $\chi_i:r_i^{-1}(0)\ra \Im\chi_i\subseteq X$ upgraded to $C^\iy$-scheme morphisms $\uchi_i:\ur_i^{-1}(0)\ra\uX$ which are isomorphisms with open $C^\iy$-subschemes $\Im\uchi_i\subseteq\uX$ for $i\in I$, such that
\e
\uchi_j\ci\utau_{ij}\vert_{\uU_{ij}\cap \ur_i^{-1}(0)}=\uchi_i\vert_{\uU_{ij}\cap\ur_i^{-1}(0)}
\label{ku7eq47}
\e
for all $i,j\in I$. This $\uX$ is {\it locally fair\/} \cite[\S 4.5]{Joyc4}, \cite[\S 3.1]{Joyc5} (locally finite-dimensional) as the $\ur_i^{-1}(0)$ are, since they are $C^\iy$-subschemes of manifolds~$\uU_i$.

Now \cite[Prop.~4.35]{Joyc4} and \cite[Prop.~3.6]{Joyc5} say that if $\uX=(X,\O_X)$ is a locally fair $C^\iy$-scheme with $X$ Hausdorff and paracompact, and $\{\uX_i:i\in I\}$ is an open cover of $\uX$, then there exists a locally finite partition of unity $\{\ueta_i:i\in I\}$ on $\uX$ subordinate to $\{\uX_i\in I\}$, in the sense of \cite[Def.~4.34]{Joyc4}, \cite[Def.~3.5]{Joyc5}. By definition $X$ is Hausdorff and second countable, and it is locally compact as it has a cover by $\mu$-Kuranishi neighbourhoods, so it is paracompact. Thus we can choose a partition of unity $\{\ueta_i:i\in I\}$ on $\uX$ subordinate to~$\{\Im\uchi_i:i\in I\}$.

Then for each $i\in I$, $\{\ueta_j\ci\uchi_i:j\in I\}$ is a partition of unity on the $C^\iy$-scheme $\ur_i^{-1}(0)$ subordinate to the open cover $\bigl\{\uU_{ij}\cap \ur_i^{-1}(0):j\in J\bigr\}$. From the proof of the existence of partitions of unity on $C^\iy$-schemes \cite[\S 4.5]{Joyc4}, we see that a partition of unity on $\ur_i^{-1}(0)\subseteq\uU_i$ subordinate to $\bigl\{\uU_{ij}\cap \ur_i^{-1}(0):j\in J\bigr\}$ can be extended to a partition of unity on $\uU_i$ subordinate to $\bigl\{\uU_{ij}:j\in J\bigr\}$, which is a partition of unity on $U_i$ in the sense of Definition \ref{ku6def}, as $U_i$ is a manifold. 

Thus, for all $i\in I$ we can choose a partition of unity $\{\eta_{ij}:j\in I\}$ on $U_i$ subordinate to $\{U_{ij}:j\in I\}$, such that $\eta_{ij}\vert_{\ur_i^{-1}(0)}=\ueta_j\ci\uchi_i$ for all $j\in I$, in the sense of $C^\iy$-schemes. If $i,j,k\in I$ then
\begin{equation*}
\eta_{ik}\vert_{\uU_{ij}\cap\ur_i^{-1}(0)}\!=\!\ueta_k\ci\uchi_i\vert_{\uU_{ij}\cap\ur_i^{-1}(0)}\!=\!\ueta_k\ci\uchi_j\ci\utau_{ij}\vert_{\uU_{ij}\cap \ur_i^{-1}(0)}\!=\!\eta_{jk}\ci\tau_{ij}\vert_{\uU_{ij}\cap\ur_i^{-1}(0)},
\end{equation*} 
using \eq{ku7eq47}. But $f\vert_{\uU_{ij}\cap\ur_i^{-1}(0)}=g\vert_{\uU_{ij}\cap\ur_i^{-1}(0)}$ for smooth $f,g:U_i\ra\R$ is equivalent to $f=g+O(r_i)$ on $U_{ij}$, so equation \eq{ku7eq45} follows.
\end{proof}

\subsubsection*{$F_\mKur^\muKur$ is faithful} 

Let $\bs f,\bs g:\bX\ra\bY$ be 1-morphisms in $\mKur$, so that $[\bs f],[\bs g]:\bX\ra\bY$ are morphisms in $\Ho(\mKur)$. Write $\bX',\bY',\bs f',\bs g'$ for the images of $\bX,\bY,[\bs f],[\bs g]$ under $F_\mKur^\muKur$. Suppose $\bs f'=\bs g'$. We must show that $[\bs f]=[\bs g]$, that is, that there exists a 2-morphism $\bs\mu:\bs f\Ra\bs g$ in $\mKur$.

Use notation
\ea
\bX&=(X,\cI),& \cI&=\bigl(I,(U_i,D_i,r_i,\chi_i)_{i\in I},\;
\Tau_{ii'}=(U_{ii'},\tau_{ii'},\hat\tau_{ii'})_{i,i'\in I},\nonumber\\
&&& \qquad \Ka_{ii'i''}=[\dot U_{ii'i''},\hat\ka_{ii'i''}]_{i,i',i''\in I}\bigr),
\label{ku7eq48}\\
\bY&=(Y,\cJ),& \cJ&=\bigl(J,(V_j,E_j,s_j,\psi_j)_{j\in J},\; \Up_{jj'}=(V_{jj'},\up_{jj'},\hat\up_{jj'})_{j,j'\in J},\nonumber\\
&&& \qquad \La_{jj'j''}=[\dot V_{jj'j''},\hat\la_{jj'j''}]_{j,j',j''\in J}\bigr),
\label{ku7eq49}
\ea
and as in Definition \ref{ku4def13} write
\begin{align*}
\bs f&=\bigl(f,\bs f_{ij,\;i\in I,\; j\in J},\; \bs F_{ii',\;i,i'\in I}^{j,\; j\in J},\; \bs F_{i,\;i\in I}^{jj',\; j,j'\in J}\bigr)\;\>\text{with}\;\>\bs f_{ij}=(U_{ij},f_{ij},\hat f_{ij}),\\
\bs g&=\bigl(g,\bs g_{ij,\;i\in I,\; j\in J},\; \bs G_{ii',\;i,i'\in I}^{j,\; j\in J},\; \bs G_{i,\;i\in I}^{jj',\; j,j'\in J}\bigr)\;\>\text{with}\;\>\bs g_{ij}=(\ti U_{ij},g_{ij},\hat g_{ij}).
\end{align*}
Then $\bs f'=\bs g'$ means that $f=g$, and for all $i\in I$ and $j\in J$
\begin{equation*}
[U_{ij},f_{ij},\hat f_{ij}]=[\ti U_{ij},g_{ij},\hat g_{ij}]
\end{equation*}
as morphisms $(U_i,D_i,r_i,\chi_i)\ra (V_j,E_j,s_j,\psi_j)$ of $\mu$-Kuranishi neighbourhoods over $(S,f)$ in the sense of \S\ref{ku21}, where~$S=\Im\chi_i\cap f^{-1}(\Im\psi_j)$.

So by Definition \ref{ku2def3} there exists an open neighbourhood $\dot U_{ij}$ of $\chi_i^{-1}(S)$ in $\check U_{ij}:=U_{ij}\cap\ti U_{ij}$ and a morphism $\hat\la_{ij}:D_i\vert_{\dot U_{ij}}\ra f_{ij}^*(TV_j)\vert_{\dot U_{ij}}$ satisfying
\e
\begin{split}
g_{ij}&=f_{ij}+\hat\la_{ij}\cdot r_i+O(r_i^2)
\;\>\text{and}\\
\hat g_{ij}&=\hat f_{ij}+\hat\la_{ij}\cdot f_{ij}^*(\d s_j)+O(r_i)\;\> \text{on $\dot U_{ij}$.}
\end{split}
\label{ku7eq50}
\e
Definition \ref{ku4def31} now implies that
\e
\bs\la_{ij}=[\dot U_{ij},\hat\la_{ij}]:\bs f_{ij}\Longra\bs g_{ij}
\label{ku7eq51}
\e
is a 2-morphism of m-Kuranishi neighbourhoods over $(S,f)$ in the sense of \S\ref{ku47}.

We would like $\bs\la=\bigl(\bs\la_{ij,\; i\in I,\; j\in J}\bigr):\bs f\Ra\bs g$ to be a 2-morphism of m-Kuranishi spaces, but there is a problem: as the $\bs\la_{ij}$ are chosen arbitrarily, they have no compatibility with the $\bs F_{ii'}^j,\bs F_i^{jj'},\bs G_{ii'}^j,\bs G_i^{jj'}$, so Definition \ref{ku4def14}(a),(b) may not hold for $\bs\la$. We will define a modified version $\bs\mu=\bigl(\bs\mu_{ij,\; i\in I,\; j\in J}\bigr)$ of $\bs\la$ which does have the required compatibility.

For $i,\ti\imath\in I$ and $j,\ti\jmath\in I$, define $\bs\la_{i\ti\imath}^{\ti\jmath j}$ to be the horizontal composition of 2-morphisms over $S=\Im\chi_i\cap \Im\chi_{\ti\imath}\cap f^{-1}(\Im\psi_j\cap\Im\psi_{\ti\jmath})$ and $f:X\ra Y$
\e
\xymatrix@C=34pt{\bs f_{ij} \ar@{=>}[rr]^(0.45){(\bs F_i^{\ti\jmath j}\od(\id*\bs F_{i\ti\imath}^{\ti\jmath}))^{-1}}_(0.45){=(\bs F_{i\ti\imath}^{j}\od(\bs F_{\ti\imath}^{\ti\jmath j}*\id)^{-1}} && {\begin{subarray}{l}\ts\Up_{\ti\jmath j}\ci \\ \ts\bs f_{\ti\imath \ti\jmath}\!\ci\!\Tau_{i\ti\imath}\end{subarray}} \ar@{=>}[r]^(0.45){\id*\bs\la_{\ti\imath \ti\jmath}*\id} & {\begin{subarray}{l}\ts\Up_{\ti\jmath j}\ci \\ \ts\bs g_{\ti\imath \ti\jmath}\!\ci\!\Tau_{i\ti\imath}\end{subarray}} \ar@{=>}[rr]^(0.55){\bs G_i^{\ti\jmath j}\od(\id*\bs G_{i\ti\imath}^{\ti\jmath})}_(0.55){=\bs G_{i\ti\imath}^j\od(\bs G_{\ti\imath}^{\ti\jmath j} *\id))} && \bs g_{ij}, }\!\!\!{}
\label{ku7eq52}
\e
where the alternative expressions for the first and third 2-morphisms come from Definition \ref{ku4def13}(g), but omitting terms like $\bs\al_{\Up_{jj'},\bs f_{i'j},\Tau_{ii'}}$ since $\mKur_S(X)$ is a strict 2-category as in \S\ref{ku47}, rather than a weak 2-category like $\Kur_S(X)$ in~\S\ref{ku41}.

Define $\dot U_{i\ti\imath}^{\ti\jmath j}=\check U_{ij}\cap U_{i\ti\imath}\cap f_{ij}^{-1}(V_{j\ti\jmath})$. Then we may write
\e
\bs\la_{i\ti\imath}^{\ti\jmath j}\!=\!\bigl[\dot U_{i\ti\imath}^{\ti\jmath j},\hat\la_{i\ti\imath}^{\ti\jmath j}\bigr],
\label{ku7eq53}
\e
for $\hat\la_{i\ti\imath}^{\ti\jmath j}:D_i\vert_{\dot U_{i\ti\imath}^{\ti\jmath j}}\ra f_{ij}^*(TV_j)\vert_{\dot U_{i\ti\imath}^{\ti\jmath j}}$ a morphism satisfying the analogue of~\eq{ku7eq50}.

Apply Lemma \ref{ku7lem1} to $\bX'=F_\mKur^\muKur(\bX)$, using $(U_{ii'},\tau_{ii'},\hat\tau_{ii'})$ to represent $\Tau'_{ii'}$. This gives a partition of unity $\{\eta_{i\ti\imath}:\ti\imath\in I\}$ on $U_i$ subordinate to $\{U_{i\ti\imath}:\ti\imath\in I\}$ for each $i\in I$, such that for all $i,i',\ti\imath\in I$ we have
\e
\eta_{i\ti\imath}\vert_{U_{ii'}}=\eta_{i'\ti\imath}\ci\tau_{ii'}+O(r_i)\quad\text{on $U_{ii'}\subseteq U_i$.}
\label{ku7eq54}
\e 
Similarly, applying Lemma \ref{ku7lem1} to $\bY'=F_\mKur^\muKur(\bY)$ gives a partition of unity $\{\ze_{j\ti\jmath}:\ti\jmath\in J\}$ on $V_j$ subordinate to $\{V_{j\ti\jmath}:\ti\jmath\in J\}$ for each $j\in J$, such that for all $j,j',\ti\jmath\in J$ we have
\begin{equation*}
\ze_{j\ti\jmath}\vert_{V_{jj'}}=\ze_{j'\ti\jmath}\ci\up_{jj'}+O(s_j)\quad\text{on $V_{jj'}\subseteq V_j$.}
\end{equation*} 

For all $i\in I$ and $j\in J$, define $\hat\mu_{ij}:D_i\vert_{\check U_{ij}}\longra f_{ij}^*(TV_j)\vert_{\check U_{ij}}$ by
\e
\hat\mu_{ij}=\ts\sum_{\ti\imath\in I}\sum_{\ti\jmath\in J} \eta_{i\ti\imath}\vert_{\check U_{ij}}\cdot f_{ij}^*(\ze_{j\ti\jmath})\vert_{\check U_{ij}}\cdot \hat\la_{i\ti\imath}^{\ti\jmath j}.
\label{ku7eq55}
\e
Here $\eta_{i\ti\imath}\vert_{\check U_{ij}}\cdot f_{ij}^*(\ze_{j\ti\jmath})\vert_{\check U_{ij}}\cdot \hat\la_{i\ti\imath}^{\ti\jmath j}$ extends smoothly by zero from the domain $\dot U_{i\ti\imath}^{\ti\jmath j}=\check U_{ij}\cap U_{i\ti\imath}\cap f_{ij}^{-1}(V_{j\ti\jmath})$ of $\hat\la_{i\ti\imath}^{\ti\jmath j}$ to $\check U_{ij}$ as $\supp \eta_{i\ti\imath}\subseteq U_{i\ti\imath}$ and $\supp f_{ij}^*(\ze_{j\ti\jmath})\subseteq f_{ij}^{-1}(V_{j\ti\jmath})$, so each term in \eq{ku7eq55} makes sense, and \eq{ku7eq55} is a locally finite sum as $\sum_{\ti\imath\in I}\eta_{i\ti\imath}$, $\sum_{\ti\jmath\in J}\ze_{j\ti\jmath}$ are. Thus $\hat\mu_{ij}$ is well-defined.

We now claim that
\e
\bs\mu_{ij}:=[\check U_{ij},\hat\mu_{ij}]:\bs f_{ij}\Longra\bs g_{ij}
\label{ku7eq56}
\e
is a 2-morphism over $(S,f)$ in the sense of \S\ref{ku47} for $S=\Im\chi_i\cap f^{-1}(\Im\psi_j)$. To see this, note that $\hat\la_{i\ti\imath}^{\ti\jmath j}$ satisfies the analogue of \eq{ku7eq50} over $\dot U_{i\ti\imath}^{\ti\jmath j}$ with $\hat\la_{i\ti\imath}^{\ti\jmath j}$ in place of $\hat\la_{ij}$, and as for $\La$ in \eq{ku6eq2}--\eq{ku6eq3}, multiplying the analogue of \eq{ku7eq50} for $\hat\la_{i\ti\imath}^{\ti\jmath j}$ by $\eta_{i\ti\imath}\cdot f_{ij}^*(\ze_{j\ti\jmath})$ and summing over $\ti\imath\in I$ and $\ti\jmath\in J$ gives the analogue of \eq{ku7eq50} for $\hat\mu_{ij}$, which is equivalent in this case to \eq{ku4eq47} for~$\bs\mu_{ij}$.

We will show that $\bs\mu=\bigl(\bs\mu_{ij,\; i\in I,\; j\in J}\bigr):\bs f\Ra\bs g$ is a 2-morphism in $\mKur$. We must verify the analogues of Definition \ref{ku4def14}(a),(b). Part (a) says that for all $i,i'\in I$ and $j\in J$ the following diagram of 2-morphisms commutes:
\e
\begin{gathered}
\xymatrix@C=100pt@R=15pt{
*+[r]{\bs f_{i'j}\ci\Tau_{ii'}} \ar@{=>}[r]_{\bs F_{ii'}^j} \ar@{=>}[d]^{\bs \mu_{i'j}*\id} & *+[l]{\bs f_{ij}} \ar@{=>}[d]_{\bs\mu_{ij}} \\
*+[r]{\bs g_{i'j}\ci\Tau_{ii'}} \ar@{=>}[r]^{\bs G_{ii'}^j} & *+[l]{\bs g_{ij}.\!} }
\end{gathered}
\label{ku7eq57}
\e
To prove this, for $\ti\imath\in I$ and $\ti\jmath\in J$ consider the diagram
\e
\begin{gathered}
\xymatrix@!0@C=72pt@R=26pt{
*+[r]{\bs f_{i'j}\ci\Tau_{ii'}} \ar@{=>}[rrrr]_{\bs F_{ii'}^j} \ar@{=>}[ddddd]^{\bs \la_{i'\ti\imath}^{\ti\jmath j}*\id} &&&& *+[l]{\bs f_{ij}} \ar@{=>}[ddddd]_{\bs\la_{i\ti\imath}^{\ti\jmath j}} 
\\
& \bs f_{\ti\imath j}\!\ci\!\Tau_{i'\ti\imath}\!\ci\!\Tau_{ii'} \ar@{=>}[ul]_(0.3){\bs F_{i'\ti\imath}^j*\id} \ar@{=>}[rr]^{\id*\Ka_{ii'\ti\imath}} && \bs f_{\ti\imath j}\!\ci\!\Tau_{i\ti\imath} \ar@{=>}[ur]^(0.3){\bs F_{i\ti\imath}^j} 
\\
& \Up_{\ti\jmath j}\!\ci\!\bs f_{\ti\imath\ti\jmath}\!\ci\!\Tau_{i'\ti\imath}\!\ci\!\Tau_{ii'} \ar@{=>}[u]_{\bs F_{\ti\imath}^{\ti\jmath j}*\id} \ar@{=>}[d]^{\id*\bs\la_{\ti\imath\ti\jmath}*\id} \ar@{=>}[rr]^{\id*\Ka_{ii'\ti\imath}} && \Up_{\ti\jmath j}\!\ci\!\bs f_{\ti\imath\ti\jmath}\!\ci\!\Tau_{i\ti\imath} \ar@{=>}[d]_{\id*\bs\la_{\ti\imath\ti\jmath}*\id} \ar@{=>}[u]^{\bs F_{\ti\imath}^{\ti\jmath j}*\id} \\
& \Up_{\ti\jmath j}\!\ci\!\bs g_{\ti\imath\ti\jmath}\!\ci\!\Tau_{i'\ti\imath}\!\ci\!\Tau_{ii'} \ar@{=>}[d]^{\bs G_{\ti\imath}^{\ti\jmath j}*\id} \ar@{=>}[rr]^{\id*\Ka_{ii'\ti\imath}} && \Up_{\ti\jmath j}\!\ci\!\bs g_{\ti\imath\ti\jmath}\!\ci\!\Tau_{i\ti\imath} \ar@{=>}[d]_{\bs G_{\ti\imath}^{\ti\jmath j}*\id}
\\
& \bs g_{\ti\imath j}\!\ci\!\Tau_{i'\ti\imath}\!\ci\!\Tau_{ii'} \ar@{=>}[dl]^(0.3){\bs G_{i'\ti\imath}^j*\id} \ar@{=>}[rr]^{\id*\Ka_{ii'\ti\imath}} && \bs g_{\ti\imath j}\!\ci\!\Tau_{i\ti\imath} \ar@{=>}[dr]_(0.3){\bs G_{i\ti\imath}^j} 
\\*+[r]{\bs g_{i'j}\ci\Tau_{ii'}} \ar@{=>}[rrrr]^{\bs G_{ii'}^j} &&&& *+[l]{\bs g_{ij}.\!} }\!\!\!{}
\end{gathered}
\label{ku7eq58}
\e
Here the hexagons commute by the definition \eq{ku7eq52} of $\bs\la_{i\ti\imath}^{\ti\jmath j}$, the top and bottom quadrilaterals by Definition \ref{ku4def13}(f) for $\bs f,\bs g$, and the central rectangles by compatibility of horizontal and vertical composition. Thus \eq{ku7eq58} commutes.

As in \eq{ku7eq51} and \eq{ku7eq53}, write
\begin{equation*}
\bs F_{ii'}^j=[\dot U_{ii'}^j,\hat F_{ii'}^j],
\end{equation*}
where $\dot U_{ii'}^j$ is an open neighbourhood of $\chi_i^{-1}(\Im\chi_i\cap\Im\chi_{i'}\cap f^{-1}(\Im\psi_j))$ in $U_{ii'}\cap \tau_{ii'}^{-1}(U_{i'j})\cap U_{ij}\subseteq U_i$, and $\hat F_{ii'}^j:E_i\vert_{\dot U_{ii'}^j}\ra f_{i'j}\vert_{\dot U_{ii'}^j}^*(TV_j)$ is a morphism. Write $\bs G_{ii'}^j$ in the same way, using $\ddot U_{ii'}^j$ and~$\hat G_{ii'}^j:E_i\vert_{\ddot U_{ii'}^j}\ra g_{i'j}\vert_{\ddot U_{ii'}^j}^*(TV_j)$. 

Then by \eq{ku4eq48}, the outer rectangle of \eq{ku7eq58} commuting is equivalent to
\e
\hat\la_{i\ti\imath}^{\ti\jmath j}+\hat F_{ii'}^j=\hat G_{ii'}^j+\tau_{ii'}^*(\hat\la_{i'\ti\imath}^{\ti\jmath j})\ci\hat\tau_{ii'}+O(r_i)
\label{ku7eq59}
\e
on $\dot U_{i\ti\imath}^{\ti\jmath j}\cap \dot U_{ii'}^j\cap \ddot U_{ii'}^j\cap\tau_{ii'}^{-1}(\dot U_{i'\ti\imath}^{\ti\jmath j})$. Multiply \eq{ku7eq59} by $\eta_{i\ti\imath}\cdot f_{ij}^*(\ze_{j\ti\jmath})$ and sum over $\ti\imath\in I$, $\ti\jmath\in J$. Using \eq{ku7eq54} and \eq{ku7eq55} for $i,j$ and $i',j$ gives 
\e
\hat\mu_{ij}+\hat F_{ii'}^j=\hat G_{ii'}^j+\tau_{ii'}^*(\hat\mu_{i'j})\ci\hat\tau_{ii'}+O(r_i)
\label{ku7eq60}
\e
on $\check U_{ij}\cap\dot U_{ii'}^j\cap \ddot U_{ii'}^j\cap\tau_{ii'}^{-1}(\check U_{i'j})$. But by \eq{ku4eq48}, this is equivalent to \eq{ku7eq57} commuting. This proves Definition \ref{ku4def14}(a) for $\bs\mu$, and part (b) is similar. Hence $\bs\mu:\bs f\Ra\bs g$ is a 2-morphism, so $[\bs f]=[\bs g]$ as morphisms in $\Ho(\mKur)$, and $F_\mKur^\muKur$ is faithful, as we have to prove.

\subsubsection*{$F_\mKur^\muKur$ is full} 

Let $\bX,\bY$ be objects in $\mKur$, and write $\bX'=F_\mKur^\muKur(\bX)$, $\bY'=F_\mKur^\muKur(\bY)$. Suppose $\bs f':\bX'\ra\bY'$ is a morphism in $\muKur$. We must show that there exists a 1-morphism $\bs f:\bX\ra\bY$ in $\mKur$ with $F_\mKur^\muKur([\bs f])=\bs f'$.

Use notation \eq{ku7eq48}--\eq{ku7eq49} for $\bX,\bY$, as in \S\ref{ku23} write $\bs f'=\bigl(f,\bs f'_{ij,\;i\in I,\; j\in J}\bigr)$, and let $(U_{ij},f_{ij},\hat f_{ij})$ represent $\bs f_{ij}'$ for all $i\in I$ and $j\in J$, so we have a morphism of $\mu$-Kuranishi neighbourhoods over $f$, as in \S\ref{ku21}:
\begin{equation*}
\bs f_{ij}'=[U_{ij},f_{ij},\hat f_{ij}]:(U_i,D_i,r_i,\chi_i)\longra (V_j,E_j,s_j,\psi_j).
\end{equation*}
Define a 1-morphism of m-Kuranishi neighbourhoods over $f$, as in \S\ref{ku47}:
\e
\bs f_{ij}=(U_{ij},f_{ij},\hat f_{ij}):(U_i,D_i,r_i,\chi_i)\longra (V_j,E_j,s_j,\psi_j).
\label{ku7eq61}
\e

Definition \ref{ku2def13}(a),(b) for $\bs f'$ imply that $\bs f_{\ti\imath j}'\ci\Tau'_{i\ti\imath}=\bs f_{ij}'$ and $\Up'_{j\ti\jmath}\ci\bs f'_{ij}=\bs f'_{i\ti\jmath}$. Therefore $\Up'_{\ti\jmath j}\ci\bs f_{\ti\imath\ti\jmath}'\ci\Tau'_{i\ti\imath}=\bs f_{ij}'$. The argument constructing $\bs\la_{ij}$ in \eq{ku7eq51} now shows that there exist 2-morphisms of m-Kuranishi neighbourhoods over $f$ 
\begin{equation*}
\bs\la_{i\ti\imath}^{\ti\jmath j}:\Up_{\ti\jmath j}\ci\bs f_{\ti\imath\ti\jmath}\ci\Tau_{i\ti\imath}\Longra \bs f_{ij}
\end{equation*}
for all $i,\ti\imath\in I$ and $j,\ti\jmath\in J$. For $i,i',\ti\imath\in I$ and $j,j',\ti\jmath\in J$, define 2-morphisms $\bs F_{ii'(\ti\imath)}^{j(\ti\jmath)}:\bs f_{i'j}\ci\Tau_{ii'}\Ra\bs f_{ij}$ over $(S,f)$ for $S=\Im\chi_i\cap\ab\Im\chi_{i'}\cap\ab\Im\chi_{\ti\imath}\cap\ab f^{-1}(\Im\psi_j\cap\ab\Im\psi_{\ti\jmath})$ and $\bs F_{i(\ti\imath)}^{jj'(\ti\jmath)}:\Up_{jj'}\ci\bs f_{ij}\Ra\bs f_{ij'}$ over $(S,f)$ for $S=\Im\chi_i\cap\ab\Im\chi_{\ti\imath}\cap\ab f^{-1}(\Im\psi_j\cap\ab\Im\psi_{j'}\cap\ab\Im\psi_{\ti\jmath})$ by the commutative diagrams
\ea
\begin{gathered}
\xymatrix@C=190pt@R=17pt{
*+[r]{\bs f_{i'j}\ci\Tau_{ii'}} \ar@{=>}[r]_(0.68){\bs F_{ii'(\ti\imath)}^{j(\ti\jmath)}} \ar@{=>}[d]^{(\bs\la_{i'\ti\imath}^{\ti\jmath j})^{-1}*\id_{\Tau_{ii'}}} & *+[l]{\bs f_{ij}} \\
*+[r]{\Up_{\ti\jmath j}\ci\bs f_{\ti\imath\ti\jmath}\ci\Tau_{i'\ti\imath}\ci\Tau_{ii'}} 
\ar@{=>}[r]^{\id_{\Up_{\ti\jmath j}\ci\bs f_{\ti\imath\ti\jmath}}*\Ka_{ii'\ti\imath}} 
& *+[l]{\Up_{\ti\jmath j}\ci\bs f_{\ti\imath\ti\jmath}\ci\Tau_{i\ti\imath},} \ar@{=>}[u]^{\bs\la_{i\ti\imath}^{\ti\jmath j}} }
\end{gathered}
\label{ku7eq62}\\
\begin{gathered}
\xymatrix@C=190pt@R=15pt{
*+[r]{\Up_{jj'}\ci\bs f_{ij}} \ar@{=>}[r]_(0.68){\bs F_{i(\ti\imath)}^{jj'(\ti\jmath)}} \ar@{=>}[d]^{\id_{\Up_{jj'}}*(\bs\la_{i\ti\imath}^{\ti\jmath j'})^{-1}} & *+[l]{\bs f_{ij'}} \\
*+[r]{\Up_{jj'}\ci\Up_{\ti\jmath j}\ci\bs f_{\ti\imath\ti\jmath}\ci\Tau_{i\ti\imath}} 
\ar@{=>}[r]^{\La_{\ti\jmath jj'}*\id_{\bs f_{\ti\imath\ti\jmath}\ci\Tau_{i\ti\imath}}} 
& *+[l]{\Up_{\ti\jmath j'}\ci\bs f_{\ti\imath\ti\jmath}\ci\Tau_{i\ti\imath}.} \ar@{=>}[u]^{\bs\la_{i\ti\imath}^{\ti\jmath j'}} }
\end{gathered}
\nonumber
\ea

As in \eq{ku7eq51} and \eq{ku7eq53}, write
\begin{equation*}
\bs F_{ii'(\ti\imath)}^{j(\ti\jmath)}=[\dot U_{ii'(\ti\imath)}^{j(\ti\jmath)},\hat F_{ii'(\ti\imath)}^{j(\ti\jmath)}],
\end{equation*}
where $\dot U_{ii'(\ti\imath)}^{j(\ti\jmath)}=U_{ij}\cap U_{ii'}\cap \tau_{ii'}^{-1}(U_{i'j})\cap U_{i\ti\imath} \cap f_{ij}^{-1}(V_{j\ti\jmath})\subseteq U_i$, and $\hat F_{ii'(\ti\imath)}^{j(\ti\jmath)}:D_i\vert_{\dot U_{ii'(\ti\imath)}^{j(\ti\jmath)}}\ra (f_{i'j}\ci\tau_{ii'})^*(TV_j)\vert_{\dot U_{ii'(\ti\imath)}^{j(\ti\jmath)}}$ is a morphism, and
\begin{equation*}
\bs F_{i(\ti\imath)}^{jj'(\ti\jmath)}=[\dot U_{i(\ti\imath)}^{jj'(\ti\jmath)},\hat F_{i(\ti\imath)}^{jj'(\ti\jmath)}],
\end{equation*}
where $\dot U_{i(\ti\imath)}^{jj'(\ti\jmath)}\!=\!U_{ij}\!\cap\! U_{ij'}\!\cap\! f_{ij}^{-1}(V_{jj'})\!\cap\!\ab U_{i\ti\imath}\ab\!\cap\ab\!f_{ij'}^{-1}(V_{j'\ti\jmath})$, and $\hat F_{i(\ti\imath)}^{jj'(\ti\jmath)}:D_i\vert_{\dot U_{ii'(\ti\imath)}^{j(\ti\jmath)}}\!\ra\! (\up_{jj'}\ci f_{ij})^*(TV_{j'})\vert_{\dot U_{ii'(\ti\imath)}^{j(\ti\jmath)}}$.

Apply Lemma \ref{ku7lem1} to $\bX'$, using $(U_{ii'},\tau_{ii'},\hat\tau_{ii'})$ to represent $\Tau'_{ij}$. This gives a partition of unity $\{\eta_{i\ti\imath}:\ti\imath\in I\}$ on $U_i$ subordinate to $\{U_{i\ti\imath}:\ti\imath\in I\}$ for each $i\in I$, such that for all $i,i',\ti\imath\in I$ we have
\e
\eta_{i\ti\imath}\vert_{U_{i\ti\imath}}=\eta_{i'\ti\imath}\ci\tau_{ii'}+O(r_i)\quad\text{on $U_{ii'}\subseteq U_i$.}
\label{ku7eq63}
\e 
Similarly, applying Lemma \ref{ku7lem1} to $\bY'$ gives a partition of unity $\{\ze_{j\ti\jmath}:\ti\jmath\in J\}$ on $V_j$ subordinate to $\{V_{j\ti\jmath}:\ti\jmath\in J\}$ for $j\in J$, such that for all $j,j',\ti\jmath\in J$ we have
\begin{equation*}
\ze_{j\ti\jmath}\vert_{V_{jj'}}=\ze_{j'\ti\jmath}\ci\up_{jj'}+O(s_j)\quad\text{on $V_{jj'}\subseteq V_j$.}
\end{equation*} 

As in \eq{ku7eq55}, for all $i,i'\in I$ and $j,j'\in J$, define vector bundle morphisms $\hat F_{ii'}^{j}:D_i\vert_{\dot U_{ii'}^{j}}\ra (f_{i'j}\ci\tau_{ii'})^*(TV_j)\vert_{\dot U_{ii'}^{j}}$ over $\dot U_{ii'}^{j}=U_{ij}\cap U_{ii'}\cap \tau_{ii'}^{-1}(U_{i'j})\subseteq U_i$ and $\hat F_{i}^{jj'}:D_i\vert_{\dot U_{i}^{jj'}}\ra (\up_{jj'}\ci f_{ij})^*(TV_{j'})\vert_{\dot U_{i}^{jj'}}$ over $\dot U_i^{jj'}=U_{ij}\cap U_{ij'}\cap f_{ij}^{-1}(V_{jj'})\subseteq U_i$ by
\ea
\hat F_{ii'}^{j}&=\ts\sum_{\ti\imath\in I}\sum_{\ti\jmath\in J} \eta_{i\ti\imath}\vert_{\dot U_{ii'}^{j}}\cdot f_{ij}^*(\ze_{j\ti\jmath})\vert_{\dot U_{ii'}^{j}}\cdot \hat F_{ii'(\ti\imath)}^{j(\ti\jmath)},
\label{ku7eq64}\\
\hat F_{i}^{jj'}&=\ts\sum_{\ti\imath\in I}\sum_{\ti\jmath\in J} \eta_{i\ti\imath}\vert_{\dot U_{i}^{jj'}}\cdot f_{ij'}^*(\ze_{j'\ti\jmath})\vert_{\dot U_{i}^{jj'}}\cdot \hat F_{i(\ti\imath)}^{jj'(\ti\jmath)}.
\nonumber
\ea

Then as for $\bs\mu_{ij}$ in \eq{ku7eq56}, 
\e
\bs F_{ii'}^{j}=[\dot U_{ii'}^{j},\hat F_{ii'}^{j}]:\bs f_{i'j}\ci\Tau_{ii'}\Longra \bs f_{ij}
\label{ku7eq65}
\e
is a 2-morphism over $(S,f)$ for $S=\Im\chi_i\cap\Im\chi_{i'}\cap f^{-1}(\Im\psi_j)$, and 
\begin{equation*}
\bs F_{i}^{jj'}=[\dot U_{i}^{jj'},\hat F_{i}^{jj'}]:\Up_{jj'}\ci\bs f_{ij}\Longra\bs f_{ij'}
\end{equation*}
is a 2-morphism over $(S,f)$ for $S=\Im\chi_i\cap f^{-1}(\Im\psi_j\cap \Im\psi_{j'})$.

We now claim that $\bs f=\bigl(f,\bs f_{ij,\;i\in I,\; j\in J}$, $\bs F_{ii',\;i,i'\in I}^{j,\; j\in J}$, $\bs F_{i,\;i\in I}^{jj',\; j,j'\in J}\bigr)$ is a 1-morphism $\bs f:\bX\ra\bY$ in $\mKur$. We must verify Definition \ref{ku4def13}(a)--(h). Parts (a)--(d) are immediate. For (e), if $i=i'$ then $\bs F_{ii(\ti\imath)}^{j(\ti\jmath)}$ in \eq{ku7eq62} is $\id_{\bs f_{ij}}$, so $\hat F_{ii(\ti\imath)}^{j(\ti\jmath)}=O(r_i)$, giving $\hat F_{ii}^{j}=O(r_i)$ in \eq{ku7eq64}, and $\bs F_{ii}^j=\id_{\bs f_{ij}}$ in \eq{ku7eq65}. Similarly $\bs F_i^{jj}=\id_{\bs f_{ij}}$, proving Definition~\ref{ku4def13}(e).

Part (f) says the following diagram of 2-morphisms commutes:
\e
\begin{gathered}
\xymatrix@C=160pt@R=15pt{ *+[r]{\bs f_{i''j}\ci\Tau_{i'i''}\ci\Tau_{ii'}} \ar@{=>}[r]_{\bs F_{i'i''}^j*\id_{\Tau_{ii'}}} \ar@{=>}[d]^{\id_{\bs f_{i''j}}*\Ka_{ii'i''}} & *+[l]{\bs f_{i'j}\ci\Tau_{ii'}} \ar@{=>}[d]_{\bs F_{ii'}^j} \\
 *+[r]{\bs f_{i''j}\ci\Tau_{ii''}} \ar@{=>}[r]^{\bs F_{ii''}^j} & *+[l]{\bs f_{ij}.\!} }
\end{gathered}
\label{ku7eq66}
\e
To prove this, let $\ti\imath\in I$, $\ti\jmath\in J$ and consider the diagram
\e
\begin{gathered}
\xymatrix@!0@C=73pt@R=35pt{ *+[r]{\bs f_{i''j}\ci\Tau_{i'i''}\ci\Tau_{ii'}} \ar@{=>}[rrrr]_{\bs F_{i'i''(\ti\imath)}^{j(\ti\jmath)}*\id} \ar@{=>}[ddd]^{\id*\Ka_{ii'i''}} &&&& *+[l]{\bs f_{i'j}\ci\Tau_{ii'}} \ar@{=>}[ddd]_{\bs F_{ii'(\ti\imath)}^{j(\ti\jmath)}} \\
& \Up_{\ti\jmath j}\!\ci\!\bs f_{\ti\imath\ti\jmath}\!\ci\!\Tau_{i''\ti\imath}\!\ci\!\Tau_{i'i''}\!\ci\!\Tau_{ii'} \ar@{=>}[ul]_(0.35){\bs\la_{i''\ti\imath}^{\ti\jmath j}*\id} \ar@{=>}[rr]_(0.55){\raisebox{-9pt}{$\st\id*\Ka_{i'i''\ti\imath}*\id$}} \ar@{=>}[d]^{\id*\Ka_{ii'i''}} &&
\Up_{\ti\jmath j}\!\ci\!\bs f_{\ti\imath\ti\jmath}\!\ci\!\Tau_{i'\ti\imath}\!\ci\!\Tau_{ii'} \ar@{=>}[ur]^(0.35){\bs\la_{i'\ti\imath}^{\ti\jmath j}*\id} \ar@{=>}[d]_{\id*\Ka_{ii'\ti\imath}} \\
& \Up_{\ti\jmath j}\ci\bs f_{\ti\imath\ti\jmath}\ci\Tau_{i''\ti\imath}\ci\Tau_{ii''} \ar@{=>}[dl]^(0.35){\bs\la_{i''\ti\imath}^{\ti\jmath j}*\id} \ar@{=>}[rr]^{\id*\Ka_{ii''\ti\imath}} && \Up_{\ti\jmath j}\ci\bs f_{\ti\imath\ti\jmath}\ci\Tau_{i\ti\imath} \ar@{=>}[dr]_(0.35){\bs\la_{i\ti\imath}^{\ti\jmath j}} \\
 *+[r]{\bs f_{i''j}\ci\Tau_{ii''}} \ar@{=>}[rrrr]^{\bs F_{ii''(\ti\imath)}^{j(\ti\jmath)}} &&&& *+[l]{\bs f_{ij}.\!} }\!\!\!{}
\end{gathered}
\label{ku7eq67}
\e
Here the top, bottom and right quadrilaterals commute by \eq{ku7eq62}, the central rectangle by Definition \ref{ku4def12}(h) for $\bX$, and the left quadrilateral by compatibility of horizontal and vertical composition. Thus \eq{ku7eq67} commutes.

By \eq{ku4eq48}, the outer rectangle of \eq{ku7eq67} commuting is equivalent to
\e
\hat F_{ii''(\ti\imath)}^{j(\ti\jmath)}+\tau_{ii''}^*(\d f_{i''j})\ci\hat\ka_{ii'i''}=\hat F_{ii'(\ti\imath)}^{j(\ti\jmath)}+\tau_{ii'}^*(\hat F_{i'i''(\ti\imath)}^{j(\ti\jmath)}) \ci \hat\tau_{ii'}+O(r_i)
\label{ku7eq68}
\e
on $U_{ij}\cap U_{ii'}\cap U_{ii''}\cap \tau_{ii'}^{-1}(U_{i'j}\cap U_{i'i''})\cap \tau_{ii''}^{-1}(U_{i''j})\ab\cap\ab U_{i\ti\imath}\ab\cap \ab f_{ij}^{-1}(V_{j\ti\jmath})
$. Multiply \eq{ku7eq68} by $\eta_{i\ti\imath}\cdot f_{ij}^*(\ze_{j\ti\jmath})$ and sum over all $\ti\imath\in I$ and $\ti\jmath\in J$. Using \eq{ku7eq63} and \eq{ku7eq64} for $i,i',j$ and $i,i'',j$ and $i',i'',j$, we deduce that 
\begin{equation*}
\hat F_{ii''}^j+\tau_{ii''}^*(\d f_{i''j})\ci\hat\ka_{ii'i''}=\hat F_{ii'}^j+\tau_{ii'}^*(\hat F_{i'i''}^j) \ci \hat\tau_{ii'}+O(r_i)
\end{equation*}
on $U_{ij}\cap U_{ii'}\cap U_{ii''}\cap\tau_{ii'}^{-1}(U_{i'j}\cap U_{i'i''})\cap \tau_{ii''}^{-1}(U_{i''j})$. But by \eq{ku4eq48}, this is equivalent to \eq{ku7eq66} commuting. This proves Definition \ref{ku4def13}(f) for $\bs f$. Parts (g),(h) are similar. Hence $\bs f:\bX\ra\bY$ is a 1-morphism in $\mKur$. By construction $F_\mKur^\muKur([\bs f])=\bs f'$, so $F_\mKur^\muKur$ is full, as we have to prove.

\subsubsection*{$F_\mKur^\muKur$ is surjective on isomorphism classes} 

Let $\bX'=(X,\cK')$ be a $\mu$-Kuranishi space, with $\cK'=\bigl(I,(V_i,E_i,s_i,\psi_i)_{i\in I}$, $\Phi'_{ij,\;i,j\in I}\bigr)$. To show $F_\mKur^\muKur$ is surjective on isomorphism classes, we must construct an object $\bX$ in $\mKur$ with $F_\mKur^\muKur(\bX)\cong\bX'$ in $\muKur$. Actually we will arrange that~$F_\mKur^\muKur(\bX)=\bX'$.

Then $(V_i,E_i,s_i,\psi_i)$ is an m-Kuranishi neighbourhood on $X$ for $i\in I$. Let $(V_{ij},\phi_{ij},\hat\phi_{ij})$ represent $\Phi'_{ij}$ for $i,j\in I$, where as $\Phi_{ii}'=\id_{(V_i,E_i,s_i,\psi_i)}$ we take $(V_{ii},\phi_{ii},\hat\phi_{ii})=(V_i,\id_{V_i},\id_{E_i})$. As in \eq{ku7eq61}, for all $i,j\in I$ we have a 1-morphism of m-Kuranishi neighbourhoods over $S=\Im\psi_i\cap\Im\psi_j$
\e
\Phi_{ij}=(V_{ij},\phi_{ij},\hat\phi_{ij}):(V_i,E_i,s_i,\psi_i)\longra (V_j,E_j,s_j,\psi_j),
\label{ku7eq69}
\e
and $\Phi_{ii}=(V_i,\id_{V_i},\id_{E_i})=\id_{(V_i,E_i,s_i,\psi_i)}$.

As $\Phi_{jk}'\ci\Phi_{ij}'=\Phi_{ik}'$ for $i,j,k\in I$ by Definition \ref{ku2def11}(f) for $\bX'$, the argument defining $\bs\la_{ij}$ in \eq{ku7eq51} implies that there exists a 2-morphism
\begin{equation*}
\Ka_{ijk}:\Phi_{jk}\ci\Phi_{ij}\Longra\Phi_{ik}
\end{equation*}
over $S=\Im\psi_i\cap\Im\psi_j\cap\Im\psi_k$, where we choose $\Ka_{iij}=\Ka_{ijj}=\id_{\Phi_{ij}}$ for $i,j\in I$. Therefore $\Ka_{iji}:\Phi_{ji}\ci\Phi_{ij}\Ra\id_{(V_i,E_i,s_i,\psi_i)}$, $\Ka_{jij}:\Phi_{ij}\ci\Phi_{ji}\Ra\id_{(V_j,E_j,s_j,\psi_j)}$ imply that $\Phi_{ij}$ is an equivalence in $\Kur_S(X)$ for $S=\Im\psi_i\cap\Im\psi_j$, and so a coordinate change over $S$, for all~$i,j\in I$.

The next elementary lemma about 2-categories is easy to prove.

\begin{lem} Suppose $f:X\ra Y$ and\/ $g,h:Y\ra Z$ are $1$-morphisms in a (strict or weak)\/ $2$-category $\bs\cC,$ with $f$ an equivalence. Then the map $\eta\mapsto\eta*\id_f=\ze$ induces a $1$-$1$ correspondence between $2$-morphisms $\eta:g\Ra h$ and\/ $2$-morphisms $\ze:g\ci f\Ra h\ci f$ in\/~$\bs\cC$.
\label{ku7lem2}
\end{lem}

Let $\ti\imath,i,j,k\in I$. Then Lemma \ref{ku7lem2} in the 2-category $\Kur_S(X)$ and $\Phi_{\ti\imath i}$ an equivalence implies that there is a unique 2-morphism
\begin{equation*}
\Ka_{ijk}^{(\ti\imath)}:\Phi_{jk}\ci\Phi_{ij}\Longra\Phi_{ik}
\end{equation*}
over $S=\Im\psi_{\ti\imath}\cap\Im\psi_i\cap\Im\psi_j\cap\Im\psi_k$ making the following diagram commute:
\e
\begin{gathered}
\xymatrix@C=150pt@R=15pt{
*+[r]{\Phi_{jk}\ci\Phi_{ij}\ci\Phi_{\ti\imath i}} \ar@{=>}[r]_{\Ka_{ijk}^{(\ti\imath)}*\id_{\Phi_{\ti\imath i}}} \ar@{=>}[d]^{\id_{\Phi_{jk}}*\Ka_{\ti\imath ij}} & *+[l]{\Phi_{ik}\ci\Phi_{\ti\imath i}} \\
 *+[r]{\Phi_{jk}\ci\Phi_{\ti\imath j}} \ar@{=>}[r]^{\Ka_{\ti\imath jk}} & *+[l]{\Phi_{\ti\imath k}.\!} \ar@{=>}[u]^{\Ka_{\ti\imath ik}^{-1}} }
\end{gathered}
\label{ku7eq70}
\e
As in \eq{ku7eq51} and \eq{ku7eq53}, we may write $\Ka_{ijk}^{(\ti\imath)}=[\dot V_{ijk}^{(\ti\imath)},\hat\ka_{ijk}^{(\ti\imath)}]$, where $\dot V_{ijk}^{(\ti\imath)}=V_{i\ti\imath}\cap V_{ij}\cap V_{ik}\cap \phi_{ij}^{-1}(V_{jk})$ and $\hat\ka_{ijk}^{(\ti\imath)}:D_i\vert_{\dot V_{ijk}^{(\ti\imath)}}\ra (\phi_{jk}\ci\phi_{ij})\vert_{\dot V_{ijk}^{(\ti\imath)}}^{-1}(TV_k)$ is a morphism. 

Apply Lemma \ref{ku7lem1} to $\bX'$, using $(V_{ij},\phi_{ij},\hat\phi_{ij})$ to represent $\Phi'_{ij}$. This gives a partition of unity $\{\eta_{i\ti\imath}:\ti\imath\in I\}$ on $V_i$ subordinate to $\{V_{i\ti\imath}:\ti\imath\in I\}$ for each $i\in I$, such that for all $\ti\imath,i,j\in I$ we have
\e
\eta_{i\ti\imath}\vert_{V_{ij}}=\eta_{j\ti\imath}\ci\phi_{ij}+O(s_i)\quad\text{on $V_{ij}\subseteq V_i$.}
\label{ku7eq71}
\e 

As in \eq{ku7eq55}, for all $i,j,k\in I$ define $\dot V_{ijk}=V_{ij}\cap V_{ik}\cap \phi_{ij}^{-1}(V_{jk})\subseteq V_i$ and a morphism $\hat\la_{ijk}:D_i\vert_{\dot V_{ijk}}\ra (\phi_{jk}\ci\phi_{ij})\vert_{\dot V_{ijk}}^{-1}(TV_k)$ by
\e
\hat\la_{ijk}=\ts\sum_{\ti\imath\in I} \eta_{i\ti\imath}\vert_{\dot V_{ijk}}\cdot\hat\ka_{ijk}^{(\ti\imath)}.
\label{ku7eq72}
\e
Then the argument for $\bs\mu_{ij}$ in \eq{ku7eq56} shows that
\begin{equation*}
\La_{ijk}:=[\dot V_{ijk},\hat\la_{ijk}]:\Phi_{jk}\ci\Phi_{ij}\Longra\Phi_{ik}
\end{equation*}
is a 2-morphism over $S=\Im\psi_i\cap\Im\psi_j\cap\Im\psi_k$.

Now let $\ti\imath,i,j,k,l\in I$. Consider the diagram of 2-morphisms \e
\begin{gathered}
\xymatrix@!0@C=73pt@R=30pt{ *+[r]{\Phi_{kl}\ci\Phi_{jk}\ci\Phi_{ij}\ci\Phi_{\ti\imath i}} \ar@{=>}[dr]^(0.65){\id*\Ka_{\ti\imath ij}} \ar@{=>}[rrrr]_{\Ka_{jkl}^{(\ti\imath)}*\id} \ar@{=>}[ddd]^(0.62){\id*\Ka_{ijk}^{(\ti\imath)}*\id} &&&& *+[l]{\Phi_{jl}\ci\Phi_{ij}\ci\Phi_{\ti\imath i}} \ar@{=>}[ddd]_(0.62){\Ka_{ijl}^{(\ti\imath)}*\id} \ar@{=>}[dl]_(0.65){\id*\Ka_{\ti\imath ij}} \\
& {\Phi_{kl}\ci\Phi_{jk}\ci\Phi_{\ti\imath j}} \ar@{=>}[rr]_{\Ka_{jkl}^{(\ti\imath)}*\id} \ar@{=>}[d]^{\id*\Ka_{\ti\imath jk}} &&
{\Phi_{jl}\ci\Phi_{\ti\imath j}} \ar@{=>}[d]_{\Ka_{\ti\imath jl}} \\
& {\Phi_{kl}\ci\Phi_{\ti\imath k}} \ar@{=>}[rr]^{\Ka_{\ti\imath kl}} && \Phi_{\ti\imath l} \\
*+[r]{\Phi_{kl}\ci\Phi_{ik}\ci\Phi_{\ti\imath i}} \ar@{=>}[ur]_(0.65){\id*\Ka_{\ti\imath ik}} \ar@{=>}[rrrr]^{\Ka_{ikl}^{(\ti\imath)}*\id} &&&& *+[l]{\Phi_{il}\ci\Phi_{\ti\imath i}} \ar@{=>}[ul]^(0.65){\Ka_{\ti\imath il}} }\!\!\!{}
\end{gathered}
\label{ku7eq73}
\e
over $S=\Im\psi_{\ti\imath}\cap\Im\psi_i\cap\Im\psi_j\cap\Im\psi_k\cap\Im\psi_l$. Here the top quadrilateral commutes by compatibility of horizontal and vertical composition, and the other four quadrilaterals commute by \eq{ku7eq70}. Hence \eq{ku7eq73} commutes.

Applying Lemma \ref{ku7lem2} to the outer rectangle of \eq{ku7eq73} and using $\Phi_{\ti\imath i}$ an equivalence shows that over $S=\Im\psi_{\ti\imath}\cap\Im\psi_i\cap\Im\psi_j\cap\Im\psi_k\cap\Im\psi_l$ we have
\e
\Ka^{(\ti\imath)}_{ikl}\od(\id_{\Phi_{kl}}*\Ka^{(\ti\imath)}_{ijk})=\Ka^{(\ti\imath)}_{ijl}\od(\Ka^{(\ti\imath)}_{jkl}*\id_{\Phi_{ij}}):\Phi_{kl}\ci\Phi_{jk}\ci\Phi_{ij}\Longra\Phi_{il}.
\label{ku7eq74}
\e
Write \eq{ku7eq74} as an equation in the form \eq{ku7eq59} involving the $\hat\ka^{(\ti\imath)}_{ijk}$. As for \eq{ku7eq60}, multiplying this equation by $\eta_{i\ti\imath}$, summing over $\ti\imath\in I$, and using \eq{ku7eq71} and \eq{ku7eq72} gives the equation implying that over $S=\Im\psi_i\cap\Im\psi_j\cap\Im\psi_k\cap\Im\psi_l$
\e
\La_{ikl}\od(\id_{\Phi_{kl}}*\La_{ijk})=\La_{ijl}\od(\La_{jkl}*\id_{\Phi_{ij}}):\Phi_{kl}\ci\Phi_{jk}\ci\Phi_{ij}\Longra\Phi_{il}.
\label{ku7eq75}
\e

Now define $\cK=\bigl(I,(V_i,E_i,s_i,\psi_i)_{i\in I}$, $\Phi_{ij,\;i,j\in I}$, $\La_{ijk,\; i,j,k\in I}\bigr)$. Definition \ref{ku4def12}(a)--(f) for $\cK$ are immediate. For (g), as $\Ka_{iij}=\Ka_{ijj}=\id_{\Phi_{ij}}$, equation \eq{ku7eq70} implies that $\Ka_{iij}^{(\ti\imath)}=\Ka_{ijj}^{(\ti\imath)}=\id_{\Phi_{ij}}$, so $\hat\ka_{iij}^{(\ti\imath)}=O(s_i)=\hat\ka_{ijj}^{(\ti\imath)}$. Thus \eq{ku7eq72} gives $\hat\la_{iij}=O(s_i)=\hat\la_{ijj}$, and $\La_{iij}=\La_{ijj}=\id_{\Phi_{ij}}$, as we have to prove. Definition \ref{ku4def12}(h) for $\cK$ is equation \eq{ku7eq75}. Hence $\cK$ is an m-Kuranishi structure on $X$, so that $\bX=(X,\cK)$ is an m-Kuranishi space. By construction $F_\mKur^\muKur(\bX)=\bX'$. Therefore $F_\mKur^\muKur$ is surjective on isomorphism classes. This completes the proof of Theorem~\ref{ku4thm8}(a).

\subsubsection{Theorem \ref{ku4thm8}(b): $F_\mKur^\Kur:\mKur\ra\KurtrG$ is an equivalence}
\label{ku732}

Theorem \ref{ku4thm7} defines a full and faithful 2-functor $F_\mKur^\Kur:\mKur\ra\KurtrG$. By construction, $F_\mKur^\Kur$ is an equivalence from $\mKur$ to the full 2-subcategory $\Kur_{\bf tr\Ga}\subset\KurtrG\subset\Kur$ of Kuranishi spaces $\bX=(X,\cK)$ such that all Kuranishi neighbourhoods $(V_i,E_i,\Ga_i,s_i,\psi_i)$ in $\cK$ have $\Ga_i=\{1\}$. Thus, to show that $F_\mKur^\Kur:\mKur\ra\KurtrG$ is an equivalence, it is enough to prove that the inclusion $\Kur_{\bf tr\Ga}\subset\KurtrG$ is an equivalence. That is, if $\bX$ is an object of $\KurtrG$, we must show that there exists $\bs{\ti X}$ in $\Kur_{\bf tr\Ga}$ with $\bs{\ti X}\simeq\bX$ in~$\KurtrG$.

Write $\bX=(X,\cK)$ with $\cK=\bigl(I,(V_i,E_i,\Ga_i,s_i,\psi_i)_{i\in I}$, $\Phi_{ij,\; i,j\in I}$, $\La_{ijk,\; i,j,k\in I}\bigr)$, and let $\Phi_{ij}=(P_{ij},\pi_{ij},\phi_{ij},\hat\phi_{ij})$ and $(\dot P_{ijk},\la_{ijk},\hat\la_{ijk})$ represent $\La_{ijk}$ for $i,j,k\in I$. 

Let $i\in I$, and define $V_i'=\bigl\{v\in V_i:\ga\cdot v\ne v$ for all $1\ne\ga\in\Ga_i\bigr\}$. Then $V_i'$ is the $\Ga_i$-invariant open subset of $V_i$ on which $\Ga_i$ acts freely. Define $E_i'=E_i\vert_{V_i'}$, with projection $\pi':E_i'\ra V_i'$, and $s_i'=s_i\vert_{V_i'}$. If $v\in s_i^{-1}(0)\subseteq V_i$ then the stabilizer group $\Stab_{\Ga_i}(v)$ of $v$ in $\Ga_i$ is $G_{\bar\psi_i(v)}\bX=\{1\}$ as $\bX\in\KurtrG$. Thus $s_i^{-1}(0)\subseteq V_i'$, and~$(s_i')^{-1}(0)=s_i^{-1}(0)$.

Since $\Ga_i$ is finite and acts freely on $V_i'$ and $E_i'$, the quotients $\ti V_i:=V_i'/\Ga_i$ and $\ti E_i:=E_i'/\Ga_i$ are manifolds, with \'etale projections $\pi_{\smash{\ti V_i}}:V_i'\ra\ti V_i$ and $\pi_{\smash{\ti E_i}}:E_i'\ra\ti E_i$. As $\pi':E_i'\ra V_i'$ and $s_i':V_i'\ra E_i'$ are $\Ga_i$-equivariant, they push down to unique smooth maps $\ti\pi:\ti E_i\ra\ti V_i$ and $\ti s_i:\ti V_i\ra\ti E_i$ with $\ti\pi\ci\pi_{\smash{\ti E_i}}=\pi_{\smash{\ti V_i}}\ci\pi'$ and $\ti s_i\ci\pi_{\smash{\ti V_i}}=\pi_{\smash{\ti E_i}}\ci s_i'$. Then $\ti\pi:\ti E_i\ra\ti V_i$ is a vector bundle, and $\ti s_i:\ti V_i\ra\ti E_i$ a smooth section. Set $\ti\Ga_i=\{1\}$. Then we have
\begin{equation*}
\ti s_i^{-1}(0)/\ti\Ga_i=\ti s_i^{-1}(0)=(s_i')^{-1}(0)/\Ga_i=s_i^{-1}(0)/\Ga_i,\end{equation*}
so $\ti\psi_i:=\psi_i:\ti s_i^{-1}(0)/\ti\Ga_i\ra X$ is a homeomorphism with the open set $\Im\psi_i\subseteq X$. Thus $(\ti V_i,\ti E_i,\ti\Ga_i,\ti s_i,\ti\psi_i)$ is a Kuranishi neighbourhood on $X$, with~$\Im\ti\psi_i=\Im\psi_i$.

Let $i,j\in I$. Define $P_{ij}'=\pi_{ij}^{-1}(V_i')\cap \phi_{ij}^{-1}(V_j')\subseteq P_{ij}$, so that $P_{ij}'\subseteq P_{ij}$ is open and $\Ga_i\t\Ga_j$-invariant. Set $\pi_{ij}':=\pi_{ij}\vert_{P_{ij}'}:P_{ij}'\ra V_i'$, $\phi_{ij}':=\phi_{ij}\vert_{P_{ij}'}:P_{ij}'\ra V_j'$, and $\hat\phi_{ij}':=\hat\phi_{ij}\vert_{P_{ij}'}:(\pi_{ij}')^*(E_i')\ra (\phi_{ij}')^*(E_j')$. Since $\bar\psi_i^{-1}(S)\subseteq V_i'$ and $\bar\psi_j^{-1}(S)\subseteq V_j'$ for $S=\Im\psi_i\cap\Im\psi_j$, the image $V_{ij}':=\pi_{ij}'(P_{ij}')$ is a $\Ga_i$-invariant open neighbourhood of $\bar\psi_i^{-1}(S)$ in~$V_i'\subseteq V_i$. 

Since $\pi_{ij}'\t\phi_{ij}':P_{ij}'\ra V_i'\t V_j'$ is $(\Ga_i\t\Ga_j)$-equivariant and $\Ga_i,\Ga_j$ act freely on $V_i',V_j'$, we see that $\Ga_i\t\Ga_j$ acts freely on $P_{ij}'$. Thus $\ti P_{ij}:=P_{ij}'/(\Ga_i\t\Ga_j)$ is a manifold, with \'etale projection $\pi_{\smash{\ti P_{ij}}}:P_{ij}'\ra\ti P_{ij}$, and $\pi_{ij}',\phi_{ij}',\hat\phi_{ij}'$ descend to unique $\ti\pi_{ij}:\ti P_{ij}\ra\ti V_i$, $\ti\phi_{ij}\ra\ti V_j$, $\hat{\ti\phi}_{ij}:\ti\pi_{ij}^*(\ti E_i)\ra \ti\phi_{ij}^*(\ti E_j)$, as for $\ti\pi,\ti s_i$. Clearly
\begin{equation*}
\ti\Phi_{ij}:(\ti V_i,\ti E_i,\ti\Ga_i,\ti s_i,\ti\psi_i)\longra(\ti V_j,\ti E_j,\ti\Ga_j,\ti s_j,\ti\psi_j)
\end{equation*}
is a 1-morphism over $S=\Im\psi_i\cap\Im\psi_j=\Im\ti\psi_i\cap\Im\ti\psi_j$.

Let $i,j,k\in I$. As $(\dot P_{ijk},\la_{ijk},\hat\la_{ijk})$ represents $\La_{ijk}:\Phi_{jk}\ci\Phi_{ij}\Ra\Phi_{ik}$, we have $\dot P_{ijk}\subseteq (P_{ij}\t_{\phi_{ij},V_j,\pi_{jk}}P_{jk})/\Ga_j$ open, and $\la_{ijk}:\dot P_{ijk}\ra P_{ik}$, and $\hat\la_{ijk}:(\pi_{ij}\ci\pi_{P_{ij}})^*(E_i)\ra (\phi_{jk}\ci\pi_{P_{jk}})^*(TV_k)$. Define
\begin{equation*}
\dot P_{ijk}'=\dot P_{ijk}\cap (\pi_{ij}\ci\pi_{P_{ij}})^{-1}(V_i')\cap (\phi_{ij}\ci\pi_{P_{ij}})^{-1}(V_j')\cap (\phi_{jk}\ci\pi_{P_{jk}})^{-1}(V_k'),
\end{equation*}
and set $\la_{ijk}':=\la_{ijk}\vert_{\smash{\dot P_{ijk}'}}$, $\hat\la_{ijk}':=\hat\la_{ijk}\vert_{\smash{\dot P_{ijk}'}}$. Then $\dot P_{ijk}'\subseteq (P'_{ij}\t_{\phi'_{ij},V'_j,\pi'_{jk}}P'_{jk})/\Ga_j$ is open, and $\la'_{ijk}:\dot P'_{ijk}\ra P'_{ik}$, and~$\hat\la'_{ijk}:(\pi'_{ij}\ci\pi_{P'_{ij}})^*(E'_i)\ra (\phi'_{jk}\ci\pi_{P'_{jk}})^*(TV'_k)$.

Since $\la'_{ijk}:\dot P'_{ijk}\ra P'_{ik}$ is $(\Ga_i\t\Ga_k)$-equivariant and $\Ga_i\t\Ga_k$ acts freely on $P'_{ik}$, we see that $\Ga_i\t\Ga_k$ acts freely on $\dot P'_{ijk}$. Thus $\dot{\ti P}_{ijk}:=\dot P'_{ijk}/(\Ga_i\t\Ga_k)$ is a manifold, with \'etale projection $\pi_{\smash{\dot{\ti P}_{ijk}}}:\dot P'_{ijk}\ra\dot{\ti P}_{ijk}$. We have
\begin{align*}
\dot{\ti P}_{ijk}&\subseteq \bigl[(P'_{ij}\t_{V'_j}P'_{jk})/\Ga_j\bigr]/(\Ga_i\t\Ga_k)\\
&=\bigl(P'_{ij}/(\Ga_i\t\Ga_j)\bigr)\t_{V_j'/\Ga_j}\bigl(P'_{jk}/(\Ga_j\t\Ga_k)\bigr)
=\ti P_{ij}\t_{\ti\phi_{ij},\ti V_j,\ti\pi_{jk}}\ti P_{jk}.
\end{align*}
Also $\la_{ijk}',\hat\la_{ijk}'$ descend to unique $\ti\la_{ijk}:\dot{\ti P}_{ijk}\ra \ti P_{ik}$ and $\hat{\ti\la}_{ijk}:(\ti\pi_{ij}\ci\pi_{\ti P_{ij}})^*(\ti E_i)\ra (\ti\phi_{jk}\ci\pi_{\ti P_{jk}})^*(T\ti V_k)$. It is now easy to check that $\ti\La_{ijk}:=[\dot{\ti P}_{ijk},\ti\la_{ijk},\hat{\ti\la}_{ijk}]$ is a 2-morphism $\ti\Phi_{jk}\ci\ti\Phi_{ij}\Ra\ti\Phi_{ik}$ over $S=\Im\ti\psi_i\cap\Im\ti\psi_j\cap\Im\ti\psi_k$. Therefore $\ti\La_{iji}:\ti\Phi_{ji}\ci\ti\Phi_{ij}\Ra\id_{\smash{(\ti V_i,\ti E_i,\ti\Ga_i,\ti s_i,\ti\psi_i)}}$, $\ti\La_{jij}:\ti\Phi_{ij}\ci\ti\Phi_{ji}\Ra\id_{\smash{(\ti V_j,\ti E_j,\ti\Ga_j,\ti s_j,\ti\psi_j)}}$ imply $\ti\Phi_{ij}$ is an equivalence in $\Kur_S(X)$ for $S=\Im\ti\psi_i\cap\Im\ti\psi_j$, and so a coordinate change over $S$, for all~$i,j\in I$.

Define $\ti\cK=\bigl(I,(\ti V_i,\ti E_i,\ti\Ga_i,\ti s_i,\ti\psi_i)_{i\in I}$, $\ti\Phi_{ij,\;i,j\in I}$, $\ti\La_{ijk,\; i,j,k\in I}\bigr)$. Then Definition \ref{ku4def12}(a)--(d) for $\ti\cK$ are immediate, and (e)--(h) follow from (e)--(h) for $\cK$. Hence $\bs{\ti X}=(X,\ti\cK)$ is a Kuranishi space, and $\bs{\ti X}\in\Kur_{\bf tr\Ga}$ as $\ti\Ga_i=\{1\}$ for all~$i\in I$.

We must show that $\bX$ and $\bs{\ti X}$ are equivalent in $\KurtrG$. Set $f=g=\id_X:X\ra X$. For $i,j\in I$, define a 1-morphism of Kuranishi neighbourhoods
\begin{equation*}
\bs f_{ij}:(V_i,E_i,\Ga_i,s_i,\psi_i)\longra (\ti V_j,\ti E_j,\ti\Ga_j,\ti s_j,\ti\psi_j)
\quad\text{over $(\Im\psi_i\cap\Im\ti\psi_j,f)$},
\end{equation*}
by $\bs f_{ij}=(\ddot P_{ij},\ddot\pi_{ij},f_{ij},\hat f_{ij})$, where $\ddot P_{ij}:=P_{ij}'/\Ga_j$ with \'etale projection $\pi_{\smash{\ddot P_{ij}}}:P_{ij}'\ra\ddot P_{ij}$, and $\pi_{ij}',\phi_{ij}',\hat\phi_{ij}'$ descend to unique $\ddot\pi_{ij},f_{ij},\hat f_{ij}$ through $\pi_{\smash{\ddot P_{ij}}}$. That is, we define $\bs f_{ij}$ from $\Phi_{ij}'$ as for $\ti\Phi_{ij}$, except that we divide by $\Ga_j$ rather than $\Ga_i\t\Ga_j$. Similarly, define a 1-morphism
\begin{equation*}
\bs g_{ji}:(\ti V_j,\ti E_j,\ti\Ga_j,\ti s_j,\ti\psi_j)\longra (V_i,E_i,\Ga_i,s_i,\psi_i)\quad\text{over $(\Im\psi_i\cap\Im\ti\psi_j,g)$},
\end{equation*}
by $\bs g_{ji}=(\dddot P_{ji},\dddot\pi_{ji},g_{ji},\hat g_{ji})$, where $\dddot P_{ji}:=P_{ji}'/\Ga_j$, dividing by $\Ga_j$ rather than $\Ga_j\t\Ga_i$, and $\pi_{ji}',\phi_{ji}',\hat\phi_{ji}'$ descend to $\dddot\pi_{ji},g_{ji},\hat g_{ji}$ through~$\pi_{\smash{\dddot P_{ji}}}:P_{ji}'\ra\dddot P_{ji}$.

For $i,i',j,j'\in I$, we now define 2-morphisms
\begin{align*}
\bs F_{ii'}^j&:\bs f_{i'j}\ci\Phi_{ii'}\Longra \bs f_{ij}, &
\bs F_i^{jj'}&:\ti\Phi_{jj'}\ci\bs f_{ij}\Longra \bs f_{ij'}, \\
\bs G_{jj'}^i&:\bs g_{j'i}\ci\ti\Phi_{jj'}\Longra \bs g_{ji}, &
\bs G_j^{ii'}&:\Phi_{ii'}\ci\bs g_{ji}\Longra \bs g_{ji'},
\end{align*}
where $\bs F_{ii'}^j,\bs F_i^{jj'},\bs G_{jj'}^i,\bs G_j^{ii'}$ are constructed as for $\ti\La_{ii'j},\ti\La_{ijj'},\ti\La_{jj'i},\ti\La_{jii'}$, except that $\dot P'_{ii'j},\dot P'_{ijj'},\dot P'_{jj'i},\dot P'_{jii'}$ are divided by $\Ga_j,\Ga_{j'},\Ga_j,\Ga_j$ rather than $\Ga_i\t\Ga_j,\Ga_i\t\Ga_{j'},\Ga_j\t\Ga_i,\Ga_j\t\Ga_{i'}$, respectively.

Set $\bs f=\bigl(f,\bs f_{ij,\;i\in I,\; j\in I}$, $\bs F_{ii',\;i,i'\in I}^{j,\; j\in I}$, $\bs F_{i,\;i\in I}^{jj',\; j,j'\in I}\bigr)$ and $\bs g=\bigl(g,\bs g_{ji,\;j\in I,\; i\in I}$, $\bs G_{jj',\;j,j'\in I}^{i,\; i\in I}$, $\bs G_{j,\;j\in I}^{ii',\; i,i'\in I}\bigr)$. Then Definition \ref{ku4def13}(a)--(d) for $\bs f,\bs g$ are immediate, and (e)--(h) are easily proved from Definition \ref{ku4def12}(g),(h). Hence $\bs f:\bX\ra\bs{\ti X}$ and $\bs g:\bs{\ti X}\ra\bX$ are 1-morphisms in $\KurtrG$. 

Similarly, for $i,i'\in I$ and $j,j'\in J$ we define 2-morphisms
\begin{equation*}
\bs\de_{iji'}:\bs g_{ji'}\ci\bs f_{ij}\Longra \Phi_{ii'},\quad \bs\ep_{jij'}:\bs f_{ij'}\ci\bs g_{ji}\Longra \ti\Phi_{jj'},
\end{equation*}
constructed as for $\ti\La_{iji'},\ti\La_{jij'}$ but dividing by $\{1\},\Ga_j\t\Ga_{j'}$. Then Definition \ref{ku4def12}(h) for $\bX$ implies that
\e
\begin{split}
\La_{ii'i''}\od(\id_{\Phi_{i'i''}}*\bs\de_{iji'})\od\bs\al_{\Phi_{i'i''},\bs g_{ji'},\bs f_{ij}}&=\bs\de_{iji''}\od(\bs G_j^{i'i''}*\id_{\bs f_{ij}}),\\
\bs\de_{iji''}\od(\id_{\bs g_{ji''}}*\bs F_{ii'}^j)\od\bs\al_{\bs g_{ji''},\bs f_{i'j},\Phi_{ii'}}&=\La_{ii'i''}\od(\bs\de_{i'ji''}*\id_{\Phi_{ii'}}),\\
\bs\de_{ij'i'}\od(\id_{\bs g_{j'i'}}*\bs F_i^{jj'})\od\bs\al_{\bs g_{j'i'},\ti\Phi_{jj'},\bs f_{ij}}&=\bs\de_{iji'}\od(\bs G_{jj'}^i*\id_{\bs f_{ij}}),\\
\ti\La_{jj'j''}\od(\id_{\ti\Phi_{j'j''}}*\bs\ep_{jij'})\od\bs\al_{\ti\Phi_{j'j''},\bs f_{ij'},\bs g_{ji}}&=\bs\ep_{jij''}\od(\bs F_i^{j'j''}*\id_{\bs g_{ji}}),\\
\bs\ep_{jij''}\od(\id_{\bs f_{ij''}}*\bs G_{jj'}^i)\od\bs\al_{\bs f_{ij''},\bs g_{j'i},\ti\Phi_{jj'}}&=\ti\La_{jj'j''}\od(\bs\ep_{j'ij''}*\id_{\ti\Phi_{jj'}}),\\
\bs\ep_{ji'j'}\od(\id_{\bs f_{i'j'}}*\bs G_j^{ii'})\od\bs\al_{\bs f_{i'j'},\Phi_{ii'},\bs g_{ji}}&=\bs\ep_{jij'}\od(\bs F_{ii'}^j*\id_{\bs g_{ji}}).
\end{split}
\label{ku7eq76}
\e

Definition \ref{ku4def15} defines compositions $\bs g\ci\bs f:\bX\ra\bX$ and $\bs f\ci\bs g:\bs{\ti X}\ra\bs{\ti X}$ in $\KurtrG\subset\Kur$ with 2-morphisms for all $i,i',j,j'\in I$
\begin{equation*}
\Th_{iji'}^{\bs g,\bs f}:\bs g_{ji'}\ci\bs f_{ij}\Longra(\bs g\ci\bs f)_{ii'},\quad
\Th_{jij'}^{\bs f,\bs g}:\bs f_{ij'}\ci\bs g_{ji}\Longra(\bs f\ci\bs g)_{jj'}.
\end{equation*}
We now claim that for all $i,i',j,j'\in I$ there are unique 2-morphisms
\begin{equation*}
\bs\eta_{ii'}:(\bs g\ci\bs f)_{ii'}\Longra\Phi_{ii'},\quad
\bs\ze_{jj'}:(\bs f\ci\bs g)_{jj'}\Longra\ti\Phi_{jj'},
\end{equation*}
such that for all $j\in I$ and $i\in I$ we have
\e
\begin{split}
\bs\eta_{ii'}\vert_{\Im\psi_i\cap\Im\psi_{i'}\cap\Im\ti\psi_j}&=\bs\de_{iji'}\od(\Th_{iji'}^{\bs g,\bs f})^{-1},\\
\bs\ze_{jj'}\vert_{\Im\ti\psi_j\cap\Im\ti\psi_{j'}\cap\Im\psi_i}&=\bs\ep_{jij'}\od(\Th_{jij'}^{\bs f,\bs g})^{-1}.
\end{split}
\label{ku7eq77}
\e
To prove this, for $\bs\eta_{ii'}$ we use the stack property Theorem \ref{ku4thm1}(a) and Definition \ref{ku4def11}(iv) for the open cover $\{\Im\psi_i\cap\Im\psi_{i'}\cap\Im\ti\psi_j:j\in I\}$ of the domain $\Im\psi_i\cap\Im\psi_{i'}$ of $\bs\eta_{ii'}$, using \eq{ku4eq18}--\eq{ku4eq20} for the $\Th_{iji'}^{\bs g,\bs f}$ and \eq{ku7eq76} to check that the prescribed values $\bs\eta_{ii'}\vert_{\Im\psi_i\cap\Im\psi_{i'}\cap\Im\ti\psi_j}$ in \eq{ku7eq77} satisfy the required compatibility on double overlaps. The $\bs\ze_{jj'}$ proof is similar.

We can now show that $\bs\eta=(\bs\eta_{ii'}:i,i'\in I):\bs g\ci\bs f\Ra\bs\id_\bX$ and $\bs\ze=(\bs\ze_{jj'}:j,j'\in I):\bs f\ci\bs g\Ra\bs\id_{\smash{\bs{\ti X}}}$ are 2-morphisms of Kuranishi spaces, using \eq{ku4eq18}--\eq{ku4eq20} for the $\Th_{iji'}^{\bs g,\bs f},\Th_{jij'}^{\bs f,\bs g}$, Definition \ref{ku4def12}(h) for $\bX,\bs{\ti X}$, equations \eq{ku7eq76} and \eq{ku7eq77}, and the stack property Theorem \ref{ku4thm1}(a) and Definition \ref{ku4def11}(iii) to verify Definition \ref{ku4def14}(a),(b). Therefore $\bs f,\bs g,\bs\eta,\bs\ze$ show that $\bX$ and $\bs{\ti X}$ are equivalent in $\KurtrG$. This completes the proof of Theorem~\ref{ku4thm8}.

\subsection{Proof of Theorem \ref{ku4thm9}}
\label{ku74}

In this section, as in \S\ref{ku712} we will again by an abuse of notation treat the weak 2-category $\Kur_S(X)$ defined in \S\ref{ku41} as if it were a strict 2-category, omitting 2-morphisms $\bs\al_{\Phi_{kl},\Phi_{jk},\Phi_{ij}},\bs\be_{\Phi_{ij}},\bs\ga_{\Phi_{ij}}$ in \eq{ku4eq4} and \eq{ku4eq6}, and omitting brackets in compositions of 1-morphisms $\Phi_{kl}\ci\Phi_{jk}\ci\Phi_{ij}$. We do this because otherwise diagrams such as \eq{ku7eq80}, \eq{ku7eq86}, \eq{ku7eq88}, \ldots\ would become too big.

Let $\cF=\bigl(A,(V_a,E_a,\Ga_a,s_a,\psi_a)_{a\in A},$ $S_{ab},\Phi_{ab,\;a,b\in A},$ $S_{abc},\ab\La_{abc,\; a,b,c\in A}\bigr)$ be a fair coordinate system of virtual dimension $n\in\Z$ on a Hausdorff, second countable topological space $X$, as in \S\ref{ku48}. Then $\cF$ satisfies either Definition \ref{ku4def36}(k) or (k$)'$. We will suppose $\cF$ satisfies Definition \ref{ku4def36}(k), and give the proof in this case. The proof for (k$)'$ is very similar, but the order of composition of 1-morphisms is reversed, and the order of horizontal composition of 2-morphisms is reversed (though vertical composition stays the same), and the order of subscripts $a,b,c,\ldots$ is reversed, so $\Phi_{ab},\La_{abc}$ are replaced by $\Phi_{ba},\La_{cba}$, and so on. We leave the details for case (k$)'$ to the interested reader.

Throughout the proof, we will use the following notation for multiple intersections of the open sets $S_{ab}$ in $X$. For $a_1,\ldots,a_k\in A$, $k\ge 3$, write
\begin{equation*}
\dot S_{a_1a_2\cdots a_k}=\bigcap\nolimits_{1\le i<j\le k}S_{a_ia_j}.
\end{equation*}
Mote generally, if we enclose a group of consecutive indices $a_la_{l+1}\cdots a_m$ in brackets, as in $\dot S_{a_1\cdots a_{l-1}(a_l\cdots a_m)a_{m+1}\cdots a_k}$, we omit from the intersection any $S_{a_ia_j}$ with both $a_i,a_j$ belonging to the bracketed group. So, for example
\begin{align*}
\dot S_{a(bc)}&=S_{ab}\cap S_{ac}, \qquad
\dot S_{(ab)(cd)}=S_{ac}\cap S_{ad}\cap S_{bc}\cap S_{bd},\\
\dot S_{a(bc)(de)}&=S_{ab}\cap S_{ac}\cap S_{ad}\cap S_{ae}\cap S_{bd}\cap S_{be}\cap S_{cd}\cap S_{ce}.
\end{align*}

In Definition \ref{ku4def36}, the 2-morphisms $\La_{abc}$ are defined on open sets $S_{abc}\subseteq S_{ab}\cap S_{ac}\cap S_{bc}\subseteq\Im\psi_a\cap\Im\psi_b\cap\Im\psi_c$. We begin by showing that we can extend the $\La_{abc}$ canonically to $\dot S_{abc}=S_{ab}\cap S_{ac}\cap S_{bc}$.

\begin{lem} There exist unique\/ $2$-morphisms $\ti\La_{abc}:\Phi_{bc}\ci\Phi_{ab}\Ra\Phi_{ac}$ defined over $\dot S_{abc}$ for all\/ $a,b,c\in A,$ such that\/ $\ti\La_{abc}\vert_{S_{abc}}=\La_{abc},$ and as in Definition\/ {\rm\ref{ku4def36}(j)} we have $\ti\La_{acd}\od(\id_{\Phi_{cd}}*\ti\La_{abc})=\ti\La_{abd}\od(\ti\La_{bcd}*\id_{\Phi_{ab}}):\Phi_{cd}\ci\Phi_{bc}\ci\Phi_{ab}\Ra\Phi_{ad}$ over\/ $\dot S_{abcd},$ for all\/~$a,b,c,d\in A$.
\label{ku7lem3}
\end{lem}

\begin{proof} Fix $a,b,c\in A$. We will construct a 2-morphism $\ti\La_{abc}:\Phi_{bc}\ci\Phi_{ab}\Ra\Phi_{ac}$ over $\dot S_{abc}$. For each $d\in A$, define
\e
\ti S_{abc}^d=S_{dab}\cap S_{dac}\cap S_{dbc}\subseteq\dot S_{abc}.
\label{ku7eq78}
\e
Then $\ti S_{abc}^d$ is open in $\dot S_{abc}$. Definition \ref{ku4def36}(k) with $B=\{a,b,c\}$ implies that for each $x\in\dot S_{abc}$, there exists $d\in A$ with $x\in\ti S_{abc}^d$. Thus, $\{\ti S_{abc}^d:d\in A\}$ is an open cover of~$\dot S_{abc}$.

Since $\Phi_{da}$ is an equivalence in the weak 2-category $\Kur_{\ti S_{abc}^d}(X)$ in Definition \ref{ku4def7}, as it is a coordinate change, Lemma \ref{ku7lem2} implies that for each $d\in A$ there is a unique 2-morphism
\e
\begin{split}
\ti\La_{abc}^d:\Phi_{bc}\ci\Phi_{ab}\Longra\Phi_{ac}\quad\text{over $\ti S_{abc}^d$, such that} \\
\ti\La_{abc}^d*\id_{\Phi_{da}}=
\La_{dac}^{-1}\od\La_{dbc}\od(\id_{\Phi_{bc}}*\La_{dab}).
\end{split}
\label{ku7eq79}
\e

For $d,e\in A$, we will show that $\ti\La_{abc}^d\vert_{\ti S_{abc}^d\cap \ti S_{abc}^e}=\ti\La_{abc}^e\vert_{\ti S_{abc}^d\cap \ti S_{abc}^e}$. Let $x\in \ti S_{abc}^d\cap \ti S_{abc}^e$. Then Definition \ref{ku4def36}(k) with $B=\{a,b,c,d,e\}$ gives $f\in A$ with $x\in S_{fab}\cap S_{fac}\cap S_{fbc}\cap S_{fda}\cap S_{fdb}\cap S_{fdc}\cap S_{fea}\cap S_{feb}\cap S_{fec}\cap \ti S_{abc}^d\cap \ti S_{abc}^e$. Consider the diagram of 2-morphisms on this intersection:
\e
\begin{gathered}
\xymatrix@!0@C=21pt@R=30pt{ &&&&&&& \Phi_{bc}\ci\Phi_{ab}\ci\Phi_{fa} \ar`l/20pt[dlllllll]`^r[dddddlllllll] [ddddd]^(0.2){\ti\La_{abc}^d*\id_{\Phi_{fa}}} 
\ar`r/20pt[drrrrrrr]`_l[dddddrrrrrrr] [ddddd]_(0.2){\ti\La_{abc}^e*\id_{\Phi_{fa}}} 
\ar[d]^(0.6){\id_{\Phi_{bc}}*\La_{fab}} &&&&&&&
\\
& *+[r]{\Phi_{bc}\!\ci\!\Phi_{ab}\!\ci\!\Phi_{da}\!\ci\!\Phi_{fd}} \ar[urrrrrr]^(0.3){\id_{\Phi_{bc}\ci\Phi_{ab}}*\La_{fda}} \ar[drrr]^(0.65){\id_{\Phi_{bc}}*\La_{dab}*\id_{\Phi_{fd}}} \ar[ddd]^{\begin{subarray}{l}\La_{abc}* \\ \id_{\Phi_{da}\ci\Phi_{fd}}\end{subarray}} &&&&&& \Phi_{bc}\!\ci\!\Phi_{fb} \ar[ddd]^(0.62){\La_{fbc}} &&&&&& *+[l]{\Phi_{bc}\!\ci\!\Phi_{ab}\!\ci\!\Phi_{ea}\!\ci\!\Phi_{fe}} \ar[dlll]_(0.65){\id_{\Phi_{bc}}*\La_{eab}*\id_{\Phi_{fe}}} \ar[ddd]_{\begin{subarray}{l}\La_{abc}* \\ \id_{\Phi_{ea}\ci\Phi_{fe}}\end{subarray}} \ar[ullllll]_(0.3){\id_{\Phi_{bc}\ci\Phi_{ab}}*\La_{fea}} &
\\
&&&& \Phi_{bc}\!\ci\!\Phi_{db}\!\ci\!\Phi_{fd} \ar[d]^{\La_{dbc}*\id_{\Phi_{fd}}} \ar[urrr]_(0.6){\begin{subarray}{l}\id_{\Phi_{bc}}* \\ \La_{fdb}\end{subarray}} &&&&&& \Phi_{bc}\!\ci\!\Phi_{eb}\!\ci\!\Phi_{fe} \ar[d]_{\La_{ebc}*\id_{\Phi_{fe}}} \ar[ulll]^(0.6){\begin{subarray}{l}\id_{\Phi_{bc}}* \\ \La_{feb}\end{subarray}}
\\
&&&& \Phi_{dc}\!\ci\!\Phi_{fd} \ar[drrr]^{\La_{fdc}} &&&&&& \Phi_{ec}\!\ci\!\Phi_{fe} \ar[dlll]_{\La_{fec}}
\\
& *+[r]{\Phi_{ac}\!\ci\!\Phi_{da}\!\ci\!\Phi_{fd}} \ar[drrrrrr]^(0.6){\id_{\Phi_{ac}}*\La_{fda}} \ar[urrr]_(0.65){\La_{dac}*\id_{\Phi_{fd}}} &&&&&& \Phi_{fc} &&&&&& *+[l]{\Phi_{ac}\!\ci\!\Phi_{ea}\!\ci\!\Phi_{fe}} \ar[dllllll]_(0.6){\id_{\Phi_{ac}}*\La_{fea}} \ar[ulll]^(0.65){\La_{eac}*\id_{\Phi_{fe}}} &
\\
&&&&&&& \Phi_{ac}\ci\Phi_{fa}. \ar[u]_{\La_{fac}} &&&&&&& }
\end{gathered}
\label{ku7eq80}
\e
Here the outer two quadrilaterals commute by \eq{ku7eq79}, and the inner eight quadrilaterals commute by Definition \ref{ku4def36}(j). So \eq{ku7eq80} commutes. 

Thus, for each $x\in \ti S_{abc}^d\cap \ti S_{abc}^e$, on an open neighbourhood of $x$ we have $\ti\La_{abc}^d*\id_{\Phi_{fa}}=\ti\La_{abc}^e*\id_{\Phi_{fa}}$, so that on an open neighbourhood of $x$ we have $\ti\La_{abc}^d=\ti\La_{abc}^e$ by Lemma \ref{ku7lem2}. Definition \ref{ku4def11}(iii) and Theorem \ref{ku4thm1}(a) now imply that $\ti\La_{abc}^d=\ti\La_{abc}^e$ on $\ti S_{abc}^d\cap \ti S_{abc}^e$. Since the $\ti S_{abc}^d$ for $d\in A$ cover $\dot S_{abc}$, Definition \ref{ku4def11}(iii),(iv) and Theorem \ref{ku4thm1}(a) show that there exists a unique 2-morphism $\ti\La_{abc}:\Phi_{bc}\ci\Phi_{ab}\Ra\Phi_{ac}$ over $\dot S_{abc}$ such that
\e
\ti\La_{abc}\vert_{\ti S_{abc}^d}=\ti\La_{abc}^d\quad\text{for all $d\in A$.}
\label{ku7eq81}
\e

When $d=a$, we see from \eq{ku7eq78}--\eq{ku7eq79} and Definition \ref{ku4def36}(h),(i) that $\ti S_{abc}^a=S_{abc}$ and $\ti\La_{abc}^a=\La_{abc}$. Hence $\ti\La_{abc}\vert_{S_{abc}}=\La_{abc}$, as we have to prove.

Suppose $a,b,c,d\in A$, and $x\in \dot S_{abcd}=S_{ab}\cap S_{ac}\cap S_{ad}\cap S_{bc}\cap S_{bd}\cap S_{cd}$. Definition \ref{ku4def36}(k) with $B=\{a,b,c,d\}$ gives $e\in A$ with $x\in\ti S_{abc}^e\cap \ti S_{abd}^e\cap\ti S_{acd}^e\cap\ti S_{bcd}^e$. So, in an open neighbourhood of $x$ we have
\begin{align*}
&\bigl[\ti\La_{acd}\od(\id_{\Phi_{cd}}*\ti\La_{abc})\bigr]*\id_{\Phi_{ea}}=(\ti\La_{acd}^e*\id_{\Phi_{ea}})\od(\id_{\Phi_{cd}}*\ti\La_{abc}^e*\id_{\Phi_{ea}})\\
&=\bigl(\La_{ead}^{-1}\od\La_{ecd}\od(\id_{\Phi_{cd}}*\La_{eac})\bigr)\\
&\qquad\od\bigl((\id_{\Phi_{cd}}*\La_{eac}^{-1})\od(\id_{\Phi_{cd}}*\La_{ebc})\od(\id_{\Phi_{cd}}*\id_{\Phi_{bc}}*\La_{eab})\bigr)\\
&=\La_{ead}^{-1}\od\La_{ebd}\od\bigl(\La_{ebd}^{-1}\od\La_{ecd}\od(\id_{\Phi_{cd}}*\La_{ebc})\bigr)\od(\id_{\Phi_{cd}\ci\Phi_{bc}}*\La_{eab})\\
&=\bigl(\La_{ead}^{-1}\od\La_{ebd}\od(\id_{\Phi_{bd}}*\La_{eab})\bigr)\\
&\quad\od\bigl((\id_{\Phi_{bd}}*\La_{eab}^{-1})\od(\ti\La_{bcd}^e*\id_{\Phi_{eb}})\od(\id_{\Phi_{cd}\ci\Phi_{bc}}*\La_{eab})\bigr)\\
&=(\ti\La_{abd}^e*\id_{\Phi_{ea}})\od(\ti\La_{bcd}^e*\id_{\Phi_{ab}}*\id_{\Phi_{ea}})=\bigl[\ti\La_{abd}\od(\ti\La_{bcd}*\id_{\Phi_{ab}})\bigr]*\id_{\Phi_{ea}},
\end{align*}
using \eq{ku7eq81} in the first, fourth and sixth steps, and \eq{ku7eq79} in the second, third and fifth. Lemma \ref{ku7lem2} now implies that $\ti\La_{acd}\od(\id_{\Phi_{cd}}*\ti\La_{abc})=\ti\La_{abd}\od(\ti\La_{bcd}*\id_{\Phi_{ab}})$ holds near $x$. Applying Definition \ref{ku4def11}(iii) and Theorem \ref{ku4thm1}(a) again shows it holds on the correct domain $\dot S_{abcd}$. This completes the lemma.
\end{proof}

Next, for all $a,b\in A$ we have a coordinate change $\Phi_{ab}:(V_a,E_a,\ab\Ga_a,\ab s_a,\ab\psi_a)\ra (V_b,E_b,\Ga_b,s_b,\psi_b)$ over $S_{ab}\subseteq\Im\psi_a\cap\Im\psi_b$. This is an equivalence in the 2-category $\Kur_{S_{ab}}(X)$ by Definition \ref{ku4def8}. Thus we may choose a quasi-inverse $\check\Phi_{ba}:(V_b,E_b,\Ga_b,s_b,\psi_b)\ra(V_a,E_a,\ab\Ga_a,\ab s_a,\ab\psi_a)$, which is also a coordinate change over $S_{ab}$, and 2-morphisms
\e
\eta_{ab}:\Phi_{ab}\ci\check\Phi_{ba}\Ra\id_{(V_a,E_a,\Ga_a,s_a,\psi_a)},\quad
\ze_{ab}:\check\Phi_{ba}\ci\Phi_{ab}\Ra\id_{(V_b,E_b,\Ga_b,s_b,\psi_b)}.
\label{ku7eq82}
\e
When $a=b$, so that $\Phi_{aa}=\id_{(V_a,E_a,\Ga_a,s_a,\psi_a)}$, we choose
\e
\check\Phi_{aa}=\id_{(V_a,E_a,\Ga_a,s_a,\psi_a)}\quad\text{and}\quad\eta_{aa}=\ze_{aa}=\id_{\id_{(V_a,E_a,\Ga_a,s_a,\psi_a)}}.
\label{ku7eq83}
\e

Now fix $a,b\in A$. For all $c\in A$, we have $\dot S_{c(ab)}=S_{ca}\cap S_{cb}\subseteq\Im\psi_a\cap\Im\psi_b$. From Definition \ref{ku4def36}(k) with $B=\{a,b\}$, we see that for each $x\in\Im\psi_a\cap\Im\psi_b$ there exists $c\in A$ with $x\in\dot S_{c(ab)}$, so $\{\dot S_{c(ab)}:c\in A\}$ is an open cover of $\Im\psi_a\cap\Im\psi_b$. For each $c\in A$, define a 1-morphism $\Psi_{ab}^c:(V_a,E_a,\Ga_a,s_a,\psi_a)\ra(V_b,E_b,\Ga_b,s_b,\psi_b)$ over $\dot S_{c(ab)}$ by $\Psi_{ab}^c=\Phi_{cb}\ci\check\Phi_{ac}$.

\begin{lem} For all\/ $a,b,c,d\in A,$ there is a unique $2$-morphism
\e
\Mu_{ab}^{cd}:\Psi_{ab}^c\Longra \Psi_{ab}^d\quad\text{over $\dot S_{(cd)(ab)}=\dot S_{c(ab)}\cap \dot S_{d(ab)},$}
\label{ku7eq84}
\e
such that for all\/ $e\in A,$ the following commutes on $\dot S_{e(cd)(ab)}\!:$
\e
\begin{gathered}
\xymatrix@!0@R=27pt@C=35pt{
*+[r]{\Phi_{cb}\ci\Phi_{ec}} \ar@{=>}[rrrr]_(0.6){\ti\La_{ecb}} \ar@{=>}[d]^{\id_{\Phi_{cb}}*\ze_{ca}^{-1}*\id_{\Phi_{ac}}} &&&&
\Phi_{eb} \ar@{=>}[rrrr]_(0.4){\ti\La_{edb}^{-1}} &&&& *+[l]{\Phi_{db}\ci\Phi_{ed}} \ar@{=>}[d]_{\id_{\Phi_{db}}*\ze_{da}^{-1}*\id_{\Phi_{ad}} } \\
*+[r]{\Phi_{cb}\ci\check\Phi_{ac}\ci\Phi_{ca}\ci\Phi_{ec}} \ar@{=>}[d]^{\id_{\Phi_{cb}\ci\check\Phi_{ac}}*\ti\La_{eca}} &&&&&&&& *+[l]{\Phi_{db}\ci\check\Phi_{ad}\ci\Phi_{da}\ci\Phi_{ed}} \ar@{=>}[d]_{\id_{\Phi_{db}\ci\check\Phi_{ad}}*\ti\La_{eda}} \\
*+[r]{\Phi_{cb}\ci\check\Phi_{ac}\ci\Phi_{ea}} \ar@{=}[rrr] &&& {\Psi_{ab}^c\ci\Phi_{ea}} \ar@{=>}[rr]^{\raisebox{6pt}{$\st \Mu_{ab}^{cd}*\id_{\Phi_{ea}}$}} && {\Psi_{ab}^d\ci\Phi_{ea}} \ar@{=}[rrr] &&& *+[l]{\Phi_{db}\ci\check\Phi_{ad}\ci\Phi_{ea}.\!\!} }
\end{gathered}
\label{ku7eq85}
\e 

\label{ku7lem4}
\end{lem}

\begin{proof} Equation \eq{ku7eq85} determines $\Mu_{ab}^{cd}*\id_{\Phi_{ea}}$ over $\dot S_{e(cd)(ab)}$, and so by Lemma \ref{ku7lem2}, determines $\Mu_{ab}^{cd}$ over $\dot S_{e(cd)(ab)}$, as $\Phi_{ea}$ is an equivalence. Write $(\Mu_{ab}^{cd})^e$ for the value for $\Mu_{ab}^{cd}$ on $\dot S_{e(cd)(ab)}$ determined by \eq{ku7eq85}. Observe that Definition \ref{ku4def36}(k) with $B=\{a,b,c,d\}$ implies that the $\dot S_{e(cd)(ab)}$ for $e\in A$ form an open cover of~$\dot S_{(cd)(ab)}$.

Let $e,f\in A$, and $x\in\dot S_{(ef)(cd)(ab)}=\dot S_{e(cd)(ab)}\cap \dot S_{f(cd)(ab)}$. Applying Definition \ref{ku4def36}(k) with $B=\{a,b,c,d,e,f\}$ and this $x$ gives $g\in A$ such that all the 1- and 2-morphisms in the following diagram are defined on $x\in\dot S_{g(ef)(cd)(ab)}$:
\e
\begin{gathered}
\xymatrix@!0@C=21pt@R=15pt{ 
&&&&&&& \Psi_{ab}^c\ci\Phi_{ga} \ar[dllllll]^(0.3){\ti\La_{gea}^{-1}} \ar[drrrrrr]_(0.3){\ti\La_{gfa}^{-1}} \ar@<-1ex>[dd]^{\ti\La_{gca}^{-1}}
\ar`l/20pt[dlllllll]`^r[ddddddddddddddlllllll] [dddddddddddddd]_(0.2){(\Mu_{ab}^{cd})^e*\id_{\Phi_{ga}}} 
\ar`r/20pt[drrrrrrr]`_l[ddddddddddddddrrrrrrr] [dddddddddddddd]^(0.2){(\Mu_{ab}^{cd})^f*\id_{\Phi_{ga}}} &&&&&&&
\\
& *+[r]{\Phi_{cb}\!\ci\!\check\Phi_{ac}\!\ci\!\Phi_{ea}\!\ci\!\Phi_{ge}} \ar[dd]^{\ti\La_{eca}^{-1}} &&&&&&&&&&&& *+[l]{\Phi_{cb}\!\ci\!\check\Phi_{ac}\!\ci\!\Phi_{fa}\!\ci\!\Phi_{gf}} \ar[dd]_{\ti\La_{fca}^{-1}} & 
\\
&&&&&&& \Phi_{cb}\!\ci\!\check\Phi_{ac}\!\ci\!\Phi_{ca}\!\ci\!\Phi_{gc} \ar[dllllll]_(0.4){\ti\La_{gec}^{-1}} \ar[drrrrrr]^(0.4){\ti\La_{gfc}^{-1}} \ar@<-1ex>[dd]^{\ze_{ca}}
\\
& *+[r]{\Phi_{cb}\!\ci\!\check\Phi_{ac}\!\ci\!\Phi_{ca}\!\ci\!\Phi_{ec}\!\ci\!\Phi_{ge}}\ar[dd]^{\ze_{ca}} &&&&&&&&&&&& *+[l]{\Phi_{cb}\!\ci\!\check\Phi_{ac}\!\ci\!\Phi_{ca}\!\ci\!\Phi_{fc}\!\ci\!\Phi_{gf}}\ar[dd]_{\ze_{ca}} 
\\
&&&&&&& \Phi_{cb}\!\ci\!\Phi_{gc} \ar[dllllll]_(0.4){\ti\La_{gec}^{-1}} \ar[drrrrrr]^(0.4){\ti\La_{gfc}^{-1}} \ar@<-1ex>[dd]^{\ti\La_{gcb}}
\\
& *+[r]{\Phi_{cb}\!\ci\!\Phi_{ec}\!\ci\!\Phi_{ge}}\ar[dd]^{\ti\La_{ecb}} &&&&&&&&&&&& *+[l]{\Phi_{cb}\!\ci\!\Phi_{fc}\!\ci\!\Phi_{gf}}\ar[dd]_{\ti\La_{fcb}} 
\\
&&&&&&& \Phi_{gb} \ar[dllllll]_(0.4){\ti\La_{geb}^{-1}} \ar[drrrrrr]^(0.4){\ti\La_{gfb}^{-1}} \ar@<-1ex>[dd]^{\ti\La_{gdb}^{-1}}
\\
& *+[r]{\Phi_{eb}\!\ci\!\Phi_{ge}}\ar[dd]^{\ti\La_{edb}^{-1}} &&&&&&&&&&&& *+[l]{\Phi_{fb}\!\ci\!\Phi_{gf}} \ar[dd]_{\ti\La_{fdb}^{-1}} 
\\
&&&&&&& \Phi_{db}\!\ci\!\Phi_{gd} \ar[dllllll]_(0.4){\ti\La_{ged}^{-1}} \ar[drrrrrr]^(0.4){\ti\La_{gfd}^{-1}} \ar@<-1ex>[dd]^{\ze_{da}^{-1}}
\\
& *+[r]{\Phi_{db}\!\ci\!\Phi_{ed}\!\ci\!\Phi_{ge}} \ar[dd]^{\ze_{da}^{-1}} &&&&&&&&&&&& *+[l]{\Phi_{db}\!\ci\!\Phi_{fd}\!\ci\!\Phi_{gf}} \ar[dd]_{\ze_{da}^{-1}} 
\\
&&&&&&& {\Phi_{db}\!\ci\!\check\Phi_{ad}\!\ci\!\Phi_{da}\!\ci\!\Phi_{gd}} \ar[dllllll]_(0.45){\ti\La_{ged}^{-1}} \ar[drrrrrr]^(0.45){\ti\La_{gfd}^{-1}} \ar@<-1ex>[dddd]^{\ti\La_{gda}}
\\
& *+[r]{\Phi_{db}\!\ci\!\check\Phi_{ad}\!\ci\!\Phi_{da}\!\ci\!\Phi_{ed}\!\ci\!\Phi_{ge}} \ar[dd]^{\ti\La_{eda}} &&&&&&&&&&&& *+[l]{\Phi_{db}\!\ci\!\check\Phi_{ad}\!\ci\!\Phi_{da}\!\ci\!\Phi_{fd}\!\ci\!\Phi_{gf}} \ar[dd]_{\ti\La_{fda}} 
\\
\\
& *+[r]{\Phi_{db}\!\ci\!\check\Phi_{ad}\!\ci\!\Phi_{ea}\!\ci\!\Phi_{ge}}\ar[drrrrrr]^(0.7){\ti\La_{gea}} &&&&&&&&&&&& *+[l]{\Phi_{db}\!\ci\!\check\Phi_{ad}\!\ci\!\Phi_{fa}\!\ci\!\Phi_{gf}} \ar[dllllll]_(0.7){\ti\La_{gfa}} 
\\
&&&&&&& \Psi_{ab}^d\ci\Phi_{ga}.\!\! &&&&&&& }
\end{gathered}
\label{ku7eq86}
\e
Here for clarity we have omitted all `$\id_{\cdots}*$' and `$*\id_{\cdots}$' terms. The two outer nine-gons commute by \eq{ku7eq85}, eight small quadrilaterals commute by Lemma \ref{ku7lem3}, and four small quadrilaterals commute by compatibility of horizontal and vertical composition. Thus \eq{ku7eq86} commutes, and $(\Mu_{ab}^{cd})^e*\id_{\Phi_{ga}}=(\Mu_{ab}^{cd})^f*\id_{\Phi_{ga}}$ near $x$, so $(\Mu_{ab}^{cd})^e=(\Mu_{ab}^{cd})^f$ near $x$ by Lemma~\ref{ku7lem2}.

As this holds for all $x\in \dot S_{e(cd)(ab)}\cap \dot S_{f(cd)(ab)}$, Definition \ref{ku4def11}(iii) and Theorem \ref{ku4thm1}(a) show that $(\Mu_{ab}^{cd})^e=(\Mu_{ab}^{cd})^f$ on $\dot S_{e(cd)(ab)}\cap \dot S_{f(cd)(ab)}$. Since the $\dot S_{e(cd)(ab)}$ for $e\in A$ cover $\dot S_{(cd)(ab)}$, Definition \ref{ku4def11}(iii),(iv) and Theorem \ref{ku4thm1}(a) imply that there is a unique 2-morphism $\Mu_{ab}^{cd}$ as in \eq{ku7eq84} with $\Mu_{ab}^{cd}\vert_{\dot S_{e(cd)(ab)}}=(\Mu_{ab}^{cd})^e$. But by definition of $(\Mu_{ab}^{cd})^e$ this holds if and only if \eq{ku7eq85} commutes. This completes the lemma.
\end{proof}

\begin{lem} For all\/ $a,b,c,d,e\in A,$ we have
\e
\begin{split}
&\Mu_{ab}^{de}\od\Mu_{ab}^{cd}=\Mu_{ab}^{ce}:\Psi_{ab}^c\Longra \Psi_{ab}^e\\ 
&\text{over $\dot S_{(cde)(ab)}=\dot S_{(cd)(ab)}\cap \dot S_{(ce)(ab)}\cap\dot S_{(de)(ab)}$.}
\end{split}
\label{ku7eq87}
\e

\label{ku7lem5}
\end{lem}

\begin{proof} Let $x\in\dot S_{(cde)(ab)}$. Definition \ref{ku4def36}(k) with $B=\{a,b,c,d,e\}$ and this $x$ gives $f\in A$ such that all the 1- and 2-morphisms in the following diagram are defined on $x\in \dot S_{f(cde)(ab)}$:
\e
\begin{gathered}
\xymatrix@!0@R=35pt@C=132pt{
& \Phi_{fb} \ar@{=>}[dl]^(0.4){\ti\La_{fcb}^{-1}} \ar@<-1ex>@{=>}[d]^(0.4){\ti\La_{fdb}^{-1}} \ar@{=>}[dr]_(0.4){\ti\La_{feb}^{-1}}
 \\
*+[r]{\Phi_{cb}\!\ci\!\Phi_{fc}} \ar@{=>}[d]^{\id_{\Phi_{cb}}*\ze_{ca}^{-1}*\id_{\Phi_{ac}}} & {\Phi_{db}\!\ci\!\Phi_{fd}} \ar@<-1ex>@{=>}[d]^{\id_{\Phi_{db}}*\ze_{da}^{-1}*\id_{\Phi_{ad}} } 
& *+[l]{\Phi_{eb}\!\ci\!\Phi_{fe}} \ar@{=>}[d]_{\id_{\Phi_{eb}}*\ze_{ea}^{-1}*\id_{\Phi_{ae}} }
\\
*+[r]{\Phi_{cb}\!\ci\!\check\Phi_{ac}\!\ci\!\Phi_{ca}\!\ci\!\Phi_{fc}} \ar@{=>}[d]^{\id_{\Phi_{cb}\ci\check\Phi_{ac}}*\ti\La_{fca}} & {\Phi_{db}\!\ci\!\check\Phi_{ad}\!\ci\!\Phi_{da}\!\ci\!\Phi_{fd}} \ar@<-1ex>@{=>}[d]^(0.4){\id_{\Phi_{db}\ci\check\Phi_{ad}}*\ti\La_{fda}} & *+[l]{\Phi_{eb}\!\ci\!\check\Phi_{ae}\!\ci\!\Phi_{ea}\!\ci\!\Phi_{fe}} \ar@{=>}[d]_{\id_{\Phi_{eb}\ci\check\Phi_{ae}}*\ti\La_{fea}} 
\\
*+[r]{\Phi_{cb}\!\ci\!\check\Phi_{ac}\!\ci\!\Phi_{fa}} \ar@{=>}[r]^(0.6){\Mu_{ab}^{cd}*\id_{\Phi_{fa}}} \ar@<-.8ex>@/_.5pc/@{=>}[rr]_{\Mu_{ab}^{ce}*\id_{\Phi_{fa}}} & {\Phi_{db}\!\ci\!\check\Phi_{ad}\!\ci\!\Phi_{fa}} \ar@{=>}[r]^(0.4){\Mu_{ab}^{de}*\id_{\Phi_{fa}}} & *+[l]{\Phi_{eb}\!\ci\!\check\Phi_{ae}\!\ci\!\Phi_{fa}.\!\!} }
\end{gathered}
\label{ku7eq88}
\e 
Here the two inner and the outer septagons commute by \eq{ku7eq85}. Thus \eq{ku7eq88} commutes, and compatibility of horizontal and vertical composition gives
\begin{equation*}
(\Mu_{ab}^{de}\od\Mu_{ab}^{cd})*\id_{\Phi_{fa}}=(\Mu_{ab}^{de}*\id_{\Phi_{fa}})\od (\Mu_{ab}^{de}*\id_{\Phi_{fa}})=\Mu_{ab}^{ce}*\id_{\Phi_{fa}}
\end{equation*}
near $x$, so \eq{ku7eq87} holds near $x$ by Lemma \ref{ku7lem2}. As this is true for all $x\in\dot S_{(cde)(ab)}$, the lemma follows from Definition \ref{ku4def11}(iii) and Theorem~\ref{ku4thm1}(a).
\end{proof}

By Lemmas \ref{ku7lem4} and \ref{ku7lem5}, as $\{\dot S_{c(ab)}:c\in A\}$ is an open cover of $\Im\psi_a\cap\Im\psi_b$, we may now apply Definition \ref{ku4def11}(v) and Theorem \ref{ku4thm1}(a) to show that for all $a,b\in A$, there exists a coordinate change $\Psi_{ab}:(V_a,E_a,\Ga_a,s_a,\psi_a)\ra(V_b,E_b,\Ga_b,s_b,\psi_b)$ over $\Im\psi_a\cap\Im\psi_b$, and 2-morphisms $\ep_{ab}^c:\Psi_{ab}^c\Ra\Psi_{ab}$ over $\dot S_{c(ab)}$ for all $c\in A$, such that for all $c,d\in A$ we have 
\e
\ep_{ab}^d\od\Mu_{ab}^{cd}=\ep_{ab}^c:\Psi_{ab}^c\Longra\Psi_{ab}\quad
\quad\text{over $\dot S_{(cd)(ab)}=\dot S_{c(ab)}\cap \dot S_{d(ab)}.$}
\label{ku7eq89}
\e
Furthermore $\Psi_{ab}$ is unique up to 2-isomorphism.

In the case when $a=b$, we have $\Psi_{aa}^a=\Phi_{aa}=\check\Phi_{aa}=\id_{(V_a,E_a,\Ga_a,s_a,\psi_a)}$ and $\dot S_{a(aa)}=\Im\psi_a$, so $\ep_{aa}^a:\id_{(V_a,E_a,\Ga_a,s_a,\psi_a)}\Ra\Psi_{aa}$ is a 2-morphism over $\Im\psi_a$. As we can choose $\Psi_{aa}$ freely in its 2-isomorphism class, we choose
\e
\Psi_{aa}=\id_{(V_a,E_a,\Ga_a,s_a,\psi_a)}\;\>\text{and}\;\>
\ep_{aa}^a=\id_{\id_{(V_a,E_a,\Ga_a,s_a,\psi_a)}},\;\>\text{for all $a\in A$.}
\label{ku7eq90}
\e

\begin{lem} For all\/ $a,b,c\in A,$ there is a unique $2$-morphism
\begin{equation*}
\Ka_{abc}:\Psi_{bc}\ci\Psi_{ab}\Longra \Psi_{ac}\quad\text{over\/ $\Im\psi_a\cap\Im\psi_b\cap\Im\psi_c,$}
\end{equation*}
such that for all\/ $d\in A,$ the following commutes over $\dot S_{d(abc)}\!:$
\e
\begin{gathered}
\xymatrix@R=13pt@C=60pt{
*+[r]{\Phi_{dc}\ci\check\Phi_{bd}\ci\Phi_{db}\ci\check\Phi_{ad}} 
\ar@{=}[r] \ar@{=>}[d]^{\id_{\Phi_{dc}}*\ze_{db}*\id_{\check\Phi_{ad}} } & \Psi_{bc}^d\ci\Psi_{ab}^d \ar@{=>}[r]_(0.45){\ep_{bc}^d*\ep_{ab}^d} & *+[l]{\Psi_{bc}\ci\Psi_{ab}} \ar@{=>}[d]_(0.45){\Ka_{abc}} \\
*+[r]{\Phi_{dc}\ci\check\Phi_{ad}} \ar@{=}[r] & \Psi_{ac}^d \ar@{=>}[r]^(0.45){\ep_{ac}^d} & *+[l]{\Psi_{ac}.\!\!} }
\end{gathered}
\label{ku7eq91}
\e 

\label{ku7lem6}
\end{lem}

\begin{proof} Fix $a,b,c\in A$. If $x\in\Im\psi_a\cap\Im\psi_b\cap\Im\psi_c$, then Definition \ref{ku4def36}(k) with $B=\{a,b,c\}$ and this $x$ gives $d\in A$ with $x\in\dot S_{d(abc)}$. Hence $\bigl\{\dot S_{d(abc)}:d\in A\bigr\}$ is an open cover of~$\Im\psi_a\cap\Im\psi_b\cap\Im\psi_c$.

For each $d\in A$, write $\Ka_{abc}^d$ for the 2-morphism over $\dot S_{d(abc)}$ determined by \eq{ku7eq91} with $\Ka_{abc}^d$ in place of $\Ka_{abc}$. We have to show that there is a unique 2-morphism $\Ka_{abc}$ over $\Im\psi_a\cap\Im\psi_b\cap\Im\psi_c$ with~$\Ka_{abc}\vert_{\dot S_{d(abc)}}=\Ka_{abc}^d$. 

\begin{figure}[htb]
\text{\begin{footnotesize}$\displaystyle
\xymatrix@!0@C=24pt@R=18pt{ 
&&&&&&& {\Psi_{bc}\!\ci\!\Psi_{ab}\!\ci\!\Phi_{fa}} \ar[ddllllll]_(0.6){(\ep_{bc}^d)^{-1}*(\ep_{ab}^d)^{-1}} \ar[ddrrrrrr]^(0.6){(\ep_{bc}^e)^{-1}*(\ep_{ab}^e)^{-1}} 
\ar`l/20pt[ddlllllll]`^r[dddddddddddddddddlllllll] [ddddddddddddddddd]_(0.2){\Ka_{abc}^d*\id_{\Phi_{fa}}} 
\ar`r/20pt[ddrrrrrrr]`_l[dddddddddddddddddrrrrrrr] [ddddddddddddddddd]^(0.2){\Ka_{abc}^e*\id_{\Phi_{fa}}} &&&&&&&
\\
\\
& *+[r]{\Phi_{dc}\!\ci\!\check\Phi_{bd}\!\ci\!\Phi_{db}\!\ci\!\check\Phi_{ad}\!\ci\!\Phi_{fa}} \ar@<-1.5ex>@/_3.5ex/[ddddddddddddd]^{\!\ze_{db}}\ar[rrrrrrrrrrrr]^{M_{bc}^{de}*M_{ab}^{de}} \ar[dd]^{\ti\La_{fda}^{-1}} 
&&& {\phantom{PPPPP_{PPP}}} \ar@/^.5pc/[dddddrrr]^(0.4){\id_{\smash{\Psi_{bc}^d}}*\Mu_{ab}^{de}}
&&&&&&&&& *+[l]{\Phi_{ec}\!\ci\!\check\Phi_{be}\!\ci\!\Phi_{eb}\!\ci\!\check\Phi_{ae}\!\ci\!\Phi_{fa}} \ar[dd]_{\ze_{eb}} & 
\\
\\
& *+[r]{{}\;\;\begin{subarray}{l}\ts \Phi_{dc}\!\ci\!\check\Phi_{bd}\!\ci\!\Phi_{db}\ci \\ \ts\check\Phi_{ad}\!\ci\!\Phi_{da}\!\ci\!\Phi_{fd}\end{subarray}}
\ar[dd]^{\ze_{da}} &&&&&&&&&&&& *+[l]{\Phi_{ec}\!\ci\!\check\Phi_{ae}\!\ci\!\Phi_{fa}} \ar[dd]_{\ti\La_{edc}^{-1}} \ar@<1.5ex>@/^1.5ex/@{=}[ddddddddddd]
\\
\\
& *+[r]{{}\;\;\begin{subarray}{l}\ts \Phi_{dc}\!\ci\!\check\Phi_{bd}\ci\\ \ts\Phi_{db}\!\ci\!\Phi_{fd}\end{subarray}}
\ar[ddddd]^{\ze_{db}} \ar[drrr]_(0.6){\ti\La_{fdb}} &&&&&&&&&&&& 
*+[l]{\begin{subarray}{l}\ts \Phi_{dc}\!\ci\!\Phi_{ed}\ci\\ \ts \check\Phi_{ae}\!\ci\!\Phi_{fa}\end{subarray}\;\;} \ar[dlll]^(0.5){\ze_{db}^{-1}}
\\
&&&& {\begin{subarray}{l}\ts \Phi_{dc}\!\ci\!\check\Phi_{bd}\\ \ts\ci\Phi_{fb}\end{subarray}} \ar[dd]_{\ti\La_{feb}^{-1}} &&& 
{\!\!\!\!\!\!\begin{subarray}{l}\ts \Phi_{dc}\!\ci\!\check\Phi_{bd}\!\ci\!\Phi_{eb}\\ \\ \ts \ci\check\Phi_{ae}\!\ci\!\Phi_{fa}\end{subarray}}
&&& {\!\!\!\begin{subarray}{l}\ts \Phi_{dc}\!\ci\!\check\Phi_{bd}\!\ci\!\Phi_{db}\ci \\ \ts \Phi_{ed}\!\ci\!\check\Phi_{ae}\!\ci\!\Phi_{fa}\end{subarray}} \ar[lll]^(0.55){\ti\La_{edb}} 
\\
\\
&&&&
{\begin{subarray}{l}\ts \Phi_{dc}\!\ci\!\check\Phi_{bd}\ci\\ \ts\Phi_{eb}\!\ci\!\Phi_{fe}\end{subarray}\;\;} \ar[rrr]^(0.45){\ze_{ea}^{-1}}
&&& {\begin{subarray}{l}\ts \Phi_{dc}\!\ci\!\check\Phi_{bd}\!\ci\!\Phi_{eb}\ci\\ \ts \check\Phi_{ae}\!\ci\!\Phi_{ea}\!\ci\!\Phi_{fe}\end{subarray}} \ar[uu]_(0.5){\ti\La_{fea}} &&& {\begin{subarray}{l}\ts \Phi_{dc}\!\ci\!\check\Phi_{bd}\!\ci\!\Phi_{db}\!\ci\!\Phi_{ed}\\ \ts\ci\check\Phi_{ae}\!\ci\!\Phi_{ea}\!\ci\!\Phi_{fe}\end{subarray}\!\!\!\!\!\!\!\!\!\!\!\!\!\!\!\!\!\!\!\!\!\!\!\!\!\!} \ar[lll]_(0.42){\ti\La_{edb}} \ar[uu]_(0.5){\ti\La_{fea}} \ar[ddlll]^(0.45){\ze_{db}*\ze_{ea}}
\\
\\
& *+[r]{\Phi_{dc}\!\ci\!\Phi_{fd}} \ar[ddrrrrrr]^{\ti\La_{fdc}} \ar[dd]^{\ze_{da}^{-1}} &&&&&& {\Phi_{dc}\!\ci\!\Phi_{ed}\!\ci\!\Phi_{fe}} \ar[llllll]_{\ti\La_{fed}} \ar[rrrrrr]^(0.55){\ti\La_{edc}} &&&&&& *+[l]{\Phi_{ec}\!\ci\!\Phi_{fe}} \ar[ddllllll]_{\ti\La_{fec}} \ar[dd]_{\ze_{ea}^{-1}}
\\
\\
& *+[r]{\Phi_{dc}\!\ci\!\check\Phi_{ad}\!\ci\!\Phi_{da}\!\ci\!\Phi_{fd}}\ar[dd]^{\ti\La_{fda}} &&&&&& {\Phi_{fc}} &&&&&& *+[l]{\Phi_{ec}\!\ci\!\check\Phi_{ae}\!\ci\!\Phi_{ea}\!\ci\!\Phi_{fe}} \ar[dd]_{\ti\La_{fea}} 
\\
\\
& *+[r]{\Phi_{dc}\!\ci\!\check\Phi_{ad}\!\ci\!\Phi_{fa}} \ar[ddrrrrrr]^(0.7){\ep_{ac}^d} \ar[rrrrrrrrrrrr]^{M_{ac}^{de}} &&&&&&&&&&&& *+[l]{\Phi_{ec}\!\ci\!\check\Phi_{ae}\!\ci\!\Phi_{fa}} \ar[ddllllll]_(0.7){\ep_{ac}^e} 
\\
\\
&&&&&&& {\Psi_{ac}\!\ci\!\Phi_{fa}.\!\!} &&&&&&& }$\end{footnotesize}}
\caption{Proof that $\Ka_{abc}^d*\id_{\Phi_{fa}}=\Ka_{abc}^e*\id_{\Phi_{fa}}$}
\label{ku7fig2}
\end{figure}
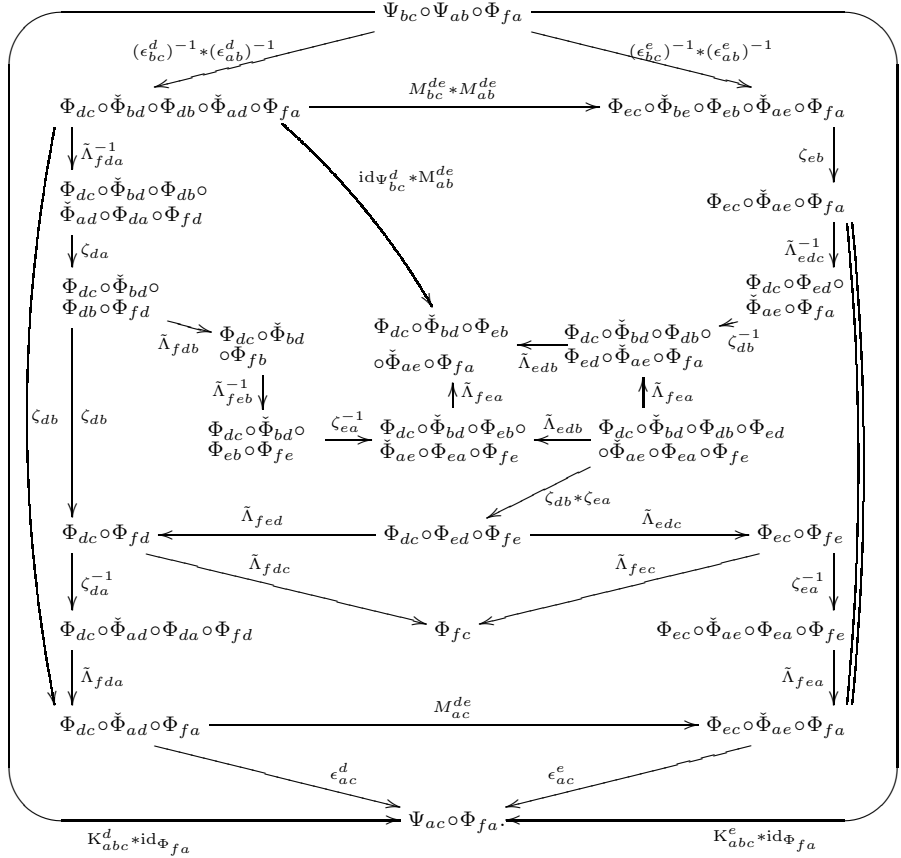

Let $d,e\in A$, and $x\in\dot S_{(de)(abc)}=\dot S_{d(abc)}\cap\dot S_{e(abc)}$. Definition \ref{ku4def36}(k) with $B=\{a,b,c,d,e\}$ and this $x$ gives $f\in A$ with $x\in\dot S_{f(de)(abc)}$. Consider the diagram of 1- and 2-morphisms Figure \ref{ku7fig2}. We have omitted most terms $*\id_{\cdots}$ and $\id_{\cdots}*$ in the 2-morphisms for clarity. The two outer crescent shapes are the definitions of $\Ka_{abc}^d,\Ka_{abc}^e$ in \eq{ku7eq91}, composed with $\Phi_{fa}$. The top and bottom triangles commute by \eq{ku7eq89}. In the interior of the figure, the three polygons with sides involving $M_{ab}^{de},M_{ac}^{de},M_{bc}^{de}$ commute by \eq{ku7eq85}. The remaining four polygons commute by Lemma \ref{ku7lem3} and compatibility of horizontal and vertical composition. Thus Figure \ref{ku7fig2} commutes, which proves that $\Ka_{abc}^d*\id_{\Phi_{fa}}=\Ka_{abc}^e*\id_{\Phi_{fa}}$ on $\dot S_{f(de)(abc)}$. Lemma \ref{ku7lem2} now shows that~$\Ka_{abc}^d=\Ka_{abc}^e$ on $\dot S_{f(de)(abc)}$. 

As the $\dot S_{f(de)(abc)}$ for $f\in A$ cover $\dot S_{d(abc)}\cap\dot S_{e(abc)}$, Definition \ref{ku4def11}(iii) and Theorem \ref{ku4thm1}(a) imply that $\Ka_{abc}^d=\Ka_{abc}^e$ on $\dot S_{d(abc)}\cap\dot S_{e(abc)}$. Since $\bigl\{\dot S_{d(abc)}:d\in A\bigr\}$ is an open cover of $\Im\psi_a\cap\Im\psi_b\cap\Im\psi_c$, Definition \ref{ku4def11}(iii),(iv) and Theorem \ref{ku4thm1}(a) show that there exists a unique 2-morphism $\Ka_{abc}$ over $\Im\psi_a\cap\Im\psi_b\cap\Im\psi_c$ such that $\Ka_{abc}\vert_{\dot S_{d(abc)}}=\Ka_{abc}^d$. Thus \eq{ku7eq91} commutes for all $d\in A$, by definition of $\Ka_{abc}^d$. This completes the proof.
\end{proof}

Putting $a,a,b,a$ in place of $a,b,c,d$ in \eq{ku7eq91} and using $\ep_{aa}^a,\ze_{aa}$ identities by \eq{ku7eq83}, \eq{ku7eq90}, and similarly putting $a,b,b,b$ in place of $a,b,c,d$ and using $\ep_{bb}^b,\ze_{bb}$ identities, yields
\e
\Ka_{aab}=\Ka_{abb}=\id_{\Psi_{ab}}.
\label{ku7eq92}
\e 

\begin{lem} For all\/ $a,b,c,d\in A$ we have $\Ka_{acd}\od(\id_{\Psi_{cd}}*\Ka_{abc})=\Ka_{abd}\od(\Ka_{bcd}*\id_{\Psi_{ab}}):\Psi_{cd}\ci\Psi_{bc}\ci\Psi_{ab}\Longra\Psi_{ad}$ over\/~$\Im\psi_a\cap\Im\psi_b\cap\Im\psi_c\cap\Im\psi_d$.
\label{ku7lem7}
\end{lem}

\begin{proof} Let $x\in\Im\psi_a\cap\Im\psi_b\cap\Im\psi_c\cap\Im\psi_d$. Definition \ref{ku4def36}(k) with $B=\{a,b,c,d\}$ and this $x$ gives $e\in A$ with $x\in\dot S_{e(abcd)}$. Consider the diagram
\e
\begin{gathered}
\xymatrix@!0@C=74pt@R=50pt{
*+[r]{\Psi_{cd}\ci\Psi_{bc}\ci\Psi_{ab}} \ar@{=>}[dr]^(0.5){\,\,\,(\ep_{cd}^e*\ep_{bc}^e*\ep_{ab}^e)^{-1}} \ar@{=>}[rrrr]_{\id_{\Psi_{cd}}*\Ka_{abc}} \ar@{=>}[ddd]^{\Ka_{bcd}*\id_{\Psi_{ab}}} &&&& *+[l]{\Psi_{cd}\ci\Psi_{ac}} \ar@{=>}[ddd]_{\Ka_{acd}} \ar@<.7ex>@{=>}[dl]^(0.65){\,\,\,(\ep_{cd}^e*\ep_{ac}^e)^{-1}} 
\\ 
& {\begin{subarray}{l}\ts \Phi_{ed}\!\ci\!\check\Phi_{ce}\!\ci\!\Phi_{ec}\ci\\ \ts\check\Phi_{be}\!\ci\!\Phi_{eb}\!\ci\!\check\Phi_{ae}\end{subarray}} \ar@{=>}[d]^(0.6){\id_{\Phi_{ed}}*\ze_{ec}*\id_{\check\Phi_{be}\ci\Phi_{eb}\ci\check\Phi_{ae}}} \ar@{=>}[rr]^(0.55){\begin{subarray}{l}\id_{\Phi_{ed}\ci\check\Phi_{ce}\ci\Phi_{ec}}* \\ \ze_{eb}*\id_{\check\Phi_{ae}}\end{subarray}} && {\begin{subarray}{l}\ts \Phi_{ed}\!\ci\!\check\Phi_{ce}\ci \\ \ts\Phi_{ec}\!\ci\!\check\Phi_{ae}\end{subarray}} \ar@{=>}[d]_(0.4){\id_{\Phi_{ed}}*\ze_{ec}*\id_{\check\Phi_{ae}}} \ar@<.7ex>@{=>}[ur]^(0.45){\ep_{cd}^e*\ep_{ac}^e\,}
\\
& {\begin{subarray}{l}\ts \Phi_{ed}\!\ci\!\check\Phi_{be}\ci \\ \ts \Phi_{eb}\!\ci\!\check\Phi_{ae}\end{subarray}} \ar@<-.7ex>@{=>}[dl]_(0.5){\ep_{bd}^e*\ep_{ab}^e\,\,} \ar@{=>}[rr]_{\id_{\Phi_{ed}}*\ze_{eb}*\id_{\check\Phi_{ae}}} && {\Phi_{ed}\!\ci\!\check\Phi_{ae}} \ar@{=>}[dr]^{\ep_{ad}^e} 
\\
*+[r]{\Psi_{bd}\ci\Psi_{ab}} \ar@<-.7ex>@{=>}[ur]_(0.5){\,\,\,(\ep_{bd}^e*\ep_{ab}^e)^{-1}} 
\ar@{=>}[rrrr]^{\Ka_{abd}} &&&& *+[l]{\Psi_{ad}.\!\!} }\!\!\!{} 
\end{gathered}
\label{ku7eq93}
\e
Here the four outer quadrilaterals commute by \eq{ku7eq91}, and the inner rectangle commutes by compatibility of horizontal and vertical multiplication. Thus \eq{ku7eq93} commutes, and the outer rectangle shows that $\Ka_{acd}\od(\id_{\Psi_{cd}}*\Ka_{abc})=\Ka_{abd}\od(\Ka_{bcd}*\id_{\Psi_{ab}})$ holds over $\dot S_{e(abcd)}$. Since the $\dot S_{e(abcd)}$ for all $e\in A$ cover $\Im\psi_a\cap\Im\psi_b\cap\Im\psi_c\cap\Im\psi_d$, the lemma follows from Definition \ref{ku4def11}(iii) and Theorem~\ref{ku4thm1}(a). 
\end{proof}

The definition of the $\Psi_{ab}$ after Lemma \ref{ku7lem5}, Lemmas \ref{ku7lem6}--\ref{ku7lem7}, and equations \eq{ku7eq90} and \eq{ku7eq92}, now imply that $\cK=\bigl(A,(V_a,E_a,\Ga_a,s_a,\psi_a)_{a\in A}$, $\Psi_{ab,\;a,b\in A}$, $\Ka_{abc,\; a,b,c\in A}\bigr)$ is a Kuranishi structure on $X$ in the sense of \S\ref{ku43}, so $\bX=(X,\cK)$ is a Kuranishi space with $\vdim\bX=n$, as we have to prove.

To give $(V_a,E_a,\Ga_a,s_a,\psi_a)$ the structure of a Kuranishi neighbourhood on the Kuranishi space $\bX$ in the sense of \S\ref{ku44} for $a\in A$, note that as $(V_a,\ab E_a,\ab\Ga_a,\ab s_a,\ab\psi_a)$ is already part of the Kuranishi structure $\cK$, we can take $\Psi_{ai,\; i\in A}$ and
$\Ka_{aij,\; i,j\in A}$ to be the implicit extra data $\Phi_{ai,\; i\in I}$, $\La_{aij,\; i,j\in I}$ in Definition~\ref{ku4def19}.

To give $\Phi_{ab}:(V_a,E_a,\Ga_a,s_a,\psi_a)\ra(V_b,E_b,\Ga_b,s_b,\psi_b)$ the structure of a coordinate change over $S_{ab}$ on the Kuranishi space $\bX$ as in \S\ref{ku44} for $a,b\in A$, we need to specify implicit extra data $\Io_{abi,\; i\in A}$ in place of $\La_{abi,\; i\in A}$ in Definition \ref{ku4def20}, where $\Io_{abi}:\Psi_{bi}\ci\Phi_{ab}\Ra\Psi_{ai}$ is a 2-morphism over $S_{ab}\cap\Im\psi_i$ for all $i\in A$ satisfying \eq{ku4eq39} over $S_{ab}\cap\Im\psi_i\cap\Im\psi_j$ for all $i,j\in A$, which becomes 
\e
\Ka_{aij}\od(\id_{\Psi_{ij}}*\Io_{abi})=\Io_{abj}\od(\Ka_{bij}*\id_{\Phi_{ab}}):\Psi_{ij}\ci\Psi_{bi}\ci\Phi_{ab}\Longra\Psi_{aj}.
\label{ku7eq94}
\e

Since $\check\Phi_{aa}=\id_{V_a,E_a,\Ga_a,s_a,\psi_a)}$ by \eq{ku7eq83} we have $\Psi_{ab}^a=\Phi_{ab}$, so the definition of $\Psi_{ab}$ gives a 2-morphism $\ep_{ab}^a:\Phi_{ab}\Ra\Psi_{ab}$ over $S_{ab}\subseteq\Im\psi_a\cap\Im\psi_b$. Define $\Io_{abi}=\Ka_{abi}\od(\id_{\Psi_{bi}}*(\ep_{ab}^a)^{-1})$. Then \eq{ku7eq94} follows from vertically composing $\id_{\Psi_{ij}\ci\Psi_{bi}}*(\ep_{ab}^a)^{-1}$ with Lemma \ref{ku7lem7} with $i,j$ in place of $c,d$. This makes $\Phi_{ab}$ into a coordinate change over $S_{ab}$ on $\bX$, as we want.

Now let $a,b,c\in A$. To show that $\La_{abc}:\Phi_{bc}\ci\Phi_{ab}\Ra\Phi_{ac}$ is the unique $2$-morphism over $S_{abc}$ given by Theorem \ref{ku4thm4}(a), we must prove that as in \eq{ku4eq40}, for all $i\in A$, over $S_{abc}\cap\Im\psi_i$ we have 
\e
\Io_{abi}\od(\Io_{bci}*\id_{\Phi_{ab}})=\Io_{aci}\od(\id_{\Psi_{ci}}*\La_{abc}):\Psi_{ci}\ci\Phi_{bc}\ci\Phi_{ab}\Longra\Psi_{ai}.
\label{ku7eq95}
\e
To prove \eq{ku7eq95}, consider the diagram of 2-morphisms over $S_{abc}\cap\Im\psi_i$:
\e
\begin{gathered}
\xymatrix@!0@C=37pt@R=25pt{
*+[r]{\Psi_{ci}\ci\Phi_{bc}\ci\Phi_{ab}} \ar@{=>}[dddrr]^(0.6){\id_{\Psi_{ci}}*\ep_{bc}^b*\ep_{ab}^a} \ar@{=>}[rrrrrrrr]_(0.5){\id_{\Psi_{ci}}*\La_{abc}=\id_{\Psi_{ci}}*\ti\La_{abc}} \ar@{=>}[dddddd]^(0.6){\Io_{bci}*\id_{\Phi_{ab}}} \ar@{=>}[drrrr]_{\id_{\Psi_{ci}}*M_{bc}^{ba}*\id_{\Phi_{ab}}} &&&&&&&& *+[l]{\Psi_{ci}\ci\Phi_{ac}} \ar@{=>}[dddddd]_{\Io_{aci}} \ar@{=>}[ddl]_(0.75){\id_{\Psi_{ci}}*\ep_{ac}^a} 
\\ 
&&&& \Psi_{ci}\ci\Phi_{ac}\ci\check\Phi_{ba}\ci\Phi_{ab} \ar@{=>}[ddll]^(0.4){\id_{\Psi_{ci}}*\ep_{bc}^a*\ep_{ab}^a} \ar@{=>}[urrrr]_(0.55){\id_{\Psi_{ci}\ci\Phi_{ac}}*\ze_{ab}} \\
&&&&&&& \Psi_{ci}\ci\Psi_{ac} \ar@{=>}[ddddr]_(0.4){\Ka_{aci}} 
\\
&& \Psi_{ci}\ci\Psi_{bc}\ci\Psi_{ab} \ar@{=>}[dd]^{\Ka_{bci}*\id_{\Psi_{ab}}} \ar@{=>}[urrrrr]_{\id_{\Psi_{ci}}*\Ka_{abc}}
\\
\\
&& \Psi_{bi}\ci\Psi_{ab} \ar@{=>}[drrrrrr]^{\Ka_{abi}} 
\\
*+[r]{\Psi_{bi}\ci\Phi_{ab}} \ar@{=>}[urr]^(0.5){\id_{\Psi_{bi}}*\ep_{ab}^a\,\,\,} 
\ar@{=>}[rrrrrrrr]^(0.4){\Io_{abi}} &&&&&&&& *+[l]{\Psi_{ai}.\!\!} }\!\!\!{} 
\end{gathered}
\label{ku7eq96}
\e
Here the bottom and rightmost triangles, and the leftmost quadrilateral, commute by definition of $\Io_{abi}$. The lower central quadrilateral commutes by Lemma \ref{ku7lem7}, the upper central quadrilateral by \eq{ku7eq91} with $d=a$, the upper left triangle by \eq{ku7eq89}, and the topmost triangle by \eq{ku7eq85} with $b,c,b,a,a$ in place of $a,b,c,d,e$, noting that of the seven morphisms in \eq{ku7eq85}, four are identities in this case, so we omit them. Also we use $\ti\La_{abc}\vert_{S_{abc}}=\La_{abc}$ from Lemma \ref{ku7lem3}. Thus \eq{ku7eq96} commutes, and the outer rectangle yields \eq{ku7eq95}. Hence $\La_{abc}:\Phi_{bc}\ci\Phi_{ab}\Ra\Phi_{ac}$ is the unique $2$-morphism over $S_{abc}$ given by Theorem \ref{ku4thm4}(a). This completes the proof of the first part of Theorem~\ref{ku4thm9}.

It remains to show that $\bX=(X,\cK)$ is unique up to equivalence in $\Kur$. To prove this, we have to consider where in the proof above we made arbitrary choices, and show that if we made different choices yielding $\bX'=(X,\cK')$, then $\bX$ and $\bX'$ are equivalent in $\Kur$. There are two places in the construction of $\bX$ where we made arbitrary choices: firstly the choice after Lemma \ref{ku7lem3} of a quasi-inverse $\check\Phi_{ba}$ for $\Phi_{ab}$ and 2-morphisms $\eta_{ab},\ze_{ab}$ in \eq{ku7eq82} (though in fact the $\eta_{ab}$ were never used in the definition of $\bX$), and secondly the choice after Lemma \ref{ku7lem5} of $\Psi_{ab}$ and 2-morphisms $\ep_{ab}^c$ satisfying~\eq{ku7eq89}.

For the first, if $\check\Phi_{ba}',\eta_{ab}',\ze_{ab}'$ are alternative choices for $\check\Phi_{ba},\eta_{ab},\ze_{ab}$, for all $a,b\in A$, then there exist unique 2-morphisms $\al_{ab}:\check\Phi_{ba}\Ra \check\Phi_{ba}'$ such that
\e
\ze_{ab}=\ze'_{ab}\od(\al_{ab}*\id_{\Phi_{ab}})\quad\text{for all $a,b\in A$,}
\label{ku7eq97}
\e
and $\al_{aa}=\id_{\id_{(V_a,E_a,\Ga_a,s_a,\psi_a)}}$. Then one can check that for the second choice we can keep $\Psi_{ab}$ unchanged and replace $\ep_{ab}^c$ by
\e
\ep_{ab}^{\prime c}=\ep_{ab}^c\od(\id_{\Phi_{cb}}*(\al_{ac})^{-1})\quad\text{for all $a,b,c\in A$.}
\label{ku7eq98}
\e
Using \eq{ku7eq97}--\eq{ku7eq98} to compare \eq{ku7eq91} for $\check\Phi_{ba},\eta_{ab},\ze_{ab},\ep_{ab}^c$ and $\check\Phi_{ba}',\eta_{ab}',\ab\ze_{ab}',\ab\ep_{ab}^{\prime c}$, we find that the two occurrences of $\al_{da}$ and of $\al_{db}$ cancel, so $\Ka_{abc}$ is unchanged. Thus, the family of possible outcomes for $\Psi_{ab},\Ka_{abc}$ and $\bX$ are independent of the first choice of $\check\Phi_{ba},\eta_{ab},\ze_{ab}$ for $a,b\in A$.

Next, regard the $\check\Phi_{ba},\eta_{ab},\ze_{ab}$ as fixed, and let $\Psi_{ab}',\ep_{ab}^{\prime c}$ be alternative possibilities for $\Psi_{ab},\ep_{ab}^c$ in the second choice, and $\Ka_{abc}'$ the corresponding 2-morphisms in Lemma \ref{ku7lem6}. Then by Theorem \ref{ku4thm1}(a) and the last part of Definition \ref{ku4def11}(v), there are unique 2-morphisms $\be_{ab}:\Psi_{ab}\Ra \Psi_{ab}'$ for all $a,b\in A$, such that
\e
\ep_{ab}^{\prime c}=\be_{ab}\od \ep_{ab}^c\quad\text{for all $a,b,c\in A$.}
\label{ku7eq99}
\e
Substituting \eq{ku7eq99} into \eq{ku7eq91} for $\Psi_{ab}',\ep_{ab}^{\prime c},\Ka_{abc}'$ and comparing with \eq{ku7eq91} for $\Psi_{ab},\ep_{ab}^c,\Ka_{abc}$, we see that
\begin{equation*}
\Ka'_{abc}=\be_{ac}\od\Ka_{abc}\od(\be_{bc}^{-1}*\be_{ab}^{-1}).
\end{equation*}

Define 1-morphisms $\bs f:\bX\ra\bX'$, $\bs g:\bX'\ra\bX$, in the notation of \eq{ku4eq15}, by
\begin{align*}
\bs f=\bigl(\id_X,\Psi_{ab,\;a\in A,\; b\in A},\; (\Ka_{aa'b})_{aa',\;a,a'\in A}^{b,\; b\in A},\; (\Ka_{abb'}\od(\be_{bb'}^{-1}*\id_{\Psi_{ab}}))_{a,\;a\in A}^{bb',\; b,b'\in A}\bigr),\\
\bs g=\bigl(\id_X,\Psi'_{ab,\;a\in A,\; b\in A},\; (\Ka'_{aa'b})_{aa',\;a,a'\in A}^{b,\; b\in A},\; (\Ka'_{abb'}\od(\be_{bb'}*\id_{\Psi'_{ab}}))_{a,\;a\in A}^{bb',\; b,b'\in A}\bigr).
\end{align*}
One can check these satisfy Definition \ref{ku4def13}(a)--(h), and so are 1-morphisms of Kuranishi spaces. Definition \ref{ku4def15} now gives a 1-morphism of Kuranishi spaces $\bs g\ci\bs f:\bX\ra\bX$, and 2-morphisms of Kuranishi neighbourhoods for all $a,b,c\in A$
\begin{equation*}
\Th_{abc}^{\bs g,\bs f}:\Psi'_{bc}\ci\Psi_{ab}\Longra(\bs g\ci\bs f)_{ac}.
\end{equation*}

We claim that there is a unique 2-morphism $\bs\varrho=(\bs\varrho_{ac,\;a,c\in A}):\bs g\ci\bs f\Ra\bs\id_\bX$ of Kuranishi spaces such that for all $a,b,c\in A$ the following diagram of 2-morphisms of Kuranishi neighbourhoods over $\Im\psi_a\cap\Im\psi_b\cap\Im\psi_c$ commutes:
\e
\begin{gathered}
\xymatrix@C=160pt@R=16pt{
*+[r]{\Psi_{bc}'\ci\Psi_{ab}} \ar@{=>}[r]_{(\be_{bc}')^{-1}*\id_{\Psi_{ab}}} \ar@{=>}[d]^{\Th_{abc}^{\bs g,\bs f}} & *+[l]{\Psi_{bc}\ci\Psi_{ab}} \ar@{=>}[d]_{\Ka_{abc}} \\
*+[r]{(\bs g\ci\bs f)_{ac}} \ar@{=>}[r]^{\bs\varrho_{ac}\vert_{\Im\psi_a\cap\Im\psi_b\cap\Im\psi_c}} & *+[l]{\Psi_{ac}=(\bs\id_\bX)_{ac}.\!\!} } 
\end{gathered}
\label{ku7eq100}
\e

To prove this, note that \eq{ku7eq100} determines $\bs\varrho_{ac}$ on the open subset $\Im\psi_a\cap\Im\psi_b\cap\Im\psi_c\subseteq \Im\psi_a\cap\Im\psi_c$. Using \eq{ku4eq18}--\eq{ku4eq20} for the $\Th_{abc}^{\bs g,\bs f}$ and Lemma \ref{ku7lem7} for the $\Ka_{abc}$, we prove that these prescribed values for $\bs\varrho_{ac}$ agree on overlaps between $\Im\psi_a\cap\Im\psi_b\cap\Im\psi_c$ and $\Im\psi_a\cap\Im\psi_{b'}\cap\Im\psi_c$, for all $b,b'\in A$. Thus, as the $\Im\psi_a\cap\Im\psi_b\cap\Im\psi_c$ for all $b\in A$ form an open cover of the correct domain $\Im\psi_a\cap\Im\psi_c$ for $a,c\in A$, Theorem \ref{ku4thm1}(a) and Definition \ref{ku4def11}(iii),(iv) imply that there is a unique 2-morphism $\bs\varrho_{ac}:(\bs g\ci\bs f)_{ac}\Ra(\bs\id_\bX)_{ac}$ such that \eq{ku7eq100} commutes for all~$b\in A$. 

We can then check that $\bs\varrho=(\bs\varrho_{ac,\;a,c\in A})$ satisfies Definition \ref{ku4def14}(a),(b), by proving that they hold on the restriction of their domains with $\Im\psi_b$ for each $b\in A$ using \eq{ku7eq100}, \eq{ku4eq18}--\eq{ku4eq20} for the $\Th_{abc}^{\bs g,\bs f}$ and Lemma \ref{ku7lem7} for the $\Ka_{abc}$, and then using Theorem \ref{ku4thm1}(a) and Definition \ref{ku4def11}(iii) to deduce that Definition \ref{ku4def14}(a),(b) hold on the correct domains. Therefore $\bs\varrho:\bs g\ci\bs f\Ra\bs\id_\bX$ is a 2-morphism of Kuranishi spaces. Similarly, exchanging $\bX,\bX'$ we construct a 2-morphism $\bs\si:\bs f\ci\bs g\Ra\bs\id_{\bX'}$. Hence $\bs f:\bX\ra\bX'$ is an equivalence, and $\bX,\bX'$ are equivalent in the 2-category $\Kur$. This completes the proof of Theorem~\ref{ku4thm9}.

\subsection{Proof of Theorem \ref{ku5thm1}}
\label{ku75}

The changes to \S\ref{ku71} to include corners, proving Theorem \ref{ku5thm1}, are very similar to \S\ref{ku63}. The only nontrivial issue is in the proof of Definition \ref{ku4def11}(v) for $\bcHom((V_i,E_i,\Ga_i,s_i,\psi_i),(V_j,E_j,\Ga_j,s_j,\psi_j))$, when we define $\phi_{ij}$ using Lemma \ref{ku6lem} in the paragraph of \eq{ku7eq12}, and this should be modified as in~\S\ref{ku63}.

\subsection{Proof of Theorem \ref{ku5thm2}}
\label{ku76}

Theorem \ref{ku5thm2} may be proved by modifying the proof of Theorem \ref{ku4thm2} in \S\ref{ku72}, following the method used to prove Theorem \ref{ku3thm3} in \S\ref{ku64} by modifying the proof of Theorem \ref{ku2thm2} in \S\ref{ku62}. We replace $TV_i,TV_j,\d s_i,\d s_j,\d\phi_{ij},\ldots$ by ${}^bTV_i,\ab{}^bTV_j,\ab{}^b\d s_i,\ab{}^b\d s_j,\ab{}^b\d\phi_{ij},\ldots$ throughout, and use Proposition \ref{ku5prop1} in place of Proposition \ref{ku3prop2} in the `only if' part in \S\ref{ku641}. We leave the details to the reader.

\appendix
\section{Previous definitions of Kuranishi space}
\label{kuA}

We discuss the various definitions of Kuranishi space, and of good coordinate system, in the work of Fukaya, Oh, Ohta and Ono \cite{Fuka,FOOO1, FOOO2,FOOO3,FOOO4,FOOO5,FOOO6,FOOO7,FOOO8,FuOn}, McDuff and Wehrheim \cite{McWe2,McWe3}, and Dingyu Yang \cite{Yang1,Yang2,Yang3}. To improve compatibility with \S\ref{ku2}--\S\ref{ku6}, we have made some small changes in notation compared to our sources, without changing the content. We hope the authors concerned will not mind this.

The author \cite{Joyc1,Joyc2} also previously attempted a definition of a category of Kuranishi spaces, but this was unsatisfactory because of problems with the use of germs, and we will not explain it. 

Examples \ref{kuAex1}, \ref{kuAex2}, \ldots\ explain the relationship between the material we explain, and some of the definitions of~\S\ref{ku4}.

\subsection{\texorpdfstring{Fukaya--Oh--Ohta--Ono's Kuranishi spaces}{Fukaya--Oh--Ohta--Ono\textquoteright s Kuranishi spaces}}
\label{kuA1}

Kuranishi spaces are used in the work of Fukaya, Oh, Ohta and Ono \cite{Fuka,FOOO1,FOOO2,FOOO3,FOOO4,FOOO5,FOOO6,FOOO7,FOOO8,FuOn} as the geometric structure on moduli spaces of $J$-holomorphic curves. Initially introduced by Fukaya and Ono \cite[\S 5]{FuOn} in 1999, the definition has changed several times as their work has evolved \cite{Fuka,FOOO1,FOOO2,FOOO3,FOOO4,FOOO5,FOOO6, FOOO7,FOOO8,FuOn}.

This section explains their most recent definition of Kuranishi space, taken from \cite[\S 4]{FOOO6}. As in the rest of our book `Kuranishi neighbourhood', `coordinate change' and `Kuranishi space' have a different meaning, we will use the terms `FOOO Kuranishi neighbourhood', `FOOO coordinate change' and `FOOO Kuranishi space' below to refer to concepts from \cite{FOOO6}.

For the next definitions, let $X$ be a compact, metrizable topological space. 

\begin{dfn} A {\it FOOO Kuranishi neighbourhood\/} on $X$ is a quintuple $(V,\ab E,\ab\Ga,\ab s,\ab\psi)$ such that:
\begin{itemize}
\setlength{\itemsep}{0pt}
\setlength{\parsep}{0pt}
\item[(a)] $V$ is a smooth manifold, which may or may not have boundary or corners.
\item[(b)] $E$ is a finite-dimensional real vector space.
\item[(c)] $\Ga$ is a finite group with a smooth, effective action on $V$, and a linear representation on $E$.
\item[(d)] $s:V\ra E$ is a $\Ga$-equivariant smooth map.
\item[(e)] $\psi$ is a homeomorphism from $s^{-1}(0)/\Ga$ to an open subset $\Im\psi$ in $X$, where $\Im\psi=\bigl\{\psi(x\Ga):x\in s^{-1}(0)\bigr\}$ is the image of $\psi$, and is called the {\it footprint\/} of~$(V,E,\Ga,s,\psi)$.
\end{itemize}
We will write $\bar\psi:s^{-1}(0)\ra\Im\psi\subseteq X$ for the composition of $\psi$ with the projection~$s^{-1}(0)\ra s^{-1}(0)/\Ga$.

Now let $p\in X$. A {\it FOOO Kuranishi neighbourhood of\/ $p$ in\/} $X$ is a FOOO Kuranishi neighbourhood $(V_p,E_p,\Ga_p,s_p,\psi_p)$ with a distinguished point $o_p\in V_p$ such that $o_p$ is fixed by $\Ga_p$, and $s_p(o_p)=0$, and $\psi_p([o_p])=p$. Then $o_p$ is unique.

\label{kuAdef1}
\end{dfn}

\begin{ex} For our Kuranishi neighbourhoods $(V',E',\Ga',s',\psi')$ in Definition \ref{ku4def1}, $\pi':E'\ra V'$ is a $\Ga'$-equivariant vector bundle, and $s':V'\ra E'$ a $\Ga'$-equivariant smooth section. Also $\Ga'$ is not required to act effectively on $V'$.

To make a FOOO Kuranishi neighbourhood $(V,E,\Ga,s,\psi)$ into one of our Kuranishi neighbourhoods $(V',E',\Ga',s',\psi')$, take $V'=V$, $\Ga'=\Ga,$ $\psi'=\psi$, let $\pi':E'\ra V'$ be the trivial vector bundle $\pi_V:V\t E\ra V$ with fibre $E$, and $s'=\id\t s:V\ra V\t E$. Thus, FOOO Kuranishi neighbourhoods correspond to special examples of our Kuranishi neighbourhoods $(V',E',\Ga',s',\psi')$, in which $\pi':E'\ra V'$ is a trivial vector bundle, and $\Ga'$ acts effectively on~$V'$.

By an abuse of notation, we will sometimes identify FOOO Kuranishi neighbourhoods with the corresponding Kuranishi neighbourhoods in \S\ref{ku41}. That is, we will use $E$ to denote both a vector space, and the corresponding trivial vector bundle over $V$, and $s$ to denote both a map, and a section of a trivial bundle. Fukaya et al.\ \cite[Def.~4.3(4)]{FOOO6} also make the same abuse of notation.
\label{kuAex1}
\end{ex}

\begin{dfn} Let $(V_i,E_i,\Ga_i,s_i,\psi_i)$, $(V_j,E_j,\Ga_j,s_j,\psi_j)$ be FOOO Kuranishi neighbourhoods on $X$. Suppose $S\subseteq \Im\psi_i\cap\Im\psi_j\subseteq X$ is an open subset of the intersection of the footprints $\Im\psi_i,\Im\psi_j\subseteq X$. We say a quadruple $\Phi_{ij}=(V_{ij},h_{ij},\vp_{ij},\hat\vp_{ij})$ is a {\it FOOO coordinate change from\/ $(V_i,E_i,\Ga_i,s_i,\psi_i)$ to $(V_j,\ab E_j,\ab\Ga_j,\ab s_j,\ab\psi_j)$ over\/} $S$ if:
\begin{itemize}
\setlength{\itemsep}{0pt}
\setlength{\parsep}{0pt}
\item[(a)] $V_{ij}$ is a $\Ga_i$-invariant open neighbourhood of $\bar\psi_i^{-1}(S)$ in~$V_i$.
\item[(b)] $h_{ij}:\Ga_i\ra\Ga_j$ is an injective group homomorphism.
\item[(c)] $\vp_{ij}:V_{ij}\hookra V_j$ is an $h_{ij}$-equivariant smooth embedding, such that the induced map $(\vp_{ij})_*: V_{ij}/\Ga_i\ra V_j/\Ga_j$ is injective. 
\item[(d)] $\hat\vp_{ij}:V_{ij}\t E_i\hookra V_j\t E_j$ is an $h_{ij}$-equivariant embedding of vector bundles over $\vp_{ij}:V_{ij}\hookra V_j$, viewing $V_{ij}\t E_i\ra V_{ij}$, $V_j\t E_j\ra V_j$ as trivial vector bundles.
\item[(e)] $\hat\vp_{ij}(s_i\vert_{V_{ij}})=\vp_{ij}^*(s_j)$, in sections of $\vp_{ij}^*(V_j\t E_j)\ra V_{ij}$.
\item[(f)] $\psi_i=\psi_j\ci(\vp_{ij})_*$ on $(s_i^{-1}(0)\cap V_{ij})/\Ga_i$.
\item[(g)] $h_{ij}$ restricts to an isomorphism $\Stab_{\Ga_i}(v)\ra \Stab_{\Ga_j}(\vp_{ij}(v))$ for all $v$ in $V_{ij}$, where $\Stab_{\Ga_i}(v)$ is the {\it stabilizer subgroup\/} $\bigl\{\ga\in\Ga_i:\ga(v)=v\bigr\}$.
\item[(h)] For each $v\in s_i^{-1}(0)\cap V_{ij}\subseteq V_{ij}\subseteq V_i$ we have a commutative diagram
\e
\begin{gathered}
\xymatrix@C=25pt@R=13pt{ 0 \ar[r] & T_vV_i \ar[rr]_{\d\vp_{ij}\vert_v} \ar[d]^{\d s_i\vert_v} && T_{\vp_{ij}(v)}V_j \ar[r] \ar[d]^{\d s_j\vert_{\vp_{ij}(v)}} & N_{ij}\vert_v \ar[r] \ar@{.>}[d]^{\d_{\rm fibre}s_j\vert_v} & 0 \\
0 \ar[r] & E_i\vert_v \ar[rr]^{\hat\vp_{ij}\vert_v} && E_j\vert_{\vp_{ij}(v)} \ar[r] & F_{ij}\vert_v \ar[r] & 0 
}
\end{gathered}
\label{kuAeq1}
\e
with exact rows, where $N_{ij}\ra V_{ij}$ is the normal bundle of $V_{ij}$ in $V_j$, and $F_{ij}=\vp_{ij}^*(E_j)/\hat\vp_{ij}(E_i\vert_{V_{ij}})$ the quotient bundle. We require that the induced morphism $\d_{\rm fibre}s_j\vert_v$ in \eq{kuAeq1} should be an isomorphism.
\end{itemize}

Note that $\d_{\rm fibre}s_j\vert_v$ an isomorphism in \eq{kuAeq1} is equivalent to the following complex being exact, which is how we write the analogous conditions in \S\ref{ku2}--\S\ref{ku6}:
\e
\begin{gathered}
\xymatrix@C=14.5pt{ 0 \ar[r] & T_vV_i \ar[rrr]^(0.35){\d s_i\vert_v\op\d\vp_{ij}\vert_v} &&& E_i\vert_v \!\op\!T_{\vp_{ij}(v)}V_j 
\ar[rrr]^(0.56){\hat\vp_{ij}\vert_v\op -\d s_j\vert_{\vp_{ij}(v)}} &&& E_j\vert_{\vp_{ij}(v)} \ar[r] & 0. }
\end{gathered}
\label{kuAeq2}
\e

If we do not specify $S$, but just say that $\Phi_{ij}$ is a {\it FOOO coordinate change from\/ $(V_i,E_i,\Ga_i,s_i,\psi_i)$ to\/} $(V_j,E_j,\Ga_j,s_j,\psi_j)$, we mean that~$S=\Im\psi_i\cap\Im\psi_j$.

Now let $(V_p,E_p,\Ga_p,s_p,\psi_p)$, $(V_q,E_q,\Ga_q,s_q,\psi_q)$ be FOOO Kuranishi neighbourhoods of $p\in X$ and $q\in\Im\psi_p\subseteq X$, respectively. We say a quadruple $\Phi_{qp}=(V_{qp},h_{qp},\vp_{qp},\hat\vp_{qp})$ is a {\it FOOO coordinate change\/} if it is a FOOO coordinate change from $(V_q,E_q,\Ga_q,s_q,\psi_q)$ to $(V_p,E_p,\Ga_p,s_p,\psi_p)$ on $S_{qp}$, where $S_{qp}$ is any open neighbourhood of $q$ in $\Im\psi_q\cap\Im\psi_p$.
\label{kuAdef2}
\end{dfn}

\begin{rem}{\bf(a)} We have changed notation slightly compared to \cite{FOOO6}, to improve compatibility with the rest of the book. Fukaya et al.\ \cite[\S 4]{FOOO6} write Kuranishi neighbourhoods as $(V,E,\Ga,\psi,s)$ rather than $(V,E,\Ga,s,\psi)$. Also, they write coordinate changes as $\Phi_{pq}=(\hat\vp_{pq},\vp_{pq},h_{pq})$, leaving $V_{pq}$ implicit, rather than as $\Phi_{qp}=(V_{qp},h_{qp},\vp_{qp},\hat\vp_{qp})$ as we do. Note that we have changed the order of $p,q$ in the subscripts compared to~\cite{FOOO6}.

Fukaya et al.\ do not require $\hat\vp_{ij}:V_{ij}\t E_i\hookra \vp_{ij}^*(V_j\t E_j)$ to come from an injective linear map of vector spaces $E_i\hookra E_j$. As in \S\ref{kuA3}, McDuff and Wehrheim do require this.

Fukaya et al.\ only impose Definition \ref{kuAdef2}(h) for Kuranishi spaces `with a tangent bundle' in the sense of \cite{FOOO1,FOOO6,FuOn}. As the author knows of no reason for considering Kuranishi spaces `without tangent bundles', and the notation appears to be merely historical, we will include `with a tangent bundle' in our definitions of FOOO coordinate changes and FOOO Kuranishi spaces.
\smallskip

\noindent{\bf(b)} Manifolds with corners were discussed in \S\ref{ku31}--\S\ref{ku33}. When we allow the $V_i$ in Kuranishi neighbourhoods $(V_i,E_i,\Ga_i,s_i,\psi_i)$ to be manifolds with boundary, or manifolds with corners, it is important that the definition of {\it embedding\/} of manifolds with corners $\vp_{ij}:V_{ij}\hookra V_j$ used in Definition \ref{kuAdef2}(c) includes the condition that $\vp_{ij}$ be {\it simple}, in the sense of \S\ref{ku31}. Here are some examples:
\begin{itemize}
\setlength{\itemsep}{0pt}
\setlength{\parsep}{0pt}
\item[(i)] Let $h:[0,\iy)\ra\R$ be smooth. Then $f:[0,\iy)\ra[0,\iy)\t\R$ mapping $f:x\mapsto\bigl(x,h(x)\bigr)$ is simple, and an embedding.
\item[(ii)] The inclusion $i:[0,\iy)\hookra\R$ is not simple, or an embedding.
\item[(ii)] The diagonal map $\De:[0,\iy)\ra[0,\iy)^2$ mapping $\De:x\mapsto(x,x)$ is not simple, or an embedding.
\end{itemize}
\label{kuArem1}
\end{rem}

We relate FOOO coordinate changes to coordinate changes in~\S\ref{ku41}.

\begin{ex} Let $\Phi_{ij}=(V_{ij},h_{ij},\vp_{ij},\hat\vp_{ij}):(V_i,E_i,\Ga_i,s_i,\psi_i)\ra(V_j,\ab E_j,\ab\Ga_j,\ab s_j,\ab\psi_j)$ be a FOOO coordinate change over $S$, as in Definition \ref{kuAdef2}. As in Example \ref{kuAex1}, regard the FOOO Kuranishi neighbourhoods $(V_i,E_i,\Ga_i,s_i,\psi_i),(V_j,\ab E_j,\ab\Ga_j,\ab s_j,\ab\psi_j)$ as examples of Kuranishi neighbourhoods in the sense of \S\ref{ku41}.

Set $P_{ij}=V_{ij}\t\Ga_j$. Let $\Ga_i$ act on $P_{ij}$ by $\ga_i:(v,\ga)\mapsto (\ga_i\cdot v,\ga h_{ij}(\ga_i)^{-1})$. Let $\Ga_j$ act on $P_{ij}$ by $\ga_j:(v,\ga)\mapsto (v,\ga_j\ga)$. Define $\pi_{ij}:P_{ij}\ra V_i$ and $\phi_{ij}:P_{ij}\ra V_j$ by $\pi_{ij}:(v,\ga)\mapsto v$ and $\phi_{ij}:(v,\ga)\mapsto\ga\cdot\vp_{ij}(v)$. Then $\pi_{ij}$ is $\Ga_i$-equivariant and $\Ga_j$-invariant. Since $\vp_{ij}$ is $h_{ij}$-equivariant, $\phi_{ij}$ is $\Ga_i$-invariant, and $\Ga_j$-equivariant.

We will define a vector bundle morphism $\hat\phi_{ij}:\pi_{ij}^*(E_i)\ra\phi_{ij}^*(E_j)$. At $(v,\ga)\in P_{ij}$, this $\hat\phi_{ij}$ must map $E_i\vert_v\ra E_j\vert_{\ga\cdot \vp_{ij}(v)}$. We define $\hat\phi_{ij}\vert_{(v,\ga)}$ to be the composition of $\hat\vp_{ij}\vert_v:E_i\vert_v\ra E_j\vert_{\vp_{ij}(v)}$ with $\ga\cdot :E_j\vert_{\vp_{ij}(v)}\ra E_j\vert_{\ga\cdot\vp_{ij}(v)}$ from the $\Ga_j$-action on $E_j$. That is, $\hat\phi_{ij}\vert_{V_{ij}\t\{\ga\}}=\ga\cdot \hat\vp_{ij}$ for each $\ga\in\Ga_j$.

It is now easy to see that $\ti\Phi_{ij}=(P_{ij},\pi_{ij},\phi_{ij},\hat\phi_{ij}):(V_i,E_i,\Ga_i,s_i,\psi_i)\ra (V_j,\ab E_j,\ab\Ga_j,\ab s_j,\ab\psi_j)$ is a 1-morphism over $S$, in the sense of \S\ref{ku41}. Theorem \ref{ku4thm2} shows that $\ti\Phi_{ij}$ is a coordinate change over $S$, in the sense of~\S\ref{ku41}.
\label{kuAex2}
\end{ex}

\begin{dfn} A {\it FOOO Kuranishi structure\/} $\cK$ on $X$ of {\it virtual dimension\/} $n\in\Z$ in the sense of \cite[\S 4]{FOOO6}, including the `with a tangent bundle' condition, assigns a FOOO Kuranishi neighbourhood $(V_p,E_p,\Ga_p,s_p,\psi_p)$ for each $p\in X$ and a FOOO coordinate change $\Phi_{qp}=(V_{qp},h_{qp},\vp_{qp},\hat\vp_{qp}):(V_q,E_q,\Ga_q,s_q,\psi_q)\ra(V_p,E_p,\Ga_p,s_p,\psi_p)$ for each $q\in \Im\psi_p$ such that the following holds:
\begin{itemize}
\setlength{\itemsep}{0pt}
\setlength{\parsep}{0pt}
\item[(a)] $\dim V_p-\rank E_p=n$ for all $p\in X$.
\item[(b)] If $q\in\Im\psi_p$, $r\in\psi_q((V_{qp}\cap s_q^{-1}(0))/\Ga_q)$, then for each connected component $(\vp_{rq}^{-1}(V_{qp})\cap V_{rp})^\al$ of $\vp_{rq}^{-1}(V_{qp})\cap V_{rp}$ there exists $\ga_{rqp}^\al\in \Ga_p$ with
\e
\begin{gathered}
h_{qp}\ci h_{rq}=\ga_{rqp}^{\al}\cdot h_{rp} \cdot (\ga_{rqp}^{\al})^{-1}, \qquad
\vp_{qp} \ci \vp_{rq} = \ga_{rqp}^{\al}
\cdot \vp_{rp}, \\
\text{and}\qquad \vp_{rq}^*(\hat\vp_{qp}) \ci \hat\vp_{rq} = \ga_{rqp}^{\al}\cdot
\hat\vp_{rp},
\end{gathered}
\label{kuAeq3}
\e
where the second and third equations hold on $(\vp_{rq}^{-1}(V_{qp})\cap V_{rp})^\al$.
\end{itemize}

The pair $\bX=(X,\cK)$ is called a {\it FOOO Kuranishi space}, of {\it virtual dimension\/} $n\in\Z$, written $\vdim\bX=n$.

If the $V_p$ for all $p\in X$ are manifolds without boundary, or with boundary, or with corners, then we call $\bX$ a {\it FOOO Kuranishi space without boundary}, or {\it with boundary}, or {\it with corners}, respectively.
\label{kuAdef3}
\end{dfn}

We prove in Theorem \ref{ku4thm10} above that a FOOO Kuranishi space $\bX$ can be made into a Kuranishi space $\bX'$ in the sense of~\S\ref{ku4}.

We will show that the elements $\ga_{rqp}^\al\in\Ga_p$ in Definition \ref{kuAdef3}(b) correspond in the setting of \S\ref{ku4} to a 2-morphism~$\La_{rqp}:\ti\Phi_{qp}\ci\ti\Phi_{rq}\Ra\ti\Phi_{rp}$.

\begin{ex}{\bf(i)} In the Fukaya--Oh--Ohta--Ono theory \cite{Fuka,FOOO1,FOOO2,FOOO3,FOOO4,FOOO5,FOOO6,FOOO7,FOOO8,FuOn}, one often relates two FOOO coordinate changes in the following way. Let $\Phi_{ij}=(V_{ij},h_{ij},\ab \vp_{ij},\ab\hat\vp_{ij}),\ab\Phi_{ij}'=(V_{ij}',\ab h_{ij}',\ab\vp_{ij}',\hat\vp_{ij}'):(V_i,E_i,\Ga_i,s_i,\psi_i)\ra(V_j,\ab E_j,\ab\Ga_j,\ab s_j,\ab\psi_j)$ be FOOO coordinate changes over $S$. Suppose there exists $\ga\in\Ga_j$ such that
\e
h_{ij}=\ga\cdot h_{ij}'\cdot\ga^{-1}, \quad
\phi_{ij}=\ga\cdot \phi_{ij}',\quad\text{and}\quad 
\hat\phi_{ij}=\ga\cdot\hat\phi_{ij}',
\label{kuAeq4}
\e
where the second and third equations hold on~$\dot V_{ij}:=V_{ij}\cap V_{ij}'$.

Let $\ti\Phi_{ij},\ti\Phi_{ij}':(V_i,E_i,\Ga_i,s_i,\psi_i)\ra (V_j,\ab E_j,\ab\Ga_j,\ab s_j,\ab\psi_j)$ be the 1-morphisms in the sense of \S\ref{ku41} corresponding to $\Phi_{ij},\Phi_{ij}'$ in Example \ref{kuAex2}. Set $\dot P_{ij}=\dot V_{ij}\t\Ga_j\subseteq P_{ij}$. Define $\la_{ij}:\dot P_{ij}=\dot V_{ij}\t\Ga_j\ra V_{ij}'\t\Ga_j=P_{ij}'$ by $\la_{ij}:(v,\ga')\mapsto(v,\ga'\ga)$, and $\hat\la_{ij}=0$. Then $(\dot P_{ij},\la_{ij},\hat\la_{ij})$ satisfies Definition \ref{ku4def3}(a)--(c), so we have defined a 2-morphism $\La_{ij}=[\dot P_{ij},\la_{ij},\hat\la_{ij}]:\ti\Phi_{ij}\Ra\ti\Phi_{ij}'$, in the sense of~\S\ref{ku41}.
\smallskip

\noindent{\bf(ii)} This enables us to interpret Definition \ref{kuAdef3}(b) in terms of a 2-morphism. In the situation of Definition \ref{kuAdef3}(b), the composition of the FOOO coordinate changes $\Phi_{rq},\Phi_{qp}$ is $\Phi_{qp}\ci\Phi_{qp}=\bigl(\vp_{rq}^{-1}(V_{qp}),h_{qp}\ci h_{rq},\vp_{qp} \ci \vp_{rq}\vert_{\vp_{rq}^{-1}(V_{qp})},\vp_{rq}^*(\hat\vp_{qp}) \ci \hat\vp_{rq}\vert_{\vp_{rq}^{-1}(V_{qp})}\bigr)$. Thus, \eq{kuAeq3} relates $\Phi_{qp}\ci\Phi_{rq}$ to $\Phi_{rp}$ in the same way that \eq{kuAeq4} relates $\Phi_{ij}$ to $\Phi_{ij}'$, except for allowing $\ga_{rqp}$ to vary on different connected components. Hence, if $\ti\Phi_{rq},\ti\Phi_{qp},\ti\Phi_{rp}$ are the coordinate changes in the sense of \S\ref{ku41} associated to $\Phi_{rq},\Phi_{qp},\Phi_{rp}$ in Example \ref{kuAex2}, then the method of (i) defines a 2-morphism $\La_{pqr}:\ti\Phi_{qp}\ci\ti\Phi_{rq}\Ra\ti\Phi_{rp}$, in the sense of~\S\ref{ku41}.
\smallskip

\noindent{\bf(iii)} In the situation of Definition \ref{kuAdef3}(b), suppose $v\in (\vp_{rq}^{-1}(V_{qp})\cap V_{rp})^\al$ is generic. Then $\Stab_{\Ga_r}(v)=\{1\}$, as $\Ga_r$ acts (locally) effectively on $V_r$ by Definition \ref{kuAdef1}(c). Hence $\Stab_{\Ga_p}(\vp_{rp}(v))=\{1\}$ by Definition \ref{kuAdef2}(g). Therefore the point $\ga_{rqp}^\al\cdot \vp_{rp}(v)=\vp_{qp}\ci\vp_{rq}(v)$ in $V_p$ determines $\ga_{rqp}^\al$ in $\Ga_p$. So the second equation of \eq{kuAeq3} determines $\ga_{rqp}^\al\in\Ga_p$ uniquely, provided it exists. Thus the 2-morphism $\La_{pqr}:\ti\Phi_{qp}\ci\ti\Phi_{rq}\Ra\ti\Phi_{rp}$ in (ii) is also determined uniquely.
\label{kuAex3}
\end{ex}

\begin{dfn} Let $\bX$ be a FOOO Kuranishi space. Then for each $p\in X$, $q\in\Im\psi_p$ and $v\in s_q^{-1}(0)\cap V_{qp}$, we have an exact sequence \eq{kuAeq2}. Taking top exterior powers in \eq{kuAeq2} yields an isomorphism \begin{equation*}
\bigl(\det T_vV_q\bigr)\ot \det\bigl(E_p\vert_{\vp_{qp}(v)}\bigr)\cong
\bigl(\det E_q\vert_v\bigr)\ot\bigl(T_{\vp_{qp}(v)}V_p\bigr),
\end{equation*}
where $\det W$ means $\La^{\dim W}W$, or equivalently, a canonical isomorphism
\e
\bigl(\det E_p^*\ot \det TV_p\bigr)\vert_{\vp_{qp}(v)}\cong 
\bigl(\det E_q^*\ot \det TV_q\bigr)\vert_v.
\label{kuAeq5}
\e
Defining the isomorphism \eq{kuAeq5} requires a suitable sign convention. Sign conventions are discussed in Fukaya et al. \cite[\S 8.2]{FOOO1} and McDuff and Wehrheim \cite[\S 8.1]{McWe2}. An {\it orientation\/} on $\bX$ is a choice of orientations on the line bundles 
\begin{equation*} 
\det E_p^* \ot \det TV_p\big\vert_{s_p^{-1}(0)} \longra s_p^{-1}(0)
\end{equation*}
for all $p\in X$, compatible with the isomorphisms~\eq{kuAeq5}.
\label{kuAdef5}
\end{dfn}

\begin{dfn} Let $\bX$ be a FOOO Kuranishi space, and $Y$ a manifold. A {\it smooth map\/} $\bs f:\bX\ra Y$ is $\bs f=(f_p:p\in X)$ where $f_p:V_p\ra Y$ is a $\Ga_p$-invariant smooth map for all $p\in X$ (that is, $f_p$ factors via $V_p\ra V_p/\Ga_p\ra Y$), and $f_p\ci\vp_{qp}=f_q\vert_{V_{qp}}:V_{qp}\ra Y$ for all $q\in\Im\psi_p$. This induces a unique continuous map $f:X\ra Y$ with $f_p\vert_{s_p^{-1}(0)}=f\ci\bar\psi_p$ for all $p\in X$. We call $\bs f$ {\it weakly submersive\/} if each $f_p$ is a submersion.

Suppose $\bX^1,\bX^2$ are FOOO Kuranishi spaces, $Y$ is a manifold, and $\bs f^1:\bX^1\ra Y$, $\bs f^2:\bX^2\ra Y$ are weakly submersive. Then as in \cite[\S A1.2]{FOOO1} one can define a `fibre product' Kuranishi space $\bW=\bX^1\t_Y\bX^2$, with topological space $W=\bigl\{(p^1,p^2)\in X^1\t X^2:f^1(p^1)=f^2(p^2)\bigr\}$, and FOOO Kuranishi neighbourhoods $(V_{p^1,p^2},E_{p^1,p^2},\Ga_{p^1,p^2},s_{p^1,p^2},\psi_{p^1,p^2})$ for $(p^1,p^2)\in W$, where $V_{p_1,p_2}=V_{p^1}^1\t_{f_{p^1}^1,Y,f_{p^2}^2}V_{p^2}^2$, $E_{p_1,p_2}=\pi_{V_{p^1}^1}^*(E^1_{p^1})\op \pi_{V_{p^2}^2}^*(E^2_{p^2})$, $\Ga_{p^1,p^2}=\Ga_{p^1}^1\t\Ga_{p^2}^2$, $s_{p_1,p_2}=\pi_{V_{p^1}^1}^*(s^1_{p^1})\op \pi_{V_{p^2}^2}^*(s^2_{p^2})$, and $\psi_{p^1,p^2}=\psi_{p^1}^1\ci(\pi_{V_{p_1}^1})_*\t\psi_{p^2}^2\ci(\pi_{V_{p_2}^2})_*$. The weakly submersive condition ensures that $V_{p_1,p_2}=V_{p^1}^1\t_YV_{p^2}^2$ is well-defined.
\label{kuAdef6}
\end{dfn}

\begin{rem}{\bf(i)} Note that Fukaya et al.\ \cite{Fuka,FOOO1,FOOO2,FOOO3,FOOO4,FOOO5,FOOO6,FOOO7,FOOO8,FuOn} {\it do not define morphisms between Kuranishi spaces}, but only $\bs f:\bX\ra Y$ from Kuranishi spaces $\bX$ to manifolds $Y$. Thus, Kuranishi spaces in \cite{Fuka,FOOO1,FOOO2,FOOO3,FOOO4,FOOO5,FOOO6,FOOO7,FOOO8,FuOn} {\it do not form a category}. 

Observe however that Fukaya \cite[\S 3, \S 5]{Fuka} (see also \cite[\S 4.2]{FOOO5}) works with a forgetful morphism $\mathfrak{forget}:\bs\cM_{l,1}(\be)\ra\bs\cM_{l,0}(\be)$, which is clearly intended to be some kind of morphism of Kuranishi spaces, without defining the concept.
\smallskip

\noindent{\bf(ii)} The `fibre product' $\bX^1\t_Y\bX^2$ in Definition \ref{kuAdef6} {\it is not a fibre product in the sense of category theory}, characterized by a universal property, since Fukaya et al.\ in\cite{Fuka,FOOO1,FOOO2,FOOO3,FOOO4,FOOO5,FOOO6,FOOO7,FOOO8,FuOn} do not have a category of FOOO Kuranishi spaces in which to state such a universal property (though we do, see \S\ref{ku4}). Their `fibre product' is really just an ad hoc construction.
\label{kuArem2}
\end{rem}

\subsection{\texorpdfstring{Fukaya--Oh--Ohta--Ono's good coordinate systems}{Fukaya--Oh--Ohta--Ono\textquoteright s good coordinate systems}}
\label{kuA2}

{\it Good coordinate systems\/} on Kuranishi spaces $\bX$ in the work of Fukaya, Oh, Ohta and Ono \cite{Fuka,FOOO1,FOOO2,FOOO3, FOOO4,FOOO5,FOOO6,FOOO7,FOOO8,FuOn} are an open cover of $\bX$ by FOOO Kuranishi neighbourhoods $(V_i,E_i,\Ga_i,s_i,\psi_i)$ for $i$ in a finite set $I$, with coordinate changes $\Phi_{ij}$ for $i,j\in I$, satisfying extra conditions. They are a tool for constructing virtual cycles for Kuranishi spaces using the method of `perturbation by multisections', and the extra conditions are included to make this virtual cycle construction work.

As with Kuranishi spaces, since its introduction in \cite[Def.~6.1]{FuOn} the definition of good coordinate system has changed several times during the evolution of \cite{Fuka,FOOO1,FOOO2,FOOO3,FOOO4,FOOO5, FOOO6,FOOO7,FOOO8,FuOn}, see in chronological order \cite[Def.~6.1]{FuOn}, \cite[Lem.~A1.11]{FOOO1}, \cite[\S 15]{FOOO2}, and \cite[\S 5]{FOOO6}. Of these, \cite{FuOn,FOOO6} work with Kuranishi neighbourhoods $(\fV_i,\fE_i,s_i,\psi_i)$ where $\fV_i$ is an orbifold (which we do not want to do), and \cite{FOOO1,FOOO2} with Kuranishi neighbourhoods $(V_i,E_i,\Ga_i,s_i,\psi_i)$ with $V_i$ a manifold.

The definition we give below is a hybrid of those in \cite{FOOO1,FOOO2,FOOO6,FOOO7}. Essentially our `FOOO weak good coordinate systems' follow the definitions in \cite{FOOO1,FOOO2}, and our `FOOO good coordinate systems' include extra conditions adapted from \cite{FOOO6,FOOO7}. We show in Theorem \ref{ku4thm11} above that given a FOOO weak good coordinate system on $X$, we can make $X$ into a Kuranishi space $\bX$ in the sense of~\S\ref{ku4}.

\begin{dfn} Let $X$ be a compact, metrizable topological space. A {\it FOOO weak good coordinate system\/ $\bigl((I,\pr),(V_i,E_i,\Ga_i,s_i,\psi_i)_{i\in I},\Phi_{\text{$ij$, $i\pr j$ in $I$}}\bigr)$ on\/ $X$ of virtual dimension\/} $n\in\Z$ consists of a finite indexing set $I$, a partial order $\pr$ on $I$, FOOO Kuranishi neighbourhoods $(V_i,E_i,\Ga_i,s_i,\psi_i)$ for $i\in I$ with $\dim V_i-\rank E_i=n$ and $X=\bigcup_{i\in I}\Im\psi_i$, and FOOO coordinate changes $\Phi_{ij}=(V_{ij},h_{ij},\vp_{ij},\hat\vp_{ij})$ from $(V_i,E_i,\Ga_i,s_i,\psi_i)$ to $(V_j,E_j,\Ga_j,s_j,\psi_j)$ on $S=\Im\psi_i\cap\Im\psi_j$ for all $i,j\in I$ with $i\pr j$ and $\Im\psi_i\cap\Im\psi_j\ne\es$, satisfying the two conditions:
\begin{itemize}
\setlength{\itemsep}{0pt}
\setlength{\parsep}{0pt}
\item[(a)] If $i\ne j\in I$ with $\Im\psi_i\cap\Im\psi_j\ne\es$ then either $i\pr j$ or $j\pr i$.
\item[(b)] If $i\pr j\pr k$ in $I$ with $\Im\psi_i\cap\Im\psi_j\cap\Im\psi_k\ne\es$ then there exists $\ga_{ijk}\in \Ga_k$ such that as in \eq{kuAeq3} we have
\e
\begin{gathered}
h_{jk} \ci h_{ij} = \ga_{ijk}
\cdot h_{ik} \cdot \ga_{ijk}^{-1}, \qquad
\vp_{jk} \ci \vp_{ij} = \ga_{ijk}
\cdot \vp_{ik}, \\
\text{and}\qquad \vp_{ij}^*(\hat\vp_{jk}) \ci \hat\vp_{ij} = \ga_{ijk}\cdot
\hat\vp_{ik},
\end{gathered}
\label{kuAeq6}
\e
where the second and third equations hold on $V_{ij}\cap V_{ik}\cap\vp_{ij}^{-1}(V_{jk})$. The $\ga_{ijk}$ are uniquely determined by \eq{kuAeq6} as in Example~\ref{kuAex3}(iii). 
\end{itemize}

We call $\bigl((I,\pr),(V_i,E_i,\Ga_i,s_i,\psi_i)_{i\in I},\Phi_{\text{$ij$, $i\pr j$ in $I$}}\bigr)$ a {\it FOOO good coordinate system on\/} $X$ if it also satisfies the extra conditions:
\begin{itemize}
\setlength{\itemsep}{0pt}
\setlength{\parsep}{0pt}
\item[(c)] If $i\pr j$ in $I$, $\Im\psi_i\cap\Im\psi_j\ne\es$ then $\psi_i\bigl((V_{ij}\cap s_i^{-1}(0))/\Ga_i\bigr)=\Im\psi_i\cap\Im\psi_j$.
\item[(d)] If $i\pr j$ in $I$ and $\Im\psi_i\cap\Im\psi_j\ne\es$ then $\mathop{\rm inc}\t\vp_{ij}:V_{ij}\ra V_i\t V_j$ is proper, where $\mathop{\rm inc}:V_{ij}\hookra V_i$ is the inclusion.
\item[(e)] If $i\pr j$, $i\pr k$ in $I$ for $j\ne k$ and $\Im\psi_i\cap\Im\psi_j\ne\es\ne\Im\psi_i\cap\Im\psi_k$, $V_{ij}\cap V_{ik}\ne\es$, then $\Im\psi_j\cap\Im\psi_k\ne\es$, and {\bf either} $j\pr k$ and $V_{ij}\cap V_{ik}=\vp_{ij}^{-1}(V_{jk})$, {\bf or} $k\pr j$ and $V_{ij}\cap V_{ik}=\vp_{ik}^{-1}(V_{kj})$.
\item[(f)] If $i\pr k$, $j\pr k$ in $I$ for $i\ne j$ and $\Im\psi_i\cap\Im\psi_k\ne\es\ne\Im\psi_j\cap\Im\psi_k$ and $v_i\in V_{ik}$, $v_j\in V_{jk}$,
$\de\in\Ga_k$ with $\vp_{jk}(v_j)=\de\cdot \vp_{ik}(v_i)$ in
$V_k$, then $\Im\psi_i\cap\Im\psi_j\ne\es$ and {\bf either} $i\pr j$, $v_i\in V_{ij}$, and there exists $\ga\in\Ga_j$ with $h_{jk}(\ga)=\de\,\ga_{ijk}$ and $v_j=\ga\cdot\vp_{ij}(v_i)$; {\bf or} $j\pr i$, $v_j\in V_{ji}$, and there exists $\ga\in\Ga_i$ with $h_{ik}(\ga)=\de^{-1}\,\ga_{jik}$ and $v_i=\ga\cdot\vp_{ji}(v_j)$, for $\ga_{ijk},\ga_{jik}$ as in~(b).
\end{itemize}

As in \cite{FOOO7}, parts (c)--(f) are equivalent to:
\begin{itemize}
\setlength{\itemsep}{0pt}
\setlength{\parsep}{0pt}
\item[(g)] Define a symmetric, reflexive binary relation $\sim$ on $\coprod_{i\in I}V_i/\Ga_i$ by $\Ga_i v\sim \Ga_j\vp_{ij}(v_i)$ if $i\pr j$, $\Im\psi_i\cap\Im\psi_j\ne\es$ and $v\in V_{ij}$. Then $\sim$ is an equivalence relation, and $\bigl(\coprod_{i\in I}V_i/\Ga_i\bigr)/\sim$ with the quotient topology is Hausdorff.
\end{itemize}

Now let $X$, $\bigl((I,\pr),(V_i,E_i,\Ga_i,s_i,\psi_i)_{i\in I},\Phi_{ij,\; i\pr j}\bigr)$ be as above (either weak or not), and $Y$ be a manifold. As in Definition \ref{kuAdef6}, a {\it smooth map $(f_i,$ $i\in I)$ from $X,$ $\bigl((I,\pr),(V_i,E_i,\Ga_i,s_i,\psi_i)_{i\in I},\Phi_{ij,\; i\pr j}\bigr)$ to\/} $Y$ is a $\Ga_i$-invariant smooth map $f_i:V_i\ra Y$ for $i\in I$, with $f_j\ci\vp_{ij}=f_i\vert_{V_{ij}}:V_{ij}\ra Y$ for all $i\pr j$ in $I$. This induces a unique continuous map $f:X\ra Y$ with $f_i\vert_{s_i^{-1}(0)}=f\ci\bar\psi_i$ for~$i\in I$.

If the $V_i$ for all $i\in I$ are manifolds without boundary, or with boundary, or with corners, then we call $\bigl((I,\pr),(V_i,E_i,\Ga_i,s_i,\psi_i)_{i\in I},\Phi_{ij,\; i\pr j}\bigr)$ a {\it FOOO\/} ({\it weak\/}) {\it good coordinate system without boundary}, or {\it with boundary}, or {\it with corners}, respectively.
\label{kuAdef7}
\end{dfn}

Using elementary topology, Fukaya, Oh, Ohta and Ono \cite{FOOO7} prove:

\begin{thm} Let\/ $\bigl((I,\pr),(V_i,E_i,\Ga_i,s_i,\psi_i)_{i\in I},\Phi_{\text{$ij$, $i\pr j$ in $I$}}\bigr)$ be a FOOO weak good coordinate system on $X$. Then we can construct a FOOO good coordinate system $\bigl((I',\pr),(V_i',E_i',\Ga_i',s_i',\psi_i')_{i\in I},\Phi'_{\text{$ij$, $i\pr j$ in $I$}}\bigr)$ on $X,$ where $I'\subseteq I,$ $V_i'\subseteq V_i,$ $V_{ij}'\subseteq V_{ij}$ are open, $\Ga_i'=\Ga_i,$ $h_{ij}'=h_{ij},$ and\/ $E_i',s_i',\psi_i',\vp_{ij}',\hat\vp_{ij}'$ are obtained from $E_i,\ldots,\hat\vp_{ij}$ by restricting from $V_i,V_{ij}$ to~$V_i',V_{ij}'$.
\label{kuAthm1}
\end{thm}

In fact Fukaya et al.\ \cite{FOOO7} work at the level of orbifolds $V_i/\Ga_i$, $V_{ij}/\Ga_i$ rather than manifolds with finite group actions, but their result easily implies Theorem \ref{kuAthm1}. The next definition is based on Fukaya et al.\ \cite[Def.~7.2]{FOOO6}, but using $(V_i,\ab E_i,\ab\Ga_i,\ab s_i,\ab\psi_i)$ for $V_i$ a manifold, rather than $(\fV_i,\fE_i,{\mathfrak s}_i,\psi_i)$ for $\fV_i$ an orbifold.

\begin{dfn} Let $\bX\!=\!(X,\cK)$ be a FOOO Kuranishi space. A FOOO (weak) good coordinate system $\bigl((I,\pr),(V_i,E_i,\Ga_i,s_i,\psi_i)_{i\in I},\Phi_{\text{$ij$, $i\pr j$ in $I$}}\bigr)$ on the topological space $X$ is called {\it compatible with the FOOO Kuranishi structure $\cK$ on\/} $\bX$ if for each $i\in I$ and each $p\in\Im\psi_i\subseteq X$ there exists a FOOO coordinate change $\Phi_{pi}$ from $(V_p,E_p,\Ga_p,s_p,\psi_p)$ to $(V_i,E_i,\Ga_i,s_i,\psi_i)$ on an open neighbourhood $S_{pi}$ of $p$ in $\Im\psi_p\cap\Im\psi_i$ (where $(V_p,E_p,\Ga_p,s_p,\psi_p)$ comes from $\cK$ and $(V_i,E_i,\Ga_i,s_i,\psi_i)$ from the good coordinate system) such that
\begin{itemize}
\setlength{\itemsep}{0pt}
\setlength{\parsep}{0pt}
\item[(a)] If $q\in\Im\psi_p\cap\Im\psi_i$ then there exists $\ga_{qpi}\in \Ga_i$ such that 
\begin{gather*}
h_{pi} \ci h_{qp} = \ga_{qpi}
\cdot h_{qi} \cdot \ga_{qpi}^{-1}, \qquad
\vp_{pi} \ci \vp_{qp} = \ga_{qpi}
\cdot \vp_{qi}, \\
\text{and}\qquad \vp_{qp}^*(\hat\vp_{pi}) \ci \hat\vp_{qp} = \ga_{qpi}\cdot
\hat\vp_{qi},
\end{gather*}
where the second and third equations hold on~$\vp_{qp}^{-1}(V_{pi})\cap V_{qp} \cap V_{qi}$.
\item[(b)] If $i\pr j$ in $I$ with $p\in\Im\psi_i\cap\Im\psi_j$ then there exists $\ga_{pij}\in \Ga_j$ such that
\e
\begin{gathered}
h_{ij} \ci h_{pi} = \ga_{pij}
\cdot h_{pj} \cdot \ga_{pij}^{-1}, \qquad
\vp_{ij} \ci \vp_{pi} = \ga_{pij}
\cdot \vp_{pj}, \\
\text{and}\qquad \vp_{pi}^*(\hat\vp_{ij}) \ci \hat\vp_{pi} = \ga_{pij}\cdot
\hat\vp_{pj},
\end{gathered}
\label{kuAeq7}
\e
where the second and third equations hold on~$\vp_{pi}^{-1}(V_{ij})\cap V_{pi} \cap V_{pj}$.
\end{itemize}
\label{kuAdef8}
\end{dfn}

\begin{rem} For the programme of\cite{Fuka,FOOO1,FOOO2,FOOO3,FOOO4,FOOO5,FOOO6,FOOO7,FOOO8,FuOn}, one would like to show:
\begin{itemize}
\setlength{\itemsep}{0pt}
\setlength{\parsep}{0pt}
\item[(i)] Any (oriented) FOOO Kuranishi space $\bX$ (perhaps also with a smooth map $\bs f:\bX\ra Y$ to a manifold $Y$) admits a compatible (oriented) FOOO good coordinate system $\bigl((I,\pr),(V_i,E_i,\Ga_i,s_i,\psi_i)_{i\in I},\Phi_{\text{$ij$, $i\pr j$ in $I$}}\bigr)$ (perhaps also with a smooth map $(f_i$, $i\in I)$ to~$Y$).
\item[(ii)] Given a compact, metrizable topological space $X$ with an oriented FOOO good coordinate system $\bigl((I,\pr),(V_i,E_i,\Ga_i,s_i,\psi_i)_{i\in I},\Phi_{\text{$ij$, $i\pr j$ in $I$}}\bigr)$ (perhaps with a smooth map $(f_i$, $i\in I)$ to a manifold $Y$), we can construct a {\it virtual cycle\/} for $X$ (perhaps in the singular homology $H_*(Y;\Q)$ or de Rham cohomology $H^*_{\rm dR}(Y;\R)$ of~$Y$).
\end{itemize}
Producing such virtual cycles is, from the point of view of symplectic geometry, the sole reason for defining and studying Kuranishi spaces.

Statements (i), for various definitions of `Kuranishi space', `good coordinate system', and `compatible', can be found in \cite[Lem.~6.3]{FuOn} (with short proof), \cite[Lem.~A1.11]{FOOO1} (with no proof), and \cite[\S 7]{FOOO6} (with long proof). Constructions (ii), again for various definitions, can be found in \cite[\S 6]{FuOn}, \cite[\S A1.1]{FOOO1}, \cite[\S 12]{FOOO3} (using de Rham cohomology), and \cite[\S 6]{FOOO6} (with long proof).
\label{kuArem3}
\end{rem}

\subsection{\texorpdfstring{McDuff--Wehrheim's Kuranishi atlases}{McDuff--Wehrheim\textquoteright s Kuranishi atlases}}
\label{kuA3}

Next we discuss an approach to Kuranishi spaces developed by McDuff and Wehrheim \cite{McWe1,McWe2,McWe3}. Their main definition is that of a ({\it weak\/}) {\it Kuranishi atlas\/} on a topological space $X$. Here are~\cite[Def.s 2.2.2 \& 2.2.8]{McWe3}.

\begin{dfn} An {\it MW Kuranishi neighbourhood\/} $(V,E,\Ga,s,\psi)$ on a topological space $X$ is the same as a FOOO Kuranishi neighbourhood in Definition \ref{kuAdef1}, except that $\Ga$ is not required to act effectively on~$V$.

As in Example \ref{kuAex1}, by an abuse of notation we will regard MW Kuranishi neighbourhoods as examples of our Kuranishi neighbourhoods in~\S\ref{ku41}.
\label{kuAdef9}
\end{dfn}

\begin{dfn} Suppose $(V_B,E_B,\Ga_B,s_B,\psi_B)$, $(V_C,E_C,\Ga_C,s_C,\psi_C)$ are MW Kuranishi neighbourhoods on a topological space $X$, and $S\subseteq \Im\psi_B\cap\Im\psi_C\subseteq X$ is open. We say a quadruple $\Phi_{BC}=(\ti V_{BC},\rho_{BC},\varpi_{BC},\hat\vp_{BC})$ is an {\it MW coordinate change from\/ $(V_B,E_B,\Ga_B,s_B,\psi_B)$ to $(V_C,\ab E_C,\ab\Ga_C,\ab s_C,\ab\psi_C)$ over\/} $S$ if:
\begin{itemize}
\setlength{\itemsep}{0pt}
\setlength{\parsep}{0pt}
\item[(a)] $\ti V_{BC}$ is a $\Ga_C$-invariant embedded submanifold of $V_C$ containing $\bar\psi_C^{-1}(S)$.
\item[(b)] $\rho_{BC}:\Ga_C\ra\Ga_B$ is a surjective group morphism, with kernel $\De_{BC}\subseteq\Ga_C$.

There should exist an isomorphism $\Ga_C\cong\Ga_B\t\De_{BC}$ identifying $\rho_{BC}$ with the projection $\Ga_B\t\De_{BC}\ra\Ga_B$.
\item[(c)] $\varpi_{BC}:\ti V_{BC}\ra V_B$ is a $\rho_{BC}$-equivariant \'etale map, with image $V_{BC}=\varpi_{BC}(\ti V_{BC})$ a $\Ga_B$-invariant open neighbourhood of $\bar\psi_B^{-1}(S)$ in $V_B$, such that $\varpi_{BC}:\ti V_{BC}\ra V_{BC}$ is a principal $\De_{BC}$-bundle.
\item[(d)] $\hat\vp_{BC}:E_B\ra E_C$ is an injective $\Ga_C$-equivariant linear map, where the $\Ga_C$-action on $E_B$ is induced from the $\Ga_B$-action by $\rho_{BC}$, so in particular $\De_{BC}$ acts trivially on~$E_B$.
\item[(e)] $\hat\vp_{BC}\ci s_B\ci\varpi_{BC}=s_C\vert_{\ti V_{BC}}:\ti V_{BC}\ra E_C$.
\item[(f)] $\psi_B\ci(\varpi_{BC})_*=\psi_C$ on $(s_C^{-1}(0)\cap\ti V_{BC})/\Ga_C$.
\item[(g)] For each $v\in\ti V_{BC}$ we have a commutative diagram
\e
\begin{gathered}
\xymatrix@C=23pt@R=13pt{ 0 \ar[r] & T_v\ti V_{BC} \ar[rr]_(0.6){\subset} \ar[d]^{\d(\varpi_{BC}^*(s_B))\vert_v} && T_vV_C \ar[r] \ar[d]^{\d s_C\vert_v} & N_{BC}\vert_v \ar[r] \ar@{.>}[d]^{\d_{\rm fibre}s_C\vert_v} & 0 \\
0 \ar[r] & E_B \ar[rr]^(0.65){\hat\vp_{BC}} && E_C \ar[r] & E_C/\hat\vp_{BC}(E_B) \ar[r] & 0 }
\end{gathered}
\label{kuAeq8}
\e
with exact rows, where $N_{BC}$ is the normal bundle of $\ti V_{BC}$ in $V_C$. We require the induced morphism $\d_{\rm fibre}s_C\vert_v$ in \eq{kuAeq8} to be an isomorphism.
\end{itemize}
\label{kuAdef10}
\end{dfn}

We relate MW coordinate changes to coordinate changes in~\S\ref{ku41}.

\begin{ex} Let $\Phi_{BC}=(\ti V_{BC},\rho_{BC},\varpi_{BC},\hat\vp_{BC}):(V_B,E_B,\Ga_B,s_B,\psi_B)\ra(V_C,E_C,\Ga_C,s_C,\psi_C)$ be an MW coordinate change over $S$, as in Definition \ref{kuAdef10}. Regard $(V_B,E_B,\Ga_B,s_B,\psi_B),(V_C,E_C,\Ga_C,s_C,\psi_C)$ as Kuranishi neighbourhoods in the sense of \S\ref{ku41}, as in Example \ref{kuAex1}. 

Set $P_{BC}=\ti V_{BC}\t\Ga_B$. Let $\Ga_B$ act on $P_{BC}$ by $\ga_B:(v,\ga)\mapsto (v,\ga_B\ga)$. Let $\Ga_C$ act on $P_{BC}$ by $\ga_C:(v,\ga)\mapsto (\ga_C\cdot v,\ga\rho_{BC}(\ga_C)^{-1})$. Define $\pi_{BC}:P_{BC}\ra V_B$ and $\phi_{BC}:P_{BC}\ra V_C$ by $\pi_{BC}:(v,\ga)\mapsto\ga\cdot\varpi_{BC}(v)$ and $\phi_{BC}:(v,\ga)\mapsto v$. Then $\pi_{BC}$ is $\Ga_B$-equivariant and $\Ga_C$-invariant, and $\phi_{BC}$ is $\Ga_B$-invariant and $\Ga_C$-equivariant.

Define $\hat\phi_{BC}:\pi_{BC}^*(V_B\t E_B)\ra\phi_{BC}^*(V_C\t E_C)$, as a morphism of trivial vector bundles with fibres $E_B,E_C$ on $P_{BC}=\ti V_{BC}\t\Ga_B$, by $\hat\phi_{BC}\vert_{\ti V_{BC}\t\{\ga\}}=\hat\vp_{BC}\ci(\ga^{-1}\cdot -)$ for each $\ga\in\Ga_B$. It is easy to see that $\ti\Phi_{BC}=(P_{BC},\pi_{BC},\phi_{BC},\hat\phi_{BC}):(V_B,E_B,\Ga_B,s_B,\psi_B)\ra (V_C,\ab E_C,\ab\Ga_C,\ab s_C,\ab\psi_C)$ is a 1-morphism over $S$, in the sense of \S\ref{ku41}. Combining Definition \ref{kuAdef10}(g) and Theorem \ref{ku4thm2} shows that $\ti\Phi_{BC}$ is a coordinate change over $S$, in the sense of~\S\ref{ku41}.
\label{kuAex4}
\end{ex}
 
\begin{dfn} Let $X$ be a compact, metrizable topological space. An {\it MW weak Kuranishi atlas\/ $\cK=\bigl(A,I,(V_B,E_B,\Ga_B,s_B,\psi_B)_{B\in I},\Phi_{BC,\;B,C\in I,\; B\subsetneq C}\bigr)$ on\/ $X$ of virtual dimension\/} $n\in\Z$, as in \cite[Def.~2.3.1]{McWe3}, consists of a finite indexing set $A$, a set $I$ of nonempty subsets of $A$, MW Kuranishi neighbourhoods $(V_B,E_B,\Ga_B,s_B,\psi_B)$ on $X$ for all $B\in I$ with $\dim V_B-\rank E_B=n$ and $X=\bigcup_{B\in I}\Im\psi_B$, and MW coordinate changes $\Phi_{BC}=(\ti V_{BC},\rho_{BC},\varpi_{BC},\hat\vp_{BC})$ from $(V_B,\ab E_B,\ab\Ga_B,\ab s_B,\ab\psi_B)$ to $(V_C,E_C,\Ga_C,s_C,\psi_C)$ on $S=\Im\psi_B\cap\Im\psi_C$ for all $B,C\in I$ with $B\subsetneq C$, satisfying the four conditions:
\begin{itemize}
\setlength{\itemsep}{0pt}
\setlength{\parsep}{0pt}
\item[(a)] We have $\{a\}\in I$ for all $a\in A$, and $I=\bigl\{\es\ne B\subseteq A: \bigcap_{a\in B}\Im\psi_{\{a\}}\ne\es\bigr\}$. Also
$\Im\psi_B=\bigcap_{a\in B}\Im\psi_{\{a\}}$ for all $B\in I$.
\item[(b)] We have $\Ga_B=\prod_{a\in B}\Ga_{\{a\}}$ for all $B\in I$. If $B,C\in I$ with $B\subsetneq C$ then $\rho_{BC}:\Ga_C\ra\Ga_B$ is the obvious projection $\prod_{a\in C}\Ga_{\{a\}}\ra \prod_{a\in B}\Ga_{\{a\}}$, with kernel $\De_{BC}\cong \prod_{a\in C\sm B}\Ga_{\{a\}}$.
\item[(c)] We have $E_B=\prod_{a\in B}E_{\{a\}}$ for all $B\in I$, with the obvious representation of $\Ga_B=\prod_{a\in B}\Ga_{\{a\}}$. If $B\subsetneq C$ in $I$ then $\hat\vp_{BC}:E_B=\prod_{a\in B}E_{\{a\}}\ra E_C=\prod_{a\in C}E_{\{a\}}$ is $\id_{E_{\{a\}}}$ for $a\in B$, and maps to zero in $E_{\{a\}}$ for $a\in C\sm B$.
\item[(d)] If $B,C,D\in I$ with $B\subsetneq C\subsetneq D$ then $\varpi_{BC}\ci\varpi_{CD}=\varpi_{BD}$ on $\ti V_{BCD}:=\ti V_{BD}\cap \varpi_{CD}^{-1}(\ti V_{BC})$. One can show using (b),(c) and Definition \ref{kuAdef10} that $\ti V_{BD}$ and $\varpi_{CD}^{-1}(\ti V_{BC})$ are both open subsets in $s_D^{-1}(\hat\vp_{BD}(E_B))$, which is a submanifold of $V_D$, so $\ti V_{BCD}$ is a submanifold of~$V_D$.
\end{itemize}

We call $\cK=\bigl(A,I,(V_B,E_B,\Ga_B,s_B,\psi_B)_{B\in I},\Phi_{BC,\;B\subsetneq C}\bigr)$ an {\it MW Kuranishi atlas on\/} $X$, as in \cite[Def.~2.3.1]{McWe3}, if it also satisfies:
\begin{itemize}
\setlength{\itemsep}{0pt}
\setlength{\parsep}{0pt}
\item[(e)] If $B,C,D\in I$ with $B\subsetneq C\subsetneq D$ then $\varpi_{CD}^{-1}(\ti V_{BC})\subseteq\ti V_{BD}$.
\end{itemize}

McDuff and Wehrheim also define {\it orientations\/} on MW weak Kuranishi atlases, in a very similar way to Definition~\ref{kuAdef5}.

Two MW weak Kuranishi atlases $\cK,\cK'$ on $X$ are called {\it directly commensurate\/} if they are both contained in a third MW weak Kuranishi atlas $\cK''$. They are called {\it commensurate\/} if there exist MW weak Kuranishi atlases $\cK=\cK_0,\cK_1,\ldots,\cK_m=\cK'$ with $\cK_{i-1},\cK_i$ directly commensurate for $i=1,\ldots,m$. This is an equivalence relation on MW weak Kuranishi atlases on~$X$.
\label{kuAdef11}
\end{dfn}

We show in Theorem \ref{ku4thm12} above that given an MW weak Kuranishi atlas on $X$, we can make $X$ into a Kuranishi space $\bX$ in the sense of~\S\ref{ku4}.

McDuff and Wehrheim argue that their concept of MW weak Kuranishi atlas is a more natural, or more basic, idea than a FOOO Kuranishi space, since in analytic moduli problems such as $J$-holomorphic curve moduli spaces, one has to construct an MW weak Kuranishi atlas (or something close to it) first, and then define the FOOO Kuranishi structure using this.

When one constructs an MW weak Kuranishi atlas $\cK$ on a moduli space of $J$-holomorphic curves $\oM$, the construction involves many arbitrary choices, but McDuff and Wehrheim expect different choices $\cK,\cK'$ to be commensurate. They prove this \cite[Rem.~6.2.2]{McWe2} for their definition of MW weak Kuranishi atlases on moduli spaces of nonsingular genus zero Gromov--Witten curves in~\cite[\S 4.3]{McWe2}.

We relate Definition \ref{kuAdef11}(d) to 2-morphisms in~\S\ref{ku41}.

\begin{ex} In the situation of Definition \ref{kuAdef11}(d), let $\ti\Phi_{BC},\ti\Phi_{BD},\ti\Phi_{CD}$ be the coordinate changes in the sense of \S\ref{ku41} associated to the MW coordinate changes $\Phi_{BC},\Phi_{BD},\Phi_{CD}$ in Example \ref{kuAex4}. The composition coordinate change $\ti\Phi_{CD}\ci\ti\Phi_{BC}=(P_{BCD},\pi_{BCD},\phi_{BCD},\hat\phi_{BCD})$ from Definition \ref{ku4def4} has
\e
\begin{split}
P_{BCD}&=\bigl[(\ti V_{BC}\t\Ga_B)\t_{V_C}(\ti V_{CD}\t\Ga_C)\bigr]\big/\Ga_C\\
&\cong (\ti V_{BC}\t_{V_C}\ti V_{CD})\t\Ga_B \cong \varpi_{CD}^{-1}(\ti V_{BC})\t\Ga_B.
\end{split}
\label{kuAeq9}
\e
Define $\dot P_{BCD}$ to be the open subset of $P_{BCD}$ identified with $\ti V_{BCD}\t\Ga_B$ by \eq{kuAeq9}, and $\la_{BCD}:\dot P_{BCD}\ra P_{BD}=\ti V_{BD}\t\Ga_B$ to be the map identified by \eq{kuAeq9} with the inclusion $\ti V_{BCD}\t\Ga_B\hookra \ti V_{BD}\t\Ga_B$, and $\hat\la_{BCD}=0$. Then as in Example \ref{kuAex3}(i), we can show that $(\dot P_{BCD},\la_{BCD},\hat\la_{BCD})$ satisfies Definition \ref{ku4def3}(a)--(c), so we have defined a 2-morphism $\La_{BCD}=[\dot P_{BCD},\la_{BCD},\hat\la_{BCD}]:\ti\Phi_{CD}\ci\ti\Phi_{BC}\Ra\ti\Phi_{BD}$ on $S_{BCD}=\Im\psi_B\cap\Im\psi_C\cap\Im\psi_D$, in the sense of~\S\ref{ku41}.
\label{kuAex5}
\end{ex}

McDuff and Wehrheim prove \cite[Th.~B]{McWe2}, \cite[Th.~A]{McWe3}:

\begin{thm} Let\/ $\cK=\bigl(A,I,(V_B,E_B,\Ga_B,s_B,\psi_B)_{B\in I},\Phi_{BC,\;B,C\in I,\; B\subsetneq C}\bigr)$ be an oriented MW weak Kuranishi atlas without boundary of dimension $n$ on a compact, metrizable topological space $X$. Then $\cK$ determines:
\begin{itemize}
\setlength{\itemsep}{0pt}
\setlength{\parsep}{0pt}
\item[{\bf(a)}] A \begin{bfseries}virtual moduli cycle\end{bfseries} $[X]_{\rm vmc}$ in the cobordism group $\Om_n^{\rm SO,\Q}$ of compact, oriented, $n$-dimensional `$\,\Q$-weighted manifolds' in the sense of\/~{\rm\cite[\S A]{McWe3}}.
\item[{\bf(b)}] A \begin{bfseries}virtual fundamental class\end{bfseries} $[X]_{\rm vfc}$ in $\check H_n(X;\Q),$ where $\check H_*(-;\Q)$ is \v Cech homology over\/~$\Q$.
\end{itemize}

Any two commensurate MW weak Kuranishi atlases $\cK,\cK'$ on $X$ yield the same virtual moduli cycle and virtual fundamental class.

If\/ $\cK$ has trivial isotropy (that is, $\Ga_B=\{1\}$ for all\/ $B\in I$) then we may instead take $[X]_{\rm vmc}\in\Om_n^{\rm SO},$ where $\Om_*^{\rm SO}$ is the usual oriented cobordism group, and\/ $[X]_{\rm vfc}\in H_n^{\rm St}(X;\Z),$ where $H_*^{\rm St}(-;\Z)$ is Steenrod homology over\/~$\Z$.
\label{kuAthm2}
\end{thm}

In part (a), the author expects that $\Om_n^{\rm SO,\Q}\cong\Om_n^{\rm SO}\ot_\Z\Q$, so that $\Om_*^{\rm SO,\Q}\cong\Q[x_4,x_8,\ldots]$ by results of Thom.

Theorem \ref{kuAthm2} is McDuff and Wehrheim's solution to the issues discussed in Remark \ref{kuArem3}. As an intermediate step in the proof of Theorem \ref{kuAthm2}, they pass to a Kuranishi atlas with better properties (a `reduction' of a `tame, metrizable' Kuranishi atlas), which is similar to a FOOO good coordinate system.

\subsection{\texorpdfstring{Dingyu Yang's Kuranishi structures, and polyfolds}{Dingyu Yang\textquoteright s Kuranishi structures, and polyfolds}}
\label{kuA4}

As part of a project to define a truncation functor from polyfolds to Kuranishi spaces, Dingyu Yang \cite{Yang1,Yang2,Yang3} writes down his own theory of Kuranishi spaces:

\begin{dfn} Let $X$ be a compact, metrizable topological space. A {\it DY Kuranishi structure\/} $\cK$ on $X$ is a FOOO Kuranishi structure in the sense of Definition \ref{kuAdef3}, satisfying the additional conditions \cite[Def.~1.11]{Yang2}:
\begin{itemize}
\setlength{\itemsep}{0pt}
\setlength{\parsep}{0pt}
\item[(a)] the {\it maximality condition}, which is essentially Definition \ref{kuAdef7}(e),(f), but replacing $i\pr j$ by~$q\in\Im\psi_p$. 
\item[(b)] the {\it topological matching condition}, which is related to Definition \ref{kuAdef7}(d), but replacing $i\pr j$ by~$q\in\Im\psi_p$.
\end{itemize}
There are a few other small differences --- for instance, Yang does not require vector bundles $E_p$ in $(V_p,E_p,\Ga_p,s_p,\psi_p)$ to be trivial.
\label{kuAdef12}
\end{dfn}

We show in Theorem \ref{ku4thm13} above that given a DY Kuranishi structure on $X$, we can make $X$ into a Kuranishi space $\bX$ in the sense of~\S\ref{ku4}.

Yang also defines his own notion of DY good coordinate system \cite[Def.~2.4]{Yang2}, which is almost the same as a FOOO good coordinate system in \S\ref{kuA2}.

One reason for these modifications is that it simplifies the passage from Kuranishi spaces to good coordinate systems, as in Remark \ref{kuArem3}(i): Yang shows \cite[Th.~2.10]{Yang2} that given any DY Kuranishi space $\bX$, one can construct a DY good coordinate system $\bigl((I,\pr),(V_p',E_p',\Ga_p',s_p',\psi_p')_{p\in I},\Phi'_{\text{$qp$, $q\pr p$ in $I$}}\bigr)$ in which $I\subseteq X$ is a finite subset, $V_p'\subseteq V_p$ is a $\Ga_p$-invariant open subset, $\Ga_p'=\Ga_p$, and $E_p',s_p',\psi_p'$ are the restrictions of $E_p,s_p,\psi_p$ to $V_p'$ for each $p\in I$, and the coordinate changes $\Phi_{qp}'$ for $q\pr p$ are obtained either by restricting $\Phi_{qp}$ to an open $V_{qp}'\subseteq V_{qp}$ if $q\in\Im\psi_p$, or in a more complicated way otherwise.

The next definition comes from Yang \cite[\S 1.6]{Yang1}, \cite[\S 5]{Yang2}, \cite[\S 2.4]{Yang3}.

\begin{dfn} Let $\cK,\cK'$ be DY Kuranishi structures on a compact topological space $X$. An {\it embedding\/} $\ep:\cK\hookra\cK'$ is a choice of FOOO coordinate change $\ep_p:(V_p,E_p,\Ga_p,s_p,\psi_p)\ra(V_p',E_p',\Ga_p',s_p',\psi_p')$ with domain $V_p$ for all $p\in X$, commuting with the FOOO coordinate changes $\Phi_{qp},\Phi'_{qp}$ in $\cK,\cK'$ up to elements of $G_p'$. An embedding is a {\it chart refinement\/} if the $\ep_p$ come from inclusions of $G_p$-invariant open sets~$V_p\hookra V_p'$.

DY Kuranishi structures $\cK,\cK'$ on $X$ are called {\it R-equivalent\/} (or {\it equivalent\/}) if there is a diagram of DY Kuranishi structures on $X$
\begin{equation*}
\xymatrix@C=30pt@R=15pt{\cK  & \cK_1 \ar[l]_\sim \ar@{=>}[r] & \cK_2 & \cK_3 \ar@{=>}[l] \ar[r]^\sim & \cK', }
\end{equation*}
where arrows $\Longra$ are embeddings, and ${\buildrel\sim\over\longra}$ are chart refinements. Using facts about existence of good coordinate systems, Yang proves \cite[Th.~1.6.17]{Yang1}, \cite[\S 11.2]{Yang2} that R-equivalence is an equivalence relation on DY Kuranishi structures.

\label{kuAdef13}
\end{dfn}

Yang emphasizes the idea, which he calls {\it choice independence}, that when one constructs a (DY) Kuranishi structure $\cK$ on a moduli space $\oM$, it should be independent of choices up to R-equivalence. 

One major goal of Yang's work is to relate the Kuranishi space theory of Fukaya, Oh, Ohta and Ono \cite{Fuka,FOOO1,FOOO2,FOOO3,FOOO4,FOOO5,FOOO6,FOOO7,FOOO8,FuOn} to the polyfold theory of Hofer, Wysocki and Zehnder \cite{Hofe,HWZ1,HWZ2,HWZ3,HWZ4,HWZ5,HWZ6}. Here is a very brief introduction to this:
\begin{itemize}
\setlength{\itemsep}{0pt}
\setlength{\parsep}{0pt}
\item An {\it sc-Banach space\/} $\cV$ is a sequence $\cV=(\cV_0\supset\cV_1\supset\cV_2\supset\cdots),$ where the $\cV_i$ are Banach spaces, the inclusions $\cV_{i+1}\hookra\cV_i$ are compact, bounded linear maps, and $\cV_\iy=\bigcap_{i\ge 0}\cV_i$ is dense in every~$\cV_i$.

An example to bear in mind is if $M$ is a compact manifold, $E\ra M$ a smooth vector bundle, $\al\in(0,1)$, and $\cV_k=C^{k,\al}(E)$ for $k=0,1,\ldots.$ 
\item If $\cV=(\cV_0\supset\cV_1\supset\cdots)$ and $\cW=(\cW_0\supset\cW_1\supset\cdots)$ are sc-Banach spaces and $\cU\subseteq\cV_0$ is open, a map $f:\cU\ra\cW_0$ is called {\it sc-smooth\/} roughly speaking if $f(\cU\cap\cV_i)\subseteq\cW_i$ for all $i=0,1,\ldots$ and $f\vert_{\cU\cap\cV_{i+k}}:\cU\cap\cV_{i+k}\ra\cW_i$ is a $C^k$-map of Banach manifolds for all~$i,k\ge 0$.

(Actually, sc-smoothness is defined differently, but has the property that if $f$ is sc-smooth then $f\vert_{\cU\cap\cV_{i+k}}:\cU\cap\cV_{i+k}\ra\cW_i$ is $C^k$ for all $i,k$, and if $f\vert_{\cU\cap\cV_{i+k}}:\cU\cap\cV_{i+k}\ra\cW_i$ is $C^{k+1}$ for all $i,k$ then $f$ is sc-smooth.)
\item Let $\cV=(\cV_0\supset\cV_1\supset\cdots)$ be an sc-Banach space and $\cU\subseteq\cV_0$ be open. An {\it sc\/$^\iy$-retraction\/} is an sc-smooth map $r:\cU\ra\cU$ with $r\ci r=r$. Set $\O=\Im r\subset\cV$. We call $(\O,\cV)$ a {\it local sc-model}.

If $\cV$ is finite-dimensional then $\O$ is just a smooth manifold. But in infinite dimensions, new phenomena occur, and the tangent spaces $T_x\O$ can vary discontinuously with $x\in\O$. This is important for `gluing'.
\item An {\it M-polyfold chart\/} $(\O,\cV,\psi)$ on a topological space $Z$ is a local sc-model $(\O,\cV)$ and a homeomorphism $\psi:\O\ra\Im\psi$ with an open set~$\Im\psi\subset Z$. 
\item M-polyfold charts $(\O,\cV,\psi),(\ti\O,\ti\cV,\ti\psi)$ on $Z$ are {\it compatible\/} if $\ti\psi^{-1}\ci\psi\ci r:\cU\ra\ti\cV$ and $\psi^{-1}\ci\ti\psi\ci\ti r:\ti\cU\ra\cV$ are sc-smooth, where $\cU\subset\cV_0$, $\ti\cU\subset\ti\cV_0$ are open and $r:\cU\ra\cU$, $\ti r:\ti\cU\ra\ti\cU$ are sc-smooth with $r\ci r=r$, $\ti r\ci\ti r=\ti r$ and $\Im r=\psi^{-1}(\Im\ti\psi)\subseteq\O$, $\Im\ti r=\ti\psi^{-1}(\Im\psi)\subseteq\ti\O$.
\item An {\it M-polyfold\/} is roughly a metrizable topological space $Z$ with a maximal atlas of pairwise compatible M-polyfold charts.
\item {\it Polyfolds\/} are the orbifold version of M-polyfolds, proper \'etale groupoids in M-polyfolds.
\item A {\it polyfold Fredholm structure\/} $\cP$ on a metrizable topological space $X$ writes $X$ as the zeroes of an sc-Fredholm section ${\mathfrak s}:\fV\ra\fE$ of a strong polyfold vector bundle $\fE\ra\fV$ over a polyfold $\fV$. 
\end{itemize}
This is all rather complicated. The motivation for local sc-models $(\O,\cV)$ is that they can be used to describe functional-analytic problems involving `gluing', `bubbling', and `neck-stretching', including moduli spaces of $J$-holomorphic curves with singularities of various kinds.

The polyfold programme  \cite{Hofe,HWZ1,HWZ2,HWZ3,HWZ4,HWZ5,HWZ6} aims to show that moduli spaces of $J$-holomorphic curves in symplectic geometry may be given a polyfold Fredholm structure, and that compact spaces with oriented polyfold Fredholm structures have virtual chains and virtual classes. One can then use these virtual chains/classes to define big theories in symplectic geometry, such as Gromov--Witten invariants or Symplectic Field Theory. Constructing a polyfold Fredholm structure on a moduli space of $J$-holomorphic curves involves far fewer arbitrary choices than defining a Kuranishi structure.

Yang proves \cite[Th.~3.1.7]{Yang1} (see also \cite[\S 2.6]{Yang3}):

\begin{thm} Suppose we are given a `polyfold Fredholm structure'\/ $\cP$ on a compact metrizable topological space $X,$ that is, we write $X$ as the zeroes of an sc-Fredholm section ${\mathfrak s}:\fV\ra\fE$ of a strong polyfold vector bundle $\fE\ra\fV$ over a polyfold\/ $\fV,$ where $\mathfrak s$ has constant Fredholm index $n\in\Z$. Then we can construct a DY Kuranishi structure $\cK$ on $X,$ of virtual dimension $n,$ which is independent of choices up to R-equivalence.
\label{kuAthm3}
\end{thm}

In the survey \cite{Yang3}, Yang announces further results for which the proofs were not available at the time of writing. These include:
\begin{itemize}
\setlength{\itemsep}{0pt}
\setlength{\parsep}{0pt}
\item[(a)] Yang defines `R-equivalence' of polyfold Fredholm structures on $X$ \cite[Def.~2.14]{Yang3}, and claims \cite[\S 2.8]{Yang3} that Theorem \ref{kuAthm3} extends to a 1-1 correspondence between R-equivalence classes of polyfold Fredholm structures on $X$, and R-equivalence classes of DY Kuranishi structures $\cK$ on $X$.
\item[(b)] In \cite[\S 2.4]{Yang3}, Yang claims that R-equivalence extends as an equivalence relation to FOOO Kuranishi structures, and every R-equivalence class of FOOO Kuranishi structures contains a DY Kuranishi structure. Hence the 1-1 correspondence in (a) also extends to a 1-1 correspondence with R-equivalence classes of FOOO Kuranishi structures.
\item[(c)] Yang claims that virtual chains or virtual classes for polyfolds and for FOOO/DY Kuranishi spaces agree under (a),(b).
\item[(d)] Yang says \cite[p.~26, p.~46]{Yang3} that in future work he will make spaces with DY Kuranishi structures into a category~$\Kur_{\rm DY}$.
\end{itemize}

These results would enable a clean translation between the polyfold and Kuranishi space approaches to symplectic geometry. It seems likely that in (d) there will be an equivalence of categories $\Kur_{\rm DY}\simeq\Ho(\Kur)$, for $\Kur$ as in~\S\ref{ku4}.

\section{Basics of 2-categories}
\label{kuB}

Finally we discuss 2-categories, both strict and weak. Some references are Borceux \cite[\S 7]{Borc}, Kelly and Street \cite{KeSt}, and Behrend et al.~\cite[App.~B]{BEFF}.

\subsection{Strict and weak 2-categories}
\label{kuB1}

\begin{dfn} A {\it strict\/} 2-{\it category\/} $\bs\cC$ consists of
a class of {\it objects\/} $\Obj(\bs\cC)$, for all $X,Y\in\Obj(\bs\cC)$ an essentially small category $\Hom(X,Y)$, for all $X,Y,Z$ in $\Obj(\bs\cC)$ a functor $\mu_{X,Y,Z}:\Hom(X,Y)\t\Hom(Y,Z)\ra\Hom(X,Z)$ called {\it composition}, and for all $X$ in $\Obj(\bs\cC)$ an object $\id_X$ in $\Hom(X,X)$ called the {\it identity $1$-morphism}. These must satisfy
the {\it associativity property}, that
\e
\mu_{W,Y,Z}\ci(\mu_{W,X,Y}\t\id_{\Hom(Y,Z)})
=\mu_{W,X,Z}\ab\ci\ab(\id_{\Hom(W,X)}\ab\t\mu_{X,Y,Z})
\label{kuBeq1}
\e
as functors $\Hom(W,X)\t\Hom(X,Y)\t\Hom(Y,Z)\ra\Hom(W,X)$, and the
{\it identity property}, that
\e
\mu_{X,X,Y}(\id_X,-)=\mu_{X,Y,Y}(-,\id_Y)=\id_{\Hom(X,Y)}
\label{kuBeq2}
\e
as functors $\Hom(X,Y)\ra\Hom(X,Y)$.

Objects $f$ of $\Hom(X,Y)$ are called 1-{\it morphisms}, written
$f:X\ra Y$. For 1-morphisms $f,g:X\ra Y$, morphisms $\eta\in
\Hom_{\Hom(X,Y)}(f,g)$ are called 2-{\it morphisms}, written
$\eta:f\Ra g$. Thus, a 2-category has objects $X$, and two kinds of
morphisms: 1-morphisms $f:X\ra Y$ between objects, and 2-morphisms
$\eta:f\Ra g$ between 1-morphisms.

A {\it weak\/ $2$-category}, or {\it bicategory}, is like a strict 2-category, except that the equations of functors \eq{kuBeq1}, \eq{kuBeq2} are required to hold only up to specified natural isomorphisms. That is, a weak 2-category $\bs\cC$ consists of data $\Obj(\bs\cC),\Hom(X,Y),\mu_{X,Y,Z},\id_X$ as above, but in place of \eq{kuBeq1}, a natural isomorphism of functors 
\e
\al:\mu_{W,Y,Z}\!\ci\!(\mu_{W,X,Y}\!\t\!\id_{\Hom(Y,Z)})
\Longra\mu_{W,X,Z}\!\ci\!(\id_{\Hom(W,X)}\!\t\!\mu_{X,Y,Z}),
\label{kuBeq3}
\e
and in place of \eq{kuBeq2}, natural isomorphisms
\e
\be:\mu_{X,X,Y}(\id_X,-)\!\Longra\! \id_{\Hom(X,Y)},\;\>
\ga:\mu_{X,Y,Y}(-,\id_Y)\!\Longra\! \id_{\Hom(X,Y)}.
\label{kuBeq4}
\e
These $\al,\be,\ga$ must satisfy identities which we give below in \eq{kuBeq8} and~\eq{kuBeq11}.

A strict 2-category $\bs\cC$ can be regarded as an example of a weak 2-category, in which the natural isomorphisms $\al,\be,\ga$ in \eq{kuBeq3}--\eq{kuBeq4} are the identities.
\label{kuBdef1}
\end{dfn}

We now unpack Definition \ref{kuBdef1}, making it more explicit.

There are three kinds of composition in a 2-category, satisfying
various associativity relations. If $f:X\ra Y$ and $g:Y\ra Z$ are
1-morphisms then $\mu_{X,Y,Z}(f,g)$ is the {\it composition of\/ $1$-morphisms}, written $g\ci f:X\ra Z$. If $f,g,h:X\ra Y$ are 1-morphisms and $\eta:f\Ra g$, $\ze:g\Ra h$ are 2-morphisms then composition of $\eta,\ze$ in $\Hom(X,Y)$ gives the {\it vertical composition of\/ $2$-morphisms} of $\eta,\ze$, written $\ze\od\eta:f\Ra h$, as a diagram\vskip -15pt\begin{equation*}
\xymatrix@C=25pt{ X \rruppertwocell^f{\eta} \rrlowertwocell_h{\ze}
\ar[rr]_(0.35)g && Y & \ar@{~>}[r] && X
\rrtwocell^f_h{{}\,\,\,\,\ze\od\eta\!\!\!\!\!} && Y.}
\end{equation*}\vskip -10pt
\noindent Vertical composition is associative.

If $f,\dot f:X\ra Y$ and $g,\dot g:Y\ra Z$ are 1-morphisms and
$\eta:f\Ra\dot f$, $\ze:g\Ra\dot g$ are 2-morphisms then
$\mu_{X,Y,Z}(\eta,\ze)$ is the {\it horizontal composition of\/
$2$-morphisms}, written $\ze*\eta:g\ci f\Ra\dot g\ci\dot f$, as a
diagram\vskip -15pt\begin{equation*}
\xymatrix@C=20pt{ X \rrtwocell^f_{\dot f}{\eta} && Y
\rrtwocell^g_{\dot g}{\ze} && Z & \ar@{~>}[r] && X \rrtwocell^{g\ci
f}_{\dot g\ci\dot f}{{}\,\,\,\ze*\eta\!\!\!\!\!} && Z. }
\end{equation*}\vskip -5pt
\noindent As $\mu_{X,Y,Z}$ is a functor, these satisfy {\it compatibility of vertical and horizontal composition\/}: given a diagram of 1- and 2-morphisms\vskip -15pt\begin{equation*}
\xymatrix@C=20pt{ X \rruppertwocell^f{\eta} \rrlowertwocell_{\ddot f}{\dot\eta} \ar[rr]^(0.3){\dot f} && Y \rruppertwocell^g{\ze} \rrlowertwocell_{\ddot g}{\dot\ze} \ar[rr]^(0.3){\dot g} && Z, }
\end{equation*}\vskip -10pt
\noindent we have
\e
(\dot\ze\od\ze)*(\dot\eta\od\eta)=(\dot\ze*\dot\eta)\od(\ze*\eta):g\ci f\Longra \ddot g\ci\ddot f.
\label{kuBeq5}
\e
There are also two kinds of identity: {\it identity\/
$1$-morphisms\/} $\id_X:X\ra X$ and {\it identity\/
$2$-morphisms\/}~$\id_f:f\Ra f$.

In a strict 2-category $\bs\cC$, composition of 1-morphisms is strictly associative, $(g\ci f)\ci e=g\ci(f\ci e)$, and horizontal composition of 2-morphisms is strictly associative, $(\ze*\eta)*\ep=\ze*(\eta*\ep)$. In a weak 2-category $\bs\cC$, composition of 1-morphisms is associative up to specified 2-isomorphisms. That is, if $e:W\ra X$, $f:X\ra Y$, $g:Y\ra Z$ are 1-morphisms in $\bs\cC$ then the natural isomorphism $\al$ in \eq{kuBeq3} gives a 2-isomorphism
\e
\al_{g,f,e}:(g\ci f)\ci e\Longra g\ci(f\ci e).
\label{kuBeq6}
\e
As $\al$ is a natural isomorphism, given 1-morphisms $e,\dot e:W\ra X$, $f,\dot f:X\ra Y$, $g,\dot g:Y\ra Z$ and 2-morphisms $\ep:e\Ra\dot e$, $\eta:f\Ra\dot f$, $\ze:g\Ra\dot g$ in $\bs\cC$, the following diagram of 2-morphisms must commute:
\e
\begin{gathered}
\xymatrix@C=150pt@R=14pt{ *+[r]{(g\ci f)\ci e} \ar@{=>}[r]_{\al_{g,f,e}} \ar@{=>}[d]^{(\ze*\eta)*\ep} & *+[l]{g\ci (f\ci e)} \ar@{=>}[d]_{\ze*(\eta*\ep)} \\
*+[r]{(\dot g\ci\dot f)\ci\dot e} \ar@{=>}[r]^{\al_{\dot g,\dot f,\dot e}} & *+[l]{\dot g\ci(\dot f\ci\dot e).\!\!} }
\end{gathered}
\label{kuBeq7}
\e

The $\al_{g,f,e}$ must satisfy the {\it associativity coherence axiom\/}: if $d:V\ra W$ is another 1-morphism, then the following diagram of 2-morphisms must commute:
\e
\begin{gathered}
\xymatrix@C=70pt@R=14pt{ *+[r]{((g\ci f)\ci e)\ci d} \ar@{=>}[r]_(0.6){\al_{g,f,e}*\id_d} \ar@{=>}[d]^{\al_{g\ci f,e,d}} & (g\ci (f\ci e)) \ci d \ar@{=>}[r]_(0.4){\al_{g,f\ci e,d}} & *+[l]{g\ci ((f\ci e) \ci d)} \ar@{=>}[d]_{\id_g*\al_{f,e,d}} \\
*+[r]{(g\ci f)\ci (e \ci d)} \ar@{=>}[rr]^{\al_{g,f,d\ci e}} && *+[l]{g\ci (f\ci (e \ci d)).\!\!}}
\end{gathered}
\label{kuBeq8}
\e

In a strict 2-category $\bs\cC$, given a 1-morphism $f:X\ra Y$, the identity 1-morphisms $\id_X,\id_Y$ satisfy $f\ci\id_X=\id_Y\ci f=f$. In a weak 2-category $\bs\cC$, the natural isomorphisms $\be,\ga$ in \eq{kuBeq4} give 2-isomorphisms
\e
\be_f:f\ci\id_X\Longra f,\qquad \ga_f:\id_Y\ci f\Longra f.
\label{kuBeq9}
\e
As $\be,\ga$ are natural isomorphisms, if $\eta:f\Ra\dot f$ is a 2-morphism we must have
\e
\begin{split}
\eta\od\be_f&=\be_{\dot f}\od(\eta*\id_{\id_X}):f\ci\id_X\Ra\dot f,\\
\eta\od\ga_f&=\ga_{\dot f}\od(\id_{\id_Y}*\eta):\id_Y\ci f\Ra\dot f.
\end{split}
\label{kuBeq10}
\e

The $\be_f,\ga_f$ must satisfy the {\it identity coherence axiom\/}: if $g:Y\ra Z$ is another 1-morphism, then the following diagram of 2-morphisms must commute:
\e
\begin{gathered}
\xymatrix@C=120pt@R=-1pt{ *+[r]{(g\ci\id_Y)\ci f} \ar@{=>}[dr]^{\be_g*\id_f} \ar@{=>}[dd]^{\al_{g,\id_Y,f}} \\
& g\ci f. \\
*+[r]{g\ci(\id_Y\ci f)} \ar@{=>}[ur]_{\id_g*\al_f} }
\end{gathered}
\label{kuBeq11}
\e

A basic example of a strict 2-category is the {\it $2$-category of categories\/} $\mathfrak{Cat}$, with objects small categories $\cC$, 1-morphisms
functors $F:\cC\ra\cD$, and 2-morphisms natural transformations
$\eta:F\Ra G$ for functors $F,G:\cC\ra\cD$. Orbifolds naturally form
a 2-category (strict or weak, depending on the definition), as in \S\ref{ku45}, and so do stacks in algebraic geometry.

In a 2-category $\bs\cC$, there are three notions of when objects $X,Y$
in $\bs\cC$ are `the same': {\it equality\/} $X=Y$, and 1-{\it
isomorphism}, that is we have 1-morphisms $f:X\ra Y$, $g:Y\ra X$
with $g\ci f=\id_X$ and $f\ci g=\id_Y$, and {\it equivalence}, that
is we have 1-morphisms $f:X\ra Y$, $g:Y\ra X$ and 2-isomorphisms
$\eta:g\ci f\Ra\id_X$ and $\ze:f\ci g\Ra\id_Y$. Usually equivalence
is the correct notion. By \cite[Prop.~B.8]{BEFF}, we can also choose
$\eta,\ze$ to satisfy some extra identities:

\begin{prop} Let\/ $\bs\cC$ be a weak\/ $2$-category, and\/ $f:X\ra
Y$ be an equivalence in $\bs\cC$. Then there exist a $1$-morphism\/
$g:Y\ra X$ and\/ $2$-isomorphisms $\eta:g\ci f\Ra\id_X$ and\/
$\ze:f\ci g\Ra\id_Y$ with\/ $\ze*\id_f=(\id_f*\eta)\od\al_{f,g,f}$ as
$2$-isomorphisms $(f\ci g)\ci f\Ra f,$ and\/ $\eta*\id_g=(\id_g*\ze)\od\al_{g,f,g}$ as
$2$-isomorphisms~$(g\ci f)\ci g\Ra g$.
\label{kuBprop}
\end{prop}

When we say that objects $X,Y$ in a 2-category $\bs\cC$ are {\it canonically equivalent}, we mean that there is a nonempty distinguished class of equivalences $f:X\ra Y$ in $\bs\cC$, and given any two such equivalences $f,g:X\ra Y$ there is a 2-isomorphism $\eta:f\Ra g$. Often there is a distinguished choice of such~$\eta$.

{\it Commutative diagrams\/} in 2-categories should in general only
commute {\it up to (specified)\/ $2$-isomorphisms}, rather than
strictly. A simple example of a commutative diagram in a 2-category
$\bs\cC$ is
\begin{equation*}
\xymatrix@C=50pt@R=8pt{ & Y \ar[dr]^g \ar@{=>}[d]^\eta \\
X \ar[ur]^f \ar[rr]_h && Z, }
\end{equation*}
which means that $X,Y,Z$ are objects of $\bs\cC$, $f:X\ra Y$, $g:Y\ra
Z$ and $h:X\ra Z$ are 1-morphisms in $\bs\cC$, and $\eta:g\ci f\Ra h$
is a 2-isomorphism.

Let $\bs\cC$ be a 2-category. The {\it homotopy category\/} $\Ho(\bs\cC)$ of $\bs\cC$ is the category whose objects are objects of $\bs\cC$, and whose morphisms $[f]:X\ra Y$ are 2-isomorphism classes $[f]$ of 1-morphisms $f:X\ra Y$ in $\bs\cC$. Then equivalences in $\bs\cC$ become isomorphisms in $\Ho(\bs\cC)$, 2-commutative diagrams in $\bs\cC$ become commutative diagrams in $\Ho(\bs\cC)$, and so on.

\subsection[2-functors, 2-natural transformations, and modifications]{2-functors, 2-natural transformations, modifications}
\label{kuB2}

Next we discuss 2-functors between 2-categories, following Borceux \cite[\S 7.2, \S 7.5]{Borc} and Behrend et al.~\cite[\S B.4]{BEFF}.

\begin{dfn} Let $\bs\cC,\bs\cD$ be strict 2-categories. A {\it strict\/ $2$-functor\/} $F:\bs\cC\ra\bs\cD$ assigns an object $F(X)$ in $\bs\cD$ for each object $X$ in $\bs\cC$, a 1-morphism $F(f):F(X)\ra F(Y)$ in $\bs\cD$ for each 1-morphism $f:X\ra Y$ in $\bs\cC$, and a 2-morphism $F(\eta):F(f)\Ra F(g)$ in $\bs\cD$ for each 2-morphism $\eta:f\Ra g$ in $\bs\cC$, such that $F$ preserves all the structures on $\bs\cC,\bs\cD$, that is,
\ea
F(g\ci f)&=F(g)\ci F(f), & \!F(\id_X)&=\id_{F(X)}, & \!F(\ze*\eta)&=\!F(\ze)\!* \!F(\eta), 
\label{kuBeq12}\\
F(\ze\od\eta)&=F(\ze)\od F(\eta), & \!F(\id_f)&=\id_{F(f)}.
\label{kuBeq13}
\ea

Now let $\bs\cC,\bs\cD$ be weak 2-categories. Then strict 2-functors $F:\bs\cC\ra\bs\cD$ are not well-behaved. To fix this, we need to relax \eq{kuBeq12} to hold only up to specified 2-isomorphisms. A {\it weak\/ $2$-functor\/} (or {\it pseudofunctor\/}) $F:\bs\cC\ra\bs\cD$ assigns an object $F(X)$ in $\bs\cD$ for each object $X$ in $\bs\cC$, a 1-morphism $F(f):F(X)\ra F(Y)$ in $\bs\cD$ for each 1-morphism $f:X\ra Y$ in $\bs\cC$, a 2-morphism $F(\eta):F(f)\Ra F(g)$ in $\bs\cD$ for each 2-morphism $\eta:f\Ra g$ in $\bs\cC$, a 2-isomorphism $F_{g,f}:F(g)\ci F(f)\Ra F(g\ci f)$ in $\bs\cD$ for all 1-morphisms $f:X\ra Y$, $g:Y\ra Z$ in $\bs\cC$, and a 2-isomorphism $F_X:F(\id_X)\Ra \id_{F(X)}$ in $\bs\cD$ for all objects $X$ in $\bs\cC$ such that \eq{kuBeq13} holds, and for all $e:W\ra X$, $f:X\ra Y$, $g:Y\ra Z$ in $\bs\cC$ the following diagram of 2-isomorphisms commutes in $\bs\cD$:
\begin{equation*}
\xymatrix@C=87pt@R=15pt{ *+[r]{(F(g)\ci F(f))\ci F(e)} \ar@{=>}[d]^{\al_{F(g),F(f),F(e)}} \ar@{=>}[r]_(0.65){F_{g,f}*\id_{F(e)}} & F(g\ci f)\ci F(e) \ar@{=>}[r]_(0.4){F_{g\ci f,e}} & *+[l]{F((g\ci f)\ci e)} \ar@{=>}[d]_{F(\al_{g,f,e})} \\
*+[r]{F(g)\ci (F(f)\ci F(e))} \ar@{=>}[r]^(0.65){\id_{F(g)}*F_{f,e}} & F(g)\ci F(f\ci e) \ar@{=>}[r]^(0.4){F_{g,f\ci e}} &
*+[l]{F(g\ci (f\ci e)),\!\!} } 
\end{equation*}
and for all 1-morphisms $f:X\ra Y$ in $\bs\cC$, the following commute in $\bs\cD$:
\begin{equation*}
\xymatrix@!0@C=25pt@R=30pt{
*+[r]{F(f)\ci F(\id_X)} \ar@{=>}[rrrrrr]_(0.58){F_{f,\id_X}} \ar@{=>}[d]^{\id_{F(f)}*F_X} &&&&&& *+[l]{F(f\ci\id_X)} \ar@{=>}[d]_{F(\be_f)} & *+[r]{F(\id_Y)\ci F(f)} \ar@{=>}[rrrrrr]_(0.58){F_{\id_Y,f} } \ar@{=>}[d]^{F_Y*\id_{F(f)}} &&&&&& *+[l]{F(\id_Y\ci f)} \ar@{=>}[d]_{F(\ga_f)} \\
*+[r]{F(f)\ci \id_{F(X)}} \ar@{=>}[rrrrrr]^(0.6){\be_{F(f)}} &&&&&& *+[l]{F(f),\!\!} &
*+[r]{\id_{F(Y)}\ci F(f)} \ar@{=>}[rrrrrr]^(0.6){\ga_{F(f)}} &&&&&& *+[l]{F(f),\!\!} } 
\end{equation*}
and if $f,\dot f:X\ra Y$ and $g,\dot g:Y\ra Z$ are 1-morphisms and
$\eta:f\Ra\dot f$, $\ze:g\Ra\dot g$ are 2-morphisms in $\bs\cC$ then the following commutes in $\bs\cD$:
\begin{equation*}
\xymatrix@C=140pt@R=15pt{ *+[r]{F(g)\ci F(f)} \ar@{=>}[d]^{F(\ze)*F(\eta)} \ar@{=>}[r]_{F_{g,f}} & *+[l]{F(g\ci f)} \ar@{=>}[d]_{F(\ze*\eta)} \\
*+[r]{F(\dot g)\ci F(\dot f)} \ar@{=>}[r]^{F_{\dot g,\dot f}} & *+[l]{F(\dot g\ci\dot f).\!\!} } 
\end{equation*}

There are obvious notions of {\it composition\/} $G\ci F$ of strict and weak 2-functors $F:\bs\cC\ra\bs\cD$, $G:\bs\cD\ra\bs\cE$, {\it identity\/ $2$-functors\/} $\id_{\bs\cC}$, and so on.

If $\bs\cC,\bs\cD$ are strict 2-categories, then a strict $2$-functor $F:\bs\cC\ra\bs\cD$ can be made into a weak 2-functor by taking all $F_{g,f},F_X$ to be identity 2-morphisms. 
\label{kuBdef2}
\end{dfn}

Here is the 2-category analogue of natural transformations of functors:

\begin{dfn} Let $\bs\cC,\bs\cD$ be weak 2-categories and $F,G:\bs\cC\ra\bs\cD$ be weak 2-functors. A {\it weak\/ $2$-natural transformation\/} (or {\it pseudo-natural transformation\/}) $\Th:F\Ra G$ assigns a 1-morphism $\Th(X):F(X)\ra G(X)$ in $\bs\cD$ for all objects $X$ in $\bs\cC$ and a 2-isomorphism $\Th(f):\Th(Y)\ci F(f)\Ra G(f)\ci\Th(X)$ in $\bs\cD$ for all 1-morphisms $f:X\ra Y$ in $\bs\cC$, such that if $\eta:f\Ra g$ is a 2-morphism in $\bs\cC$ then
\begin{align*}
(G(\eta)*\id_{\Th(X)})\od\Th(f)&=\Th(g)\od(\id_{\Th(Y)}*F(\eta)):\\
&\Th(Y)\ci F(f)\longra G(g)\ci\Th(X),
\end{align*}
and if $f:X\ra Y$, $g:Y\ra Z$ are 1-morphisms in $\bs\cC$ then the following diagram of 2-isomorphisms commutes in $\bs\cD$:
\begin{equation*}
\xymatrix@C=60pt@R=17pt{ *+[r]{(\Th(Z)\ci F(g))\ci F(f)} \ar@{=>}[r]_(0.7){\raisebox{-8pt}{$\st\al_{\Th(Z),F(g),F(f)}$}} \ar@{=>}[d]^{\Th(g)*\id_{F(f)}} & {\Th(Z)\ci(F(g)\ci F(f))} \ar@{=>}[r]_(0.45){\raisebox{-8pt}{$\st\id_{\Th(Z)}*F_{g,f}$}} & *+[l]{\Th(Z)\ci(F(g\ci f))} \ar@{=>}[d]_{\Th(g\ci f)} \\
*+[r]{(G(g)\ci\Th(Y))\ci F(f)} \ar@{=>}[d]^{\al_{G(g),\Th(Y),F(f)}} && *+[l]{G(g\ci f)\ci\Th(X)}
\\
*+[r]{G(g)\ci(\Th(Y)\ci F(f))} \ar@{=>}[r]^(0.6){\raisebox{8pt}{$\st\id_{G(g)}*\Th(f)$}} & {G(g)\ci(G(f)\ci\Th(X))} \ar@{=>}[r]^(0.4){\raisebox{8pt}{$\st\al_{G(g),G(f),\Th(X)}^{-1}$}} & *+[l]{(G(g)\ci G(f))\ci\Th(X)),\!\!} \ar@{=>}[u]^{G_{g,f}*\id_{\Th(X)}} }
\end{equation*}
and if $X\in\bs\cC$ then the following diagram of 2-isomorphisms commutes in $\bs\cD$:
\begin{equation*}
\xymatrix@C=75pt@R=15pt{ *+[r]{\Th(X)\ci F(\id_X)} \ar@{=>}[r]_(0.6){\Th(\id_X)} \ar@{=>}[d]^{ \id_{\Th(X)}* F_X} & {G(\id_X)\ci \Th(X)} \ar@{=>}[r]_(0.45){G_X*\id_{\Th(X)}} & *+[l]{\id_{G(X)}\ci \Th(X)} \ar@{=>}[d]_{\ga_{\Th(X)}} \\
*+[r]{\Th(X)\ci\id_{F(X)}} \ar@{=>}[rr]^{\be_{\Th(X)}} && *+[l]{\Th(X).\!\!} }
\end{equation*}

\label{kuBdef3}
\end{dfn}

Just as the `category of (small) categories' is actually a (strict) 2-category, so the `category of (weak) 2-categories' is actually a 3-category (which we will not define). The 3-morphisms in this 3-category, morphisms between weak 2-natural transformations, are called {\it modifications}.

\begin{dfn} Let $\bs\cC,\bs\cD$ be weak 2-categories, $F,G:\bs\cC\ra\bs\cD$ be weak 2-functors, and $\Th,\Phi:F\Ra G$ be weak 2-natural transformations. A {\it modification\/} $\aleph:F\RRa G$ assigns a 2-isomorphism $\aleph(X):\Th(X)\Ra\Phi(X)$ in $\bs\cD$ for all objects $X$ in $\bs\cC$, such that for all 1-morphisms $f:X\ra Y$ in $\bs\cC$ we have
\begin{align*}
(\id_{G(f)}*\aleph(X))\od\Th(f)&=\Phi(f)\od(\aleph(Y)*\id_{F(f)}):\\
&\Th(Y)\ci F(f)\Longra \Phi(X)\ci G(Y).
\end{align*}
There are obvious notions of composition of modifications, identity modifications, and so on.

A weak 2-natural transformation $\Th:F\Ra G$ is called an {\it equivalence of\/ $2$-functors\/} if there exist a weak 2-natural transformation $\Phi:G\Ra F$ and modifications $\aleph:\Phi\ci\Th\RRa \id_F$ and $\beth:\Th\ci\Phi\RRa\id_G$. Equivalence of 2-functors is a good notion of when weak 2-functors $F,G:\bs\cC\ra\bs\cD$ are `the same'.

A weak 2-functor $F:\bs\cC\ra\bs\cD$ is called an {\it equivalence of weak\/ $2$-categories\/} if there exists a weak 2-functor $G:\bs\cD\ra\bs\cC$ and equivalences of 2-functors $\Th:G\ci F\Ra \id_{\bs\cC}$, $\Phi:F\ci G\Ra\id_{\bs\cD}$. Equivalence of weak 2-categories is a good notion of when weak 2-categories $\bs\cC,\bs\cD$ are `the same'.
\label{kuBdef4}
\end{dfn}

Here are some well-known facts about 2-categories:
\begin{itemize}
\setlength{\itemsep}{0pt}
\setlength{\parsep}{0pt}
\item[(i)] Every weak 2-category $\bs\cC$ is equivalent as a weak 2-category to a strict 2-category $\bs\cC'$, that is, weak 2-categories can always be strictified. Thus, in Theorem \ref{ku1thm}, it is not surprising that the weak 2-category of Kuranishi spaces $\Kur$ is equivalent to the strict 2-category of d-orbifolds $\dOrb$.
\item[(ii)] If $\bs\cC,\bs\cD$ are strict 2-categories, and $F:\bs\cC\ra\bs\cD$ is a weak 2-functor, it may not be true that $F$ is equivalent to a strict 2-functor $F':\bs\cC\ra\bs\cD$ (though this does hold if $\bs\cD=\mathfrak{Cat}$, the strict 2-category of categories). That is, weak 2-functors cannot necessarily be strictified. 

Even if one is working with strict 2-categories, weak 2-functors are often the correct notion of functor between them.
\item[(iii)] A weak 2-functor $F:\bs\cC\ra\bs\cD$ is an equivalence of weak 2-categories, as in Definition \ref{kuBdef4}, if and only if for all objects $X,Y$ in $\bs\cC$, the functor $F_{X,Y}:\Hom_{\bs\cC}(X,Y)\ra\Hom_{\bs\cD}(F(X),F(Y))$ is an equivalence of categories, and the map induced by $F$ from equivalence classes of objects in $\bs\cC$ to equivalence classes of objects in $\bs\cD$ is surjective (and hence a bijection).

\end{itemize}

\subsection{Fibre products in 2-categories}
\label{kuB3}

We define fibre products in 2-categories, following
\cite[Def.~B.13]{BEFF}.

\begin{dfn} Let $\bs\cC$ be a strict 2-category and $g:X\ra Z$, $h:Y\ra Z$ be 1-morphisms in $\bs\cC$. A {\it fibre product\/} $X\t_ZY$ in $\bs\cC$
consists of an object $W$, 1-morphisms $\pi_X:W\ra X$ and $\pi_Y:W\ra Y$ and a 2-isomorphism $\eta:g\ci\pi_X\Ra h\ci\pi_Y$ in $\bs\cC$ with the
following universal property: suppose $\pi_X':W'\ra X$ and
$\pi_Y':W'\ra Y$ are 1-morphisms and $\eta':g\ci\pi_X'\Ra
h\ci\pi_Y'$ is a 2-isomorphism in $\bs\cC$. Then there should exist a
1-morphism $b:W'\ra W$ and 2-isomorphisms $\ze_X:\pi_X\ci
b\Ra\pi_X'$, $\ze_Y:\pi_Y\ci b\Ra\pi_Y'$ such that the following
diagram of 2-isomorphisms commutes:
\e
\begin{gathered}
\xymatrix@C=100pt@R=14pt{ *+[r]{g\ci\pi_X\ci b} \ar@{=>}[r]_{\eta*\id_b}
\ar@{=>}[d]^{\,\id_g*\ze_X} & *+[l]{h\ci\pi_Y\ci b}
\ar@{=>}[d]_{\id_h*\ze_Y} \\ *+[r]{g\ci\pi_X'}
\ar@{=>}[r]^{\eta'} & *+[l]{h\ci\pi_Y'.\!\!} }
\end{gathered}
\label{kuBeq14}
\e
Furthermore, if $\ti b,\ti\ze_X,\ti\ze_Y$ are alternative choices of
$b,\ze_X,\ze_Y$ then there should exist a unique 2-isomorphism
$\th:\ti b\Ra b$ with
\begin{equation*}
\ti\ze_X=\ze_X\od(\id_{\pi_X}*\th)\quad\text{and}\quad
\ti\ze_Y=\ze_Y\od(\id_{\pi_Y}*\th).
\end{equation*}
If a fibre product $X\t_ZY$ in $\bs\cC$ exists then it is unique up to
canonical equivalence. That is, if $W,W'$ are two choices for $X\t_ZY$ then there is an equivalence $b:W\ra W'$ in $\bs\cC$, which is canonical up to 2-isomorphism.

If instead $\bs\cC$ is a weak 2-category, we must replace \eq{kuBeq14} by
\begin{equation*}
\xymatrix@C=75pt@R=14pt{ 
*+[r]{(g\ci\pi_X)\ci b} \ar@{=>}[d]^{\al_{g,\pi_X,b}}
\ar@{=>}[r]_(0.6){\eta*\id_b} & {(h\ci\pi_Y)\ci b}
\ar@{=>}[r]_(0.4){\al_{h,\pi_Y,b}} & *+[l]{h\ci(\pi_Y\ci b)\!\!}
\ar@{=>}[d]_{\id_h*\ze_Y} \\ 
*+[r]{g\ci(\pi_X\ci b)} \ar@{=>}[r]^(0.6){\id_g*\ze_X} &
g\ci\pi_X' \ar@{=>}[r]^(0.4){\eta'} & *+[l]{h\ci\pi_Y'.\!\!} }
\end{equation*}

\label{kuBdef5}
\end{dfn}

Orbifolds, and stacks in algebraic geometry, form 2-categories, and
Definition \ref{kuBdef5} is the right way to define fibre products
of orbifolds or stacks.

\medskip

\noindent{\small\sc The Mathematical Institute, Radcliffe
Observatory Quarter, Woodstock Road, Oxford, OX2 6GG, U.K.

\noindent E-mail: {\tt joyce@maths.ox.ac.uk.}}


\begin{thebibliography}{99}
\addcontentsline{toc}{section}{References}


\bibitem{ALR} A. Adem, J. Leida and Y. Ruan, {\it Orbifolds and
Stringy Topology}, Cambridge Tracts in Math. 171, Cambridge
University Press, Cambridge, 2007.

\bibitem{BEFF} K. Behrend, D. Edidin, B. Fantechi, W. Fulton, L.
G\"ottsche and A. Kresch, {\it Introduction to stacks}, book in
preparation, 2010.

\bibitem{BeXu} K. Behrend and P. Xu, {\it Differentiable stacks and
gerbes}, J. Symplectic Geom. 9 (2011), 285--341. \href{http://arxiv.org/abs/math/0605694}{math.DG/0605694}.

\bibitem{BoVo} M. Boardman and R. Vogt, {\it Homotopy invariant algebraic structures on topological spaces}, Lectures Notes in Math. 347, Springer Verlag, 1973.

\bibitem{Borc} F. Borceux, {\it Handbook of categorical algebra 1.
Basic category theory}, Encyclopedia of Mathematics and its
Applications 50, Cambridge University Press, 1994.

\bibitem{Cerf} J. Cerf, {\it Topologie de certains espaces de
plongements}, Bull. Soc. Math. France 89 (1961), 227--380.

\bibitem{ChRu} W. Chen and Y. Ruan, {\it Orbifold Gromov--Witten
theory}, pages 25--86 in A. Adem, J. Morava and Y. Ruan, editors,
{\it Orbifolds in mathematics and physics}, Cont. Math. 310, A.M.S.,
Providence, RI, 2002. \href{http://arxiv.org/abs/math/0103156}{math.AG/0103156}.

\bibitem{Doua} A. Douady, {\it Vari\'et\'es \`a bord anguleux et
voisinages tubulaires}, S\'eminaire Henri Cartan 14 (1961-2), exp.
1, 1--11.

\bibitem{Fuka} K. Fukaya, {\it Cyclic symmetry and adic convergence in Lagrangian Floer theory}, Kyoto J. Math. 50 (2010), 521--590. \href{http://arxiv.org/abs/0907.4219}{arXiv:0907.4219}.

\bibitem{FOOO1} K. Fukaya, Y.-G. Oh, H. Ohta and K. Ono, {\it Lagrangian intersection Floer theory --- anomaly and obstruction}, Parts I \& II. AMS/IP Studies in Advanced Mathematics, 46.1 \& 46.2, A.M.S./International Press, 2009.

\bibitem{FOOO2} K. Fukaya, Y.-G. Oh, H. Ohta and K. Ono, {\it Lagrangian Floer theory on compact toric manifolds I}, Duke Math. J. 151 (2010), 23--174. \href{http://arxiv.org/abs/0802.1703}{arXiv:0802.1703}.

\bibitem{FOOO3} K. Fukaya, Y.-G. Oh, H. Ohta and K. Ono, {\it Lagrangian Floer theory on compact toric manifolds II: bulk deformations}, Selecta Math. 17 (2011), 609--711. \href{http://arxiv.org/abs/0810.5654}{arXiv:0810.5654}.

\bibitem{FOOO4} K. Fukaya, Y.-G. Oh, H. Ohta and K. Ono, {\it Anti-symplectic involution and Floer cohomology}, \href{http://arxiv.org/abs/0912.2646}{arXiv:0912.2646}, 2009.

\bibitem{FOOO5} K. Fukaya, Y.-G. Oh, H. Ohta and K. Ono, {\it Lagrangian Floer theory and mirror symmetry on compact toric manifolds}, \href{http://arxiv.org/abs/1009.1648}{arXiv:1009.1648}, 2010.

\bibitem{FOOO6} K. Fukaya, Y.-G. Oh, H. Ohta and K. Ono, {\it Technical details on Kuranishi structure and virtual fundamental chain}, \href{http://arxiv.org/abs/1209.4410}{arXiv:1209.4410}, 2012.

\bibitem{FOOO7} K. Fukaya, Y.-G. Oh, H. Ohta and K. Ono, {\it Shrinking good coordinate systems associated to Kuranishi structures}, \href{http://arxiv.org/abs/1405.1755}{arXiv:1405.1755}, 2014.

\bibitem{FOOO8} K. Fukaya, Y.-G. Oh, H. Ohta and K. Ono, {\it Kuranishi structure, Pseudo-holomorphic curve, and Virtual fundamental chain: Part 1}, \hfil\break \href{http://arxiv.org/abs/1503.07631}{arXiv:1503.07631}, 2014.

\bibitem{FuOn} K. Fukaya and K. Ono, {\it Arnold Conjecture and
Gromov--Witten invariant}, Topology 38 (1999), 933--1048.

\bibitem{GiMo} W.D. Gillam and S. Molcho, {\it Log differentiable spaces and manifolds with corners}, \href{http://arxiv.org/abs/1507.06752}{arXiv:1507.06752}, 2015.

\bibitem{Hart} R. Hartshorne, {\it Algebraic Geometry}, Graduate
Texts in Math. 52, Springer, New York, 1977.

\bibitem{HeMe} A. Henriques and D.S. Metzler, {\it Presentations of
noneffective orbifolds}, Trans. A.M.S. 356 (2004), 2481--2499.
\href{http://arxiv.org/abs/math/0302182}{math.AT/0302182}.

\bibitem{HeSt} H. Herrlich and G.E. Strecker, {\it Category Theory, an introduction}, Allyn and Bacon Inc., Boston, 1973.

\bibitem{Hofe} H. Hofer, {\it Polyfolds and Fredholm Theory}, \href{http://arxiv.org/abs/1412.4255}{arXiv:1412.4255}, 2014.

\bibitem{HWZ1} H. Hofer, K. Wysocki and E. Zehnder, {\it A general
Fredholm theory I: A splicing-based differential geometry}, J. Eur.
Math. Soc. 9 (2007), 841--876. \href{http://arxiv.org/abs/math/0612604}{math.FA/0612604}.

\bibitem{HWZ2} H. Hofer, K. Wysocki and E. Zehnder, {\it Integration theory for zero sets of polyfold Fredholm sections}, \href{http://arxiv.org/abs/0711.0781}{arXiv:0711.0781}, 2007.

\bibitem{HWZ3} H. Hofer, K. Wysocki and E. Zehnder, {\it A general Fredholm theory II: Implicit function theorems}, Geom. Funct. Anal. 18 (2009), 206--293. \href{http://arxiv.org/abs/0705.1310}{arXiv:0705.1310}.

\bibitem{HWZ4} H. Hofer, K. Wysocki and E. Zehnder, {\it A general Fredholm theory III: Fredholm functors and polyfolds}, Geom. Topol. 13 (2009), 2279--2387. \href{http://arxiv.org/abs/0810.0736}{arXiv:0810.0736}.

\bibitem{HWZ5} H. Hofer, K. Wysocki and E. Zehnder, {\it
Applications of polyfold theory I: the polyfolds of Gromov--Witten
theory}, \href{http://arxiv.org/abs/1107.2097}{arXiv:1107.2097}, 2011.

\bibitem{HWZ6} H. Hofer, K. Wysocki and E. Zehnder, {\it Polyfold and Fredholm theory I: basic theory in M-polyfolds}, \href{http://arxiv.org/abs/1407.3185}{arXiv:1407.3185}, 2014.

\bibitem{Jech} T. Jech, {\it Set Theory}, third millennium edition, Springer-Verlag, Berlin, 2003.

\bibitem{Joya} A. Joyal, {\it Quasi-categories and Kan complexes}, J. Pure Appl. Algebra 175 (2002), 207--222.

\bibitem{Joyc1} D. Joyce, {\it Kuranishi homology and Kuranishi
cohomology}, \href{http://arxiv.org/abs/0707.3572}{arXiv:0707.3572}, 2007.

\bibitem{Joyc2} D. Joyce, {\it Kuranishi homology and Kuranishi
cohomology: a User's Guide}, \href{http://arxiv.org/abs/0710.5634}{arXiv:0710.5634}, 2007.

\bibitem{Joyc3} D. Joyce, {\it On manifolds with corners}, pages
225--258 in S. Janeczko, J. Li and D.H. Phong, editors, {\it
Advances in Geometric Analysis}, Advanced Lectures in Mathematics
21, International Press, Boston, 2012. \href{http://arxiv.org/abs/0910.3518}{arXiv:0910.3518}.

\bibitem{Joyc4} D. Joyce, {\it Algebraic Geometry over\/ $C^\iy$-rings},
\href{http://arxiv.org/abs/1001.0023}{arXiv:1001.0023}, 2010.

\bibitem{Joyc5} D. Joyce, {\it An introduction to $C^\iy$-schemes
and\/ $C^\iy$-algebraic geometry}, pages 299--325 in H.-D. Cao and
S.-T. Yau, editors, {\it In memory of C.C. Hsiung: Lectures given at
the JDG symposium, Lehigh University, June 2010}, Surveys in
Differential Geometry 17, 2012. \href{http://arxiv.org/abs/1104.4951}{arXiv:1104.4951}.

\bibitem{Joyc6} D. Joyce, {\it An introduction to d-manifolds and
derived differential geometry}, pages 230--281 in L. Brambila-Paz, O. Garcia-Prada, P. Newstead and R.P. Thomas, editors, {\it Moduli spaces}, L.M.S. Lecture Notes 411, Cambridge University Press, 2014. \href{http://arxiv.org/abs/1206.4207}{arXiv:1206.4207}.

\bibitem{Joyc7} D. Joyce, {\it D-manifolds, d-orbifolds and derived
differential geometry: a detailed summary}, \href{http://arxiv.org/abs/1208.4948}{arXiv:1208.4948}, 2012.

\bibitem{Joyc8} D. Joyce, {\it D-manifolds and d-orbifolds: a
theory of derived differential geometry}, to be published by Oxford University Press, 2016. \hfil\break
Preliminary version (2012) available at \hfil\break \href{http://people.maths.ox.ac.uk/~joyce/dmanifolds.html}{\tt http://people.maths.ox.ac.uk/$\sim$joyce/dmanifolds.html}.

\bibitem{Joyc9} D. Joyce, {\it A generalization of manifolds with corners}, \href{http://arxiv.org/abs/1501.00401}{arXiv:1501.00401}, 2015.

\bibitem{Joyc10} D. Joyce, {\it Some new homology and cohomology theories of manifolds}, \href{http://arxiv.org/abs/1509.05672}{arXiv:1509.05672}, 2015.

\bibitem{Joyc11} D. Joyce, {\it Kuranishi spaces as a\/ $2$-category}, \href{http://arxiv.org/abs/1510.07444}{arXiv:1510.07444}, 2015.

\bibitem{Joyc12} D. Joyce, in preparation, 2016.

\bibitem{KeSt} G.M. Kelly and R.H. Street, {\it Review of the elements
of\/ $2$-categories}, pages 75--103 in Lecture Notes in Math. 420,
Springer-Verlag, New York, 1974.

\bibitem{KoMe} C. Kottke and R.B. Melrose, {\it Generalized blow-up
of corners and fibre products}, Trans. A.M.S. 367 (2015), 651--705. \href{http://arxiv.org/abs/1107.3320}{arXiv:1107.3320}.

\bibitem{Lerm} E. Lerman, {\it Orbifolds as stacks?},
Enseign. Math. 56 (2010), 315--363. \href{http://arxiv.org/abs/0806.4160}{arXiv:0806.4160}.

\bibitem{MaOu} J. Margalef-Roig and E. Outerelo Dominguez, {\it
Differential Topology}, North-Holland Math. Studies 173,
North-Holland, Amsterdam, 1992.

\bibitem{Mau} S. Ma'u, {\it Gluing pseudoholomorphic quilted disks}, \href{http://arxiv.org/abs/0909.3339}{arXiv:0909.3339}, 2009.

\bibitem{MaWo} S. Ma'u and C. Woodward, {\it Geometric realizations of the multiplihedra}, Compos. Math. 146 (2010), 1002--1028. \href{http://arxiv.org/abs/0802.2120}{arXiv:0802.2120}.

\bibitem{McDu} D. McDuff, {\it Orbifold atlases of Kuranishi type}, \href{http://arxiv.org/abs/1506.05350}{arXiv:1506.05350}, 2015.

\bibitem{McWe1} D. McDuff and K. Wehrheim, {\it The topology of Kuranishi atlases}, \hfil\break 
\href{http://arxiv.org/abs/1508.01844}{arXiv:1508.01844}, 2015.

\bibitem{McWe2} D. McDuff and K. Wehrheim, {\it The fundamental class of smooth Kuranishi atlases with trivial isotropy}, \href{http://arxiv.org/abs/1508.01560}{arXiv:1508.01560}, 2015.

\bibitem{McWe3} D. McDuff and K. Wehrheim, {\it Smooth Kuranishi atlases with isotropy}, \href{http://arxiv.org/abs/1508.01556}{arXiv:1508.01556}, 2015.

\bibitem{Melr1} R.B. Melrose, {\it Calculus of conormal
distributions on manifolds with corners}, IMRN 1992 (1992), 51--61.

\bibitem{Melr2} R.B. Melrose, {\it The Atiyah--Patodi--Singer Index
Theorem}, A.K. Peters, Wellesley, MA, 1993.

\bibitem{Melr3} R.B. Melrose, {\it Differential Analysis on
Manifolds with Corners}, unfinished book available at \href{http://math.mit.edu/~rbm}{\tt http://math.mit.edu/$\sim$rbm}, 1996.

\bibitem{Metz} D.S. Metzler, {\it Topological and smooth stacks},
\href{http://arxiv.org/abs/math/0306176}{math.DG/0306176}, 2003.

\bibitem{Moer} I. Moerdijk, {\it Orbifolds as groupoids: an
introduction}, pages 205--222 in A. Adem, J. Morava and Y. Ruan,
editors, {\it Orbifolds in Mathematics and Physics}, Cont. Math.
310, A.M.S., Providence, RI, 2002. \href{http://arxiv.org/abs/math/0203100}{math.DG/0203100}.

\bibitem{MoPr} I. Moerdijk and D.A. Pronk, {\it Orbifolds, sheaves
and groupoids}, K-theory 12 (1997), 3--21.

\bibitem{Mont} B. Monthubert, {\it Groupoids and pseudodifferential calculus on manifolds with corners}, J. Funct. Anal. 199 (2003),
243--286.

\bibitem{Pron} D. Pronk, {\it Etendues and stacks as bicategories of
fractions}, Compositio Math. 102 (1996), 243--303.

\bibitem{Sata} I. Satake, {\it On a generalization of the notion of
manifold}, Proc. Nat. Acad. Sci. U.S.A. 42 (1956), 359--363.

\bibitem{Shul} M.A. Shulman, {\it Set Theory for Category Theory}, \href{http://arxiv.org/abs/0810.1279}{arXiv:0810.1279}, 2008.

\bibitem{Thur} W. Thurston, {\it The geometry and topology of
three-manifolds}, Princeton lecture notes, Princeton, 1980.
Available at \hfil\break \href{http://library.msri.org/books/gt3m}{\tt http://library.msri.org/books/gt3m.}

\bibitem{WeWo1} K. Wehrheim and C. Woodward, {\it Pseudoholomorphic quilts}, \hfil\break \href{http://arxiv.org/abs/0905.1369}{arxiv:0905.1369}, 2009.

\bibitem{WeWo2} K. Wehrheim and C. Woodward, {\it Quilted Floer cohomology}, Geometry and Topology 14 (2010), 833--902. \href{http://arxiv.org/abs/0905.1370}{arXiv:0905.1370}.

\bibitem{WeWo3} K. Wehrheim and C. Woodward, {\it Functoriality for Lagrangian correspondences in Floer theory}, Quantum Topol. 1 (2010), 129--170. \hfil\break \href{http://arxiv.org/abs/0708.2851}{arxiv:0708.2851}.

\bibitem{Yang1} D. Yang, {\it A choice-independent theory of Kuranishi structures
and the polyfold--Kuranishi correspondence}, PhD thesis, New York University, 2014. Available online at \hfil\break
\href{http://webusers.imj-prg.fr/~dingyu.yang/thesis.pdf}{\tt http://webusers.imj-prg.fr/$\sim$dingyu.yang/thesis.pdf}.

\bibitem{Yang2} D. Yang, {\it The polyfold--Kuranishi correspondence I: A choice-independent theory of Kuranishi structures}, \href{http://arxiv.org/abs/1402.7008}{arXiv:1402.7008}, 2014.

\bibitem{Yang3} D. Yang, {\it Virtual harmony}, \href{http://arxiv.org/abs/1510.06849}{arXiv:1510.06849}, 2015.

\end{thebibliography}
\end{document}